%% file: MEMO784_V.-A.Nguyen.tex
\documentclass{memo-l}  

\usepackage{amstext    }
\usepackage{amsthm    }
\usepackage[mathscr]{eucal}
\usepackage{mathrsfs}
\usepackage{hyperref}

\usepackage{amsmath}
\usepackage{amssymb}
\usepackage{amscd}
\usepackage[refpage]{nomencl}

\newtheorem{theorem}{Theorem}[chapter]
\newtheorem{lemma}[theorem]{Lemma}
\newtheorem{proposition}[theorem]{Proposition}
\newtheorem{corollary}[theorem]{Corollary}

\theoremstyle{definition}
\newtheorem{definition}[theorem]{Definition}
\newtheorem{example}[theorem]{Example}

\theoremstyle{remark}
\newtheorem{remark}[theorem]{Remark}

\numberwithin{section}{chapter}
\numberwithin{equation}{chapter}

\newcommand{\cali}[1]{\mathscr{#1}}

\newcommand{\Leb}{{\rm Leb}}

\newcommand{\Vol}{{\rm Vol}}
\newcommand{\Har}{{\rm Har}}

\newcommand{\const}{{\rm const}}
\newcommand{\dist}{{\rm dist}}

\newcommand{\GL}{{\rm GL}}
\newcommand{\Gr}{{\rm Gr}}

\newcommand{\id}{{\rm id}}
\newcommand{\pr}{{\rm pr}}

\newcommand{\class}{{\rm class}}

\newcommand{\Satur}{{\rm Satur}}
\newcommand{\Dom}{{\rm Dom}}
\newcommand{\Hom}{{\rm Hom}}

\renewcommand{\Re}{{\rm Re}}

\newcommand{\esup}{{\rm ess.\ sup}} 
\newcommand{\einf}{{\rm ess.\ inf}} 
\newcommand{\otextbf}{{\rm \mathbf{1}}}

\newcommand{\Ac}{\cali{A}}
\newcommand{\Bc}{\cali{B}}
\newcommand{\Cc}{\cali{C}}
\newcommand{\Dc}{\cali{D}}
\newcommand{\Ec}{\cali{E}}
\newcommand{\Fc}{\cali{F}}
\newcommand{\Gc}{\cali{G}}
\newcommand{\Hc}{\cali{H}}

\newcommand{\Lc}{\cali{L}}
\renewcommand{\Mc}{\cali{M}}

\newcommand{\Uc}{\cali{U}}

\newcommand{\Sc}{\cali{S}}
\newcommand{\Tc}{\cali{T}}
\newcommand{\Nc}{\cali{N}}

\newcommand{\C}{\mathbb{C}}
\newcommand{\D}{\mathbb{D}}

\newcommand{\N}{\mathbb{N}}
\newcommand{\Z}{\mathbb{Z}}
\newcommand{\R}{\mathbb{R}}
\newcommand{\K}{\mathbb{K}}
\newcommand{\T}{\mathbb{T}}
\newcommand{\B}{\mathbb{B}}
\newcommand{\U}{\mathbb{U}}
\newcommand{\V}{\mathbb{V}}
\renewcommand{\S}{\mathbb{S}}
\renewcommand{\P}{\mathbb{P}}
\newcommand{\Q}{\mathbb{Q}}
\newcommand{\G}{\mathbb{G}}
\renewcommand{\O}{\mathbb{O}}

\renewcommand{\Dom}{{\rm Dom\ }}
\newcommand{\Range}{{\rm Range\ }}

\newcommand{\Et} {\mathbb{E}}  

\makeindex
\makenomenclature

\begin{document}
\frontmatter
\title{Oseledec   multiplicative ergodic theorem   for   laminations}

\dedicatory{To Lan-Anh, Minh-Anh and Qu\^oc-Anh,\\
for their understanding}
\author{Vi{\^e}t-Anh Nguy{\^e}n}
\address{Universit{\'e} Paris-Sud, Laboratoire de
 Math{\'e}matique, UMR 8628, B{\^a}timent 425,
F-91405 Orsay, France
}
\email{VietAnh.Nguyen@math.u-psud.fr}
\urladdr{http://www.math.u-psud.fr/$\sim$vietanh}

\date{}

\begin{abstract}
   Given  a  $n$-dimensional lamination endowed  with a Riemannian metric, we introduce the notion  of    a multiplicative  cocycle of rank $d,$
   where $n$ and $d$ are   arbitrary  positive integers. The   holonomy  cocycle of a  foliation and  its exterior  powers as well as  its  tensor powers   provide   examples   of
 multiplicative  cocycles.
Next, we define  the Lyapunov exponents of such a  cocycle   with respect  to    a   harmonic probability  measure directed by the lamination. We also
prove   an   Oseledec multiplicative  ergodic theorem   in this  context. This  theorem  implies
the  existence  of   an Oseledec decomposition almost  everywhere  which  is holonomy invariant. Moreover, in the  case  of  differentiable  cocycles
we   establish effective integral estimates  for  the Lyapunov exponents. These  results find  applications  in the
geometric and dynamical  theory  of laminations.  They are also  applicable  to (not necessarily closed) laminations  with singularities.
Interesting holonomy   properties of  a  generic  leaf
of a foliation 
are   obtained.  The main ingredients  of our   method  are the theory of   Brownian motion, the analysis  of the heat diffusions on
Riemannian manifolds, the  ergodic   theory in discrete dynamics and a geometric study of  laminations.  
\end{abstract}

\maketitle

\setcounter{page}{4}
\tableofcontents

\include{preface}

\mainmatter
\include{chapter1}

\include{chapter2}

\include{chapter3}

\include{chapter4}

\include{chapter5}

\include{chapter6}

\include{chapter7}

\include{chapter8}

\include{chapter9}

\include{Appendix}

\backmatter
\include{biblio}

\printindex
\printnomenclature[.3\textwidth]

\end{document}

%% file: preface.tex
%

\chapter*{Acknowledgement}

The author would like to thank    Nessim  Sibony    and  Tien-Cuong Dinh   
 for many interesting  discussions. He also would like  to express his  thanks to    George Marinescu, Peter Pflug and El Hassan Youssfi  for their kind  help.
 Special thanks also go to Laurent  Dang and Jo\"el Merker, both of whom offered   helpful comments.
 This work was partially prepared 
during several visits of the   author at the Max-Planck Institute  for Mathematics in Bonn and  at the University of Cologne upon  a Humboldt foundation research grant.
 He would like to express  his gratitude to these organizations for hospitality and  for  financial support.   
 He is also grateful to  University of Aix-Marseille, University of Paris 11 and  University of Paris 6 for providing him with very  good  working  conditions 
over the years.


\aufm{Vi{\^e}t-Anh Nguy{\^e}n}


%% file: chapter1.tex

\chapter{Introduction} \label{intro}


We first recall  the    definition  of  a  multiplicative cocycle\index{cocycle!multiplicative $\thicksim$}  in the context of  discrete dynamics.  Let $T$ be a measurable  transformation of a  probability measure space
$(X,\Bc,\mu).$
  Assume that   $\mu$ is  $T$-invariant (or equivalently, $T$ preserves $\mu$), that is,  $\mu(T^{-1}B)=\mu(B)$  for all $B\in \Bc$ (or equivalently, $T_*\mu=\mu$).
\begin{definition}\label{defi_cocycle_for-maps}
\rm 
 Let  $\G$ be either  $\N$ or $\Z.$ 
\nomenclature[a1]{$\N$}{semi-group of  non-negative integers}
\nomenclature[a2]{$\Z$}{ring of  integers}
In case $\G=\Z$ we  assume further that $T$ is bi-measurable invertible\footnote{
An invertible map $T$ is  said to be  bi-measurable   if both $T$ and $T^{-1}$ are measurable.}.\index{measurable!bi-$\thicksim$ map} 
A measurable function  $\mathcal{A}: \ X\times \G\to \GL(d,\R)$
 \nomenclature[a5]{$\GL(d,\K)$}{general linear group of degree $d$ over a field $\K$} is  called  a 
{\it multiplicative cocycle\index{cocycle!multiplicative $\thicksim$|(}  over $T$} or simply a {\it cocycle} if
for every $x\in X,$    $\mathcal{A}(x,0)=\id$  and  the following  {\it multiplicative law} holds   
$$   \mathcal{A}(x,k+l)=\mathcal{A}(T^l(x),k)\mathcal{A}(x,l),\qquad k,l\in \G.$$
\end{definition}
 Throughout  the Memoir,  we use  the  notation $\log^+:=\max(0,\log).$  
  \nomenclature[a8]{$\log^+ $}{$:=\max(0,\log)$}
 Moreover,  the  {\it angle}  between two subspaces  $V,W$ of $\R^d$ (resp. $\C^d$)  is, by definition,
$$
\measuredangle \big (V,W\big):=\min\left\lbrace  \arccos\langle v,w\rangle:\  v\in V,\ w \in W, \| v\|=\|w\|=1 \right\rbrace.
$$\nomenclature[a8a]{$ \measuredangle \big (V,W\big)$}{ angle between two subspaces  $V,W$ of $\R^d$ (resp. $\C^d$)}
  Here $ \langle \cdot,\cdot\rangle$ (resp.  $\|\cdot\|$) denotes the standard  Euclidean inner product (resp. Euclidean norm) of $\R^d$ or of $\C^d$.
  Now  we are  able to
 state the   classical Oseledec Multiplicative Ergodic Theorem \cite{Oseledec}\index{Oseledec!$\thicksim$ multiplicative ergodic theorem}\index{theorem!Oseledec multiplicative ergodic $\thicksim$}  (see also \cite{KatokHasselblatt,Ruelle}).
 
 \begin{theorem} \label{thm_Oseledec}
 Let $T$ be as  above and   let   $\mathcal A:\   X\times \G\to \GL(d,\R)$  be  a cocycle
   such that the real-valued functions  $x\mapsto \log^+ \|\mathcal A^{\pm 1}(x,1)\|$  are $\mu$-integrable.
 Then there exists $Y\in\Bc$ with $TY\subset Y$ and $\mu(Y)=1$ such that  the following properties   hold:
 \\(i)  There is  a  measurable function  $m:\ Y\to \N$  with $m\circ T=m.$
 \\(ii) For   each $x\in Y$  
  there   are real numbers 
$$\chi_{m(x)}(x)<\chi_{m(x)-1}(x)<\cdots
<\chi_2(x)<\chi_1(x)$$
with  $ \chi_i(Tx)= \chi_i(x)$  when $1\leq i\leq  m(x),$ and    the function $x\mapsto \chi_i(x)$ is measurable on $\{x\in Y:\ m(x)\geq i\}.$ These numbers are called 
 the {\rm  Lyapunov exponents}\index{Lyapunov!$\thicksim$ exponent}\index{exponent|see{Lyapunov $\thicksim$}} associated to the cocycle $\mathcal{A}$  at the point $x.$
\\ (iii) 
For each  $x\in Y$ there is  a decreasing sequence of   linear  subspaces
$$   \{0\}\equiv V_{m(x)+1}(x)\subset   V_{m(x)}(x)\subset \cdots\subset  V_2(x)\subset V_1(x)=\R^d,
$$
of $\R^d$ 
  such that $\mathcal{A}(x,1 ) V_i(x)= V_i(Tx)$  and that
  $x\mapsto  V_i(x)$ is   a  measurable map from $\{x\in Y:\ m(x)\geq i\}$ into the  corresponding  Grassmannian of $\R^d.$
  This  sequence of subspaces is called  the
 {\rm  Lyapunov filtration}\index{filtration!Lyapunov $\thicksim$}\index{Lyapunov!$\thicksim$ filtration} associated to the cocycle $\mathcal{A}$  at the point $x.$
 \\ (iv) For each  $x\in Y$  and  $v\in V_i(x)\setminus V_{i+1}(x),$    
$$\lim\limits_{n\to \infty} {1\over  n}  \log {\| \mathcal A(x,n)v   \|\over  \| v\|}  =\chi_i(x).    
$$
\\ (v) Suppose now that $\G=\Z$ and that $T$ is bi-measurable invertible. Then, for every $x\in Y,$
 there   exists $m(x)$  linear subspaces $H_1(x),\ldots, H_{m(x)}(x)$   of $\R^d$  such that  
$$V_j(x)=\oplus_{i=j}^{m(x)} H_i(x),
$$
 with $\mathcal{A}(x,1) H_i(x)= H_i(Tx),$  
and $x\mapsto  H_i(x)$ is   a  measurable map from $\{x\in Y:\ m(x)\geq i\}$ into the corresponding Grassmannian of $\R^d.$
Moreover,
$$\lim\limits_{n\to \pm\infty} {1\over  | n|}  \log {\| \mathcal{A}(x,n)v   \|\over  \| v\|}  =\pm \chi_i(x),    
$$
uniformly  on  $v\in H_i(x)\setminus \{0\},$  and  the following limit holds
$$
\lim\limits_{n\to \infty} {1\over n}  \log\sin {\big |\measuredangle \big (H_S(T^nx), H_{N\setminus S} (T^nx)\big ) \big |}=0,
$$  where, for any subset   $S$  of $  N:=\{1,\ldots,m(x)\},$ we define $H_S(x):=\oplus_{i\in S} H_i(x).$  
\end{theorem}
 The definition of  cocycles  and  Theorem   \ref{thm_Oseledec}\index{Oseledec!$\thicksim$ multiplicative ergodic theorem}  can also be  formulated   using the  action of $\GL(d,\C)$ on $\C^d$
instead  of the action of   $\GL(d,\R)$ on $\R^d.$
 The   above  fundamental theorem   together with Pesin's work  in \cite{Pesin}\index{Pesin}    constitute  the  nonuniform hyperbolicity  theory of maps.
This  theory    is now   one  of the major parts 
of the  general  dynamical theory  and one  of the main tools  in studying  highly sophisticated behavior  associated with ``deterministic chaos". 
Nonuniform hyperbolicity conditions  can be  expressed  in terms of the  Lyapunov exponents.\index{Lyapunov!$\thicksim$ exponent}\index{exponent|see{Lyapunov $\thicksim$}} Namely,
a  dynamical system   is {\it nonuniformly hyperbolic} if it admits an invariant measure  such that the  
 Lyapunov exponents\index{Lyapunov!$\thicksim$ exponent}\index{exponent|see{Lyapunov $\thicksim$}}  associated  to  a certain representative 
cocycle of the  system  are nonzero almost everywhere.

The  ergodic  theory  of      laminations
  is  not  so  developed  as  that of maps  or  flows, 
  because
of  at least two reasons. The first   one is that   laminations  which have  invariant  measures  are rather   scarce.
In fact,   invariant measures  for maps  should be  replaced  by   harmonic  measures for laminations.
On the other hand,   
  there is the additional problem with ``time". In 
the dynamics of laminations, the concept of linearly or time-ordered trajectories is replaced with the vague notion of
multi-dimensional futures for points, as defined by the leaves through the points. The geometry of
the leaves thus plays a fundamental role in the study of  lamination dynamics, which is a fundamentally
new aspect of the subject, in contrast to the study of diffeomorphisms, or $\Z$-actions.
This is  the  second  reason. The  reader is invited  to  consult  the surveys by  Forn\ae ss-Sibony \cite {FornaessSibony2},\index{Forn\ae ss}\index{Sibony}
by Ghys\index{Ghys} \cite{Ghys}, and by Hurder\index{Hurder}  \cite{Hurder}
for  a recent account on this  subject.  
 We mention here     
the  approach  using Brownian motion\index{Brownian motion} which  has been first  introduced  by Garnett  \cite{Garnett} in order to explore   the  dynamics  of  compact  smooth laminations endowed  with a transversally continuous Riemannian metric.
This  method  has been  further   pursued  by many authors (see, for example, Candel \cite{Candel2}, Kaimanovich \cite{Kaimanovich} etc).
\index{Candel}\index{Kaimanovich} 

The purpose  of this  Memoir is to  establish  an  Oseledec multiplicative   ergodic theorem\index{Oseledec!$\thicksim$ multiplicative ergodic theorem}    for laminations.
This  will be   a starting 
 attempt  in order to  develop   a  nonuniform hyperbolicity  theory for    laminations. 
The natural  framework of  our  study is    a given    lamination endowed  with a Riemannian metric
which    directs a  harmonic probability  measure.
Our  purpose consists of two tasks. The first  one is  to formulate  a  good  notion of (multiplicative) cocycles.
Secondly,   we  define the Lyapunov exponents\index{Lyapunov!$\thicksim$ exponent}\index{exponent|see{Lyapunov $\thicksim$}} for such cocycles  
    and prove  an Oseledec multiplicative   ergodic theorem\index{Oseledec!$\thicksim$ multiplicative ergodic theorem}  in this context. 
 The main examples of laminations   we  have  in mind  come from  two  sources.
The first one consists of  all  compact smooth  laminations.  It is easy to   endow each  such a lamination  with a transversally continuous  Riemannian  metric.
The  second  source   comprises   (possibly singular)  foliations by Riemann surfaces
in the  complex projective space $\P^k$    or
in algebraic manifolds. Such a foliation  often  admits a  canonical Riemannian metric, namely, the Poincar\'e metric.
However, this metric is   transvesally  measurable, it is  continuous only in some good cases (see Dinh-Nguyen-Sibony \cite{DinhNguyenSibony1},\index{Dinh}\index{Nguyen}\index{Sibony} 
Forn\ae ss-Sibony \cite{FornaessSibony2} etc).\index{Forn\ae ss}   Our  main examples of cocycles
are
 the  holonomy cocycles\index{cocycle!holonomy $\thicksim$}    
(or their tensor powers) of such  foliations.

 A recent   result  in our  direction is  obtained  by Candel\index{Candel} in \cite{Candel2} who  defines 
 the Lyapunov exponent\index{Lyapunov!$\thicksim$ exponent}\index{exponent|see{Lyapunov $\thicksim$}} of additive  cocycles\index{cocycle!additive $\thicksim$}  .  
To state  his   result  in the context of  multiplicative  cocycles, we  need to introduce some terminology and definition.  
     A precise  formulation will be recalled  in Chapter   \ref{section_background}  below.
       Let $(X,\Lc)$ be a       lamination endowed with a    Riemannian metric tensor $g$ on leaves. 
 Let  $\Omega:=\Omega(X,\Lc) $  be  the space consisting of  all continuous  paths  $\omega:\ [0,\infty)\to  X$ with image fully contained  in a  single   leaf. 
Consider  the  semi-group $(T^t)_{t\in\R^+}$ of shift-transformations\index{shift-transformation} 
  $T^t:\  \Omega\to\Omega$ defined for  all $t,s\in\R^+$ by 
\begin{equation}\label{eq_shift}
   T^t(\omega)(s):=\omega(s+t),\qquad  \omega\in \Omega.
   \end{equation}
   \nomenclature[c2a]{$\{T^t:\ t\in\R^+\}$}{semi-group of shift-transformations of time $t$ acting on either $X^{[0,\infty)}$ or its subspace $\Omega;$  shift-transformation of unit-time
   $T^1$ is often denoted by $T$} 
For $x\in X,$ let $\Omega_x$  be the  subspace consisting of all  paths in $\Omega$ starting from $x.$ 
We endow $\Omega_x$ with a canonical  probability measure: the {\it Wiener measure} $W_x.$
Let $\alpha$ be  a   closed  one-form  on  the leaves  of $(X,\Lc).$   
Define   a map $\mathcal A:\  \Omega\times \R^+\to  \C^*$  by
$$
\mathcal A(\omega,t):=e^{\int_{\omega[0,t]} \alpha}.
$$
Clearly,  the following  multiplicative  property holds  $ \mathcal A(\omega,s+t)=\mathcal A( T^t\omega,s)\mathcal A(\omega,t)$ for all $s,t\in\R^+.$ 
$\mathcal  A$ is  called  the {\it multiplicative  cocycle}  associated to $\alpha.$  

$\mathcal A$ can be  defined  in the following manner.
Let $L$ be  a leaf   of $(X,\Lc).$ Since  $\alpha$ is  closed  on $L,$  it is   exact when lifted to  the universal cover $\pi:\ \widetilde L\to L$ of $L,$ that is,
there is   a  complex-valued function  $f$  on $\widetilde L$  such that  $df=\pi^*\alpha.$
Then, if  $\omega$ is a path in $L,$
$$
\mathcal A(\omega,t):=e^{f(\tilde \omega(t))-f(\tilde\omega(0))}, 
$$
where $\tilde \omega$ is  any lift  of $\omega$ to $\widetilde L.$ The value 
of $\mathcal A(\omega,t)$ is  independent of  the lift $\tilde \omega$ and $f.$
It depends only on $\alpha$ and  the homotopy class of   the  curve $\omega|_{[0,t]}.$
Let $\Delta$ be the  Laplace operator associated to the metric tensor $g$ of $(X,\Lc).$
Consider  the operator  $\delta$   which sends every closed one-form  $\alpha$  to   the  function  
$$
\delta\alpha(x):=(\tilde\Delta  f)(\tilde x),
$$
where $x$ is  a point in the leaf $L,$ $\tilde x$ is a  lift of $x$  to  $\widetilde L,$ i.e,  $\tilde x\in\pi^{-1}(x)$,  $f$ is  related to $\alpha$ as above, and $\tilde \Delta$ is the lift to $\widetilde L$ of   the Laplace  operator
$\Delta$  on $L.$

The  following result of Candel  gives the asymptotic value  of multiplicative   cocycles of rank $1$ (see \cite[Section 8]{Candel2}). 
\begin{theorem}\label{thm_Candel}\index{Candel!$\thicksim$'s theorem}\index{theorem!Candel's $\thicksim$} 
 Let $(X,\Lc)$ be a  compact $\Cc^2$-smooth     lamination endowed with a  transversally continuous  Riemannian metric $g.$ 
 Let $\mu$ be  a  harmonic probability measure  directed by $(X,\Lc).$
 Let $\alpha$  be  a closed   one-form   such that  both  $\alpha$ and    $\delta \alpha$ are  bounded.
 Then  for  $\mu$-almost  every $x\in X,$ the {\rm asymtotic  value}
 $$
 \chi(x):= \lim\limits_{t\to \infty} {1\over  t} \log  \| \mathcal A(\omega,t) \|  
 $$
 exists for    $W_x$-almost every $\omega\in \Omega_x.$
  The real numbers  $\chi(x)$ is called  the {\rm  Lyapunov  exponent\index{Lyapunov!$\thicksim$ exponent}\index{exponent|see{Lyapunov $\thicksim$}}} associated  to the  cocycle $\mathcal A$ at $x.$
  Moreover,
  $$
 \int_X \chi(x) d\mu(x)   =\Re\int_X\delta\alpha(x) d\mu(x),
 $$
 where $\Re$  denotes the real part  of a complex number. 
  If,  moreover,  $\mu$ is  ergodic, then $\chi(x)$ is  constant  for  $\mu$-almost  every $x\in X.$
 \end{theorem}
An immediate application of  Candel's  theorem\index{Candel!$\thicksim$'s theorem}\index{theorem!Candel's $\thicksim$}   is the case where $(X,\Lc)$ is  a  foliation of transversal (real or complex) dimension  $1$ and
$\mathcal A$ is   its  holonomy cocycle\index{cocycle!holonomy $\thicksim$}   (see \cite{Candel,  Ghys}). Deroin \cite{Deroin}\index{Deroin} also  obtains  some similar results in the  last case.
Since  $\mathcal A$ takes its values  in  $\C^*$ which  is  naturally identified  with  $\GL(1,\C)$, 
Candel's result\index{Candel!$\thicksim$'s theorem}\index{theorem!Candel's $\thicksim$}   may be  considered as    Oseledec's theorem\index{Oseledec!$\thicksim$ multiplicative ergodic theorem}  for  compact $\Cc^2$-smooth laminations endowed with a transversally continuous Riemannian metric  in the case $d=1.$ Our  purpose may be rephrased  as      generalizing 
  Candel's theorem\index{Candel!$\thicksim$'s theorem}\index{theorem!Candel's $\thicksim$} to the  context of more general laminations   in  arbitrary  cocycle dimension $d.$ 
More  concretely, we  need to 

$\bullet$  introduce a large class of laminations for which  neither  the compactness of $X$  nor the transversal  smoothness of the associated  metric is  required, the new  class should include not only  compact smooth laminations, but also (possibly singular)  foliations by Riemann surfaces in algebraic  manifolds;

$\bullet$  introduce a
  notion  of  (multiplicative) cocycles  of arbitrary ranks
which  is  natural  and which  
  captures the  essential  features  of  Candel's definition of  cocycles\index{cocycle!multiplicative $\thicksim$}\index{Candel} of rank $1$
as  well as  the  definition of cocycles  for maps; 

$\bullet $ construct  an Oseledec decomposition\index{Oseledec!$\thicksim$ decomposition} and Lyapunov exponents  \index{Lyapunov!$\thicksim$ exponent}\index{exponent|see{Lyapunov $\thicksim$}} at  almost  every point
in the  spirit  of Theorem  \ref{thm_Oseledec}\index{Oseledec!$\thicksim$ multiplicative ergodic theorem}  and  
Theorem \ref{thm_Candel}.\index{Candel!$\thicksim$'s theorem}\index{theorem!Candel's $\thicksim$}

 To overcome  the  major  difficulty   with  time-ordered trajectories
we  follow partly the approach by  Garnett\index{Garnett} and  Candel\index{Candel} using the Wiener measure $W_x$ on $\Omega_x.$
More precisely, our idea is  that the  asymptotic behavior  of  a  cocycle $\mathcal  A$  at a point $x$ and a vector $v\in \R^d$ is
determined  by the asymptotic behavior of   $\mathcal  A(\omega,\cdot)v$ where $\omega$ is a {\it typical} path in $\Omega_x$  that is, it is an element of      a certain subset
of $\Omega_x$ of full    $W_x$-measure.  However, in order to  make this idea  work  in the  context of arbitrary rank $d,$ we  have to develop
  new techniques  based on the  so-called leafwise Lyapunov exponents \index{Lyapunov!leafwise $\thicksim$ exponent}
\index{leafwise!$\thicksim$ Lyapunov exponent} and the  
Lyapunov forward\index{filtration!Lyapunov forward $\thicksim$} and  backward\index{filtration!Lyapunov backward $\thicksim$} filtrations.
  These techniques  are partly  inspired  by  Ruelle's work\index{Ruelle} in \cite{Ruelle}.
  Another  crucial  ingredient is   the construction of  weakly harmonic measures
which  maximize  (resp. minimize) some  Lyapunov exponents   functionals\index{Lyapunov!$\thicksim$ exponent functional}.
  The  next   important tool  is  a  procedure of splitting  invariant subbundles. In fact,
we are inspired by  the  methods of Ledrappier \cite{Ledrappier}\index{Ledrappier}, Walters  \cite{Walters}\index{Walters}  in  discrete  dynamics.
We  improve   the  random ergodic  theorems  of  Kakutani  \index{Kakutani!random ergodic theorem of $\thicksim$}\index{theorem!random ergodic $\thicksim$ of Kakutani} and adapt it to the context of laminations
in order to  study  totally  invariant sets. This  is  a key tool to  explore the holonomy  of the cocycles.
  We also   establish    new  measure, harmonic  measure and ergodic   theories  on the  sample-path space $\Omega.$    Since  the  description  of  our method is  rather involved, we postpone  it until  the  two next  chapters.

Let us   review  shortly the main   results of this  work.  A full  development and explanation  will be given  in Chapter  \ref{section_background}  
and Chapter \ref{section_main_results} below.
 A cocycle  of rank $d$ on  a  lamination $(X,\Lc)$ is   a map
$\mathcal A$ defined on    $\Omega\times\G$ ($\Omega:=\Omega(X,\Lc)$ and $\G\in\{\N,\R^+\}$)  with  matrix-valued  $\GL(d,\K)$ ($\K
\in\{\R,\C\}$)  satisfying identity, homotopy, multiplicative  and measurable  laws. Now  we are in the position to state  the Main Theorem  in an incomplete and informal  formulation:


\begin{theorem} \label{th_main_0}
 Let $(X,\Lc,g)$ be   a       lamination satisfying some  reasonable  standing hypotheses.  Let
$\mathcal A:\ \Omega\times\G\to \GL(d,\K)$  be  a  cocycle\index{cocycle!multiplicative $\thicksim$|)}.
  Let $\mu$ be a   harmonic probability measure.   Assume 
that  $\mathcal A$ satisfies some  integrability condition with respect to $\mu.$  
Then  there exists  a   leafwise  saturated\index{leafwise!$\thicksim$ saturated set}  Borel  set $Y\subset X$  with $\mu(Y)=1$    such that
the following properties hold:
\\(i)  There is  a  measurable function  $m:\ Y\to \N$   which is   leafwise  constant.\index{leafwise!$\thicksim$ constant function}
 \\(ii)  For   each $x\in Y$  
 there   exists a  decomposition of $\K^d$  as  a direct sum of $\K$-linear subspaces 
$$\K^d=\oplus_{i=1}^{m(x)} H_i(x),
$$
 such that $\mathcal{A}(\omega, t) H_i(x)= H_i(\omega(t))$ for all $\omega\in  \Omega_x$ and $t\in \G.$     
Moreover,  the map $x\mapsto  H_i(x)$ is   a  measurable map from $\{x\in Y:\ m(x)\geq i\}$ into the Grassmannian of $\K^d.$
Moreover, there   are real numbers 
$$\chi_{m(x)}(x)<\chi_{m(x)-1}(x)<\cdots
<\chi_2(x)<\chi_1(x)$$
 such that    the function $x\mapsto \chi_i(x)$ is measurable  and  leafwise constant on $\{x\in Y:\ m(x)\geq i\},$
  and 
\begin{equation*}
\lim\limits_{t\to \infty, t\in \G} {1\over  t}  \log {\| \mathcal{A}(\omega,t)v   \|\over  \| v\|}  =\chi_i(x),    
\end{equation*}
uniformly  on  $v\in H_i(x)\setminus \{0\},$ for  $W_x$-almost every  $\omega\in\Omega_x.$
The   numbers  $\chi_{m(x)}(x)<\chi_{m(x)-1}(x)<\cdots
<\chi_2(x)<\chi_1(x)$ are called  the {\rm  Lyapunov exponents}\index{Lyapunov!$\thicksim$ exponent}\index{exponent|see{Lyapunov $\thicksim$}} associated to the cocycle $\mathcal{A}$ at the point $x.$\\
(iii) For  $S\subset  N:=\{1,\ldots,m(x)\}$ let $H_S(x):=\oplus_{i\in S} H_i(x).$ Then
\begin{equation*}
\lim\limits_{t\to \infty,\ t\in \G} {1\over t}  \log\sin {\big |\measuredangle \big (H_S(\omega(t)), H_{N\setminus S} (\omega(t))\big ) \big |}=0
\end{equation*}
for  $W_x$-almost every  $\omega\in\Omega_x.$
\\
(iv) If, moreover,   $\mu$ is    ergodic, then the  functions  $Y\ni x\mapsto m(x)$ as well as $Y\ni x\mapsto \dim H_i(x)$ and  $Y\ni x\mapsto \chi_i(x)$ are all constant. In this  case $\chi_{m}<\chi_{m-1}<\cdots
<\chi_2<\chi_1$ are called  the {\rm  Lyapunov exponents}\index{Lyapunov!$\thicksim$ exponent}\index{exponent|see{Lyapunov $\thicksim$}} associated to the cocycle $\mathcal{A}.$
\end{theorem}
 
 It is  worthy  noting    that the  decomposition $\K^d=\oplus_{i=1}^{m(x)} H_i(x)
$ in (ii) depends  only on $x,$ in particular, it does  not  depend on paths $\omega\in\Omega_x.$
We will see later that   Theorem \ref{th_main_0} (i)--(iii)
 is the abridged version of     Theorem \ref{th_main_1}, whereas  Theorem \ref{th_main_0} (iv) is  the abridged version of Corollary \ref{cor1_th_main_1} below. Moreover,  Theorem \ref{th_main_0}  seems to be the  right  counterpart  
 of Theorem \ref{thm_Oseledec}\index{Oseledec!$\thicksim$ multiplicative ergodic theorem} in the  context of laminations.  Further   remarks as well as applications of the Main Theorem  will be given  after  Theorem \ref{th_main_1} below. 

 We will also see  later that  Theorem \ref{th_main_2} below  generalizes
 Theorem \ref{thm_Candel}\index{Candel!$\thicksim$'s theorem}\index{theorem!Candel's $\thicksim$} to the context of cocycles  of arbitrary  ranks. Since  the framework of the former theorem
 requires  a good deal of preparations, we prefer to  state  it in the full  form and to discuss  its applications
in  Chapter \ref{section_main_results}.  Moreover, Theorem \ref{thm_Ledrappier} below   gives   a characterization of 
Lyapunov spectrum\index{Lyapunov!$\thicksim$ spectrum}\index{spectrum|see{Lyapunov $\thicksim$}}
 in the  spirit of   Ledrappier's work in \cite{Ledrappier}.\index{Ledrappier} All our  results   demonstrate that there is  a strong  analogue between
 the dynamical theory of  maps  and that of laminations.

 The  Memoir  is  organized  as  follows. In Chapter  \ref{section_background}   we  develop  the background
 for our study. In particular, we  introduce  a new  $\sigma$-algebra $\Ac$  on $\Omega$  which 
is  one of the main objects of our study.
   Chapter \ref{section_main_results} is  started 
with  the  notion of multiplicative  cocycles.  
  Next, we state  the   main results and their corollaries.  This  chapter is  ended with a discussion on the perspectives
of this  work and   an outline of 
  our method.    The  measure theory  as well as the harmonic  measure theory and the ergodic  theory  for the  measure  space  $(\Omega,\Ac)$
   will be developed in  Chapter  \ref{section_measurability} and   Appendices. 

Our techniques  as  well as  the proofs of the main theorems and their corollaries are  presented  in Chapters  
 \ref{section_leaf}--\ref{section_Main_Theorems}.

%% file: chapter2.tex

\chapter{Background} \label{section_background} 

 In this chapter, we  develop   the   background   for  our study.
We first 
 introduce the
   Riemannian  laminations  which are the main objects of our  study.     Some   basic properties of  these objects and related notions  such as positive  harmonic  measures and covering laminations  are also presented. Next, we    recall the heat equations 
on leaves  and review the
theory of   Brownian   motion\index{Brownian motion} in the  context of Riemannian laminations.
This preparation allows  us to develop  a new measure theory for some  sample-path spaces.
A  comprehensive and  modern exposition on the  theory of laminations   could be  found  in the  two volumes by Candel-Conlon\index{Candel}\index{Conlon}
\cite{CandelConlon1,CandelConlon2} or  in Walczak's book \cite{Walczak}.\index{Walczak}


 \section{(Riemannian) laminations and  Laplacians}
 \label{subsection_Riemannian_laminations_and_Laplacians}
Let $X$ be a separable locally compact space. Consider an atlas $\Lc$ of $X$ with charts 
$$\Phi_i:\U_i\rightarrow \B_i\times \T_i,$$
where $\T_i$ is a locally compact metric space, $\B_i$ is a domain in $\R^n$
and $\U_i$  is an open  subset of $X.$\index{space!metric $\thicksim$}\index{space!separable $\thicksim$}\index{space!locally compact $\thicksim$}

\begin{definition}\label{defi_lamination}\rm
We say that $(X,\Lc)$ is a
{\it  lamination\index{lamination!lamination, continuous $\thicksim$}  (of dimension $n$)}, or equivalently, a  {\it continuous lamination (of dimension $n$)}   if
all 
$\Phi_i$ are homeomorphism  and
 all the 
changes of coordinates $\Phi_i\circ\Phi_j^{-1}$ are of the form
$$(x,t)\mapsto (x',t'), \quad x'=\Psi(x,t),\quad t'=\Lambda(t)$$
where $\Psi,\Lambda$ are continuous functions.
Moreover, we say  that  $(X,\Lc)$ is  {\it  $\Cc^k$-smooth} for some $k\in\N\cup\{\infty\}$ if
   $\Psi$ is $\Cc^k$-smooth   with
respect to $x$ and its partial derivatives of any order $\leq k$ with respect to $x$ are jointly continuous
with respect  to $(x,t).$
\nomenclature[b1]{$(X,\Lc)$}{a (continuous) lamination or a measurable lamination} 
 \end{definition}
  In Chapter \ref{section_Lyapunov_filtration} and Chapter \ref{section_Main_Theorems} below,  the  following weak  form of laminations  is needed.
  \begin{definition}
  \label{defi_measurable_lamination}\rm 
  We say that  $(X,\Lc)$ is a  {\it measurable lamination\index{lamination!measurable $\thicksim$} of dimension $n$} if all $\Phi_i:\U_i\rightarrow \B_i\times \T_i$ are bijective and 
Borel bi-measurable\index{measurable!bi-$\thicksim$ map}
and if   all the 
changes of coordinates $\Phi_i\circ\Phi_j^{-1}$ are of the form
$$(x,t)\mapsto (x',t'), \quad x'=\Psi(x,t),\quad t'=\Lambda(t)$$
where $\Psi,\Lambda$ are Borel measurable functions\index{Borel!$\thicksim$ measurable function (or map)}\index{function!Borel measurable $\thicksim$} and $\Phi_i\circ\Phi_j^{-1}$ is  homeomorphic on the  variable $x$  when $t$ is  fixed. Moreover, we say  that  $(X,\Lc)$ is  {\it  $\Cc^k$-smooth} for some $k\in\N\cup\{\infty\}$ if
   $\Psi$ is $\Cc^k$-smooth   with
respect to $x.$ 
  \end{definition}

The open set $\U_i$ is called a {\it flow
  box}\index{flow box} and the manifold $\Phi_i^{-1}\{t=c\}$ in $\U_i$ with $c\in\T_i$ is a {\it
  plaque}\index{flow box!plaque of a $\thicksim$}\index{plaque!$\thicksim$ of a flow box}. The property of the above coordinate changes insures that
the plaques in different flow boxes are compatible in the intersection of
the boxes.
A {\it leaf}\index{leaf} $L$ is a minimal connected subset of $X$ such
that if $L$ intersects a plaque, it contains the plaque. So, a leaf $L$
is a connected real manifold of dimension $n$ immersed in $X$ which is a
union of plaques.  
For   a point $x\in X$  let $L_x$ denote the leaf  passing through $x.$
\nomenclature[b9]{$L_x$ (resp. $\widetilde L_{\tilde x}$)}{leaf passing through a given point $x$ (resp. $\tilde x$) in a lamination $(X,\Lc)$ (resp. in its covering lamination  
  $(\widetilde X,\widetilde \Lc)$)}
We will only consider {\it oriented laminations}, i.e. the case where the
$\Phi_i$ preserve the canonical orientation on $\R^n$. So, the leaves
of $X$ inherit the orientation given by the one of $\R^n$.
A {\it transversal}\index{flow box!transversal of a $\thicksim$}\index{transversal!$\thicksim$ of a flow box} in a flow box is a closed set of the box which intersects every
plaque in one point. In particular, $\Phi_i^{-1}(\{x\}\times \T_i)$ is a
transversal in $\U_i$ for any $x\in \B_i$.  In order to
simplify the notation, we often identify $\T_i$
with $\Phi_i^{-1}(\{x\}\times \T_i)$ for some $x\in \B_i$ or even
identify $\U_i$ with $\B_i\times\T_i$ via the map $\Phi_i$.

When a lamination $(X,\Lc)$ satisfies that $X$ is a Riemannian manifold and that the leaves of $\Lc$ are manifolds immersed in $X$, we say that $(X,\Lc)$ is a {\it foliation}\index{foliation}. Moreover,  $(X,\Lc)$ is called a  {\it transversally  $\Cc^k$-smooth foliation} if
there is an atlas $\Lc$ of $X$ with charts 
$$\Phi_i:\U_i\rightarrow \B_i\times \T_i,$$
with $\T_i$  an  open set of some $\R^d$ (or $\C^d$) such that 
 each above map  $\Psi$ is a diffeomorphism  of class $\Cc^k.$ 
   For  a  transversally  $\Cc^1$-smooth foliation $(X,\Lc),$ a {\it transversal  section}\index{foliation!transversal section of a $\thicksim$}\index{transversal!$\thicksim$ section}\index{section!transversal $\thicksim$} is  a  submanifold $S$ of $X$ such that
for every  flow box $\U$ and for every plaque $P$ of $\U,$ either $S\cap \U$ does not intersect $P,$ or
 $S\cap\U$ is  transverse to  $P$ at  their unique  intersection.
 
 We introduce  the class of $\Cc^k$-smooth (real or complex-valued)  functions defined on
 a $\Cc^k$-smooth lamination  $(X,\Lc).$   
 Let $Z$ be a separable, locally compact metrizable space, and let $U$ be an open subset of the
product $\R^n\times  Z.$ A function $f :\ U\to\R$ is said to be {\it $\Cc^k$-smooth}  at a
point $(x_0,z_0)$  if
there is a neighborhood of this point of the form $D \times Z_0$ such that
the function $z\mapsto f(\cdot,z)\in \Cc^k(D)$ is continuous on $Z,$ where $\Cc^k(D)$ has the topology of
uniform convergence on compact  subsets of all derivatives of order $\leq k.$ The function $f$ is said to be  
 $\Cc^k$-smooth  in $U$ if it is $\Cc^k$-smooth  at every point of $U.$
Generalizing this  definition, given  an open subset $U$ of $X,$
 a function  $f:\  U\to  \R$ (or $\C$) is   said  to be  {\it $\Cc^k$-smooth} if 
for  any  $\Cc^k$  atlas with charts 
$$\Phi_i:\U_i\rightarrow \B_i\times \T_i,$$
the  functions $ f_i:= f\circ \Phi_i^{-1}:\  \Phi_i(\U_i\cap U)\to   \R$ (or $\C$) are   $\Cc^k$-smooth  
in the  previous  sense. Denote by $\Cc^k(U)$ 
the space  of $\Cc^k$-smooth (real or complex-valued)  functions defined on $U.$ Moreover,
 $\Cc^k_0(U)$ denotes  those  elements of $\Cc^k(U)$ which are  compactly  supported in $U.$

  \begin{definition}
  Let $(X,\Lc)$ be a  measurable  lamination.
 A tensor $g$ on
leaves of $\Lc$  is  said to be a   {\it Riemannian metric}
\index{metric!Riemannian $\thicksim$}   on $(X,\Lc)$ if, 
    using  the charts $\Phi_i:\ \U_i\to\B_i\times \T_i,$ $g$ can be expressed as  a  collection of    tensors  $(\omega_i)$  with the following properties:
 
$\bullet$ (Metric condition)
\index{condition!metric $\thicksim$} $\omega_i$  is defined on $\B_i\times \T_i$ and  has
the  following    expression
$$
\omega_i:= \sum_{p,q=1}^n   g^i_{pq}(x,t) dx_p\otimes dx_q,\qquad x=(x_1,\ldots,x_n)\in\B_i,\ t\in \T_i
$$
where   the  matrix  of  functions   $(g^i_{pq})$ is  symmetric and positive  definite,  and  the  functions  $(g^i_{jk})$   are $\Cc^2$-smooth with respect to  $x;$ 

$\bullet$  (Compatibility condition)
 \index{condition!compatibility $\thicksim$}$(\Phi_i\circ \Phi_j^{-1})_* \omega_j=\omega_i.$
We often    write  $\omega_i:=(\Phi_i)_*g.$

Moreover, $g$ is  said  to be {\it  transversally  measurable} 
 \index{metric!transversally measurable $\thicksim$} if 
  
$\bullet$  (Measurable 
  condition) 
\index{condition!measurable $\thicksim$}
the  functions  $(g^i_{jk})$     are Borel  measurable  
with respect to  $(x,t).$ 

When $(X,\Lc)$ is a lamination,  we say that a metric $g$ on $(X,\Lc)$ is    {\it  transversally continuous}
\index{metric!transversally continuous $\thicksim$} if 

$\bullet$
(Continuity  condition) 
\index{condition!continuity $\thicksim$}
the  functions  $(g^i_{jk})$     are    continuous 
with respect to  $(x,t).$ 
\end{definition}

 Roughly speaking, transversal measurability (resp.  transversal continuity)  means that  the  metric  
  depends in a measurable (resp.  continuous) way  on transversals.
    Since $X$ is  paracompact, we can use a partition of
unity in order to construct  a  Riemannian metric tensor $g$ on any $\Cc^2$-smooth lamination  $(X,\Lc)$ such  that $g$   is transversally  continuous.  

Now  we  come  to one of the main concepts of this  chapter.

 \begin{definition}
We say that  a  triplet   $(X,\Lc,g)$  consisting of a $\Cc^2$-smooth measurable lamination  $(X,\Lc)$   and
a tensor $g$  on leaves of $\Lc$  is  a {\it  Riemannian measurable lamination,}\index{lamination!Riemannian measurable $\thicksim$} 
if   $g$   is a {\it  Riemannian metric} on $(X,\Lc)$ which is  {\it  transversally measurable.}
If, moreover, $(X,\Lc)$ is  a  lamination, then  we say that $(X,\Lc,g)$ is a {\it Riemannian lamination}, or equivalently,
a {\it Riemannian continuous lamination}\index{lamination!Riemannian continuous $\thicksim,$ Riemannian $\thicksim$}.
\nomenclature[b2]{$(X,\Lc,g)$}{a Riemannian (continuous) lamination  or a Riemannian measurable lamination} 
\end{definition}

Let  $(X,\Lc, g)$ be  a    Riemannian measurable lamination. Then $g$ induces   a  metric  tensor $g|_L$  on  each leaf $L$ of $(X,\Lc),$  and thus  a corresponding  {\it leafwise Laplacian}
 $\Delta_L.$\index{leafwise!$\thicksim$ Laplacian}
 If $u$ is  a  function   on $X$  that   is of class  $\Cc^2$ along  each leaf, then  $\Delta u$   is, by definition,  the aggregate of the  leafwise Laplacians  
$\Delta_L u.$ We say that $\Delta$ is the {\it  Laplacian}, or equivalently, the  {\it Laplace operator}.\index{Laplacian, Laplace operator}
 \nomenclature[b6aa]{$\Delta$}{Laplace operator on a Riemannian  measurable lamination}
 \section{Covering laminations}
 \label{subsection_Covering_laminations}
The  {\it covering lamination}\index{lamination!covering $\thicksim$}\index{covering!$\thicksim$ lamination} $(\widetilde X,\widetilde\Lc)$ 
\nomenclature[b3]{$(\widetilde X,\widetilde \Lc)$}{covering lamination of a measurable lamination $(X,\Lc)$} 
of  a lamination $(X,\Lc)$  is, in some  sense, its  universal cover.
We give  here its construction.
For  every leaf $L$ of $(X,\Lc)$ and every point  $x\in L,$  let $\pi_1(L,x)$ denotes  as usual the first  fundamental  group of
 all continuous closed paths $\gamma:\  [0,1]\to L$ based  at $x,$ i.e., $\gamma(0)=\gamma(1)=x.$  Let   $[\gamma]\in \pi_1(L,x)$ be  the   class of   a    closed path $\gamma$ based  at $x.$
 Then the pair  $(x,[\gamma])$ represents      a point in   $(\widetilde X,\widetilde\Lc).$
  Thus
the set of points of $ \widetilde X$ is well-defined. The  leaf  $\widetilde L$ passing through a given point $(x,[\gamma])\in\widetilde X,$ is  by definition, the set
$$
 \widetilde L:=\left\lbrace (y, [\delta]):\ y\in L_x,\  [\delta]\in \pi_1(L,y)  \right\rbrace,
$$
which is the {\it  universal cover} of the leaf $L_x.$\index{leaf!universal cover of a $\thicksim$}
 We  put the following  topological structure on  $\widetilde X$ by describing
  a  basis of open  sets. Such a  
   basis  consists of all  sets $\Nc(U,\alpha),$
  $U$ being an open subset of  $X$ and $\alpha$  being a {\it homotopy  on $U.$}\index{homotopy!$\thicksim$ as a function}
  Here   a {\it homotopy  $\alpha$ on $U$}  is  a  continuous  function $\alpha:\ U\times [0,1]\to X$ such that  $\alpha_x:=\alpha(x,\cdot)$ is a  closed  path in $L_x$ based at  $x$
 for all $x\in U$ (that is, $\alpha_x[0,1]\subset L_x$ and  $\alpha(x,0)=\alpha(x,1)=x,$ $\forall x\in U$), and  
  $$
  \Nc(U,\alpha):=  \left\lbrace (x,[\alpha_x]):\ x\in U \right\rbrace.
  $$
   The projection $\pi : \widetilde X\to  X$ is defined by $\pi(x,[\gamma]) :=x. $  It is  clear that $\pi$ is  locally homeomorphic. 
 Let 
$\Phi_i:\U_i\rightarrow \B_i\times \T_i$
be a chart  of the  atlas $\Lc$ of the lamination $(X,\Lc).$ By shrinking $\U_i$ if necessary, we may assume   without
loss of generality that there is  a homotopy $\alpha_i$ on $\U_i.$ Consider the following chart on $\widetilde X:$
$$
 \Phi_{i,\alpha}:\  \Nc(\U_i,\alpha)\to\B_i\times \T_i
 $$
 given by   $  \Phi_{i,\alpha}(\tilde x) =\Phi_i(\pi(\tilde x)),             $ $\tilde x\in   \Nc(\U_i,\alpha).$
  Using   these  charts, an atlas $\widetilde \Lc$ of $\widetilde X$ is  well-defined. 
Since  $\pi:\ (\widetilde X,\widetilde\Lc)\to (X,\Lc) $ maps leaves to leaves,   $(\widetilde X,\widetilde\Lc)$  inherits the  differentiable structure on leaves and   the   lamination structure from $(X,\Lc).$
If  $(X,\Lc,g)$ is   a  Riemannian lamination, then   we  equip $(\widetilde X,\widetilde\Lc)$ with the  metric tensor $\pi^*g$
so that   $(\widetilde X,\widetilde\Lc,\pi^*g)$ is also a  Riemannian lamination.
\nomenclature[b4]{$(\widetilde X,\widetilde \Lc,\pi^*g)$}{Riemannian covering lamination of a Riemannian measurable lamination $(X,\Lc,g),$ 
where $\pi:\ (\widetilde X,\widetilde\Lc)\to (X,\Lc)$  is the   covering lamination projection of $(X,\Lc)$}
We  call $\pi:\ (\widetilde X,\widetilde\Lc)\to (X,\Lc)$  the {\it  covering lamination projection} of $(X,\Lc).$

Now   we  discuss the  general case  of  measurable laminations.
\begin{definition}\label{defi_covering_measurable_lamination}
Let  $(X,\Lc)$ be  a measurable   lamination.
We  say that  a  mesurable  lamination  $(\widetilde X,\widetilde\Lc)$ is  a {\it covering (measurable) lamination}\index{lamination!covering (measurable) $\thicksim$}
\index{covering!$\thicksim$ (measurable) lamination} of $(X,\Lc)$ 
if there is a  surjective Borel measurable  projection $\pi : \widetilde X\to  X$ which maps leaves to leaves  and  which is  locally homeomorphic on each  leaf and  whose   each fiber  $\pi^{-1}(x),$ $x\in X,$  is  at most  countable.  
We  also call $\pi:\ (\widetilde X,\widetilde\Lc)\to (X,\Lc)$ the  associated  {\it  covering lamination projection}\index{covering!$\thicksim$ lamination projection}. 
\end{definition}
In contrast to  the  class of (continuous) laminations,   a  covering measurable  lamination  of a   measurable  lamination does not exist in general, and  if  it exists  it may  not be  unique.

 \section[Heat kernels and (weakly) harmonic measures]{Heat kernels, (weakly) harmonic measures and Standing Hypotheses}

For  every point  $x$ in an arbitrary leaf $L$ of a Riemannian measurable lamination $(X,\Lc,g),$
 consider  the   {\it heat  equation}
\index{heat!$\thicksim$ equation} on $L$
 $$
 {\partial p(x,y,t)\over \partial t}=\Delta_y p(x,y,t),\qquad  \lim_{t\to 0} p(x,y,t)=\delta_x(y),\qquad   y\in L,\ t\in \R_+.
 $$
Here  

$\bullet$ $\Delta_y=\Delta|_L$ denotes the  leafwise Laplace operator (or equivalently, the  leafwise Laplacian) 
\index{Laplacian, Laplace operator!leafwise $\thicksim$} on $L$ induced by the metric  tensor $g|_L;$

$\bullet$ 
 $\delta_x$  denotes  the  Dirac mass at $x,$  and  the  limit  is  taken  in the  sense of distribution, that is,
$$
 \lim_{t\to 0+
}\int_L p(x,y,t) \phi(y) d\Vol_L (y)=\phi(x)
$$
 \nomenclature[a92]{$\Vol_L$}{Lebesgue measure on a  leaf (a Riemannian manifold) $L$} 
for  every  smooth function  $\phi$   compactly supported in $L,$  where $\Vol_L$ denotes the (Lebesgue) volume form  on $L$ induced by the metric tensor $g|_L.$\index{Lebesgue!$\thicksim$ volume form}
\index{Lebesgue!$\thicksim$ measure|see{$\thicksim$ volume form}}

The smallest positive solution of the  above  equation, denoted  by $p(x,y,t),$ is  called  {\it the heat kernel}\index{heat!$\thicksim$ kernel}. Such    a  solution   exists   when $L$ is
complete and   of bounded  geometry  (see, for example,  \cite{Chavel,CandelConlon2}).  
 \nomenclature[b5]{$p(x,y,t)$}{heat kernel}  
\begin{definition}
We say that a Riemannian measurable  lamination {\it  $(X,\Lc,g)$ satisfies   Hypothesis (H1)}\index{Hypothesis!$\thicksim$ (H1)}     
 if   the leaves of $\Lc$ are all complete and  of uniformly bounded  geometry with respect to $g.$
  The  assumption of uniformly bounded  geometry
\index{geometry!bounded $\thicksim$}  means that  there are  real numbers  $r>0,$ and  $a,b$  such that 
for every point $x\in X,$ the injectivity radius 
 of the leaf $L_x$ at $x$  is $\geq r$ and all  sectional  curvatures belong to the interval $[a,b].$
 \end{definition}
 Assuming this  hypothesis  then the  heat kernel $p(x,y,t)$ exists  on  all leaves.
 The  heat kernel   gives  rise to   a one-parameter  family $\{D_t:\ t\geq 0\}$ of  diffusion  operators\index{diffusion!$\thicksim$ operator}
 \nomenclature[b6]{$\{D_t:\ t\in\R^+\}$}{one-parameter family of diffusion operators}    defined on bounded functions  on $X:$
 \begin{equation}\label{eq_diffusions}
 D_tf(x):=\int_{L_x} p(x,y,t) f(y) d\Vol_{L_x} (y),\qquad x\in X.
 \end{equation}
 We record here  the  semi-group property\index{diffusion!semi-group of $\thicksim$s}  of this  family: 
\begin{equation}\label{eq_semi_group}
D_0=\id\quad \textrm{ and}\quad D_{t+s}=D_t\circ D_s \quad \textrm{for}\ t,s\geq 0.
\end{equation}

We note the following  relation between the diffusion in a   complete  Riemannian manifold $L$
\index{complete!$\thicksim$ Riemannian manifold} and in its universal cover $\widetilde L.$
In this  Memoir,  we often identify  the fundamental group $\pi_1(L)$\index{leaf!fundamental group of a $\thicksim$} with the group of  
deck-transformations\index{deck-transformation!group of $\thicksim$s of a leaf} 
 \nomenclature[b7]{$ \pi_1(L)$}{fundamental group of a leaf (a Riemannian  manifold) $L$}
   of the {\it universal covering projection}\index{leaf!universal covering projection of a $\thicksim$}\index{leaf!group of deck-transformations of a $\thicksim$}
$\pi:\  \widetilde L\to L.$ 
 \nomenclature[b8]{$ \widetilde L$}{universal cover of a leaf (a Riemannian  manifold) $L$}
It is  well-known that  $\pi_1(L)$ is  at most  countable. Recall  that
 $\widetilde L$ is  endowed  with  the  metric  $\pi^*(g|_L).$ 
 The  Laplace operator $\Delta$ on the Riemannian manifold $(L,g|_L)$ lifts
to $\tilde \Delta$ on the Riemannian manifold $(\widetilde L,\pi^*(g|_L)),$  which commutes with $\pi.$ 
\nomenclature[b6a]{$\Delta_L,$  $\Delta$ (resp. $\widetilde\Delta_L,$ 
 $\widetilde\Delta$)}{ Laplace operator on a Riemannian manifold $(L,g)$ (resp. on its universal cover
$ (\widetilde L, \pi^*g),$ where $\pi:\ \widetilde L\to L$ is the universal covering projection)}
  To the operator $\widetilde\Delta$ is  associated  the heat kernel 
$\tilde p(\tilde x, \tilde y,t),$ which is  related   to $p(x,y,t)$  on $L$ by
\begin{equation}\label{eq1_heat_kernel}
p(x,y,t)=\sum_{\gamma \in \pi_1(L)}\tilde p(\tilde x, \gamma \tilde y,t) 
\end{equation}
where $\tilde x,$ $\tilde  y$ are  lifts of $x$ and $y$ respectively.
Moreover, we infer from (\ref{eq1_heat_kernel}) that the  heat kernel is  invariant  under deck-transformations,
\index{deck-transformation!invariant under $\thicksim$s}
\index{invariant!$\thicksim$ under deck-transformations} that is,
\begin{equation}\label{eq2_heat_kernel}
\tilde p(\gamma\tilde x, \gamma \tilde y,t)= \tilde p(\tilde x,  \tilde y,t)
\end{equation}
for  all $\gamma \in \pi_1(L)$ and $\tilde x,\tilde y\in \tilde L$ and $t\geq 0.$
As an immediate consequence of identity (\ref{eq1_heat_kernel}), we obtain the following relation  between $D_t$ and the heat
diffusions $\widetilde D_t$ on $\widetilde L.$ 
\begin{proposition}\label{prop_heat_difusions_between_L_and_its_universal_covering}
For every  bounded  measurable function $f$ defined on $L$ and  every $t\in \R^+,$
$$ \widetilde D_t(f\circ \pi)=   (D_tf ) \circ \pi\qquad\text{on}\  \widetilde L. 
$$ 
\end{proposition}
The following  definitions will be used  throughout the  article.  
\begin{definition}\label{defi_Standing_Hypotheses_harmonicity}
 Let $(X,\Lc,g)$ be a Riemannian measurable lamination. Let  $\Delta$ be its Laplacian\index{Laplacian, Laplace operator}, that is,   the   aggregate of the leafwise Laplacians
 $\{\Delta_x\}_{x\in X}.$  

We say  that   {\it  $(X,\Lc,g)$  satisfies  Hypothesis  (H2)}\index{Hypothesis!$\thicksim$ (H2)}   if $(X,\Lc,g)$ is a Riemannian lamination and
if  $\Delta u$  is    bounded for every  $u\in \Cc^2_0(X).$ 

We say  that   {\it  $(X,\Lc,g)$  satisfies the Standing Hypotheses } 
if this triplet  satisfies both Hypotheses  (H1) and (H2).

 When  $(X,\Lc,g)$  satisfies   Hypothesis (H1),    a positive finite Borel measure $\mu$  on $X$
  is   called {\it very weakly   harmonic}\index{harmonic measure!very weakly $\thicksim$}\index{measure!very weakly harmonic $\thicksim$|see{harmonic measure}} 
(resp.  {\it weakly harmonic\index{harmonic measure!weakly $\thicksim$}}\index{measure!weakly harmonic $\thicksim$|see{harmonic measure}}) if the following    property (i) is satisfied
for $t=1$ (resp.  for all $t\in\R^+$):
\\
(i)
$\int_X  D_t f d\mu=\int_X fd\mu
$   
for all  bounded  measurable functions $f$ defined on $X.$ 

When  $(X,\Lc,g)$  satisfies the Standing Hypotheses, a measure  $\mu$  is   called {\it harmonic}\index{harmonic measure}\index{measure!harmonic $\thicksim$|see{harmonic measure}}  if  it is  weakly harmonic and
if it satisfies the following  additional   property:
 \\
 (ii)
$
\int_X  \Delta f d\mu=0
$   
for all  functions $f\in \Cc^2_0(X).$
\end{definition}
\begin{remark}\rm
 Hypothesis (H2)  guarantees   that  the function $\Delta f$ is bounded, hence  $\mu$-integrable
for  every   $f\in \Cc^2_0(X).$  Consequently, the  integral in property (ii) makes sense.


In the literature  the  term {\it  diffusion-invariant}    is  often  used  instead of {\it weakly harmonic.} 

Originally, Garnett's definition\index{Garnett}  of harmonic  measures\index{harmonic measure!$\thicksim$ in the sense of Garnett}  (see \cite{Garnett}) consists only of 
property (ii). However,  she  only deals  with the  following context:   
   $(X,\Lc)$ is  a compact $\Cc^2$-smooth lamination  endowed with  a transversally continuous  
Riemannian metric $g.$ 
In this  context it is  known (see, for instance,  \cite{CandelConlon2, Candel2}) that (i) being valid for all $t\in\R^+$
 is equivalent to (ii), that is,
weakly harmonic measures,   harmonic measures in  our  sense  and  harmonic  measures in the sense of  Garnett  are all  equivalent.
On the  other hand, the same  equivalence holds when $(X,\Lc)$ is the regular part of a compact foliation by Riemann surfaces
with  linearizable  singularities (see  \cite{DinhNguyenSibony1}).

It is  worthy  noting here that,  for  every  $(X,\Lc,g)$ satisfying Hypothesis (H1), we  have that 
\begin{equation}\label{eq3_heat_kernel}
\| D_tf\|_{L^\infty(X)}\leq \| f\|_{L^\infty(X)} 
\end{equation} for every  bounded function $f$ and every $t\in \R^+,$ because  $\int_{L_x} p(x,y,t)d\Vol_{L_x} (y)=1$ (see  Chavel\index{Chavel} \cite{Chavel}). This, combined  with (i) and  an interpolation argument,
implies that for every  $(X,\Lc,g)$ satisfying the Standing Hypotheses and  every weakly harmonic measure $\mu,$
we  get that 
\begin{equation}\label{eq4_heat_kernel}
 \| D_tf\|_{L^q(X,\mu)}\leq \| f\|_{L^q(X,\mu)} 
\end{equation}
 for every $1\leq q\leq\infty, $ $t\in \R^+$  and every  function $f\in L^q(X,\mu).$  
 In other  words, the norm of  the operator $D_t$ on $L^q(X,\mu)$ is $\leq  1.$  
\end{remark}

We have the  following  decomposition (see  Proposition  4.7.9 in \cite{Walczak}).
\begin{proposition} \label{prop_current_local}
Let $\mu$ be a   harmonic  measure on  $X$. Let $\U\simeq \B\times\T$ be a flow
box as above which is relatively compact in $X$. Then, there is a positive Radon
measure $\nu$ on $\T$ and for $\nu$-almost every $t\in \T$ there is a
positive harmonic function $h_t$ on $\B$ 
such that if $K$ is compact in $\B,$ 
the integral $\int_\T \|h_t\|_{L^1(K)}d\nu(t)$ is finite and
$$\int  fd\mu=\int_\T \Big(\int_\B h_t(y) f(y,t) d\Vol_t(y)\Big) d\nu(t)
$$
for every  continuous compactly supported function    $f$ on $\U.$
Here $\Vol_t(y)$  denotes the  volume form on $\B$ induced by the metric tensor $g|_{\B\times \{t\}}.$
\end{proposition}
\begin{proof}
The local decomposition is provided by the disintegration of the measure with
respect to the fibration $\pi:\ \U\simeq \B\times\T\to\T$ which is constant on the leaves. This allows to
find a measure $\nu:=\pi_*(\mu|_{\U})$ on $\T$  and a measurable assignment of a probability measure $\lambda_t$ on $\B$
to $\nu$-almost all $t\in \T$  such that
$$
\int  fd\mu=\int_\T \Big(\int_{y\in\B}  f(y,t) d\lambda_t(y)\Big) d\nu(t)
$$
for every  continuous compactly supported function    $f$ on $\U.$
 This is a point where the local compactness of $X$ is used.
Next, we consider  the convex  closed set of all  harmonic probability measures supported in $\U.$
Let $\theta$ be   an extremal  element  of the  last set.
Writing  $\theta= (1-\chi)\theta+\chi\theta$  for some continuous  function  $t\mapsto \chi(t)$ which is defined  and compactly  supported
on $\T$ such that $0\leq \chi\leq 1,$ it follows that
$\theta$ is  supported  on  a single fiber $t,$  that is,
$\theta=\lambda \otimes \delta_t,$ where $\lambda$ is a probability  measure on $\B\times\{t\}$
and $\delta_t$ is the Dirac mass at $t.$ 
 Since $\theta$ is  harmonic  we deduce from Definition \ref{defi_Standing_Hypotheses_harmonicity} (ii) and from
  the regularity results for weak solutions to elliptic
differential equations that  $\lambda=h_t\Vol_t,$
where $h_t$ is  a harmonic function  on the plaque  $\B\times \{t\}$ and  $\Vol_t$
is the  Riemannian  volume  form  on this plaque. Using this representation, 
the Choquet  decomposition theorem \cite{Choquet}\index{Choquet!$\thicksim$ decomposition theorem}\index{theorem!Choquet decomposition $\thicksim$} allows us to conclude  the proof.
\end{proof}
 
Throughout the  Memoir unless otherwise  specified we   assume that 

{\bf $(X,\Lc,g)$   satisfies  the Standing  Hypotheses (H1) and (H2).}\index{Hypothesis!Standing $\thicksim$s}

 \section{Brownian motion and Wiener  measures without holonomy}
 \label{subsection_Brownian_motion_without_holonomy}
 \index{Brownian motion}

In this  section we follow the   expositions  given in \cite{CandelConlon2}  and  \cite{Candel2}.
Recall  first the  following terminology.  An {\it algebra}\index{algebra!algebra}  $\Ac$  on a set $\Sigma$ is  a  family  of subsets of $\Sigma$ such that
  $\Sigma\in  \Ac$ and that $X\setminus A\in\Ac$ and $A\cap B\in \Ac$  for all $A,\ B\in \Ac.$ 
  If, moreover, $\Ac$ is  stable  under countable  intersections, i.e.,
  $\bigcap_{n=1}^\infty A_n\in\Ac$
  for any sequence $(A_n)_{n=1}^\infty\subset  \Ac,$ then  $\Ac$ is  said to be  a {\it  $\sigma$-algebra.}\index{algebra!$\sigma$-algebra}
The ($\sigma$-) algebra  generated  by a family $\Sc$ of subsets  of $\Sigma$ is, by definition, the  smallest  ($\sigma$-)algebra  
containing  $\Sc.$\index{algebra!$\sigma$-algebra!$\thicksim$ generated by a family}\index{algebra!algebra!$\thicksim$ generated by a family}
If $\Sigma$ is  a  topological space, then  the Borel  ($\sigma$-)algebra is  the ($\sigma$-)algebra generated by  all open sets of $\Sigma.$  
 The  Borel $\sigma$-algebra\index{algebra!$\sigma$-algebra!Borel $\thicksim$}\index{Borel!$\thicksim$ $\sigma$-algebra}  of $\Sigma$ is  denoted by $\Bc(\Sigma).$
 The elements of  $\Bc(\Sigma)$ are called {\it Borel sets}.\index{Borel!$\thicksim$ set}\index{set!Borel $\thicksim$}
  \nomenclature[a95]{$\Bc(\Sigma)$}{Borel $\sigma$-algebra of a topological space $\Sigma$}

Let $(X,\Lc,g)$ be  a   Riemannian measurable lamination  satisfying  Hypothesis (H1).
Let  $\Omega:=\Omega(X,\Lc) $  be  the space consisting of  all continuous  paths  $\omega:\ [0,\infty)\to  X$ with image fully contained  in a  single   leaf.%
 \nomenclature[c1]{$\Omega(X,\Lc)$ or simply $\Omega$}{sample-path space associated to a (continuous or measurable) lamination $(X,\Lc)$} 
This  space  is  called {\it the sample-path space} associated to  $(X,\Lc)$\index{space!sample-path $\thicksim$}. Observe that
$\Omega$  can be  thought of  as the  set of all possible paths that a 
Brownian particle, located  at $\omega(0)$  at time $t=0,$ might  follow as time  progresses. The  heat kernel  will be used  to construct  a family  $\{W_x\}_{x\in X}$ of probability measures   on $\Omega.$ 

  The construction of the measures on $\Omega$ needs to be done first in the space of all
maps from the half-line $\R^+=[0,\infty)$ 
\nomenclature[a4]{$\R^+$}{half real-line $[0,\infty)$} 
into $X,$ which is denoted by $X^{[0,\infty)}.$ The natural
topology of this space is the product topology, but its associated Borel $\sigma$-algebra is
too large for most purposes. Instead, we will use  the $\sigma$-algebra $ \mathfrak C$ generated by {\it cylinder sets (with non-negative times)}\index{cylinder!$\thicksim$ set (with non-negative times)}.
Recall that a  {\it cylinder  set}\index{cylinder!$\thicksim$ set with non-negative times}\index{set!cylinder $\thicksim$ with non-negative times} is a 
 set of the form
$$
C=C(\{t_i,B_i\}:1\leq i\leq m):=\left\lbrace \omega \in X^{[0,\infty)}:\ \omega(t_i)\in B_i, \qquad 1\leq i\leq m  \right\rbrace,
$$
\nomenclature[d1]{$C(\{t_i,B_i\}:1\leq i\leq m)$}{cylinder set associated to  Borel sets $B_1,\ldots,B_m$ and  a set of increasing non-negative times $0\leq t_1<t_2<\cdots<t_m$} 
where   $m$ is a positive integer  and the $B_i$ are Borel subsets\index{set!Borel $\thicksim$} of $X,$ 
and $0\leq t_1<t_2<\cdots<t_m$ is a  set of increasing times. 
In other words, $C$ consists of all elements of $X^{[0,\infty)}$ which can be found within $B_i$ at time $t_i.$

The structure of the measure space  $(X^{[0,\infty)}, \mathfrak C )$ is best understood by viewing it as
an inverse limit. To do so, let the collection of finite subsets of $[0,\infty)$ be partially
ordered by inclusion. Associated to each finite subset $F$ of $[0,\infty)$ is the measure
space $(X^F,{\mathfrak X}^F),$ where $\mathfrak{X}^F$ is the Borel $\sigma$-algebra of the product topology on $X^F.$ 
Each inclusion of finite sets $E\subset F$ canonically defines a projection $\pi_{EF}\ :X_F\to X_E$
which drops the finitely many coordinates in $F\setminus E.$ These projections are continuous,
hence measurable, and consistent, for if $E\subset F\subset G,$ then  $\pi_{EF} \circ\pi_{FG}=  \pi_{EG} .$   The family
$\{ (X^F,{\mathfrak X}^F),\pi_{EF}|\  E\subset F\subset [0,\infty) \ \text{finite}\}$ is an inverse  system of spaces, and its inverse limit is 
  $X^{[0,\infty)}$  with canonical projections $\pi_F\ :X^{[0,\infty)} \to X^F.$  The $\sigma$-algebra $ \mathfrak C$ generated by the
cylinder sets is the smallest one making all the projections $\pi_F$ measurable. 
\nomenclature[d2]{$\mathfrak C$}{$\sigma$-algebra on $X^{[0,\infty)}$ generated by all cylinder sets with non-negative times} 

For each $x\in X,$ a probability measure\index{measure!probability $\thicksim$} $W_x$ on the measure space $(X^{[0,\infty)}, \mathfrak C )$ will
now be defined. If $ F=\{0 \leq t_1<\cdots<t_m\}$ is  a finite  subset of  $[0,\infty)$ and
   $C^F:=
B_1\times \cdots\times B_m$ is a cylinder set of $(X^F, \mathfrak{X}^F),$  define
\begin{equation}\label{eq_formula_W_x_without_holonomy}
W^F_x(C^F):=\Big (D_{t_1}(\chi_{B_1}D_{t_2-t_1}(\chi_{B_2}\cdots\chi_{B_{m-1}} D_{t_m-t_{m-1}}(\chi_{B_m})\cdots))\Big) (x),
\end{equation}
where $\chi_{B_i}$
is the characteristic function of $B_i$ and $D_t$ is the diffusion operator
given  by  (\ref{eq_diffusions}).
It is an obvious consequence of the semi-group property of $D_t$ (see (\ref{eq_semi_group}) that if $E\subset F$ are
finite subsets of $[0,\infty)$  and $C^E$ is a cylinder subset of $X^E,$ then
$$ W_x^E(C^E)=W_x^F( \pi_{EF}^{-1} (C^E)).$$
Let  $\mathfrak S$ be the  (non $\sigma$-) algebra generated by the cylinder sets in $X^{[0,\infty)}.$ %
\nomenclature[d3]{$\mathfrak S$}{algebra on $X^{[0,\infty)}$ generated by all cylinder sets with non-negative times} 
The  above identity  implies that $W^F_x$ given in (\ref{eq_formula_W_x_without_holonomy})     
 extends  to    a  countably additive, increasing, non-negative-valued function\index{function!countably additive $\thicksim$}\index{function!increasing $\thicksim$}
 \index{function!non-negative-valued $\thicksim$} 
 a measure $W_x$ on $\mathfrak S$ (see Kolmogorov's theorem \cite[Theorem 12.1.2]{Dudley}),\index{Kolmogorov!$\thicksim$'s theorem}\index{theorem!Kolmogorov's $\thicksim$}
 hence to  an outer measure\index{measure!outer $\thicksim$}
on the family of all subsets of  $X^{[0,\infty)}.$ The   $\sigma$-algebra of sets  that are measurable with respect to this  outer  measure contains the  cylinder sets, hence contains  the $\sigma$-algebra $\mathfrak C.$
The  Carath\'eodory-Hahn extension theorem 
\index{theorem!Carath\'eodory-Hahn extension $\thicksim$}\index{Carath\'eodory-Hahn!$\thicksim$ extension theorem}\cite{Wheeden}
 then  guarantees that the  restriction of  this  outer measure\index{measure!outer $\thicksim$} to $\mathfrak C$ is the unique measure agreeing with $W_x$ on  the  cylinder sets.
 This measure $W_x$ gives the set of paths $\omega\in X^{[0,\infty)}$ with $\omega(0)=x$   total probability.

\begin{theorem}\label{thm_Brownian_motions} 
The  subset $\Omega$
 of  $X^{[0,\infty)}$  has outer measure\index{measure!outer $\thicksim$} $1$ with respect  to $W_x.$
\end{theorem}
\begin{proof} Although this result is  stated  in Theorem C.2.13 in \cite{CandelConlon2}, we  still  give here
a more complete argument for the reader's convenience.
The   case  where $X$ is  a  single  leaf has been proved  in Appendix  C4 in \cite{CandelConlon2}.
Note  that   here is the place  where  we make use  of the  Hypothesis (H1).
The  general  case of a lamination $(X,\Lc)$ follows  almost along  the  same lines.  More precisely,
Lemma  C.4.2  in  \cite{CandelConlon2} still holds in the context of  a lamination  $(X,\Lc)$
noting that  given  a  countable  subset $F$ of  $[0,\infty),$
then  
$$ W_x \Big (\left\lbrace \omega\in X^{[0,\infty)}:\  \omega(t)\not\in L_x\ \text{for some}\  t\in F\right\rbrace \Big)=0.
  $$  
\end{proof}
Let $\widetilde\Ac:=\widetilde\Ac(\Omega)=\widetilde\Ac(X,\Lc)$ be the $\sigma$-algebra on  $\Omega$  consisting of all  sets $A$ of the  form $A=C\cap \Omega,$
with $C\in\mathfrak C.$ 
\nomenclature[e1]{$\widetilde \Ac$}{$\sigma$-algebra on the sample-path space $\Omega$ generated by all cylinder sets, or equivalently, trace  of $\sigma$-algebra $\mathfrak C$ on $\Omega$}
Then we  define  {\it the Wiener measure} at  a  point $x\in X$ by the  formula:\index{measure!Wiener $\thicksim$ (without holonomy)}\index{Wiener!$\thicksim$ measure (without holonomy)}
\nomenclature[g1]{$W_x$}{Wiener measure  without holonomy,
 it is  defined on the measurable spaces
$(X^{[0,\infty)},\mathfrak C)$ and
 $(\Omega,\widetilde\Ac)$} 
\begin{equation} \label{eq_defi_W_x}
 W_x(A)=W_x(C\cap \Omega):=W_x(C).
 \end{equation}
$W_x$ is   well-defined on  $\widetilde\Ac.$ Indeed, the $W_x$-measure of any  measurable  subset  of $ X^{[0,\infty)}\setminus \Omega$
is  equal to $0$   by 
   Theorem \ref{thm_Brownian_motions}. If $C, C'\in \mathfrak C$  and
   $C\cap \Omega=C'\cap \Omega,$ then  the  symmetric difference $(C\setminus C')\cup (C'\setminus C)$ is  contained in $  X^{[0,\infty)}\setminus \Omega
  , $ so $W_x(C)=W_x(C').$
 Hence, $W_x$ produces  a  probability measure on $(\Omega,\widetilde\Ac).$  
 We say that $A\in \widetilde\Ac$ is  a {\it  cylinder set (in $\Omega$)}  if  $A=C\cap \Omega$ for  some  cylinder set $C\in \mathfrak C.$
 The  measure space $(\Omega,\widetilde\Ac)$ has been thoroughly  investigated  in the  works of Candel and Conlon
in \cite{CandelConlon2, Candel2}. \index{Candel}\index{Conlon}
We record  here  a   useful  property of cylinder sets (in $\Omega$).
\begin{proposition}\label{prop_cylinder_sets}
1) If $A$ and $B$ are two cylinder  sets, then $A\cap B$ is  a  cylinder set and $\Omega\setminus A$
is a  finite  union of mutually disjoint  cylinder sets.
In particular,  the family of all  finite unions of cylinder sets forms an algebra  on $\Omega.$
\\
2)  If $A$ is  a countable  union of cylinder sets, then it is also a countable  union of mutually disjoint cylinder sets.
\end{proposition}
\begin{proof}
Part 1) follows  easily from  the  definition of cylinder sets.

To prove Part 2) let $A=\cup^\infty_{n=1}A_n$ with $A_n$ a cylinder set. Write
$A=\cup_{n=1}^\infty B_n,$ where $B_1:=A_1$ and  $B_n:=  A_n\setminus B_{n-1}$ for $n>1.$
Then $B_n\cap B_m=\varnothing$ for $n\not=m.$
 On the  other hand, using  Part 1) we can show by induction on $n$ that   
each $B_n$ is a  finite  union of mutually disjoint  cylinder sets.  This proves Part 2).
\end{proof}

 \section{Wiener  measures with holonomy}
 \label{subsection_Wiener_measures_with_holonomy}
 
 Let $(X,\Lc,g)$ be  a  Riemannian  measurable lamination  satisfying  Hypothesis (H1) and let $\Omega:=\Omega(X,\Lc).$
 Assume  in addition  that  there is    a covering measurable lamination $(\widetilde X,\widetilde\Lc)$  of $(X,\Lc).$
This  assumption is  automatically satisfied  when, for  example,  $(X,\Lc) $ is  a lamination.   
 The measure space   $(\Omega,\widetilde\Ac)$ defined  in the previous section does not  detect  the holonomy of the leaves.
 Here is  a simple  example.
\begin{example}\rm
  Given two points $x_0,\ x_1$ in  a common leaf $L,$   the  cylinder set
$$
C=C\Big (\{0,\{x_0\}\},\{1,\{x_1\}\}\Big)=\left\lbrace  \omega\in\Omega:\ \omega(0)=x_0,\ \omega(1)=x_1   \right\rbrace 
$$
 does  not distinguish the homotopy  type  of  the path  $\omega|_{[0,1]}$ in $L.$  
  \end{example}
 Now we  introduce  a new   $\sigma$-algebra $\Ac$ on $\Omega$ which  contains $\widetilde \Ac $
and  which has the  advantage  of  taking into account the  holonomy of the leaves.  
This new object will play a vital role in this Memoir. 
  Let $\pi:\ (\widetilde X,\widetilde \Lc)\to (X,\Lc)$  be the covering lamination projection. It is  a leafwise map.\index{leafwise!$\thicksim$ map}
Fix an arbitrary
$\tilde x\in \widetilde X$ and  $x:=\pi(\tilde x)\in X.$
Let $\Omega_x=\Omega_x(X,\Lc)$ be the  space  of all continuous
leafwise paths  starting at $x$ in $(X,\Lc),$ that is,
$$\Omega_x:=\left\lbrace   \omega\in \Omega:\  \omega(0)=x\right\rbrace.$$
\nomenclature[c1a]{$\Omega_x$}{space  of all continuous
leafwise paths  starting at $x$ in a measurable lamination $(X,\Lc)$}
 Analogously, let   $\widetilde\Omega_{\tilde x}=\Omega_{\tilde x}(\widetilde X,\widetilde\Lc)$
\nomenclature[c1d]{$\widetilde\Omega_{\tilde x}$}{space  of all continuous
leafwise paths  starting at $\tilde x$ in a covering  lamination $(\widetilde X,\widetilde \Lc)$ of a  measurable lamination $(X,\Lc)$}
 be the  space  of all continuous
leafwise paths  starting at $\tilde x$ in $(\widetilde X,\widetilde\Lc).$
Every  path $\omega\in \Omega_x$ lifts uniquely  to
a path $\tilde\omega\in \widetilde\Omega_{\tilde x}$  in the sense  that $\pi\circ \tilde\omega=\omega,$
that is, $\pi(\tilde\omega(t))=\omega(t)$ for all $t\geq 0.$
In what follows this bijective lifting  is  denoted by $\pi^{-1}_{\tilde x}:\  \Omega_x\to \widetilde\Omega_{\tilde x}.$ So  $\pi\circ (\pi^{-1}_{\tilde x}(\omega))=\omega,$
 $\omega\in \Omega_x.$ 
  By  Section \ref{subsection_Brownian_motion_without_holonomy}, we  construct  a $\sigma$-algebra $\widetilde\Ac(\widetilde \Omega)$ on $\widetilde \Omega:=\Omega(\widetilde X,\widetilde \Lc)$
  \nomenclature[c1b]{$\widetilde \Omega(X,\Lc)$ or simply $\widetilde \Omega$}{$:=\Omega(\widetilde X, \widetilde \Lc),$ where $(\widetilde X, \widetilde \Lc)$ is a covering lamination of 
  a measurable lamination $(X,\Lc)$}
  which is  the $\sigma$-algebra generated by all cylinder sets in  $\widetilde \Omega.$
  \begin{definition} \label{defi_algebras_Ac} Let   
  $ \Ac=\Ac(\Omega)$ be  the  $\sigma$-algebra generated by
all sets  of  following family
 $$ \left  \lbrace \pi\circ \tilde A:\ \text{cylinder set}\ \tilde A\  \text{in} \ \widetilde \Omega    \right\rbrace,$$ 
  where  $\pi\circ \tilde A:= \{ \pi\circ \tilde \omega:\ \tilde\omega\in \tilde A\}.$
 \nomenclature[e2]{$\Ac$}{$\sigma$-algebra on the sample-path space $\Omega$ generated by the images of all  cylinder sets in $\widetilde\Omega:=\Omega(\widetilde X,\widetilde\Lc)$}

  For  a  point $x\in X,$ we apply  the previous   definition to the lamination  consisting  of a  single  leaf $L:=L_x$ with its  covering  projection $\pi:\ \widetilde L \to L.$ By   setting $\Omega:=\Omega(L)$ and  $\widetilde \Omega:=\Omega(\widetilde L)$
\footnote{ When  a lamination $(X,\Lc)$  consists of a single leaf $L,$  i.e.  $(X,\Lc)=(L,L),$  we often write $\Omega(L)$ instead of $\Omega(L,L).$}   in this context,  we can define  the $\sigma$-algebra $\Ac(L):= \Ac(\Omega(L)).$
 Let   $\Ac_x$  be  the  restriction of  $\Ac(L)$ on the  set $\Omega_x.$  
\nomenclature[e3]{$\Ac_x$}{trace of the $\sigma$-algebra $\Ac$ on  the subspace $\Omega_x\subset \Omega$}
So $\Ac_x$ is
 a  $\sigma$-algebra on $\Omega_x.$
  \end{definition}
 Observe that    $\widetilde \Ac\subset  \Ac$ and that the  equality holds if  all leaves  of $(X,\Lc)$ have
trivial holonomy. Moreover, we also have that $\Ac_x\subset \Ac.$ 
 Note that $ \Ac(\widetilde\Omega)=\widetilde \Ac(\widetilde\Omega).$

Now  we  construct  a family  $\{W_x\}_{x\in X}$ of probability Wiener measures   on $(\Omega,\Ac).$ 
  For   a  point $ x \in X,$  we want to  define formally the  so-called {\it  Wiener measure (with holonomy) at $x$}\index{measure!Wiener $\thicksim$ (with holonomy)}
\index{Wiener!$\thicksim$ measure (with holonomy)} as follows:
  \nomenclature[g2]{$W_x$}{Wiener measure   with holonomy, it is  defined on the measurable space $(\Omega,\Ac)$}
\begin{equation}\label{eq_formula_W_x}
 W_x(A):= W_{\tilde x}( \pi^{-1}_{\tilde x} A),\qquad A\in\Ac,
  \end{equation}
  where   
   $\tilde x$ is  a  lift  of $x$ under  the projection $\pi:\ \widetilde L\to L=L_x,$
     and  
   $$\pi^{-1}_{\tilde x}(A):=\left\lbrace \pi^{-1}_{\tilde x}\omega:\  \omega\in A\cap \Omega_x\right\rbrace,
$$ and  $W_{\tilde x}$ is the probability measure  on $(\Omega(\widetilde L),\widetilde \Ac(\widetilde L))$
given    by (\ref{eq_defi_W_x}).
 
 Given  a $\sigma$-finite positive   Borel measure $\mu$  on $X,$ we  want to  construct formally  a  $\sigma$-finite positive   measure  $\bar\mu$ on $(\Omega,\Ac)$  as  follows:
       \begin{equation}\label{eq_formula_bar_mu}
   \bar\mu(A):=\int_X\left ( \int_{\omega\in A\cap   \Omega_x}  dW_x \right ) d\mu(x)= \int_X W_x( A) d\mu(x) ,\qquad  A\in\Ac. 
\end{equation}
The  measure $\bar\mu$ (if  well-defined) is  called  
the  {\it Wiener measure with initial distribution $\mu$.}\index{measure!Wiener $\thicksim$ with a given initial distribution}\index{Wiener!$\thicksim$ measure with a given initial distribution}
\nomenclature[h1]{$\bar\mu$}{Wiener measure with initial distribution $\mu$, it is  defined on the measurable space $(\Omega,\Ac)$}


When   $(X,\Lc,g) $ is  a Riemannian lamination endowed  with  its  covering lamination projection
$\pi:\  (\widetilde X,\widetilde\Lc,\pi^*g)\to (X,\Lc,g),$
the  next  two  results   show  that formulas (\ref{eq_formula_W_x}) and (\ref{eq_formula_bar_mu}) are, in fact, well-defined.

\begin{theorem}\label{prop_Wiener_measure}
Let $(X,\Lc,g) $ be  a  Riemannian lamination.\\
(i)  Then  for every $x\in X$ and $\tilde x\in \pi^{-1}(x)$ and $A\in\Ac,$ $\pi^{-1}_{\tilde x} A\in \widetilde \Ac(\widetilde\Omega)$  
 and the value  of $W_x(A)$ defined  in (\ref{eq_formula_W_x}) is  independent of the choice
 of $\tilde x\in \pi^{-1}(x).$  Moreover, $W_x$ is a  probability  measure     
on $(\Omega,\Ac).$
\\
(ii) $\pi^{-1} A\in \widetilde \Ac(\widetilde\Omega)$ for every $A\in\Ac,$
where 
$$\pi^{-1}(A):=\left\lbrace \tilde\omega\in\widetilde\Omega  :\  \pi\circ \tilde\omega\in A\right\rbrace.
$$
\end{theorem}  
 Since the proof of this theorem  is  somehow  technical,  we  postpone it to  Appendix 
\ref{subsection_algebra_on_a_leaf} and  \ref{subsection_algebra_on_a_lamination}   
  below for the sake of clarity.

    \begin{theorem}\label{thm_Wiener_measure_measurable}
 If  $(X,\Lc,g) $ is  a Riemannian lamination.\\
(i)  Then   for each  element $ A\in\Ac,$ the  function  $X\ni x\mapsto W_x(A)\in[0,1]$ 
 is   Borel  measurable.   
  \\
(ii)  If $\mu$ is   a $\sigma$-finite positive  Borel  measure (resp. a Borel probability  measure) on $X,$ then  $\bar\mu$ 
given in (\ref{eq_formula_bar_mu}) is   $\sigma$-finite positive   measure (resp.   a probability measure) on $(\Omega,\Ac).$
\end{theorem}
 The  above  result will be  proved in   Appendix \ref{subsection_algebra_on_a_lamination}.

At first  reading the  reader may skip the  remainder of this  chapter and  jump ahead to the  next one.
In Chapter \ref{section_Lyapunov_filtration} and Chapter \ref{section_Main_Theorems} below,  we  need  to work with measurable laminations 
which  possess nice properties as the Riemannian (continuous)  laminations. This  gives  rise to the  following 

\begin{definition}\label{defi_continuity_like}
We say that a Riemannian measurable lamination $(X,\Lc,g)$ is  a {\it (Riemannian) continuous-like  lamination}\index{lamination!Riemannian continuous-like $\thicksim$}
if   there is    a covering measurable lamination $(\widetilde X,\widetilde\Lc)$  of $(X,\Lc)$
together  with  its covering lamination projection $\pi:\ (\widetilde X,\widetilde \Lc,\pi^*g)\to (X,\Lc,g)$  such that
the following properties  hold:
\\
(i-a) There is  a  {\it global section}\index{section!global $\thicksim$}  of $\pi,$ that is, there is a
 Borel measurable map  
\index{Borel!$\thicksim$ measurable function (or map)}\index{map!Borel measurable $\thicksim$}
$s:\ X\to\widetilde X$ such that
$\pi(s(x))=x,$ $x\in X.$
\\
(i-b) There is a  family of maps $s_i:\ E_i\to\widetilde X,$  with $E_i\subset X,$ indexed by a (at most) countable set $I,$ satisfying the  following three properties: 
\begin{itemize}
\item [$\bullet$]  each $s_i$ is a {\it local section}
\index{section!local $\thicksim$}  of $\pi,$ that is, $\pi(s_i(x))=x$ for all  $x\in E_i$ and $i\in I;$
\item [$\bullet$]  for each $i\in I,$ both  $E_i$ and $s_i(E_i)$  are Borel sets, and  the surjective  map $s_i:\ E_i\to s_i(E_i)$ is Borel bi-measurable;\index{measurable!bi-$\thicksim$ map}
\item [$\bullet$] the family $(s_i)_{i\in I}$
  {\it generates all fibers of $\pi,$} that is,
$$
\pi^{-1}(x):=\left  \lbrace  s_i(x):\  x\in E_i\ \text{and}\ i\in I  \right\rbrace,\qquad x\in X.
$$
\end{itemize}
(ii) For every $x\in X$ and $\tilde x\in \pi^{-1}(x)$ and $A\in\Ac,$ $\pi^{-1}_{\tilde x} A\in \widetilde \Ac(\widetilde\Omega)$  
 and the value  of $W_x(A)$ defined  in (\ref{eq_formula_W_x}) is  independent of the choice
 of $\tilde x\in \pi^{-1}(x).$  Moreover, $W_x$ is a  probability  measure     
on $(\Omega,\Ac).$
\\
(iii)
For every $A\in \Ac,$  $\pi^{-1} A\in \widetilde \Ac(\widetilde\Omega).$
\\
 (iv) For each  element $ A\in\Ac,$ the  function  $X\ni x\mapsto W_x(A)\in[0,1]$ 
 is   Borel  measurable. For each  element $ A\in\Ac(\widetilde\Omega),$ the  function  $\widetilde X\ni \tilde x\mapsto W_{\tilde x}(\tilde A)\in[0,1]$ 
 is   Borel  measurable.  
\\
(v)  If $\mu$ is   a $\sigma$-finite positive  Borel measure (resp. a  Borel probability measure) on $X,$ then  $\bar\mu$ given in (\ref{eq_formula_bar_mu}) is   $\sigma$-finite positive   measure (resp.   a probability measure) on $(\Omega,\Ac).$
\end{definition} 

The  following result justifies  the terminology {\it Riemannian continuous-like}.
\begin{proposition}\label{P:lami-is-cont-like} 
A Riemannian (continuous) lamination $(X,\Lc,g)$  endowed  with  its  covering lamination projection
$\pi:\  (\widetilde X,\widetilde\Lc,\pi^*g)\to (X,\Lc,g)$ is   Riemannian continuous-like. 
\end{proposition}
\begin{proof}
 By Theorem  \ref{prop_Wiener_measure}, Definition \ref{defi_continuity_like} (ii)-(iii) are fulfilled.
 By Theorem \ref{thm_Wiener_measure_measurable}, Definition \ref{defi_continuity_like} (iv)-(v) are fulfilled.
 So it remains to check   Definition \ref{defi_continuity_like} (i-a) and (i-b).
 
 To prove Definition \ref{defi_continuity_like} (i-a), consider  a family of (at most) countable pairs $(U_i,\alpha_i)_{i\in I}$ such that  $(U_i)_{i\in I}$ is a cover of $X$ by  open subsets  and that
$\alpha_i$ is  a homotopy on $U_i$ for each $i\in I$ (see  
Section 
 \ref{subsection_Covering_laminations}).
 It is  easy to  define  a partition $(F_i)_{i\in I}$ of $X$  by Borel subsets such that $F_i\subset U_i$ for each $i\in I.$
 Now  it suffices  to define $s:\ X\to\widetilde X$ as  follows:
 \begin{equation*}
 s(x):=(x,[\alpha_{i,x}]),\qquad  x\in F_i,\ i\in I.
 \end{equation*}
 This map satisfies  Definition \ref{defi_continuity_like} (i-a).
 
 To check  Definition \ref{defi_continuity_like} (i-b), recall from  Section \ref{subsection_Covering_laminations} that
 there is a  countable  basis  of open sets on $(\widetilde X,\widetilde\Lc)$
 which  consists of a family $(\Nc(U_i,\alpha_i))_{i\in \N},$
where $U_i$ is  an open subset of $X$ and $\alpha_i$ is  a homotopy on $U_i.$  
Now for each $i\in\N,$ let $E_i:=U_i$ and let $s_i:\ E_i\to \widetilde X$ be given by
 \begin{equation*}
 s_i(x):=(x,[\alpha_{i,x}]),\qquad  x\in E_i,\ i\in I.
 \end{equation*}
 The family $(s_i)_{i\in\N}$ satisfies  Definition \ref{defi_continuity_like} (i-b).
 \end{proof}
 \begin{remark}
 The converse of Proposition \ref{P:lami-is-cont-like}  is,  in general, not true. Indeed,
 in Theorem \ref{T:cylinder_lami_is_conti_like} below, we investigate a  class of  Riemannian   continuous-like laminations
 which  are, in most cases,  only Riemannian measurable  (see also Remark \ref{rem_cylinder_lamination}).
 \end{remark}
  
  Recall  from  (\ref{eq_shift}) the  shift-transformations $T^t:\ \Omega\to\Omega,$ $t\in\R^+,$\index{shift-transformation}
  and  let  $T:=T^1$ be the shift-transformation of unit-time.\index{shift-transformation!$\thicksim$ of unit-time}
   The following  result relates the weak harmonicity 
of probability measures defined on $(X,\Lc)$
to the invariance  
 of the corresponding  measures  on $\Omega:=\Omega(X,\Lc).$

\begin{theorem}\label{thm_invariant_measures}
1)  Let $(X,\Lc,g) $ be  a Riemannian continuous-like lamination.
If   $\mu$ is   a very weakly harmonic  (resp. weakly harmonic)  measure on  $(X,\Lc),$ then   $\bar\mu$    is 
$T$-invariant (resp.    $T^t$-invariant for all $t\in \R^+$) on $(\Omega,\Ac).$  
\\
2) Let $(X,\Lc,g) $ be  a Riemannian (continuous) lamination.
 If   $\mu$ is   a very weakly harmonic  (resp. weakly harmonic)  measure on  $(X,\Lc),$ then   $\bar\mu$    is 
$T$-invariant (resp.    $T^t$-invariant for all $t\in \R^+$) on $(\Omega,\Ac).$  
\end{theorem}

Clearly,  Part 2) of Theorem \ref{thm_invariant_measures} follows  from a combination of Part 1) and Proposition \ref{P:lami-is-cont-like}.
The proof of Part 1) of Theorem \ref{thm_invariant_measures} will be provided  in  Appendix  \ref{section_invariance}.

%% file: chapter3.tex

\chapter{Statement of  the main results}
\label{section_main_results}

 
First, we   introduce a  notion of  multiplicative  cocycles
for Riemannian  laminations. 
 Next,  we state   our main
 results  as  well as  their  applications.  
 Finally   we  outline  their proofs. In this  chapter  $\K$ denotes  either $\R$ or $\C.$ 
 
 \section{Multiplicative  cocycles}
 \label{section_cocycles}

Observe that  the orbit  $L_x$ of  a  point $x\in X$ by a transformation  $T:\ X\to X$  is  the  set $\{  T^nx:\  n\in \N\},$
and  hence can be ordered  by the unique   map $\omega:\  \N\to  L_x$ given by  $\omega(n)=T^nx, $  $n\in\N.$
 The  case of   laminations is  quite different:  the orbit  $L_x$ of  a  point $x\in X$  by a   lamination $(X,\Lc)$
is,  as expected,  the whole leaf  passing through  $x.$  However, this  leaf  is   a  manifold  and  hence  it cannot be  time-ordered  by  $\N.$
Therefore,  it is  natural to   replace   the unique   map $\omega$ in the context of  a transformation $T$    by  the   space  $\Omega_x. $  Hence,  a plausible   (multiplicative) cocycle on $(X,\Lc)$ should be a  multiplicative  map
$\mathcal{A}:\ \Omega \times \R^+ \to  \GL(d,\K)     $  such that $\mathcal{A}(\omega,0)=\id$  for all $\omega\in\Omega  .$
Obviously, this  temporary definition is  still  not good  enough. Indeed,  since  the space  $\Omega $ is   too  large,  there are plenty of
pairs $(\omega,t)$  consisting of  an $\omega\in  \Omega $ and  a $t\in \R^+ $ such that $\omega(0)=\omega(t)(=x),$  should  we  
request that     $$\mathcal{A}(\omega,t)= \mathcal{A}(\omega,0)\qquad ?$$  
  At this  stage the  topology of  the leaf  $L_x$     comes  into play as  suggested  to us  by  Candel's definition of cocycles for $\GL(1,\K).$ \index{Candel} So
it is   quite natural  to  assume   the last identity  when  $\omega|_{[0,t]}$   is   null-homotopic
 \index{homotopy!null-homotopic} in $L_x,$ and hence
      the matrix $\mathcal{A}(\omega,t)$ should  depend  only on the  class of homotopy  of paths $\omega|_{[0,t]}$ 
with two  fixed ends-points $\omega(0)$ and  $\omega(t).$     
So   a reasonable  definition of  (multiplicative) cocycles  should  reflect the topology  of  the leaves  of $(X,\Lc).$      
The notion of homotopy  for  paths in $\Omega $ can be  made precise  as follows.

\begin{definition}\rm
 Let   $\omega_1:\ [0,t_1]\to  X$ and  $\omega_2:\ [0,t_2]\to  X$      be  two continuous paths with  image fully contained in a single leaf 
$L$ of $(X,\Lc).$
We say that   $\omega_1$ is {\it  homotopic}
\index{homotopy!homotopic} to $\omega_2$ if  there  exists a  continuous map $\omega:  I \to L$ with
$I:= \left \{(t,s)\in \R_{+}\times [0,1] :   0   \leq  t \leq  (1-s) t_1+st_2  \right \}$
such that
  $\omega(  0 ,s)=\omega_1(0)=\omega_2(0)$  and
  $\omega(  (1-s) t_1+st_2 ,s)=\omega_1(t_1)=\omega_2(t_2)$   for all $s\in [0,1],$
and   $\omega(\cdot, 0)=\omega_1$ and $\omega(\cdot, 1)=\omega_2.$
In other words, the path $\omega_1$ may be  deformed  continuously on  $L$ to  $\omega_2,$ the two ends of $\omega_1$  being kept fixed  during the deformation.
\end{definition}
The  last  point of our discussion  is that    the parameter semi-group $\N$ in the context of     transformations in discrete dynamics may be  replaced by  
either $\N t_0$ $(t_0>0)$ or  $\R^+$ in the context of  laminations.

Taking all the above considerations,  we are able  to formulate   a  good notion of multiplicative cocycles for laminations.
\begin{definition} \label{defi_cocycle} \rm  
Let  $\G$ be  either  $\N t_0$ (for some $t_0>0$) or $\R^+,$  and  $d\geq 1$  an integer, and  $\K\in\{\R,\C\}.$ 
\nomenclature[a3]{$\K$}{denotes either the field $\R$ of real  numbers or the field $\C$ of complex numbers}
A  {\it (multiplicative) cocycle of rank $d$} on $\Omega  $ is  a   map  
$$\mathcal{A}:\ \Omega\times \G \to  \GL(d,\K)      $$
such that\\  
(1)  (identity law) \index{law!identity $\thicksim$} 
$\mathcal{A}(\omega,0)=\id$  for all $\omega\in\Omega ;$\\
(2) (homotopy law)\index{law!homotopy $\thicksim$}\index{homotopy!$\thicksim$ law}   if  $\omega_1,\omega_2\in \Omega$ and $t_1,t_2\in \G$ such that 
   $\omega_1(0)=\omega_2(0)$ and  $\omega_1(t_1)=\omega_2(t_2)$
and $\omega_1|_{[0,t_1]}$ is  homotopic  to  $\omega_2|_{[0,t_2]},$ then 
$$
\mathcal{A}(\omega_1,t_1)=\mathcal{A}(\omega_2,t_2);
$$
(3) (multiplicative law) \index{law!multiplicative $\thicksim$}   $\mathcal{A}(\omega,s+t)=\mathcal{A}(T^t(\omega),s)\mathcal{A}(\omega,t)$  for all  $s,t\in \G$ and $\omega\in \Omega;$\\
(4) (measurable law) \index{law!measurable $\thicksim$} $\mathcal{A}(\cdot,t):\  \Omega\ni\omega\mapsto  \mathcal{A}(\omega,t)$
is  measurable   for every  $t\in \G.$
\index{cocycle!multiplicative $\thicksim$}%
\end{definition}
Observe that if 
$\mathcal{A}:\ \Omega\times \G \to  \GL(d,\K)$ is  a cocycle, then   the map $\mathcal{A}^{*-1}:\ \Omega\times \G \to  \GL(d,\K),$ defined by    
  $\mathcal{A}^{*-1}(\omega,t):= \big (\mathcal A(\omega,t)\big)^{*-1},$ is  also a cocycle,
where  $A^*$ (resp. $A^{-1}$) denotes as  usual the transpose (resp. the inverse)  of a square  matrix $A.$ 

As  a  fundamental example, we  define   the  holonomy   cocycle  
of  
a  $\Cc^1$ transversally  smooth foliation $(X,\Lc)$ of codimension $d$  in a  Riemannian  manifold $(X,g).$  
Let $T(\Lc)$ be the  tangent bundle   to the  leaves of the foliation, i.e., each fiber $T_x(\Lc)$ is  the tangent space  $T_x(L_x)$  for each point $x\in X.$
 The normal  bundle $N(\Lc)$  is, by definition,  the  quotient of $T(X)$ by  $T(\Lc),$
 that is, the  fiber  $N_x (\Lc)$ is  the quotient    $T_x(X)/ T_x(\Lc)$ for each $x\in X.$
 Observe that  the  metric $g$ on $TX$ induces a metric (still denoted by $g$) on  $N(\Lc).$
 For every transversal  section  $S$ at a point $x\in X,$ the tangent space $T_x(S)$ is  canonically  identified  with  $N_x (\Lc)$ through the composition $T_x(S)\hookrightarrow T_x(X)\to T_x(X)/ T_x(\Lc) .$

For every $x\in X$ and $\omega\in\Omega_x$  and $t\in \R^+,$
let  $h_{\omega,t}$  be the  holonomy map%
\index{holonomy!$\thicksim$ map}
\index{map!holonomy $\thicksim$}
 along the path  $\omega|_{[0,t]}$ from a fixed transversal section $S_0$ at $\omega(0)$ to
a  fixed transversal section $S_t$ at $\omega(t)$ (see  Appendix \ref{subsection_holonomy_maps} below). Using   the  above  identification,  the  derivative  $D h_{\omega,t}:\
T_{\omega(0)}(S_0)\to T_{\omega(t)}(S_t) $ induces 
a  map (still denoted  by)  $D h_{\omega,t}:\
N_{\omega(0)}(\Lc)\to N_{\omega(t)}(\Lc). $
The last  map depends  only on the path  $\omega|_{[0,t]},$ in particular, it does  not depend  on the choice  of transversal sections $S_0$ and $S_t.$ 
   
An {\it identifier}\index{identifier} $\tau$  of $(X,\Lc)$ is  a   smooth    map  which  associates, to  each point $x\in X,$
a linear   isometry\index{isometry, linear isometry} $\tau(x):\ N_x(\Lc) \to \R^d,$ that is, a linear morphism  
such that $$ \|\tau(x)v\|=\|v\|,\qquad v\in  N_x(\Fc),\ x\in X.$$
Here we have  used  the Euclidean norm on the  left-hand  side and the  $g$-norm on the  right hand  side.
 We identify  every  fiber   $N_x(\Lc)$    with $\R^d$ via    $\tau.$ 
The {\it  holonomy   cocycle}%
\index{foliation!holonomy cocycle of a $\thicksim$} 
\index{holonomy!$\thicksim$ cocycle}
\index{cocycle!holonomy $\thicksim$}
$\mathcal A$      
of $(X,\Lc)$ with respect  to the identifier  $\tau$ is  defined by
\begin{equation}\label{eq_holonomy_cocycle}
\mathcal A(\omega,t):=   \tau(\omega(t))\circ (D h_{\omega,t})(\omega(0))\circ \tau^{-1}(\omega(0)), \qquad \omega\in \Omega,\ t\in\R^+.
\end{equation}
It is  clear that the  holonomy cocycle is unique  up to a conjugacy  class.
\begin{proposition}\label{prop_derivative_cocycles}
The  holonomy   cocycle  $\mathcal A$ is  a multiplicative  cocycle.
\end{proposition}
\begin{proof}
Since $h_{\omega,0}=\id,$ we have that $\mathcal A(\omega,0)=\id.$
To prove that the homotopy law is  fulfilled, let    $\omega_1,\omega_2\in \Omega$ and $t_1,t_2\in \R^+$ such that 
   $\omega_1(0)=\omega_2(0)=x$ and  $\omega_1(t_1)=\omega_2(t_2)$
and $\omega_1|_{[0,t_1]}$ is  homotopic  to  $\omega_2|_{[0,t_2]}.$ 
By Proposition 2.3.2 in \cite{CandelConlon1}, $h_{\omega_1,t_1}=h_{\omega_2,t_2}$ on an open neighborhood of $x$ in a fixed  transversal section through $x.$ Hence,   $(Dh_{\omega_1,t_1})(x)=(Dh_{\omega_2,t_2})(x),$ which proves the homotopy law.

For $s,t\in \R^+$ and $\omega\in \Omega,$ we have,  by the chain rule,  
$$
 (Dh_{\omega,s+t})(\omega(0))=(Dh_{T^t(\omega),s})(\omega(t)) \circ(Dh_{\omega,t})(\omega(0)).
$$
Combining  this and the definition of $\mathcal A,$  the  multiplicative law follows.

The measurable  law is  an immediate  consequence of  Proposition \ref{prop_cocycle_criterion} below.
\end{proof}

One can perform    basic  operations on the category of cocycles  such as the tensor product, the direct sum  and the  wedge-product.
In this  Memoir  we are only  concerned  with the last operation.
Let   $\mathcal A_1$  and $\mathcal A_2$  be two cocycles   defined on   $\Omega\times\G$  with values in
     $\GL(d_1,\R)$ and   $\GL(d_2,\R)$ respectively.  Then  their wedge-product is  the map  $\mathcal A_1\wedge \mathcal A_2:\ \Omega\times\G\to\GL( \R^{d_1}\wedge \R^{d_2})$ given by  the  formula
$$
(\mathcal  A_1\wedge \mathcal A_2)(\omega,t)(v_1\wedge v_2):=\mathcal  A(\omega,t)v_1\wedge  \mathcal  A_2(\omega,t)v_2,\quad  \omega\in\Omega,\ t\in\G,\ v_1\in\R^{d_1},\ v_2\in  \R^{d_2}.
$$
The operation is  defined  analogously when $\mathcal A_1$ and $\mathcal A_2$  are with values in   $\GL(d_1,\C)$ and   $\GL(d_2,\C)$
respectively. We leave  to the  reader  to  prove  the following  result:
\begin{proposition}\label{prop_wedge_product_cocycles}
   $\mathcal  A_1\wedge \mathcal A_2$ is  a (multiplicative)  cocycle.
\end{proposition}

 \section{First Main Theorem  and  applications}\index{theorem!First Main $\thicksim$}

The following  notions  are needed.
\begin{definition}\label{defi_null_full_measure}\rm
Let $(S,\Sc,\nu)$ be a $\sigma$-finite positive  measure space.
 A subset  $Z\subset S$ is  said to be  {\it of null $\nu$-measure}\index{set!$\thicksim$ of null $\nu$-measure}
(resp.  {\it of full $\nu$-measure})\index{set!$\thicksim$ of full $\nu$-measure}    if    $\nu(Z)=0$ (resp.  $\nu(S\setminus Z)=0$).
\end{definition}

\begin{definition}\label{defi_ergodic_measures}\rm
Let $(X,\Lc,g)$ be a Riemannian measurable lamination.

A subset $Z\subset X$
is  said to be 
{\it  leafwise saturated} %
\index{leafwise!$\thicksim$ saturated} 
 if $a\in Z$ implies that  the whole leaf $L_a$
 is  contained  in $Z.$ 

A function $f$
defined on a leafwise  saturated set $Y\subset  X$   is  called {\it  leafwise  constant}%
 \index{leafwise!$\thicksim$ constant function}\index{function!leafwise constant $\thicksim$}
if 
it is   constant 
 on   each restriction of $f$  to $  L_a$ for each $a\in Y.$
   
   A positive  finite Borel  measure\index{measure!Borel $\thicksim$}  $\mu$  on $X$ is  said  to be  {\it ergodic}\index{measure!ergodic $\thicksim$} if  every  leafwise  saturated  measurable subset of $X$
   either has  full $\mu$-measure  or null $\mu$-measure. 
   \end{definition}
   
Now  we are in the position to state  our first (or abstract)  Oseledec Multiplicative Ergodic Theorem for  Riemannian laminations.
%

\begin{theorem} \label{th_main_1}
 Let $(X,\Lc,g)$ be   a       lamination satisfying the  Standing Hypotheses.  
 Let $\mu$ be a   harmonic probability measure\index{measure!probability $\thicksim$}. Let  $\G$ be either $\N t_0$ or $\R^+,$  
where  $t_0>0$ is  a given  number. 
Let
$\mathcal{A}:\ \Omega\times \G \to  \GL(d,\R)      $ be  a   cocycle on $\Omega.$  Assume 
that  $\mathcal A$ satisfies the  {\rm integrability condition},
\index{condition!integrability $\thicksim$}
 that is, 

$\bullet$
 if  $\G=\N t_0$ then
     $\int_\Omega \log^+ \|\mathcal{A}^{\pm 1}(\omega,t_0)\|  d\bar\mu(\omega)<\infty ;$
     
     $\bullet$
     if $\G=\R^+,$ then 
      $\int_\Omega \sup_{t\in [0,t_0]}\log^+ \|\mathcal{A}^{\pm 1}(\omega,t)\|  d\bar\mu(\omega)<\infty .$

Then  there exists  a   leafwise  saturated  Borel  set $Y\subset X$ of  full $\mu$-measure    such that
the following properties hold:
\\(i)  There is  a  measurable function  $m:\ Y\to \N$   which is   leafwise  constant.
 \\(ii)  For   each $x\in Y$  
 there   exists a  decomposition of $\R^d$  as  a direct sum of $\R$-linear subspaces 
$$\R^d=\oplus_{i=1}^{m(x)} H_i(x),
$$
 such that $\mathcal{A}(\omega, t) H_i(x)= H_i(\omega(t))$ for all $\omega\in  \Omega_x$ and $t\in \G.$     
Moreover,  the map $x\mapsto  H_i(x)$ is   a  measurable map from $\{x\in Y:\ m(x)\geq i\}$ into the Grassmannian of $\R^d.$
Moreover, there   are real numbers 
$$\chi_{m(x)}(x)<\chi_{m(x)-1}(x)<\cdots
<\chi_2(x)<\chi_1(x)$$
 such that    the function $x\mapsto \chi_i(x)$ is measurable  and  leafwise constant on $\{x\in Y:\ m(x)\geq i\},$
  and 
\begin{equation}\label{eq_property_ii}
\lim\limits_{t\to \infty, t\in \G} {1\over  t}  \log {\| \mathcal{A}(\omega,t)v   \|\over  \| v\|}  =\chi_i(x),    
\end{equation}
uniformly  on  $v\in H_i(x)\setminus \{0\},$ for  $W_x$-almost every  $\omega\in\Omega_x.$
The   numbers  $\chi_{m(x)}(x)<\chi_{m(x)-1}(x)<\cdots
<\chi_2(x)<\chi_1(x)$ are called  the {\rm  Lyapunov exponents}\index{Lyapunov!$\thicksim$ exponent} 
 associated to the cocycle $\mathcal{A}$ at the point $x.$\\
(iii) For  $S\subset  N:=\{1,\ldots,m(x)\}$ let $H_S(x):=\oplus_{i\in S} H_i(x).$ Then
\begin{equation}\label{eq_property_iii}
\lim\limits_{t\to \infty,\ t\in \G} {1\over t}  \log\sin {\big |\measuredangle \big (H_S(\omega(t)), H_{N\setminus S} (\omega(t))\big ) \big |}=0
\end{equation}
for  $W_x$-almost every  $\omega\in\Omega_x.$
\end{theorem}

\begin{remark}\label{R:th_main_1} Some remarks  are in order.

 $\bullet$ The  decomposition  
of  $\R^d$  as  a direct sum of subspaces 
$\R^d=\oplus_{i=1}^{m(x)} H_i(x),
$  given in (ii) is   called the {\it Oseledec  decomposition}%
\index{decomposition!Oseledec $\thicksim$}
\index{Oseledec!$\thicksim$ decomposition}
   at  a  point $x\in Y.$
 If  we  apply   Theorem  \ref{thm_Oseledec} to the shift-transformation\index{shift-transformation!$\thicksim$ of unit-time} $T:=T^1$  acting  
on  the probability  measure  space $(\Omega(X,\Lc), \Ac, \bar\mu),$  we only obtain  
a  much weaker conclusion   that for $\bar\mu$-almost  every path $\omega,$  there  is    an Oseledec  decomposition   at  the   point $x=\omega(0).$ But  this  decomposition depends on each path $\omega \in\Omega_x.$
A  remarkable  point  of  Theorem  \ref{th_main_1}   is that {\bf the  following stronger statement  still holds:}
for each point $x$ in  a leafwise saturated  set $Y\subset X$ of full $\mu$-measure,    we have  a  {\it common} Oseledec decomposition\index{Oseledec!$\thicksim$ decomposition}   at  the point $x$  
for $W_x$-almost every path  $\omega\in \Omega_x.$ We can even show that for each $x\in Y,$ there is  a  set $\Fc_x\subset \Omega_x$ of full $W_x$-measure such that identity (\ref{eq_property_ii}) 
  and identity  (\ref{eq_property_iii}) above  hold  for {\bf all}   $\omega\in \Fc_x.$

$\bullet$ For $x\in Y$ and  $1\leq i\leq  m(x),$ set 
$$
V_i(x):=\bigoplus_{j=i}^{m(x)}H_j(x).
$$ 
The  decreasing sequence of subspaces of $\R^d:$
  $$\{0\}\equiv V_{m+1}(x)\subset V_m(x)\subset \cdots\subset V_1(x)=\R^d$$
is  called the {\it  Lyapunov  filtration}\index{filtration!Lyapunov $\thicksim$}
\index{Lyapunov!$\thicksim$ filtration} associated to $\mathcal A$ at  a given  point $ x\in Y.$
 Theorem  \ref{th_main_1}, combined with  the previous  $\bullet$  and Theorem \ref{th_Lyapunov_filtration}, gives  the following  reinforcement of (\ref{eq_property_ii}) which can  also be  regarded a characterization of the
 Lyapunov exponents:   
\begin{equation}\label{eq_property_ii_new}
\lim\limits_{t\to \infty, t\in \G} {1\over  t}  \log {\| \mathcal{A}(\omega,t)v   \|\over  \| v\|}  =\chi_i(x),\qquad 
 x\in Y,\  v\in V_i(x)\setminus V_{i+1}(x),\ \omega\in \Fc_x.
\end{equation}

 $\bullet$ 
The  identity  $\mathcal{A}(\omega, 1) H_i(\omega(0))= H_i(\omega(1))$ for every continuous leafwise path $\omega$ contained in $Y$ is   known as  the {\it holonomy invariant property}%
\index{holonomy!$\thicksim$ invariant}
\index{holonomy!$\thicksim$ invariant property}
 of  
   the Oseledec decomposition\index{Oseledec!$\thicksim$ decomposition}  at each point of  $Y.$  
 
  $\bullet$    Theorem   \ref{th_main_1} can be formulated  using the  action of $\GL(d,\C)$ on $\C^d$
instead  of the action of   $\GL(d,\R)$ on $\R^d.$ Consequently,  we obtain  an Oseledec  decomposition  
of  $\C^d$  as  a direct sum of complex subspaces 
$\C^d=\oplus_{i=1}^{m(x)} H_i(x)
$    at  every  point $x\in Y.$ 

$\bullet$ A  question naturally  arises  whether  we may  weaken  a little bit the assumptions on the Riemannian  lamination $(X,\Lc,g)$
and on the measure $\mu$ so that  Theorem  \ref{th_main_1} still remains valid.
We will  discuss this  issue just after  its proof   (see Remark  \ref{R:End}).
\end{remark}

We deduce  from Theorem  \ref{th_main_1}
  the  following important  consequence.

\begin{corollary}\label{cor1_th_main_1}
We keep the  hypotheses of Theorem  \ref{th_main_1} and   suppose in addition that  $\mu$ is  ergodic. Then  there are  a leafwise  saturated Borel  set $Y$ of  full $\mu$-measure,   an integer  $m\geq 1,$ and $m$ real numbers  $ \chi_i$
and  $m$ integers $d_i$ for $1\leq i\leq  m$ such that  the  conclusion  of  Theorem  \ref{th_main_1} holds for $Y$ and  that 
  $m(x)=m$  and $ \chi_i(x)=\chi_i$ and   $\dim  H_i(x)=d_i$   for every $x\in Y$ and $1\leq i\leq  m.$  Moreover,
$$
\chi_1= \lim\limits_{t\to \infty, t\in \G} {1\over  t}  \log {\| \mathcal{A}(\omega,t) \|}\quad\text{and}\quad
\chi_m=-\lim\limits_{t\to \infty, t\in \G} {1\over  t}  \log {\| \mathcal{A}^{*-1}(\omega,t) \|}
$$ for $\bar\mu$-almost  every $\omega\in\Omega.$
\end{corollary}

Now  we apply Theorem  \ref{th_main_1} in order to investigate  the  $k$-fold  exterior product 
$\mathcal  A^{\wedge k}$   ($1\leq k\leq  d$) of a  given  cocycle  $\mathcal A.$ 
     Recall  that $\mathcal  A^{\wedge k}$ is   a map defined  $\Omega\times\G$  with values in
     $\GL((\R^d)^{\wedge k})$ (resp.   $\GL((\C^d)^{\wedge k})$), given by  the  formula
$$
\mathcal  A^{\wedge k} :=\mathcal  A \wedge \cdots  \wedge \mathcal  A\quad  \text{($k$ times)}.
$$
We keep   the  hypotheses and  notation  of Corollary
\ref{cor1_th_main_1}. Consider  $d$  functions
 $\chi:\ \Omega\times (\R^d)^k\to\R$  for $1\leq k\leq d,$ given by
\begin{equation}\label{eq_k_dimensional_Lyapunov_exponents} 
\chi(\omega;v_1,\ldots,v_k):=\limsup_{t\to\infty,\ t\in \G}{1\over t} \| \mathcal  A(\omega, t)v_1\wedge \cdots \wedge \mathcal  A(\omega, t)v_k\|
\end{equation}
for  $\omega\in\Omega$ and  $v_1,\ldots,v_k\in \R^d.$ 
\begin{corollary} \label{cor2_th_main_1}
  There exist     
  a   leafwise  saturated  Borel  set  $Y\subset X$ of  full $\mu$-measure  and    $d$  functions
 $\chi:\ Y\times (\R^d)^k\to\R$  for $1\leq k\leq d$ such that
 all the conclusions of  Corollary
\ref{cor1_th_main_1}  hold for $Y$ and that the  following properties also hold:
\\(i)   For each $x\in Y$ there  exists  
a  set $\Fc_x\subset \Omega_x$ of full $W_x$-measure    such that
 for  any vectors  $v_1,\ldots,v_k\in \R^d$  and  any path $\omega\in \Fc_x,$   the right hand side  in  formula
(\ref{eq_k_dimensional_Lyapunov_exponents})   is, in fact,  a true limit. Moreover,
$$
\chi(x;v_1,\ldots,v_k)= 
\chi(\omega;v_1,\ldots,v_k),\qquad  \omega\in\Fc_x.
$$
 The number $\chi(x;v_1,\ldots,v_k)$
is  called  {\rm the $k$-dimensional Lyapunov exponent\index{Lyapunov!$k$-dimensional $\thicksim$ exponent} 
of the  vectors   $v_1,\ldots,v_k$ at  $x.$}
 \\ (ii)  $\chi(x;v)=\chi_i$  for  $v\in \big (\oplus_{j=i}^m  H_j(x)\big )\setminus (\oplus_{j=i+1}^m  H_j(x)\big ).$ 
 \\ (iii) if $v_1,\ldots,v_k\in \bigcup_{i=1}^mH_i(x)$
 and  $v_1\wedge \cdots\wedge v_k\not=0,$  then
 $$
 \chi(x;v_1,\ldots,v_k) =\sum_{i=1}^k\chi_i(x;v_i).
 $$
 (iv)  For   each $x\in Y$  we have  the  following  Oseledec decomposition\index{Oseledec!$\thicksim$ decomposition} 
  for  the cocycle $\mathcal  A^{\wedge k}$ at $x:$
$$\R^d=\oplus_{1\leq i_1,\ldots,i_k\leq  m} H_{i_1}(x)\wedge \cdots \wedge H_{i_k}(x).
$$
 In particular, the  Lyapunov  exponents  of $ \mathcal  A^{\wedge k}$  form the set
$$
\left\lbrace  \chi'_{i_1}+\cdots+\chi'_{i_k} :\ 1\leq i_1<\cdots <i_k\leq  d \right\rbrace ,
$$
where   $\chi'_d\leq  \cdots\leq \chi'_1$  are exactly the  Lyapunov exponents  $\chi_m<  \cdots<\chi_1,$ 
each $\chi_i$ being
counted   with multiplicity $ d_i.$
In particular, 
  we  have  that
 $$
 \lim_{t\to\infty,\ t\in \G}{1\over t} \log \|
\mathcal  A(\omega, t)^{\wedge k}\|=\sum_{i=1}^k  \chi'_i,\qquad 1\leq k\leq d.
$$
\end{corollary}

Another important consequence of Theorem  \ref{th_main_1} is a  characterization of Lyapunov spectrum%
\index{Lyapunov!$\thicksim$ spectrum}
 in the  spirit of   Ledrappier's work in \cite{Ledrappier}.
\index{Ledrappier} 
We will  establish  this  result in  Theorem  \ref{thm_Ledrappier} in  Chapter
\ref{section_Main_Theorems} below   after  developing   necessary materials.

 \section{Second Main Theorem  and  applications}\index{theorem!Second Main $\thicksim$}

 In order to state   the Second Main Theorem\index{theorem!Second Main $\thicksim$} we  need to introduce  some new  notions.
 Let $(X,\Lc)$ be  a Riemannian lamination satisfying the  Standing Hypotheses and  set $\Omega:=\Omega(X,\Lc)$
 and let $\G$ be  either $\N s$ (for some $s>0$)  or $\R^+.$
 \begin{definition} \label{defi_local_expression}
 \rm
 Let $\mathcal A:\ \Omega\times \G\to \GL(d,\R)$   be  a   map that  satisfies the identity, homotopy and
multiplicative laws in Definition  \ref{defi_cocycle}. Fix an arbitrary element $t_0\in \G\setminus \{0\}.$
In any flow  box  
$\Phi_i:  \U_i\to \B_i\times \T_i$ with  $\B_i$ simply connected,  consider the map
$\alpha_i:\  \B_i\times \B_i\times \T_i\to\GL(d,\R)$  defined  by
$$
\alpha_i(x,y,t):=\mathcal A(\omega,t_0),
$$
where  $\omega$ is  any leaf path  such that $\omega(0)=\Phi_i^{-1}(x,t),$ $\omega(1)=\Phi_i^{-1}(y,t)$ 
and  $\omega[0,t_0]$ is  contained in the simply connected  plaque   $\Phi_i^{-1}(\cdot,t).$
We  say that $\alpha_i$  is   the {\it local expression}\index{flow box!local expression of a $\thicksim$} of  $\mathcal A$ on the  flow  box  $\Phi_i.$
 By the homotopy law in  Definition \ref{defi_cocycle}, the local  expression   of  $\mathcal A$ on the  flow  box  $\Phi_i$
 does not depend on the  choice of $t_0\in\G\setminus\{0\}.$ 

Suppose now that $(X,\Lc)$ is  smooth lamination of class $\Cc^k$  $(k\in\N).$ Then 
a  map $\mathcal A$  as   above  is  said to be  {\it $\Cc^{l}$-differentiable cocycle}\index{cocycle!$\Cc^{l}$-differentiable $\thicksim$} (or  equivalently,   {\it $\Cc^{l}$-smooth cocycle})
\index{cocycle!$\Cc^{l}$-smooth $\thicksim$|see{$\Cc^{l}$-differentiable $\thicksim$}} for some $l\in\N$ with $l\leq k$     if, for any   flow  box $\Phi_i$
of a $\Cc^k$-smooth atlas for $(X,\Lc),$  the local expression  of $\mathcal A$  is   $\Cc^{l}$-differentiable. Clearly, this definition does not depend on the choice of
a smooth atlas for $(X,\Lc).$  
\end{definition}

Given a $\Cc^2$-differentiable  cocycle  $\mathcal A,$ we define  two functions  $\bar\delta(\mathcal A),\ \underline\delta(\mathcal A):\ X\to\R$ as well  as four quantities  $\bar\chi_{\max}( \mathcal A ),$
$ \underline\chi_{\max}( \mathcal A ),$    $ \bar\chi_{\min}( \mathcal A ),$
$\underline\chi_{\min}( \mathcal A )$
   as  follows.
  Fix  a point $x\in  X,$ an element $u\in\R^d\setminus \{0\}$  and     a  simply connected plaque $K$   of $(X,\Lc)$ passing through
$x.$
Consider   the  function  $f_{u,x}:\  K\to \R$ defined by
\begin{equation}\label{eq_function_f}
f_{u,x}(y):= \log {\| \mathcal A(\omega,1)u \|\over  \| u\|} ,\qquad  y\in K,\ u\in\R^d\setminus\{0\},
\end{equation}
where  $\omega\in \Omega$ is any path  such that $\omega(0)=x,$ $\omega(1)=y$ 
and that $\omega[0,1]$ is  contained in $K.$  Then define
\begin{equation}\label{eq_formulas_delta}
\bar \delta(\mathcal A)(x):=\sup_{u\in \R^d:\ \|u\|=1} (\Delta f_{u,x})(x)\ \ \text{and}\  \
 \underline \delta(\mathcal A)(x):=\inf_{u\in \R^d:\ \|u\|=1} (\Delta f_{u,x})(x),\\
\end{equation}
where $\Delta$ 
 is, as  usual,   the  Laplacian   
on the leaf $L_x$ induced  by the metric tensor $g$ on  $(X,\Lc).  $
We also define
\begin{equation}\label{eq_formulas_chi}
\begin{split}
\bar\chi_{\max}=\bar\chi_{\max} (\mathcal A)&:=\int_X \bar\delta(\mathcal A)  (x)   d\mu(x),\\
\underline\chi_{\max}=\underline\chi_{\max} (\mathcal A)&:=\int_X \underline\delta(\mathcal A)  (x)   d\mu(x);\\
\underline\chi_{\min}=\underline\chi_{\min}( \mathcal A )&:=-    \bar\chi_{\max}(\mathcal A^{*-1}) ,\\
\bar\chi_{\min}=\bar\chi_{\min}( \mathcal A )&:=-   \underline\chi_{\max}(\mathcal A^{*-1}) .
\end{split}
\end{equation}
Note that our functions  $\bar\delta,$ $\underline\delta$ are  the multi-dimensional generalizations  of the operator $\delta$ introduced by Candel \cite{Candel2}%
\index{Candel!operator $\delta$}
which has been recalled  in Chapter \ref{intro}.

 There is    another equivalent  characterization  of ergodicity for
the class of all    harmonic probability  measures   on  a      compact  $\Cc^2$-smooth lamination $(X,\Lc)$ endowed  with a  transversally  continuous  Riemannian metric.
 This class forms a compact convex cone in the  space of all Radon measures on $X.$
 Proposition  2.6.18  in \cite{CandelConlon2} says that 
the  ergodic  measures  are exactly the   extremal  members of this  cone.

We are in the position to state  our second  main result.

 \begin{theorem} \label{th_main_2}
 Let $(X,\Lc)$ be   a      compact  $\Cc^2$-smooth lamination endowed with   a transversally continuous Riemannian  metric $g.$
 Let $\mu$ be a   harmonic probability measure which is  ergodic.
Let
$\mathcal{A}:\ \Omega\times \R^+ \to  \GL(d,\R)      $ be  a $\Cc^1$-differentiable cocycle.  
Then  there exists  a   leafwise  saturated  Borel  set $Y\subset X$ of  full $\mu$-measure  and a number $m\in\N$  and $m$ integers  $d_1,\ldots,d_m\in \N$  such that
the following properties hold:
\\(i)  For   each $x\in Y$  
 there   exists a  decomposition of $\R^d$  as  a direct sum of $\R$-linear subspaces 
$$\R^d=\oplus_{i=1}^m H_i(x),
$$
 such that $\dim H_i(x)=d_i$ and  $\mathcal{A}(\omega, t) H_i(x)= H_i(\omega(t))$ for all $\omega\in  \Omega_x$ and $t\in \G.$   
Moreover,  $x\mapsto  H_i(x)$ is   a  measurable map from $  Y $ into the Grassmannian of $\R^d.$
Moreover, there   are real numbers 
$$\chi_m<\chi_{m-1}<\cdots
<\chi_2<\chi_1$$
 such that    
$$\lim\limits_{t\to \infty, t\in \R^+} {1\over  t}  \log {\| \mathcal{A}(\omega,t)v   \|\over  \| v\|}  =\chi_i,    
$$
uniformly  on  $v\in H_i(x)\setminus \{0\},$ for  $W_x$-almost every  $\omega\in\Omega_x,$
where  $\|  \cdot\|$  denotes  any norm in $\R^d.$  
The   numbers  $\chi_m<\chi_{m-1}<\cdots
<\chi_2<\chi_1$ are called  the {\rm  Lyapunov exponents} of the cocycle $\mathcal{A}.$\index{Lyapunov!$\thicksim$ exponent} 
\\(ii) For  $S\subset  N:=\{1,\ldots,m\}$ let $H_S(x):=\oplus_{i\in S} H_i(x).$ Then
$$
\lim\limits_{t\to \infty,\ t\in \R^+} {1\over t}  \log\sin {\big |\measuredangle \big (H_S(\omega(t)), H_{N\setminus S} (\omega(t))\big ) \big |}=0
$$
for  $W_x$-almost every  $\omega\in\Omega_x.$
\\(iii) If, moreover, the  cocycle $\mathcal{A}     $ is  $\Cc^2$-differentiable, then  following  inequalities hold $$
\underline\chi_{\max}\leq \chi_1\leq \bar\chi_{\max} \quad\text{and}\quad
\underline\chi_{\min} \leq \chi_m\leq \bar\chi_{\min} .$$ 
\end{theorem} 
  We leave it to the  interested  reader  to reformulate  Theorem  \ref{th_main_2} in the  case  when $\mathcal A$ takes values in $\GL(d,\C).$
  
  Theorem  \ref{th_main_2} generalizes  Theorem \ref{thm_Candel}
\index{Candel!$\thicksim$'s theorem}\index{theorem!Candel's $\thicksim$}
 to the higher dimensions.    
 On the  other hand,  assertion (iii) of Theorem \ref{th_main_2}, combined  with
   Corollary  \ref{cor2_th_main_1}, implies  effective  integral  estimates for  Lyapunov exponents of
    a $\Cc^{2}$-differentiable  cocycle.
    \begin{corollary}\label{cor_formulas_Lyapunov_exponents}
Let $(X,\Lc)$ be   a      compact  $\Cc^2$-smooth lamination endowed with   a transversally continuous Riemannian  metric $g.$
 Let $\mu$ be a   harmonic probability measure which is  ergodic.
Let
$\mathcal{A}:\ \Omega\times \R^+ \to  \GL(d,\R)      $ be  a $\Cc^2$-differentiable cocycle.
Let
  $\chi'_d\leq  \cdots\leq \chi'_1$  be  the  Lyapunov exponents  $\chi_m<  \cdots<\chi_1$ given by Theorem  \ref{th_main_2}, 
each $\chi_i$ being
counted   with multiplicity $d_i.$  Then
$$
\underline\chi_{\max} (\mathcal A^{\wedge k} )\leq \sum_{i=1}^k  \chi'_i \leq \bar\chi_{\max} (\mathcal A^{\wedge k} ),\qquad  1\leq  k\leq d.
$$    
    \end{corollary}
  When  $k=d,$  $\mathcal A^{\wedge d}$  is  a cocycle  of dimension $1,$ and hence  Corollary \ref{cor_formulas_Lyapunov_exponents}
gives that
$$
\sum_{i=1}^d  \chi'_i=\underline\chi_{\max} (\mathcal A^{\wedge k} )=\bar\chi_{\max} (\mathcal A^{\wedge k} ).
  $$  
  So we obtain an effective integral  formula  for the sum  of all Lyapunov exponents counted with multiplicity.

      Now  we  apply  Theorem \ref{th_main_2} to  the holonomy cocycle
of  a  compact  $\Cc^2$ transversally  smooth foliation $(X,\Lc)$ of codimension $d$  in a  Riemannian  manifold $(X,g).$ 
Let  $N(\Lc)$ be   the  normal  bundle of this foliation.
 We  say
that a leaf $L$ is   {\it   holonomy  invariant} if there  exists  a  measurable  decomposition
of $ x\ni L \mapsto N (\Lc)_x$ into the direct sum of $d$ lines  $ H_1(x)\oplus\cdots \oplus H_d(x)$ such that these lines  are invariant
with respect to the  differential  of the  holonomy map
along every closed continuous   path.  More concretely, the last invariance means that for every $x\in L$ and for every  path $\gamma\in\Omega$  with
$\gamma(0)=\gamma(1)=x,$ 
it holds that  $ Dh_{\gamma,1} H_i(x)=H_i(x)$ for all   $i=1,\ldots, d,$  where $Dh_{\gamma,1}$ is  defined in  Section \ref{section_cocycles}. Clearly,  if $L$ has  {\it trivial holonomy}
 (i.e.  $h_{\gamma,1}=\id$ for every $x\in L$ and every path $\gamma$ as  above), then it  is  holonomy invariant. However, the converse  statement is, in general,  not true.
 
 We  get  the  following consequence of Theorem    \ref{th_main_2}.
\begin{corollary}\label{cor_th_main_2}
Let $\mu$ be an ergodic harmonic probability measure  directed by  a  compact $\Cc^2$  transversally  smooth foliation $(X,\Lc)$  of codimension $d$  in a  Riemannian  manifold $(X,g).$ Suppose that  the   holonomy cocycle of   $(X,\Lc)$ admits  $d$ distinct  Lyapunov  exponents with respect to $\mu.$
Then,  for $\mu$-almost  every $x\in X,$ the leaf $L_x$  is  holonomy  invariant. 
\end{corollary}
It is  relevant  to  mention  here  a well-known theorem due  to  G. Hector, D.-B.-A. Epstein, K. Millet
and D.  Tischler \index{Hector}\index{Epstein}\index{Millet}\index{Tischler} (see Theorem 2.3.12  in \cite{CandelConlon1}) which states that   a {\it generic}  leaf  of  a  lamination  has
trivial  holonomy. Recall that  a  subset  of  leaves  of $(X,\Lc)$  is  said to be {\it  generic}
if its union  contains  a  countable  intersection of  open dense leafwise  saturated  sets of $X.$
This  theorem  may be  viewed as  a topological  counterpart of  Corollary \ref{cor_th_main_2}.

We conclude this  section with a   discussion on the perspectives of this Memoir.
In  the companion   paper  
\cite{NguyenVietAnh2} we   
investigate  the   multiplicative cocycles  of  
laminations by hyperbolic Riemann surfaces.
 In particular, we   establish    
Theorem  \ref{th_main_1} and \ref{th_main_2}, and find  geometric  interpretations of Lyapunov  exponents.
We   also  compare  our characteristic exponents  with
other definitions in  the literature. In  some  forthcoming works
we plan  to  investigate  the holonomy  cocycle  of (possibly singular)  foliations by hyperbolic Riemann surfaces.
In this  context, the holonomy of leaves  is closely related   to the  uniformizations of leaves and their Poincar\'e metric.  This subject  has   received  a lot of attention  in the recent years
(see, for example, the  works by Candel \cite{Candel}, Candel-G\'omez Mont \cite{CandelGomezMont}, Dinh-Nguyen-Sibony \cite{DinhNguyenSibony1,DinhNguyenSibony2,DinhNguyenSibony3},
Forn\ae ss-Sibony \cite{FornaessSibony1,FornaessSibony2,FornaessSibony3}, Neto  \cite{Neto} etc).\index{Candel}\index{G\'omez-Mont} \index{Nguyen}\index{Dinh}\index{Sibony}\index{Forn\ae ss}\index{Neto}   We also hope that the  results of this Memoir  may find  applications
in the   dynamics of moduli spaces and in the geometric dynamics of laminations  and  foliations. 
 

 \section{Plan of the proof}


We  use the method of   Brownian motion\index{Brownian motion} which was initiated  by Garnett \cite{Garnett} and    developed further  by Candel \cite{Candel2}.\index{Candel} 
More  precisely, 
 we  want to  prove a Multiplicative  Ergodic Theorem  for  the  shift-transformations  $T^t$ ($t>0$)\index{shift-transformation}
 defined  on the sample-path space $\Omega(X,\Lc)$  such that the Oseledec decomposition\index{Oseledec!$\thicksim$ decomposition} exists
 at  almost  every point $x\in X,$ that is,   such a  decomposition is 
   common for  $W_x$-almost  every  path $\omega\in\Omega_x.$
  Chapter \ref{section_measurability} is  devoted to  some  aspects of the  measure theory   and the ergodic theory 
on sample-path spaces. The  results  of this  chapter will be  used  throughout the  article.  
Most  of these  results are  stated  in this  chapter, but  their proofs are given  in  Appendices  below.
 Chapter \ref{section_leaf}  focuses  the  study of Lyapunov exponents  on  a single leaf. 
In that chapter  we  establish a Lyapunov  filtration, that is, a  weak form of an Oseledec decomposition,\index{Oseledec!$\thicksim$ decomposition}
at almost every point in a  single leaf. The main ingredients  are  an appropriate definition of leafwise Lyapunov
exponents  using the Brownian motion\index{Brownian motion} and the Markov property\index{Markov!$\thicksim$ property} of stochastic processes.   
  
 Following Ruelle's proof\index{Ruelle}  of Oseledec's theorem\index{Oseledec!$\thicksim$ multiplicative ergodic theorem} (see \cite{Ruelle}) we  need  to construct  a  forward filtration  and a  backward 
  filtration at almost every point  so that these filtrations are compatible  with the considered   cocycle.
  Chapter \ref{section_splitting} introduces  the notion of an  invariant  bundle. The   usefulness of this  notion is
   illustrated  by   a splitting theorem  which  reduces the   study of Lyapunov exponents of a  cocycle to  that of  
 splitting  bundles which are easier to handle.
 Using  the  results of the previous  chapters  and appealing  to an  argument  of  Walters\index{Walters} in \cite{Walters} we prove  the
existence  of Lyapunov forward  filtrations 
  in Chapter 
 \ref{section_Lyapunov_filtration}. 
 The  existence  of    Lyapunov backward filtrations is much  harder  to obtain; it will be   established 
  in Chapter
 \ref{section_backward_filtration}  thanks to an involved calculus on heat diffusions.  It will be shown in Chapter
\ref{section_Main_Theorems} that  the  intersection of these  two  filtrations   forms  the Oseledec decomposition.\index{Oseledec!$\thicksim$ decomposition}  
To do this  we  develop  a  new technique of constructing very weakly  harmonic measures on cylinder laminations and  a  new technique of
splitting  invariant subbundles.    This  approach is  inspired  by the  somehow similar device of Ledrappier\index{Ledrappier} \cite{Ledrappier} and
Walters\index{Walters} \cite{Walters} in the context of  discrete dynamics.  
The  proofs of the main results  as  well  as  their  corollaries are also presented in this  chapter.

%% file: chapter4.tex
 
\chapter{Preparatory results}
\label{section_measurability}

The first part of this chapter  deals  with  measurability questions  that   arise  in  the  study of
a Riemannian  lamination $(X,\Lc,g)$ satisfying the  Standing Hypotheses.
In particular,   we give a  sufficient  and simple  criterion  for  multiplicative  cocycles. The remainder of the  chapter discusses the Markov property\index{Markov!$\thicksim$ property} of the Brownian motion.\index{Brownian motion}

Before going further we  fix  several  standard notion and  terminology on  Measure Theory  which will be  used throughout
this  Memoir (see,  for example, the book  by  Dudley \index{Dudley}\cite{Dudley} and the lecture notes by  Castaing and Valadier\index{Castaing}\index{Valadier} \cite{CastaingValadier} for more details).
A positive   measure space\index{measure!positive $\thicksim$}\index{space!measure $\thicksim$}  $(S,\Sc,\nu)$ is  said to be  {\it finite}\index{measure!finite $\thicksim$} (resp. {\it $\sigma$-finite}\index{measure!$\sigma$-finite $\thicksim$}) if
$\nu(S)<\infty$  (resp. if there exists a sequence $(S_n)_{n=1}^\infty\subset\Sc$ such that
$\nu(S_n)<\infty$ and  $S=\bigcup_{n=1}^\infty S_n$).
  
Let $(S,\Sc,\nu)$ be  a $\sigma$-finite positive measure space.
A subset $N\subset S$ is  said to be {\it $\nu$-negligible}\index{set!negligible $\thicksim$} if there exists $A\in\Sc$ such that
$N\subset A$ and $\nu(A)=0.$ 
So  the notion of $\nu$-negligible sets  is more general than  the notion of   sets of null $\nu$-measure presented in  Definition \ref{defi_null_full_measure}.
The {\it $\nu$-completion}\index{completion!completion of a $\sigma$-algebra w.r.t. a measure}
\index{algebra!$\sigma$-algebra!completion of a $\sigma$-algebra w.r.t. a measure|see{completion}}   of $\Sc$  is the $\sigma$-algebra generated by $\Sc$
and  the $\nu$-negligible sets, it is  denoted by $\Sc_\nu$. 
The  elements of  $\Sc_\nu$ are said  {\it $\nu$-measurable}.\index{measurable!$\thicksim$ w.r.t. a measure}\index{measure!measurable set w.r.t. a $\thicksim$}
The measure $\nu$ admits a unique  extension (still denoted by $\nu$) to  $\Sc_\nu,$
and the measure space  $(S,\Sc_\nu,\nu)$ is said to be  the {\it completion } of  $(S,\Sc,\nu).$\index{measure!completion of a $\thicksim$ space}  
The  measure space $(S,\Sc,\nu)$ is  said  to be  {\it complete} if $\Sc_\nu=\Sc.$\index{measure!complete $\thicksim$ space}
When $S$ is a topological space, $\Bc(S)$ denotes  as  usual the  $\sigma$-algebra of Borel sets  of $S.$\index{Borel!$\thicksim$ $\sigma$-algebra}\index{algebra!$\sigma$-algebra!Borel $\thicksim$}
 
  Let  $(T,\Tc)$ and $(S,\Sc)$  be  two measurable  spaces.
  A function $f:\ T\to S$  is  said to be {\it  measurable} 
\index{measurable!$\thicksim$ function (or map)}  
 if  $f^{-1}(A)\in \Tc$ is    for every $A\in \Sc.$
 $f$ is  said to be  {\it bi-measurable} if $f$ is invertible with its inverse $f^{-1}$  and if both $f$ and $f^{-1}$ are measurable.
In particular, when $T$ and $S$ are topological spaces and $\Tc:=\Bc(T)$ and $\Sc:=\Bc(S),$   a measurable (resp. bi-measurable) function is also called
{\it  Borel measurable}\index{Borel!$\thicksim$ measurable function (or map)}\index{measurable!Borel $\thicksim$ function (or map)} (resp. 
 {\it Borel bi-measurable}).\index{Borel!$\thicksim$ bi-measurable function (or map)}\index{bi-measurable!Borel $\thicksim$ function (or map)}\index{measurable!Borel bi-$\thicksim$ function (or map)}
\index{map!Borel bi-measurable $\thicksim$}  
  Suppose in addition that 
    $(T,\Tc,\mu)$  is a   positive $\sigma$-finite measure space. Then
a function $f:\ T\to S$  is  said to be {\it $\mu$-measurable}\index{measurable!$\thicksim$ function w.r.t a measure}   
 if  $f^{-1}(A)$ is  $\mu$-measurable  for every $A\in \Sc.$
 
 \section{Measurability  issue}
 \label{subsection_measurability_issue}
 Let  $\pi:\ (\widetilde X,\widetilde\Lc)\to (X,\Lc)$ be  the  covering lamination projection.
 A  set $A\subset X$ is  said to be a {\it cylinder image}\index{cylinder!$\thicksim$ image} if  $A=\pi\circ \tilde A$ for  some cylinder set $\tilde A\subset \widetilde\Omega:=\Omega(\widetilde X,\widetilde\Lc),$ 
 Recall  from Definition  \ref{defi_algebras_Ac}  that  the  $\sigma$-algebra  $\Ac$ (resp. $\widetilde\Ac$) on $\Omega:=\Omega(X,\Lc)$  is  generated  by  all cylinder  images  (resp. by all cylinder  sets) and that   for  a  point $x\in X,$ let $\Ac_x$  be  the restriction of   $\Ac$
   on 
 $\Omega_x.$
 Let $\mu$ be a  positive $\sigma$-finite Borel measure  
on $X,$ and $\bar\mu$  the Wiener measure on $(\Omega, \Ac)$ with initial  distribution $\mu$   given by formula (\ref{eq_formula_bar_mu}).

Now  we  state  the first main result of this  chapter.
   \begin{proposition}\label{prop_algebras}
(i) For  every $x\in X$ 
and for every $A\in \Ac_x,$ there exists a decreasing sequence
$(A_n),$  each $A_n$ being  a  countable union of    mutually disjoint cylinder  images   such that $A\subset A_n$ and  that   $W_x(A_n\setminus A)\to 0$ as $n\to\infty.$  \\
(ii)   Suppose in addition  that $X=\widetilde X.$ So  cylinder images coincide with cylinder sets, and hence
$\widetilde\Ac=\Ac.$  Then for every $A\in \Ac,$ there exists a decreasing  sequence
$(A_n),$  each $A_n$ being  a  countable union of mutually disjoint  cylinder sets such that $A\subset A_n$ and  that   $\bar\mu(A_n\setminus A)\to 0$ as $n\to\infty.$  
\end{proposition}
Proposition  \ref{prop_algebras} will play an important role  in the  sequel.
Since the proof of this proposition  is  somehow  involved and  technical,  we  postpone it to Appendix 
\ref{subsection_algebra_on_a_leaf} and Appendix  \ref{subsection_algebra_on_a_lamination}  below for the sake of clarity.  
 Note, however, that Proposition \ref{prop_algebras},   together with 
Theorem \ref{prop_Wiener_measure} and Theorem \ref{thm_Wiener_measure_measurable}, give  fundamental properties  of the measures $W_x$ and $\bar\mu.$

\begin{proposition}\label{prop_measurability_W_x}
Let $S$ be a topological space. 
 For   any  measurable set $F$ of the measurable space  $(\Omega\times S,  \Ac\otimes \Bc(S)),$ let $\Phi(F)$ be
 the function $$ X\times S\ni(x,s)\mapsto  W_x(\{\omega\in\Omega_x:\ (\omega,s)\in F\})\in[0,1].   $$
Then $\Phi(F)$  is
   measurable.
 \end{proposition}
 \begin{proof}
We  argue as in the proof of 
  Proposition \ref{prop_integral_dependance_measurably_on_parameter} by replacing the  integral  with the family of Wiener  measures. 
 
Let $\mathfrak{A}$ be the family of   all sets  $A=\cup_{i\in I} \Omega_i\times S_i,$ where $\Omega_i\in \Ac$ and $S_i\in \Bc(S),$
and the index set $I$ is  finite. Note that
$\mathfrak{A}$ is an algebra on $\Omega\times S$ which generates the  $\sigma$-algebra $\Ac\otimes\Bc(S).$
Moreover, each   such set $A$ can be  expressed  as  a disjoint finite union
$A=\sqcup_{i\in I} \Omega_i\times S_i.$
 Using  the above  expression  for  such a set $A,$  we infer that 
 $$
  \Phi(A)(x,s)=\sum_{i\in I} W_x(\Omega_i)\otextbf_{S_i}(s),\qquad  (x,s)\in X\times S.
  $$
 On the other hand,  by  Theorem     \ref{thm_Wiener_measure_measurable} (i),  $X\ni x\mapsto W_x(\Omega_i)$ is  Borel measurable.
Consequently, $\Phi(A)$ is  measurable for all $A\in \mathfrak{A}.$

Let $\mathcal A$ be  the  family  of  all sets $A\subset  \Omega\times S$ such that   
$\Phi(A)$ is    measurable.
 The  previous paragraph shows that $ \mathfrak{A}\subset\mathcal A.$

 Next, suppose that  $(A_n)_{n=1}^\infty\subset \mathcal A$  
and that   either $A_n\searrow A$ or $A_n\nearrow A.$    
By Lebesgue dominated  convergence\index{Lebesgue!$\thicksim$ dominated convergence theorem}\index{theorem!Lebesgue dominated convergence $\thicksim$}, we get that either $\Phi(A_n)\searrow \Phi(F)$ or $\Phi(A_n)\nearrow \Phi(A).$ So $\Phi(A)$ is  also measurable.
Hence, $A\in \mathcal A.$
Consequently,  by Proposition \ref{prop_criterion_sigma_algebra},
$\Ac\otimes\Bc(S)\subset\mathcal A.$
In particular, $\Phi(A)$ is  well-defined and  measurable  for  each $A\in\Ac\otimes\Bc(S).$
This completes the proof.
\end{proof}

For every $x\in X$  let  $L:=L_x$ be  the  leaf passing through $x,$ or more generally  let $(L,g)$ be a  complete  Riemannian manifold
of bounded  geometry. Recall from Chapter  \ref{section_background} that $\Omega(L)$ (resp.   $\Omega_x$) is the space of
continuous paths $\omega: \  [0,\infty)\to L$ (resp. the  subspace of  $\Omega(L)$ consisting of all paths originated  at $x$).
Recall also that  $\Omega(L)$ (resp.  $\Omega_x$) is  endowed  with  the $\sigma$-algebra $\Ac(L)=\Ac(\Omega(L))$ (resp.  $\Ac_x=\Ac(\Omega_x)$). Let   $W_x$ be the probability  Wiener  measure  on $ \Omega_x.$
For any function  $f:\   \Omega(L)\to \R\cup\{\pm\},$ let  
$\esup  f$ denote    the  {\it essential  supremum} of $f$  with  respect  to the Wiener measure $W_x,$ that is,
\nomenclature[a9]{$\esup$ (resp. $\einf$)}{essential  supremum (resp. essential  infimum) w.r.t. the Wiener measure}
\begin{equation}\label{eq_esup}
\esup_{\omega\in \Omega_x} f(\omega):=\inf\limits_{E\in \Ac(\Omega_x),\ W_x(E)=1} \sup\limits_{\omega\in E} f(\omega).
\end{equation}
Similarly, we  define  the  {\it essential  infimum} of $f$  with  respect  to $W_x,$  and  we denote it by $\einf f.$ 

Now let $S$ be a topological space\index{space!topological $\thicksim$} and  consider the measurable space  $(\Omega\times S,  \Ac\otimes \Bc(S)).$ 
For any measurable function  $f:\  \Omega\times S \to[-\infty,\infty],$ define  two functions  $\overline{f}$ and $\underline{f}$
$: X\times  S \to   [-\infty,\infty],$ by
$$
\overline f (x,s):=\esup_{\omega\in\Omega_x} f(\omega,s)\qquad\text{and}\qquad \underline f (x,s):=\einf_{\omega\in\Omega_x} f(\omega),\ (x,s)\in X\times S.
$$
 We are in the  position  to state  the second main result  of this  chapter.
 \begin{proposition}\label{prop_measurability}
 Let $f$ be  a  measurable  function on $\Omega\times S.$ Then   $\overline{f}$ and $\underline{f}$
 are  measurable  on $X\times S.$
 \end{proposition}
 \begin{proof}
 Since  $\overline{(-f)}=-\underline f,$ 
we only need  to  prove   that
 $\overline f$ is  measurable.  The  measurability of  $\overline f$ will follow   if we  can  show that  $\{ (x,s)\in X\times S:\ \overline f\leq r\}$     
is a   measurable  set in $X\times S$ for all  $r\in\R.$ 
  Observe that
  $$
 \left \{(x,s)\in X\times S:\   \overline f\leq r\}=\{(x,s)\in X\times S:\  W_x\big( \{\omega\in \Omega_x:\ (\omega,s)\in A_r\}    \big)    =1 \right\},
  $$
  where   $A_r:=\{ (\omega,s)\in\Omega\times S:\ f(\omega,s)\leq  r \}\in\Ac\otimes \Bc(S).$
  On the other hand,  applying    Proposition \ref{prop_measurability_W_x}    yields that the function $X\times S\ni (x,s)\mapsto W_x\big( \{\omega\in \Omega_x:\ (\omega,s)\in A_r\}    \big)$ is  measurable. Hence, $
  \{(x,s)\in X\times S:\   \overline f\leq r\}$ is  measurable as desired.
  \end{proof}

The following result shows that the  measurable law of a  cocycle is  equivalent to
the measurability of  its local  expressions  on  each flow boxes. The latter condition is  
very easy to  check in practice.

\begin{proposition}\label{prop_cocycle_criterion}
Let $\G$ be either $\N t_0$ (for some $t_0>0$) or $\R^+.$
Let $\mathcal A:\  \Omega(X,\Lc)\times\G\to  \GL(d,\R)$  be  a  map which satisfies the  identity, homotopy and multiplicative  
laws  in Definition \ref{defi_cocycle}.  
Then  $\mathcal A$ is  a  multiplicative  cocycle if and only if 
   the local  expression of  $\mathcal A$  on every flow  box  is  measurable
(see Definition \ref{defi_local_expression} above).
\end{proposition}
 We  postpone the proof of Proposition \ref{prop_cocycle_criterion} to  Appendix  
 \ref{subsection_algebra_on_a_lamination}   
   below.

 \section{Markov property of  Brownian motion}\index{Markov!$\thicksim$ property}
 \label{subsection_Markov_property}
 
  We establish some facts on   Brownian motion\index{Brownian motion} using the Wiener  measures with holonomy. Let $(L,g)$ be  a  complete  Riemannian manifold of bounded  geometry.
Let $F$ be  a  measurable  bounded  function  defined on  $\Omega(L)$
   and  let $x\in L$  be  a  point. 

  The {\it  expectation}
\index{expectation} of $F$ at $x$ is  the  quantity
$$
\Et_x[F]:=\int_{\Omega_x} F(\omega)dW_x(\omega).
$$
\nomenclature[a9a]{$\Et_x[F]$ (resp. $\Et_x[F,\Fc]$)}
{expectation of $F$ at $x$
 (resp. conditional expectation of $F$ w.r.t. $\Fc$ at $x$)}
Let $\Fc$ be a $\sigma$-subalgebra of $\Ac(L)=\Ac(\Omega(L)).$  
The  {\it conditional expectation}\index{expectation!conditional $\thicksim$} of the  function $F$  
with respect  to $\Fc$ at  $x$  is  a  function    
$
\Et_x[F|\Fc]$ defined on $\Omega(L)$ and  measurable with respect to $\Fc$  such that
$$
\int_A \Et_x[F|\Fc](\omega)dW_x(\omega)=\int_A  F(\omega)dW_x(\omega) 
$$
for all $A\in \Fc.$ Note that $
\Et_x[F|\Fc]$ is  unique  in  the ``$W_x$-almost  everywhere" sense.
Therefore, we may restrict ourselves to all $A\in \Fc\cap \Ac_x.$

For $r\geq 0$ let  $\pi_r:\ \Omega(L)\to L$ be the projection  given by $\pi_r(\omega):=\omega(r),$ $\omega\in\Omega(L).$
 For  $s\geq 0$ let
 $\Fc_s$ be the  smallest $\sigma$-algebra  making all the projections $\pi_r:\ \Omega(L)\to L$  with $0\leq r\leq s$  measurable.  For $t\geq 0$ let 
$\Fc_{t^+}:=\cap_{s>t} \Fc_s$  and recall from  (\ref{eq_shift})  the shift-transformation\index{shift-transformation} $T^t:\ \Omega(L)\to\Omega(L).$

   The   Markov property\index{Markov!$\thicksim$ property} says the following  
   \begin{theorem}\label{thm_Markov}
  Let $F$ be  a measurable  bounded  function defined on  $\Omega(L).$ Then  for every $x\in L$ and $t>0$  the following  
equality $
\Et_x[F\circ T^t|\Fc_{t^+}]=\Et_\bullet [F]\circ \pi_t
$ 
holds  $W_x$-almost  everywhere, i.e., 
$$
\Et_x[F\circ T^t|\Fc_{t^+}](\omega)=\Et_{\omega(t)}[F]
$$
holds for $W_x$-almost  $\omega\in \Omega(L).$
\end{theorem}
\begin{proof}
Let $\pi:\  \widetilde L\to L$ be the universal cover of $L.$
Fix  arbitrary points $x\in L,$ and $\tilde x\in \pi^{-1}(x),$ and a number $t>0.$
Consider the function $\tilde F:\  \Omega(\widetilde X,\widetilde L)\to \R$  given by
$$\tilde F(\tilde\omega):= F(\pi\circ \tilde\omega),\qquad \tilde\omega\in \Omega(\widetilde X,\widetilde L). 
$$
For every element $A\in  \Fc_{t^+}\cap \Ac_x,$ 
 let $\tilde A:=\pi^{-1}_{\tilde x}(A).$
The Markov property\index{Markov!$\thicksim$ property}  for  Brownian  motion\index{Brownian motion}  without holonomy (see, for instance, Theorem C.3.4 in \cite{CandelConlon2}),
 applied  to $\widetilde L$ with the reference point $\tilde x,$ yields that
 $$
\int_{\tilde A}
(\tilde F\circ T^t)  dW_{\tilde x}(\tilde\omega)=\int_{\tilde A}\Et_{\tilde\omega(t)}[\tilde F] dW_{\tilde x}(\tilde\omega)
 $$
 Moreover, using the bijective lifting  $\pi^{-1}_{\tilde x}:\ \Omega_x\to  \widetilde\Omega_{\tilde x},$  and applying  Lemma  \ref{lem_change_formula} (iii) below, we  see easily that
$$
\Et_{\tilde\omega(t)}[\tilde F]= \Et_{\omega(t)}[ F],\qquad  \tilde \omega:=\pi^{-1}_{\tilde x}(\omega),\ \omega\in  \Omega_x.
$$ Another  application of this  lemma  to  $F\circ T^t$ yields that
$$ \int_{\tilde A}
(\tilde F\circ T^t) dW_{\tilde x}(\tilde\omega)=\int_A
 (F\circ T^t)  dW_x(\omega).$$
 Combining  the last three identities,  we infer that,
for every element $A\in  \Fc_{t^+}\cap \Ac_x,$  
$$
\int_A (F\circ T^t)  dW_x(\omega) =\int_A  \Et_{\omega(t)}[ F]   dW_x(\omega),
$$
which  completes the proof.
\end{proof}
 It is  worthy noting that the continuity of the sample  paths in $\Omega(L)$ plays the crucial role in the proof
 of   Theorem \ref{thm_Markov}. 
  As an important consequence of   Markov property\index{Markov!$\thicksim$ property}, the following  result     relates the   ergodicity 
of (resp. very weakly)  harmonic probability measures defined on  a (resp.  continuous-like) continuous  lamination $(X,\Lc)$
to     that 
 of the corresponding  extended measures  on $\Omega:=\Omega(X,\Lc).$   The shift-transformations $T^t,$ for $t\in\R^+$,\index{shift-transformation}
 are defined in (\ref{eq_shift}), and  we write for short $T$ the shift-transformation of unit-time\index{shift-transformation!$\thicksim$ of unit-time}  instead of $T^1.$ 
\begin{theorem}\label{thm_ergodic_measures}  
1)   
If   $\mu$ is   a  very weakly harmonic probability measure on  a Riemannian continuous-like lamination $(X,\Lc,g),$ then   $\bar\mu$    is ergodic for $T$ acting on $(\Omega,\Ac)$ if and only if
$\mu$ is ergodic. 
\\
2)
Let  $\mu$ be a probability measure which is  weakly harmonic  on  a Riemannian  (continuous) lamination    $(X,\Lc,g).$ 
 Then $\mu$ is  ergodic if and only if    $\bar\mu$   is  ergodic for some (equivalently, for all) shift-transformation\index{shift-transformation}
$T^t$  with $t\in \R^+\setminus \{0\}.$   
\end{theorem} 
We postpone  the proof  of  Theorem \ref{thm_ergodic_measures} to    Appendix \ref{subsection_Ergodicity}    below. Note however that
the analogue  version of  this theorem in  the case of the $\sigma$-algebra $\widetilde \Ac:= \widetilde \Ac(\Omega)$ has been outlined  in  Theorem  3 in  \cite{Garnett}.

%% file: chapter5.tex
 
 \chapter{Leafwise Lyapunov exponents}
 
 \label{section_leaf}
  
  We first   introduce   an alternative  definition  of Lyapunov exponents in the  discrete version of  the First Main Theorem,\index{theorem!First Main $\thicksim$}  and  then study  this  notion  on  a  fixed  leaf of a lamination.
 This  approach  permits us to apply  the Brownian motion\index{Brownian motion} theory
  more efficiently. Consequently, we   obtain  important invariant  properties of  Lyapunov exponents. 

In this  chapter,  $(X,\Lc, g)$  is  a  ($n_0$-dimensional) Riemannian lamination satisfying  Hypothesis (H1)
and $\mathcal A:\ \Omega(X,\Lc)\times \G\to \GL(d,\R)$  is a  cocycle with   $\G:=\N t_0$ for some $t_0>0.$  
 Suppose without loss of generality that $t_0=1,$ that is,  $\G=\N.$
 Consider  the function
 $\chi:\  X\times \R^d\to \R\cup\{\pm\infty\}$  defined by
 \begin{equation}\label{eq_functions_chi}
 \chi(x,v):=\esup_{\omega\in \Omega_x} \chi(\omega,v),\qquad  (x,v)\in X\times \R^d,
\end{equation}
where, for  each fixed $(x,v)\in X\times \R^d,$ the operator $\esup_{\omega\in \Omega_x}$ has been defined  in (\ref{eq_esup}) and
 \begin{equation}\label{eq_functions_chi_new}
 \chi(\omega,v):=\limsup_{n\to\infty}
{1\over n} 
 \log \| \mathcal{A}(\omega, n)  v   \|  ,\qquad  \omega\in \Omega_x.
 \end{equation}
 The following elementary lemma  will be   useful.
 \begin{lemma}\label{lem_esup}
  Let $(L,g)$ be  a complete Riemannian manifold of bounded geometry    and    $f:\ \Omega( L)\to \R\cup\{\pm\}$  a measurable  function. Then, for every $x\in L,$  there  exists  a  set $E\in \Ac(\Omega_x)$ of full  $W_x$-measure such that
 $$
 \esup_{\omega\in \Omega_x} f(\omega)=\sup_{\omega\in  E} f(\omega).
 $$
 In particular, for every  set  $Z\subset \Omega_x$ of null $W_x$-measure,
$$
\sup_{\omega\in  E} f(\omega)=\sup_{\omega\in  E\setminus Z} f(\omega)
$$ 
 \end{lemma}
 \begin{proof}
 For every $n\in\N\setminus \{0\}$ let $E_n\in\Ac(\Omega_x)$  of full $W_x$-measure such that  
 $$
 \sup_{\omega\in  E_n} f(\omega)\leq  \esup_{\omega\in \Omega_x} f(\omega)+{1\over n}.
 $$
 Setting  $E:=\cap_{n\geq 1} E_n,$ we  see that  $W_x(E)=1$ and  $
\sup_{\omega\in  E} f(\omega)\leq  
 \esup_{\omega\in \Omega_x} f(\omega).$ On the  other hand, the  inverse  inequality  also holds  by the definition of 
 $\esup.$ Hence, the first equality of the lemma follows. The  second  equality  follows  by combining the  first  one  with the  equality  $W_x(  E\setminus Z)=1.$
  \end{proof}
 The  fundamental properties of $\chi$  are given below.
 \begin{proposition}\label{prop_chi}
 (i)  $\chi$ is  a measurable function.
  \\ (ii) $\chi(x,0)=-\infty$   and 
    $\chi(x,v)=\chi(x,\lambda v)$  for  $x\in X,$   $v\in\R^d,$ $\lambda\in\R\setminus \{0\}.$
   So we   can define  a  function, still  denoted  by $\chi,$  defined on $X\times\P(\R^d)$ by
$$ \chi(x,[v]):=\chi(x,v),\qquad    x\in X,\ v\in\R^d\setminus \{0\},$$
where $[\cdot]:\  \R^d \setminus \{0\}\to\P(\R^d)$ which maps  $v\mapsto  [v]$ is  the  canonical projection.\\
(iii) $\chi(x,v_1+v_2)\leq  \max\{ \chi(x,v_1), \chi(x,v_2)\},  $ $x\in X,$    $v_1,v_2\in\R^d.$
\\ (iv)  For all $x\in X$ and $t\in\R\cup\{\pm\}$ the  set
$$ V(x,t):= \{ v\in\R^d:\    \chi(x,v)\leq  t \}$$
is  a linear  subspace of $\R^d.$ Moreover,    $s\leq t$
implies $V(x,s)\subset V(x,t).$ 
\\ (v)  For every $x\in X,$  $\chi(x,\cdot):\  \R^d\to\R\cup\{-\infty\}$ takes only finite  $m(x)$  different  values   
 $$\chi_{m(x)}(x)<\chi_{m(x)-1}(x)<\cdots <\chi_2(x)<\chi_1(x).$$
 \\ (vi) If, for  $x\in X,$  we  define  $V_i(x)$ to be   $V(x,\chi_i(x))$ for $1\leq  i\leq m(x),$
then 
$$
\{0\}\equiv  V_{m(x)+1}(x)\subset  V_{m(x)}(x)\subset\cdots\subset V_2(x)\subset  V_1(x)\equiv\R^d
$$   
and
$$
v\in V_i(x)\setminus V_{i+1}(x)\Leftrightarrow \sup_{\omega\in E_x}\limsup_{n\to\infty} {1\over n} \log \| \mathcal{A}(\omega, n)   v   \| =\chi_i(x)
$$
for some   set $E_x\in \Ac(\Omega_x)$ of full $W_x$-measure, $E_x$ depends only on $x$  (but it does  not depend on  $v\in \R^d$).
 \end{proposition}
  \begin{proof}
 Since  we  know by the measurable law in Definition  \ref{defi_cocycle}  that  $\mathcal A(\cdot,n)$ is measurable on $\Omega(X,\Lc)$ for every $n\in\N,$
the  function $\Omega(X,\Lc)\times \R^d\ni (\omega,v)\mapsto \chi(\omega,v)$ is also measurable. Consequently,   assertion  (i) follows from Proposition \ref{prop_measurability}. The proof of  (ii) is clear  since  $\mathcal A(\omega,n)\in \GL(d,\R).$

Now we turn to assertion  (iii). By Lemma  \ref{lem_esup},  pick sets $E_0, E_1,E_2\in \Ac(\Omega_x)$  of full   $W_x$-measure such that
$$
\chi (x,v_i)=\sup_{\omega\in E_i} \limsup_{n\to\infty} {1\over n} \log \| \mathcal{A}(\omega, n)   v_i   \|,
$$
where we put $v_0:=v_1+v_2.$ Now  setting  $E:= E_0\cap E_1\cap E_2,$ $E$ is  an element of $\Ac(\Omega_x)$   of full   $W_x$-measure.

The  following elementary result is  needed.
\begin{lemma}  \label{lem_elementary_inequality}
If  $a_n,b_n\geq 0$ for $n\geq 1$ then
$$
\limsup_{n\to\infty}{1\over n}  \log{(a_n+b_n)}=  \max\left\lbrace  \limsup_{n\to\infty}{1\over n}  \log{a_n},  \limsup_{n\to\infty}{1\over n}  \log{b_n} \right\rbrace,
$$
and
$$
\liminf_{n\to\infty}{1\over n}  \log{(a_n+b_n)}\geq  \max\left\lbrace  \liminf_{n\to\infty}{1\over n}  \log{a_n},  \liminf_{n\to\infty}{1\over n}  \log{b_n} \right\rbrace.
$$
\end{lemma}
 Using the  equality of the   above  lemma it follows that 
\begin{multline*}
\limsup_{n\to\infty} {1\over n} \log \| \mathcal{A}(\omega, n)  ( v_1+v_2)  \|\\
\leq
 \max\left\lbrace\limsup_{n\to\infty} {1\over n} \log \| \mathcal{A}(\omega, n)   v_1  \|,
\limsup_{n\to\infty} {1\over n} \log \| \mathcal{A}(\omega, n)   v_2   \|
\right\rbrace.
\end{multline*}
Taking the  supremum of both sides of the last inequality over all $\omega\in E$ and  using Lemma    \ref{lem_esup},
we obtain  $\chi(x,v_1+v_2)\leq  \max\{ \chi(x,v_1), \chi(x,v_2)\},  $
which proves (iii).

Each  $V(x,t)$ is   a linear subspace of $\R^d$ by  (ii) and (iii). The inclusion   
  $V(x,s)\subset V(x,t)$  for $s\leq t$ is  also clear. Hence, (iv) follows.

 Fix  $x\in X.$  Since   $s<t$  implies   $V(x,s)\subset V(x,t)$ and  hence  $\dim V(x,s)\leq  \dim V(x,t),$
we can enumerate  all the values of $t:$ $\chi_{m(x)}(x)<\chi_{m(x)-1}(x)<\cdots <\chi_2(x)<\chi_1(x),$
where  $t\mapsto  \dim V(x,t)$ changes.  Therefore,
   $\chi (x,\cdot)$  can only take the values  $\chi_{m(x)}(x),\chi_{m(x)-1}(x),\ldots ,\chi_2(x),\chi_1(x).$
    This  proves (v).
 
 To prove (vi)  it suffices to  find a  set $E_x\in \Ac(\Omega_x)$ with the required properties.
 To do this   fix  a point $x\in X$ and a  basis  $\{v_1,\ldots,v_d\}$ of $\R^d$ such that
  $\{v_1,\ldots,v_{k_j}\}$ is a  basis of $V_{m(x)-j+1}(x)$ for $j=1,\ldots,m(x),$ where  $k_j:= \dim V_{m(x)-j+1}(x).$
  By Lemma  \ref{lem_esup},  pick  a  set $E_j\in \Ac(\Omega_x)$  of full   $W_x$-measure such that
$$
\chi (x,v_i)=\sup_{\omega\in E_i} \limsup_{n\to\infty} {1\over n} \log \| \mathcal{A}(\omega, n)   v_i   \|,
$$
 Now  set  $E_x:= \bigcap_{j=1}^d E_j.$ The desired conclusion  follows from the  above  equalities for $\chi(x,v_i)$
and from  (ii), (iii) and  Lemma  \ref{lem_esup}.
  \end{proof}

  Now  let $L$ be  a  fixed  leaf  of a lamination $(X,\Lc)$   and $\Vol$  the Lebesgue  measure  induced by its Riemannian
  metric.  Let $T:=T^1:\ \Omega(L)\to\Omega(L)$  be given in (\ref{eq_shift}).
Fix a point $x\in L.$ 
\begin{proposition} \label{prop_Markov}
(i) For any measurable  set $A\subset \Omega(L),$
$$
W_x(A)  \leq \int_{y\in L} p(x,y,1) W_y(T(A)) \Vol(y).
$$
  If,  moreover, $T^{-1}(T(A))=A,$ then  the above inequality becomes an equality.
  \\
  (ii)
Given  a set  $A\subset \Omega_x$ of full  $W_x $-measure, then  for $\Vol$-almost  every $y\in L,$  $T(A)$ is 
of full $W_y $-measure.
\end{proposition}
\begin{proof}
Consider two bounded  measurable functions $F,G:\  \Omega(L)\to \R$ defined by
$$
F(\omega):=
\begin{cases}
1, &  \omega\in A;\\
0,  & \omega\not\in A.
\end{cases}
$$
and
$$
G(\omega):=
\begin{cases}
1, &  \omega\in T( A);\\
0,  & \omega\not\in T( A).
\end{cases}
$$
It is  clear that $F\leq  G\circ T.$ Moreover, if $T^{-1}(T(A))=A$ then  $F= G\circ T.$
  Consequently, we have that
$$
W_x(A)=\Et_x[F]=\Et_x[ \Et_x[F|\Fc_{1^+}]]\leq \Et_x[ \Et_x[G\circ T|\Fc_{1^+}]] 
 ,
$$
where the second equality holds by the  projection rules of  the  expectation operation (see Theorem C.1.6 in \cite{CandelConlon2}),
the inequality  follows from  the  estimate $F\leq  G\circ T.$
By the  Markov property\index{Markov!$\thicksim$ property} (see Theorem  \ref{thm_Markov}),  we get  that
$$
\Et_x[G\circ T|\Fc_{1^+}]= \Et_\bullet [G]\circ \pi_1.
$$
Inserting  this  into the previous inequalities  we obtain that
$$
W_x(A)\leq \Et_x[ \Et_x[G\circ T|\Fc_{1^+}]] =\Et_x[ \Et_\bullet [G]\circ \pi_1].
$$
This, combined with
$$
\Et_y[G]=\int_{\omega\in T( A)} dW_y(\omega)=W_y(T (A)),\qquad y\in L, 
$$
 implies that 
 $$
W_x(A) \leq \Et_x[ \Et_\bullet [G]\circ \pi_1]=\int_{y\in L} p(x,y,1) W_y(T( A))\Vol(y),
 $$
 which  proves  the  first  assertion.
 
 The  second  assertion follows  by combining  the  first one, and
   the identity  $\int_{y\in L} p(x,y,1) \Vol(y)=1$  (see \cite{Chavel}),  and the inequality   $0\leq  W_y(T( A))\leq 1.$  
 \end{proof}
Here  is the main result of this  chapter.
\begin{proposition}\label{prop_leafwise_Oseledec}
Let $(X,\Lc)$ be a  Riemannian  lamination  satisfying  Hypothesis (H1).
Let $L$ be a leaf of $(X,\Lc)$ and $\mathcal A$  a cocycle on $(X,\Lc).$ 
For  every $x\in X,$ let  $m(x), \ \chi_j(x), \ V_j(x)$  be given by Proposition 
\ref{prop_chi}.
Then  there  exist  a  number $m\in \N$ and $m$  integers $1\leq d_m<d_{m-1}<\cdots<d_1=d$ and $m$ real numbers
$\chi_m<\chi_{m-1}<\cdots< \chi_1$  and   a subset $Y\subset L$  with the  following properties:
\\ (i)    $\Vol(L\setminus Y)=0;$
\\(ii) for every $x\in Y$  and  every $1\leq j\leq m,$   we have
$m(x)=m$ and  $\chi_j(x)=\chi_j$ and  $\dim V_j(x)=d_j;$ 
\\(iii)  for  every $x,y\in Y,$ and  every $\omega\in \Omega_x$ such that $\omega(1)=y,$
we have  $\mathcal A(\omega,1)V_i(x)=V_i(y)$   with $1\leq i\leq m.$    
\end{proposition}
\begin{remark}
The  real numbers $\chi_m<\chi_{m-1}<\cdots< \chi_1$ are called the {\it leafwise Lyapunov exponents} associated to the cocycle $\mathcal A$
on the leaf $L.$
\index{leafwise!$\thicksim$ Lyapunov exponent}
 The  decreasing sequence of subspaces of $\R^d:$
  $$\{0\}\equiv V_{m+1}(x)\subset V_m(x)\subset \cdots\subset V_1(x)=\R^d,\qquad x\in L,$$
is  called the {\it  leafwise  Lyapunov  filtration}\index{leafwise!$\thicksim$ Lyapunov filtration}\index{filtration!leafwise Lyapunov $\thicksim$} associated to $\mathcal A$ at  a given  point $ x\in L.$
\end{remark}

Prior to the  proof  it is  worthy noting that   that  property (iii) is  a  primitive  version of   the  holonomy invariance  of
the Oseledec  decomposition\index{Oseledec!$\thicksim$ decomposition}.
\begin{proof}
First  we  prove   that there is  a constant $\chi_1$ such that  $\chi_1(x)=\chi_1$ for $\Vol$-almost every $x\in L.$
Let $\{e_1,\ldots,e_d\}$ be the canonical basis of $\R^d.$ Since we know that $\chi:\ L\times\R^d\to\R$ is  measurable, it follows that
$\chi_1(x)=\sup_{u\in \R^d} \chi(x,u)=\sup_{1\leq j\leq d} \chi(x,e_j)$ is  also measurable. Let
 $$\chi_1:=\esup_{x\in L}\chi_1(x):=\inf_{E\subset L:\ \Vol(L\setminus E)=0}  \sup_{x\in E} \chi_1(x)  .  $$
Fix   a  point  $x\in L.$ By Lemma
\ref{lem_esup}, 
  there  exists  a  set $A\in \Ac(\Omega_x)$ of full  $W_x$-measure such that
$$
\chi_1(x)=\max_{1\leq j\leq d} \sup_{\omega\in  A}\chi(\omega,e_j).
$$
On the one hand,  it follows from the  definition that 
\begin{equation}\label{eq_invariance_chi}
\chi(T\omega,\mathcal A(\omega,1)v)=\chi(\omega,v),\qquad  (\omega,v)\in \Omega(X,\Lc)\times \R^d.
\end{equation}
Note  that $\{ \mathcal A(\omega,1)e_j:\  1\leq j\leq d \}$ forms a basis of $\R^d.$ Therefore,
we infer  from Proposition \ref{prop_chi}  (ii) and (iii) that for $y=\omega(1),$
$$
\chi_1(y)=  \max_{1\leq j\leq d}  \chi (y, \mathcal A(\omega,1)e_j). 
$$
 On the other hand,
  by Proposition  \ref{prop_Markov}, for  $\Vol$-almost every $y\in L,$  $T (A)$  is of full $W_y$-measure.
 Consequently, for such  $y$  we   have that  $W_y(A_y)=1,$ where   $A_y:=\{ \omega\in A: \omega(1)=y\}. $ 
Hence,  for such   $y$  we  get that 
  $$
\chi_1(y)\leq \max_{1\leq j\leq d} \sup_{\omega\in A_y}\chi(T\omega, \mathcal A(\omega,1)e_j)= \max_{1\leq j\leq d} \sup_{\omega\in A_y}\chi(\omega,e_j)\leq  \chi_1(x),
$$
where the first inequality follows  from Proposition \ref{prop_chi}  (ii) and (iii), and the  equality holds by (\ref{eq_invariance_chi}).
Hence  $\chi_1(x)\geq \chi_1(y)$ for  $\Vol$-almost every $y\in L.$
So  $\chi_1(x)  \geq  \chi_1.$  On the other hand, by definition,   $\chi_1(x)\leq \chi_1$ for $\Vol$-almost every $x\in L.$
Therefore, there exists a Borel set $Y_1\subset L$  such that  $\Vol(L\setminus Y_1)=0$ and  that $\chi_1(x)= \chi_1$ for  every $x\in Y_1.$

Consider  
$$
\Lambda_2:=\left\lbrace (x,v)\in Y_1\times \R^d:\  \chi(x,v)<\chi_1 \right\rbrace\subset \Leb(Y_1)\times \Bc(\R^d),
$$
where    $\Bc(\R^d)$ denotes, as  usual, the  Borel $\sigma$-algebra, and $ \Leb(Y_1)$  denotes the completion of the Borel $\sigma$-algebra
 of $Y_1$  equipped  with the Lebesgue measure, $Y_1$ being  endowed with the induced  topology from $L.$
Let $\Pi_1:\ Y_1\times \R^d\to Y_1$ be the natural projection, then by Theorem \ref{thm_measurable_projection} below, $\Pi_1(\Lambda_2)\in\Leb(L). $
Also $\Pi_1(\Lambda_2)=\{x\in Y_1:\  m(x)>1\}.$ 
If  $\Vol(\Pi_1(\Lambda_2))=0,$ then  the  proof of the proposition is  complete   with $m=1$ and $Y=Y_1$ and  $V_1=\R^d.$

Suppose  now  that $\Vol(\Pi_1(\Lambda_2))>0.$  For $y\in L,$ let  $V_2(y)$ be the   proper vector subspace $\Lambda_2\cap  \Pi_1^{-1} (y)$ of $\R^d$
with the convention that $V_2(y):=\{0\}$ if  $y\not\in \Pi_1(\Lambda_2).$ 
Fix  a point  $x$ in the  set $ \Pi_1(\Lambda_2).$ So $\dim V_2(x) >0 .$ There are  two cases  to consider.

\noindent {\bf Case  1:} {\it  $L$ is  simply connected.}\index{leaf!simply connected $\thicksim$}

 Fix   a  basis $u_1(x),\ldots,u_k(x)$ of $V_2(x).$
For every $y\in L,$  let $V'(y):=  \mathcal A(\omega,1) V_2(x)$  and  $u_j(y):=   \mathcal A(\omega,1) u_j(x),$   
where  $1\leq  j\leq k$ and  $\omega$ is  any element  of $\Omega_x$ such that $\omega(1)=y.$
The   simple connectivity of $L$  and the homotopy law for $\mathcal A$ ensure  that  this  definition   is  independent of  the  choice of $\omega.$
 Note that 
$$
\chi_2(x)=\sup\{ \chi(x,v):\  (x,v)\in \Lambda_2\}=\max_{1\leq j\leq k} \chi(x,u_j(x)).
$$
 By Lemma  \ref{lem_esup} there is   a  set $A\in \Ac(\Omega_x)$  of full $W_x$-measure such that
$$ \chi_2(x)=   \max_{1\leq j \leq k} \sup_{\omega\in A} \chi(\omega,u_j(x)).  $$
By Proposition    \ref{prop_Markov}, for  $\Vol$-almost every $y\in L,$  $T(A)$  is of full $W_y$-measure. Consequently,
using this  and (\ref{eq_invariance_chi}) we infer that,  for  all  such $y,$  
 $$
\sup_{v\in  V'(y)} \chi(y,v )\leq  \max_j \sup_{\omega\in T( A)}\chi(\omega,u_j(y)) =  \max_j \sup_{\omega\in A}\chi(\omega,u_j(x))=   \chi_2(x)<\chi_1.
$$  
Hence,  the  above  inequality    implies that $V'(y)\subset  V_2(y)$ for  $\Vol$-almost every $y\in L.$  
 Since $\mathcal A$  is  with values in $\GL(d,\R),$  we have clearly that $\dim V'(y)=\dim  V_2(x)>0.$
Thus,  $\dim V_2(y)\geq  \dim V'(y)=\dim  V_2(x)>0$ 
  for all  such $y.$  So  all  such $y$ belong to $\Pi_1(\Lambda_2).$ Summarizing what has been done so far, we have shown that  $\Vol (L\setminus \Pi_1(\Lambda_2))=0$
and that for each $x\in  \Pi_1(\Lambda_2),$  $0<\dim V_2(x)
\leq  \dim V_2(y)<d$   for $\Vol$-almost  every  $y\in 
\Pi_1(\Lambda_2).$   So there is   an integer  $d_2<d$  and   a  set $Y_x\subset L$  such that 
$\Vol(L\setminus Y_x)=0$ and   that for every  $y\in Y_x,$ 
  $\dim V_2(y)=d_2$ and $V_2(y)=V'(y).$ 
This, combined   with the previous estimate $\sup_{v\in  V'(y)} \chi(y,v )\leq \chi_2(x),$ implies that
$\chi_2(y)\leq  \chi_2(x)$ for        $y\in Y_x.$
Using that  $\Vol$-almost every $x$ is  contained in   $\Pi_1(\Lambda_2),$        
we may find a Borel set    $Y_2\subset Y_1$ and $\chi_2\in \R\cup\{\pm\infty\}$   such that  $\Vol(L\setminus Y_2)=0$ and
$\chi_2(x)=  \chi_2$ for   every $x\in Y_2.$


\noindent {\bf Case  2:} {\it  $L$ is  not necessarily simply connected.}

 The holonomy problem   arises. More concretely, 
given  two points $x$ and $y\in L$ and two   paths $\omega_1,\ \omega_2\in\Omega(L)$ such that   $\omega_1(0)=\omega_2(0)=x$ and $\omega_1(t_1)=\omega_2(t_2)=y,$ then 
  $ \mathcal A(\omega_1,t_1) V_2(x)$ is not necessarily equal to    $\mathcal A(\omega_2,t_2) V_2(x).$   

Let  $\pi:\  \widetilde L\to L$  be  the universal  cover. 
Fix $x\in L$  and let  $\tilde x\in \widetilde L$ be a lifting of $x.$ Recall from Lemma  \ref{lem_change_formula} (ii) below
that  
$\pi^{-1}_{\tilde x}:\ \Omega_x\to \widetilde\Omega_{\tilde x}$  is a canonical identification of the  two paths  spaces 
which identifies the  respective  Wiener measures  $W_x$ and $W_{\tilde x}$   on them. More precisely,
for $\tilde E\in  \Ac(\Omega_{\tilde x} ),$ we have  that  $E:=\pi(\tilde E)\in \Ac(\Omega_x)$ and
$W_{\tilde x}(\tilde E)=W_x(E).$

We  construct a  cocycle $\widetilde{\mathcal A}$ on $\widetilde L$ as  follows:
\begin{equation}\label{eq_cover_cocycle}
\widetilde{\mathcal A}(\tilde  \omega,t):=\mathcal A(\pi (\tilde  \omega),t),\qquad  t\in\R^+,\ \tilde  \omega\in \Omega( \widetilde L  ).
\end{equation}
For  $\tilde x\in \widetilde L$ we  define  $V_i(\tilde x)$ relative  to the  cocycle   $\widetilde{\mathcal A}$ thanks to   Proposition  \ref{prop_chi}.
Using the   above  canonical  identification  and the  definition of $\widetilde{\mathcal A},$  we see that
$$
\sup_{\tilde\omega\in \tilde  E}\chi(\tilde\omega,v)=\sup_{\omega\in E}\chi(\omega,v)
$$
for  every $\tilde E\in  \Ac(\Omega_{\tilde x} )$ and $v\in \R^d.$
By taking the infimum of the  above  equality over  all $\tilde E$ of full $W_{\tilde x}$-measure, we get that
$  \chi(\tilde x,v)=\chi(x,v).$ Hence,  
\begin{equation}\label{eq_identification_V_2}
 V_2(\tilde x)=V_2(x)=V_2(\pi (\tilde x)).
\end{equation}
 Since the  cocycle $\widetilde{\mathcal A}$ is  defined   on the  simply connected  manifold   $\widetilde L,$  we may apply    Case 1.
 Consequently, there is  a   set $\tilde Y_2\subset \widetilde L$   such that $\Vol (\widetilde L\setminus \tilde Y_2)=0$ and that
 the  assertions (i)--(iii) hold  for $m=2.$
 Now let $x$ and $y$ be two points in $Y_2:=\pi(\tilde Y_2)$  and let  $\omega_1,\ \omega_2\in\Omega(L)$ be two   paths such that   $\omega_1(0)=\omega_2(0)=x$ and $\omega_1(t_1)=\omega_2(t_2)=y.$
 Since  $x\in Y_2,$ we  fix  a  lift  $\tilde x\in  \tilde Y_2$ of $x.$
 Let $\tilde\omega_1:= \pi^{-1}_{\tilde x}(\omega_1),$ $\tilde\omega_2:= \pi^{-1}_{\tilde x}(\omega_2),$
and  $  
 \tilde y_1:=\tilde\omega_1(t_1) ,$  $\tilde y_2:=\tilde\omega_1(t_2).$
We consider two  subcases.

\noindent {\bf Subcase  2a:} {\it  
 Both  $\tilde y_1$ and $\tilde y_2$   belong to $  \tilde Y_2.$}
  
By  assertion  (iii) and (\ref{eq_cover_cocycle}) and  (\ref{eq_identification_V_2}), we get that
 $$ \mathcal A(\omega_1,t_1) V_2(x)= \widetilde{\mathcal A}(\tilde\omega_1,t_1) V_2(\tilde x)=V_2(\tilde y_1)\ \text{and}\
 \mathcal A(\omega_2,t_2) V_2(x)= \widetilde{\mathcal A}(\tilde\omega_2,t_2) V_2(\tilde x)=V_2(\tilde y_2).$$
Since $\pi(\tilde y_1)=\pi(\tilde y_2)=y,$  we obtain, by  (\ref{eq_identification_V_2}) again, that  $V_2(\tilde y_1)=V_2(\tilde y_2)=V_2(y).$
Hence, $ \mathcal A(\omega_1,t_1) V_2(x)=\mathcal A(\omega_2,t_2) V_2(x).$ So  there is  no holonomy problem in this  subcase.

\noindent {\bf Subcase  2b:} {\it  
  Either  $\tilde y_1$ or $\tilde y_2$  is  outside $  \tilde Y_2.$}
  
  Assume  without loss of generality that $t_1=t_2=1.$
Since $\Vol (\widetilde L\setminus \tilde Y_2)=0,$ it follows that
$\Vol(L\setminus Y_2)=0.$ Consequently, 
by   re-parameterizing $\omega_1|_{[0,1]}$ and  $\omega_2|_{[0,1]}$ and by  replacing  $\omega_1|_{[0,1]}$ (resp. $\omega_2|_{[0,1]}$) by    a path of the  same  homotopy class  if necessary
(see the homotopy law in Definition \ref{defi_cocycle}),  we may  choose   
$z\in Y$  close to  $y$ such that     
\\ $\bullet$
$\omega_1(1/2)=\omega_2(1/2)=z$ and  $\tilde z_1:=\tilde \omega_1(1/2)\in   \tilde Y_2,$  $\tilde z_2:=\tilde\omega_2(1/2)\in   \tilde Y_2;$
\\ $\bullet$  $\omega_1|_{[1/2,1]}$ is  homotopic with  $\omega_2|_{[1/2,1]}$ in $L_x.$

By the first  $\bullet$ we may apply Subcase 2a 
  to  $\tilde z_1$ and $\tilde z_2$  in place of $\tilde y_1$ and $\tilde y_2$. Hence, using   (\ref{eq_identification_V_2})
we obtain that 
$$
  \mathcal A(\omega_1,1/2) V_2(x)=\mathcal A(\omega_2,1/2) V_2(x)=V_2(\tilde z_1)=V_2(\tilde z_2)=V_2(z).
$$
On the other hand, the  second $\bullet$  implies that 
$$
  \mathcal A(T^{1/2}\omega_1,1/2) V_2(z)=\mathcal A(T^{1/2}\omega_2,1/2) V_2(z).$$
 Combining the  equalities in the  last two lines and  appealing  to the  multiplicative law  of $\mathcal A,$ we get that  
 \begin{eqnarray*}
 \mathcal A(\omega_1,1) V_2(x)&=&\mathcal A(T^{1/2}\omega_1,1/2) \mathcal A(\omega_1,1/2) V_2(x)\\
&=&
 \mathcal A(T^{1/2}\omega_1,1/2) V_2(z)\\
&=&\mathcal A(T^{1/2}\omega_2,1/2) V_2(z)\\
&=&\mathcal A(T^{1/2}\omega_2,1/2) \mathcal A(\omega_2,1/2) V_2(x)\\
&=&\mathcal A(\omega_2,1) V_2(x).
\end{eqnarray*}
 This  completes    Subcase 2b. Hence,  the proposition  is  proved  for  $m\leq 2.$
 
 Consider  
$$
\Lambda_3:=\left\lbrace (x,v)\in Y_1\times \R^d:\  \chi(x,v)<\chi_2 \right\rbrace\subset \Leb(Y_1)\times \Leb(\R^d).
$$
 Let $V_3(y)$ be the   proper vector subspace $\Lambda_3\cap  \Pi_1^{-1} (y)$ of $\R^d$ for $y\in  \Pi_1(\Lambda_3),$ and let
 $V_3(y):=\{0\}$ otherwise.
 We argue  as  above   and  use  that $\dim V_3(y)<\dim V_2(y)<\dim V_1(y)$   when  $ \Vol(\Pi_1(\Lambda_3))>0  .$ 
Consequently, the proposition  is  proved for $m\leq 3.$  We  continue   this  process. It  will be  finished  after   a  finite $m$  steps.  This  completes the proof.  
\end{proof}

%% file: chapter6.tex


 \chapter{Splitting   subbundles}
 \label{section_splitting}
 
 
 In this chapter  we  are  given  a Riemannian lamination $(X,\Lc,g)$ satisfying the Standing Hypotheses     and
 a very  weakly harmonic  probability measure  $\mu$ directed by $(X,\Lc).$ We also fix   a  number $d\in \N$ and
 let $\G:=\N.$
 \begin{definition}\rm
 A  {\it  measurable  bundle  of rank $k$}\index{bundle!bundle, measurable $\thicksim$} 
is  a  Borel  measurable  map $V:$  $Y\ni x\mapsto V_x$
of $Y$  into  the Grassmannian $\Gr_k(\R^d)$ of vector subspaces of dimension $k$  for some  $k\leq d,$
where  $Y\subset X$ is  a  subset of full $\mu$-measure.
 A   measurable  bundle  $U$ of rank $l:$ $Y\ni x\mapsto U_x$  is  said  to be a {\it  measurable  subbundle of $V$}\index{bundle!sub-$\thicksim$}
 \index{subbundle|see{bundle}}
 if  $U_x\subset V_x,$ $x\in X.$
The {\it  trivial  bundle on $Y$ }is defined  by $Y\ni x\mapsto \R^d,$  and  is  denoted  by 
 $Y\times \R^d.$\index{bundle!trivial $\thicksim$}

For   a  subset  $Y\subset X$ of full $\mu$-measure, let
$$
\Omega(Y):=\left\lbrace \omega\in\Omega(X,\Lc):\ \pi_n\omega\in Y,\ \forall n\in\N  \right\rbrace,
$$
where $\pi_n:\ \Omega(X,\Lc)\to X$ is, as usual, the projection  given by $\pi_n\omega:=\omega(n),$ $\omega\in\Omega(X,\Lc).$  
Given  a  cocycle  $\mathcal A:\ \ \Omega(X,\Lc)\times \N \to \GL(d,\R)$   and  a  subset $ Y\subset X$ of full $\mu$-measure, 
a  measurable subbundle  $ Y \ni x\mapsto V_x$ of $Y\times \R^n$  is  said to be  {\it $\mathcal A$-invariant} if\index{bundle!$\mathcal A$-invariant $\thicksim$} 
$$
\mathcal A(\omega,n)V_{\omega(0)}=V_{\omega(n)},\qquad  \omega\in\Omega( Y).
$$
\end{definition}
 Using formula (\ref{eq_formula_W_x_without_holonomy}) we  see easily that  for   a  subset  $Y\subset X$ of full $\mu$-measure,
 $\Omega(Y)$ is  a  subset of $\Omega(X,\Lc)$  of full  $\bar \mu$-measure.

We may rephrase  Proposition \ref{prop_leafwise_Oseledec} as  follows.

\begin{corollary}\label{cor_leafwise_Oseledec}
 Suppose that $\mu$ is  ergodic.
 Let   $\mathcal A$ be  a cocycle on $(X,\Lc).$ Then  there  exist  a  Borel set  $Y\subset X$ of full $\mu$-measure and  a    number $m\in \N$ and
 $m$ integers  $1\leq d_m<d_{m-1}<\cdots<d_1=d$ and $m$ real numbers
$\chi_m<\chi_{m-1}<\cdots< \chi_1$
     with the  following properties:
\\ (i)  $m(x)=m$ for  every $x\in Y;$
\\ (ii)   the map $Y\ni \mapsto V_i(x)$  is  an $\mathcal A$-invariant  subbundle  of rank $d_i$ of  $Y\times \R^d$  for $1\leq i\leq m;$ 
\\ (iii) for every $x\in Y$  
and   $1\leq i\leq  m,$
$\chi_i(x)=\chi_i.$     
\end{corollary}
\begin{proof} By Proposition \ref{prop_leafwise_Oseledec}, for each leaf $L$  of $(X,\Lc)$  we can find
a  subset $Y_L\subset L$  and  an integer $m_L$  such that  $\Vol(L\setminus Y_L)=0$ and that all properties (i)--(iii) hold for $m_L$
 maps  $L\ni x \mapsto V_i(x)$  with  $1\leq i\leq  m_L.$ 
 Let $Y:=\cup Y_L,$ the  union being taken over all leaves of $(X,\Lc).$ So $Y$ is  of  full $\mu$-measure.
 Consider  the  leafwise constant function
 $\bar m:\  X\to\N$ given by  $\bar m:=m_L$  on  any   leaf $L.$
 So  $\bar m(x)=m(x)$ for $\mu$-almost every $x\in X.$
  By the ergodicity of $\mu,$  $\bar m$ is equal to a constant $m$ $\mu$-almost everywhere.
By removing   from $Y$ a  subset of null $\mu$-measure  if necessary while  still keeping $Y$ leafwise saturated, 
 we may assume  that $m(x)=m$ for all $x\in Y.$ 
By Part 5) of Proposition  \ref{prop_current_local_consequence}, we may also assume that $Y$ is a Borel set.  This proves  assertion (i).


 Using the same argument   for $m$ maps  
 $Y\ni x \mapsto V_i(x)$ with $1\leq i\leq m,$
 the corollary follows.
 \end{proof}

 The purpose of this  chapter is  to split an $\mathcal A$-invariant bundle   into a  direct sum of $\mathcal A$-invariant components. This  splitting will  enable us
 to apply  the ergodic Birkhoff theorem\index{Birkhoff!$\thicksim$ ergodic theorem}
\index{theorem!Birkhoff ergodic $\thicksim$} in the next chapters.  Throughout  the Memoir, for a  real-valued  function $h,$ $h^+$ denotes $\max(0, h).$ 

\begin{lemma}\label{lem_tail_term}
Let $h:\  \Omega(X,\Lc)\to [0,\infty)$ be  a  measurable function such that
$(h-h\circ T)^+\in L^1(\bar\mu).$ Then ${1\over n} h(T^n\omega)\to 0$
for  $\bar\mu$-almost every $\omega\in \Omega(X,\Lc).$
\end{lemma} 
\begin{proof}
Observe  that  $h(T^n\omega)=h(\omega)-\sum_{i=0}^{n-1} (h-h\circ T)(T^i\omega).$
Since  $(h-h\circ T)^+\in L^1(\bar\mu),$ the  classical  Birkhoff ergodic  theorem\index{Birkhoff!$\thicksim$ ergodic theorem}
\index{theorem!Birkhoff ergodic $\thicksim$} gives that $\lim_{n\to\infty} {1\over n} h(T^n\omega)$
exists for $\bar\mu$-almost every $\omega\in \Omega(X,\Lc),$ but could take the value $\infty.$ We need to prove that this limit is  equal to $0$ almost
everywhere. To do this let $A_k:=\{\omega\in  \Omega(X,\Lc):\ | h(\omega)|\leq  k\}$ for $k\in\N.$
Then $\cup_{k=1}^\infty A_k= \Omega(X,\Lc).$ If $\bar\mu(A_k)>0$ then by the recurrence theorem, for $\bar\mu$-almost every $\omega\in A_k,$
there exist $n_1(\omega)<n_2(\omega)<\cdots$ with  $T^{n_i(\omega)}(\omega)\in A_k,$ $i\geq  1.$ Hence,  $|h(T^{n_i(\omega)}(\omega))|\leq k$ and  so
$
\liminf_{n\to\infty}{1\over n} |h(T^n\omega)|=0.
$
This holds for   $\bar\mu$-almost every $\omega\in \cup_{k=1}^\infty A_k.$
\end{proof}

In what follows  $Y$  denotes the  set of full $\mu$-measure  given by  Corollary \ref{cor_leafwise_Oseledec}.
For $x\in Y$ let $\Omega_x(Y)$  denotes the space of all paths in $\Omega(Y)$ originated  at $x,$  that is,  $\Omega_x(Y):=\Omega(Y)\cap \Omega_x.$ For  a  matrix  $A\in \GL(d,\R)$ and  a vector subspace  $U\subset \R^d,$
let $\| A|_U\|$ be the Euclidean  norm of the linear homomorphism $A|_U:\ U\to\R^d.$
\begin{lemma}\label{lem_subadditive_estimate}
Let $\mathcal A$ be  a cocycle  such that $\int_{\Omega(X,\Lc)} \log^+\| \mathcal A (\omega,1)\| d\bar\mu(\omega)<\infty.$  
Suppose that $\mu$ is  ergodic and that  $Y\ni x\mapsto  U(x)$ is a measurable $\mathcal A$-invariant subbundle of $Y\times\R^d.$ \\
 (i) Then 
$
\lim_{n\to\infty}{1\over n} \log\| \mathcal A(\omega,n)|_{U(\pi_0\omega)}  \| 
$
exists and is  constant for $\bar\mu$-almost every $\omega\in \Omega(Y),$  but the limit could be  $-\infty;$  
\\ (ii) Suppose that the value of the above limit is  less than
or equal to $\alpha\in\R.$ For $\epsilon>0$ define
$$
a_\epsilon(\omega):=\sup_{n\in\N}\big (\| \mathcal A (\omega,n)|_{U(\pi_0\omega)}  \| \cdot e^{-n(\alpha+\epsilon)}\big).
$$
Then $\lim_{n\to\infty}{1\over n}\log a_\epsilon(T^n\omega)=0$  for $\bar\mu$-almost every $\omega\in \Omega(Y).$
\end{lemma}
\begin{proof}
For $n\in\N$ let $f_n:\  \Omega(Y)\to \R$  defined  by
$$
f_n(\omega):=\log  \|  \mathcal A(\omega,n)|_{U(\pi_0\omega)}  \|,\qquad  \omega\in \Omega(Y).
$$
By  the hypothesis, $\int_{\Omega(Y)} f_1^+(\omega)d\bar\mu(\omega)<\infty.$
Since $\mathcal A$ is  a   cocycle  and the  subbundle $x\mapsto U(x)$ is $\mathcal A$-invariant, we  see that
$$
f_{n+m}(\omega)\leq  f_n(\omega)+f_m(T^n\omega),\qquad  \omega\in \Omega(Y).
$$
Applying  the subadditive ergodic theorem \cite{Krengel} to  the  sequence  $(f_n),$  assertion  (i) follows

We turn to assertion (ii). By the choice of $\alpha$ we have that $0\leq  a_\epsilon(\omega)<\infty.$ Also
$$
{a_\epsilon(\omega)\over a_\epsilon(T\omega)}\leq  \max\big (\| \mathcal A (\omega,1)|_{U(\pi_0\omega)}  \| \cdot e^{-(\alpha+\epsilon)}, 1\big)
$$
so that 
$$
\log a_\epsilon(\omega)-\log a_\epsilon(T\omega)\leq  \max\big (\log^+\| \mathcal A (\omega,1)|_{U(\pi_0\omega)}  \| -(\alpha+\epsilon), 0\big).
$$
Recall from  the hypothesis that $\int_{\Omega(X,\Lc)} \log^+\| \mathcal A (\omega,1)\| d\bar\mu(\omega)<\infty.$
Hence, $\omega\mapsto \big (\log a_\epsilon(\omega)-\log a_\epsilon(T\omega)\big )^+$ is $\bar\mu$-integrable and  we  can apply Lemma
\ref{lem_tail_term}.
\end{proof}

 For two  vector subspaces  $A,$ $B$ of $\R^d,$ let $\Hom(A,B)$  denote the vector space of all linear  homomorphisms 
 from $A$ to $B.$
 Now  we  are in the position to  state  the main result of this chapter.
 \begin{theorem}\label{thm_splitting}
 Let $\mu$ be  an ergodic  harmonic  probability measure, and
    $\mathcal A:\ \ \Omega(X,\Lc)\times \N \to \GL(d,\R)$   a  cocycle,
 and  $Y\subset X$  a set of full $\mu$-measure. 
 Assume  that $\int_{\Omega(X,\Lc)} \log^+\| \mathcal A (\omega,1)\| d\bar\mu(\omega)<\infty.$  
 Assume also that
   $Y\ni x\mapsto U(x)$ and  $Y\ni x\mapsto  V(x)$ are two measurable  $\mathcal A$-invariant  subbundles of $Y\times \R^d$ with $V(x)\subset U(x),$ $x\in Y.$  Define a new measurable     subbundle  $Y\ni x\mapsto  W(x)$ of $Y\times \R^d$  
  by  splitting 
$U(x)=V(x)\oplus W(x)$  so that $W(x)$ is  orthogonal  to $V(x)$  with respect to the  Euclidean inner product of $\R^d.$
 Let   $\alpha,$ $\beta$ be two  real  numbers   with $\alpha<\beta$ such that
 
 $\bullet$  $ \chi(x, v)\leq \alpha$ for every $ x\in Y,$  $ v\in V(x)\setminus\{0\};$ 
 
$\bullet$
 $
  \chi(\omega, w)  \geq \beta$ for every $x\in Y,$ every  $w\in W(x)\setminus \{0\}$ and for 
every $\omega\in \Gc_{x,w}.$   
 Here   $\Gc_{x,w}$ is a  subset of $\Omega_x(Y)$  depending  on $x$ and $w$  with $W_x(\Gc_{x,w})>0,$ and   the  functions   $ \chi(x, v)$ and  $
  \chi(\omega, w) $ have  been  defined  in (\ref{eq_functions_chi})-(\ref{eq_functions_chi_new}).

Let $\mathcal A(\omega,1)|_{U(\pi_0\omega)}:\ U( \pi_0\omega    )\to  U( \pi_1\omega    ) $ induce the linear maps
 $\mathcal C(\omega):\    W( \pi_0\omega    )\to  W( \pi_1\omega    )$  and  $\mathcal B(\omega):\  W( \pi_0\omega    )\to  V( \pi_1(\omega)    )$ 
 by
 $$
 \mathcal A(\omega,1)w=\mathcal B(\omega)w\oplus \mathcal C(\omega)w,\qquad  \omega\in\Omega_x(Y),\ w\in W(x).
 $$
(i) Then  the map $\mathcal C$ defined   on $ \Omega(Y)\times \N$ by the formula 
$$\mathcal C(\omega,n):= \mathcal C(T^n\omega)\in \Hom( W( \pi_0\omega    ),  W( \pi_n\omega    )),\qquad \omega\in \Omega(Y),\ n\in \N,$$
satisfies  $\mathcal C (\omega,m+k)=\mathcal C(T^k\omega,m)\mathcal C(\omega,k),$  $m,k\in\N.$ Moreover, 
$\mathcal C(\omega,n)$ is invertible.
 
 There exists a   subset $Y'$ of $Y$ of full $\mu$-measure  with the  following properties:
 \\ (ii) for each  $ x\in Y'$ and  for each $ w\in W(x)\setminus \{0\},$
 there  exists a  set $\Fc_{x,w}\subset \Gc_{x,w}$    such that: $W_x(\Fc_{x,w})=W_x(\Gc_{x,w})$ and 
 that for each $ v\in V(x)$ and each $\omega\in\Fc_{x,w},$    we have 
 $$
 \chi(\omega,v\oplus w)=\chi(\omega,w)=\limsup_{n\to\infty} {1\over n} \log \|  \mathcal C(\omega,n)w\|;
 $$
 (iii)  
  if   for   some  $x\in Y'$  and some $w\in W(x)\setminus \{0\}$ and some $ v\in V(x)$ 
  and some  $\omega\in \Fc_{x,w}$ the limit $\lim_{n\to\infty} {1\over n} \log \|\mathcal C(\omega,n)w\| $ exists, then
 $\lim_{n\to\infty} {1\over n} \log \|\mathcal A(\omega, n)(v\oplus w)\| $ exists  
 and  is  equal to    the previous limit.
 \end{theorem}
\begin{proof}
We use   $v$ to denote  a general element of some $V(x)$ and  $w$ a  general element of some $W(x).$
Using  the multiplicative  property of the  cocycle  $\mathcal A,$  we obtain the following  formula,  for  $\omega\in \Omega_x(Y),$
\begin{equation}\label{eq_iterations}
  \mathcal A(\omega,n)(v\oplus w)= \big ( \mathcal A(\omega,n)v+ \mathcal D(\omega,n)w\big ) \oplus \mathcal C(\omega,n)w,
  \end{equation}
where $\mathcal D(\omega,n):\  W(\pi_0\omega)\to V(\pi_n\omega)$  is given by
$$ \mathcal D(\omega,n):=\sum_{i=0}^{n-1} \mathcal A(T^{i+1}\omega,n-i-1)\circ \mathcal B(T^i\omega)\circ \mathcal C(\omega, i).  $$ 

To prove  assertion (i) pick an arbitrary $w\in W(x).$ Using  (\ref{eq_iterations}) and the assumption that  both maps $x\mapsto  U(x),$  $x\mapsto V(x)$ are $\mathcal A$-invariant
subbundles  of $Y\times \R^d,$  we  see that 
$ \mathcal C(\omega,n)w$ is the image   of  $\mathcal A(\omega,n)( w)$ by the  projection of
$U_{\pi_n\omega} =V_{\pi_n\omega}\oplus W_{\pi_n\omega}$
onto the second summand. Hence,
$$\mathcal A(\omega,m+k)(  V_{\pi_0\omega}\oplus w)=  V_{\pi_{m+k}\omega}\oplus \mathcal C(\omega,m+k)w.$$
 Moreover, using the $\mathcal A$-invariant assumption   again we  have    that
\begin{eqnarray*}
\mathcal A(\omega,m+k)(  V_{\pi_0\omega}\oplus w)&=& \mathcal A(T^k\omega,m)\mathcal A(\omega,k)(  V_{\pi_0\omega}\oplus w) \\
&=& \mathcal A(T^k\omega,m)(   V_{\pi_k\omega}\oplus \mathcal C(\omega,k)  w)\\
&=  &  V_{\pi_{m+k}\omega}\oplus \mathcal C(T^k\omega,m)  \mathcal C(\omega,k)  w.
\end{eqnarray*}
This, combined  with  the previous  equality, implies  assertion (i).

Now we prove  assertions  (ii) and  (iii). 
Lemma  \ref{lem_elementary_inequality}, applied to identity (\ref{eq_iterations}), yields that
\begin{equation}\label{eq_thm_splitting_1}
\begin{split}
&\quad \limsup_{n\to\infty}{1\over n} \log \|  \mathcal A(\omega,n)(v\oplus w)\|\\
&=\max\big ( \limsup_{n\to\infty} {1\over n} \log \|  \mathcal A(\omega,n)v +\mathcal D(\omega,n)w\|,
\limsup_{n\to\infty} {1\over n} \log \|  \mathcal C(\omega,n)w\|\big).
\end{split}
\end{equation}
Letting $v=0$ and $w\not=0$ in  (\ref{eq_thm_splitting_1}),  we  deduce that
\begin{equation}\label{eq_thm_splitting_2}
\begin{split}
&\limsup_{n\to\infty} {1\over n}\log \|  \mathcal A(\omega,n)w\|\\
&=\max\big ( \limsup_{n\to\infty} {1\over n}\log \|  \mathcal D(\omega,n)w\|,
\limsup_{n\to\infty} {1\over n}\log \|  \mathcal C(\omega,n)w\|\big).
\end{split}
\end{equation}
For $\epsilon>0$  let $a_\epsilon(\omega):=\sup_{n\in\N}\big (\| \mathcal A (\omega,n)|_{V(\pi_0\omega)}  \| \cdot e^{-n(\alpha+\epsilon)}\big).$ By the first  assumption $\bullet,$ we may apply  Lemma \ref{lem_subadditive_estimate}
to $a_\epsilon(\omega),$ and  Lemma \ref{lem_tail_term} to $h(\omega):= \|\mathcal A(\omega,1)\|,$ $\omega\in \Omega(X,\Lc).$
Let $(\epsilon_m)_{m=1}^\infty$ be  a sequence decreasing strictly to  $ 0. $
  By  Lemma \ref{lem_subadditive_estimate} and  Lemma \ref{lem_tail_term},
we may find, for each  $m\geq 1,$  a  subset  $\Omega_m$  of $\Omega(Y)$ of full $\bar\mu$-measure such that 
\begin{equation}\label{eq_tend_to_0}
 {1\over n} a_{\epsilon_m}(T^n\omega)\to 0\quad\textrm{and}\quad 
{1\over n} \log \| \mathcal A(T^n\omega,1)\|\to 0 \quad\textrm{for all  $\omega\in\Omega_m.$}
\end{equation} 
For every $x\in Y$  set  $\Fc'_x:= \Omega_x\cap \cap_{m=1}^\infty \Omega_m\subset \Omega_x(Y).$
Since  $\cap_{m=1}^\infty \Omega_m$ is of full $\bar\mu$-measure, there exists a subset $Y'\subset Y$
of full $\mu$-measure  such that for every $x\in Y',$ $\Fc'_x$  is  of full $W_x$-measure.
By the first assumption $\bullet$ combined with Proposition \ref{prop_chi} (ii)-(iii),   for every $x\in Y',$ there  exists a set  $\Fc_x\subset \Fc'_x$    of full $W_x$-measure such that,
for every  $\omega\in\Fc_x,$ 
\begin{equation}\label{eq_thm_splitting_bullets}
\chi(\omega,v)\leq \alpha <\beta  ,\qquad v\in V(x). 
\end{equation}
By the second assumption $\bullet,$  for every $x\in Y'$ and for every $w\in W(x)\setminus \{0\},$ there  exists a set  $\Fc_{x,w}:= \Gc_{x,w}\cap  \Fc_x\subset  \Omega_x(Y)$     such that,
for every  $\omega\in\Fc_{x,w},$
\begin{equation}\label{eq_thm_splitting_second_bullets}
 \alpha <\beta \leq \chi(\omega,w). 
\end{equation}
Since  $W_x(\Fc_x)=1,$  we see that $W_x(\Fc_{x,w})=W_x(\Gc_{x,w})>0.$
We  will prove that  for every $x\in Y'$ and for every $w\in W(x)\setminus \{0\},$  and for   every $\omega\in  \Fc_{x,w},$
 \begin{equation} \label{eq_thm_splitting_3}
\limsup_{n\to\infty}{1\over n} \log \|  \mathcal C(\omega,n)w\|
= 
\limsup_{n\to\infty}{1\over n} \log \|  \mathcal A(\omega,n)w\|. 
\end{equation}
Let  $\tau$ be  the left side limit. 
By (\ref{eq_thm_splitting_2}), $\tau$ is  smaller than the right hand  side.
  By (\ref{eq_thm_splitting_second_bullets}),  $\alpha $   is strictly smaller than the right hand  side. So
  $\max(\tau,\alpha)$ is smaller than the right hand  side  of (\ref{eq_thm_splitting_3}).  
Hence, by (\ref{eq_thm_splitting_2}) again, $\limsup_{n\to\infty}{1\over n} \log \|  \mathcal D(\omega,n)w\|\geq \max(\tau,\alpha).$
We will  prove that 
\begin{equation}\label{eq_thm_splitting_4}
\limsup_{n\to\infty}{1\over n} \log \|  \mathcal D(\omega,n)w\|\leq \max(\tau,\alpha).  
\end{equation}
Taking (\ref{eq_thm_splitting_4}) for granted,   the above reasoning  shows that the inequality (\ref{eq_thm_splitting_4}) is,
in fact,  an equality.  Hence,  it will follow  from (\ref{eq_thm_splitting_2}) that  
$\limsup_{n\to\infty}{1\over n} \log \|  \mathcal A(\omega,n)w\|=\max(\tau,\alpha). $
Recall again from (\ref{eq_thm_splitting_second_bullets}) that $\limsup_{n\to\infty}{1\over n} \log \|  \mathcal A(\omega,n)w\| >\alpha.$ 
Hence, 
\begin{equation}\label{eq_thm_splitting_tau_alpha}
\tau>\alpha
\end{equation}
 and (\ref{eq_thm_splitting_3}) follows.
  So the proof of  (\ref{eq_thm_splitting_3})  is  reduced   to  the proof of (\ref{eq_thm_splitting_4}).

In order  to show (\ref{eq_thm_splitting_4}), fix  an arbitrary   $m\geq 1.$ Then    there exists $N$  depending   on $\omega, w$ and $m$  such that
$n\geq N$ implies that $\|\mathcal C(\omega,n)w\|<e^{n(\tau+\epsilon_m)}.    $
If  we write $\mathcal L(\omega,n):\  V(\pi_0\omega)\to  V(\pi_n\omega)$ instead of $\mathcal A(\omega, n)|_{ V(\pi_0\omega)},$ then
\begin{eqnarray*}
\| \mathcal D(\omega,n)w  \| &\leq & \sum_{i=0}^{n-1} \| \mathcal L(T^{i+1} \omega,n-i-1) \| \cdot \|\mathcal B(T^i\omega)  \|\cdot \| \mathcal C(\omega,i)w\| \\
&\leq &  n\max_{0\leq i\leq n-1}  \| \mathcal L(T^{i+1} \omega,n-i-1) \| \cdot \|\mathcal B(T^i\omega)  \|\cdot \| \mathcal C(\omega,i)w\| \\
&=& n \| \mathcal L(T^{i_n+1} \omega,n-i_n-1) \| \cdot \|\mathcal B(T^{i_n}\omega)  \|\cdot \| \mathcal C(\omega,i_n)w\| 
\end{eqnarray*}
for  some  $0\leq i_n\leq n-1,$ which depends also on $\omega$ and $w.$ Note that $(i_n)$ is  an increasing sequence.  
\\
\noindent   {\bf Case 1: } {\it     $(i_n)$ is unbounded.}

 So  $i_n\geq N$ for $n$ large  enough.
Consequently, we have that
$$
{1\over n} \log^+ \|  \mathcal B(T^{i_n}\omega)\|\leq  {i_n\over n}{1\over i_n}\log^+\| \mathcal A(T^{i_n}\omega,1)\|\leq {1\over i_n}\log^+\| \mathcal A(T^{i_n}\omega,1)\| \to 0
$$
by the  membership  $\omega\in\Fc_x$ and by the second estimate of (\ref{eq_tend_to_0}).   Moreover, for every $m\geq 1,$  the first estimate of (\ref{eq_tend_to_0}) gives that
$$
{1\over n}\log a_{\epsilon_m}(T^{i_n+1}\omega)={i_n+1\over n}{1\over i_n+1}\log a_{\epsilon_m}(T^{i_n+1}\omega)\to 0.
$$  
These inequalities, combined  with the  above  estimate for $\| \mathcal D(\omega,n)w  \|,$ imply that, for every $m\geq 1,$ 
 \begin{multline*}
 {1\over n} \log \|  \mathcal D(\omega,n)w\|\leq {1\over n}\log n + {1\over n}\log a_{\epsilon_m}(T^{i_n+1}\omega) +{n-1-i_n\over n}(\alpha+\epsilon_m)\\
 +{1\over n} \log^+ \|  \mathcal B(T^{i_n}\omega,n)\| +{i_n\over n}(\tau+\epsilon_m).
\end{multline*}
So  $\limsup_{n\to\infty} {1\over n} \log \|  \mathcal D(\omega,n)w\|\leq \max(\tau,\alpha)+\epsilon_m.$
By letting $m\to\infty,$ we get  (\ref{eq_thm_splitting_4}) as  desired. 
   \\
\noindent   {\bf Case 2: }{\it    $(i_n)$ is  bounded, say $i_n\leq M$ for all $n.$} 

We see easily that 
\begin{multline*}
 {1\over n} \log \|  \mathcal D(\omega,n)w\|\leq {1\over n}\log n + \max_{0\leq i\leq M}{1\over n} \log \|  \mathcal L(T^{i+1}\omega,n-i-1)\| \\
+
 \max_{0\leq i\leq M}{1\over n} \log \|  \mathcal B(T^i\omega)\|+\max_{0\leq i\leq M}{1\over n} \log \|  \mathcal C(\omega,i)\|.
\end{multline*}
Since  on the right hand side, the limsup of the second  term is  smaller than $\alpha$ by  (\ref{eq_thm_splitting_bullets}),  whereas  other terms tend to $0$ as $n\to\infty,$
it follows  that  $ {1\over n} \log \|  \mathcal D(\omega,n)w\|\leq\alpha,$    proving (\ref{eq_thm_splitting_4}).
Hence, the proof of (\ref{eq_thm_splitting_3}) is  complete.

 Lemma   \ref{lem_elementary_inequality}, applied to the first term in the  right hand side of 
(\ref{eq_thm_splitting_1}), yields that
\begin{multline*}
\limsup_{n\to\infty} {1\over n} \log \|  \mathcal A(\omega,n)v +\mathcal D(\omega,n)w\|\leq  \max\big (
\limsup_{n\to\infty} {1\over n} \log \|  \mathcal A(\omega,n)v  \|,\\
\limsup_{n\to\infty} {1\over n} \log \|  \mathcal D(\omega,n)w\|
\big ).
\end{multline*}
Observe in the  last line that the first term in the right hand side  is  smaller than  $\alpha$ by  
(\ref{eq_thm_splitting_bullets}),  
whereas the  second term $\leq \max(\tau,\alpha)$ by (\ref{eq_thm_splitting_4}). This, combined with (\ref{eq_thm_splitting_tau_alpha}),
implies that  the left hand  side of the  last line is $\leq  \tau.$
This, coupled  with (\ref{eq_thm_splitting_1}) and  (\ref{eq_thm_splitting_3}), gives that
for every $x\in Y'$ and every $w\in W(x)\setminus\{0\}$ and every $\omega\in \Fc_{x,w},$
 \begin{equation}\label{eq_thm_splitting_5}
 \begin{split}
&\quad \limsup_{n\to\infty}{1\over n} \log \|  \mathcal A(\omega,n)(v\oplus w)\|=
\limsup_{n\to\infty}{1\over n} \log \|  \mathcal C(\omega,n)w\|\\
&= \tau= \limsup_{n\to\infty}{1\over n} \log \|  \mathcal A(\omega,n)w  \| 
,\qquad \forall v\in V(x).
\end{split}
 \end{equation}
 This proves  assertion (ii).
 
 Now  suppose that $\lim_{n\to\infty} {1\over n} \log \|  \mathcal C(\omega,n)w\|$ exists for some $x\in Y',$ some $w\in W(x)\setminus\{0\}$ and some $\omega\in\Fc_{x,w}.$  
  By  the inequality in Lemma \ref{lem_elementary_inequality} and (\ref{eq_iterations}), we have, for every $v\in V(x),$ that
  \begin{eqnarray*}
&& \liminf_{n\to\infty}{1\over n} \log \|  \mathcal A(\omega,n)(v\oplus w)\|\\
&\qquad & \geq \max\big ( \liminf_{n\to\infty} {1\over n} \log \|  \mathcal A(\omega,n)v +\mathcal D(\omega,n)w\|,
\liminf_{n\to\infty} {1\over n} \log \|  \mathcal C(\omega,n)w\|\big)\\
&\qquad & \geq \liminf_{n\to\infty} {1\over n} \log \|  \mathcal C(\omega,n)w\|.
\end{eqnarray*}
  This, combined  with (\ref{eq_thm_splitting_5}), implies  assertion (iii).
  \end{proof}

%% file: chapter7.tex


\chapter{Lyapunov forward filtrations}

\label{section_Lyapunov_filtration}


The first part   of this  chapter  makes the reader  familiar with some new  terminology and  auxiliary results
which  are constantly  present  in this  work.   
The second part is  devoted  to two  Oseledec type theorems.  The  first one  is  a  direct consequence of
Oseledec  Multiplicative Ergodic  Theorem \ref{thm_Oseledec}\index{Oseledec!$\thicksim$ multiplicative ergodic theorem}. The  second  theorem  is  the main result of this  chapter.
Its proof occupies the  last parts  of the  chapter where   new   techniques  such as 
totally  invariant sample-path sets and  stratifications are introduced. 
We  will see  in this proof     that the holonomy    of the leaves comes into 
 action. Before proceeding further we  need the following  terminology.
 \begin{definition}\label{defi_totally_invariant_set}
Let $T$ be a  measurable transformation defined on a measurable space   $\Omega.$
  A measurable set $F\subset\Omega$ is  said to be  {\it $T$-invariant}
(resp.  {\it $T$-totally invariant})
if $T^{-1}F=F$  (resp. $TF=T^{-1}F= F$).\index{set!totally invariant $\thicksim$}
\end{definition} When $T $  is  surjective,   a set $F$ is  $T$-invariant if and only if it is  $T$-totally invariant.

\section{Oseledec type theorems}
\label{subsection_Oseledec_theorems}
 Consider  a       lamination $(X,\Lc)$   satisfying the  Standing Hypotheses endowed with
 a   harmonic probability measure  $\mu$ which is  ergodic. 
Consider also  a (multiplicative) cocycle $\mathcal{A}:\ \Omega\times \N \to  \GL(d,\R)      ,$ where $\Omega:=\Omega(X,\Lc)  .$ 
   Assume that
     $\int_\Omega \log^+ \|\mathcal{A}^{\pm 1}(\omega,1)\|  d\bar\mu(\omega)<\infty.
 $
Let  $T:=T^1$ be  the  shift-transformation of unit-time\index{shift-transformation!$\thicksim$ of unit-time} on $\Omega$  given in (\ref{eq_shift}).
 By Theorem \ref{thm_ergodic_measures}, $\bar\mu$ is ergodic with respect to $T$  acting on  the
measure space $(\Omega,\Ac,\bar\mu).$ Consequently, we  deduce  from Theorem  \ref{thm_Oseledec} (i)-(iv) the  following result.  
 \begin{theorem} \label{th_Lyapunov_filtration_Brownian_version}
There exists  a   subset $\Phi$ of $\Omega$ of  full $\bar\mu$-measure 
 and  a    number $l\in \N$ and
 $l$ integers  $1\leq r_l<r_{l-1}<\cdots<r_1=d$ and $l$ real numbers
$\lambda_l<\lambda_{l-1}<\cdots< \lambda_1$
such that
the following properties hold:
 \\ (i)   
  For   each $\omega\in \Phi$ there   are linear subspaces
$$   \{0\}\equiv V_{l+1}(\omega)\subset   V_l(\omega)\subset \cdots\subset  V_2(\omega)\subset V_1(\omega)=\R^d,
$$
of $\R^d$ 
  such that $\mathcal{A}(\omega, 1) V_i(\omega)= V_i(T\omega)$ and that  $\dim V_i(\omega)=r_i$  for all $\omega\in  \Phi.$  
Moreover,  $\omega\mapsto  V_i(\omega)$ is   a  measurable map from $\Phi$ into the Grassmannian of $\R^d.$
 \\ (ii) For each  $\omega\in \Phi$ and $v\in V_i(\omega)\setminus V_{i+1}(\omega),$    
$$\lim\limits_{n\to \infty} {1\over  n}  \log {\| \mathcal A(\omega,n)v   \|\over  \| v\|}  =\lambda_i,    
$$
  for   every  $\omega\in\Phi.$
\\ (iii) 
$\lambda_1=\lim_{n\to\infty}{1\over  n}  \log {\| \mathcal A(\omega,n)}\|$  for   every  $\omega\in\Phi.$
 \end{theorem}
 On the other hand, by Corollary \ref{cor_leafwise_Oseledec},
 there is a Borel set $Y\subset X$ of full $\mu$-measure and
there are integers $m\geq 1,$ $1\leq d_m<\cdots<d_1=d$ and  real numbers $\chi_m<\cdots<\chi_1$ such that
 $m(x)=m,$  $\dim V_i(x)=d_i$ and $\chi_i(x)=\chi_i$ for  every $x\in Y.$
Moreover, 
$Y\ni x\mapsto \dim V_i(x)$  is  an $\mathcal A$-invariant subbundle of $Y\times\R^d.$ 
 
The  purpose  of this  chapter is  to unify  the  above  two results.
More concretely, we  want to  compare $m$  with $l,$  $\{  \chi_1,\ldots,\chi_m\}$ with $ \{\lambda_1,\ldots, \lambda_l\}$
and  $\{V_1(x),\ldots,V_m(x)\}$  with    $\{V_1(\omega), \ldots, V_l(\omega)\}$ for $x\in Y$ and $\omega\in \Omega_x\cap\Phi$  respectively.
Here is the  main  result of this  chapter.
\begin{theorem} \label{th_Lyapunov_filtration}
Under  the above hypotheses and notation  we have  that $m\leq l$ and
$\{  \chi_1,\ldots,\chi_m\}\subset \{\lambda_1,\ldots, \lambda_l\}$ and 
$\chi_1=\lambda_1.$  Moreover,  there  exists 
 a Borel set $Y_0\subset Y$ of full $\mu$-measure  such that for every $x\in Y_0$  and  for every $u\in V_i(x)\setminus V_{i+1}(x),$
 \begin{equation*} 
\lim_{n\to\infty} {1\over n} \log \|  \mathcal A(\omega,n)u\|=\chi_i
 \end{equation*}
 for $W_x$-almost  every $\omega\in \Omega_x.$
In particular,   
$$\chi(\omega,u)=\chi(x,u)=\chi_i$$
 for $W_x$-almost  every $\omega\in \Omega_x.$
Here  the  functions  $\chi(\omega,u)$ and  $\chi(x,u)$ are defined in (\ref{eq_functions_chi})-(\ref{eq_functions_chi_new}).
\end{theorem}
\begin{remark}
This  result   may be  considered as  the first  half of Theorem
\ref{th_main_1}.   The  decreasing sequence of subspaces of $\R^d:$
  $$\{0\}\equiv V_{m+1}(x)\subset V_m(x)\subset \cdots\subset V_1(x)=\R^d$$
is  called the {\it  Lyapunov forward filtration} associated to $\mathcal A$ at  a given  point $ x\in Y_0.$ The remarkable point of Theorem \ref{th_Lyapunov_filtration} is that
this filtration depends only on the point $x,$ and  not on paths $\omega\in\Omega_x.$\index{filtration!Lyapunov forward $\thicksim$}\index{Lyapunov!$\thicksim$ forward filtration}
\end{remark}
\smallskip

\noindent {\bf Proof of Theorem \ref{th_Lyapunov_filtration}.} 

  The  proof  is  divided into  several steps.

 \noindent{\bf  Step 1:} {\it Construction of 
 a Borel set $Y_0\subset Y$ of full $\mu$-measure. Proof  that  
$\{  \chi_1,\ldots,\chi_m\}\subset \{\lambda_1,\ldots, \lambda_l\}$ and    $\chi_1=\lambda_1.$}

   By Corollary  \ref{cor_leafwise_Oseledec}    and  Theorem   \ref{th_Lyapunov_filtration_Brownian_version},
 there is a Borel set $Y_0\subset Y$ of full $\mu$-measure such that for every $x\in Y_0,$  the set
$\Phi\cap \Omega_x$ is  of  full $W_x$-measure. By the  definition and  by Lemma  \ref{lem_esup}, we obtain, for $x\in  Y_0$ and $u\in \R^d\setminus\{0\},$   a  set $E\subset \Omega_x\cap \Phi$ of full $W_x$-measure such that
$$
\chi(x,u)=\esup_{\omega\in \Omega_x} \chi(\omega,u)=\sup_{\omega\in E} \chi(\omega,u)\in   \{\lambda_1,\ldots, \lambda_l\}.
$$  
   This  implies that $\{  \chi_1,\ldots,\chi_m\}\subset \{\lambda_1,\ldots, \lambda_l\}.$ In particular, we get that $m\leq l$
   and $\chi_1\leq \lambda_1.$
    Therefore, in order to prove  that   $\chi_1=\lambda_1,$ it suffices to show that  $\chi_1\geq \lambda_1.$ 
By  Proposition
 \ref{prop_chi}, for  every $x\in Y_0$  there  exists a set $E_x\in \Ac(\Omega_x)$ of full $W_x$-measure such that 
 $$
v\in V_i(x)\setminus V_{i+1}(x)\Leftrightarrow \sup_{\omega\in E_x}\limsup_{n\to\infty} {1\over n} \log \| \mathcal{A}(\omega, n)   v   \| =\chi_i.
$$
 Therefore, by the definition we get that  for every $v\in \R^d\setminus \{0\},$  for  every $x\in Y_0$ and every $\omega\in E_x,$
 $  \limsup_{n\to\infty} {1\over n} \log \| \mathcal{A}(\omega, n)   v   \| \leq  \max\{ \chi_m,\ldots,\chi_1  \}=\chi_1 .$ Hence, 
$\lambda_1\leq \chi_1,$  as  desired.


 \noindent{\bf  Step 2:} {\it By shrinking  $Y_0$ a little,   for every $x\in Y_0$ and  every $u\in V_m(x)\setminus \{0\},$   the  equalities 
\begin{equation}\label{eq_Step2_Chapter7}
\lim_{n\to\infty} {1\over n} \log \|  \mathcal A(\omega,n)u\|=\chi(\omega,u)=\chi(x,u)=\chi_m
\end{equation} 
hold for $W_x$-almost  every $\omega\in \Omega_x.$ }

 The proof of Step 2  will be  given  in  Section \ref{subsection_Stratifications} below.

  \noindent{\bf  Step 3:} {\it End  of the proof.}

If  $m=1$ then  Step 2  completes the proof of the  theorem.
Therefore,  assume  that $m\geq 2.$ For every $x\in Y_0,$  
consider the  orthogonal decomposition $V_{m-1}(x)=V_m(x)\oplus W(x)$
with respect to  the Euclidean inner product in $\R^d.$ 
We will apply  Theorem \ref{thm_splitting} to the cocycle $\mathcal A$ in the following setting: 
$$ V(x):= V_m(x)\quad \text{and}\quad  U(x):=V_{m-1}(x),\qquad  x\in Y_0.$$
In order  to ensure  the  two  $\bullet$ conditions in   Theorem \ref{thm_splitting},
we choose $ \alpha:=\chi_m$ and $  \beta>\alpha$  so that
$$\beta<\min \{ \lambda\in \{\lambda_1,\ldots,\lambda_l\}:\ \lambda>\alpha\}.$$ 
This choice is  possible  using  Step 1 and the assumption   $m\geq 2.$ 
 Recall from Theorem   \ref{th_Lyapunov_filtration_Brownian_version}  that  $\chi(\omega,w)\in\{\lambda_1,\ldots,\lambda_l\}$ 
for 
$\omega\in\Phi$ and $w\in\R^d\setminus\{0\}.$  This,  coupled   with   Corollary  \ref{cor_leafwise_Oseledec}, guarantees that
for every $x\in Y_0$ and  every $w\in V_{m-1}(x) \setminus V_m(x),$ there is a  set
$\Gc_{x,w}\subset \Omega_x(Y_0)\cap \Phi$ such that 
\begin{equation*}
W_x(\Gc_{x,w})>0\quad \textrm{and}\quad  \chi(\omega,w)=\chi_{m-1},\qquad  \omega\in \Gc_{x,w}.
\end{equation*}
So the  hypotheses  of  Theorem \ref{thm_splitting} are  fulfilled.
Consequently,  we deduce from assertion (ii) of the latter theorem that
\begin{equation}\label{eq_C}
\esup_{\omega\in \Omega_x(Y)} \limsup_{n\to\infty} {1\over n} \log\|  \mathcal C(\omega,n)  w \| = \chi_{m-1}
\end{equation}
for $\mu$-almost  every $x\in X$ and for all $w\in W(x)\setminus\{0\}.$

By   Theorem \ref{thm_selection_Grassmannian}  there is  a bimeasurable bijection  between the $\mathcal A$-invariant
subbundle  $Y\ni x\mapsto W(x)$ of rank $d_{m-1}-d_m$  and $Y\times \R^{d_{m-1}-d_m}$ covering  the identity and which is a linear isometry\index{isometry, linear isometry}  on each fiber. Using this   bijection, it follows from
 Theorem \ref{thm_splitting} (i) that
  $\mathcal C$ is multiplicative cocycle  induced  by  $\mathcal A.$
 Since  $ Y\ni x\mapsto  W(x)$  is  a measurable $\mathcal A$-invariant subbundle
and  $\| \mathcal C(\omega,1)\| \leq  \| \mathcal A(\omega,1)\|$ and  $\| \mathcal C(\omega,1)^{-1}\| \leq  \| \mathcal A(\omega,1)^{-1}\|,$
we infer from the $\bar\mu$-integrability of $\omega\mapsto\mathcal A(\omega,1)$ that
 $\omega\mapsto\mathcal C(\omega,1)$ is also $\bar\mu$-integrable.
Consequently, this,  together  with (\ref{eq_C}) allows us  to apply   to  the cocycle  $\mathcal C$ the same arguments used in    Step 2.   
Hence,  we can show that, for    $\mu$-almost every  $x\in X$ and for every $w\in W(x)\setminus\{0\},$
$$
 \lim_{n\to\infty} {1\over n} {\log\|  \mathcal C(\omega,n)  w\| \over \| w\| } = \chi_{m-1}
$$ 
   for $W_x$-almost every  $\omega\in \Omega_x.$ 
Then,  by     Theorem  \ref{thm_splitting} (iii) we get  that,
for  $\mu$-almost every  $x\in X$ and for every $v\in V_{m-1}(x)\setminus V_m  (x),$
$$
 \lim_{n\to\infty} {1\over n} {\log\|  \mathcal A(\omega,n)  v\| \over \| v\|} = \chi_{m-1}
$$ 
 for  $W_x$-almost  every $\omega\in \Omega_x(Y_0).$
 So,  for  $\mu$-almost every  $x\in X$ and for every $v\in V_{m-1}(x)\setminus V_m  (x),$
$$
 \chi (\omega,v) =\chi(x,v) = \chi_{m-1}
$$ 
 for  $W_x$-almost  every $\omega\in \Omega_x(Y_0).$ 
 
 The  next  case where  $v\in V_{m-2}(x)\setminus V_{m-1}  (x)$
 can be proved  in the  same way.  Repeating the above  process  a  finite  number of times,
the proof of the  theorem is thereby completed  modulo Step 2.  \hfill $\square$

\section{Fibered laminations and totally  invariant  sets}
\label{subsection_total_invariant_sets}

 In  this  section we first  introduce   new objects:  the {\it fibered laminations}.
Next,   we study totally invariant measurable sample-path  sets   of  these  new objects.
This,   together  with  the  study of stratifications  given in the next  section, will be  the main ingredients
in the  proof of the  remaining  Step 2 of Theorem \ref{th_Lyapunov_filtration}. 
  Now let $(X,\Lc,g)$ be a Riemannian continuous-like lamination. 
  Let $\pi:\ (\widetilde X,\widetilde\Lc)\to (X,\Lc)$
  be the corresponding   covering  lamination projection.
 \begin{definition}\label{defi_fibered_laminations}
 A {\it weakly  fibered lamination}\index{lamination!weakly fibered $\thicksim$} $\Sigma$ over $(X,\Lc,g)$  is  the  data of  a  
Hausdorff  topological space  $\Sigma$ and  a  measurable projection $\iota:\  \Sigma\to \widetilde X$   
 \nomenclature[b4a]{$\iota:\  \Sigma\to \widetilde X$}{a weakly fibered lamination over a Riemannian continuous-like lamination
$(X,\Lc,g),$ 
where $\pi:\ (\widetilde X,\widetilde\Lc,\pi^*g)\to (X,\Lc,g)$  is the   covering lamination projection of $(X,\Lc,g)$}
 such that  for every $y\in \Sigma,$  there exists a  set $\Sigma_y\subset \Sigma$
 satisfying  the  following properties (i)--(iii): 
\\ 
(i) $y\in \Sigma_y,$ and if $y_1\in \Sigma_{y_2}$ then  $\Sigma_{y_1}= \Sigma_{y_2};$
 \\
(ii) The restriction   $\iota|_{\Sigma_y}:\ \Sigma_y\to \widetilde L_{\iota(y)}$ is homeomorphic,
where $\widetilde L_{\tilde x}$ is  as  usual the leaf of the lamination $(\widetilde X,\widetilde\Lc)$ passing through $\tilde x ;$
\\
(iii)   
There is a  family of maps $\tilde s_i:\ \widetilde E_i\to\Sigma,$  with $\widetilde E_i\subset\widetilde  X,$  indexed by a (at most) countable set $I,$ satisfying
 the  following three properties: 
\begin{itemize}
\item [$\bullet$]  each $\tilde s_i$ is a {\it local section}
\index{section!local $\thicksim$}  of $\iota,$ that is, $\iota(\tilde s_i(\tilde x))=\tilde x$ for all  $\tilde x\in \widetilde E_i$  and $i\in I;$\footnote{It is  worthy pointing out  that the set $\widetilde E_i$ is  not necessarily open.}  
\item [$\bullet$]  for each $i\in I,$ both  $\widetilde E_i$ and $\tilde s_i(\widetilde E_i)$  are Borel sets, and  the   surjective map
 $\tilde s_i:\ \widetilde E_i\to \tilde s_i(\widetilde E_i)$ is  Borel bi-measurable;\index{measurable!bi-$\thicksim$ map}
\item [$\bullet$] the family $(\tilde s_i)_{i\in I}$
  {\it generates all fibers of $\iota,$} that is,
  $$
\iota^{-1}(\tilde x):=\left  \lbrace  \tilde s_i(\tilde x):\  \tilde x\in \widetilde E_i\ \text{and}\ i\in I  \right\rbrace,\qquad \tilde x\in \widetilde X.
$$
\end{itemize}

 Each  set $\Sigma_y$ $(y\in \Sigma)$ is called  a {\it leaf} of $\Sigma.$
 \nomenclature[b9a]{$ \Sigma_y$}{leaf passing through a given point $y$ in a weakly fibered lamination $\iota:\ \Sigma\to\widetilde X$}
 Since $(\widetilde X,\widetilde \Lc,\pi^*g)$  is a Riemannian measurable lamination,  we  equip  each leaf $\Sigma_y$ with the  differentiable  structure
by pulling  back  via $\iota|_{\Sigma_y}$ the differentiable  structure on $L_{\iota(y)}.$
Similarly, we  endow  each leaf $\Sigma_y$ with
the  metric tensor $(\iota|_{\Sigma_y})^*(\pi^*g|_{L_{\iota(y)}}).$
 \end{definition}
 \begin{remark}\label{rem_fibered_laminations}
 Definition \ref{defi_fibered_laminations} (iii) is  very similar to Definition \ref{defi_continuity_like} (i-b).
 
 By Definition \ref{defi_fibered_laminations} (iii),
the  cardinal  of every fiber  $\iota^{-1}(\tilde x)$  ($\tilde x\in \widetilde X$) is  at most  countable; it may  eventually
be  empty at some  fibers (that is, $\iota(\Sigma)$ may be a proper subset of $\widetilde X$); and it may vary from fibers to fibers.

 Although  a  weakly  fibered  lamination  has  the structure of leaves,
 it does not admit, in general, an atlas  with  charts as well as  their related notions such as flow boxes etc. So
a  weakly  fibered  lamination is usually not  a  measurable   lamination and  vice versa.
 
 A  trivial  example  of a  weakly  fibered  lamination is  $\Sigma:=\widetilde X$ and $\iota:=\id.$
 Indeed,  conditions (i) and (ii)  in Definition \ref{defi_fibered_laminations} are clearly fulfilled in this  context,
whereas Definition  
 \ref{defi_continuity_like} (i-b) implies   condition (iii) in Definition \ref{defi_fibered_laminations}. 
 In this  example, $\Sigma$  is also a  measurable  lamination  with  the atlas $\widetilde \Lc.$
 \end{remark}
  Let $\iota:\  \Sigma\to \widetilde X$ be  a  weakly fibered  lamination over $(X,\Lc,g).$
  Observe that every leaf $\Sigma_y$  is simply connected  as it is  diffeomorphic  to the leaf $ L_{\iota(y)}$
which is  simply  connected  since $(\widetilde X,\widetilde\Lc) $ is a  covering measurable lamination of $(X,\Lc).$ Moreover, since  the leaf $\Sigma_y$ is  endowed with
the  metric tensor $(\iota|_{\Sigma_y})^*(\pi^*g|_{L_{\iota(y)}})$ for which  $\Sigma_y$
is a  complete  Riemannian manifold of bounded geometry, we can define  the  heat  diffusion associated to the leaves
of $\Sigma.$ Consequently,    we can carry out the constructions given
 in Section \ref{subsection_Brownian_motion_without_holonomy}.

More concretely, 
  we first construct   the  sample-path  space $\Omega(\Sigma)\subset \Sigma^{[0,\infty)}$ consisting of all continuous paths
  $\omega:\ [0,\infty)\to\Sigma$ with image  fully contained  in a single  leaf.
   \nomenclature[c1a]{$\Omega(\Sigma)$}{sample-path space associated to a  weakly fibered lamination $\iota:\ \Sigma\to\widetilde X$} 
Next,  we define the notion of cylinder set\index{cylinder!$\thicksim$ set}\index{set!cylinder $\thicksim$}  in a similar way as in Section \ref{subsection_Brownian_motion_without_holonomy}.  
Next, we construct the algebra $\mathfrak S$ (resp. the $\sigma$-algebra $\mathfrak C$) on $\Sigma^{[0,\infty)}$ generated by  all cylinder sets.

We can define, for each $y\in \Sigma,$ a Wiener probability measure $W_y$ on $( \Sigma^{[0,\infty)}, \mathfrak C )$  following formula  (\ref{eq_formula_W_x_without_holonomy}).
Similarly as in Theorem \ref{thm_Brownian_motions},
we can show that the subset $ \Omega(\Sigma)\subset \Sigma^{[0,\infty)} $ has  outer measure\index{measure!outer $\thicksim$} $1$ with respect  to $W_y$ for each $y\in\Sigma.$
Let $\Ac(\Sigma)=\Ac(\Omega(\Sigma))$ be the $\sigma$-algebra on  $\Omega(\Sigma)$  consisting of all  sets $A$ of the  form $A=C\cap \Omega(\Sigma),$
with $C\in\mathfrak C.$ Then we  define  {\it the Wiener measure} at  a  point $y\in \Sigma$ by the  formula:\index{measure!Wiener $\thicksim$ for weakly fibered lamination}
\begin{equation} \label{eq_defi_W_y}
 W_y(A)=W_y(C\cap \Omega(\Sigma)):=W_y(C).
 \end{equation}
$W_y$ is   well-defined on  $\Ac(\Sigma)$ since  $ \Omega(\Sigma)$ has  full outer $W_y$-measure.

Finally, we say that $A\in \Ac(\Sigma)$ is  a {\it  cylinder set} ({\it in} $\Omega(\Sigma)$)  if  $A=C\cap \Omega(\Sigma)$ for  some  cylinder set $C\in \mathfrak C.$
 Similarly as in 
Proposition \ref{prop_cylinder_sets}, we can show the following fact.
\begin{proposition}\label{prop_cylinder_sets_new}
1) If $A$ and $B$ are two cylinder  sets, then $A\cap B$ is  a  cylinder set and $\Omega(\Sigma)\setminus A$
is a  finite  union of mutually disjoint  cylinder sets.
In particular,  the family of all  finite unions of cylinder sets forms an algebra $\mathfrak{S}(\Omega)$ on $\Omega(\Sigma).$
\\
2)  If $A$ is  a countable  union of cylinder sets, then it is also a countable  union of mutually disjoint cylinder sets.
\end{proposition}


\begin{definition}\label{defi_fibered_laminations_new}
     A  weakly fibered lamination $\Sigma$ over $(X,\Lc,g)$ 
   is  said to be  a {\it fibered lamination}\index{lamination!fibered $\thicksim$} 
 if it satisfies
   the following additional  property:
   \\
   (iv)  
for every   set $A\in\Ac(\Sigma),$
 the function $$\Sigma\ni y\mapsto  W_y( A)\in[0,1]$$ is  Borel measurable.

A $\sigma$-finite positive Borel  measure $\mu$ on $X$ is  said to {\it respect} a fibered  lamination
$\Sigma$ over $(X,\Lc,g)$
 if it satisfies
   the following   property:
   \\
   (v)  
for every   set $A\in\Ac(\Sigma),$
 the  image $\tau\circ A\subset \Omega(X,\Lc)$ is  $\bar\mu$-measurable,
  where $\bar\mu$ is  given in (\ref{eq_formula_bar_mu}) and, for a  set  $A\subset \Omega(\Sigma),$  $\tau\circ A:=\{\tau\circ \omega:\ \omega\in A\}.$
  \end{definition} 
  \begin{remark}\label{rem_fibered_laminations_new}
We continue   the discussion with  the   trivial weakly  fibered lamination  $\Sigma:=\widetilde X$ and $\iota:=\id$
given in Remark   \ref{rem_fibered_laminations}. Since  $(X,\Lc,g)$ is a Riemannian continuous-like lamination 
with its  corresponding   covering  lamination projection 
   $\pi:\ (\widetilde X,\widetilde\Lc)\to (X,\Lc),$ it follows  from  Definition  
 \ref{defi_continuity_like} (ii)-(v) that condition (iv) and condition (v)  in Definition \ref{defi_fibered_laminations_new} are fulfilled.
Hence, this  is  a  fibered  lamination which  all $\sigma$-finite positive Borel measures $\mu$ respect. 
\end{remark} 
  In what follows,   let $\iota:\  \Sigma\to \widetilde X$ be  a   fibered  lamination over $(X,\Lc,g).$  Let $\mu$ be  a $\sigma$-finite positive Borel   measure  which  respects  this 
  fibered lamination.
  Let $$\tau:=\pi\circ\iota:\  \Sigma\to X.$$ So $\tau$ maps leaves to leaves and the cardinal of every  fiber 
   $\tau^{-1}( x)$  ($ x\in  X$) is  at most  countable because of Definition \ref{defi_fibered_laminations} (iii) and  of the fact that  the cardinal of every  fiber 
   $\pi^{-1}( x)$  ($ x\in  X$) is  at most  countable (see Definition \ref{defi_covering_measurable_lamination}). 
  \begin{proposition}\label{P:measure_nu_on_Sigma}
  There exists a    $\sigma$-finite positive measure $\nu$ on $(\Sigma, \Bc(\Sigma))$  which is formally defined  by
$\nu:=\tau^*\mu.$
\end{proposition}
 \begin{proof}
By composing the local sections given by Definition \ref{defi_fibered_laminations} (iii)  with those  given by Definition \ref{defi_continuity_like} (i-b),
 we may find  a  family of maps $ s_i:\  E_i\to\Sigma,$  with $ E_i\subset  X,$  indexed by a (at most) countable set $I,$ satisfying
 the  following three properties: 
\begin{itemize}
\item [$\bullet$]  each $ s_i$ is a  local section
\index{section!local $\thicksim$}  of $\tau,$ that is, $\tau( s_i( x))=x$ for all  $ x\in  E_i$  and $i\in I;$
\item [$\bullet$]  for each $i\in I,$ both  $ E_i$ and $s_i( E_i)$  are Borel sets, and  the surjective  map $ s_i:\ E_i\to s_i(E_i) $ is Borel bi-measurable;\index{measurable!bi-$\thicksim$ map}
\item [$\bullet$] the family $( s_i)_{i\in I}$
   generates all fibers of $\tau,$ that is,
  $$
\iota^{-1}( x):=\left  \lbrace  s_i( x):\   x\in E_i\ \text{and}\ i\in I  \right\rbrace,\qquad  x\in  X.
$$
\end{itemize}
We may assume  without loss of generality that $I=\N.$
Using these  properties, we  define  a countable partition  $(B_i)_{i=0}^\infty$ of $\Sigma$ by
Borel sets  as  follows. Set  $B_0:= s_0(E_0),$ and for $i\geq  1$ set
$B_i:= s_i(E_i)\setminus \bigcup_{j=0}^{i-1} s_j(E_j).$
Since  the restriction of $\tau$ on $B_i\subset  s_i(E_i)$ is one-to-one and onto  its image $\tau(B_i)\subset X,$ we define
$\nu:=\tau^*\mu$  as follows:
\begin{equation}\label{e:formula_tau_*_mu}
\nu(A)=\sum_{i=0}^\infty \nu(A\cap B_i):= \sum_{i=0}^\infty \mu(\tau(A\cap B_i)),\qquad  A\in \Bc(\Sigma).
\end{equation}
  This  is  clearly a well-defined $\sigma$-finite positive measure on $(\Sigma, \Bc(\Sigma)).$
\end{proof}

 Using  formula (\ref{eq_formula_bar_mu}),
   define  the $\sigma$-finite measure
  $\bar\nu:= \tau^*\bar\mu$ on $( \Omega(\Sigma),  \Ac(\Sigma))$  as  follows:
   \begin{equation}\label{eq_formula_bar_mu_new}
   \bar\nu(B):=\int_\Sigma\left ( \int_{\omega\in B\cap   \Omega_y}  dW_y \right ) d\nu(y)= \int_\Sigma W_y( B) d\nu(y) ,\qquad  B\in\Ac(\Sigma). 
\end{equation}
 In fact,   Definition \ref{defi_fibered_laminations_new} (iv) ensures that the  integral on the right hand side is  well-defined.
The following result is  a  generalization of    Proposition \ref{prop_algebras} (ii). Its proof will be  provided  in Appendix  \ref{subsection_algebra_on_a_lamination}.
\begin{proposition} \label{prop_algebras_new}  For every $ A \in  \Ac(\Sigma),$ there  exists a  decreasing sequence    sequences  $(A_n)_{n=1}^\infty,$  
      each $ A_n$  being   a countable union of  mutually disjoint cylinder sets in $\Omega(\Sigma)$ such that $ A\subset  A_n $    and that  $\bar\nu( A_n\setminus  A)\to 0$
      as $n\to\infty.$
      \end{proposition}
For  a  bounded measurable   function  $ f:\  \Sigma\to\R^+,$  
  consider  the  {\it maximal function on fibers}  $M[ f]:\  X\to\R^+$ given  by
  \begin{equation}\label{e:M_f}
  M[f](x):=  \sup_{ y\in\tau^{-1}(x)} f( y),\qquad  x\in X,
  \end{equation}
  with the convention that $M[f](x):=0$ if $\tau^{-1}(x)=\varnothing.$ 
    For  a  bounded measurable   function  $F:\  \Omega(\Sigma)\to\R^+$  on  $(\Omega(\Sigma), \Ac(\Sigma)),$
  consider the  {\it $\ast$-norm}:
   \nomenclature[a9e]{$\Vert \cdot\Vert_*$}{$*$-norm}
  \begin{equation}\label{eq_star_norm}
 \| F  \|_\ast:=\int_X M[ f]d\mu,
  \end{equation}
  where the  function  $  f:\  \Sigma\to\R^+$ is  defined  by
  \begin{equation}\label{e:f_F}
   f(y):= \int_{ \Omega_y} F(\omega)  dW_y( \omega),\qquad y \in\Sigma.
  \end{equation}
  For  a  set $A\in\Ac(\Sigma),$ let $\otextbf_{ A}$ denote  the characteristic  function of $A.$\footnote{
 In this Memoir,  we  use the symbol $\otextbf_{ A}$ (resp. $\chi_B$) to denote  the characteristic  function of  a  subset $A$ of a  sample-path space (resp. of  a  subset $B$ of a measurable
or weakly fibered lamination).}
 \nomenclature[a9d]{$\otextbf_{ A}$ (resp. $\chi_B$)}{characteristic  function of  a  subset $A$ of a  sample-path space (resp. of  a  subset $B$ of a measurable
or weakly fibered lamination)} 
  
 The following result  illustrates the necessity of Definition  \ref{defi_continuity_like} (i-b) and 
 Definition \ref{defi_fibered_laminations} (iii). 

\begin{proposition}\label{P:max_fiber_mesu}
 1)  For  every  bounded measurable   function  $ f:\  \Sigma\to\R^+,$  
    the   maximal function on fibers  $M[ f]:\  X\to\R^+$ given in (\ref{e:M_f})
  is  measurable.
  \\
  2)  For  every  bounded measurable   function  $F:\  \Omega(\Sigma)\to\R^+$  on  $(\Omega(\Sigma), \Ac(\Sigma)),$
  the  function $f$
given in (\ref{e:f_F}) is  bounded and  measurable.
In particular, 
the $\ast$-norm  $\| F  \|_\ast$ is  well-defined.
   \end{proposition}
 \begin{proof}
 By  Definition  \ref{defi_continuity_like} (i-b) and Definition  \ref{defi_fibered_laminations} (iii),
  there is a (at most) countable family of  local sections  of $\tau$ which generates all fibers of $\tau,$ that is,
there is a (at most)  countable family $( s_i)_{i\in I}$ of  Borel measurable maps $ s_i:\  E_i\to\Sigma,$ such that
each $ E_i$ is   a Borel subset of $ X,$ and that
$\tau( s_i( x))= x$ for all  $ x\in E_i,$ and that
$$
\tau^{-1}( x):=\left  \lbrace   s_i( x):\   x\in  E_i\ \text{and}\ i\in I  \right\rbrace,\qquad  x\in  X.
$$
For each $i\in I,$ consider  the measurable  function $f_i:\ X\to\R^+$  given by
$$
f_i(x):=
\begin{cases}
f(s_i(x)), &  x\in E_i;\\
0, &  x\not\in E_i.
\end{cases}
$$
Since  $M[f]=\sup_{i\in I} f_i,$ it follows that  $M[f]$  is  also  measurable. This proves Part 1).

 By  Part 1), the proof of Part 2) is  reduced to showing  that the  function $f$  given by (\ref{e:f_F})
 is measurable. Using  Proposition \ref{prop_simple_functions}, we  only need  to consider the  case  where
$F:=  \otextbf_{ A} $ for some  $A\in\Ac(\Sigma).$ Using  Definition  \ref{defi_fibered_laminations_new}  (iv),
 we  proceed as in the proof of Proposition \ref{prop_measurability_W_x}. Consequently, the measurability of $f$ follows.
\end{proof}
 
 \begin{remark}\label{rem_norm_ast} \rm  Clearly,  
  $\|F\|_\ast=0$ if and only if for $\mu$-almost every $x\in X,$ and  for every $y\in \tau^{-1}(x),$   it holds that $F(\omega)=0$ for $W_y$-almost  every
  $  \omega.$ 
  \end{remark}

  \begin{lemma}\label{lem_mu_nu}
  For  every $A\in  \Ac(\Sigma),$ 
  it holds that 
  $\bar\mu(\tau\circ  A)\leq  \bar \nu  ( A) .$ 
  \end{lemma}
  \begin{proof} Let $A':=\tau\circ A.$ Since, by  our assumption,  the measure $\mu$ respects the   fibered  lamination  $\iota:\  \Sigma\to \widetilde X$   over $(X,\Lc,g),$ it follows from
 Definition \ref{defi_fibered_laminations_new} (v) that $A'$ belongs to the $\bar\mu$-completion
of $ \Ac.$
  Note that for every $x\in X$ and  every $\omega'\in  A'\cap\Omega_x,$ there  exists $y\in  \tau^{-1}(x)$ and  $\omega \in A$
such that $\omega'=\tau\circ\omega.$ 
  Hence, using  Lemma  \ref{lem_change_formula} (ii) we  infer that, for  every $x\in X,$   $$ W_x(A') \leq  \sum_{y\in \tau^{-1}(x)} W_y( A),\qquad x\in X. $$
  Integrating  both sides  of the above  inequality and using (\ref{eq_formula_bar_mu_new}) and using Proposition
\ref{P:measure_nu_on_Sigma}, in particular, the explicit formula   $\nu=\tau^*\mu$ given in (\ref{e:formula_tau_*_mu}), 
 the lemma follows.
  \end{proof}
  
 \begin{lemma}\label{lem_norm_star_vs_norm_bar_nu}
  For  a  measurable   functions  $F,G:\  \Omega(\Sigma)\to\R^+$  on  $(\Omega(\Sigma), \Ac(\Sigma)),$ it holds that
   $$
  \| F +G \|_\ast\leq \| F  \|_\ast +\| G  \|_\ast\quad\text{and}\quad        \| F  \|_\ast\leq \int_{\Omega(\Sigma)} Fd\bar\nu.
  $$
 \end{lemma}
 \begin{proof} 
Consider  the  functions  $f,g:\ \Sigma\to\R^+ $ defined by  
 $$
 f(y):= \int_{ \Omega_y} F(\omega)  dW_y( \omega)\quad \text{and}\quad g(y):= \int_{ \Omega_y} G(\omega)  dW_y( \omega),\quad   y\in \Sigma.
 $$
 The first assertion is  an immediate consequence of 
 the  estimate
 $$
   \sup_{ y\in\tau^{-1}(x)} (f( y) +g(y))  \leq \sup_{ y\in\tau^{-1}(x)} f( y) + \sup_{ y\in\tau^{-1}(x)} g( y),\qquad  x\in X.
 $$
 The  second one   follows from combining (\ref{eq_formula_bar_mu_new})  with
 $$
   \sup_{ y\in\tau^{-1}(x)} f( y)\leq   \sum_{ y\in\tau^{-1}(x)}  f(y),\qquad  x\in X.
  $$
 \end{proof}
 
 For  $t\geq  0$ let $T^t$ be  the  shift-transformation of time $t$\index{shift-transformation} on $\Omega(\Sigma).$
 We   write  $T$ instead  of the shift-transformation of unit-time $T^1.$\index{shift-transformation!$\thicksim$ of unit-time} Recall from  Definition  \ref{defi_totally_invariant_set} that   a  set $A\subset \Ac(\Sigma)$ is  said to be  ($T$-)totally invariant 
 if $A=T^{-1}A =TA.$

   Now   we are in the position to state the  main result of this  section.
 \begin{theorem} \label{thm_totally_invariant_set_in_covering_lamination}
Let $\iota:\ \Sigma\to\widetilde X$  be a fibered lamination over $(X,\Lc,g).$
 Let $\mu$ be    a  very weakly harmonic probability measure
  on $(X,\Lc)$  which respects the  above fibered lamination. Assume in addition that $ \mu$  is  ergodic   on $ (X,\Lc).$
\\
 1)  For any  sets $ A, B\in \Ac(\Sigma),$ we have   
 $$ \left \|\liminf_{n\to\infty}{1\over n}\sum_{k=0}^{n-1} (F\circ T^k)G \right\|_\ast\leq\| F\|_\ast  \| G\|_\ast,
$$
where $F:=\otextbf_{ A}$ and  $G:=\otextbf_{ B}.$   \\
2) If $ A\in \Ac(\Sigma)$ is    $T$-totally invariant,  where $T$ is   the  shift-transformation\index{shift-transformation!$\thicksim$ of unit-time} on $\Omega(\Sigma),$
 then  $  \| \otextbf_{ A}  \|_\ast$ is equal to  either $0$ or $1.$
   \end{theorem}
   We  are inspired by 
Kakutani's method in the proof of  \cite[Theorem 3]{Kakutani}.\index{Kakutani!random ergodic theorem of $\thicksim$}\index{theorem!random ergodic $\thicksim$ of Kakutani}
 \begin{proof}
 Assuming first   Part 1),  we   will prove  Part 2).
 Indeed,  applying  Part 1) to  functions  $F=G= \otextbf_A$ yields that 
$$\left  \|\liminf_{n\to\infty} {1\over n}\sum_{k=0}^{n-1} (\otextbf_A\circ T^k) \otextbf_A \right \|_\ast\leq
 \big \| \otextbf_A \big \|_\ast^2.
$$
Since  $A=  T^{-1}(A),$ the left-hand side is  equal to $ \| \otextbf_A  \|_\ast.$
Hence,   we obtain that $  \| \otextbf_A  \|_\ast\leq  \| \otextbf_A  \|_\ast^2.$ This, coupled with the obvious  inequality  $ \| \otextbf_A  \|_\ast\leq 1$
(since  $\mu$ is a probability measure)
yields that  $ \| \otextbf_A  \|_\ast$ is  equal to either $0$ or $1,$ as  desired.
 
 Recall  from Proposition \ref{prop_cylinder_sets_new}
 that $\mathfrak S(\Omega)$ is the (non $\sigma$-) algebra  generated  by  all cylinder  sets in $\Omega(\Sigma).$
 To prove  Part 1)  we first  assume that $ A, B\in \mathfrak S(\Omega).$ 
 By Part 1) of Proposition \ref{prop_cylinder_sets_new}, each  element of  $\mathfrak S(\Omega)$ may be  represented as   the finite union of  mutually  disjoint  cylinder sets. Therefore, we may write
 $$A:=\bigcup_{p\in P}  A^p  :=\bigcup_{p\in P}  C(\{t_i,  A^p_i\}:  m)\quad\text{and}\quad   B:=\bigcup_{q\in Q}  B^q  :=\bigcup_{q\in Q}  C(\{s_j,  B^q_j\}:  l), $$
where  the cylinder sets  on the right hand sides  are mutually disjoint and the index set $P$ and $Q$ are  finite.
Consequently, for $k\geq  k_0:= s_l,$  we have that  
$$(F\circ T^k)\cdot G=\sum_{p\in P,\  q\in Q}  \otextbf_{C^{p,q}},
$$
where $C^{p,q}$ is the  cylinder set  $ C(\{s_1,  B_1\},\ldots,\{s_l,  B_l\},   \{t_1+k,  A_1\},\ldots,
\{t_m+k,  A_m\}
:  l+m).$
By (\ref{eq_formula_W_x_without_holonomy}), we get, for  every $y\in \Sigma,$ that
\begin{multline*}
W_y(C^{p,q})=\Big ( D_{s_1}(\chi_{ B^q_1} D_{s_2-s_1}(\chi_{ B^q_2}\cdots\chi_{ B^q_{l-1}}  D_{s_l-s_{l-1}}(\chi_{ B^q_l}  D_{t_1+k-s_l}(\chi_{ A^p_1}\\
   D_{t_2-t_1}(\chi_{ A^p_2}\cdots\chi_{ A^p_{m-1}}  D_{t_m-t_{m-1}}(\chi_{ A^p_m}                                   )\cdots)\Big) (y).
\end{multline*}
   Consider the function  $ H:\ \Sigma\to [0,1]$ given  by
   $$
    H(y):=\sum_{p\in P}  D_{t_1}\Big (\chi_{ A^p_1}
   D_{t_2-t_1}(\chi_{ A^p_2}\cdots\chi_{ A^p_{m-1}}  D_{t_m-t_{m-1}}(\chi_{ A^p_m}    )\cdots)\Big)(y)= W_y( \bigcup_{p\in P} A^p).                            
   $$
Consider  also  the linear integral operator $\Dc:\ L^\infty(\Sigma)\to L^\infty(\Sigma)$  given by
\begin{equation*}
 \Dc(  f):=\sum_{q\in Q} D_{s_1}\Big(\chi_{ B^q_1} D_{s_2-s_1}(\chi_{ B^q_2}\cdots\chi_{ B^q_{l-1}}  D_{s_l-s_{l-1}}(\chi_{ B^q_l}  f)\cdots )\Big),\quad  f\in \ L^\infty(\Sigma).
\end{equation*}
Summarizing what has been done so far, we have  shown 
that for every $k\geq  k_0$ and  every $y\in\Sigma,$
$$
\int_{\Omega_y} F(T^k \omega) G( \omega) dW_y=
\Dc(    D_{k-s_l} H)(y),\quad
\text{where}\  H(y)=  W_y( A).$$
Observe that   $ H\leq  K\circ\tau,$ where $K:\  X\to \R^+$ is given by 
\begin{equation}\label{eq_K}
K:=M[ W_{\bullet}( A)],
\end{equation}
 the  function $W_{\bullet}( A):\ \widetilde X\to [0,1]$  being  given by  $\widetilde X\ni \tilde x\mapsto W_{\tilde x}( A).$  
  This,  combined with the previous equality, implies that for all $n>k_0,$
\begin{equation}\label{eq_estimate_norm_ast}
\int_{\Omega_y} {1\over n}\sum_{k=k_0}^{n-1}(F\circ T^k)\cdot G  dW_y \leq     \Dc\left(
 \big(  {1\over n}\sum_{k=k_0}^{n-1}  D_{k-s_l} K\big)\circ \pi\right ) (y),\qquad   y\in \Sigma.
\end{equation}
Since we get from (\ref{eq_K}) that $0\leq K\leq 1,$  it follows that 
$
   \sup_{n\geq 1} {1\over n}\sum_{k=0}^{n-1}D_k K   \leq 1 . $
 On the other hand, by Akcoglu's  ergodic theorem\index{Akcoglu!$\thicksim$'s ergodic theorem}\index{theorem!Akcoglu's ergodic $\thicksim$}  (see Theorem  \ref{lem_Akcoglu}), the sequence 
 ${1\over n}\sum_{k=0}^{n-1}D_k K$ converges $\mu$-almost  everywhere  as $n\to\infty$ 
to  $\int_X Kd\mu=   \big \| \otextbf_A \big \|_\ast.$
Putting these altogether  and  using  the  explicit formula  of $ \Dc,$
we deduce  from Lebesgue's dominated convergence\index{Lebesgue!$\thicksim$ dominated convergence theorem}\index{theorem!Lebesgue dominated convergence $\thicksim$} that, for $\mu$-almost  every $x\in X$ and  for every  $ y\in\tau^{-1}(x)\subset\Sigma,$
\begin{equation}\label{eq_Dc}
\begin{split}
 \lim_{n\to\infty} \Dc\left(
 \big (  {1\over n}\sum_{k=k_0}^{n-1}  D_{k-s_l} K\big )\circ \pi\right ) (y)
&= \lim_{n\to\infty} \Dc\left(
 (  {1\over n}\sum_{k=0}^{n-1}  D_{k-s_l} K)\circ \pi\right ) (y)\\
&=\Dc(  \big \| \otextbf_A \big \|_\ast \cdot\otextbf)(y) ,
\end{split}
\end{equation}
where $\otextbf$ is the  function identically  equal to $1$ on $\Sigma.$
The right hand side is  equal to
$$
 \big \| \otextbf_A \big \|_\ast \Dc( \otextbf) (y)=  \big \| \otextbf_A \big \|_\ast
W_y( B). 
$$ 
This,  coupled  with (\ref{eq_estimate_norm_ast}), implies that,  for $\mu$-almost  every $x\in X$ and  for every  $ y\in\tau^{-1}(x),$
\begin{equation}\label{eq_estimate_limsup}
\limsup_{n\to\infty}\int_{\Omega_y} {1\over n}\sum_{k=0}^{n-1}(F\circ T^k)\cdot G dW_y   \leq  \big \| \otextbf_A \big \|_\ast
W_y( B).
\end{equation}
By Fatou's lemma,\index{Fatou!$\thicksim$'s lemma}  the  left  hand side  is greater  that $\int_{\Omega_y} \liminf_{n\to\infty}{1\over n}\sum_{k=0}^{n-1}(F\circ T^k)\cdot G dW_y.$
Consequently,  Part 1) follows.

It remains  to treat the general case  where $ A, B\in \Ac(\Sigma).$ Recall that
all leaves  of $\Sigma$ are  simply connected. Therefore, by Proposition \ref{prop_algebras_new}, for every $ A,  B\in  \Ac(\Sigma)$ there  exist  two sequences  $(A_n)_{n=1}^\infty$ and $( B_n)_{n=1}^\infty$
  such that  each $ A_n$ (as well as  each $ B_n$) is  a countable union of  elements in $\mathfrak S(\Omega)$ and that $ A\subset  A_n, $  $ B\subset  B_n, $  and $\bar\nu( A_n\setminus  A)\to 0,$
$\bar\nu( B_n\setminus  B)\to 0$
 as $n\to\infty.$
  Fix an  arbitrary $0<\epsilon<1.$ The above discussion  shows that  there exists $n\geq 1$ large enough and $ A',$ $ B'\in\mathfrak S(\Omega)$ such that
 $A'\subset A_n,$ $B'\subset B_n$ and that
  \begin{equation}\label{eq_choice_n_A'_B'}
  \bar\nu( A_n\setminus  A)<{\epsilon\over 4},\ \bar\nu( A_n\setminus  A')<{\epsilon\over 4},\  \bar\nu( B_n\setminus  B)<{\epsilon\over 4},\ \bar\nu( B_n\setminus  B')<{\epsilon\over 4}.
  \end{equation}
  Hence,
  $$\bar\nu( A\setminus  A')\leq   \bar\nu( A_n\setminus  A')     <{\epsilon\over 4}\quad\text{ and} \quad \bar\nu( B\setminus  B')\leq   \bar\nu( B_n\setminus  B')  <{\epsilon\over 4}.$$
   Using this  and  applying  Proposition \ref{prop_algebras} (ii) to  both sets  $A\setminus  A'$ and   $B\setminus  B',$
  we obtain two sets $A'''$ and $B''',$ each of them  being  a countable  union of cylinder sets,
  such that
  $$
   A\setminus  A'\subset A''',\ \bar\nu(A''')<\epsilon/2\quad\text{and}\quad  B\setminus  B'\subset B''',\ \bar\nu(B''')<\epsilon/2.
  $$
  Let $A'':=\tau\circ ( A''')$ and $B'':=\tau\circ ( B''').$    
 Consequently, we deduce  from  Lemma \ref{lem_mu_nu} that
   \begin{equation}\label{eq_A_A'}
 \bar\mu(A'')\leq   \bar\nu  ( A''') <{\epsilon\over 2}\quad\text{and}\quad   \bar\mu(B'')\leq   \bar\nu  ( B''') <{\epsilon\over 2} .
 \end{equation}
 On the one hand, it follows from (\ref{eq_choice_n_A'_B'}) that
 $$\bar\nu( A'\setminus  A)\leq   \bar\nu( A_n\setminus  A)     <{\epsilon\over 4}\quad\text{ and} \quad \bar\nu( B'\setminus  B)\leq  \bar\nu( B_n\setminus  B)  <{\epsilon\over 4}.$$
 This, combined with Lemma   \ref{lem_norm_star_vs_norm_bar_nu}, implies that
 \begin{equation}\label{eq_norm_star_A_vs_A'}
  \|\otextbf_{ A'} \|_\ast \leq  \|\otextbf_A \|_\ast+\|\otextbf_{A'\setminus A} \|_\ast
\leq
\|\otextbf_A \|_\ast +\epsilon/4\quad\text{and}\quad     \|\otextbf_{  B'} \|_\ast
 \leq  \|\otextbf_B \|_\ast +\epsilon/4.
 \end{equation}
 On the  other hand,  
   since $
    \otextbf_A - \otextbf_{  A'}\leq \otextbf_{A'''}\leq   \otextbf_{\tau^{-1}(A'')}  
   $ and  $
    \otextbf_B - \otextbf_{  B'}\leq \otextbf_{B'''}\leq   \otextbf_{\tau^{-1}(B'')}  
   ,$  we  deduce  that, for every $x\in X$ and $y\in\tau^{-1}(x)$ and $n\geq 1,$
\begin{multline*}
\int_{\Omega_y} {1\over n}\sum_{k=0}^{n-1}(F\circ T^k)\cdot G dW_y 
\leq  \int_{\Omega_y} {1\over n}\sum_{k=0}^{n-1}(\otextbf_{  A'}\circ T^k)\cdot \otextbf_{  B'}dW_y\\
+ \int_{\Omega_y} \big ( {1\over n}\sum_{k=0}^{n-1}(\otextbf_{\tau^{-1}A''}\circ T^k) \big ) \otextbf_B 
dW_y +
   \int_{\Omega_y}\big (  {1\over n}\sum_{k=0}^{n-1}(F\circ T^k) \big ) \otextbf_{ B\setminus  B'}
dW_y.
\end{multline*}
Since  $A',\ B'\in\mathfrak{C},$ it follows from the previous  case (see (\ref{eq_estimate_limsup})) that
 the $\limsup_{n\to\infty}$ of first  term  on the right hand side is $\leq \|\otextbf_{ A'} \|_\ast  W_y(  B').$
Since  $F\leq 1,$  the  third term is bounded  from above by 
$$  \int_{\Omega_y}\otextbf_{ B\setminus  B'}
dW_y=W_y ( B\setminus  B')\leq  W_y ( B''')  \leq  W_x(B'').$$
The  second term is  dominated  by  
$$
  \int_{\Omega_y} {1\over n}\sum_{k=0}^{n-1}(\otextbf_{\tau^{-1}A''}\circ T^k) 
dW_y= \int_{\Omega_x}  {1\over n}\sum_{k=0}^{n-1}(\otextbf_{A''}\circ T^k)   
dW_x, 
$$
where we recall that  $x=\tau(y)$  and the  equality holds by Lemma \ref{lem_change_formula}
(i). 
Consequently,
\begin{multline*}
\int_{\Omega_y} \liminf_{n\to\infty}{1\over n}\sum_{k=0}^{n-1}(F\circ T^k)\cdot G dW_y
\leq \limsup_{n\to\infty}\int_{\Omega_y} {1\over n}\sum_{k=0}^{n-1}(F\circ T^k)\cdot G dW_y\\
\leq   \|\otextbf_{ A'} \|_\ast  W_y(  B')+W_x ( B'')
+ \limsup_{n\to \infty}\int_{\Omega_x}  {1\over n}\sum_{k=0}^{n-1}(\otextbf_{A''}\circ T^k) 
dW_x.
\end{multline*}
By  Fatou's lemma,\index{Fatou!$\thicksim$'s lemma}
$$
\limsup_{n\to \infty}\int_{\Omega_x}  {1\over n}\sum_{k=0}^{n-1}(\otextbf_{A''}\circ T^k)dW_x 
\leq  \int_{\Omega_x} \limsup_{n\to \infty} {1\over n}\sum_{k=0}^{n-1}(\otextbf_{A''}\circ T^k) dW_x
$$ 
On the other hand, by  Theorem   \ref {thm_invariant_measures}, the probability measure $\bar\mu$ is $T$-invariant on $\Omega(X,\Lc).$
Applying the Birkhoff ergodic theorem\index{Birkhoff!$\thicksim$ ergodic theorem}
\index{theorem!Birkhoff ergodic $\thicksim$} yields that
 $$\lim_{n\to\infty}{1\over n}\sum_{k=0}^{n-1}(\otextbf_{A''}\circ T^k)(\omega)= R(\omega)$$
 for  $\mu$-almost every $x\in X$ and $W_x$-almost every $\omega\in\Omega_x,$ and that
 $$
 \int_X\big (\int_{\Omega_x} R(\omega)dW_x(\omega)\big) d\mu(x)=   \int_{\Omega} \otextbf_{A''}d\bar\mu=
\bar\mu(A'').
 $$
  Summarizing what has been done so far,  we have  shown that
\begin{eqnarray*}
\left \|\liminf_{n\to\infty}{1\over n}\sum_{k=0}^{n-1} (F\circ T^k)G \right\|_\ast
&\leq&   \|\otextbf_{ A'} \|_\ast  \|\otextbf_{  B'} \|_\ast + \int_X W_x(B'')d\mu(x)\\
&&+\int_X \big ( \int_{\Omega_x} R(\omega)dW_x(\omega)\big)d\mu(x)\\
&=&  \|\otextbf_{ A'} \|_\ast  \|\otextbf_{  B'} \|_\ast + \int_X W_x(B'')d\mu(x)+\bar\mu(A'').
\end{eqnarray*}
Using (\ref{eq_A_A'}), the last line is dominated by $  \|\otextbf_{ A'} \|_\ast  \|\otextbf_{  B'} \|_\ast+\epsilon,$
which is,  in turn, bounded by 
  $  \|\otextbf_A \|_\ast  \|\otextbf_B \|_\ast+4\epsilon$
in virtue of  (\ref{eq_norm_star_A_vs_A'}).
Since $0<\epsilon<1$ is arbitrarily chosen, Part 1) in the  general case  where $ A, B\in \Ac(\Sigma)$ follows. 
\end{proof}

\section{Cylinder laminations and end  of the  proof}
\label{subsection_Stratifications}
Let $(X,\Lc,g)$ be a Riemannian lamination satisfying  the Standing Hypotheses and  set  $\Omega:=\Omega(X,\Lc)$ as usual.
The purpose  of this section is   to complete Step 2  in the proof of Theorem     \ref{th_Lyapunov_filtration}.  Throughout the Memoir,
given a $\K$-finite dimensional vector space $V$ with $K\in\{\R,\C\}$ and a positive integer $k,$
 $\Gr_k(V)$ denotes the Grassmannian of all $\K$-linear subspaces of $V$ of given dimension $k.$ 
When $k=1,$  $\Gr_1(V)$ coincides with the projectivisation $\P V$ of $V.$ 
 \nomenclature[a6]{$\Gr_k(V)$}{Grassmannian of all $\K$-linear subspaces of $V$ of given dimension $k,$ where $V$ is  a  $\K$-finite dimensional vector space and
$\K\in\{\R,\C\}$} 
 \nomenclature[a7]{$\P V$}{projectivisation of a $\K$-finite dimensional vector space $V,$ where
$K\in\{\R,\C\}$} 
   \begin{definition}\label{defi_cylinder_lamination}\rm For every $1\leq k\leq d,$
the {\it   cylinder lamination of rank $k$}\index{lamination!cylinder $\thicksim$} 
\nomenclature[b4b]{$(X_{k,\mathcal A},\Lc_{k,\mathcal A})$}{cylinder lamination of rank $k$ of a cocycle   $\mathcal A$}
of a cocycle   $\mathcal A:\ \Omega(X,\Lc)\times \G\to \GL(d,\R),$  denoted  by $(X_{k,\mathcal A},\Lc_{k,\mathcal A}),$ is  defined  as follows.
The  ambient topological space  of the  cylinder lamination is   $X\times \Gr_{k}(\R^d)$ which is independent of $\mathcal A.$
Its leaves are defined  as  follows.
For a  point $(x,U)\in X\times \Gr_{k}(\R^d)$ and  for every   simply connected   plaque  $K$  of $(X,\Lc)$ 
passing through $x,$ we define 
the plaque $\mathcal K$  of $(X\times  \Gr_{k}(\R^d),\Lc_{k,\mathcal A})  $  passing  through  $(x,U)$  by 
$$
\mathcal K=\mathcal K(K,x,U):=\left\lbrace (y,\mathcal A(\omega,1)U):\ y\in K,\ \omega\in\Omega_x,\ \omega(1)=y,\ \omega[0,1]\subset K   \right\rbrace,
$$  
where we also denote by  $\mathcal A(\omega,1)$ its  induced action on $\Gr_{k}(\R^d), $ that is,
$$
\mathcal A(\omega,1)U:=\left\lbrace \mathcal A(\omega,1)u:\ u\in U  \right\rbrace.
$$ 
\end{definition}
\begin{remark} \label{rem_cylinder_lamination}  Since the local expression of  $\mathcal A$  on flow  boxes  is,   in general,  only measurable,
the  cylinder lamination  $(X\times  \Gr_{k}(\R^d),\Lc_{k,\mathcal A})$ is  a   measurable lamination 
in the  sense of Definition  \ref{defi_measurable_lamination}. Moreover, it  is a continuous  lamination  
in the  sense of Definition  \ref{defi_lamination} if and only if the local expression of  $\mathcal A$  on flow  boxes  is continuous,
that is, $\mathcal A$ is $\Cc^0$-smooth. 

 Let  $\Omega_{k,\mathcal A}:=\Omega  (X_{k,\mathcal A},\Lc_{k,\mathcal A}).$ 
 Clearly, when $k=d$ we have that $(X,\Lc)\equiv (X_{d,\mathcal A}, \Lc_{d,\mathcal A}).$
\nomenclature[c1a]{$\Omega_{k,\mathcal A}$}{sample-path space associated to  the cylinder lamination 
$(X_{k,\mathcal A},\Lc_{k,\mathcal A})$}

Note that
 the projection on the  first factor $\pr_1:\ X\times \Gr_{k}(\R^d)\to X $   maps $\mathcal K$ onto  $K$  homeomorphically.
 We endow  the plaque $\mathcal K$ with  the  metric  $(\pr_1|_{\mathcal K})^*(g|_K).$
 By this  way, the leaves  of $(X\times   \Gr_{k}(\R^d),\Lc_{k,\mathcal A})  $ are  equipped with the  metric $\pr_1^*g,$
and hence $(X\times  \Gr_{k}(\R^d),\Lc_{k,\mathcal A},\pr_1^*g )  $ is a  Riemannian  measurable lamination. The  Laplacian
and  the  one  parameter  family  $\{D_t:\ t\geq 0\}$ of the  diffusion operators  are defined  using  the leafwise
metric\index{leafwise!$\thicksim$ metric} $\pr_1^*g.$
\end{remark}

 In what follows, let  $\mathcal{A}:\ \Omega\times \N \to  \GL(d,\R)      $ be     a (multiplicative) cocycle.
 Now  we  discuss  the notion of  saturations.
The {\it (leafwise) saturation}  of a set $Z\subset Y$ in  a  measurable  lamination $(Y,\Lc)$ (resp. in a  weakly  fibered lamination
$\iota:\ Y\to\widetilde X$) is  the
leafwise  saturated  set 
$$\Satur(Z):=\cup_{y\in Z} L_y.$$
\nomenclature[b9e]{$ \Satur(Z)$}{(leafwise) saturation  of a set $Z$  in  a  measurable (resp.  weakly  fibered) lamination}
For a  set $\Sigma\subset  X\times  \Gr_{k}(\R^d),$  the  {\it saturation}  of $\Sigma$  with respect  to
the cocycle $\mathcal A$ 
 is the saturation of  $\Sigma$  in the lamination $(X_{k,\mathcal A},\Lc_{k,\mathcal A}).  $ 

   We have the following  natural identification.
\begin{lemma}\label{lem_identifications_spaces}
 The transformation  $\Omega_{k,\mathcal A}\to  \Omega \times \Gr_k(\R^d)$ 
which maps $\eta$ to  $(\omega,U(0)),$  
where  $\eta(t)=(\omega(t), U(t)), $  $t\in [0,\infty),$
is  bijective.
 \end{lemma}
\begin{proof} The   identification,
follows  from  the fact that $\eta$ is  uniquely determined  in terms of $\omega$ and $U(0).$
Indeed,  we have that
$$\eta(t)= (\omega(t), U(t))=  \big (\omega(t), \mathcal A(\omega,t)(U(0)) \big).
$$
 \end{proof}

Let $T$ be  as  usual the  shift-transformation of  unit-time\index{shift-transformation!$\thicksim$ of unit-time} on $\Omega_{k,\mathcal A}.$ Following Definition \ref{defi_totally_invariant_set},
a  set $\widehat F\subset \Omega_{k,\mathcal A}  $  is said to be {\it $T$-totally invariant} if
$T\widehat F=T^{-1}\widehat F=F.$ 
Using Lemma \ref{lem_identifications_spaces}, we may define $T$ and $T^{-1}$ on 
$\Omega\times   \Gr_{k}(\R^d)$  as follows:
\begin{eqnarray*}
T(\omega,u)&:= & (T\omega,\mathcal A(\omega,1)u),\qquad (\omega,u)\in  \Omega\times \Gr_{k}(\R^d);\\
T\widehat F&:=&\{ T (\omega,u):\ (\omega,u)\in \widehat F\}\quad\text{and}\quad
T^{-1}\widehat F:=\{ (\omega,u):\  T(\omega,u)\in \widehat F\}.
\end{eqnarray*}  
 Here $\widehat F$ is a subset of $\Omega\times   \Gr_{k}(\R^d).$ 
 Given  a  set  $\widehat F\subset \Omega\times \Gr_{k}(\R^d),$ let $F$ be  the  projection of $\widehat F$ onto  the  first factor,
that is, $$F:=\{\omega\in\Omega:\ \exists u\in\P(\R^d): (\omega,u)\in \widehat F\}.$$
We  see easily that if   $\widehat F$ is  $T$-totally invariant, so is  $F.$

By   Theorem \ref{thm_selection_Grassmannian} below there is  a bimeasurable bijection  between the $\mathcal A$-invariant
sub-bundle  $Y\ni x\mapsto V_m(x)$ of rank $d_m$  and $Y\times \R^{d_m}$ covering  the identity and which is  linear on fibers.
Therefore, we may assume without loss of generality that $V_m(x)=\R^{d_m}$  everywhere in Step 2 of  the  proof of Theorem      \ref{th_Lyapunov_filtration}.        By Corollary \ref{cor_leafwise_Oseledec} (iii),
 we have,  for every $x\in Y$ and   for every $v\in V_m(x),$ that 
\begin{equation}\label{eq_Step1}
  \esup_{\omega\in \Omega_x(Y)}  \limsup_{n\to\infty} {1\over n} \log\|  \mathcal A(\omega,n) v \| =\chi_m .
\end{equation}
  Consider the following   measurable set $\widehat F\subset  \Omega\times \P(\R^{d_m}):$
\begin{equation}\label{eq_total_invariant_set_F}
 \widehat F:=\left\lbrace   (\omega,u)\in  \Omega\times  \P(\R^{d_m}):\    
 \chi(\omega, u)<\chi_m\right\rbrace.
 \end{equation} 
In what follows, we  identify      $\Gr_1(\R^{d_m})$ with $\P(\R^{d_m}),$
and  we also  identify  a vector $u\in \R^{d_m}\setminus \{0\}$ with its  image $[u]\in \P(\R^{d_m})$
under the  canonical projection $[\cdot]:\  \R^{d_m}\setminus \{0\}\to \P(\R^{d_m}).$
   Pick an  arbitrary  $(\omega,u)\in \widehat F.$  So 
$  \chi(\omega,u)<\chi_m.$  
  Let  $\eta\in \Omega$ be an arbitrary path such that $T\eta=\omega$  
 and  choose $v\in \P(\R^{d_m})$ such that $
 u= {\mathcal A}(\eta,1)v.$ We  infer from (\ref{eq_invariance_chi}) that $\chi(\eta, v)= \chi(\omega, u)<\chi_m.$ Hence, $(\eta,v)\in \widehat F.$
So we have  just  shown that  $T^{-1}\widehat F\subset \widehat F.$
Similarly, we  also  obtain that $T\widehat F\subset \widehat F.$
 
In summary,  we get that $T^{-1}\widehat F\subset \widehat F$ and  $T\widehat F\subset\widehat F.$ So  $\widehat F=T^{-1}\widehat F= T\widehat F,$ that is, $\widehat F$ is  $T$-totally invariant.
So the image $ F:=\pi_1 (\widehat F)$ of  $\widehat F$ onto  the  first factor $\Omega$
is $T$-totally invariant.

For  each integer $1\leq  k\leq d_m$   let   
\begin{equation}\label{eq_Nc_k}
\Nc_k:=\left\lbrace x\in X:\ \exists U\in \Gr_k(\R^{d_m})\ \text{such that}\  W_x(\Fc_{x,U})>0 \right\rbrace,
\end{equation}
where, for  each point $x\in  X$  and each vector subspace of dimension $k$ in $\R^{d_m}$   $U\in\Gr_k(\R^{d_m}),$   
\begin{equation}\label{eq_Fc_k}
\Fc_{x,U}:=\left\lbrace \omega\in\Omega_x:\   \ \forall u\in U\setminus \{0\}:\ (\omega,u)\in    \widehat F \right\rbrace,
\end{equation}
where   $\widehat F$ is given by   (\ref{eq_total_invariant_set_F}). 
Note that
$$
\Nc_{d_m}\subset  \Nc_{d_m-1}\subset\cdots\subset  \Nc_1\subset X.
$$
In  what follows, let $\mu$ be a  harmonic probability measure on $X.$

\begin{lemma}\label{lem_sup_measurable} Let  $k$ be an integer with $1\leq k\leq d_m.$ 
\\
1) Then the  map
$M_k:\ X\times\Gr_k(\R^{d_m})\to [0,1]$ given by
$$ M_k(x,U):=  W_x(\{ \omega\in \Omega:\ \forall u\in U:\ (\omega,u)\in \widehat F\}),$$
is  Borel measurable,\index{measurable!Borel $\thicksim$ function (or map)}\index{Borel!$\thicksim$ measurable function (or map)}
\index{map!Borel measurable $\thicksim$} and the   
 map $N_k:\  X\to  [0,1]$  given by
$$ N_k(x):= \sup_{U\in\Uc_k(x)}   W_x(\Fc_{x,U}),\qquad  x\in X, 
$$
is  $\mu$-measurable.
\\
2)   $\Nc_k$ is $\mu$-measurable and 
$\Vol_{L_x}(L_x\cap \Nc_k)>0$ for every $x\in\Nc_k.$
\\
3) $\Satur(\Nc_k)$ is   $\mu$-measurable and there is a leafwise saturated Borel set $E$ such that
$\Satur(\Nc_k)\subset E$ and that $\mu(E\setminus\Satur(\Nc_k))=0.$ 
\\
4) If $\mu(\Nc_k)=0$ if and only if  $\mu(\Satur(\Nc_k))=0.$
\end{lemma}
\begin{proof}
By   Theorem   \ref{T:measurable_selection},   we may find  $k$ Borel measurable   functions $b_1,\ldots, b_k:\  \Gr_k(\R^{d_m})\to  \Gr_1(\R^{d_m})$ such that
for each  $U\in \Gr_k(\R^{d_m}),$ the  $k$  lines $b_1(U),\ldots,b_k(U)$ span  $U.$
For   every $1\leq i\leq k$ consider  the following subset of $\Omega\times \Gr_k(\R^{d_m}):$
$$
F_i:=\left\lbrace (\omega,U)\in \Omega\times \Gr_k(\R^{d_m}):  \chi(\omega, b_i(U))<\chi_m   \right\rbrace.
$$
Since  $\chi$ and  $b_i$ are  measurable  functions,  each $F_i$ is  a  measurable  subset of $\Omega\times \Gr_k(\R^{d_m}).$
Let 
\begin{equation}\label{eq_F_0}
F_0:= \cap_{i=1}^k F_i.
\end{equation}
 So  $F_0$ is  also measurable.
Observe  that
$$
M_k(x,U)=  W_x(\{ \omega\in \Omega:\ (\omega,U)\in F_0\}),\qquad (x,U)\in  X\times\Gr_k(\R^{d_m}).
$$
On the other hand,   the  function $X\times\Gr_k(\R^{d_m}) \ni(x,U)\mapsto  W_x(\{ \omega\in \Omega:\ (\omega,U)\in F_0\})$ is Borel measurable\index{measurable!Borel $\thicksim$ function (or map)}\index{Borel!$\thicksim$ measurable function (or map)}
\index{function!Borel measurable $\thicksim$}
by  Proposition  \ref{prop_measurability_W_x}.  Consequently, $M_k$ is Borel measurable.

To prove  that  $N_k$ is $\mu$-measurable 
observe  that  $$W_x(\Fc_{x,U})=W_x(\{ \omega\in \Omega:\ (\omega,U)\in F_0\}).$$ 
We deduce that
$$
N_k(x)=\sup_{U\in \Gr_k(\R^{d_m})} W_x(\{ \omega\in \Omega:\ (\omega,U)\in F_0\}).
$$
Recall that the  function $X\times\Gr_k(\R^{d_m}) \ni(x,U)\mapsto  W_x(\{ \omega\in \Omega:\ (\omega,U)\in F_0\})$ is  measurable.
Consequently,   applying  Lemma   \ref{lem_sup_measurable_book}    to   the  last  equality yields that
  $N_k$ is $\mu$-measurable. This  completes Part 1).
 
Since $\Nc_k=\{x\in X:\ N_k(x)>0\},$ the  $\mu$-measurability of $\Nc_k$ follows from Part 1).
To prove the other assertion of Part 2),  fix $x_0\in \Nc_k$ and $U_0\in \Gr_k(\R^{d_m})$ such that
$\Fc_{x_0,U_0}>0.$
In other words,
  \begin{equation}\label{eq_W_x_0_positive}
W_{x_0}\Big ( F_0\cap \Omega((L_{k,\mathcal A})_{( x_0, U_0)} ) \Big) >0,
\end{equation}
where $F_0$ is  given in (\ref{eq_F_0}) and   $(L_{k,\mathcal A})_{( x_0, U_0)}  $ denotes the leaf of $ (X_{k,\mathcal A},\Lc_{k,\mathcal A})$ passing through the point  
 $(x_0, U_0)$ and  $\Omega((L_{k,\mathcal A})_{( x_0, U_0)}  )$ denotes the  space of all  continuous  paths $\omega$ defined on $[0,\infty)$ with image fully contained in  this leaf.

Since  $\widehat F$ given in (\ref{eq_total_invariant_set_F})  is $T$-totally invariant, it is  easy to see that so is $F_0.$  
Consequently, applying  Proposition \ref{prop_Markov} (i) to inequality (\ref{eq_W_x_0_positive}) yields that for $L:=L_{x_0},$ 
$$
\Vol_L\Big(\left\lbrace   x\in  L:\  W_x\big( F_0\cap \Omega((L_{k,\mathcal A})_{( x_0, U_0)})\big) >0\right\rbrace\Big)  >0,
$$
proving  the last assertion of Part  2).

The proof of Part 3) and Part 4)  will be provided in Appendix  \ref{subsection_harmonic_measures}.
\end{proof}

\begin{lemma}
\label{lem_Nc_0_equal_0} If $\chi_m\not=\lambda_l,$ then
  $\mu(\Nc_{d_m})=0.$
\end{lemma}
\begin{proof}
By Theorem  \ref{th_Lyapunov_filtration_Brownian_version} and by Step  1, there is  $1\leq  s\leq  l$ such that  $\chi_m=\lambda_s.$
Our  assumption  $\chi_m\not=\lambda_l$  implies that $s<l.$ Consequently, for $\mu$-almost every $x\in X,$ 
for  $W_x$-almost  every  $\omega\in \Omega_x$ we have that
$\chi(\omega,u)=\lambda_s=\chi_m$ for all $u\in V_s(\omega) \subset\R^{d_m}.$ As $ V_s(\omega)\not=\{0\},$
we infer from  
(\ref{eq_Nc_k})
 and (\ref{eq_Fc_k})
  that  $\mu(\Nc_{d_m})=0.$
\end{proof}
For  each integer $1\leq  k\leq  d_m$ and each point $x\in X,$ let   
    \begin{equation}\label{eq_Uc_k}
\Uc_k(x):=\left\lbrace  U\in \Gr_k(\R^{d_m}) :\    W_x(\Fc_{x,U})>0 \right\rbrace.
\end{equation}
Now  we arrive  at the  following stratifications.
\begin{lemma}\label{lem_stratifications}
Let $x\in X\setminus \Nc_{k+1}.$ Then  for every $U,V\in \Uc_k(x)$ with $U\not=V,$
it holds that $W_x(\Fc_{x,U}\cap \Fc_{x,V})=0.$
In particular, $0<\sum_{U\in \Uc_k(x)}  W_x(\Fc_{x,U})\leq 1$ and hence  the  cardinal of  $\Uc_k(x)$
is  at most countable.
 \end{lemma}
 \begin{proof}
Suppose  that  there  exist $U,V\in \Uc_k(x)$ such that $U\not=V,$
and  that $W_x(\Fc_{x,U}\cap \Fc_{x,V})>0.$
Let $W$ be  the vector space  spanned by both $U$ and $V.$ Since $U\not=V,$ $W$  is  of dimension $\geq  k+1.$
Let $w$ be an arbitrary element in $ W.$  So we may find  $u\in U$ and $v\in V$ such that $w=u+v.$
Arguing as  in the  proof of  Proposition \ref{prop_chi} (iii), we get that
$$\chi(\omega,u+v)\leq  \max \{\chi(\omega,u),\chi(\omega,v)\},\qquad \omega\in\Omega_x.$$
Consequently,  for  every $\omega\in \Fc_{x,U}\cap \Fc_{x,V},$ we infer that
$$
\chi(\omega,w)=\chi(\omega,u+v)\leq  \max \{\chi(\omega,u),\chi(\omega,v)\}<\chi_m.
$$
Hence,   $\Fc_{x,U}\cap \Fc_{x,V}\subset \Fc_{x,W}.$
This, combined  with  the  assumption  that $W_x(\Fc_{x,U}\cap \Fc_{x,V})>0,$ implies that $W_x(\Fc_{x,W})>0,$
that is,
 $x\in  \Nc_{k+1},$ which contradicts the hypothesis.   Hence, the first assertion of the lemma follows.

The  second  assertion follows   from the  first  one since $W_x$ is  a  probability measure.
To prove  that  the  cardinal of  $\Uc_k(x)$
is  at most countable,  consider,  for each $N\geq 1,$  the  following  subset of $\Uc_k(x):$
$$
\Uc^N_k(x):=  \left\lbrace  U\in \Uc_k(x):\  W_x(\Fc_{x,U})>1/N \right\rbrace .
$$
By the  first  assertion, the  cardinal  of $\Uc^N_k(x)$ is  at most $N.$ Since
 $\Uc_k(x)=\bigcup_{N=1}^\infty\Uc^N_k(x),$ the  last assertion   of the lemma follows.
\end{proof}

\noindent{\bf End of the proof of Step 2  of  Theorem    \ref{th_Lyapunov_filtration}.}
 Recall from Theorem   \ref{th_Lyapunov_filtration_Brownian_version}  that  $\chi(\omega,u)\in\{\lambda_1,\ldots,\lambda_l\}$ 
for $W_x$-almost every
$\omega$  and  for $\mu$-almost  every $x\in X.$  Moreover, $\lambda_l<\cdots <\lambda_1.$
  Consequently,   
if $\chi_m=\lambda_l,$ then  by   Theorem   \ref{th_Lyapunov_filtration_Brownian_version},  
for $\mu$-almost  every $x\in X,$   we have that $\chi(\omega,v)=\lambda_l=\chi_m$ for all $v\in V_m(x)$ and that $V_l(\omega)=V_m(x)$ 
for $W_x$-almost every
$\omega.$ Hence,  Step  2 is  finished.  Therefore, in the  sequel  we  assume  that $\chi_m\not=\lambda_l.$
Consequently, by  Lemma \ref{lem_Nc_0_equal_0} we get that
  $\mu(\Nc_{d_m})=0.$

In the   remaining  part of the  proof  we let $k$ descend from $d_m-1$ to $1.$
So we begin  with $k=d_m-1$  and recall that  $\mu( \Nc_{d_m})=0.$ The remaining proof is  divided into two sub-steps.

\noindent {\bf Sub-step 1:}
{\it   If   $\mu( \Nc_{k+1})=0$ and $k \geq 1,$ then $\mu(\Nc_k)=0.$}

Suppose  in order to reach a contradiction that $\mu(\Nc_k)>0.$ Let  $\pi:\ (\widetilde X,\widetilde\Lc)\to (X,\Lc)$ be the  covering lamination projection  of  $(X,\Lc),$ and  set
 $\widetilde\Omega:=  \Omega(\widetilde X,\widetilde\Lc).$
Let 
\begin{equation}\label{eq_Sigma_k}
\Sigma_k:=\left\lbrace (\tilde x,U)\in  \pi^{-1}(\Nc_k\setminus \Satur(\Nc_{k+1})) \times \Gr_k(\R^{d_m}):\ U\in\Uc_k(\pi(\tilde x)) \right\rbrace,
\end{equation}
where $\Uc_k(\pi(\tilde x))$ is  defined  by (\ref{eq_Uc_k}).

We construct a  cocycle  $\widetilde{\mathcal A}$ on $(\widetilde X,\widetilde\Lc)$ as  follows:
\begin{equation}\label{eq_cocycle_on_cover}
\widetilde{\mathcal A}(\tilde\omega,t):=\mathcal A(\pi\circ \tilde\omega,t),\qquad t\in\R^+,\ \tilde\omega\in\widetilde\Omega .
\end{equation}
Consider the cylinder lamination of rank $k$ $(\widetilde{X}_{k,\widetilde{\mathcal A}},\widetilde{\Lc}_{k,\widetilde{\mathcal A}})  $
of the  cocycle  $
\widetilde{\mathcal A}.$ Note that
$\Sigma_k\subset \widetilde{X}_{k,\widetilde{\mathcal A}} =\widetilde X\times \Gr_k(\R^{d_m}).$
Let $\overline\Sigma_k$ be  the  saturation  of  $\Sigma_k$  in this measurable lamination.

Before going further, we make the following modification on $\Sigma_k.$
By Part 2) of  Lemma \ref{lem_sup_measurable}, we may shrink  $\Nc_k$  a  little  bit so that $\mu(\Nc_k)$ does  not change
and that  $\Nc_k$ is a Borel set. By Part 3) of  the same lemma, we may add to $\Satur(\Nc_{k+1})$ a  set of null $\mu$-measure such that
$\Satur(\Nc_{k+1})$ is a Borel set.
So  we may assume  that $\Nc_k\setminus \Satur(\Nc_{k+1})$ is a  Borel set.
Moreover, by Part 4) of Proposition \ref{prop_current_local_consequence}, we may shrink the  last set a  little bit
so that  its $\mu$-measure does not change and that  its saturation is a Borel set.
Putting this  discussion together with  (\ref{eq_Sigma_k}), (\ref{eq_Uc_k}) and  the measurability of $M_k$ stated in Part 1) of
  Lemma \ref{lem_sup_measurable},    we may  assume  without loss of generality that
  \begin{equation}\label{eq_Borel_Sigma}
  \Sigma_k\  \text{and}\ \overline\Sigma_k\ \text{are Borel subsets of}\ \widetilde X\times \Gr_k(\R^{d_m}).
\end{equation}
 The projection of $\overline\Sigma_k\subset \widetilde X\times \Gr_k(\R^{d_m})$ onto the  first factor is denoted by
$\pr_1(\overline\Sigma_k).$ 
  Since $\mu( \Nc_{k+1})=0$ and $\mu( \Nc_k)>0,$ if follows from Part 4) of Lemma \ref{lem_sup_measurable} that
$\mu(\Nc_k\setminus \Satur(\Nc_{k+1}))>0.$
On the other hand, $\pi(\pr_1(\overline\Sigma_k))\subset X$ 
  is  equal to the  saturation of $\Nc_k\setminus \Satur(\Nc_{k+1})$ in the lamination $(X,\Lc).$
  Putting all these together and  noting  that $\mu$ is  ergodic, we infer that
the leafwise saturated set $\pi(\pr_1(\overline\Sigma_k))\subset X$ 
  is of full $\mu$-measure.

For  any point  $(\tilde x,U)\in\overline\Sigma_k,$ let $(\overline\Sigma_k)_{(\tilde x,U)}$
be  the  saturation 
of  $(\tilde x,U)$ in  $(\widetilde{X}_{k,\widetilde{\mathcal A}},\widetilde{\Lc}_{k,\widetilde{\mathcal A}}),  $  
and  we   call it  the {\it  leaf}  of $\overline\Sigma_k$ passing  through $(\tilde x,U).$
 Since this leaf  is  also a leaf in the  cylinder lamination $(\widetilde{X}_{k,\widetilde{\mathcal A}},\widetilde{\Lc}_{k,\widetilde{\mathcal A}}),  $ 
 it is  endowed  with  the natural metric $\pr^*_1 \tilde g,$ where $\tilde g:=\pi^*g.$
 Let $\Omega (\overline\Sigma_k)$ be the  space of all  continuous  paths $\omega:\ [0,\infty)\to\overline\Sigma_k$ with image fully contained in a single leaf. Using the canonical  identification given by Lemma \ref{lem_identifications_spaces} we identify $\Omega (\overline\Sigma_k)$ with 
 a  subspace  of 
    $\widetilde\Omega \times \Gr_k(\R^d).$

  Let   
 \begin{equation}\label{eq_widehat A_k}
\widehat A_k:=\left\lbrace (\omega, U)\in\Omega \times \Gr_k(\R^{d_m}) :\   (\omega,U)\subset    \widehat F \right\rbrace.
\end{equation}
where $ \widehat F$ is given  by (\ref{eq_total_invariant_set_F}).
We can  easily  show that $\widehat A_k$ is  $T$-totally invariant. Moreover, it is  also measurable   since  so is  $\widehat F.$ 
Let 
\begin{equation}
\label{eq_formula_widehat_widetilde_A_k}
\widehat{\widetilde A}_k:= \left\lbrace(\tilde \omega,U)\in \widetilde\Omega \times \Gr_k(\R^{d_m}):\
(\pi\circ\tilde\omega,U)\in \widehat A_k\ \text{and}\ (\tilde \omega(0),U)\in \overline\Sigma_k  \right\rbrace.
\end{equation}

\begin{proposition}\label{prop_Sigma_k_is_a_weakly_fibered_lamination} Suppose  as in Sub-step 1 that  $\mu( \Nc_{k+1})=0$ and $\mu( \Nc_k)>0.$
\\ 
1) Then  $\widehat{\widetilde A}_k$
is  a $T$-totally invariant measurable subset of  $\Omega(\overline\Sigma_k).$
\\
2)
 $\iota:\ \overline\Sigma_k\to\widetilde X$ is  a  weakly fibered  lamination over $(X,\Lc,g),$
where $\iota$ is  the  canonical projection onto the  first  factor.
  \end{proposition}
\begin{proof}
Recall  that  $\widehat A_k$ and $\widehat{\widetilde A}_k$ are  defined in  (\ref{eq_widehat A_k}) and (\ref{eq_formula_widehat_widetilde_A_k}) and that
 $\widehat A_k$ is $T$-totally invariant measurable.

To prove  Part 1)   observe that
$\widehat{\widetilde A}_k= \Omega(\overline\Sigma_k) \cap A_k,$
 where 
 \begin{equation*}
 A_k:=\left\lbrace (\tilde\omega,U)\in \widetilde\Omega \times \Gr_k(\R^{d_m}):\ (\pi\circ\tilde\omega, U)\in \widehat A_k \right\rbrace.
\end{equation*}
 Since $\widehat A_k$ is $T$-totally invariant measurable, so are $A_k$ and $\widehat{\widetilde A}_k.$
 
Now  we  turn to Part 2).
Recall from  the construction of   $\overline\Sigma_k,$
that it is the saturation of  $\Sigma_k$ in the cylinder lamination $(\widetilde{X}_{k,\widetilde{\mathcal A}},\widetilde{\Lc}_{k,\widetilde{\mathcal A}}) . $
Consequently, we only need  to check Definition \ref{defi_fibered_laminations} (iii).
For the moment, we only give  the proof of the following  weaker  statement:
\begin{equation}\label{e:coutable_fiber}
\textit{The  cardinal  of every fiber  $\iota^{-1}(\tilde x)$  ($\tilde x\in \widetilde X$) is  at most  countable.}
\end{equation}
The verification of the whole Definition \ref{defi_fibered_laminations} (iii) will be  provided in Appendix \ref{subsection_fibered_laminations_new}.
 To prove the above  statement,  fix  an arbitrary point $\tilde x_0\in\widetilde X$ and
let $\iota^{-1}(\tilde x_0)=\{ (\tilde x_0, U_i):\ i\in I   \}.$ We need to show that the index set $I$ is at most countable.
Suppose without loss of generality that $\tilde x_0\not\in\pi^{-1}(\Satur(\Nc_{k+1}))$ since otherwise $\iota^{-1}(\tilde x_0)=\varnothing$
by the construction of $\Sigma_k$ and $\overline{\Sigma}_k.$
For each $i\in I$  there exists  $(\tilde x_i,V_i)\in \Sigma_k$ such that 
 $(\tilde x_i,V_i)$ on the  same leaf as $ (\tilde x_0, U_i)$ in $\overline\Sigma_k.$
The membership  $(\tilde x_i,V_i)\in \Sigma_k$ implies, by the  definition of $\Sigma_k,$ that
$$
W_{\pi(\tilde x_i)}(\Fc_{\pi(\tilde x_i), V_i} ) >0,\qquad i\in I.
$$
  In other words,
  \begin{equation}\label{e:W_positive}
W_{\pi(\tilde x_i)}\Big(\widehat{\widetilde A}_k\cap \Omega((\overline\Sigma_k)_{(\tilde x_0, U_i)} )\Big) >0,\qquad i\in I,
\end{equation}
where $(\overline\Sigma_k)_{(\tilde x, U)}$ denotes the leaf of $\overline\Sigma_k$ passing through the point $(\tilde x, U),$
and  $\Omega((\overline\Sigma_k)_{(\tilde x, U)} )$ denotes the  space of all  continuous  paths $\omega$ defined on $[0,\infty)$ with image fully contained in  this leaf.
By Part 1) the set $\widehat{\widetilde A}_k\cap \Omega\big((\overline\Sigma_k)_{(\tilde x_0, U_i)}\big)$ is $T$-totally invariant.
Consequently, applying  Proposition \ref{prop_Markov} (i) yields that for each $i\in I,$
$$
\Vol\Big(\left\lbrace  \tilde x\in\widetilde L:\  W_{\tilde x}\Big(\widehat{\widetilde A}_k\cap \Omega((\overline\Sigma_k)_{(\tilde x_0, U_i)} )\Big) >0\right\rbrace\Big)  >0.
$$
Here $\widetilde L$ is the leaf $\widetilde L_{\tilde x_0}$ passing through $\tilde x_0$ in $(\widetilde X,\widetilde\Lc)$ and
$\Vol$ is the Lebesgue measure  induced by the metric  $\tilde g:=\pi^*g$ on  the leaf $\widetilde L.$

Next, we cover $\widetilde L$ by a   countable  family of open sets $(O_n)_{n=1}^\infty$ such that
$0<\Vol(O_n)<1$ for each $n.$
The  previous  estimates show that  for each $i\in I,$ there  is  an integer $n\geq 1$ such that
\begin{equation}\label{eq_coutability_I}
0<\int_{O_n}  W_{\tilde x}\Big(\widehat{\widetilde A}_k\cap \Omega((\overline\Sigma_k)_{(\tilde x_0, U_i)} )\Big) d\Vol(\tilde x)\leq \Vol(O_n)<1.
\end{equation}
Note that  $\widetilde L\cap \pi^{-1}(\Nc_{k+1})=\varnothing $ because
  $\tilde x_0\not\in\pi^{-1}(\Satur(\Nc_{k+1})).$
Consequently,
we deduce from Lemma \ref{lem_stratifications}
 that, for each $n\geq 1$ and for each $\tilde x\in O_n,$
 \begin{equation}\label{e:sum<1}
 \sum_{i\in I} W_{\tilde x}\Big(\widehat{\widetilde A}_k\cap \Omega((\overline\Sigma_k)_{(\tilde x_0, U_i)} )\Big)\leq 1.  
 \end{equation}
 Here we make  the convention that for  a collection $(a_i)_{i\in I}\subset \R^+,$
\begin{equation*}
\sum_{i\in I} a_i:=\sup_{J\subset I,\ J\ \textrm{ finite }}\sum_{j\in J} a_j.
\end{equation*}
 Integrating  the above inequality  over $O_n,$ we get that
 \begin{equation}\label{e:integration_sum}
\sum_{i\in I}\int_{O_n}  W_{\tilde x}\Big(\widehat{\widetilde A}_k\cap \Omega((\overline\Sigma_k)_{(\tilde x_0, U_i)} ) \Big)d\Vol(\tilde x)\leq \Vol(O_n)<1.
\end{equation}
So for each $n\geq 1,$ there is at most a  countable  number of $i\in I$ such that
$$
0<\int_{O_n}  W_{\tilde x}\Big(\widehat{\widetilde A}_k\cap \Omega((\overline\Sigma_k)_{(\tilde x_0, U_i)} )\Big) d\Vol(\tilde x) .
$$
 Using  this and  varying  $n\in \N,$ and combining  them with  (\ref{eq_coutability_I}), the countability of $I$ follows.
\end{proof}

In what follows,  let $\iota:\  \overline\Sigma_k\to \widetilde X$ be  the above  weakly  fibered  lamination over $(X,\Lc,g).$
  Let $\tau:=\pi\circ\iota:\  \overline\Sigma_k\to X.$ So $\tau$ maps leaves to leaves and the cardinal of every  fiber 
   $\tau^{-1}( x)$  ($ x\in  X$) is  at most  countable. 
   Let $\mu$ be a harmonic  probability measure on $(X,\Lc).$
  Consider the $\sigma$-finite measures $\nu:=\tau^*\mu$ on $(\overline\Sigma_k, \Bc(\overline\Sigma_k))$ and  
  $\bar\nu:= \tau^*\bar\mu$ on $( \Omega(\overline\Sigma_k),  \Ac(\overline\Sigma_k))$  as  in  (\ref{eq_formula_bar_mu_new}).

The following stronger  version of Proposition
\ref{prop_Sigma_k_is_a_weakly_fibered_lamination}
will be proved  in  Appendix \ref{subsection_fibered_laminations}   below.
  
\begin{proposition}\label{prop_Sigma_k_is_a_fibered_lamination}
Suppose  as in Sub-step 1 that  $\mu( \Nc_{k+1})=0$ and $\mu( \Nc_k)>0.$
\\
1) Then $\iota:\ \overline\Sigma_k\to\widetilde X$ is  a  fibered  lamination over $(X,\Lc,g);$  
\\ 
2)      $ \mu$ respects this fibered lamination. 
\end{proposition}
Taking for granted Proposition \ref{prop_Sigma_k_is_a_fibered_lamination},
 we resume the  proof  Step 2 of  Theorem    \ref{th_Lyapunov_filtration}.
By Proposition \ref{prop_Sigma_k_is_a_fibered_lamination} and Part 1) of Proposition \ref{prop_Sigma_k_is_a_weakly_fibered_lamination},
we may apply   Theorem  \ref{thm_totally_invariant_set_in_covering_lamination} to $\widehat{\widetilde A}_k.$ This yields 
 that $\|\otextbf_{\widehat{\widetilde A}_k }\|_\ast $ is  either  $0$ or $1.$
  Recall  that  for every $x\in \Nc_k,$ there is $U\in \Gr_k(\R^{d_m})$ such that  $(\omega,U)\subset \widehat F$  for all $\omega\in
 \Fc_{x,U}$ and  that $W_x(\Fc_{x,U})>0.$ 
On the other hand, since     $\mu( \Nc_{k+1})=0$ it follows from  Part 4) of Lemma \ref{lem_sup_measurable} that
$\mu(\Satur(\Nc_{k+1}))=0.$
This, combined  with  the  assumption that
 $\mu( \Nc_k)>0,$ and  formula (\ref{eq_star_norm}) , implies  that 
   $\|\otextbf_{\widehat{\widetilde A}_k }\|_\ast>0. $
Hence, $\|\otextbf_{\widehat{\widetilde A}_k }\|_\ast=1. $ 
  Consequently, we deduce  from  formula  (\ref{eq_star_norm}) again that  for $\mu$-almost every $x\in X,$
$$
\sup_{y\in \tau^{-1}(x)} W_y(  {\widehat{\widetilde A}_k }  )=1.
$$
Using  that $\tau=\pi\circ \iota,$ where $\iota:\ \overline\Sigma_k\to \widetilde X$ is the canonical projection
onto the first factor,  we rewrite  the above  identity as follows
\begin{equation}\label{eq_norm_star_widehat_widetilde_A_k}
\sup_{\tilde x\in \pi^{-1}(x),\ U\in\Gr_k(\R^d)} W_{\tilde x,U}(  {\widehat{\widetilde A}_k }  )=1.
\end{equation}
Note  from  (\ref{eq_formula_widehat_widetilde_A_k}) that 
$(\tilde\omega',U)\in  {\widehat{\widetilde A}_k }$ if and only if $(\tilde\omega'',U)\in  {\widehat{\widetilde A}_k }$ 
for every $\omega\in\Omega$ and $U\in  \Gr_k(\R^d)$ and  $\tilde\omega', \tilde\omega''\in \pi^{-1}(\omega).$ 
Therefore, for every $x\in X$ and  every $x',x''\in \pi^{-1}(x),$ and every $U\in \Gr_k(\R^d),$
$$
W_{x'}\left(\left\lbrace  \tilde \omega\in \widetilde\Omega_{\tilde x'}:\ (\tilde \omega,U)\in \widehat{\widetilde A}_k \right\rbrace
\right)
=
W_{x''}\left(\left\lbrace  \tilde \omega\in \widetilde\Omega_{\tilde x''}:\ (\tilde \omega,U)\in \widehat{\widetilde A}_k \right\rbrace
\right)
$$
because  both members are equal to $W_x(\Fc_{x,U})$ by Lemma \ref{lem_change_formula}.
This, combined  with (\ref{eq_norm_star_widehat_widetilde_A_k}), implies that
$$
\sup_{U\in \Uc_k(x)} W_x(  \Fc_{x,U}  )=\sup_{U\in \Gr_k(\R^d)} W_x(  \Fc_{x,U}  )=1.
$$
Consequently,
there  exists a  sequence  $(U_N)\subset  \Uc_k(x)$  such that  $\lim_{N\to\infty} W_x(\Fc_{x, U_N})=1.$ 
On the  other hand, we infer from $\mu(\Satur(\Nc_{k+1}))=0$  and from Lemma \ref{lem_stratifications} that 
    $W_x(\Fc_{x,U_N}\cap \Fc_{x,U_{N'}})=0$  when   $U_N\not=U_{N'}$  for $\mu$-almost every $x\in X.$
    Consequently, there  exists an element $U\in  \Uc_k(x)$  such that  $  W_x(\Fc_{x, U})=1.$
So  for $\mu$-almost every $x\in  X,$ there  exists  a $k$-dimensional   subspace $U_x$
such that  $\chi(\omega,u) <\chi_m$ for all $u\in U(x)$ and for $W_x$-almost every $\omega.$
 Recall from Theorem   \ref{th_Lyapunov_filtration_Brownian_version}  that  $\chi(\omega,u)\in\{\lambda_1,\ldots,\lambda_l\}$ for 
for $W_x$-almost every
$\omega.$ 
  Consequently,  we deduce  from  the definition of   $\chi(x, u),$  
  that  $\chi(x, u )<\chi_m$ for  $\mu$-almost every  $x\in  Y_0$ and for every  $u\in U(x).$
 But this contradicts the fact that   $\chi(x,u)=\chi_m$ for  $\mu$-almost every $x\in X$ and for every $u\in\P(\R^{d_m}).$
So $\mu(\Nc_k)=0.$ This  completes   Sub-step 1. 

\noindent {\bf Sub-step 2:}
{\it End of the proof.}

We  repeat Sub-step 1  by descending $k$ from $d_m-1$ to $1.$ Finally, we obtain that $\mu( \Nc_1)=0. $
So  for $\mu$-almost every $x\in  X$ and for all $u\in\P(\R^{d_m}),$  
  $\chi(\omega,u) =\chi_m$ for  $W_x$-almost every $\omega.$
  This  completes Sub-step 2.   
The  proof of Step 2 of  Theorem    \ref{th_Lyapunov_filtration} is thereby ended. \hfill $\square$

%% file: chapter8.tex

 
\chapter{Lyapunov backward filtrations}
 \label{section_backward_filtration}
 

 This  chapter is  devoted  to  the construction of   the Lyapunov backward filtrations associated  to 
a  cocycle $\mathcal A$  defined on a lamination $(X,\Lc).$  
The  word ``backward" means that  we go to the past, i.e., the time   $n$ tends to $-\infty.$
We will see that some  important  properties  of
the forward filtrations also hold 
for
 the  backward filtrations. However, the corresponding  proof in the  backward context  is   much harder since
 there are many difficulty and  difference  in comparison with the forward situation.

 \section{Extended sample-path spaces}
 \label{subsection_extended_sample_path_spaces}
   
   The  aim  of this  section is  to introduce  the notion of {\it extended sample-path space} associated to a lamination, and
   to develop a measure theory on this  space. This new theory  may be considered as a  natural extension  to the backward context  of the  measure theory  on sample-path spaces
   which   has  been described in Section   \ref{subsection_Brownian_motion_without_holonomy},  \ref{subsection_Wiener_measures_with_holonomy} and
\ref{subsection_measurability_issue}.
   This  section contains 4  subsections.
   
 \subsection{General context} 
 \label{ss:Wiener-I}
 In this  subsection,  $(X,\Lc,g)$ is a Riemannian  continuous-like lamination satisfying Hypothesis (H1), and 
   $\mu$ is  a very weakly harmonic  probability measure on $(X,\Lc,g)$ which is also ergodic.
Recall first that the shift-transformation  $T$ of unit-time\index{shift-transformation!$\thicksim$ of unit-time} is  an  endomorphism of the  probability space $(\Omega, \Ac, \bar\mu),
$ where  we denote, as usual, $\Omega:=\Omega(X,\Lc),$  $\Ac=\Ac(\Omega)$ and  $\bar\mu$ is the Wiener measure with initial
distribution $\mu$ given by (\ref{eq_formula_bar_mu}).   
 Consider the {\it  natural  extension}\index{extension!natural $\thicksim$}  $(\widehat\Omega,\widehat\Ac, \hat\mu)$ of this  space which  is constructed  as  follows
(see \cite{CornfeldFominSinai}). In the  sequel,  the $\sigma$-algebra $ \widehat\Ac$ (resp. the measure $ \hat\mu$)  is  called the {\it  natural  extension}  of the $\sigma$-algebra $\Ac$
 (resp. the measure $\bar\mu$).

Each element of  $\widehat\Omega$  is   a continuous  path $\hat\omega:\ \R\to X$  with image  fully  contained 
in a  single leaf of $(X,\Lc).$  We  say that $\widehat\Omega$ is the  {\it  extended sample-path space} associated to  $(X,\Lc).$
\nomenclature[c1e]{$\widehat\Omega(X,\Lc)$ or simply $\widehat\Omega$}{extended sample-path space associated to a (continuous or measurable) lamination $(X,\Lc)$}
 Consider  the  group $(T^t)_{t\in\R}$ of shift-transformations\index{shift-transformation} 
  $T^t:\ \widehat \Omega\to\widehat\Omega$ defined for  all $t,s\in\R$ by 
$$   T^t(\hat\omega)(s):=\hat\omega(s+t),\qquad \hat \omega\in \widehat\Omega.$$
Observe that  all $T^t$  are  invertible and $(T^t)^{-1}=T^{-t},$ $t\in\R.$
Consider also the canonical restriction
$\hat\pi:\  \widehat \Omega\to \Omega$   which, to each  path $\hat\omega,$  associates its restriction on $[0,\infty),$ 
that is,
\begin{equation}\label{eq_pi_hat}
\hat\pi(\hat \omega):=\hat\omega|_{[0,\infty)},\qquad  \hat\omega\in \widehat \Omega.
\end{equation}
 For  every $i\in\N,$ consider the  $\sigma$-algebra $\widehat\Ac_{-i}$ consisting of  
all  sets of the form
\begin{equation}\label{eq_A_i_C}
A=A_{i,C}:=\left\lbrace \hat\omega \in \widehat\Omega:\  \hat \pi (T^i\hat\omega)\in C \right\rbrace,
\end{equation}
 where  $C\in \Ac.$ In other words,  
\begin{equation}\label{eq_A_i_C_new}
\widehat\Ac_{-i}=(\hat \pi\circ T^i)^{-1}\Ac.
\end{equation} 
  \nomenclature[e3]{$\widehat\Ac_{-i},$ $i\in\N$}{$\sigma$-algebra on $\widehat\Omega$ given by $(\hat \pi\circ T^i)^{-1}\Ac,$
where $\hat\pi:\  \widehat \Omega\to \Omega$ is the canonical restriction}
 So  $(\widehat\Ac_{-i})_{i=1}^\infty$ is an increasing  sequence of  $\sigma$-algebras on $X.$
 Let $\widehat \Ac=\widehat \Ac(\Omega)$  denote the $\sigma$-algebra  generated  by the  union $\bigcup_{i=1}^\infty \widehat\Ac_{-i}.$
 \nomenclature[e4]{$\widehat\Ac$}{$\sigma$-algebra on $\widehat\Omega$  generated  by the  algebra $\bigcup_{i=1}^\infty \widehat\Ac_{-i}$}
On  each $\widehat\Ac_{-i}$ there is  a  natural   probability  measure $\hat\mu$  defined by
\begin{equation}\label{e:hat_mu}
\hat\mu\big ( A_{i,C} )
:=\bar\mu(C),
\end{equation}
for every $C\in \Ac,$ where  $A_{i,C}$ is  defined above.   This  relation gives a compatible family of finite-dimensional probability distribution which, according  to Kolmogorov's theorem%
\index{Kolmogorov!$\thicksim$'s theorem}\index{theorem!Kolmogorov's $\thicksim$}  
(see  \cite[Theorem 12.1.2]{Dudley}), may be  extended to a 
probability measure $   \hat\mu$ on the  $\sigma$-algebra  $\widehat\Ac.$
 We say that  $   \hat\mu$ is  {\it the extended Wiener measure with initial distribution $\mu.$}\index{Wiener!extended $\thicksim$ measure with a given initial distribution}
 \index{measure!extended Wiener $\thicksim$ with a given initial distribution}
\nomenclature[h2]{$\hat\mu$}{extended Wiener measure with initial distribution $\mu$, or equivalently, natural extension of $\bar\mu,$ it is  defined on the measurable space $(\widehat\Omega,\widehat\Ac)$}
We record  here  a useful  characterization  of  $\hat\mu$-negligible   sets\index{set!negligible $\thicksim$}.
\begin{lemma}\label{lem_null_set}
Let $A$ be a subset  of $\widehat\Omega.$  Then   $\hat\mu(A)=0$ if and only if  for every $\epsilon>0$
there  exists an increasing  sequence  $(A_i)_{i=1}^\infty\subset \widehat \Ac$ such that $A_i\in \widehat\Ac_{-i}$
and that  $A\subset  \cup_{i=1}^\infty A_i$ and that  $\hat\mu(A_i)<\epsilon.$ 
\end{lemma}
\begin{proof} Consider the  (non $\sigma$-) algebra
$\Sc:=\bigcup_{i=1}^\infty \widehat\Ac_{-i}.$ It is  clear that $\hat\mu$ is  finitely additive  on $\Sc.$ 
By Theorem  12.1.2  in \cite{Dudley},  $\hat\mu$ is  countably additive  on $\Sc.$  So we  are  in the position to apply 
Part 1) of Proposition \ref{prop_measure_theory}.
Consequently, $\hat\mu(A)=0$ if and only if
$$
\inf \left\lbrace \sum_{i=1}^\infty \hat\mu(B_i):\  B_i\in  \Sc,\  A\subset \bigcup_{i=1}^\infty B_i  \right\rbrace
=0.
$$
Letting $A_i:=B_1\cup\cdots\cup B_i$ and by reorganizing  the elements $A_i$ if necessary, we obtain the desired
conclusion.  
\end{proof}
 
Since $\mu$ is very weakly harmonic,
it follows from Theorem  \ref{thm_invariant_measures} that $\bar\mu$ is $T$-invariant.
Consequently, we deduce  from (\ref{e:hat_mu}) that $\hat\mu$ is also $T$-invariant. 
Finally, since the probability   measure  $\mu$ is also ergodic on $(X,\Lc)$ by our assumption, 
it follows  from Theorem \ref{thm_ergodic_measures} that
$\bar\mu$ is ergodic for $T$ acting on  $(\Omega,\Ac).$   Consequently, 
we deduce  from    \cite[p. 241]{CornfeldFominSinai} that $\hat\mu$ is  also ergodic for $T$ acting on  $(\widehat\Omega,\hat\Ac).$

\subsection{Wiener  backward  measures: motivations and candidates}
\label{ss:Wiener-II}

While studying  leafwise  Lyapunov  exponents using  forward orbits, we  not  only made a repeated use of the 
family of Wiener measures  $\{W_x\}_{x\in X},$  but also exploited  fully  the  following    
 relations  between this  family and  a  very  weakly  harmonic probability  measure $\mu.$

\noindent{\bf Property (i).} 
{\it  If $A\subset \Omega$ is a subset of  full $\bar\mu$-measure, then for $\mu$-almost  every $x\in X,$
the set $A_x:= A\cap \Omega_x$ is  of full  $W_x$-measure. }

\noindent{\bf  Property (ii).}
{\it If $A\in\Ac(\Omega)$ then  the function  $X\ni x\mapsto  W_x(A)\in [0,1]$  is  Borel measurable.}

Property (i) is  an  immediate consequence of formula (\ref{eq_formula_bar_mu}). It has  been widely  used in Chapter   \ref{section_splitting} and  \ref{section_Lyapunov_filtration}.
Property  (ii) has  been  proved in  Theorem     \ref{thm_Wiener_measure_measurable}. It  has  been  used  in  Chapter \ref{section_measurability}
(more concretely, in  Proposition \ref{prop_measurability_W_x} and Proposition \ref{prop_measurability}) and
in Chapter  \ref{section_leaf} (more  specifically, in  Proposition \ref{prop_chi} (i)).

This  discussion   shows that if  we want to study  leafwise  Lyapunov  exponents using  backward  orbits,  
  we need to find an  analogue family of  measures    in the setting  of extended paths  $\widehat\Omega$
  which  possesses the similar  properties as (i)-(ii)   above.
  
  The  purpose  of this  subsection is  to  provide   two candidates $(\widehat W_x)_{x\in X}$
  and   $(\widehat W^*_x)_{x\in X}.$ After  studying their     properties we  will see  that each one  has    some  advantages and   some  disadvantage as well.  
 Unifying   these  two candidates,  we   will obtain, in    Subsection \ref{ss:Wiener-II},   a  family of  Wiener  backward type measures
satisfying  both  properties (i)-(ii) listed  above.

Let $(L,g)$ be a   complete  Riemannian  manifold of bounded  geometry. Let $\widehat \Omega(L)$ be
the  space  of continuous paths $\omega:\ \R\to L.$
 For  every $i\in\N,$ the  $\sigma$-algebra $\widehat\Ac_{-i}(L)$ consists of  
all  sets of the form
$$
A=A_{i,C}:=\left\lbrace \hat\omega \in \widehat\Omega(L):\  \hat \pi (T^i\hat\omega)\in C \right\rbrace,
$$ 
 where  $C\in \Ac(L)$  (see Definition  \ref{defi_algebras_Ac}  for the $\sigma$-algebra $\Ac(L)$). In other words,  $\widehat\Ac_{-i}(L)=(\hat \pi\circ T^i)^{-1}( \Ac (L)).$  
 So  $(\widehat\Ac_{-i}(L))_{i=1}^\infty$ is an increasing  sequence of  $\sigma$-algebras on $\widehat\Omega(L).$
 Let $\widehat \Ac(L)$  denote the $\sigma$-algebra on $\widehat \Omega(L)$  generated  by the  union $\bigcup_{i=1}^\infty \widehat\Ac_{-i}(L).$ 
  
  In the remainder of this  subsection,  
 let $(X,\Lc,g)$ be  a  Riemannian  measurable lamination  satisfying  Hypothesis (H1), and 
 assume  in addition  that  there is    a covering measurable lamination $(\widetilde X,\widetilde\Lc)$  of $(X,\Lc).$
  
  \begin{definition}\label{defi_nested_covering}
    A  sequence  $(A_n)_{n=1}^\infty\subset \widehat\Ac(L)$ is  said  to be {\it  a nested covering}   of 
  a set $A\subset  \widehat\Omega(L)$ if $A_n\in   \widehat\Ac_{-n}(L)$ and 
   $A_n\subset A_{n+1}$ for  every $n,$ and  $A\subset\bigcup_{n=1}^\infty A_n.$
   
    A  sequence  $(A_n)_{n=1}^\infty\subset \widehat\Ac$ is  said  to be {\it  a nested covering}   of 
  a set $A\subset  \widehat\Omega $ if $A_n\in   \widehat\Ac_{-n}$ and 
   $A_n\subset A_{n+1}$ for  every $n,$ and  $A\subset\bigcup_{n=1}^\infty A_n.$
  \end{definition}
  
  Let $L$ be  a  leaf.
For  a point $x\in L$ and  a  set $A\in \widehat\Ac_{-i}(L)$ (resp.  a set   $A\in \widehat\Ac_{-i}$),     let $W_x(T^iA)$ denotes the (Wiener) $W_x$-measure
of  the set  $(\hat\pi\circ T^i)A\in \Ac(L)$ (resp. $\in \Ac$). Recall from  formula (\ref{eq_diffusions})  that  $\{D_t:\ t\in \R^+\}$  is  the semi-group of diffusion operators 
on  $L.$ In what  follows, let $\Vol$ be the  Lebesgue measure induced  by the metric $g$ on $L.$
 \begin{lemma}\label{lem_increasing_Theta}
  Let    $\alpha:=(A_n)_{n=1}^\infty$ be  a  nested  covering  of  a  set  $A \subset  \widehat\Omega(L)$
  (resp.    a  set  $A \subset  \widehat\Omega $).
 \\
 1) Then  the  sequence  of functions $\Theta_n:\ X\to[0,1]$ given by
 $$\Theta_n(x):= \big(D_n (W_\bullet (T^nA_n))\big) (x),\qquad  x\in X,$$  satisfies
  $0\leq \Theta_n\leq  \Theta_{n+1}\leq 1.$     Here, for $B\in  \widehat\Ac_{-n},$
  $W_\bullet (T^nB)$ is  the  function  $X\ni x\mapsto W_x (T^nB)\in[0,1].$
 So  the  function  $$\Theta(\alpha)(x):= \lim_{n\to\infty}\Theta_n(x),\qquad  x\in X,$$
is   well-defined. 
\\
2) In particular, for every  set $A\in \widehat\Ac_{-m},$
the sequence of functions $X\ni x\mapsto  \big(D_n (W_\bullet (T^nA_n))\big) (x)\in[0,1]$ is  increasing  in $n\geq m.$    
\\
3)  If     $A_n \subset  \widehat\Omega(L)$ for all $n,$ then each $\Theta_m$  is  equal to $0$ outside  the leaf $L.$
\end{lemma}
  \begin{proof}  Part 3)  of the  lemma  is  clear since $\Omega(L)$ is  of full $W_x$-measure for all $x\in L.$

To prove Part 1) at a given point $x\in X,$ we may assume without loss of generality that  $A,A_n \subset  \widehat\Omega(L)$ for all $n,$
where $L:=L_x.$
 By Proposition   \ref{prop_Markov}  applied  to the set $T^nA_n\in\Ac(L),$ we get 
  $$
  W_z(T^n A_n)\leq  \int_L p(z,y,1) W_y(T^{n+1}A_n) d\Vol(y)=(D_1 W_\bullet(T^{n+1}A_n))(z),\qquad z\in L.
  $$
  Since $A_n\subset  A_{n+1}$ and
$T^{n+1} A_n,\ T^{n+1}A_{n+1}\in  \Ac(L),$ it follows that 
$$(D_1 W_\bullet(T^{n+1}A_n))(z) \leq (D_1 W_\bullet(T^{n+1}A_{n+1}))(z),\qquad z\in L.$$
  This, combined  with the previous  estimate, implies that
  $$
  W_z(T^n A_n)\leq D_1 W_\bullet(T^{n+1}A_{n+1}))(z),\qquad z\in L.$$
  Acting  $D_n$  on both sides of the  last  estimate and  evaluating them at $x,$ we get $ \Theta_n(x)\leq  \Theta_{n+1}(x).$
  The    estimate $0\leq \Theta_n(x)\leq 1$  is  evident.
  
  Part 2) of the lemma  follows from  Part 1) applied to the following nested covering  $\alpha:=(A_n)_{n=1}^\infty:$ 
  $A_n:=A$ if $n\geq m$ and  $A_n:=\varnothing$ if $n<m.$
  \end{proof}

   
 Now we are  in the position to  introduce  the following family of functions   defined   on the  family of all subsets    of $\widehat\Omega$
  which models on the Wiener measures  on $\Omega.$
  \begin{definition}\label{def_outer_measure}
  For  every  $x\in X,$  consider the  following {\it Wiener backward function}\index{Wiener!$\thicksim$ backward function}\index{function! Wiener backward $\thicksim$}
$\widehat W_x$ whose  domain  of definition  is  the family of all subsets of $\widehat \Omega$ and which has  values in $[0,1].$
 For $A\subset  \widehat\Omega,$
 \begin{equation}\label{eq_outer_measure}
  \widehat W_x(A):= \inf\left\lbrace   \big (D_1\Theta (\alpha)\big)(x) \right\rbrace,
  \end{equation}
  the infimum being  taken over all nested  coverings  $\alpha$   of $A.$ 
  \nomenclature[g3]{$\widehat W_x$}{Wiener  backward function, its domain  of definition  is the family of  all subsets of $\widehat \Omega$} 
  \end{definition}
  \begin{remark}
  Clearly, $0\leq  \widehat W_x(A)\leq  1.$ Moreover, $ \widehat W_x(A)=0$ if   $  A\cap \widehat\Omega(L_x)=\varnothing.$
  
   The  value  of $\widehat W_x(A)$ does not change  if the     infimum  in formula  (\ref{eq_outer_measure}) is  taken over   all nested  coverings  $\alpha$  of $  A\cap \widehat\Omega(L_x).$
  \end{remark}
    Fundamental properties of the family of functions  $\{\widehat W_x\}_{x\in X}$
are  studied  in the  next     propositions.
\begin{proposition}
\label{prop_outer_measures}
  For  every $x\in X,$ $\widehat W_x$ is an  outer measure\index{measure!outer $\thicksim$}  on $\widehat\Omega.$ 
 \end{proposition}
   \begin{proof} For  every $x\in X$ we  need  to prove  the following:
   \\
   (i)  $\widehat W_x(\varnothing)=0;$\\
   (ii) (monotonicity)  if $A\subset B\subset \widehat \Omega$ then  $\widehat W_x(A)\leq \widehat W_x(B);$
   \\
   (iii)  (countable subadditivity)  if $(A_n)_{n=1}^\infty\subset \widehat \Omega$ then  $ \widehat W_x(\cup_{n=1}^\infty A_n)\leq \sum_{n=1}^\infty\widehat W_x(A_n).$
 
 Assertion (i) is  trivial since  we  may choose  the  trivial  nested covering $\alpha:=(A_n=\varnothing)_{n=1}^\infty.$
 
 To prove  the monotonicity, let $\epsilon>0$  be  arbitrary and  choose   a nested covering $\beta$ of $B$  such that  
$(D_1\Theta (\beta)\big)(x) <  \widehat W_x(B)+\epsilon.$ Since $A\subset B,$  $\beta$ is also   a nested covering of $A,$
and hence by Definition \ref{def_outer_measure},  $\widehat W_x(A)\leq (D_1\Theta (\beta)\big)(x).$
This, combined  with the previous estimate, shows that  $\widehat W_x(A) <  \widehat W_x(B)+\epsilon.$
Letting $\epsilon\to 0,$ assertion (ii) follows.

To prove the  countable subadditivity,  write $L:=L_x$ and  
pick  an arbitrary   number $\epsilon>0.$ Fix   a sequence  $(\epsilon_n)_{n=1}^\infty$ of positive numbers
such that $\sum_{n=1}^\infty \epsilon_n<\epsilon.$
By Definition \ref{def_outer_measure}, for every $n\geq 1,$
     there exists
 a  nested  covering $\alpha_n:=  (A_n^i)_{i=1}^\infty$  of  $A_n$
 such that
 \begin{equation}\label{eq_sigma_subadd}
 (D_1  \Theta( \alpha_n))(x)\leq \widehat W_x(A_n)+\epsilon_n. 
 \end{equation}
Consider  the  nested covering   $\alpha:=  (B_i)_{i=1}^\infty$  of  $A$ given by
$$  B_i:=   \bigcup_{n=1}^\infty A_n^i. $$ 
To complete  the proof of assertion (iii) it suffices   to check that  $(D_1  \Theta( \alpha))(x)<\sum_{n=1}^\infty \widehat W_x(A_n)+\epsilon.$ 
To this end, we write
\begin{eqnarray*}
(D_1  \Theta( \alpha))(x)&=&\int_L p(x,y,1)   \Theta( \alpha)(y) d\Vol(y)\\
&=&\lim_{i\to\infty}\int_L p(x,y,1)   D_i (W_\bullet (T^i B_i))(y) d\Vol(y)\\
  &\leq &\lim_{i\to\infty}\int_L p(x,y,1) \sum_{n=1}^\infty  D_i (W_\bullet (T^i A^i_n))(y) d\Vol(y).
  \end{eqnarray*}
  The last line  may be written as
  \begin{eqnarray*}
  &&  \sum_{n=1}^\infty \lim_{i\to\infty}\int_L p(x,y,1)  D_i (W_\bullet (T^i A^i_n))(y) d\Vol(y)\\
  &=&\sum_{n=1}^\infty  \int_L p(x,y,1) \Theta(\alpha_n) (y)d\Vol(y)\\
&=& \sum_{n=1}^\infty  (D_1  \Theta( \alpha_n))(x)  < \sum_{n=1}^\infty\big ( \widehat W_x(A_n) +  \epsilon_n\big)\\
  &<&   \sum_{n=1}^\infty \widehat W_x(A_n) +   \epsilon,
\end{eqnarray*}
where the  inequality in the  third  line  follows from (\ref{eq_sigma_subadd}).
This completes the proof.
\end{proof}
Proposition \ref{prop_outer_measures} allows us to define the concept of measurability as follows: 
\begin{definition}\label{def_mesurability} For  every $x\in X,$
a subset $E$ of $\widehat\Omega$ is called   {\it  $\widehat W_x$-measurable} if  for every  set $A\subset \widehat\Omega$ 
$$
 \widehat W_x(A)=  \widehat W_x(A\cap E)+ \widehat W_x(A\setminus E).
$$
\end{definition}

\begin{proposition}\label{prop_Caratheodory}
For  every $x\in X,$ all  elements   of $\widehat\Ac$ are   $\widehat W_x$-measurable.
\end{proposition}
\begin{proof} Since  we know by the   Carath\'eodory-Hahn extension theorem 
\index{theorem!Carath\'eodory-Hahn extension $\thicksim$}\index{Carath\'eodory-Hahn!$\thicksim$ extension theorem}\cite{Wheeden}
that the  family  of all $\widehat W_x$-measurable sets  forms a  $\sigma$-algebra,
it suffices  to   check that  every set $E\in \widehat \Ac_{-m}$ for any $m\in\N$ is  $\widehat W_x$-measurable.
 In fact, we only  need  to show  that for every  set $A\subset\widehat\Omega$ and  for every $\epsilon >0,$
\begin{equation}\label{eq_subadditive_Cara}
 \widehat W_x(A\cap E)+ \widehat W_x(A\setminus E)< \widehat W_x(A)+\epsilon.
\end{equation}
since  this  will imply that  $\widehat W_x(A\cap E)+ \widehat W_x(A\setminus E)\leq \widehat W_x(A),$  and hence, by 
the countable subadditivity of  $\widehat W_x$ 
established in    Proposition \ref{prop_outer_measures},  we  will  obtain  that $\widehat W_x(A\cap E)+ \widehat W_x(A\setminus E)= \widehat W_x(A).$
By Definition \ref{def_outer_measure},  
     there exists
 a  nested  covering $\alpha:=  (A_n)_{n=1}^\infty$  of  $A\cap\widehat\Omega(L_x)$
 such that
 \begin{equation}\label{eq_mesurable}
 (D_1  \Theta( \alpha))(x)\leq \widehat W_x(A)+\epsilon. 
 \end{equation}
 Consider   the    nested covering $\beta :=  (B_n)_{n=1}^\infty$  of  $A\cap E\cap\widehat\Omega(L_x)$ given by 
$$
B_n:=\begin{cases}
A_n\cap E, & n\geq m,\\
\varnothing, & n<m;  
\end{cases}
$$
and  the    nested covering  $\gamma :=  (C_n)_{n=1}^\infty$  of  $(A\setminus E)\cap\widehat\Omega(L_x)$ given by
$$
C_n:=\begin{cases}
A_n\setminus  E, & n\geq m,\\
\varnothing, & n<m. 
\end{cases}
$$
Since  $B_n\cap  C_n=\varnothing$ and  $B_n\cup C_n=A_n$ for $n\geq m,$  it follows that
\begin{equation*}
 D_n (W_\bullet (T^nB_n))+  D_n (W_\bullet (T^nC_n))    = D_n (W_\bullet (T^nA_n))  \qquad  \text{on}\  L_x \ \text{for}\ n\geq m. 
\end{equation*}
So  $\Theta(\beta)+\Theta(\gamma)=\Theta(\alpha),$ and hence
$(D_1  \Theta( \beta))(x)+ (D_1  \Theta( \gamma))(x) =(D_1  \Theta( \alpha))(x).$
  This, coupled  with (\ref{eq_mesurable}), implies (\ref{eq_subadditive_Cara}) as  desired.
\end{proof}

The construction of the  family of Wiener-type measures $(\widehat  W_x)_{x\in X}$ in  
Definition \ref{def_outer_measure}
and Proposition \ref{prop_Caratheodory} has the disadvantage  that  it does  not give  an  explicit integral  formula
as the  one given in (\ref{eq_formula_W_x_without_holonomy}).
Consequently, one cannot prove the backward version of  Property (ii) stated  at the  beginning of the  subsection.  
To remedy this  inconvenience,  we develop another family of  Wiener-type measures  $(\widehat  W^*_x)_{x\in X}$
which possesses  this  desired  property. We  will  finally prove in the  last  subsection  that the two families are, in fact, equal 
when $(X,\Lc,g)$ is a  Riemannian continuous-like lamination 
(see Proposition  \ref{prop_comparison}).

 For every $m\in \N$ and   and  every $x\in X,$ consider  the  function  $\widehat W^{*,m}_x:\  \widehat\Ac_{-m}\to [0,1]$  given by
\begin{equation}\label{eq_formula_W_m_extended}
 \widehat W^{*,m}_x(A):= \lim_{n\to\infty} \big(  D_{n+1} W_\bullet (T^nA)\big ) (x),\qquad  A\in \widehat\Ac_{-m}.
\end{equation}
Here, for $n\geq  m,$   $W_\bullet (T^nA)$  is  the  function   $X\ni z\mapsto W_z (T^nA)\in[0,1].$
 \begin{proposition}
\label{prop_formula_widehat_W_x} $\widehat W^{*,m}_x$ is  a  probability measure on $(T^{-m}\Omega, \Ac_{-m})$
and  is   supported  on $T^{-m}\Omega(L_x).$
Moreover,  $\widehat W^{*,m}_x(A)=\widehat W^{*,m'}_x(A)$  for $A\in \Ac_{-m}$ and $m\leq m'.$
\end{proposition}
Prior to the  proof, the  following elementary lemma  is  needed.
\begin{lemma}\label{lem_comparison}
Let $f,$ $(f_n)_{n=1}^\infty,$  $(g_i)_{i=1}^\infty$ and  $(g_{ni})_{n,i=1}^\infty$ be   non-negative-valued functions
on a space $S.$
Assume that  $f_n\nearrow f$ as $n\to\infty$ and  $g_{ni}\nearrow g_i$ as  $i\to\infty$ and $f_n=\sum_{i=1}^\infty  g_{ni}$  for every $n\geq 1.$
Then $f= \sum_{i=1}^\infty  g_{i}.$
\end{lemma}
\begin{proof}
On the one hand,  since $f_n=\sum_{i=1}^\infty  g_{ni}\leq \sum_{i=1}^\infty  g_{i}$  and $f_n\nearrow f,$  we get that $f\leq  \sum_{i=1}^\infty  g_{i}.$

On the  other hand, for each $N\geq  1,$
$$ \sum_{i=1}^N g_{i}=\lim_{n\to\infty}  \sum_{i=1}^N g_{ni}\leq \lim_{n\to\infty} f_n=f.  $$
Letting $N\to\infty,$ we obtain  $  \sum_{i=1}^\infty  g_{i}\leq  f.$
\end{proof}
Now  we  arrive  at  the proof of Proposition \ref{prop_formula_widehat_W_x}.
\begin{proof}
The  only nontrivial  verification is  to show that $\widehat W^{*,m}_x$ is  countably additive.
Let  $A=\cup_{i=1}^\infty A_i,$ where   $A_i\in \Ac_{-m}$ and  $ A_i\cap A_j=\varnothing$ for $i\not=j.$
To  this end  we  consider, for all $n\geq m$ and $i\geq 1,$  the following functions 
\begin{eqnarray*}
f_n&:=&D_n (W_\bullet (T^nA)),\qquad  f:= \lim_{n\to\infty} f_n,\\
g_{ni} &:=&D_n (W_\bullet (T^nA_i)),\qquad  g_i:= \lim_{n\to\infty} g_{ni}.
\end{eqnarray*}
By    Lemma \ref{lem_increasing_Theta} and using the  countable additivity of the Wiener measures,
the assumption  of Lemma \ref{lem_comparison} is  fulfilled.
The conclusion of this  lemma    says that
$$
\lim_{n\to\infty}D_n (W_\bullet (T^nA))=\sum_{i=1}^\infty \lim_{n\to\infty}D_n (W_\bullet (T^nA_i)).
$$
This, coupled with
 (\ref{eq_formula_W_m_extended}), implies that  $\widehat W^{*,m}_x(A)=\sum_{i=1}^\infty \widehat W^{*,m}_x(A_i),$ as  desired.
\end{proof}

By  formula (\ref{eq_formula_W_m_extended}) and Proposition \ref{prop_formula_widehat_W_x}, we  obtain, for every $x\in X,$ a 
function  
 $\widehat W^{*}_x:\  \cup_{m=0}^\infty\widehat\Ac_{-m}\to [0,1]$  given by
\begin{equation}\label{eq_formula_W_extended}
\widehat W^{*}_x(A):= \widehat W^{*,m}_x(A), \qquad  A\in \widehat\Ac_{-m}.
\end{equation}
At this  stage  we  still  do not know  if $\widehat W^{*}_x$ is countably additive  on the  algebra $\cup_{m=0}^\infty\widehat\Ac_{-m}.$
By Proposition \ref{prop_measure_theory}, this information  is  very important for us  in order  to   extend  $W^{*}_x$ to  a probability measure  on the space $(\widehat\Omega,\widehat\Ac).$

\subsection{Wiener  backward  measures without holonomy}
\label{ss:Wiener-III}

 To   extend  $W^{*}_x$ to  a probability measure  on the space $(\widehat\Omega,\widehat\Ac),$ we  adapt the  strategy,  which has been  developed previously in
 Section \ref{subsection_Brownian_motion_without_holonomy}
 and   Section  \ref{subsection_Wiener_measures_with_holonomy} in the  forward context,  to the   present extended context.
 In this  subsection, let $(X,\Lc,g)$ be a Riemannian measurable lamination satisfying Hypothesis (H1),  and 
 assume  in addition  that  there is    a covering measurable lamination $(\widetilde X,\widetilde\Lc)$  of $(X,\Lc).$
 Recall  from   Section \ref{subsection_Brownian_motion_without_holonomy} that the $\sigma$-algebra $\mathfrak C$ on $X^{[0,\infty)}$ is  generated by all 
cylinder sets with non-negative times. So in this  subsection, we  will extend   $\mathfrak C$ to  the $\sigma$-algebra $\widehat{\mathfrak C}$ on $X^\R$   which  is  generated by all 
cylinder sets with real times.
 \nomenclature[d2b]{$\widehat{\mathfrak C}$}{$\sigma$-algebra on $X^\R$ generated by all cylinder sets with real times}
More precisely,  a {\it  cylinder  set} ({\it with real times})\index{set!cylinder $\thicksim$ with real times}\index{cylinder!$\thicksim$ set with real times}  is a 
 set of the form
$$
C=C(\{t_i,B_i\}:1\leq i\leq m):=\left\lbrace \omega \in X^{\R}:\ \omega(t_i)\in B_i, \qquad 1\leq i\leq m  \right\rbrace,
$$
where   $m$ is a positive integer  and the $B_i$ are Borel subsets of $X,$ 
and $ t_1<t_2<\cdots<t_m$ is a  set of increasing real times. 
\nomenclature[d1a]{$C(\{t_i,B_i\}:1\leq i\leq m)$}{cylinder set associated to  Borel sets $B_1,\ldots,B_m$ and  a set of increasing  real times $ t_1<t_2<\cdots<t_m$} 
In other words, $C$ consists of all elements of $X^{\R}$ which can be found within $B_i$ at time $t_i.$
But unlike   Section \ref{subsection_Brownian_motion_without_holonomy},  the time  $t_i$ is now allowed  to be negative.
   
   Following  the model in (\ref{eq_shift}), consider  the group 
     $(T^t)_{t\in\R}$ of shift-transformations\index{shift-transformation} 
  $T^t:\  X^\R\to X^\R$ defined for  all $t,s\in\R$ by 
\begin{equation}\label{eq_shift_real}
   T^t(\omega)(s):=\omega(s+t),\qquad  \omega\in X^\R.
   \end{equation}  
   \nomenclature[c2b]{$\{T^t:\ t\in\R\}$}{group of shift-transformations of time $t$ acting on either $X^\R$ or its subspace $\widehat\Omega;$  shift-transformation of unit-time
   $T^1$
is often denoted by $T$} 
   For each $n\in \N,$  we proceed as in the  previous  paragraph replacing  $\R$ by  the interval $[-n,\infty).$
   Consequently, we obtain  the $\sigma$-algebra  $\widehat{\mathfrak C}_{-n}$ on $X^\R$   which  is  generated by all 
cylinder sets $
C=C(\{t_i,B_i\})$ with $\min t_i\geq -n.$
\nomenclature[d2a]{$\widehat{\mathfrak C}_{-n},$ $n\in\N$}{$\sigma$-algebra on $X^\R$ generated by all cylinder sets with times $\geq -n$}
 Clearly,
$\widehat{\mathfrak C}_0=\mathfrak C$ and  $\widehat{\mathfrak C}_{-n}=T^{-n} \mathfrak C,$  where, for a  family $\Fc$ of elements of $X^\R$
and for $t\in\R,$
we note  $T^t\Fc:=\{ T^tA:\ A\in\Fc\}.$

The structure of the measure space  $(X^{\R}, \widehat{\mathfrak C} )$ is best understood by viewing it as
an inverse limit. To do so, let the collection of finite subsets of $\R$ be partially
ordered by inclusion. Associated to each finite subset $F$ of $\R$ is the measure
space $(X^F,{\mathfrak X}^F),$ where $\mathfrak{X}^F$ is the Borel $\sigma$-algebra of the product topology on $X^F.$ 
Each inclusion of finite sets $E\subset F$ canonically defines a projection $\pi_{EF}\ :X_F\to X_E$
which drops the finitely many coordinates in $F\setminus E.$ These projections are continuous,
hence measurable, and consistent, for if $E\subset F\subset G,$ then  $\pi_{EF} \circ\pi_{FG}=  \pi_{EG} .$   The family
$\{ (X^F,{\mathfrak X}^F),\pi_{EF}|\  E\subset F\subset \R \ \text{finite}\}$ is an inverse  system of spaces, and its inverse limit is 
  $X^\R$  with canonical projections $\pi_F\ :X^{\R} \to X^F.$  The $\sigma$-algebra $ \widehat{\mathfrak C}$ generated by the
cylinder sets is the smallest one making all the projections $\pi_F$ measurable.

For each $x\in X,$ a probability measure $\widehat W^*_x$ on the measure space $(X^\R,\widehat{ \mathfrak C} )$ will
now be defined. If $ F=\{ t_1<\cdots<t_m\}$ is  a finite  subset of  $\R$ and
   $C^F:=
B_1\times \cdots\times B_m$ is a cylinder set of $(X^F, \mathfrak{X}^F),$  define
\begin{equation}\label{eq_formula_W_x_without_holonomy_new}
\widehat W^{*,F}_x(C^F):=
\widehat W^{*}_x(C^F\cap \widehat \Omega),
\end{equation}
where on  the  right hand  side  we  have used formula  (\ref{eq_formula_W_extended}) since  $C^F\cap  \widehat \Omega$  belongs to $\widehat\Ac_{\min\{[t_1],0\}}.$ Here $[t_1]$ denotes the integer part of $t_1.$
 
 By  Proposition
\ref{prop_formula_widehat_W_x}, if $E\subset F$ are
finite subsets of $\R$  and $C^E$ is a cylinder subset of $X^E,$ then
$$\widehat W_x^{*,E}(C^E)=\widehat W_x^{*,F}( \pi_{EF}^{-1} (C^E)).$$
Let  $\widehat{\mathfrak S}$ be the  (non $\sigma$-) algebra generated by the cylinder sets in $X^{\R}.$
\nomenclature[d3a]{$\widehat{\mathfrak S}$}{algebra on $X^\R$ generated by all cylinder sets with real times}
The  above identity  implies that $\widehat W^{*,F}_x$ given in (\ref{eq_formula_W_x_without_holonomy_new})  extends  to    a  countably additive, increasing, non-negative-valued function  
 $\widehat W^*_x$ on $\widehat{\mathfrak S}$ (see Kolmogorov's theorem \cite[Theorem 12.1.2]{Dudley}),\index{Kolmogorov!$\thicksim$'s theorem}\index{theorem!Kolmogorov's $\thicksim$} 
hence to  an outer measure\index{measure!outer $\thicksim$}
on the family of all subsets of  $X^{\R}.$ The   $\sigma$-algebra of sets  that are measurable with respect to this  outer  measure contains the  cylinder sets, hence contains 
 the $\sigma$-algebra $\widehat{\mathfrak C}.$
The  Carath\'eodory-Hahn extension theorem  
\index{theorem!Carath\'eodory-Hahn extension $\thicksim$}\index{Carath\'eodory-Hahn!$\thicksim$ extension theorem}\cite{Wheeden}
 then  guarantees that the  restriction of  this  outer measure\index{measure!outer $\thicksim$} to $\widehat{\mathfrak C}$ is the unique measure agreeing with $\widehat W^*_x$ on  the  cylinder sets.
 This measure    $\widehat W^*_x$ gives the set of paths $\omega\in X^\R$   total probability.

The following result gives  the counterpart of   Theorem \ref{thm_Brownian_motions} in the backward context. Its  proof  will be the main theme of  Appendix  \ref{section_full_outer_measure}.   
\begin{theorem}\label{thm_Brownian_motions_new}
The  subset $\widehat\Omega$
 of  $X^{\R}$  has outer measure\index{measure!outer $\thicksim$} $1$ with respect  to $\widehat W^*_x.$
\end{theorem}
 Let $\widehat{\widetilde\Ac}:=\widehat{\widetilde\Ac}(\Omega)=\widehat{\widetilde\Ac}(X,\Lc)$ be the $\sigma$-algebra on  $\widehat\Omega$  consisting of all  sets $A$ of
 the  form $A=C\cap \widehat\Omega,$
with $C\in\widehat{\mathfrak C}.$ Then we  define   the  so-called {\it  Wiener backward measure (without holonomy)}:\index{measure!Wiener backward $\thicksim$ (without holonomy)}
\index{Wiener!$\thicksim$ backward measure (without holonomy)}
\begin{equation} \label{eq_defi_widehat_W_star_x}
 \widehat W^*_x(A)=\widehat W^*_x(C\cap \Omega):=\widehat W^*_x(C).
 \end{equation}
 \nomenclature[g4]{$\widehat W^*_x$}{Wiener backward measure  without holonomy, it is  defined on the measurable spaces
$(X^\R,\widehat{\mathfrak C})$ and
 $(\widehat \Omega,\widehat{\widetilde\Ac})$} 
An important  consequence of Theorem \ref{thm_Brownian_motions_new} is  that $\widehat W^*_x$ is   well-defined on  $\widehat{\widetilde \Ac}.$ Indeed, the $\widehat W^*_x$-measure of any  measurable  subset  of $ X^{\R}\setminus\widehat{\widetilde  \Omega}$
is  equal to $0$   by 
   Theorem \ref{thm_Brownian_motions_new}. If $C, C'\in \widehat{\mathfrak C}$  and
   $C\cap\widehat \Omega=C'\cap \widehat\Omega,$ then  the  symmetric difference $(C\setminus C')\cup (C'\setminus C)$ is  contained in $  X^{\R}\setminus \widehat\Omega
  , $ so $\widehat W^*_x (C)=\widehat W^*_x(C').$
 Hence, $\widehat W^*_x$ produces  a  probability measure on $(\widehat\Omega,\widehat{\widetilde \Ac}).$
We say that $A\in \widehat{\widetilde\Ac}$ is  a {\it  cylinder set (in $\widehat\Omega$)}  if  $A=C\cap \widehat\Omega$ for  some  cylinder set $C\in \widehat{\mathfrak C}.$
\subsection{Wiener  backward   measures with holonomy}
\label{ss:Wiener-IV}

In this  subsection,  let $(X,\Lc,g)$ be  a  Riemannian continuous-like lamination and  let $$\pi:\ (\widetilde X,\widetilde \Lc,\pi^*g)\to (X,\Lc,g)$$ 
be  its covering lamination projection.
By  Subsection \ref{ss:Wiener-III}, we  construct  a $\sigma$-algebra $\widehat{\widetilde\Ac}(\widetilde \Omega)$ on $\widehat{\widetilde \Omega}:=\widehat\Omega(\widetilde X,\widetilde \Lc)$
  which is  the $\sigma$-algebra generated by all cylinder sets with real times in  $\widehat{\widetilde \Omega}.$
  Comparing the  construction of $\widehat \Ac$  carried  out in Subsection \ref{ss:Wiener-I} with
the  above  construction,   we obtain
  the  following result.
  \begin{proposition} \label{prop_algebras_widehat_Ac}    
  The  $\sigma$-algebra $ \widehat\Ac=\widehat\Ac(\Omega)$ is generated by
all sets  of  following family
 $$ \left  \lbrace \pi\circ \tilde A:\ \text{cylinder set}\ \tilde A\  \text{in} \ \widehat{\widetilde \Omega}    \right\rbrace,$$ 
  where  $\pi\circ \tilde A:= \{ \pi\circ \tilde \omega:\ \tilde\omega\in \tilde A\}.$
  \end{proposition}
  \begin{proof}
  Combining  Definition \ref{defi_algebras_Ac}  and (\ref{eq_A_i_C})-(\ref{eq_A_i_C_new}) together,  we infer that   
 for each $i\in \N,$ $ \widehat\Ac_{-i}$ is   the  $\sigma$-algebra generated by
  all sets  of  following family
 \begin{equation*}
 \left  \lbrace \pi\circ \tilde A:\ \text{cylinder set $\tilde A$ with times $\geq  -i$  in $\widehat{\widetilde \Omega}$}    \right\rbrace.
 \end{equation*}
  Since  the $\sigma$-algebra $\widehat \Ac$  is  generated  by the  union $\bigcup_{i=1}^\infty \widehat\Ac_{-i},$ the proposition follows.
  \end{proof}
  For  a  point $x\in X,$ we apply  the previous    proposition  to the lamination  consisting  of a  single  leaf $L:=L_x$ with its  covering  projection $\pi:\ \widetilde L \to L.$  
Consequently,   
 we infer that  the  $\sigma$-algebra $\widehat\Ac(L):= \widehat\Ac(\Omega(L))$
  is generated by
all sets  of  following family
 $$ \left  \lbrace \pi\circ \tilde A:\ \text{cylinder set}\ \tilde A\ \text{in}\  \widehat  \Omega (\widetilde L)    \right\rbrace.$$ 
Recall from  Subsection \ref{ss:Wiener-IV}
the $\sigma$-algebra   $\widehat{\widetilde\Ac}:=\widehat{\widetilde\Ac}(\Omega)=\widehat{\widetilde\Ac}(X,\Lc)$ on $\widehat\Omega.$
 Observe that    $\widehat{\widetilde \Ac}\subset  \widehat\Ac$ and that the  equality holds if  all leaves  of $(X,\Lc)$ have
trivial holonomy.  
 Note that $ \widehat\Ac(\widetilde\Omega)=\widehat{\widetilde \Ac}(\widetilde\Omega),$
 where $\widetilde\Omega:=\Omega(\widetilde X,\widetilde \Lc).$

Now  we  construct  a family  $\{\widehat W^*_x\}_{x\in X}$ of Wiener  backward type probability measures   on $(\widehat\Omega,\widehat\Ac).$ 
  Let $x$  be  a  point in $X$ and 
$C$ an element of $\widehat\Ac.$ Then we  define the  so-called {\it  Wiener backward  measure (with holonomy)}:\index{measure!Wiener backward $\thicksim$ (with holonomy)}
\index{Wiener!$\thicksim$ backward measure (with holonomy)}
 \nomenclature[g5]{$\widehat W^*_x$}{Wiener backward  measure   with holonomy, it is  defined on the measurable space $(\widehat\Omega,\widehat\Ac),$ it coincides with $\widehat W_x$ 
 by Proposition \ref{prop_comparison}}
\begin{equation}\label{eq_formula_widehat_W_star_x}
 \widehat W^*_x(C):= \widehat W^*_{\tilde x}( \pi^{-1} C),
  \end{equation}
  where  
   $\tilde x$ is  a  lift  of $x$ under  the projection $\pi:\ \widetilde L\to L=L_x,$
     and  
   $$\pi^{-1} (C):=\left\lbrace \tilde\omega\in \widehat\Omega(\widetilde L) :\  \pi\circ\tilde \omega\in C\right\rbrace,
$$ and  $\widehat W^*_{\tilde x}$ is the probability measure  on $(\widehat\Omega(\widetilde L),\widehat \Ac(\widetilde L))$
given    by (\ref{eq_defi_widehat_W_star_x}).
\begin{proposition}\label{prop_Wiener_backward_measure}
 The value  of $\widehat W^*_x(C)$ defined  in (\ref{eq_formula_widehat_W_star_x}) is  independent of the choice
 of $\tilde x.$  Moreover, $\widehat W^*_x$ is a  probability  measure     
on $(\widehat\Omega,\widehat\Ac).$
\end{proposition}
\begin{proof}
We only give the proof of the independence of formula  (\ref{eq_formula_widehat_W_star_x}).
To do this, fix a  point $x\in X.$ Next, let $\widehat\Cc$ be the family of all elements $C\in \widehat\Ac$ such that 
the value  of $\widehat W^*_x(C)$ defined  in (\ref{eq_formula_widehat_W_star_x}) is  independent of the choice
 of $\tilde x.$
Putting  together  Definition  \ref{defi_continuity_like} (ii), formula (\ref{eq_formula_widehat_W_star_x}) as well as  formulas 
 (\ref{eq_formula_W_m_extended})-(\ref{eq_formula_W_extended}), we infer that $\widehat\Ac_0\subset\widehat\Cc.$

Next,  using  (\ref{eq_A_i_C})-(\ref{eq_A_i_C_new})   and combining again  formula (\ref{eq_formula_widehat_W_star_x}) as well as  formulas 
 (\ref{eq_formula_W_m_extended})-(\ref{eq_formula_W_extended}), we infer from the  previous paragraph  that
 $\widehat\Ac_{-i}\subset \widehat\Cc$ for all $i\in \N.$
 So the (non $\sigma$-)algebra $\bigcup_{i=1}^\infty\widehat\Ac_{-i}$ is  contained in $ \widehat\Cc.$
 
 Next,
observe  that  if $(C_n)_{n=1}^\infty\subset \widehat\Cc$ such that  $C_n\nearrow C$ (resp.  $C_n\searrow C$)  as $n\nearrow \infty,$
then  $C\in\widehat\Cc$ because  $\widehat W^*_{\tilde x}( \pi^{-1} C_n) \nearrow \widehat W^*_{\tilde x}( \pi^{-1} C)$  (resp. 
$\widehat W^*_{\tilde x}( \pi^{-1} C_n) \searrow \widehat W^*_{\tilde x}( \pi^{-1} C)$) for  all $\tilde x\in\pi^{-1}(x).$
On the other hand, recall that the $\sigma$-algebra $\widehat\Ac$ is generated by  $\bigcup_{i=1}^\infty\widehat\Ac_{-i}.$
Consequently,
applying Proposition \ref{prop_criterion_sigma_algebra} yields that  $\widehat\Ac\subset \widehat\Cc.$
Hence, $\widehat\Ac=\widehat\Cc.$
 This  completes the proof.
\end{proof}

\begin{proposition} \label{prop_comparison}
Let $x\in X.$
\\
1) Then
$\widehat W_x(A)=\widehat W^*_x(A)$ for all $A\in\widehat \Ac.$
\\
2) For every $A\in \widehat\Ac_{-m},$ we have that
$$ 
\widehat W_x(A)= \lim_{n\to\infty} \big ( D_{n+1} W_\bullet (T^nA)\big ) (x),
$$
where, for $n\geq  m,$   $W_\bullet (T^nA))$  is  the  function   $X\ni z\mapsto W_z (T^nA)\in[0,1].$
  \end{proposition} 
\begin{proof} 
We only need  to prove  Part 1) since  Part 2)  follows from combining  Part 1) with formula (\ref{eq_formula_W_m_extended}),
and Proposition \ref{prop_formula_widehat_W_x} and  the above construction of the probability measure $\widehat W^*_x$   on $(\widehat\Omega,\widehat\Ac).$

To prove Part 1), pick an arbitrary $\epsilon>0.$ 
 Let  $(A_n)_{n=1}^\infty\subset \mathfrak S$ be an  increasing sequence  such that $A\subset \cup_{n=1}^\infty A_n$ and  that $ \lim_{n\to\infty} W^*_x(A_n)< \widehat W^*_x(A)+\epsilon.$
 Arguing as  in the proof of Part 1) of Proposition \ref{prop_cylinder_sets},
 we see  that each $A_n$ is a  finite union of cylinder sets. By  repeating  each $A_n$ finitely many times if necessary,  we may suppose without loss of generality that 
$A_n\in\widehat \Ac_{-n}.$
So $\alpha:=(A_n)_{n=1}^\infty$ is a nested covering of $A.$
By Part 2) of Lemma \ref{lem_increasing_Theta}, we have that
$$
D_n (W_\bullet (T^nA_n))\leq  \lim_{i\to\infty} D_i (W_\bullet (T^iA_n)).
$$
So by formulas (\ref{eq_formula_W_m_extended})-(\ref{eq_formula_W_extended}), we obtain that
$$
\Big (D_1\big (D_n (W_\bullet (T^nA_n))\big) \Big)(x)\leq   W^*_x(A_n).
$$
Letting $n\to\infty,$ we get that $$(D_1\Theta(\alpha))(x)\leq    \lim_{n\to\infty} W^*_x(A_n)< \widehat W^*_x(A)+\epsilon.$$
By Definition    \ref{def_outer_measure}, we  infer that  $\widehat W_x(A)<\widehat W^*_x(A)+\epsilon.$
Since $\epsilon>0$ is  arbitrary, we have  shown that   $\widehat W_x(A)\leq \widehat W^*_x(A).$

To prove  the converse  inequality,
 pick an arbitrary $\epsilon>0.$ Let  $\alpha:=(A_n)_{n=1}^\infty$ be a nested covering of $A$ such  that 
 $ (D_1\Theta(\alpha))(x)< \widehat W_x(A)+\epsilon.$
 Since  $(A_n)_{n=1}^\infty$ covers $A,$ we have that
 $$
 \widehat W^*_x(A)\leq \lim_{n\to\infty} \widehat W^*_x(A_n). 
 $$
 Since $A_n\in\widehat \Ac_{-n},$ it follows from formulas (\ref{eq_formula_W_m_extended})-(\ref{eq_formula_W_extended})  that
$$
\widehat W^*_x(A_n)=           \Big  (D_1  \lim_{i\to\infty} \big(D_i (W_\bullet (T^iA_n))\big)  \Big)(x).
$$
 Since the sequence $(A_n)_{n=1}^\infty$ is  increasing, we see that for each $n,$  the right hand  side is   smaller than 
$$
  \Big (D_1 \lim_{i\to\infty} \big(D_i (W_\bullet (T^iA_i))\big) \Big)(x)=(D_1\Theta(\alpha))(x).
$$ 
So we have shown that 
$$
 \widehat W^*_x(A)\leq \lim_{n\to\infty} \widehat W^*_x(A_n)\leq (D_1\Theta(\alpha))(x)< \widehat W_x(A)+\epsilon.
$$
Since $\epsilon>0$ is  arbitrary, we  infer that $
 \widehat W^*_x(A)\leq \widehat W_x(A).$ This completes the proof.
\end{proof}

\begin{remark}
As  an immediate  consequence of Part 2) of Proposition \ref{prop_comparison},  we  see    that  $\widehat W_x\not= W_x$  on $\Ac$  in general
when  we identify, via formula (\ref{eq_A_i_C_new}), $\Ac$ with $\widehat\Ac_0.$
\end{remark}

\begin{proposition}
\label{prop_measurability_widehat_W_x}
1) If $A\in\widehat\Ac,$ then  the function  $X\ni x\mapsto  \widehat W_x(A)\in [0,1]$  is   Borel measurable.
\\
2) Let $S$ be a topological space.  
 For   any  measurable set $F$ of the measurable space  $(\widehat\Omega\times S,  \widehat\Ac\otimes \Bc(S)),$ let $\Phi(F)$ be 
 the function $$ X\times S\ni(x,s)\mapsto  \widehat W_x(\{\omega\in\widehat\Omega:\ (\omega,s)\in F\})\in[0,1].   $$
Then $\Phi(F)$  is
 measurable.
 \end{proposition}
\begin{proof}
 Let $A\in \widehat \Ac_{-m}$ for some $m\in\N.$
Using the  construction  of $\widehat \Ac_{-m}$  given in 
(\ref{eq_pi_hat})
 and (\ref{eq_A_i_C}),  
  Definition \ref{defi_continuity_like} (iv)   tells us   that each map
$X\ni x\mapsto  W_x(T^n A)$ is  Borel  measurable for every $n\geq m.$
So,  the function on the  right hand  side  of  (\ref{eq_formula_W_m_extended}) is  also  Borel  measurable.
 This, combined  with (\ref{eq_formula_W_extended}) and Proposition  \ref{prop_comparison}, implies that
   the function $X\ni x\mapsto
 \widehat W_x(A)$ is Borel   measurable. 

Next, using the  previous  paragraph and using the transfinite induction
given in Proposition \ref{prop_criterion_sigma_algebra},    we argue as  in Step 2 and Step 3  of the  proof of  Theorem \ref{thm_Wiener_measure_measurable}. 
This  completes  the proof of Part 1).

  Finally, we  argue as  in the proof of   Proposition \ref{prop_measurability_W_x}  using Part  1) instead of Theorem \ref{thm_Wiener_measure_measurable},
  Part 2) follows.
\end{proof}

     The  following simple terminology  will be  useful  later on.
 \begin{definition}\label{defi_set_null_measure}
  \rm 

  Let $A\in  \widehat\Ac(L).$ 
 \\ 
   $\bullet$ $A$ is said to be  {\it of null measure in $L$}  if   $\widehat  W_x(A)=0$ for some reference point $x\in L.$\index{set!$\thicksim$ of null measure in a leaf}
 \\
  $\bullet$  $A$ is said to be  {\it  of positive measure  in $L$}  if it is not of null measure in $L.$\index{set!$\thicksim$ of positive measure in a leaf}
\\
 $\bullet$  $A$ is said to be {\it  of full measure  in $L$}  if  $\widehat\Omega(L)\setminus A$ is of null measure in $L.$\index{set!$\thicksim$ of full measure in a leaf}
 \\
 $\bullet$ We say that  a property $\Hc$ holds {\it for almost every $\omega\in \widehat\Omega(L)$} if there  is  a  set $A\subset \widehat\Omega(L) $  of full measure  in $L$ such that  $\Hc$ holds for every $\omega\in A.$ 
 \end{definition}

\begin{remark}\label{rem_intersection_non_empty}\rm
It follows immediately from  the above definition  that the intersection of a set of full measure and  a set  of positive
measure  in  the  same leaf is  always  nonempty.

Apparently, the  notion of  sets of null, positive or  full  measures  given in Definition \ref{defi_set_null_measure}
depends on the  choice of   a reference point $x\in L.$ However,
Part 2)  of the  next  proposition  shows that the above definition  is, in fact,  independent of  the  choice  of such a reference point.   
\end{remark}

 \begin{proposition}\label{prop_backward_set_classification}
 1) Let $\mu$ be a  very weakly harmonic  probability measure on $(X,\Lc,g).$ If $A\in \widehat\Ac$ is  a  set of null  $\hat\mu$-measure, then  for  $\mu$-almost every $x\in X,$
 $A\cap \widehat\Ac(L_x)$ is of null measure in $L_x.$ Equivalently, if $A\in \widehat\Ac$ is  a  set of full $\hat\mu$-measure, then  for  $\mu$-almost every $x\in X,$
 $A\cap \widehat\Ac(L_x)$ is of full measure in $L_x.$
 \\
 2) Let   $A\in \widehat\Ac(L)$  where $L$ is a  leaf.  If   $\widehat  W_x(A)=0$ for some  point  $x\in L,$
 then $\widehat  W_y(A)=0$ for all   $y\in L.$
 Similarly, if   $\widehat  W_x(A)=1$ for some  point $x\in L,$
 then $\widehat  W_y(A)=1$ for all   $y\in L.$  If   $\widehat  W_x(A)>0$ for some  point $x\in L,$
 then $\widehat  W_y(A)>0$ for all   $y\in L.$
 \\
 3) For   a  sequence  $(A_n)_{n=1}^\infty\subset\widehat\Ac(L)$ of null measure   in a leaf $L,$  its union
 $\bigcup_{n=1}^\infty A_n$ is also of null measure in $L.$
 Similarly, for   a  sequence  $(A_n)_{n=1}^\infty\subset\widehat\Ac(L)$ of full measure   in a leaf $L,$  its  intersection
 $\bigcap_{n=1}^\infty A_n$ is also of full measure in $L.$
\end{proposition}
\begin{proof}
First we prove assertion 1). Fix a  sequence  $(\epsilon_n)\searrow 0$ as $n\nearrow\infty.$
By Lemma \ref{lem_null_set} and  Definition \ref{defi_nested_covering}, there  exists, for every $n$,
 a  nested  covering $\alpha_n:=(A_i^n)_{i=1}^\infty$  of  $A$
    such that  $\hat\mu(\bigcup_{i=1}^\infty A_i^n)<\epsilon_n.$ Since $\mu$ is very weakly harmonic, it follows from Definition \ref{defi_Standing_Hypotheses_harmonicity} that 
it is $D_i$-invariant for all $i\in \N.$
 So we get that
\begin{eqnarray*}
\int_X  D_i (W_\bullet (T^iA^n_i))(x)d\mu(x)&=&\int_X  W (T^iA^n_i)(x)d\mu(x)=\bar\mu(T^iA^n_i)=\hat\mu(T^iA^n_i)\\
&=&\hat\mu(A^n_i)
\leq \hat\mu(\bigcup_{i=1}^\infty A_i^n)<\epsilon_n,
\end{eqnarray*}
because $\hat\mu$ is $T$-invariant  
 and $\mu$ is  $D_i$-invariant.     
Since we know from Part 1) of  Lemma \ref{lem_increasing_Theta} that $ D_i (W_\bullet (T^iA^n_i))$ converge pointwise to
 $\Theta(\alpha_n)$  as $i\to\infty,$ it follows from the Lebesgue dominated  convergence
\index{Lebesgue!$\thicksim$ dominated convergence theorem}\index{theorem!Lebesgue dominated convergence $\thicksim$}  that
$
\int_X \Theta(\alpha_n)(x)d\mu(x)\leq \epsilon_n.
$ So the bounded sequence  $ (\Theta(\alpha_n))_{n=1}^\infty$ converges in $L^1(X,\mu)$ to $0.$
By  extracting  a subsequence  if necessary, we may assume  without loss of generality that
the sequence  $ (\Theta(\alpha_n))_{n=1}^\infty$ converges pointwise to $0$
$\mu$-almost everywhere. Hence, for $\mu$-almost every $x\in X,$
 \begin{equation*}
 \lim_{n\to\infty}\Theta(\alpha_n)(y)=0\qquad\text{for
$\Vol_{L_x}$-almost every $y\in L_x.$}
\end{equation*} 
For such  a  point $x\in X,$ consider
the sequence  of nested coverings  $(\alpha_{x,n})_{n=1}^\infty$  of $A\cap \widehat\Omega(L_x)$ defined  by
$$
\alpha_{x,n}:=(A_i^{x,n})_{i=1}^\infty,\qquad\text{where}\  A_i^{x,n}:=A_i^n\cap  \widehat\Omega(L_x).
$$
Clearly, the last limit implies that
$$
\lim_{n\to\infty}\Theta(\alpha_{x,n})(y)=0\qquad\text{for
$\Vol_{L_x}$-almost every $y\in L_x.$}
$$
So  $\widehat  W_x(A\cap \widehat\Omega(L_x))=0,$ which, in turn, gives that $\widehat  W_x(A)=0.$    
By Definition \ref{defi_set_null_measure}, $A$ is of null measure in $L_x.$
 This  finishes assertion 1).

Now  we turn to    the first part of assertion 2).  
By Definition \ref{def_outer_measure},  there exists
 a sequence of nested  coverings $(\alpha_n)_{n=1}^\infty$  of  $A$  
 such that 
$$
\lim_{n\to\infty} \int_L   p(x, z,1)  \Theta(\alpha_n)(z) d\Vol(z)=0.
$$
Recall  that  $\int_L p(x, z,1)   d\Vol(z)$ is a probability measure on $L.$ Consequently,
by passing to a  subsequence  if necessary, we infer from the last limit  that  $ \Theta(\alpha_n)$ converges  pointwise to $0$   $\Vol(L)$-almost everywhere in $L.$
Combining  this  and   the  estimates  $0\leq \Theta(\alpha_n)\leq 1,$ and  applying the Lebesgue dominated convergence,
\index{Lebesgue!$\thicksim$ dominated convergence theorem}\index{theorem!Lebesgue dominated convergence $\thicksim$} we get that
 $$
\lim_{n\to\infty} \int_L   p(y, z,1)  \Theta(\alpha_n)(z) d\Vol(z)=0,\qquad  y\in L.
$$
So, $(D_1 \Theta(\alpha_n))(y) \to 0$
as $n\to\infty.$ Hence,  $\widehat  W_y(A)=0$ for all   $y\in L.$
The first part of assertion 2) follows.
The  second part and  the  third one can be proved in the  same  way.

To prove  the first part of assertion 3),  it suffices  to notice that
by the first part of assertion 2), for every
point $x\in L,$    $\widehat W_x(A_n)=0,$ $n\geq 1.$
Hence,  by the countable subadditivity of   $\widehat W_x$ established in Proposition \ref{prop_outer_measures},
$\widehat W_x(\cup_{n=1}^\infty A_n)=0.$
Hence, by  Definition \ref{defi_set_null_measure}, $\cup_{n=1}^\infty A_n$ is  of null measure in $L.$
The second  part of assertion 3)  can be  proved  similarly.  \end{proof}

The following result will be  very useful  later on.
\begin{lemma}\label{lem_projection_null_measure}
Let $(L,g)$ be  a complete Riemannian manifold of bounded geometry  and  $\pi:\  \widetilde L\to L$ its  universal cover.\\
1)
Let $\widetilde\Fc\subset \widehat\Omega(\widetilde L)$ be a set of positive  measure in $\widetilde L.$  Then
$\pi\circ\widetilde\Fc  \subset  \widehat\Omega(  L)$ is a set of positive  measure in $ L,$ where
$
\pi\circ\widetilde\Fc:=\{ \pi\circ \hat\omega:\ \hat\omega\in \widetilde\Fc \}.$
\\
2) A set $\Fc\subset \widehat\Omega( L)$ is of full measure in $ L$  if and only if the set 
$\pi^{-1}\Fc  \subset  \widehat\Omega(  \widetilde L)$ is  of full  measure in $\widetilde L,$
where 
$
\pi^{-1}\Fc:=\{  \hat\omega\in \widehat\Omega(  \widetilde L) :\ \pi\circ\hat\omega\in \Fc \}.$
\end{lemma}
\begin{proof} 
To prove Part 1) suppose  in order to reach a contradiction that $\Fc:= \pi\circ\widetilde\Fc  \subset  \widehat\Omega(  L)$ is a set of null  measure in $ L.$ By  Definition  \ref{defi_set_null_measure},
     there exist  a set $F\subset L$ and a sequence  of 
   nested  coverings $(\alpha_n)_{n=1}^\infty$ of $\Fc$
 such that $\Vol(L\setminus F)=0$ and  $\lim_{n\to\infty}\Theta(\alpha_n)(x)=0$ for every $x\in F.$

 Write   $\alpha_n=  (A_n^i)_{i=1}^\infty.$  
Consider  the  sequence of nested  coverings $(\tilde\alpha_n)_{n=1}^\infty$ of $\pi^{-1}(\Fc)$  defined  by     
   $$ 
 \tilde\alpha_n:=(\tilde A_n^i)_{i=1}^\infty,\qquad  \tilde A_n^i:=\pi^{-1}(A_n^i)  . $$
Clearly,  $\widetilde  \Fc\subset \pi^{-1}\Fc.$
To  complete the proof of Part 1) it suffices to show that
 $\lim_{n\to\infty}\Theta(\tilde\alpha_n)(\tilde x)=0$ for every $\tilde x\in \tilde F:=\pi^{-1}(F).$
 This  will follow  immediately from the equality 
\begin{equation}\label{eq_lem_projection_null_measure}
\Theta(\tilde\alpha_n)(\tilde x)= \Theta(\alpha_n)( x),\qquad  x\in L,\  \tilde x\in \pi^{-1}(x),
\end{equation}
and the  above mentioned  property of $(\alpha_n)_{n=1}^\infty.$ 

 To prove  (\ref{eq_lem_projection_null_measure}) we  start  with the following   immediate consequence of    Lemma
 \ref{lem_change_formula}
(i) below
 $$
 W_x(B)=W_{\tilde x}(\pi^{-1} B),\qquad  x\in L,\  \tilde x\in \pi^{-1}(x),\ B\in \Ac(L).  
 $$
 Using  this and  the  equality  $ \pi^{-1} (T^iA^i_n)=T^i(\pi^{-1} A^i_n),$ we get that
 $$
  W_y (T^i A_n^i)= W_{\tilde  y} (T^i \tilde A_n^i),\qquad  y\in L,\ \tilde y\in \pi^{-1}(y).
 $$
  This, combined  with  (\ref{eq1_heat_kernel}), implies that
$$  D_i (W_\bullet (T^i A_n^i))(x) =  D_i (W_\bullet (T^i \tilde A_n^i))(\tilde x), \qquad x\in L,\  \tilde x\in \pi^{-1}(x).
$$
   Hence, (\ref{eq_lem_projection_null_measure}) follows.
   
   Next, we turn to the ``only if"  part of assertion   2). Observe that we only need  to show that   the  set $   \widehat\Omega(\widetilde L)\setminus \pi^{-1}(\Fc)$ is   of null measure in $\widetilde L.$ Suppose  the contrary in order to get a  contradiction.
   Then by Part 1),  the  set $   \widehat\Omega( L)\setminus \Fc$ is   of  positive measure in $ L,$
   which is  impossible since  $\Fc$ is  of  full measure in $L.$
   
   Finally,  we  establish the ``if" part of assertion 2).
   Since  $\widehat\Omega(\widetilde L)\setminus \pi^{-1}(\Fc)$ is   of null measure in $\widetilde L,$
   we deduce  from   Definition  \ref{defi_set_null_measure} that 
     there exist  a set $\widetilde F\subset \widetilde L$ and a sequence  of 
   nested  coverings $(\tilde\alpha_n)_{n=1}^\infty$ of $\widehat\Omega(\widetilde L)\setminus\pi^{-1}\Fc$
 such that $\Vol(\widetilde L\setminus\widetilde F)=0$ and  $\lim_{n\to\infty}\Theta(\tilde\alpha_n)(\tilde x)=0$ for every $\tilde x\in\widetilde F.$
 Write  $ 
 \tilde\alpha_n:=(\tilde A_n^i)_{i=1}^\infty.$
 For  a deck-transformation $\gamma\in \pi_1(L)$ and a set $\tilde A\subset \widehat\Omega(\widetilde L),$
 let  $\gamma\circ\tilde A:= \{ \gamma\circ\hat\omega:\ \hat\omega\in \tilde A \}.$
 Replacing  each $\tilde A_n^i$ with its  subset $\bigcap_{\gamma\in \pi_1(L)}  \gamma\circ \tilde A_n^i$
 and noting that $\pi_1(L)$ is  at most  countable,
 we may assume  without loss of generality that  $\gamma\circ \tilde A_n^i=\tilde A_n^i$ for all $\gamma\in \pi_1(L).$
 So   there  exists a set $A_n^i\in \widehat\Ac_{-n}(L)$ such that $\tilde A_n^i=\pi^{-1} (A_n^i).$
Consider  the  following  sequence  of 
   nested  coverings
 $\alpha_n=  (A_n^i)_{i=1}^\infty$ of  $\widehat\Omega( L)\setminus\Fc.$   Let $\widetilde F':=\cap_{\gamma\in \pi_1(L)} \gamma( \widetilde F),$ and  $F=\pi(\widetilde F')\subset L.$  So we obtain    that $\pi^{-1}(F)\subset \widetilde F$ and 
$\Vol(L\setminus F)=0.$ Using (\ref{eq_lem_projection_null_measure}) and  the  equality
$$\lim_{n\to\infty}\Theta(\tilde \alpha_n)( \tilde x)=0,\qquad  \tilde x\in \widetilde F,$$
we see  that 
$\lim_{n\to\infty}\Theta(\alpha_n)( x)=0$ for every $ x\in F.$
 This  completes the proof.
\end{proof}
\section[Lyapunov backward exponents and  Oseledec backward   theorem]{Leafwise Lyapunov backward exponents and  Oseledec backward type theorem}
Let $(X,\Lc,g)$ be a Riemannian  lamination satisfying the  Standing Hypotheses.
By Proposition \ref{P:lami-is-cont-like},  $(X,\Lc,g)$ endowed  with  its  covering lamination projection
$\pi:\  (\widetilde X,\widetilde\Lc,\pi^*g)\to (X,\Lc,g)$ is     Riemannian continuous-like.
So the result of  Section  \ref{subsection_extended_sample_path_spaces} is  valid in the present context.
Let $\mathcal  A:\ \Omega(X,\Lc)\times \R^+\to \GL(d,\R)$ be  a cocycle.
We  extends it to   a map (still  denoted by)  $\mathcal  A:\ \widehat\Omega(X,\Lc)\times \R\to \GL(d,\R)$ by
the following formula
\begin{equation}\label{eq_formula_extended_cocycle}
\mathcal A(\omega,t):=
\begin{cases}
\mathcal A(\hat\pi\omega,t), &  t\geq 0;\\
\mathcal A(\hat\pi( T^t\omega),|t|)^{-1},  & t<0,
\end{cases}
\end{equation}
where  $\hat\pi$ is given  by (\ref{eq_pi_hat}).
 It can be checked that the multiplicative law   
$$\mathcal A(\omega,s+t)=\mathcal A(T^t\omega,s)\mathcal A(\omega,t)$$
still holds for all $s,t\in \R$ and  $\omega \in \widehat\Omega(X,\Lc).$
By the same  way  we extend a  cocycle  $\mathcal  A:\ \Omega(X,\Lc)\times \N\to \GL(d,\R)$ to  a cocycle
(still denoted by) $\mathcal A:\ \widehat \Omega(X,\Lc)\times \Z\to \GL(d,\R).$  

Let $(L,g)$ be a complete   Riemannian  manifold of bounded  geometry.   
 For any function  $f:\  \widehat \Omega(L)\to \R\cup\{\pm\},$ let  
$\esup  f$ denote    the  {\it essential  supremum} of $f$ {\it (with respect to the Wiener backward measure)} given by the following  formula:
\nomenclature[a91]{$\esup$}{essential  supremum   w.r.t. the Wiener backward measure}
\begin{equation}\label{eq_esup_extended_version}
\esup f=\esup_L f:=\inf\limits_E \sup\limits_{\omega\in E} f(\omega),
\end{equation}
 the  infimum being taken over all  elements $E\in  \widehat\Ac(L)$  that are of full measure in $L$
 (see Definition  \ref{defi_set_null_measure}).

 The following result is the counterpart of   Lemma \ref{lem_esup} in the  backward setting.
 \begin{lemma}\label{lem_backward_esup}
 Let $(L,g)$ be  a complete Riemannian manifold of bounded geometry    and    $f:\ \widehat\Omega( L)\to \R\cup\{\pm\}$  a measurable  function. Then  there  exists  a  set $E\in \widehat\Ac(L)$ of full  measure in $L$ such that
 $$
 \esup  f=\sup_{\omega\in  E} f(\omega).
 $$
 In particular, for every  subset  $Z\subset \widehat\Omega(L)$ of null measure in $L,$
$$
\sup_{\omega\in  E} f(\omega)=\sup_{\omega\in  E\setminus Z} f(\omega).
$$ 
 \end{lemma}
 \begin{proof} We proceed  as in the proof of    Lemma \ref{lem_esup}
 using (\ref{eq_esup_extended_version}) and  assertion 3) of  Proposition \ref{prop_backward_set_classification}.
  \end{proof}

Now let $S$ be a topological space\index{space!topological $\thicksim$} and  consider the measurable space  $(\widehat\Omega\times S,  \widehat\Ac\otimes \Bc(S)),$ 
where  $\Bc(S)$ denotes, as  usual,   the Borel $\sigma$-algebra of $S.$
For any measurable function  $f:\  \widehat\Omega\times S \to[-\infty,\infty],$ define  the  function  $\esup f:\ X\times  S \to   [-\infty,\infty]$ by
$$
\esup f (x,s):=\esup_{ L_x} f_{x,s},\qquad (x,s)\in X\times S.
$$
where the function $f_{x,s}:\ \widehat\Omega(L_x)\to [-\infty,\infty]$ is  given by 
$$
f_{x,s}(\omega):= f(\omega,s),\qquad  \omega\in \widehat\Omega(L_x).
$$
The  following result  may be  regarded as  the  counterpart  of Proposition \ref{prop_measurability}
in the   setting of extended paths.
 \begin{proposition}\label{prop_measurability_extended}
 Let $f$ be  a  measurable  function on $\widehat\Omega\times S.$
Let $\mu$ be a very weakly  harmonic  probability  measure on $(X,\Lc).$ Then   $\esup f$ 
 is  measurable  on the  measurable space $(X\times S, \Bc(X)\otimes\Bc(S)).$
 \end{proposition}
 \begin{proof}
 We argue  as  in the proof of  Proposition \ref{prop_measurability}
using  Part  2) of Proposition \ref{prop_measurability_widehat_W_x} instead of  Proposition \ref{prop_measurability_W_x}.
 \end{proof}

  Now  we  will introduce  the notion of Lyapunov backward exponents of the  cocycle $\mathcal A$
with respect to  a leaf $L.$
 For $(x,v) \in  L\times\R^d ,$ and  $\omega\in \widehat\Omega(L)$ with $\omega(0)=x$  introduce the quantity 
 \begin{equation}\label{eq_chi-_x,v,omega}
  \chi^-_{x,v}(\omega):=\limsup_{n\to\infty}{1\over n} \log\| \mathcal A(\omega,-n)v \|. 
   \end{equation}
In what follows  we want  to  extend this  definition  to  the case  where  $\omega(0)$ is  not equal to $x.$
Having   at hands this  extension,
we will be  able  to  define  a function  $\chi^-:\  L\times\R^d  \to \R\cup\{\pm\infty\},$  which plays  the analogue role
in the  backward setting 
as  the  function $\chi:\  L\times\R^d  \to \R\cup\{\pm\infty\}$  given by (\ref{eq_functions_chi}) does  in the  forward setting.
 Consider  two cases.
\\
{\bf Case 1:}  {\it $L$ is simply connected.}
 
 We introduce  the following   equivalent relation.
 \begin{definition}\label{D:equivlent_relation_1} For two pairs $(x,v),$ $(y,u)\in  L\times\R^d,$ we  write $(x,v)\overset{\mathcal A}{\sim}(y,u)$
 if  there is  a  path $\omega\in \Omega(L)$  with $\omega(0)=x,$ $\omega(1)=y$ and $\mathcal A(\omega,1)v=u.$
 \end{definition}
 \nomenclature[i1]{$\overset{\mathcal A}{\sim}$}{equivalent relation w.r.t. a cocycle $\mathcal A$}
 Consider the function
$\chi^-_{x,v}:\   \widehat \Omega(L) \to \R\cup\{\pm\infty\}$ defined by
\begin{equation}\label{eq_chi-_x_v_omega}
\chi^-_{x,v}(\omega):=\limsup_{n\to\infty}{1\over n} \log\| \mathcal A(\omega,-n)u_{x,v,\omega}\|,\qquad  \omega \in\widehat\Omega(L),
\end{equation}
  where $u_{x,v,\omega}\in\R^d$ is uniquely determined  by the  condition that $(\omega(0),u_{x,v,\omega})\overset{\mathcal A}{\sim} (x,v).$
  The uniqueness is an immediate consequence of the  simple  connectivity of $L$ and  the homotopy law for $\mathcal A.$
  Clearly, if  $\omega(0)=x$ then formula (\ref{eq_chi-_x_v_omega}) becomes the usual  formula (\ref{eq_chi-_x,v,omega}).
  
 Using  (\ref{eq_esup_extended_version}) and  the  functions  $\chi^-_{x,v}$  given in (\ref{eq_chi-_x_v_omega}),  consider the function    $\chi^-:\  L\times\R^d  \to \R\cup\{\pm\infty\}$ given by
 \begin{equation}\label{eq_chi-_x_v}
\chi^-(x,v):=\esup \chi^-_{x,v} , \qquad (x,v)\in  L\times\R^d .
\end{equation}
{\bf Case 2:}  {\it $L$ is arbitrary.}

 Let  $\pi:\  \widetilde L\to L$  be  the universal  cover.  
We  construct a  cocycle $\widetilde{\mathcal A}$ on $\widetilde L$ as  follows:
\begin{equation}\label{eq_cocycle_covering_manifold}
\widetilde{\mathcal A}(\tilde  \omega,t):=\mathcal A(\pi (\tilde  \omega),t),\qquad  t\in\R,\ \tilde  \omega\in \widehat\Omega( \widetilde L  ).
\end{equation}
Since $\widetilde L$ is  simply connected, we may apply  Case 1. More  exactly,  we can  define  
  $\chi^-:\  \widetilde L\times\R^d  \to \R\cup\{\pm\infty\}$  by  formula (\ref{eq_chi-_x_v}):
 \begin{equation}\label{eq_chi-_x_v_new}
\chi^-(\tilde x,v):=\esup  \chi^-_{\tilde x,v},\qquad (\tilde x,v)\in  \widetilde L\times\R^d  .
\end{equation}
Here   $\esup$  is  defined  by (\ref{eq_esup_extended_version}) and the function 
$ \chi^-_{\tilde x,v}:\   \widehat \Omega(\widetilde L) \to \R\cup\{\pm\infty\}$ is given by
\begin{equation}\label{eq_chi-_x_v_omega_new}
\chi^-_{\tilde x,v}(\tilde \omega):=\limsup_{n\to\infty}{1\over n} \log\| \widetilde{\mathcal  A}(\tilde \omega,-n)u_{\tilde x,v,\tilde \omega}\|,\qquad  \tilde \omega \in\widehat\Omega(\widetilde L),
\end{equation}
  where $u_{\tilde x,v,\tilde \omega}\in\R^d$ is uniquely determined  by the  condition that $(\tilde \omega(0),u_{\tilde x,v,\tilde \omega})\overset{\widetilde {\mathcal A}}{ \sim} (\tilde x,v),$ namely,
$ u_{\tilde x,v,\tilde \omega} = \widetilde{\mathcal  A}(\tilde \eta,1)v$ for some (and hence  every) path $\tilde \eta\in \Omega(\widetilde L)$  with  $\tilde\eta(0)=\tilde x$ and  $\tilde\eta(1)=\tilde \omega(0).$ 
  
  Continuing  the  prototype of Definition \ref{D:equivlent_relation_1}, we have the following
 \begin{definition}\label{D:equivlent_relation_2}
For   $x\in L$  and $u,v\in  \P(\R^d),$ we  write $u\overset{x,\mathcal A}{\sim} v$
 if  there is  a  path $\omega\in \Omega( L)$  with $\omega(0)=\omega(1)=x$ and $\mathcal A(\omega,1)u=v.$
 This  is  an equivalent relation.
 \nomenclature[i2]{$\overset{x,\mathcal A}{\sim}$}{equivalent relation at  a point $x$ w.r.t. a cocycle $\mathcal A$}
 
 For $x\in L$ and $u\in \P(\R^d),$ let $\class_{x,\mathcal A}(u)$ denote the  (at most countable) set  of all 
 $v\in  \P(\R^d)$  such that $u\overset{x,\mathcal A}{\sim} v.$
 \end{definition}
 \begin{definition}\label{defi_invariance_under_deck-transformations} \rm
$\bullet$
 A set
 $\widetilde\Fc \subset \widehat\Omega(\widetilde L)$  is said to be {\it invariant  under deck-transformations}
\index{invariant!$\thicksim$ under deck-transformations}
 \index{deck-transformation!invariant under $\thicksim$s}
 if $\gamma\circ \widetilde\Fc=\widetilde\Fc$
for all $\gamma\in\pi_1(L).$ Clearly,  this  property  is  equivalent  to  the condition $\widetilde\Fc=\pi^{-1}(  \Fc)$
for some set $\Fc\subset \widehat\Omega( L).$
\\
$\bullet$
 A set
 $\widetilde\Fc \subset \widehat\Omega(\widetilde X,\widetilde\Lc)$  is said to be {\it invariant  under deck-transformations} if  $\widetilde\Fc \cap \widehat\Omega(\widetilde L)$  is  invariant  under deck-transformations for each leaf $L$ of $(X,\Lc).$ Clearly,  this  property  is  equivalent  to  the condition $\widetilde \Fc=\pi^{-1}( \Fc)$ for some $\Fc\subset  \widehat\Omega,$ where $\pi:\ (\widetilde X,\widetilde\Lc)\to (X,\Lc)$ is the covering lamination projection.
\\
 $\bullet$ The above two  definitions can be  adapted in a natural way to the  case where $\widetilde\Fc \subset \Omega(\widetilde L)$
and to the  case where   $\widetilde\Fc \subset \Omega(\widetilde X,\widetilde\Lc).$ 
\end{definition}
  \begin{lemma}\label{lem_backward_holonomy_invariant}
 (i)  Let $\gamma \in \pi_1(L)$ be a  deck-transformation and let   $\tilde x_1,\tilde x_2\in \widetilde L$ and $v_1,v_2\in\R^d$ be such that
$\gamma(\tilde x_1)=\tilde x_2$ and that
$\widetilde {\mathcal A}(\tilde\gamma,1)v_1=v_2,$    where $\tilde\gamma\in \Omega(\widetilde L)$ is  a path such that $\tilde\gamma(0)=\tilde x_1$ and  $\tilde\gamma( 1)=\tilde x_2.$
 Then,  for every $\tilde \omega\in \widehat\Omega(\widetilde L),$
 $$\chi^-_{\tilde x_1,v_1}(\tilde\omega)= \chi^-_{\tilde x_2,v_2} (\tilde\omega)\quad\text{and}\quad \chi^-_{\tilde x_1,v_1}(\tilde\omega)= \chi^-_{\tilde x_2,v_1}(\gamma\circ \tilde \omega).$$
(ii)  Suppose now that  $\tilde a,\tilde b\in \pi^{-1}(x)$ for some $x\in L$  and  $ u,v\in\P(\R^d)$  such that
$u\overset{x,\mathcal A}{\sim} v.$ 
Then
 $\chi^-(\tilde a,u)= \chi^-(\tilde b,v).$
\\
(iii)
For   $x\in X$ and $u\in\R^d\setminus\{0\},$  there  exists a  set $\Fc=\Fc_{x,u}\subset \widehat\Omega(\widetilde L)$ which is of full measure in $\widetilde L_x$ and which is  invariant  under deck-transformations
such that for every  $\tilde  a,\tilde b\in\pi^{-1}(x),$ and  every $v \in\class_{x,\mathcal A}(u),$ 
it holds that  
  $$
\chi^-(\tilde a,u):=\sup_{\Fc}  \chi^-_{\tilde b,v}(\tilde\omega) .
$$
\end{lemma}
\begin{proof}
Note that    for all $\tilde \omega \in\widehat\Omega(\widetilde L),$ we have that
$$(\tilde\omega(0),u_{\tilde x_1,v_1,\tilde \omega})\overset{\widetilde{\mathcal A}}{\sim}(\tilde x_1,v_1)\quad \text{and}\quad
 (\tilde\omega(0),u_{\tilde x_2,v_2,\tilde \omega})\overset{\widetilde{\mathcal A}}{\sim}(\tilde x_2,v_2).$$
 This, combined  with 
 $(\tilde x_1,v_1)\overset{\widetilde{\mathcal A}}{\sim}(\tilde x_2,v_2),$ implies that
  $ u_{\tilde x_1,v_1,\tilde \omega}=u_{\tilde x_2,v_2,\tilde \omega}.$ 
Hence, by (\ref{eq_chi-_x_v_omega_new}) we have that
$
 \chi^-_{\tilde x_1,v_1}(\tilde \omega)= \chi^-_{\tilde x_2,v_2}(\tilde \omega),$ as  asserted.
 Next, suppose  without loss of generality that $\tilde\omega(0)=\tilde x_1.$
So 
$ \chi^-_{\tilde x_1,v_1}(\tilde\omega)=\limsup_{n\to\infty}{1\over n} \log\| \widetilde{\mathcal  A}(\tilde \omega,-n) v_1\|,$ which is also  equal to  
$\limsup_{n\to\infty}{1\over n} \log\| \widetilde{\mathcal  A}(\gamma\circ \tilde \omega,-n) v_1\|,$
by the  definition of the cocycle $\widetilde{\mathcal A}$ in (\ref{eq_cocycle_covering_manifold}) and by the identity $\pi\circ \gamma=\pi$ on $\widetilde L.$
Hence, $ \chi^-_{\tilde x_1,v_1}(\tilde\omega)= \chi^-_{\tilde x_2,v_1}(\gamma\circ \tilde\omega)$  as $\gamma\circ \tilde\omega(0)=\gamma ( \tilde x_1)=\tilde x_2.$  This   proves the last  identity of assertion (i).

We turn to the proof of   assertion (ii).  
 Since   $u\overset{x,\mathcal A}{\sim} v,$ there  exists $\tilde c\in\pi^{-1}(x)$   such that
 $v:= \widetilde {\mathcal A}(\tilde\gamma,1)u,$    where $\tilde\gamma\in \Omega(\widetilde L)$ is  a path such that  $\tilde\gamma(0)=\tilde a$ 
and  $\tilde\gamma(1)=\tilde c.$
  We deduce from  the  first  identity of assertion (i) and the definition of $\chi^-(x,v)$ in Case 1 (see  formula  (\ref{eq_chi-_x_v_omega}))  that
  $
\chi^-(\tilde a,u)= \chi^-(\tilde c,v  ).$

Let $w:= \widetilde {\mathcal A}(\tilde\beta,1)u,$    where $\tilde\beta\in \Omega(\widetilde L)$ is  a path such that  $\tilde\beta(0)=\tilde a$ and $\tilde\beta(1)=\tilde b.$
 Let  $\tilde\alpha$ be  a path such that    $\tilde\alpha|_{[0,1]}$ is  the concatenation $
\tilde \gamma^{-1}|_{[0,1]}\circ \tilde \beta|_{[0,1]}.$ Thus,  $\tilde\alpha|_{[0,1]}$ connects  $\tilde c$ to $\tilde b.$
 By Lemma   \ref{lem_backward_esup},  let  $\tilde E \subset\widehat\Omega(\widetilde L)$ be a set of full measure in $\widetilde L$ such that
 $$\chi^-(\tilde c,v)=\sup_{\tilde\omega\in \tilde E} \chi^-_{\tilde c, v}(\tilde\omega).$$
 Let $\alpha$ be  the deck-transformation sending  $\tilde c$ to $\tilde b.$
Since $\tilde E$ is  of full measure in  $\widetilde L,$ we can check using Definition \ref{defi_set_null_measure} that   $\alpha\circ \tilde E$
is also  of full measure in  $\widetilde L.$ This, combined with the  second  identity of assertion (i), implies that 
$
\chi^-(\tilde c,v)\geq  \chi^-(\tilde b,v).$
Arguing as  above for $\alpha^{-1},$ we get that $
\chi^-(\tilde c,v)= \chi^-(\tilde b,v).$
This, combined  with  the identity $
\chi^-(\tilde a,u)= \chi^-(\tilde c,v  ),$  completes assertion (ii).

Now we  prove  assertion (iii). Observe by assertion (ii) that we may assume  without loss of generality that
$\tilde a=\tilde b.$
By Lemma \ref{lem_backward_esup}, for each 
$v
\in\class_{x,\mathcal A}(u),$ 
there  is a  set $\Fc_v\subset  \widehat\Omega(\widetilde L)$  of full measure in $\widetilde L_x$  with the  following
property:
  $$
\chi^-(\tilde a,u):=\sup_{\Fc_v}  \chi^-_{\tilde a,v}(\tilde\omega) .
$$
Consider  the set
$$
\Fc:=  \bigcap_{v\in\class_{x,\mathcal A}(u)}  \big (\bigcap_{\gamma\in \pi_1(L)}  \gamma\circ \Fc_v\big) .
$$
Observe that  each set   $\gamma\circ \Fc_v$ is  of full measure in  $\widetilde L_x$ and that  the above  intersection is  at most countable.
Clearly, by Part 3) of Proposition    \ref{prop_backward_set_classification} $\Fc$ 
is of full measure in $\widetilde L_x$ and  is  invariant  under deck-transformations. 
Using Lemma \ref{lem_backward_esup} again and  the above  property  of  $\Fc_v,$  we   see that $\Fc$ satisfies the  conclusion of  assertion (iii).
\end{proof}
In view of Lemma \ref{lem_backward_holonomy_invariant} (ii) we  are able  to define the function 
 $\chi^-:\   L\times\R^d  \to \R\cup\{\pm\infty\}$ as follows.  By convention   $
\chi^-( x,0):=-\infty.$  For $u\in \R^d\setminus \{0\},$  set 
 \begin{equation} \label{eq_functions_chi-}
\chi^-( x,u):=\chi^-(\tilde x, v) ,\qquad (x,v)\in  L\times\R^d ,
 \end{equation}
where $\chi^-(\tilde x, v)$  is  calculated  using  (\ref{eq_chi-_x_v_new})-(\ref{eq_chi-_x_v_omega_new}),
and $\tilde x$ is an arbitrary  element in   $\pi^{-1}(x)$ and $v$ is an arbitrary  element in  $\class_{x,\mathcal A}(u).$ 

  We record   here the  properties of $\chi^-.$ Some of them are analogous to those of $\chi$  stated in  Proposition \ref{prop_chi}.
 \begin{proposition}\label{prop_chi-}
  (i)  $\chi^-$ is  a measurable function  on the  measurable space $(X\times \R^d, \Bc(X)\otimes\Bc(\R^d)).$
\\
(ii)   $\chi^-(x,u)=\chi^-(y,v)$  if there exists a path $\omega\in \Omega(L)$
  such that $\omega(0)=x,$  $\omega(1)=y$ and  $\mathcal A(\omega,1)u=v.$
  \\ (iii) $\chi^-(x,0)=-\infty$   and 
    $\chi^-(x,v)=\chi^-(x,\lambda v)$  for  $x\in X,$   $v\in\R^d,$ $\lambda\in\R\setminus \{0\}.$
   So we   can define  a  function, still  denoted  by $\chi^-,$  defined on $X\times\P(\R^d)$ by
$$ \chi^-(x,[v]):=\chi^-(x,v),\qquad    x\in X,\ v\in\R^d\setminus \{0\}. $$ 
(iv) $\chi^-(x,v_1+v_2)\leq  \max\{ \chi^-(x,v_1), \chi^-(x,v_2)\},  $ $x\in X,$    $v_1,v_2\in\R^d.$
\\ (v)  For all $x\in X$ and $t\in\R\cup\{\pm\}$ the  set
$$ V^-(x,t):= \{ v\in\R^d:\    \chi^-(x,v)\leq  t \}$$
is  a linear  subspace of $\R^d.$ Moreover,    $s\leq t$
implies $V(x,s)\subset V(x,t).$ 
\\ (vi)  For every $x\in X,$  $\chi^-(x,\cdot):\  \R^d\to\R\cup\{-\infty\}$ takes only finite  $m^-(x)$  different  values   
 $$\chi^-_{m^-(x)}(x)>\chi^-_{m^-(x)-1}(x)>\cdots >\chi^-_2(x)>\chi^-_1(x).$$
 \\ (vii) If, for  $x\in X,$  we  define  $V^-_i(x)$ to be   $V^-(x,\chi^-_i(x))$ for $1\leq  i\leq m^-(x),$
then 
$$
\{0\}\equiv  V^-_0(x)\subset  V^-_1(x)\subset\cdots\subset V^-_{m^-(x)-1}(x)\subset  V^-_{m^-(x)}(x)\equiv\R^d
$$   
and
$$
v\in V^-_i(x)\setminus V^-_{i-1}(x)\Leftrightarrow \sup_{\tilde\omega\in E_x}\limsup_{n\to\infty} {1\over n} \log \| \widetilde{\mathcal A}(\omega, -n)  u_{\tilde x,\tilde v,\tilde \omega}  \| =\chi^-_i(x)
$$
for some (and hence every)  $\tilde x\in\pi^{-1}(x)$ and some (and hence every)  $\tilde v\in \class_{x,\mathcal A}(v).$ Here  $E_x\in \widehat\Ac(\widetilde L_x)$ is a set of full  measure in $\widetilde L_x$
which is  also  invariant  under deck-transformations,  $E_x$ depends only on the point $x$  (but it does  not depend on  $v\in \R^d$).
 \end{proposition}
 \begin{proof}    Arguing   as  in the proof of  Proposition  \ref{prop_chi} (i)
  and  using Proposition \ref{prop_measurability_extended} instead of Proposition \ref{prop_measurability},
  the  measurability of $\chi^-$  follows.     Assertion (ii) is  an  immediate consequence of Lemma \ref{lem_backward_holonomy_invariant}.
 
 Applying Part 3) of Proposition  \ref{prop_backward_set_classification} and Lemma \ref{lem_backward_esup}, we proceed  as in the proof of 
 assertions (ii)--(v) of Proposition \ref{prop_chi}. Consequently, assertions (iii)--(vi)  follow.
 Arguing as in the proof  of assertion (vi) of Proposition \ref{prop_chi} and  applying  Lemma \ref{lem_backward_holonomy_invariant} (iii),
 assertion  (vii) follows.  
 \end{proof}
 As  a  consequence  we obtain an  analogue of  Proposition \ref{prop_leafwise_Oseledec} in the  backward setting.  
 \begin{proposition}\label{prop_leafwise_Oseledec_backward_orbit}
Let $L$ be a leaf of $(X,\Lc)$ and $\mathcal A$  a cocycle on $(X,\Lc).$ Then  there  exist  a  number $m^-\in \N$ and $m^-$  integers $1\leq d^-_1<\cdots <d^-_{m^--1}<d^-_{m^-}=d$ and $m^-$ real numbers
$\chi^-_1< \cdots \chi^-_{m^--1}< \chi^-_{m^-}$    with the  following properties:
 \\ (i) for every $x\in L$  and  every $1\leq j\leq m^-,$   we have
$m^-(x)=m^-$ and  $\chi^-_j(x)=\chi^-_j$ and  $\dim V^-_j(x)=d^-_j;$ 
\\ (ii)  for  every $x,y\in L,$ and  every $\omega\in \Omega_x$ such that $\omega(1)=y,$
we have  $\mathcal A(\omega,1)V^-_i(x)=V^-_i(y)$   with $1\leq i\leq m^-.$    
\end{proposition}
\begin{proof}
It follows  from the  definition of $\chi^-$ and  assertion  (ii) of Proposition \ref{prop_chi-}.
\end{proof}
\begin{remark}
The reader  should  notice  the difference  between the  conclusion of Proposition  \ref{prop_leafwise_Oseledec} and  that of Proposition  \ref{prop_leafwise_Oseledec_backward_orbit}.
Indeed, in the  former  the desired  properties  only hold  $\Vol_L$-almost everywhere, whereas in the latter these  properties hold everywhere in $L.$ This  difference  is  a consequence of  the
rather special  definition of the  function $\chi^-.$ Note that this peculiar definition allows us  to avoid  the  use of Proposition \ref{prop_Markov}
since  the  following facts holds (please check it using the very definition!): If $A\in \widehat\Ac(L)$ is  of full measure in the leaf $L,$ then so is  $T^{-1}A.$ 

Given a leaf $L$ we often write  $m^-(L)$  instead of $m^-$ given by Proposition \ref{prop_leafwise_Oseledec_backward_orbit} in order to 
 emphasize the dependence of  $m^-$ on  the leaf $L.$ The same  rule also applies  as  well to other quantities obtained by this  proposition.

 The  real numbers  $\chi^-_1< \cdots \chi^-_{m^--1}< \chi^-_{m^-}$  are called the {\it  leafwise Lyapunov backward exponents} associated to the cocycle $\mathcal A$
on the leaf $L.$
\index{leafwise!$\thicksim$ Lyapunov backward exponent}
 The  increasing sequence of subspaces of $\R^d:$
 $$
\{0\}\equiv  V^-_0(x)\subset  V^-_1(x)\subset\cdots\subset V^-_{m^-(x)-1}(x)\subset  V^-_{m^-(x)}(x)\equiv\R^d
 ,\qquad x\in L,$$
is  called the {\it  leafwise  Lyapunov backward filtration}\index{leafwise!$\thicksim$ Lyapunov backward filtration}\index{filtration!leafwise Lyapunov backward $\thicksim$}
 associated to $\mathcal A$ at  a given  point $ x\in L.$
\end{remark}

For the  rest of the chapter,
 we consider  a       lamination $(X,\Lc,g)$   satisfying the  Standing Hypotheses endowed with
 a   harmonic probability measure  $\mu$ which is  ergodic. Set  $\Omega:=\Omega(X,\Lc)$ and   $\widehat\Omega:=\widehat\Omega(X,\Lc).$
 So  $\hat\mu$ is invariant and ergodic  with respect to $T$  acting on  the
probability space $(\widehat\Omega,\widehat\Ac,\hat\mu).$  
We also  consider  a (multiplicative) cocycle $\mathcal{A}:\ \Omega\times \N \to  \GL(d,\R)  ,    $  
and 
   assume that
     $\int_\Omega \log^+ \|\mathcal{A}^{\pm 1}(\omega,1)\|  d\bar\mu(\omega)<\infty.
 $ We extend it naturally  to  a cocycle (still denoted  by) $\mathcal A:\  \widehat\Omega(X,\Lc)\times\Z\to\GL(d,\R)$  using formula  (\ref{eq_formula_extended_cocycle}).
  Applying Theorem  \ref{thm_Oseledec} (i)-(v), we obtain the following complement  to Theorem   \ref{th_Lyapunov_filtration_Brownian_version}.

\begin{theorem} \label{thm_Oseledec_Brownian_version}
We keep the  notation  introduced  by Theorem    \ref{th_Lyapunov_filtration_Brownian_version}.
There exists  a   subset $\Psi$ of $\widehat\Omega$ of  full $\hat\mu$-measure 
such that  the  following properties hold:
\\
(i) For  each $\omega\in \Psi$
 there   are $l$  linear subspaces $H_1(\omega),\ldots, H_l(\omega)$   of $\R^d$  such that  
$$V_j(\omega)=\oplus_{i=j}^l H_i(\omega), 
$$
 with $\mathcal{A}(\omega,1) H_i(\omega)= H_i(T\omega),$  
and   $\omega\mapsto  H_i(\omega)$ is   a  measurable map from $\Psi$ into the Grassmannian of $\R^d.$
 \\ (ii) For each  $\omega\in \Psi$ and $v\in H_i(\omega)\setminus \{0\},$    
$$\lim\limits_{n\to \pm\infty} {1\over  |n|}  \log {\| \mathcal A(\omega,n)v   \|\over  \| v\|}  =\pm\lambda_i,    
$$
and  the following limit holds
$$
\lim\limits_{n\to \infty} {1\over n}  \log\sin {\big |\measuredangle \big (H_S(T^n\omega), H_{N\setminus S} (T^n\omega)\big ) \big |}=0,
$$  where, for any subset   $S$  of $  N:=\{1,\ldots,l\},$ we define $H_S(\omega):=\oplus_{i\in S} H_i(\omega).$  
 \\ (iii) 
$\lambda_l=-\lim_{n\to\infty}{1\over  n}  \log {\| \mathcal A(\omega,-n)}\|$  for   every  $\omega\in\Psi.$
 \end{theorem}

We proceed as  we did in deducing   Corollary \ref{cor_leafwise_Oseledec} from Proposition \ref{prop_leafwise_Oseledec}.
Consequently,  we infer  from Proposition \ref{prop_leafwise_Oseledec_backward_orbit}  the following
\begin{corollary}\label{cor_leafwise_Oseledec_backward}
  There  exist  a  leafwise saturated Borel set $Y\subset X$ of full $\mu$-measure and  a    number $m^-\in \N$ and
 $m^-$ integers  $1\leq d^-_1<d^-_2<\cdots<d^-_{m^-}=d$ and $m^-$ real numbers
$\chi^-_1<\chi^-_2<\cdots< \chi^-_{m^-}$
     with the  following properties:
\\ (i)  $m^-(L_x)=m^-$ for  every $x\in Y;$
\\ (ii)   the map $Y\ni \mapsto V^-_i(x)$  is  an $\mathcal A$-invariant  subbundle  of rank $d^-_i$ of  $Y\times \R^d$  for $1\leq i\leq m^-;$ 
\\ (iii) for every $x\in Y$  
and   $1\leq i\leq  m,$
$\chi^-_i(L_x)=\chi^-_i.$     
\end{corollary}

 Now we  are in the  position  to state  the  analogue version of Theorem  \ref{th_Lyapunov_filtration} in the  backward setting.
 \begin{theorem} \label{th_Lyapunov_filtration_backward}
Under  the above hypotheses and notation  we have  that $m^-\leq l$ and
$\{  \chi^-_1,\ldots,\chi^-_{m^-}\}\subset \{-\lambda_1,\ldots, -\lambda_l\}$ and 
$\chi^-_{m^-}=-\lambda_l.$  Moreover,  there  exists 
 a leafwise  saturated Borel set $Y_0\subset Y$ of full $\mu$-measure  such that for every $x\in Y_0$ and
 for every $v\in V^-_i(x)\setminus V^-_{i-1}(x)$ with $1\leq i\leq m^-,$ and for every 
$\tilde x\in \pi^{-1}(x),$ there is 
a  set $\Fc=\Fc_{\tilde x,v}\subset\widehat\Omega(\widetilde L_x)$  of positive  measure in $\widetilde L_x$  such that  
  $\chi^-_{\tilde x,v}(\tilde\omega) =\chi^-_i$ for every $\tilde\omega\in \Fc.$ 
   \end{theorem}
   \begin{remark}
   The increasing  sequence  of subspaces of $\R^d:$
    $$\{0\}\equiv V^-_0(x)\subset V^-_1(x)\subset \cdots\subset V^-_{m^-}(x)=\R^d$$
is  called the {\it  Lyapunov backward filtration} associated to $\mathcal A$ at  a given  point $ x\in Y_0.$
  The remarkable point of Theorem \ref{th_Lyapunov_filtration_backward} is that
this filtration depends only on the point $x,$ and  not on paths $\omega\in\widehat\Omega(\widetilde{L}_x).$\index{filtration!Lyapunov backward $\thicksim$}\index{Lyapunov!$\thicksim$ backward filtration}
\end{remark}
\begin{proof}
   By Corollary  \ref{cor_leafwise_Oseledec_backward}    and  Theorem   \ref{thm_Oseledec_Brownian_version} 
 and Part 1) of  Proposition \ref{prop_backward_set_classification},
 there is a leafwise saturated Borel set $Y_0\subset Y$ of full $\mu$-measure such that for every $x\in Y_0,$  the set
$\Psi\cap \widehat\Omega(L_x)$ is  of  full measure in $L_x.$ By  Lemma  \ref{lem_backward_holonomy_invariant} (iii), we obtain, for each $x\in  Y_0$ and $v\in \R^d\setminus\{0\},$   a  set $\widetilde E\subset \widehat\Omega(\widetilde L_x)$ which is of full measure in $\widetilde L_x$ and  which is also invariant  under deck-transformations  such that   for an arbitrary  $\tilde x\in \pi^{-1}(x),$ 
\begin{equation}\label{eq_chi_minus}
\chi^-(x,v)=\esup \chi^-_{\tilde x,v}(\tilde \omega)=\sup_{\tilde \omega\in \widetilde E} \chi^-_{\tilde x,v}(\tilde \omega) .
\end{equation}
Let  $E:=\pi\circ \widetilde E \subset \widehat \Omega(L).$ So $\pi^{-1}(E)=\widetilde E.$
By  Lemma   \ref{lem_projection_null_measure}, $E$ is  of full measure in the leaf $L_x.$ 
Recall from above that $\Psi\cap \widehat\Omega(L_x)$ is  of  full measure in $L_x.$
So  by Part 1) of Proposition \ref{prop_backward_set_classification} and by Part 2) of Lemma  \ref{lem_projection_null_measure} again,   
the  set  $E\cap  \Psi$ is     of full measure in the leaf $L_x,$ and the  set  $\pi^{-1}(E\cap  \Psi)$ is   of full measure in the leaf $\widetilde L_x,$
Hence, by  Theorem      \ref{thm_Oseledec_Brownian_version} (ii),  we have, for $\tilde \omega \in \pi^{-1}(E\cap  \Psi),$ that
\begin{equation}\label{eq_chi_minus_omega}
\chi^-_{\tilde x,v}(\tilde \omega)=\lim_{n\to\infty}{1\over  n}  \log {\| \mathcal A(\omega,n)u   \|\over  \| u\|}   
 \in   \{-\lambda_1,\ldots, -\lambda_l\}.
\end{equation}
 Here  $\omega:=\pi\circ \tilde\omega$ and $u:= u_{\tilde x,v,\tilde \omega}.$  
   This, combined with (\ref{eq_chi_minus}) and Corollary \ref{cor_leafwise_Oseledec_backward},  implies that $\{  \chi^-_1,\ldots,\chi^-_{m^-}\}\subset \{-\lambda_1,\ldots, -\lambda_l\}.$ In particular, we get that $m^-\leq l$
   and $\chi^-_{m^-}\leq -\lambda_l.$
    Therefore, in order to prove  that  $\chi^-_{m^-}=-\lambda_l,$   it suffices to show that $\chi^-_{m^-}\geq-\lambda_l.$  
By  Proposition \ref{prop_chi-} (vii),
  for  every $x\in Y_0$  there  exists set $\widetilde E_x\in \widehat\Ac(\widetilde L_x)$ which is  full measure in $\widetilde L_x$ and  which
is also invariant  under deck-transformations   such that 
 $$
v\in V^-_i(x)\setminus V^-_{i-1}(x)\Leftrightarrow \sup_{\tilde\omega\in \widetilde E_x}\limsup_{n\to\infty} {1\over n} \log \| \widetilde{\mathcal{A}}(\tilde\omega, -n)   u_{\tilde x,v,\tilde \omega}   \| =\chi^-_i.
$$
Let $E_x:=\pi\circ \widetilde E_x.$ By Lemma  \ref{lem_projection_null_measure},  $E_x$ is  of full measure in $L_x.$
 Therefore, by the definition, we get,  for every $v\in \R^d\setminus \{0\},$  for  every $x\in Y_0$ and every $\omega\in E_x,$  that
 \begin{eqnarray*}
 \limsup_{n\to\infty} {1\over n} \log \| \mathcal{A}(\omega,- n)    u_{\tilde x,v,\tilde \omega}    \| &=& \limsup_{n\to\infty} {1\over n} \log \| \widetilde{\mathcal{A}}(\tilde\omega,- n)    u_{\tilde x,v,\tilde \omega}    \| \\
&\leq & \max \{  \chi^-_1,\ldots,\chi^-_{m^-}\}= \chi^-_{m^-},
\end{eqnarray*} where $\tilde\omega$ is any path in $ \widehat \Omega(\widetilde L)$ such that $\pi\circ \tilde\omega=\omega.$ This, coupled with (\ref{eq_chi_minus_omega}), gives that 
 $\chi^-_{m^-}\geq-\lambda_l,$ as  desired. 
 
  Finally, the existence  of 
a  set $\Fc=\Fc_{\tilde x,v}\subset\widehat\Omega(\widetilde L_x)$ with the desired  property stated  in the  theorem
is  an immediate consequence of 
combining  (\ref{eq_chi_minus}) and  (\ref{eq_chi_minus_omega}) and the equality $\chi^-_{m^-}=-\lambda_l$ and Corollary  \ref{cor_leafwise_Oseledec_backward}.
 \end{proof}
  Finally, we conclude the chapter with  the following backward version of Theorem  \ref{thm_splitting}.
 \begin{theorem}\label{thm_splitting_backward}
 Let  $Y\subset X$ be a set of full $\mu$-measure. 
 Assume    that
   $Y\ni x\mapsto U(x)$ and  $Y\ni x\mapsto  V(x)$ are two measurable  $\mathcal A$-invariant  subbundles of $Y\times \R^d$ with $V(x)\subset U(x),$ $x\in Y.$ Define a new measurable     subbundle  $Y\ni x\mapsto  W(x)$ of $Y\times \R^d$  
  by  splitting 
$U(x)=V(x)\oplus W(x)$  so that $W(x)$ is  orthogonal  to $V(x)$  with respect to the  Euclidean inner product of $\R^d.$  Using (\ref{eq_cocycle_covering_manifold}), 
we define the cocycle  $\widetilde{\mathcal A}$     on $\widetilde\Omega\times \Z$
in terms of   $\mathcal A.$  Let   $\alpha,$ $\beta$ be two  real  numbers   with $\alpha<\beta$ such that
 
 $\bullet$  $ \chi^-(x, v)\leq \alpha$ for every $ x\in Y,$  $ v\in V(x)\setminus\{0\};$ 
 
$\bullet$
 $
  \chi^-_{\tilde x, w}(\tilde \omega)  \geq \beta$ for every $x\in Y,$ every $\tilde x\in\pi^{-1}(x),$
 every $w\in W(x),$ 
  and for  every $\tilde\omega\in \widehat\Gc_{\tilde x,w}.$
Here  $\widehat\Gc_{\tilde x,w}\subset   \widehat\Omega(\widetilde L_x)$ (depending on $\tilde x$ and $w$)  is  of positive  measure in $\widetilde L_x,$   the  function   $ \chi^-(x, v)$  (resp.  $
  \chi^-_{\tilde x, u}(\tilde \omega) $) is  defined  in (\ref{eq_functions_chi-})  (resp.  in (\ref{eq_chi-_x_v_omega})).

Let $\mathcal A(\omega,-1)|_{U(\pi_0\omega)}:\ U( \pi_0\omega    )\to  U( \pi_{-1}\omega    ) $ induce the linear maps
 $\mathcal C(\omega):\    W( \pi_0\omega    )\to  W( \pi_{-1}\omega    )$  and  $\mathcal B(\omega):\  W( \pi_0\omega    )\to  V( \pi_{-1}(\omega)    )$ 
 by
 $$
 \mathcal A(\omega,-1)w=\mathcal B(\omega)w\oplus \mathcal C(\omega)w,\qquad  \omega\in\widehat\Omega,\ w\in W(x).
 $$
(i) Then  the map $\mathcal C$ defined   on $ \widehat\Omega\times (-\N)$ by the formula 
$$\mathcal C(\omega,-n):= \mathcal C(T^{-n}\omega)\in \Hom( W( \pi_0\omega    ),  W( \pi_{-n}\omega    ),\qquad \omega\in \widehat\Omega,\ n\in \N,$$
satisfies  $\mathcal C (\omega,-(m+k))=\mathcal C(T^{-k}\omega,m)\mathcal C(\omega,-k),$  $m,k\in \N.$ Moreover, 
$\mathcal C(\omega,-n)$ is invertible.

Using (\ref{eq_cocycle_covering_manifold}), 
we define the cocycle     $\widetilde{\mathcal C}$ on $\widetilde\Omega\times \Z$
in terms of    $\mathcal C.$
Then
 there exists a   subset $Y'$ of $Y$ of full $\mu$-measure  with the  following properties:
 \\ (ii)  For each $ x\in Y'$ and  for each $\tilde x\in \pi^{-1}(x)$ and for each $ w\in W(x)\setminus \{0\},$ there  exists a  set $\widehat\Fc_{\tilde x,w}\subset \widehat \Gc_{\tilde x,w}$  such that  $\widehat\Fc_{\tilde x,w}$ is     of positive measure
 in  $\widetilde L_x$ and that  for  each $ v\in V(x)$ and  each $\tilde\omega\in  \widehat\Fc_{\tilde x,w},$   we have 
 $$
 \chi^-_{\tilde x, v\oplus w}(\tilde \omega)= \chi^-_{\tilde x,  w}(\tilde \omega)=\limsup_{n\to\infty} {1\over n} \log \|\widetilde{  \mathcal C}(\tilde\omega,-n)u_{\tilde x, w,\tilde\omega}\|;
 $$
 (iii)  
  if   for   some  $x\in Y'$  and some $w\in W(x)\setminus \{0\}$ and some $ v\in V(x)$ 
  and some  $\tilde\omega\in \widehat\Fc_{x,w}$ the limit $\lim_{n\to\infty} {1\over n}\log  \|\widetilde{\mathcal  C}(\tilde\omega,-n)u_{\tilde x, w,\tilde\omega}\|  $ exists, then
 $$\lim_{n\to\infty} {1\over n} \log \|  \widetilde{\mathcal A}(\tilde\omega,-n)u_{\tilde x,v\oplus w,\tilde\omega}\| $$ exists  
 and  is  equal to    the previous limit.
 \end{theorem}
\begin{proof}
Using  the multiplicative  property of the  cocycle  $\mathcal A,$  we obtain the following  formula, which is 
the backward version of  (\ref{eq_iterations}) in  Theorem \ref{thm_splitting} above:
for  $\omega\in \widehat\Omega$ and $n\in\N,$
\begin{equation}\label{eq_iterations_backward}
  \mathcal A(\omega,-n)(v\oplus w)= \big ( \mathcal A(\omega,-n)v+ \mathcal D(\omega,-n)w\big ) \oplus \mathcal C(\omega,-n)w,
 \end{equation}
where $\mathcal D(\omega,-n):\  W(\pi_0\omega)\to V(\pi_{-n}\omega)$  is given by
$$ \mathcal D(\omega,-n):=\sum_{i=0}^{n-1} \mathcal A(T^{-(i+1)}\omega,-(n-i-1))\circ \mathcal B(T^{-i}\omega)\circ \mathcal C(\omega, -i).  $$ 
Next,  observe that Lemma \ref{lem_subadditive_estimate}
 and  Lemma \ref{lem_tail_term}  remain valid if we replace $\Omega(X,\Lc),$ $T$ and $\bar\mu$ with
$\widehat\Omega(X,\Lc),$ $T^{-1}$ and $\hat\mu$ respectively. 
For $\epsilon>0$  let $$a_\epsilon(\omega):=\sup_{n\in\N}\big (\| \mathcal A (\omega,-n)|_{V(\pi_0\omega)}  \| \cdot e^{-n(\alpha+\epsilon)}\big).$$ By the first  assumption $\bullet,$ we may apply  Lemma \ref{lem_subadditive_estimate}
to $a_\epsilon(\omega),$ and  Lemma \ref{lem_tail_term} to $h(\omega):= \|\mathcal A(\omega,-1)\|,$ $\omega\in \widehat\Omega(X,\Lc).$
Let $(\epsilon_m)_{m=1}^\infty$ be  a sequence decreasing strictly to  $ 0. $
  By  Lemma \ref{lem_subadditive_estimate} and  Lemma \ref{lem_tail_term},
we may find, for each  $m\geq 1,$  a  subset  $\widehat\Omega_m$  of $\widehat\Omega(Y)$ of full $\hat\mu$-measure such that  ${1\over n} a_{\epsilon_m}(T^{-n}\omega)\to 0$ and 
${1\over n} \log \| \mathcal A(T^{-n}\omega,1)\|\to 0$ for all  $\omega\in\widehat\Omega_m.$ 
For every $x\in Y$  set  $\widehat\Fc'_x:= \widehat\Omega(L_x)\cap \cap_{m=1}^\infty \widehat\Omega_m\subset \widehat\Omega(L_x) .$
Since  $\cap_{m=1}^\infty \widehat\Omega_m$ is of full $\hat\mu$-measure, it follows from  Part 1) of Proposition \ref{prop_backward_set_classification} that there exists a subset $Y'\subset Y$
of full $\mu$-measure  such that for every $x\in Y',$ $\widehat\Fc'_x$  is  of full measure in $L_x.$ 
By the first assumption $\bullet$ combined with Proposition \ref{prop_chi-} (iii)-(iv),   for every $x\in Y',$ there  exists a set  $\widehat\Fc_x\subset \widehat\Fc'_x$    of full  measure in $L_x$ such that, for every $\tilde x\in\pi^{-1}(x)$
and
for every  $\tilde \omega\in\pi^{-1}\widehat\Fc_x,$ 
\begin{equation}\label{eq_thm_splitting_bullets_backward}
\chi^-_{\tilde x,v}(\tilde\omega)\leq \alpha <\beta ,\qquad v\in V(x). 
\end{equation}
 By the second assumption $\bullet,$  for every $x\in Y'$ and for every $\tilde x\in \pi^{-1}(x)$ and for every $w\in W(x)\setminus \{0\},$ there  exists a set  $\widehat\Fc_{\tilde x,w}:= \widehat\Gc_{\tilde x,w}\cap  \pi^{-1}\widehat\Fc_x\subset  \widehat\Omega(\widetilde L_x)$     such that,
for every  $\omega\in\widehat\Fc_{\tilde x,w},$
\begin{equation}\label{eq_thm_splitting_second_bullets_backward}
 \alpha <\beta \leq \chi^-_{x,w}(\tilde\omega). 
\end{equation}
Since $\widehat\Fc_x$ is   of full  measure in $L_x,$ we infer from Part 2) of Lemma  \ref{lem_projection_null_measure}
 that $\pi^{-1}\widehat\Fc_x$ is of full measure in $\widetilde L_x.$  Since
  $\widehat\Fc_{\tilde x,w}$  is the intersection of a set of positive measure  and  a set of full measure in $\widetilde L_x,$
 we deduce from  Remark \ref{rem_intersection_non_empty} that  $\widehat\Fc_{\tilde x,w}$ is  of positive  measure in $\widetilde L_x.$ 

 Using (\ref{eq_iterations_backward})-(\ref{eq_thm_splitting_bullets_backward})-(\ref{eq_thm_splitting_second_bullets_backward}) and making the necessary changes (for example, $n$ is  replaced with $-n$), we argue as in the proof of Theorem \ref{thm_splitting}
from  (\ref{eq_thm_splitting_3}) to the end of that proof. 
\end{proof}

%% file: chapter9.tex

  
 \chapter{Proof of the main  results}
 \label{section_Main_Theorems}
 

In this  chapter  we    prove  the First Main Theorem (Theorem \ref{th_main_1})\index{theorem!First Main $\thicksim$} and  the Second Main Theorem\index{theorem!Second Main $\thicksim$} (Theorem \ref{th_main_2})
as well as their  corollaries. In addition, we also  give a  Ledrappier type  characterization\index{Ledrappier!$\thicksim$ type characterization} of Lyapunov spectrum
(Theorem \ref{thm_Ledrappier} below). 
The  chapter is  organized as follows. In Section \ref{subsection_canonical_cocycle} we introduce  some  terminology, notation
and auxiliary results
which will be  of  constant use  later on. Section  \ref{subsection_weakly_harmonic_measures_and_splitting}
introduces two important techniques.
The first one is designed in order to construct weakly harmonic  measures which maximize (resp. minimize)  certain Lyapunov exponent
 functionals\index{Lyapunov!$\thicksim$ exponent functional}\index{functional|see{Lyapunov}}.
Using the  first  technique  we develop  the second one  which aims at   splitting  invariant  subbundles
 (see Theorem \ref{thm_splitting_model} and  Theorem \ref{thm_dual_splitting_model} below).
 Having at  hand  all needed tools and  combining  them with  the results  established in  Chapter
\ref{section_Lyapunov_filtration} and  Chapter \ref{section_backward_filtration} above,
Section 
\ref{subsection_First_Main_Theorem} and  Section 
\ref{subsection_Second_Main_Theorem} are devoted to the proof of the main results of this  work.
 
In what follows,   for a   linear (real or complex) vector space $V,$
we denote by $\P V$ its  projectivisation.

\section{Canonical  cocycles and specializations}
\label{subsection_canonical_cocycle}
Consider  a       lamination $(X,\Lc,g)$   satisfying the  Standing Hypotheses endowed with
 a   harmonic probability measure  $\mu$ which is  ergodic. Set  $\Omega:=\Omega(X,\Lc)$ and   $\widehat\Omega:=\widehat\Omega(X,\Lc).$
 So  $\hat\mu$ is invariant and ergodic  with respect to $T$  acting on  the
probability space $(\widehat\Omega,\widehat\Ac,\hat\mu).$  
Consider  a (multiplicative) cocycle $\mathcal{A}:\ \Omega\times \N \to  \GL(d,\R)   $
(resp.         $\mathcal{A}:\ \Omega\times \R^+ \to  \GL(d,\R)   $).  
Using formula (\ref{eq_formula_extended_cocycle}), we  extend it to the cocycle (still denoted by)   $\mathcal{A}:\ \widehat\Omega\times \Z \to  \GL(d,\R)   $
(resp.         $\mathcal{A}:\ \widehat\Omega\times \R \to  \GL(d,\R)   $).  

In this  section we consider the cylinder lamination of rank $1$
$(X_{\mathcal A},\Lc_{\mathcal A}):= (X_{1,\mathcal A}, \Lc_{1,\mathcal A}).$
\nomenclature[b4c]{$(X_{\mathcal A},\Lc_{\mathcal A})$}{$:=(X_{1,\mathcal A},\Lc_{1,\mathcal A})$ cylinder lamination of rank $1$ of a cocycle   $\mathcal A$}
We identify  $\Gr_1(\R^d)$ with $\P(\R^d)$ and  write $P:=\P(\R^d).$
\nomenclature[a9c]{$P$}{$:=\P(\R^d)$ in Chapter \ref{section_Main_Theorems}}
 So $X_{\mathcal A}\equiv X\times P$
and  we   will write  $( X\times P,\Lc_{\mathcal A})$ instead of $(X_{\mathcal A},\Lc_{\mathcal A}).$
Let $\Omega:=\Omega(X,\Lc),$  $\widehat\Omega:=\widehat\Omega(X,\Lc),$  $\Omega_{\mathcal A}:=\Omega  (X_{\mathcal A},\Lc_{\mathcal A})$ and    $\widehat\Omega_{\mathcal A}:=\widehat\Omega  (X_{\mathcal A},\Lc_{\mathcal A}).$ 
\nomenclature[c1b]{$\Omega_{\mathcal A}$}{$:=\Omega_{1,\mathcal A}$ sample-path space associated to  the cylinder lamination 
of rank $1$ of a cocycle   $\mathcal A$}
\nomenclature[c1f]{$\widehat\Omega_{\mathcal A}$}{extended sample-path space associated to  the cylinder  lamination of rank $1$
of a cocycle  $\mathcal A$}
Let $\pi:\ (\widetilde X,\widetilde\Lc)\to (X,\Lc)$ be the  covering lamination projection. Let $\widetilde\Omega:=\Omega( \widetilde X,\widetilde\Lc).$ We have the following  natural identifications.
\begin{lemma}\label{lem_identifications_spaces_dim1}
1) The transformation  $\Omega_{\mathcal A}\to  \Omega\times P$ 
which maps $\eta$ to  $(\omega,u(0)),$  
where  $\eta(t)=(\omega(t), u(t)), $  $t\in [0,\infty),$
is  bijective.
\\
2) The transformation $\widehat\Omega_{\mathcal A}=\widehat\Omega\times P\to \widehat\Omega$
which maps $\hat\eta$ to  $(\hat\omega,u(0)),$  
where  $\hat\eta(t)=(\hat\omega(t), u(t)), $  $t\in [0,\infty),$
is  bijective.
\end{lemma}
\begin{proof} The first part is Lemma \ref{lem_identifications_spaces} in the special case when $k=1.$
 
  Part 2)  can be proved  in exactly the same way as  Part 1). 
\end{proof}
Using Lemma \ref{lem_identifications_spaces_dim1} we construct  the {\it canonical cocycle}\index{cocycle!canonical $\thicksim$}
associated  to $\mathcal A$ which is  a cocycle of rank $1$ on $(X_{1,\mathcal A},  ,\Lc_{1,\mathcal A})  $ as follows.
\nomenclature[a9b]{$\mathcal C_{\mathcal A}$}
{canonical cocycle associated to  a  cocycle $\mathcal A$}
For  $(\omega,u)\in \widehat\Omega\times  P $    and $t\in \R^+,$   let
\begin{equation}\label{eq_canonical_cocycle}
\mathcal C_{\mathcal A}((\omega, u), t):=\| \mathcal A(\omega,t)u \|,\qquad u\in P,
\end{equation}
where the right hand side is  given  by
\begin{equation}\label{eq_cocycle_on_P}
\| \mathcal A(\omega,t)u \|:={\| \mathcal A(\omega,t)\tilde u \|\over \|\tilde  u\|},\qquad \tilde u\in\R^d\setminus \{0\},\ u=[\tilde u],
\end{equation}
with $[\cdot]:\ \R^d\setminus\{0\}\to P$ the canonical projection.
Since   $\mathcal A$ is  a  cocycle,  the above definition implies  that
$\log \mathcal C_{\mathcal A}$ is     an additive  cocycle\index{cocycle!additive $\thicksim$}, that is, 
        $$\log  \mathcal C_{\mathcal A}( (\omega,u),n+m)=\log  \mathcal C_{\mathcal A}(T^n(\omega,u),m)+\log  \mathcal C_{\mathcal A}((\omega,u),n),\ \  (\omega,u)\in \widehat\Omega_{\mathcal A },\ n\in \Z.$$
        
Given a  point $x\in X,$ let  $\pi:\ \widetilde  L\to L=L_x$ be the universal cover of the leaf $L_x$ and let $\tilde x\in\widetilde L$ be  such that $\pi(\tilde x)=x.$    
As in (\ref{eq_cocycle_covering_manifold}) we  construct a  cocycle $\widetilde{\mathcal A}$ on the leaf $\widetilde L$ as  follows:
\begin{equation}\label{eq_cocycle_covering_manifold_new}
\widetilde{\mathcal A}(\tilde  \omega,t):=\mathcal A(\pi (\tilde  \omega),t),\qquad  t\in\R,\ \tilde  \omega\in \Omega( \widetilde L  ).
\end{equation}
Given  an element $u\in\P(\R^d),$
  the {\it specialization}
\index{specialization}  of $\mathcal A$ at $(\widetilde  L,\tilde x;u)$ is     the  function  $f=f_{u,\tilde x}:\  \widetilde L\to \R$ defined by
\begin{equation}\label{eq_function_f_stepII}
f_{u,\tilde x}(\tilde y):= \log \| \widetilde{\mathcal A}(\tilde\omega,1)u \|,\qquad  \tilde y\in \widetilde L,
\end{equation}
where  $\tilde\omega\in \widetilde\Omega_{\tilde x}$ is any path  such that  $\tilde\omega(1)=\tilde y.$ 
This   definition is  well-defined   because of the homotopy law for $\mathcal A$  and of the  simple connectivity of $\widetilde L.$
Using  (\ref{eq_cocycle_covering_manifold_new}) and the identity law for $\mathcal A,$ we  get that
\begin{equation}\label{eq_normalization}
f_{u,\tilde x}(\tilde x)=0.
\end{equation}
Let $\tilde z\in \widetilde L.$ Let  $v\in P$ such that  $(\tilde x,u)\overset{\widetilde{\mathcal A}}{\sim}(\tilde z,v),$ i.e., $v= \widetilde {\mathcal A}(\tilde\omega,1)u,$ where  $\tilde\omega\in \widetilde\Omega_{\tilde x}$ is any path  such that   $\tilde\omega(1)=\tilde z.$
Let $\tilde\eta\in \widetilde{\Omega}_{\tilde z}$ be  such that  $\tilde \eta(1)=\tilde y.$
We concatenate $\tilde\omega|_{[0,1]}$  and  $\tilde\eta$ in order  to obtain   a path
$$\tilde\xi(t):=
\begin{cases}
\tilde\omega(2t),  & 0\leq t\leq 1/2;\\
\tilde \eta(2t-1), & t\geq 1/2.
\end{cases}
$$
Note that $\tilde\xi\in\widetilde \Omega_{\tilde x}$ and $  \tilde  \xi(1) =\tilde y.$
Therefore, using the   multiplicative law and  homotopy law  for $\mathcal A,$ we see that
\begin{equation}\label{eq_specialization_comparison}
\begin{split}
 f_{v,\tilde z}(\tilde y)&= \log{\| \widetilde {\mathcal A}(\tilde \eta,1) v\| }
= \log{\|  \widetilde {\mathcal A}(\tilde\xi,1) u\| } - \log{\|  \widetilde {\mathcal A}(\tilde\omega,1) u\| }\\
&=
f_{u,\tilde x}(\tilde y) -f_{u,\tilde x}(\tilde z).
\end{split}
\end{equation}
This, combined  with (\ref{eq_normalization}), implies that
\begin{equation}\label{eq_conversion_opposite}
f_{u,\tilde x}(\tilde z)=- f_{v,\tilde z}(\tilde x) .
\end{equation}
Using    the  homotopy law for $\widetilde {\mathcal A}$  and  using  the  simple connectivity of $\widetilde L,$ we see that  
$$
f_{u,\tilde x}(\pi_t(\tilde \omega))= \log {\|\widetilde{ \mathcal A}(\tilde\omega,t) u\|  },\qquad  \tilde \omega\in \Omega_{\tilde x}(\widetilde L),\ t\in\R^+.
$$  
On the  other hand, recall  from Proposition C.3.8 in  \cite{CandelConlon2} the  following   identity 
\begin{equation}\label{eq_expectation_vs_diffusion}
\Et_{\tilde x}[f\circ\pi_t(\omega) ]=D_tf(\tilde x).
\end{equation} 
Consequently, we obtain  the following  conversion rule:
  \begin{equation}\label{eq_E_x_log_A_t} 
\Et_x[\log {\|\mathcal A(\cdot,t)u\|    }  ]=\Et_{\tilde x} [\log {\|\widetilde{\mathcal A}(\cdot,t)u\| }  ]=  (D_t f_{u,\tilde x})(\tilde x)=
 (D_t f_{v,\tilde z})(\tilde x)- f_{v,\tilde z}(\tilde x),
\end{equation}
where the  first   equality holds 
by (\ref{eq_cocycle_covering_manifold_new}) and an application of Proposition \ref{prop_heat_difusions_between_L_and_its_universal_covering},
the  second  equality  holds by  (\ref{eq_expectation_vs_diffusion}), and the last one  follows from a combination of
  (\ref{eq_specialization_comparison})  and  (\ref{eq_conversion_opposite}).

 Now  we  compare   the specializations (\ref{eq_function_f_stepII})  with the function 
$f_{u,x}$ constructed  in  
(\ref{eq_function_f}), where $x\in X$ and $u\in \R^d\setminus\{0\}.$ Recall that  
 we fix  an arbitrary  point $x\in X$ and  an arbitrary point $\tilde x\in\pi^{-1}(x)\subset \widetilde L,$  where   $\pi:\ \widetilde  L\to L=L_x$ is the universal cover of the leaf $L_x.$ 
 Let  $y$ be  an arbitrary point in  a     simply connected, connected open neighborhood $K$ of $x$ in $L.$ 
On $K$ a branch of $\pi^{-1}$  such that $\pi^{-1}(x)=\tilde x$ is  well-defined.    
Setting $\tilde y:=\pi^{-1}(y)$  for $y\in K,$ 
  we see that
the function 
$f_{u,x}$ constructed  in  
(\ref{eq_function_f}) satisfies   
\begin{equation}\label{eq_identity_f_f}
 f_{u,x}(y) = f_{[u],\tilde x}(\tilde y),\qquad  y\in K ,
 \end{equation}
        where $[\cdot]:\ \R^d\setminus \{0\}\to \P(\R^d)$  denotes,  as  usual,  the canonical projection.

\section{$\mathcal A$-weakly harmonic measures and  splitting invariant bundles}
 \label{subsection_weakly_harmonic_measures_and_splitting}
 
 We recall from Walters \cite{Walters}\index{Walters}  some  results  about  dual spaces.
Let $(X,\Bc(X) ,\mu)$  be  a  probability Borel space.
Let $E$ be  a  separable  Banach space\index{space!Banach $\thicksim$}\index{Banach!$\thicksim$ space} with  dual space\index{space!dual $\thicksim$} $E^\ast.$
Let $L^1_\mu(E)$  be  the space of all $\mu$-measurable  functions\index{function!measurable $\thicksim$ w.r.t. a measure}\index{measurable!$\thicksim$ function w.r.t. a measure}
$f:\ X\to E$  ($x\mapsto f_x$)   such that  $\|f\|:= \int_X  \|f_x\| d\mu(x)<\infty.$
\nomenclature[f1a]{$ L^1_\mu(\Cc(P,\R))$}{$:=L^1_\mu(E),$ where $E:=\Cc(P,\R)$}
This  is  a  Banach space  with the norm $f\mapsto\|f\|,$  where two functions  $f$ and $g$
are identified  if  $f=g$ for $\mu$-almost everywhere.
Let $L_\mu^\infty(E^\ast,E)$  be  the space of all maps $f:\  X\to E^\ast$  ($x\mapsto f_x$) 
for which  the  function  $X\ni x\mapsto f_x(v)$ is  bounded  and  measurable  for each $v\in E,$
where  two such functions $f,$ $g$  are identified  if $X\ni x\mapsto f_x(v)$
and $X\ni x\mapsto g_x(v)$ are equal $\mu$-almost everywhere for every $v\in E.$
This  is  a Banach space   with the norm  
$$\|f \|_\infty :=\esup_{x\in X} \| f_x\|=  \inf_{Y\in\Bc(X):\ \mu(Y)=1}\sup_{x\in Y}\| f_x\|,  $$
\nomenclature[a91a]{$\esup$}{essential  supremum   w.r.t. a measure $\mu$}
which is finite by the principle of uniform boundedness.
Consider the  map $  \Lambda:\ L_\mu^\infty(E^\ast,E)\to (L^1_\mu(E))^*,$ given by
$$
(\Lambda\gamma)( f):=\int_X \gamma_x(f_x)d\mu(x),
$$  
where $\gamma:\ X\to E^*$ which maps $x\mapsto \gamma_x$  is in  $L_\mu^\infty(E^\ast,E),$
and   $f:\ X\to E$ which maps $x\mapsto f_x$  is in  $L^1_\mu(E).$  
By \cite{Bourbaki}  $\Lambda$ is an isomorphism of Banach spaces.
In what follows,  for a  locally compact   metric space $\Sigma,$  we denote by $\Mc(\Sigma)$ the  space of all positive  Radon measures\index{measure!Radon $\thicksim$}\index{Radon!$\thicksim$ measure} on $\Sigma$ with mass $\leq 1.$
\nomenclature[f1a]{$ \Mc(\Sigma)$}{space of all positive  Radon measures  with mass $\leq 1$ on a  locally compact   metric space $\Sigma$}

We will be  interested  in the case where  $E:= \Cc(P,\R)$ for a       compact metric space $P.$ 
 So  $\Mc(P)$ is  the closed    unit ball of $E^\ast.$ 
The set
 $L_\mu(\Mc(P))$
\nomenclature[f1b]{$L_\mu(\Mc(P))  $}{set of  all   measurable  maps $\alpha:\ X\to \Mc(P),$ where $P$ is a     compact metric space}
  of all   measurable  maps $\alpha:\ X\to \Mc(P)$ is contained in the unit ball of  $L_\mu^\infty(E^\ast,E),$
and is closed with respect to the  weak-star topology  $L_\mu^\infty(E^\ast,E).$ 
Hence, $L_\mu(\Mc(P))$
is  compact with respect to this  topology.
The    set   $L_\mu(\Mc(P))$ can be  identified with  a  subset of the  following space 
$$
\Mc_\mu(X\times P):=\left\lbrace \lambda\in\Mc(X\times P):\ \lambda\ \text{projects to $m$ on $X$}   \right\rbrace.
$$
 via  the map  $L_\mu(\Mc(P))\ni \nu\mapsto \lambda\in \Mc(X\times P),$ where for $X\ni x\mapsto f_x$ in  $L^1_\mu(\Cc(P,\R)),$ we have
 \begin{equation}  \label{eq_formula_nu_x}
\int_{X\times P} f_x(u) d\lambda(x,u) =\int_X \Big (\int_P f_x(u) d\nu_x(u)\Big)d\mu (x).
\end{equation}

In the  remaining part of the section,  let $(X,\Lc,g)$ be a Riemannian lamination  
satisfying the Standing Hypotheses, and 
   let
    $P:=\P(\R^d)$ and $\G\in\{\N,\R^+\}.$  Consider a  harmonic probability measure $\mu$ on $(X,\Lc)$  which is also ergodic.      
For each  cocycle   $\mathcal A:\ \Omega(X,\Lc)\times \G\to \GL(d,\R),$  we  consider  its  {\it   cylinder lamination} of dimension $1,$  denoted  by $(X_{\mathcal A},\Lc_{\mathcal A}),$ which is   given by
$$(X_{\mathcal A},\Lc_{\mathcal A}):=   ( X\times \Gr_1(\R^d),\Lc_{1,\mathcal A} ),$$
where  the measurable lamination on the  right hand side is given by Definition \ref{defi_cylinder_lamination}.
Recall that  $\Omega_{\mathcal A}$ is the  sample-path space  $\Omega_{1,\mathcal A}=\Omega( X_{\mathcal A},\Lc_{\mathcal A}   ).$
Using the  identification $\Gr_1(\R^d)=\P(\R^d)=P,$ we  may write $X_{\mathcal A}=X\times P.$ 
\begin{remark}\rm 
 Since  the cylinder lamination $(X_{\mathcal A},\Lc_{\mathcal A})$ is a measurable lamination,  we can speak of very
   weakly harmonic measures  in the sense of  Definition \ref{defi_Standing_Hypotheses_harmonicity} on $(X_{\mathcal A},\Lc_{\mathcal A}).$  
 \end{remark}

\begin{definition}\label{defi_weakly_harmonic_measures_which_is_ergodic}
\rm 
Let    $\mathcal A$ be a cocycle and $\mu$  a measure as  above.
A positive finite Borel measure $\nu$ on $X\times P$ is  said to be  {\it  $\mathcal A$-weakly harmonic with respect to $\mu$}\index{harmonic measure!$\mathcal A$-weakly $\thicksim$}  
   if  it satisfies  the  following   two  conditions (i)-(ii):
   \\
   (i) it belongs to  $L_\mu(\Mc(P));$
   \\
   (ii) it is very weakly  harmonic, i.e, 
$\int_{X_{\mathcal A}}  D_1 f d\nu=\int_{X_{\mathcal A}} fd\nu
$   
for all  bounded  measurable functions $f$ defined on $X\times P.$

Denote  by  $\Har_\mu (X_{\mathcal A})$ the convex closed cone  of all     $\mathcal A$-weakly harmonic positive finite  Borel measures on $X\times P$ with respect to $\mu.$
\nomenclature[f2]{$\Har_\mu (X_{\mathcal A})$}{convex closed cone  of all     $\mathcal A$-weakly harmonic positive finite  Borel measures on $X\times P$ w.r.t. a harmonic probability measure $\mu$ on a Riemannian lamination $(X,\Lc,g).$
When the cocycle $\mathcal A$ is clear from the context, this space is often denoted by $\Har_\mu (X\times P)$ with $P:=\P(\R^d)$}
Clearly, the mass of every element in $\Har_\mu (X_{\mathcal A})$ is  $\leq 1.$
When  the cocycle $\mathcal A$ and the probability  measure $\mu$ are  clear  from the context,  we often  write ``$\mathcal A$-weakly harmonic" (resp. $\Har_\mu (X\times P)$)  instead of  ``$\mathcal A$-weakly harmonic
with respect to $\mu$"   (resp. $\Har_\mu (X_{\mathcal A})$).

An element $\nu\in \Har_\mu(X\times P)$ is  said to be {\it extremal} 
if it is  an extremal point
of    this convex closed cone, that is, if $\nu=t\nu_1+(1-t)\nu_2$ for some  $0< t < 1$ and  $\nu_1,\nu_2\in \Har_\mu(X\times P),$
then $\nu_1$ and $\nu_2$ are constants  times  $\nu .$ Clearly,  if $\Har_\mu(X\times P)\not=\{0\},$ then the  set of extremal  points of  $\Har_\mu(X\times P)$ which  are also
probability measures 
is  always  nonempty.

 \end{definition}
 
 In  what follows, for  any positive  measure $\lambda$ on $P,$ let  $\|\lambda\|  $ denotes its mass.

 \begin{remark}\label{rm_continuity_to_measurability}
 Since  $X_{\mathcal A}=X\times P$ is a locally compact  metric space, we can approximate  a bounded  $\nu$-integrable  function 
 defined on $X_{\mathcal A}$
 by continuous compactly  supported  ones  in the norm $L^1(X_{\mathcal A},\nu).$
 Therefore,   condition (ii) in Definition \ref{defi_weakly_harmonic_measures_which_is_ergodic}
 is  equivalent  to the  following  (apparently weaker) condition (ii)':
$$\int_{X_{\mathcal A}}  D_1 f d\nu=\int_{X_{\mathcal A}} fd\nu,\qquad  \forall f\in \Cc_0(X_{\mathcal A}).
$$   
 \end{remark}
 
 \begin{proposition}\label{prop_extremality_implies_ergodicity}
 Let $\nu\in \Har_\mu(X\times P)$  be  an  extremal element. Then $\nu$ is  ergodic.
In particular,  if $\Har_\mu(X\times P)\not=\{0\},$ then there exists an element of  $ \Har_\mu(X\times P)$ which is an ergodic probability measure.
 \end{proposition}
\begin{proof}
Suppose  in order to reach a contradiction that $\nu$ is not   ergodic and $\nu(X\times P)=1.$ Then there 
 is a  leafwise  saturated Borel set $Y\subset X\times P$ with $0<\nu(Y)<1.$
 Consider  two probability measures  
$$\nu_1:= {1\over \nu(Y)}\nu|_Y\qquad\text{and}\qquad  \nu_2:= {1\over 1-\nu( Y)}\nu|_{(X\times P)\setminus Y}.$$ 
Clearly, $\nu=\nu(Y) \nu_1+(1-\nu(Y))\nu_2.$  Moreover, using  Definition \ref{defi_weakly_harmonic_measures_which_is_ergodic}
and  the assumption that $Y$ is leafwise saturated  Borel set, we can show easily that both $\nu_1$ and $\nu_2$ are very weakly   harmonic
and that they
 belong to 
$\Mc_\mu(X\times P).$  

We will  prove  that   both $\nu_1$ and $\nu_2$ belong to 
$L_\mu(\Mc(P)).$ Taking for granted  this  assertion, it follows that   
$\nu_1$ and $\nu_2$
belong to $ \Har_\mu(X\times P).$ Hence, $\nu$ is  not extremal, which is the desired contradiction.  

To prove the above remaining assertion, it  suffices  to show that $\|(\nu_1)_x\|=\const$  and $\|(\nu_2)_x\|=\const$ for $\mu$-almost every $x\in X$
since  these  equalities will imply that   $\|(\nu_1)_x\|=1$  and $\|(\nu_2)_x\|=1$ for $\mu$-almost every $x\in X.$ 
Observe that 
$$
\| (D_1\nu_1)_x\|=\int_{L_x} p(x,y,1)\| \nu_y\| d\Vol_{L_x}(y)=  (D_1  \|(\nu_1)_{\bullet}\|)(x) ,\qquad x\in X,
$$
where $\|(\nu_1)_{\bullet}\|$ is the function which maps $y\in X$ to $\|(\nu_1)_{y}\|.$
Since $\nu_1$ is  very weakly  harmonic, it follows that $D_1\nu_1=\nu_1.$
This, combined  with the previous   equality, implies that
$D_1  \|(\nu_1)_{\bullet}\|=   \|(\nu_1)_{\bullet}\|.$
 Applying   Theorem  \ref{lem_Akcoglu} (i) to the  last equality yields that
 $ \|(\nu_1)_{\bullet}\|=\const$  $\mu$-almost everywhere.
The same argument also   gives that $ \|(\nu_2)_{\bullet}\|=\const$  $\mu$-almost everywhere.
 Hence, the  desired assertion is  proved.
\end{proof}

\begin{lemma} \label{lem_D_t_invariant_L_mu} For every $t\geq 0,$ the operators $D_t:\ L_\mu(\Mc(P))\to L_\mu(\Mc(P))$ 
and  $D_t:\ L^1_\mu(\Cc(P,\R))\to L^1_\mu(\Cc(P,\R))$   are contractions, that is, their norms are $\leq 1.$
\end{lemma}
\begin{proof} Let $E:=\Cc(P,\R)$.
Let  $\nu\in L_\mu(\Mc(P))$ and $x\in X$ and $t\geq 0.$ 
 Then, for $\mu$-almost every $x\in X,$ we have
that
$$
\| (D_t\nu)_x\|=\int_{L_x} p(x,y,t)\| \nu_y\| d\Vol_{L_x}(y)\leq \int_{L_x} p(x,y,t)\| \nu\|_\infty d\Vol_{L_x}(y) =\| \nu\|_\infty
$$
Hence, $\| D_t\nu\|_\infty\leq \| \nu\|_\infty.$

Now  we turn to    the  second assertion.  For $\psi\in L^1_\mu(E)$ and $x\in X,$
we have that
$$
\| (D_t\psi)(x)\|_E\leq \int_{L_x} p(x,y,t)\| \psi(y)\|_E d\Vol_{L_x}(y).
$$
Integrating both sides with respect to $\mu$ over $X,$ we get that
$$
\|D_t\psi\|_{  L^1_\mu(E)  }\leq \int_X D_t (\| \psi ({\bullet})\|_E)(x)d\mu(x)= \int_X \| \psi(x)\|_Ed\mu(x)=
\|\psi\|_{  L^1_\mu(E)  }<\infty,
$$
where $\| \psi ({\bullet})\|_E$ is the  function $X\ni x\mapsto \| \psi (x)\|_E,$ and 
where the first equality holds since $\mu$ is  harmonic.
Hence, to complete 
  the  second assertion, it suffices to show that given each $t\geq 0$ and $\psi\in L^1_\mu(E),$
we have  $(D_t\psi)(x)\in E$   for $\mu$-almost every $x\in X.$ To do this
observe  from the last argument that  for $\mu$-almost every $x\in X,$
 we have that  $ ( D_t (\| \psi (\bullet)\|_E))(x)<\infty.$ 
Moreover, for  $\mu$-almost every $x\in X,$ $\psi(y,\cdot)\in E$  for $\Vol_{L_x}$-almost every $y\in L_x.$
Fix   any point $x$ possessing  the last two  properties and  write $L:=L_x.$
 We are reduced  to the  following problem:
 
 {\it Let $\psi:\   L\times P\to\R$ be a  measurable  function such that
\\
$\bullet$ $\psi( y,\cdot)$ is  continuous  on $ P$ for  $\Vol_{L}$-almost every     $ y\in L;$
\\
$\bullet$  $\int_{ L} p( x, y,t)\max_P| \psi( y,\cdot)| <\infty.$

Then the function $P\ni u \mapsto \int_{ L} p( x, y,t) \psi( y,u)$  is  continuous.}

Since the conclusion  of the problem  follows  easily from the Lebesgue dominated convergence theorem\index{Lebesgue!$\thicksim$ dominated convergence theorem}\index{theorem!Lebesgue dominated convergence $\thicksim$}, the proof is complete.
\end{proof}

From now on  we   assume the integrability condition $$\int_{\Omega(X,\Lc)} \log^\pm\| \mathcal A (\omega,1)\| d\bar\mu(\omega)<\infty.$$
 Consider  the   functions $\varphi$ and  $\varphi_n:\ X\times P\to\R$    given by
 \begin{equation}\label{eq_varphi_n}
\begin{split} 
\varphi(x,u)&:=  \int_{\Omega_x} \log{\|\mathcal A(\omega,1)u\|} dW_x(\omega),\\
\varphi_n&:={1\over n}\sum_{i=0}^{n-1}D_i\varphi. 
\end{split}
 \end{equation}
In fact, the functions $\varphi_n$  are the Lyapunov exponent functionals that we mention at the  beginning of the chapter.\index{Lyapunov!$\thicksim$ exponent functional} We will maximize/minimize them in the proof of Theorem \ref{thm_splitting_model}
 and Proposition   \ref{prop_dual_splitting_model} below.
 
 We obtain the  following   ergodic  property of the canonical cocycle $\mathcal C_{\mathcal A}$ of a cocycle $\mathcal A.$
\begin{theorem}\label{thm_ergodic_for_C}       
      Let $\nu$ be an  element of $\Har_\mu(X\times P)$ and let $\alpha_0:=\int_{X\times P} \varphi d\nu.$  
        Then   there exists 	a leafwise  constant measurable   function
        $\alpha:\  X_{\mathcal A}=X\times P\to\R$ with  the  following properties:
\\
(i)  $\lim_{n\to\infty}{1\over n} \log  \mathcal C_{\mathcal A}((\omega,u),n)=\alpha(x)$ 
 for $\nu$-almost every $(x,u)\in X\times P$ and for $W_x$-almost every $\omega\in\Omega_x,$
 or equivalently, $\lim_{n\to\infty}{1\over n} \log  \mathcal C_{\mathcal A}((\omega,u),n)=\alpha(\omega(0))$ for $\bar\nu$-almost every $(\omega,u)\in \Omega_{\mathcal A}$;
\\
(ii) $\lim_{n\to\infty}{1\over n} \log  \mathcal C_{\mathcal A}((\hat\omega,u),-n)=-\alpha(\hat\omega(0))$ 
for $\hat\nu$-almost every $(\hat\omega,u)\in \widehat\Omega_{\mathcal A};$ 
\\ 
(iii)
$$
\lim_{n\to\infty} \int_{\widehat\Omega_{\mathcal A}}   {1\over n} \log  \mathcal C_{\mathcal A}((\hat\omega,u),n) d\hat\nu(\hat\omega,u)=
\lim_{n\to\infty} \int_{\Omega_{\mathcal A}}   {1\over n} \log  \mathcal C_{\mathcal A}((\omega,u),n) d\bar\nu(\omega,u)=\alpha_0
$$
and $\lim_{n\to\infty} \int_{\widehat\Omega_{\mathcal A}}   {1\over n} \log  \mathcal C_{\mathcal A}((\hat\omega,u),-n) d\hat\nu(\hat\omega,u)=-\alpha_0.
$
\\
(iv) If, moreover, $\nu$ is  ergodic, then $\alpha=\alpha_0$ $\nu$-almost everywhere.
 \end{theorem} 
 \begin{proof} 
First  we   consider the case when $\nu$ is ergodic. 
  Consider the function $f:\  \Omega_{\mathcal A}\to\R$ given by
 $f(\omega,u):=\log \mathcal C_{\mathcal A}((\omega,u),1),$ $(\omega,u)\in  \Omega_{\mathcal A}$
   (resp. the function $f:\  \widehat\Omega_{\mathcal A}:\to\R$ given by the function $f(\hat\omega,u):=\log \mathcal C_{\mathcal A}((\hat\omega,u),1),$  $(\hat\omega,u)\in \widehat \Omega_{\mathcal A}$).
Observe that for $n\in\N$ and  $\omega\in \Omega_{\mathcal A}$
(resp.  for $n\in\Z$ and  $\hat\omega\in \widehat\Omega_{\mathcal A}$),  
 $$\log \mathcal C_{\mathcal A}(\omega,n)=\sum_{i=0}^{n-1}f(T^i(\omega)) \quad \Big (\ \text{resp.}\  \log \mathcal C_{\mathcal A}(\hat\omega,n)=\sum_{i=0}^{n-1}f(T^i(\hat\omega))\  \Big).$$
On the  other hand,  by 
(\ref{eq_canonical_cocycle})-(\ref{eq_cocycle_on_P})
and  the  integrability condition, we  infer that
 $$\int_{\Omega} \log^\pm\|   \mathcal C_{\mathcal A}(\omega,1)\| d\bar\mu(\omega)<\infty.$$
Consequently,  by  Corollary \ref{cor_Birkhoff_ergodic_thm}  
 we get  a   real number  $\alpha $ such  that 
 $$\lim_{n\to\infty}{1\over n} \log \mathcal C_{\mathcal A}((\omega,u),n)=\alpha\quad\text{and}\quad
\lim_{n\to\infty}{1\over n} \log  \mathcal C_{\mathcal A}((\hat\omega,u),n)=\alpha $$ 
 for $\bar\nu$-almost every $(\omega,u)\in \Omega_{\mathcal A}$ and  for $\hat\nu$-almost every $(\hat\omega,u)\in \widehat\Omega_{\mathcal A}.$
Moreover, by Birkhoff ergodic theorem,\index{Birkhoff!$\thicksim$ ergodic theorem}
\index{theorem!Birkhoff ergodic $\thicksim$}
$$\lim_{n\to\infty} \int_{\Omega_{\mathcal A}}   {1\over n} \log  \mathcal C_{\mathcal A}((\omega,u),n) d\bar\nu(\omega,u)=
\lim_{n\to\infty} \int_{\widehat\Omega_{\mathcal A}}   {1\over n} \log  \mathcal C_{\mathcal A}((\widehat\omega,u),n) d\hat\nu(\hat\omega,u)=\alpha.
$$
Consider the function $g:\  \widehat\Omega_{\mathcal A}:\to\R$ given by the function $g(\hat\omega,u):=\log \mathcal C_{\mathcal A}((\hat\omega,u),-1),$  $(\hat\omega,u)\in \widehat \Omega_{\mathcal A}$.
Observe that 
 $$\log \mathcal C_{\mathcal A}(\cdot,-n)=\sum_{i=0}^{n-1}g(T^{-i}(\cdot)),\qquad  n\in\N.$$
 Consequently,  by  Corollary \ref{cor_Birkhoff_ergodic_thm},  
 we get  a real number $\beta$  such  that  
$$\lim_{n\to\infty}{1\over n} \log  \mathcal C_{\mathcal A}((\hat\omega,u),-n)=-\beta$$
 for $\hat\nu$-almost every $(\hat\omega,u)\in \widehat\Omega_{\mathcal A}$ and that 
  $$\lim_{n\to\infty} \int_{\widehat\Omega_{\mathcal A}}   {1\over n} \log  \mathcal C_{\mathcal A}((\hat\omega,u),-n) d\hat\nu(\hat\omega,u)=-\beta.
 $$
On the other hand, since $\hat\nu$ is $T$-invariant, we also get that 
 \begin{eqnarray*}
 -\beta&=& \lim_{n\to\infty} \int_{\widehat\Omega_{\mathcal A}}   {1\over n} \log \mathcal C_{\mathcal A}(T^n(\hat\omega,u),-n) d\hat\nu(\hat\omega,u)\\
 &=&-\lim_{n\to\infty} \int_{\widehat\Omega_{\mathcal A}}   {1\over n} \log  \mathcal C_{\mathcal A}((\hat\omega,u),n) d\hat\nu(\hat\omega,u)\\
 &=&
-\alpha,
 \end{eqnarray*}
 where the second equality follows from  the  identity $  \mathcal C_{\mathcal A}(T^n(\omega,u),-n)  \mathcal C_{\mathcal A}((\omega,u),n)=\mathcal C_{\mathcal A}((\omega,u),0)=1.  $
 This  implies that $\alpha=\beta.$
 
 Now  we consider the general case where $\nu$ is not necessarily ergodic.
 It suffices  to apply Choquet decomposition theorem\index{Choquet!$\thicksim$ decomposition theorem}\index{theorem!Choquet decomposition $\thicksim$}  in order to     decompose   $\nu$ into  extremal measures.
 By Proposition \ref{prop_extremality_implies_ergodicity}, these  measures are ergodic. Therefore,
  applying the previous case to each  component measure of this  decomposition and  combining  the  obtained results, the theorem follows.
\end{proof}
 
 \begin{lemma}\label{eq_relations_varphi_n}
  $\varphi\in L^1_\mu(\Cc(P,\R)),$ where $\varphi$ is given by (\ref{eq_varphi_n}).
 \end{lemma}
 \begin{proof} 
Recall  from   (\ref{eq_varphi_n})  that
  $$
\varphi(x,u)= \int_{\Omega_x} \log{\|\mathcal A(\omega,1)u\|} dW_x(\omega),\qquad (x,u)\in X\times P. 
 $$
 By the integrability condition, for $\mu$-almost every $x\in X,$
$ \int_{\Omega_x} |{\log\|\mathcal A(\omega,1)\|}| dW_x(\omega)<\infty.$
Putting this together  with the continuity of each map $P\in u\mapsto  \log{\|\mathcal A(\omega,1)u\|},$
we may apply the Lebesgue dominated convergence theorem\index{Lebesgue!$\thicksim$ dominated convergence theorem}\index{theorem!Lebesgue dominated convergence $\thicksim$}. Consequently,
 $
\varphi(x,\cdot)$ is  continuous  on $P$ for  such a point $x,$ and
$$
\|\varphi\|_{L^1_\mu(\Cc(P,\R))}\leq  \int_{x\in X}  \big (   \int_{\Omega_x} {\log\|\mathcal A(\omega,1)\|} dW_x(\omega)\big)d\mu(x)<\infty,
$$
where the last inequality holds by the  integrability condition. This  completes the proof.
\end{proof}
 
  \begin{lemma} \label{lem_varphi_n}
 For every $n\geq1,$
$$
 \varphi_n(x,u) 
= {1\over n}  \int_{\Omega_x} \log{\|\mathcal A(\omega,n)u\|} dW_x(\omega),\qquad
(x,u)\in X\times P,$$
where $\varphi_n$ is given by (\ref{eq_varphi_n}).
 \end{lemma}
 \begin{proof} Fix an arbitrary    point $x_0\in X$  and an arbitrary  element $u\in P.$
 Let  $\pi:\ \widetilde  L\to L=L_{x_0}$ be the universal cover, and fix  a  point $\tilde x_0\in\widetilde L$   such that $\pi(\tilde x_0)=x_0.$    
 Let $f=f_{u,\tilde x_0}$ be 
  the  specialization  of $\mathcal A$ at $(\widetilde  L,\tilde x_0;u)$  given by (\ref{eq_function_f_stepII}).
  For every $n\geq 1$ let
 $$
 \psi_n(\tilde x) 
= {1\over n}  \int_{\widetilde\Omega_{\tilde x}} \log{\|\widetilde{\mathcal A}(\tilde\omega,n)u_{\tilde x}\|} dW_{\tilde x}(\tilde\omega),\qquad
\tilde x\in\widetilde L,$$
where  $u_{\tilde x}\in P$ is determined  by  $(\tilde x,u_{\tilde x} )\overset{\widetilde{\mathcal A}}{\sim}(\tilde x_0,u).$
We only  need to show  that  $\psi_n(\tilde x_0)={1\over n} \sum_{i=0}^{n-1}(D_i\psi_1)(\tilde x_0).$ 
To do this    recall from  (\ref{eq_E_x_log_A_t})  that
 \begin{equation*}  
\psi_n(\tilde x)=\Et_{\tilde x} \left[\log {\|\widetilde{\mathcal A}(\cdot,n)u_{\tilde x}\|}  \right]=  
 D_n f(\tilde x)- f(\tilde x).
\end{equation*}
 We deduce from the above identity that
 $$
 \psi_n(\tilde x_0)=(D_n f)(\tilde x_0)- f(\tilde x_0)={1\over n} \sum_{i=0}^{n-1}D_i ( D  f  -f)    (x_0)=  {1\over n} \sum_{i=0}^{n-1}(D_i  \psi_1)(\tilde x_0),
 $$
 as  desired.
  \end{proof}

By Theorem \ref{T:cylinder_lami_is_conti_like},   
$(X\times P,\Lc_{\mathcal A},\pr_1^*g )  $
is   a  Riemannian   continuous-like lamination, and its  covering lamination projection  is 
$$ \Pi:\ (\widetilde X\times P,\widetilde\Lc_{\mathcal A},\Pi^*(\pr_1^*g) )\to   (X\times P,\Lc_{\mathcal A},\pr_1^*g ).  $$  
 In the next lemma,  we  follow  the notation given in Appendix  \ref{section_invariance}
  and let $(D_t)_{t\in\R^+}$ be the semi-group of heat diffusions on $(X\times P,\Lc_{\mathcal A},\pr_1^*g ) . $
  
 \begin{lemma}\label{lem_existence_harmonic_measures}
 Let $(\nu_n)_{n=1}^\infty\subset  L_\mu(\Mc(P)).$\\
1) Then there is a  subsequence $ ( \nu_{n_j} )_{j=1}^\infty$  such that ${1\over n_j} \sum_{i=0}^{n_j-1}D_i\nu_{n_j}$ converges  weakly  to  a    measure 
 $\nu\in\Har_\mu(X\times P) .$ In particular, there is always a probability measure which belongs to $\Har_\mu(X\times P).$ 
 \\
 2) Suppose  in addition 
  that 
 for each $n,$ for $\mu$-almost every $x\in X,$  $(\nu_n)_x$ is   a  Dirac mass at  some  point $u_n(x)\in P.$ Then
 the above   sequence  $(n_j)_{j=1}^\infty$  satisfies
  $$\lim_{j\to\infty}  \int_X \varphi_{n_j}(x,u_{n_j}(x))d\mu(x)=\int_{X\times P} \varphi d\nu,$$
  where $\varphi_n$ and $\varphi$ are given by (\ref{eq_varphi_n}).
 \end{lemma}
 \begin{proof}
 Let  $\nu^n:=  {1\over n} \sum_{i=0}^{n-1}D_i\nu_n,$ $n\geq 1.$ By Lemma \ref{lem_D_t_invariant_L_mu},  $\nu^n\in L_\mu(\Mc(P))  .$  
 The  sequence $ ( \nu^n )_{n=1}^\infty$ has  a convergent sequence in the  weak star topology (see the  discussion  at the  beginning of this  section). Therefore, there is
 $n_j\nearrow\infty$ and  $\nu\in  L_\mu(\Mc(P)) $ such that
 $$
 \int \psi d\nu^{n_j}\to  \int \psi d\nu,\qquad  \forall\psi\in L^1_\mu(\Cc(P,\R)).   
 $$
   To prove  that $\nu$ is  $\mathcal A$-weakly harmonic, by Definition \ref{defi_weakly_harmonic_measures_which_is_ergodic}  we only need   to  show  that
 $
 \int D\psi d\nu=\int\psi d\nu.
 $
 Since  $\psi\in L^1_\mu(\Cc(P,\R))$ it follows from  Lemma \ref{lem_D_t_invariant_L_mu}   that
 $D_1\psi\in L^1_\mu(\Cc(P,\R)).$ Therefore, 	applying the  last limit
 to  both $\psi$ and $D\psi,$ we 
 need to   show that
 $$
 \int \psi d D\nu^{n_j}  - \int \psi d\nu^{n_j}\to 0 ,\qquad  \forall\psi\in L^1_\mu(\Cc(P,\R)).   
 $$
 Observe that
 \begin{eqnarray*}
 \int \psi d D\nu^{n_j}  - \int \psi d\nu^{n_j} &=&{1\over n_j} \sum_{i=1}^{n_j}\int \psi d(D_i\nu_{n_j})
- {1\over n_j} \sum_{i=0}^{n_j-1}\int \psi d(D_i\nu_{n_j})\\
&=& {1\over n_j}\int \psi d(D_{n_j}\nu_{n_j}) - {1\over n_j}\int \psi d\nu_{n_j}\\
&=& {1\over n_j}\int (D_{n_j} \psi- \psi)d\nu_{n_j}
.   
 \end{eqnarray*}
 By Lemma \ref{lem_D_t_invariant_L_mu},     the  last line tends to $0$ as $n_j\nearrow\infty.$  
 Hence, $\nu$ is $\mathcal A$-weakly harmonic, proving  Part 1).
 
 Since   we know  by  Lemma \ref{eq_relations_varphi_n}
  that  $\varphi\in L^1_\mu(\Cc(P,\R))$    it follows from Part 1) that
  $
 \int \varphi d\nu^{n_j}\to  \int \varphi d\nu. 
 $ Using  the  explicit  formula  for  $\nu^{n_j}$  and $\nu_{n_j}$ as  well as   formula (\ref{eq_varphi_n}) for  $\varphi_{n_j},$ the leaf-hand side is  equal to
 $$
 \int \varphi  d\Big({1\over n_j}  \sum_{i=0}^{n_j-1}D_i\nu_{n_j}\Big)=  \int    {1\over n_j}  \sum_{i=0}^{n_j-1}D_i\varphi d\nu_{n_j} =\int \varphi_{n_j} d\nu^{n_j}.
 $$
  This completes the proof.
  \end{proof}

 \begin{lemma}\label{lem_harmonic_measures_for_cylinder_laminations}
 Let  $\nu\in  \Har_\mu(X\times P).$
     Let $Q$ be  Borel subset of $X\times P $  such that
$\nu(Q)>0.$ 
 Let $\alpha$  and $\beta$ be two real   numbers   such that for every $(x,u)\in Q ,$  we have
 \\
 (i)  $\chi(\omega,u)=\alpha$ for $W_x$-almost every $\omega\in\Omega_x;$ 
 \\
  (ii) 
$  \chi^-_{\tilde x,u}(\tilde \omega)=-\beta$
  for every $\tilde x\in \pi^{-1}(x)$ and
 for every $\tilde\omega\in  \widehat\Fc_{\tilde x,u},$
where $ \widehat\Fc_{\tilde x,u}\subset \widehat\Omega({\widetilde L}_x)$ (depending on $\tilde x$ and $u$) is of positive measure in $\widetilde L_x.$   
 \\ Here  the  function $
  \chi(\omega, u) $  (resp. $ \chi^-_{\tilde x,u}(\tilde \omega)$) has  been  defined  in (\ref{eq_functions_chi_new})   (resp.   in 
(\ref{eq_chi-_x_v_omega_new})).
 Then $\alpha=\beta.$
 \end{lemma}
 \begin{proof} 
  First, by Proposition \ref{prop_extremality_implies_ergodicity} we may assume without loss of generality that $\nu$ is ergodic.
 Let $\gamma:=\int_{X\times P} \varphi d\nu.$
By Theorem \ref{thm_ergodic_for_C},  
 we have that  $\lim_{n\to\infty}{1\over n} \log  \mathcal C_{\mathcal A}((\omega,u),n)=\gamma$  for $\nu$-almost every $(x,u)$
and  $W_x$-almost every $\omega.$
This, combined  with assumption (i), implies that $\gamma=\alpha.$  So  it remains  to  show that $\gamma=\beta.$

By Theorem \ref{thm_ergodic_for_C}  again,  
 $\lim_{n\to\infty}{1\over n} \log  \mathcal C_{\mathcal A}((\hat\omega,u),-n)=-\gamma$  for $\hat\nu$-almost every $(\hat\omega,u)\in \widehat\Omega_{\mathcal A}.$  
This, coupled with  Part 1) of Proposition  \ref{prop_backward_set_classification} applied  to the cylinder lamination
$(X\times P,\Lc_{\mathcal A}),$ implies that  
for $\nu$-almost every $(x,u)\in X\times P,$   the set
\begin{multline*}\Fc_{x,u}:=\left\lbrace (\hat\omega,v)\in \widehat \Omega( L_x)\times P:\    (x,u)\ \text{and}\ (\hat\omega(0),v)\
\text{are on the same leaf of}\ (X_{\mathcal A},\Lc_{\mathcal A}) \right. \\
\left.\text{and}\  \lim_{n\to\infty}{1\over n} \log  \mathcal C_{\mathcal A}((\hat\omega,v),-n)=-\gamma     \right\rbrace
\end{multline*}
is  of full measure in $ L_x.$ So by Part 2) of Lemma \ref{lem_projection_null_measure},
$\widetilde\Fc_{x,u}:=\pi^{-1}\Fc_{x,u}\subset \widehat\Omega(\widetilde L_x)$ is  of full measure in $\widetilde L_x.$
Note  that for  an arbitrary $\hat\omega\in \widehat \Omega( L_x),$
$$
 \lim_{n\to\infty}{1\over n} \log  \mathcal C_{\mathcal A}((\hat\omega,v),-n)=
\chi^-_{\tilde y,v}(\tilde\omega),
$$
where $\tilde y$ is an arbitrary point in $\pi^{-1}(y)$ with $y:= \hat\omega(0)$ 
and  $\tilde\omega:=\pi^{-1}_{\tilde y} \hat\omega.$
 Consequently, for $\nu$-almost every $(x,u)\in X\times P,$   the set
\begin{multline*}
\left\lbrace \tilde\omega\in \widehat \Omega(\widetilde L_x):\ \exists \tilde x\in \pi^{-1}(x), \exists v\in P:
 (\tilde x,u)\ \text{and}\ (\tilde\omega(0),v) \right.\\
\left. \text{are on the same leaf of}\ (\widetilde X_{\widetilde {\mathcal A}},\widetilde\Lc_{\widetilde{\mathcal A}}) 
\ \text{and}\ 
 \chi^-_{\tilde \omega(0),v}(\tilde\omega)=-\gamma                 \right\rbrace
\end{multline*}
is  of full measure in $\widetilde L_x.$
On the  other hand, by assumption (ii), for  each $(x,u)\in Q$ and for  each $\tilde x\in \pi^{-1}(x),$  the set
\begin{multline*}
\left\lbrace \tilde\omega\in \widehat \Omega(\widetilde L_x):\  \chi^-_{\tilde x,u}(\tilde\omega)=-\beta                  \right\rbrace\\
=\left\lbrace \tilde\omega\in \widehat \Omega(\widetilde L_x):\  \exists v\in P:
 (\tilde x,u)\ \text{and}\ (\tilde\omega(0),v)  \right.\\
\left. \text{are on the same leaf of}\ (\widetilde X_{\widetilde {\mathcal A}},\widetilde\Lc_{\widetilde{\mathcal A}})  \ \text{and}\ 
 \chi^-_{\tilde \omega(0),v}(\tilde\omega)=-\beta                \right\rbrace
\end{multline*}
is  of positive  measure in $\widetilde L_x.$
 Recall from  Remark \ref{rem_intersection_non_empty} that the  intersection of a set of full measure and a  set of positive measure in the   leaf $\widetilde L_x$ 
  is nonempty.    Applying  this  to the  last  two subsets of  $\widehat\Omega(\widetilde L_x)$ and using 
the assumption that $\nu(Q)>0$ yields that
  for every path $\tilde\omega$ in their intersection,
$$
-\gamma=\chi^-_{\tilde \omega(0),v}(\tilde\omega)=-\beta.
$$ 
 Hence, $\gamma=\beta,$ which implies that $\alpha=\beta=\gamma.$ 
 \end{proof}

 \begin{corollary}\label{cor_harmonic_measures_for_cylinder_laminations}
 Let $Y\subset X$ be  a  Borel set of full $\mu$-measure.
 Let $\alpha$  and $\beta$ be two real   numbers   such that for every $x\in Y$ and $u\in P ,$  we have
 \\
 (i)  $\chi(\omega,u)=\alpha$ for $W_x$-almost every $\omega\in\Omega_x;$ 
 \\
  (ii) 
$  \chi^-_{\tilde x,u}(\tilde \omega)=-\beta$
  for every $\tilde x\in \pi^{-1}(x)$ and
 for every $\tilde\omega\in  \widehat\Fc_{\tilde x,u},$
where $ \widehat\Fc_{\tilde x,u}\subset \widehat\Omega({\widetilde L}_x)$ (depending on $\tilde x$ and $u$) is of positive measure in $\widetilde L_x.$   
  
 Then $\alpha=\beta.$
 \end{corollary}
 \begin{proof} 
 By Part 1) of Lemma \ref{lem_existence_harmonic_measures},
  let $\nu$ be  a probability measure on $X\times P$ which belongs to $\Har_\mu(X\times P).$ 
Set  $Q:= Y\times P.$  Since  $Y\subset X$ is  a  Borel set of full $\mu$-measure,  it follows that $\nu(Q)=1.$ 
 Consequently,   applying  Lemma  \ref{lem_harmonic_measures_for_cylinder_laminations} yields that $\alpha=\beta.$
 \end{proof}
 \begin{remark}
 \label{rem_finite_values}
 The following remark will be  very useful.
 For $\mu$-almost every $x\in X,$ and for every $v\in \R^d\setminus \{0\}$ and  for   every 
$\tilde x\in \pi^{-1}(x),$ there is 
a  set $\Fc=\Fc_{\tilde x,v}\subset\widehat\Omega(\widetilde L_x)$  of positive  measure in $\widetilde L_x$  such that  
  $\chi^-_{\tilde x,v}(\tilde\omega) =\chi^-(x,v)$ for every $\tilde\omega\in \Fc.$ 
 This  is  an immediate consequence of 
 Theorem \ref{th_Lyapunov_filtration_backward}.
  \end{remark}

Now  we arrive  at the first result  on splitting invariant bundles.

 \begin{theorem}\label{thm_splitting_model}
  Let $Y\subset X$ be  a  Borel set of full $\mu$-measure.
Assume  also that    
     $Y\ni x\mapsto  V(x)$  is a  measurable  $\mathcal A$-invariant  subbundles of $Y\times \R^d$
 with $\dim V(x)=d'<d$ for all $x\in Y.$  
 Let   $\alpha,\beta,\gamma$ be   three real numbers   with $\alpha<\beta$ such that
 \\
 1) $ \chi^-(x,u)=\gamma$ for  every $ x\in Y,$ every $ u\in \R^d\setminus\{0\};$ 
\\ 
2)  $ \chi(\omega, v)= \alpha$ for every $ x\in Y,$ every $ v\in V(x)\setminus\{0\}$ 
  and for $W_x$-almost every $\omega\in \Omega_x;$
\\
 3) $
  \chi(\omega, u)  = \beta$ for every $x\in Y,$ every $u\in \R^d\setminus V(x),$  and for $W_x$-almost every $\omega\in \Omega_x.$   
 
Here  the  function $
  \chi(\omega, u) $  (resp. $ \chi^-(x,u)$) has  been  defined  in (\ref{eq_functions_chi_new}) (resp.   in (\ref{eq_functions_chi-})).
   Then $\beta=-\gamma$ and  $V(x)=\{0\}$ for $\mu$-almost every $x\in Y.$ 
 \end{theorem}
 \begin{proof}
 We are in the position to apply Theorem  \ref{thm_splitting}.
 Define a new measurable     subbundle  $Y\ni x\mapsto  W(x)$ of $Y\times \R^{d}$  
  by  splitting 
$\R^d=V(x)\oplus W(x)$  so that $W(x)$ is  orthogonal  to $V(x)$  with respect to the  Euclidean inner product of $\R^d.$ 
By (\ref{eq_iterations}), the linear map  $\mathcal A(\omega,n)  $ induces two other linear maps
 $\mathcal C(\omega,n):\    W( \pi_0\omega    )\to  W( \pi_n\omega    )$  and  $\mathcal D(\omega,n):\  W( \pi_0\omega    )\to  V( \pi_n\omega    )$ 
 satisfying  
 \begin{equation}\label{eq_iterations_application}
  \mathcal A(\omega,n)(v\oplus w)= \big ( \mathcal A(\omega,n)v+ \mathcal D(\omega,n)w\big ) \oplus \mathcal C(\omega,n)w
  \end{equation}
for $ x\in Y,$ $v\in V(x),$ $ w\in W(x), $  $\omega\in \Omega_x(Y).$ 
 By   Theorem \ref{thm_splitting} (i), 
 $\mathcal C$ defined   on $\Omega(Y)\times \N$  satisfies the  multiplicative law  
   $$\mathcal C (\omega,m+k)=\mathcal C(T^k\omega,m)\mathcal C(\omega,k),\qquad   m,k\in\N.$$
 Moreover, 
$\mathcal C(\omega,n)$ is invertible.  Using assumption 2) and 3) and the  inequality $\alpha<\beta,$
  we may apply Theorem \ref{thm_splitting} (ii). Consequently, 
 there exists a   subset $Y'$ of $Y$ of full $\mu$-measure  with the  following properties:
 for every $ x\in Y'$ and $ w\in \P W(x) ,$      we have 
 \begin{equation}\label{eq_limit_C}
  \limsup_{n\to\infty} {1\over n} \log \|  \mathcal C(\omega,n)w\|=\beta,
 \end{equation}
 for every $\omega\in \Fc_{x,w},$
 where $\Fc_{x,w}\subset \Omega_x(Y)$ is  a set of  full $W_x$-measure.

By Theorem \ref{thm_selection_Grassmannian}, by shrinking  $Y$ a little bit, there is a bimeasurable bijection between 
the  bundle $Y\ni x\mapsto
 W(x)$ and $Y\times \R^{d-d'}$   covering the identity and which is a linear isometry\index{isometry, linear isometry} on  each fiber.
 Using this  
 and  applying Lemma \ref{lem_existence_harmonic_measures}  and then  applying   Proposition \ref{prop_extremality_implies_ergodicity}, we may find  an ergodic  $\mathcal C$-weakly harmonic probability
 measure $\lambda$ living on    the 
leafwise  saturated  set (with respect to $\mathcal C$)    $\{(x,\P W(x)):\ x\in X\}.$   
Using  $\lambda$ and   (\ref{eq_limit_C})  and applying Theorem \ref{thm_ergodic_for_C}  yields that  
 \begin{equation}\label{eq_limit_C_beta}
 \begin{split}
  & \lim_{n\to\infty} {1\over n} \int_X  \Big ( \int_{\Omega_x} \big (\int_{u\in \P W(x)}   \log{ \|  \mathcal C(\omega,n)u\|} d\lambda_x(u)\big) dW_x(\omega) \Big) d\mu(x)\\
&=  \lim_{n\to\infty} \int   {1\over n} \log \|  \mathcal C(\omega,n)u\|d\overline\lambda(\omega,u)= \beta.
\end{split}  \end{equation}
 On the  other hand,
let
$$
M_n(x):= \sup_{u\in P}\varphi_n(x,u),
$$
where $\varphi_n$ is the  Lyapunov exponent functional given in (\ref{eq_varphi_n}),\index{Lyapunov!$\thicksim$ exponent functional} and $P$ stands, as usual, for $\P(\R^d).$ 
Set $$
\Delta_n:=\left\lbrace (x,u)\in X\times P:\ \varphi_n(x,u)=M_n(x)\in \Bc(X)\times\Bc(P)\right\rbrace.
$$
We have  $\pr:\ \Delta_n\to X,$ where $\pr:\ X\times P\to X$ is  the natural projection.
Since  for each $x\in X,$  $\{ u\in P:\ (x,u)\in\Delta_n\}$ is  closed,  we can choose by   Theorem \ref{thm_measurable_selection}  a measurable 
 map  $u_n:\ X\to P$ such that $(x,u_n(x))\in \Delta_n$ for $\mu$-almost every $x\in X.$
 We may apply Lemma  \ref{lem_existence_harmonic_measures}  to the  sequence $(u_n)_{n=1}^\infty .$
 Next, we  apply  Proposition \ref{prop_extremality_implies_ergodicity}.  Consequently, we obtain an  ergodic $\mathcal A$-weakly harmonic measure $\nu$ on $X\times P$  and a real  number 
$$\beta':=\int_{X\times P}  \varphi d\nu$$ such that
  $$
   \lim_{n\to\infty} {1\over n} \int_{\Omega\times P}\log { \|  \mathcal C(\omega,n)u\|} d\bar\nu(\omega,u)=\beta'
  $$
 and  that  $
   \lim_{n\to\infty} {1\over n}  \log \|  \mathcal C(\omega,n)u)\|=\beta'$ for $\bar\nu$-almost every $(w,u)\in\Omega\times P .$
 Note that  by 2) and  3) and  the  assumption $\alpha<\beta$ and applying Theorem \ref{thm_ergodic_for_C}  to  $\nu,$
 we have  \begin{equation}\label{eq_beta_beta}
\beta'\leq \beta.
\end{equation} 
  Recall from  (\ref{eq_iterations_application}) that  
$\| \mathcal A(\omega,n) w\|\geq \| \mathcal C(\omega,n)w\|$ for all $w\in W(x).$
Hence, $ \log{\|\mathcal A(\omega,i)u\|\over \|u\|}\geq \log{\|\mathcal C(\omega,i)u\|\over \|u\|}$ for $u\in W(x).$
Consequently, we  deduce from  Lemma \ref{lem_varphi_n} that
\begin{equation}\label{eq_varphi_n_choice}
 \int_X \varphi_n(x,u_n(x))d\mu(x) 
= {1\over n} \int_X\Big (  \int_{\Omega_x} \log{\|\mathcal A(\omega,n)u_n(x)\|} dW_x(\omega)\Big)d\mu(x).
\end{equation}
By the choice of $u_n$ and the fact that 
each $\lambda_x$ is a probability measure on $\P(W(x))$ for $\mu$-almost every $x\in X,$ the  right hand side is  greater than
\begin{eqnarray*}
&\quad & {1\over n} \int_X\Big ( \int_{v\in \P(W(x))} \big ( \int_{\Omega_x}  \log{\|\mathcal A(\omega,n)v\|}dW_x(\omega) \big) d\lambda_x(v)\Big)d\mu(x)\\
&\geq & {1\over n} \int_X\Big ( \int_{v\in \P(W(x))} \big (  \int_{\Omega_x}  \log{\|\mathcal C(\omega,n)v\|}dW_x(\omega)\big)  d\lambda_x(v)\Big)d\mu(x)\\
&=& 
  {1\over n}  \int_X  \Big ( \int_{\Omega_x} \big (\int_{v\in \P( W(x))}   \log {\|  \mathcal C(\omega,n)v\|} d\lambda_x(v)\big) dW_x(\omega) \Big) d\mu(x) ,
  \end{eqnarray*}
    By (\ref{eq_limit_C_beta}) the limit when $n\to\infty$ of   the last  expression is  equal to $\beta.$
  This, combined  with (\ref{eq_varphi_n_choice}) and Part 2) of  Lemma \ref{lem_existence_harmonic_measures}, implies that
  $\beta'=\int_{X\times P} \phi d\nu\geq \beta.$ Putting this  together with (\ref{eq_beta_beta})
 we get that  $\beta'=\beta.$
  So   we have shown that 
  \begin{equation}\label{eq_convergence_beta}
   \lim_{n\to\infty} {1\over n}  \log \|  \mathcal A(\omega,n)u\|= \beta
   \end{equation}
   for $\bar\nu$-almost every $(w,u)\in\Omega\times P .$
  Let $Q:=\{  (x,u):\ x\in X,\ u\in \P(\R^d \setminus V(x))\}.$
This  is  a  leafwise  saturated Borel set.
Since $\nu$ is  ergodic, $\nu(Q)$ is either $0$ or $1.$ 
  
 \noindent{\bf Case  $\nu(Q)=0:$}  then  $\nu((X\times P)\setminus  Q)= 1.$ Hence,  for $\mu$-almost every $x,$  there  exists $u\in   V(x),$
such that $
   \lim_{n\to\infty} {1\over n}  \log \|  \mathcal A(\omega,n)u\|= \beta$ for $W_x$-almost every $\omega.$
This, combined  with     assumption 2)  and the  assumption that $\alpha <\beta,$ implies that    $V(x)=\{0\}$ for $\mu$-almost every $x\in X.$ 
Consequently, we  deduce  from Lemma \ref{lem_harmonic_measures_for_cylinder_laminations} and  1) and 3) and  Remark
 \ref{rem_finite_values} that $\beta=-\gamma,$
 as desired.

\noindent{\bf Case    $\nu(Q)=1:$} By (\ref{eq_convergence_beta}),   for $\nu$-almost every $(x,u)\in Q,$   we  have  that $$
   \lim_{n\to\infty} {1\over n}  \log \|  \mathcal A(\omega,n)u\|=\beta$$ for $W_x$-almost every $\omega.$
  We are in the position to apply  Lemma \ref{lem_harmonic_measures_for_cylinder_laminations}.
  Consequently, we get that $\beta=-\gamma.$
  
  It remains to show  that $V(x)=\{0\}$ for $\mu$-almost every $x\in X.$ Suppose  the contrary.
  So  $\dim V(x)=d'\geq 1$ for $\mu$-almost every $x\in X.$
  Restricting  $\mathcal A(\omega,\cdot) $ on $V(x)$ for every $w\in\Omega_x,$ and
  applying    Lemma \ref{lem_harmonic_measures_for_cylinder_laminations}, we get that $\alpha=-\gamma.$
  Hence, $\alpha=\beta (=-\gamma),$ which  contradicts the  hypothesis that $\alpha<\beta.$
 \end{proof}

\begin{proposition}\label{prop_dual_splitting_model}
  Let $Y\subset X$ be  a  Borel set of full $\mu$-measure and 
     $Y\ni x\mapsto  V(x)$ a measurable  $\mathcal A$-invariant  subbundle of $Y\times \R^d$ 
 such that $\dim V(x)=d'<d$ for all $x\in Y.$ 
 Let   $\alpha,\beta,\gamma$ be three real numbers   with $\alpha<\beta$ such that
 \\
 1)  $ \chi(\omega,u)=\gamma$ for  every $ x\in Y,$ every $ u\in \R^d\setminus\{0\},$ and for $W_x$-almost every $\omega\in\Omega;$
\\ 
2)  $ \chi^-(x, v)= \alpha$ for every $ x\in Y,$ every $ v\in V(x)\setminus\{0\};$ 
   \\
 3) $  \chi^-(x, u)= \beta$ for every $ x\in Y,$ every $ u\in  \R^d\setminus V(x).$ 
 
 Here  the  function $
  \chi(x, u) $  (resp. $ \chi^-(x,u)$) has  been  defined  in (\ref{eq_functions_chi}) (resp.   in (\ref{eq_functions_chi-})).
   Then $\beta=-\gamma$ and  $V(x)=\{0\}$ for $\mu$-almost every $x\in Y.$
    \end{proposition}

 Observe  that if $V(x)=\{0\}$ for $\mu$-almost every $x\in Y,$ then 
 we  deduce  from
Corollary \ref{cor_harmonic_measures_for_cylinder_laminations}
 and  1) and 3)  that $\beta=-\gamma,$
 as desired.  Therefore,  we only consider the  case where  
 $d'\geq 1$   in the  sequel. 
 
 Prior to the proof of Proposition \ref{prop_dual_splitting_model}  we  need   to introduce some  preparation.
  Define a     subbundle  $Y\ni x\mapsto  W(x)$ of $Y\times \R^d$  
   and the multiplicative  
 $\mathcal C(\omega,n):\    W( \pi_0\omega    )\to  W( \pi_n\omega    )$   by
 (\ref{eq_iterations_application}).
   Next, using  formula (\ref{eq_formula_extended_cocycle}) we extend 
 $\mathcal A$ and $\mathcal C$ to $\widehat\Omega(Y)\times \Z$ such that its extension still satisfies the  multiplicative law. 
By assumption 3) and  Remark
 \ref{rem_finite_values},  for every $x\in Y,$ every $\tilde x\in\pi^{-1}(x),$
 every $w\in W(x),$ there is  a  set $\widehat\Gc_{\tilde x,w}\subset   \widehat\Omega(\widetilde L_x)$     of positive  measure in $\widetilde L_x$ such that  
 $
  \chi^-_{\tilde x, w}(\tilde \omega)  = \beta$ 
   for  every $\tilde\omega\in \widehat\Gc_{\tilde x,w}.$
 This, combined with  assumption 2) and the inequality $\alpha<\beta,$  allows us to apply   Theorem \ref{thm_splitting_backward} (ii). Consequently, 
 there exists a   subset $Y'$ of $Y$ of full $\mu$-measure  with the  following properties:
 for every $ x\in Y'$ and every $\tilde x\in \pi^{-1}(x)$ and every $ w\in W(x)\setminus \{0\},$    we have 
 \begin{equation}\label{eq_limit_C_dual}
\limsup_{n\to\infty} {1\over n} \log \|  \mathcal C(\tilde\omega,-n)w\|=\beta,
 \end{equation}
 for every $\tilde\omega\in \Fc_{\tilde x,w},$
where $\Fc_{\tilde x,w}\subset \widehat\Omega({\widetilde L}_x)$ is of positive measure in $\widetilde L_x.$

On the other hand, by Theorem \ref{thm_selection_Grassmannian}, by shrinking  $Y$ a little bit, there is a bimeasurable bijection between the   bundle $Y\ni x\mapsto
 W(x)$ and $Y\times \R^{d-d'},$  which   covers the identity and which is a linear isometry\index{isometry, linear isometry}  on  each fiber.
 Using this and  applying Lemma \ref{lem_existence_harmonic_measures},
 we may find  an ergodic probability
 measure $\lambda$  which is  also  $\mathcal C$-weakly harmonic  living on    the 
leafwise  saturated  set (with respect to $\mathcal C$)    $\{(x,\P W(x)):\ x\in X\}.$

 For each $n\geq 1$ consider  the   function   $m_n:\ X\to\R$    given by 
 \begin{equation}\label{eq_min_A}
 m_n(x):=\min_{u\in \P V(x)}\varphi_n(x,u) 
 =  {1\over n} \min_{v\in V(x)\setminus \{0\}}\int_{\Omega_x} \log{\|\mathcal A(\omega,n)v \| \over \| v \|} dW_x(\omega),\qquad x\in X,
\end{equation}
 where $\varphi_n$ is the  Lyapunov exponent functional given in (\ref{eq_varphi_n}),\index{Lyapunov!$\thicksim$ exponent functional}
 and the last equality follows  from  Lemma 
\ref{lem_varphi_n}.
  \begin{lemma}\label{eq_subsequence_gamma}
We keep  the  hypotheses of Proposition \ref{prop_dual_splitting_model} and the above notation. Then
  there is  a  subsequence $(n_j)_{j=1}^\infty$ such that
  $$
   \lim_{j\to\infty}{1\over n_j}\int_X m_{n_j}(x)d\mu(x)=\gamma.
   $$
  \end{lemma}
  \begin{proof}
  Set $$
\Delta_n:=\left\lbrace (x,u)\in X\times \P V(x):\ \varphi_n(x,u)=m_n(x)\right\rbrace\in \Bc(X)\times\Bc(P).
$$
We have  $\pr:\ \Delta_n\to X,$ where $\pr:\ X\times P\to X$ is  the natural projection.
Since  for each $x\in X,$  $\{ u\in \P V(x):\ (x,u)\in\Delta_n\}$ is  closed,  we can choose by   Theorem \ref{thm_measurable_selection}  
a measurable  map  $u_n:\ X\to  P$ such that $(x,u_n(x))\in \Delta_n$ for $\mu$-almost every $x\in X.$
 We may apply Part 1) of Lemma  \ref{lem_existence_harmonic_measures}  to the  sequence $(u_n)_{n=1}^\infty .$
 Next, we  apply   Proposition  \ref{prop_extremality_implies_ergodicity}.
  Consequently, we obtain  an ergodic  probability  measure $\tau\in\Har_\mu(X\times P)$ and  a  sequence $(n_j)\nearrow\infty$ as $j\nearrow \infty$  such that
  \begin{equation*}
   \lim_{j\to\infty} \int_X  \varphi_{n_j}(x,u_{n_j}(x))d\mu(x) = \int_{X\times P}\varphi d\tau.
\end{equation*}
This, coupled  with the choice  of $u_{n_j}$ and  the definition of $\Delta_n$ and formula (\ref{eq_min_A}), implies that
$$
  \lim_{j\to\infty}{1\over n_j}\int_X m_{n_j}(x)d\mu(x)= \int_{X\times P}\varphi d\tau.
  $$
  On the other hand, by Theorem \ref{thm_selection_Grassmannian}, by shrinking  $Y$ a little bit, there is a bimeasurable bijection between the   bundle $Y\ni x\mapsto
 V(x)$ and $Y\times \R^{d'}$   covering the identity and which is a linear isometry\index{isometry, linear isometry}  on  each fiber.
 Using this as well as 
  assumption 1), we  may   apply  Theorem \ref{thm_ergodic_for_C}.  Consequently, we  get that
$$
 \int_{X\times P}\varphi d\tau=\gamma.
$$ 
This, combined with the previous  equality, completes the proof.
  \end{proof}

 For  $(x,u)\in X\times P$ such that $u\not\in \P V(x),$ let $\pr_xu:= [\pr_x\tilde u]\in \P W(x),$   where $\tilde u\in \R^d\setminus\{0\}$  such that  $[\tilde u]=u$ and
 $\pr_x\tilde u$ is the component of $\tilde u$   in $W(x)$  from  the direct sum decomposition  $\tilde u\in V(x)\oplus W(x).$   
For  each $n\geq 1$ consider  the   function   $\psi_n:\ X\times P\to\R$    given by
 \begin{equation}\label{eq_psi_n}
\psi_n(x,u):=
\begin{cases} 
 \min\left\lbrace  1/n\cdot \int_{\Omega_x}   \log{\|\mathcal C(\omega,n)\pr_xu\|} dW_x(\omega), m_n(x)  \right\rbrace,
&  u\not\in \P V(x),\\
m_n(x)  , &u\in  \P V(x),
\end{cases}
 \end{equation}
 where  the function $m_n$ is  given in (\ref{eq_min_A}).
 The functions  $\psi_n$ are, in general, only upper semi-continuous with respect to the  variable $u\in P.$  
 For each $n\geq 1$  
we also  consider  the  following continuous   regularizations $(\psi_{n,N})_{N=1}^\infty$ of $\psi_n,$ which are defined by 
 \begin{equation}\label{eq_psi_n_N}
\psi_{n,N}(x,u):=
 \begin{cases} 
 \min\left\lbrace   1/n\cdot \int_{\Omega_x}  \log{\|\mathcal C(\omega,n)\pr_xu\|} dW_x(\omega), m_n(x) -1/N \right\rbrace,
&  u\not\in \P V(x),\\
m_n(x) -1/N , &u\in  \P V(x). 
\end{cases}
 \end{equation}  
 The properties  of the  functions  $\psi_n$ and $\psi_{n,n}$ are collected in the  following  result.
  \begin{lemma}\label{eq_relations_psi_n} For each $n\geq 1$  $\psi_{n,N}\nearrow \psi_n$ as $N\nearrow\infty.$ 
  Moreover,  
  $\psi_{n,N}\in L^1_\mu(\Cc(P,\R)).$
  \end{lemma}
 \begin{proof} The limit $\psi_{n,N}\nearrow \psi_n$ as $N\nearrow\infty$ follows from the  definition of  $\psi_{n,N}$ and  $\psi_n$
given in (\ref{eq_psi_n})--(\ref{eq_psi_n_N}). Next, we  show that  $\psi_{n,N}(x,\cdot)\in \Cc(P,\R)$ for each $x\in X.$
 Indeed, by (\ref{eq_min_A}),  we  see that $\psi_{n,N}(x,u)=  m_n(x) -1/N$ when $u$ varies  in a small  neighborhood
of $\P V(x)$ in $P.$ On the other hand, for each $x\in X$ fixed, the functions   $\psi_{n,N}(x,\cdot)$  are continuous    outside $\P V(x). $  Consequently,   $\psi_{n,N}(x,\cdot)\in \Cc(P,\R)$ for each $x\in X.$
Using this and arguing  as in the proof of 
 Lemma \ref{eq_relations_varphi_n} we can show that $\psi_{n,N}\in L^1_\mu(\Cc(P,\R)).$
 \end{proof}

 \begin{lemma}\label{lem_relation_psi_n_N_and_psi_1_N} 
We keep  the  hypotheses of Proposition \ref{prop_dual_splitting_model} and the above notation.
Then for every $n,N\geq 1,$ we have that
$$\psi_{n,N}\geq  {1\over n}  \sum_{i=0}^nD_i\psi_{1,N}.$$
\end{lemma}
\begin{proof}
 Fix  arbitrary $N,n_0\geq 1$ and an arbitrary    point $x_0\in X$   and an arbitrary  element $u\in P$
and   set $L:=L_{x_0}.$
 Let  $\pi:\ \widetilde  L\to L $ be a universal cover, and fix  a  point $\tilde x_0\in\widetilde L$   such that $\pi(\tilde x_0)=x_0.$ There  are two  cases  to consider.

\noindent {\bf Case  $ \psi_{n_0,N}(x_0,u) =   1/n_0\cdot \int_{\Omega_{x_0}}  \log{\|\mathcal C(\omega,n_0)\pr_xu\|} dW_{x_0}(\omega).$}   
 
 In this case $u\not\in \P V(x_0).$ 
For every $n\geq 1$ and  $\tilde x\in \widetilde L,$ let
  $$
\theta_n(\tilde x):=  
  {1\over n}  \int_{\widetilde\Omega_{\tilde x}}   \log{\|\mathcal C(\pi\circ\tilde\omega,n)\pr_{\tilde x}(  u_{\tilde x}) \|}  dW_{\tilde x}(\omega),
 $$
where  $u_{\tilde x}\in P$ is determined  by  $(\tilde x,u_{\tilde x} )\overset{\widetilde{\mathcal A}}{\sim}(\tilde x_0,u)$ (see Definition \ref{D:equivlent_relation_1}).

Since   we  have  already  identified the bundle $Y\ni x\mapsto W(x)$
as   $Y\times \R^{d-d'}$ and  under this  identification $\mathcal C$ is a  cocycle, applying  Lemma 
\ref{lem_varphi_n}  to  $\mathcal C$ yields that
\begin{equation}\label{eq_psi_min}
 \theta_{n_0}(\tilde x_0)= {1\over n_0} \sum_{i=0}^{n_0-1}(\widetilde D_i\theta_1)( \tilde x_0),
\end{equation}
where $\widetilde D_i$ are the diffusion operators on $\widetilde L.$
On the other hand,  by our assumption   $\theta_{n_0}(\tilde x_0)=\psi_{n_0,N}(x_0,u)  ,$ and
by  formula   (\ref{eq_psi_n})  $$\theta_1( \tilde x)\geq  \psi_{1,N}(x,u_{\tilde x})$$
for every $\tilde x\in \widetilde L$ and $x=\pi(\tilde x).$
  This, combined with (\ref{eq_psi_min}), gives the lemma in this first case.
  
 \noindent {\bf Case  $ \psi_{n_0,N}(x_0,u) =    m_{n_0}(x_0)-1/N.$}

 In this  case  let $v\in \P V(x_0)$  be such that
 $\varphi_{n_0}(x_0,v)=m_{n_0}(x_0).$  For every $n\geq 1$ and  $\tilde x\in \widetilde L,$ let
  $$
\theta_n(\tilde x):=  
  {1\over n} \int_{\widetilde\Omega_{\tilde x}}    \log{\|\mathcal A(\pi\circ\tilde\omega,n)(  v_{\tilde x}) \|}  dW_{\tilde x}(\omega),
 $$
where  $v_{\tilde x}\in P$ is determined  by  $(\tilde x,v_{\tilde x} )\overset{\widetilde{\mathcal A}}{\sim}(\tilde x_0,v)$ (see Definition \ref{D:equivlent_relation_1}).

Since in  Lemma \ref{eq_subsequence_gamma} we  have  already  identified the bundle $Y\ni x\mapsto V(x)$
as   $Y\times \R^{d'}$ and  under this  identification $\mathcal A$ is a  cocycle, applying  Lemma 
\ref{lem_varphi_n}  to  $\mathcal A$ yields that
\begin{equation}\label{eq_psi_min_new}
 \theta_n(\tilde x_0)= {1\over n_0} \sum_{i=0}^{n_0-1}(\widetilde D_i\theta_1)( \tilde x_0),
\end{equation}
where $\widetilde D_i$ are the diffusion operators on $\widetilde L.$
  On the other hand,  by our assumption   $\theta_{n_0}(\tilde x_0)= \psi_{n_0,N}(x_0,u)+1/N    ,$ and
by  formula   (\ref{eq_psi_n})  $$\theta_1( \tilde x)\geq  \psi_{1,N}(x,v_{\tilde x})+1/N$$
for every $\tilde x\in \widetilde L$ and $x=\pi(\tilde x).$
  This, combined with (\ref{eq_psi_min_new}), gives the lemma in this last case.
 \end{proof}

\noindent \noindent{\bf End of the  proof of Proposition \ref{prop_dual_splitting_model}.} 
 Applying   Theorem  \ref{thm_ergodic_for_C}  and using (\ref{eq_limit_C_dual}), we deduce  that  
  \begin{equation}\label{eq_limit_C_beta_dual}
\begin{split}
  &    \lim_{n\to\infty} {1\over n} \int_X  \Big ( \int_{\Omega_x} \big (\int_{u\in \P W(x)}   \log{ \|  \mathcal C((\omega,u),n)\|}
 d\lambda_x(u)\big) dW_x(\omega) \Big) d\mu(x)\\
&=  \lim_{n\to\infty} \int   {1\over n} \log \|  \mathcal C(\omega,n)u\|d\overline\lambda(\omega,u)=
-\lim_{n\to\infty} \int   {1\over n} \log \|  \mathcal C(\omega,-n)u\|d\hat\lambda(\omega,u)
=-\beta.
  \end{split}  \end{equation}
   Now
let
$$
\hat m_n(x):= \inf_{u\in\P W(x)}\psi_n(x,u).
$$
Set $$
\widehat\Delta_n:=\left\lbrace (x,u)\in X\times P:\ \psi_n(x,u)=\hat m_n(x)\right\rbrace\in \Bc(X)\times\Bc(P).
$$
We have  $\pr:\ \widehat\Delta_n\to X,$ where $\pr:\ X\times P\to X$ is  the natural projection.
Since  for each $x\in X,$  $\{ u\in P:\ (x,u)\in\Delta_n\}$ is  closed,  we can choose by   Theorem \ref{thm_measurable_selection}  a measurable  map  $\hat u_n:\ X\to P$ 
such that $(x,\hat u_n(x))\in \widehat\Delta_n$ for $\mu$-almost every $x\in X.$
 We may apply Part 1) of Lemma  \ref{lem_existence_harmonic_measures}  to the  sequence $(\hat u_n)_{n=1}^\infty .$
 Next, we  apply Proposition  \ref{prop_extremality_implies_ergodicity}.
  Consequently, we obtain  an ergodic  probability  measure $\nu\in\Har_\mu(X\times P)$ and  a  sequence $(n_j)\nearrow\infty$ as $j\nearrow \infty$  such that
  \begin{equation}\label{eq_limit_measure_nu_dual}
   \lim_{j\to\infty} \int_X \big({1\over n_j}\sum_{i=0}^{n_j-1} D_i   \psi \big )(x,\hat u_{n_j}(x))d\mu(x) = \int_{X\times P}\psi d\nu
\end{equation}
for every $\psi\in L^1_\mu(\Cc(P,\R)).$ The following result is needed.
 
\begin{lemma}\label{lem_upper_bound_integral_m_n}
$
 \lim_{n\to\infty}\int_X  \hat m_n(x)d\mu(x) 
\leq   -\beta  .$
\end{lemma}
\begin{proof}
Observe from (\ref{eq_psi_n}) that for  $v\in \P W(x),$
$$
\psi_n(x,v)\leq  \int_{\Omega_x}   \log{\|\mathcal C(\omega,n)v \| } dW_x(\omega).
$$
Using this  we  get that 
\begin{eqnarray*}
\int_X \hat m_n(x)d\mu(x)
&=&     \int_X \psi_n(x,\hat u_n(x))d\mu(x)   \\
&\leq& {1\over n} \int_X\Big ( \int_{v\in \P W(x)} \big ( \int_{\Omega_x}  \log{\|\mathcal C(\omega,n)v\|}dW_x(\omega) \big) d\lambda_x(v)\Big)d\mu(x)
 ,
  \end{eqnarray*}
  where the   inequality   follows from the construction of $\hat{u}_n$ and the fact that 
each $\lambda_x$ is a probability measure on $\P W(x)$ for $\mu$-almost every $x\in X.$
By (\ref{eq_limit_C_beta_dual}), the  limit of the last  integral as $n\to\infty$ is equal to $-\beta.$ 
Hence,  the lemma follows.
\end{proof}

Resuming the proof of  Proposition \ref{prop_dual_splitting_model},  we deduce from Lemma   \ref{lem_relation_psi_n_N_and_psi_1_N} that  for  an arbitrary $N\geq 1,$
 $$
 \int_X \psi_{n,N}(x,\hat u_n(x))d\mu(x) 
\geq  \int_X    {1\over n}  \sum_{i=0}^n D_i\psi_{1,N} (x,\hat u_n(x))d\mu(x) .$$
By (\ref{eq_limit_measure_nu_dual}) and Lemma  \ref{eq_relations_psi_n}, the  right hand side  tends to   $\int_{X\times P} \psi_{1,N} d\nu$ as $n\to\infty.$
On the  other hand,  by Lemma  \ref{eq_relations_psi_n}, $\psi_{n,N}\leq  \psi_n$ and  $\psi_{1,N}\nearrow\psi_1$ as $N\to\infty.$
Putting  these  estimates  together and letting $N\to\infty,$  yields that
 $$
 \lim_{n\to\infty}\int_X \psi_n(x,\hat u_n(x))d\mu(x) 
\geq  \int_{X\times P} \psi_1 d\nu  .$$
In other words,
\begin{equation}\label{eq_last_estimate_dual_prop}
 \lim_{n\to\infty}\int_X \hat m_n(x)d\mu(x) 
\geq  \int_{X\times P} \psi_1 d\nu  .
\end{equation}
 On the  other hand, using  assumption 1) and 2)  and  applying   Theorem \ref{thm_selection_Grassmannian} and by shrinking  $Y$ a little bit,
 we may apply  Lemma \ref{lem_harmonic_measures_for_cylinder_laminations}  to
the  $\mathcal A$-invariant  bundle  $Y\ni x\mapsto V(x).$ Therefore,  we get that 
\begin{equation}\label{eq_gamma_alpha}
\gamma=-\alpha.
\end{equation}
Let $Q:=\{(x,u)\in X\times P:\ u\not\in \P V(x)\}.$
Note that $Q$ is leafwise  saturated.
There are two cases  to consider.

\noindent {\bf  Case  $\nu(Q)>0.$}

   By Lemma \ref{lem_harmonic_measures_for_cylinder_laminations} and assumption  1) and 3), we get that $\beta=-\gamma.$
This, combined  with (\ref{eq_gamma_alpha}), implies that $\alpha=\beta$ which contradicts the assumption that  $\alpha<\beta.$
Hence,  this case cannot happen.

\noindent {\bf  Case  $\nu(Q)=0.$}

Lemma \ref{lem_upper_bound_integral_m_n}, combined  with  (\ref{eq_last_estimate_dual_prop}), implies that 
$\int_{X\times P} \psi_1 d\nu \leq -\beta .$
Since  $\nu$ is  supported on  $(X\times P)\setminus Q=\{(x,u):\ u\in V(x)\},$
it follows  from   the last estimate and the  formula  of  $\psi_1$ in (\ref{eq_psi_n}) that 
 \begin{equation}\label{eq_min_1}
  \int_{x\in X}  m_1(x)d\mu(x)= \int_{x\in X} \Big( \min_{v\in V(x)\setminus \{0\}}  \int_{\Omega_x} \log{\|\mathcal A(\omega,1)v \| \over \| v \|} dW_x(\omega) \Big)d\mu(x) \leq -\beta.
\end{equation}
Next,  we  scale  the cocycle  $\mathcal A,$ that is,  for each $n\geq 1$ we  consider the  cocycle
$\mathcal A_n$ given by
$$
\mathcal A_n(\omega,t):=\mathcal A(\omega,nt),\qquad  \omega\in\Omega,\ t\in\R^+.
$$
Arguing as  in the proof of (\ref{eq_min_1}) but for $\mathcal A_n$ instead of $\mathcal A,$ we  get that
\begin{equation*}\label{eq_min_1_new}
 \int_{x\in X}   m_n(x)d\mu(x)
= \int_{x\in X}\Big({1\over n}\min_{v\in V(x)\setminus \{0\}}  \int_{\Omega_x} \log{\|\mathcal A(\omega,n)v \| \over \| v \|} dW_x(\omega)\Big) d\mu(x)\leq -\beta.
\end{equation*}
By Lemma  \ref{eq_subsequence_gamma},
  there is  a  subsequence $(n_j)_{j=1}^\infty$ through which   the  limit of the left hand  side  is  equal to $\gamma.$
So $\gamma\leq -\beta.$ This, combined  with (\ref{eq_gamma_alpha}),  
 implies that  $-\alpha\leq -\beta.$
  But this contradicts the assumption $\alpha<\beta.$ Hence, the second case cannot happen.
\hfill $\square$ 

Now  we  arrive  at the  second main result of this  section. The next  theorem, together  with  Theorem \ref{thm_splitting_model},
    constitute  the  indispensable toolkit in order to obtain splitting invariant subbundles in the  next sections.

 \begin{theorem}\label{thm_dual_splitting_model}
  Let $Y\subset X$ be  a  Borel set of full $\mu$-measure and $1\leq k\leq d$  an integer.
Assume that   
     $Y\ni x\mapsto  V^{-i}(x)$  for $1\leq i\leq  k$  and $Y\ni x\mapsto  U(x)=V^0(x)$ are $(k+1)$ measurable  $\mathcal A$-invariant  subbundles of $Y\times \R^d$ 
 such that 
$$\{0\}= V^{-k}(x)\subset\cdots \subset V^{-1}(x)\subsetneq U(x),\qquad x\in Y.$$ 
 Let   $\alpha_1,\ldots,\alpha_k$ and $\gamma$ be $(k+1)$ real numbers   with $\alpha_1>\cdots >\alpha_k$ such that
 \\
 1) $ \chi(\omega,u)=\gamma$ for  every $ x\in Y,$ every $ u\in U(x)\setminus\{0\},$ and for $W_x$-almost every $\omega\in\Omega;$ 
\\ 
2)  $ \chi^-(x, v)= \alpha_i$ for every $1\leq i\leq k,$ every $ x\in Y,$ every $ v\in V^{-(i-1)}(x)\setminus V^{-i}(x).$ 
  
   Then $\alpha_1=-\gamma$ and  $V^{-k}(x)=\cdots= V^{-1}(x)=\{0\}$ for all $x\in Y.$
   \end{theorem}
\begin{proof}
Suppose without loss of generality  that  the sequence $(V^{-i}(x))_{i=0}^k$ is  strictly  decreasing in $i.$
Observe that for each $1\leq i\leq k,$ the leafwise  constant  function $Y\ni x\mapsto  \dim V^{-i}(x)$ is, in fact,  constant $\mu$-almost everywhere 
because of the ergodicity of $\mu.$ 
There are two cases to consider.
\\ {\bf Case : $k=1.$}
By Theorem \ref{thm_selection_Grassmannian}, there is a bimeasurable bijection between the   bundle $Y\ni x\mapsto
 U(x)$ and $Y\times \R^{d'}$ with $\dim U(x)=d'$   covering the identity and which is a  linear  isometry\index{isometry, linear isometry} on each fiber.
 Using this  bijection, we are  able to apply Corollary \ref{cor_harmonic_measures_for_cylinder_laminations}.
 Consequently, $\alpha_1=-\gamma,$ as  asserted.
 \\ {\bf Case : $k>1.$}  So  $\{0\}\subsetneq  V^{-(k-1)}(x)\subsetneq V^{-(k-2)}\subset U(x)$ for each $x\in Y.$
 By Theorem \ref{thm_selection_Grassmannian}, there is a bimeasurable bijection $\Lambda$ from the   bundle $Y\ni x\mapsto
 V^{-(k-2)}(x)$  onto $Y\times \R^{d''}$ with $\dim V^{-(k-2)}(x)=d''$   covering the identity and which is a linear isometry\index{isometry, linear isometry} on each fiber.
 Using this  bijection, we are in the position to apply
 Proposition \ref{prop_dual_splitting_model} to the  following  situation:
 $d$ is replaced  with $d',$  $V(x):=\Lambda(x,V^{-(k-1)}(x)).$ 
 Consequently, we  obtain that $V(x) =0,$ hence  $ V^{-(k-1)}(x)=0$ for all $x\in Y,$  which is  a  contradiction. So this case cannot happen.
\end{proof}

\section[First Main Theorem and  Ledrappier type  characterization]{First Main Theorem and  Ledrappier type  characterization of Lyapunov spectrum}\index{theorem!First Main $\thicksim$} 
\label{subsection_First_Main_Theorem}

Assume  without loss of generality that $\mu$ is ergodic. We are in the position to apply the results  obtained in Chapter
\ref{section_Lyapunov_filtration} and  Chapter \ref{section_backward_filtration} and Section  \ref{subsection_weakly_harmonic_measures_and_splitting}.  
The proof of Theorem   \ref{th_main_1} is  divided into two cases  which correspond to  the following two subsections.

\subsection{Case I:    $\G= \N t_0$ for some  $t_0>0$}
\label{subsection_First_Theorem_Case_I}
 
 Without loss of generality  we may assume that $t_0=1,$ that is, $\G=\N.$ In what follows
 we  will make full use of the  results as well as the notation   given  in:
\\
$\bullet$  Theorem \ref{th_Lyapunov_filtration_Brownian_version} and  Theorem  \ref{th_Lyapunov_filtration}
in the  forward setting;
 \\
$\bullet$ 
 Theorem \ref{thm_Oseledec_Brownian_version} and  Theorem  \ref{th_Lyapunov_filtration_backward}
   in the  backward setting; 
 \\
$\bullet$  Theorem \ref{thm_splitting_model} and  Theorem \ref{thm_dual_splitting_model}  for splitting  invariant   sub-bundles.

For example,  $\Phi\subset \Omega(X,\Lc)$  is  the
set  of full $\bar\mu$-measure   introduced  by  Theorem  \ref{th_Lyapunov_filtration_Brownian_version}.  This  case is  divided  into 4 steps.

 \noindent{\bf Step 1: }{\it Proof that  $\chi_m=\lambda_l .$  Moreover, we have that
$V_m(x)=V_l(\omega)$ for $\mu$-almost every $x\in X$ and  for $W_x$-almost every path
 $\omega\in\Phi.$}  
  
  By  Theorem  \ref{th_Lyapunov_filtration_backward},  $\lambda_l=-\chi^-_{m^-}.$ So it is  sufficient to show that $\chi_m=-\chi^-_{m^-}.$
 Recall from  Theorem  \ref{th_Lyapunov_filtration} 
 that for $\mu$-almost  every $x\in X,$
  \begin{equation}\label{eq_characterization_V_m}
  V_m(x):=\left\lbrace v\in\R^d:\  \lim_{n\to\infty} {1\over n} \log\|  \mathcal A(\omega,n)v\|= \chi_m \ \text{$W_x$-almost every}\ \omega \in \Omega(L_x) \right\rbrace.
  \end{equation}
  Moreover, by Theorem    \ref{th_Lyapunov_filtration}, $\chi_m\in\{\lambda_1,\ldots,\lambda_l\}.$
 On the other hand,   recall also from Theorem  \ref{th_Lyapunov_filtration_backward} that for $\mu$-almost  every $x\in X,$
   and for every $\tilde x\in \pi^{-1}(x),$ and for every $v\in 
  \R^d\setminus  V^-_{m^--1}(x),$  
 there  exists
a  set $\Fc_{\tilde x,v}\subset\widehat\Omega(\widetilde L_x)$  of positive  measure in $\widetilde L_x$   such that 
\begin{equation}\label{eq_set_Fc}
 \limsup_{n\to\infty} {1\over n} \log\|  \widetilde{\mathcal A}(\tilde\omega,-n)u_{\tilde x, v,\tilde \omega}\| 
= \chi^{-}_{m^-}
\end{equation}
for every  $ \tilde\omega\in \Fc_{\tilde x,v}. $
  
  Therefore, we deduce from Theorem \ref{th_Lyapunov_filtration_Brownian_version} and Theorem \ref{thm_Oseledec_Brownian_version}
  and  Part 1) of Proposition  \ref{prop_backward_set_classification}
  that   there  exists  $1\leq s\leq l$ such that $V_m(\omega(0))=V_s(\omega)=\bigoplus_{j=s}^l  H_j(\omega)$ for $\mu$-almost every $x\in X$ and for almost every $\omega\in \widehat\Omega( L_x)$ (see Definition \ref{defi_set_null_measure} for the notion of almost everywhere in  $\widehat\Omega( L_x)$). Similarly,  by Theorem  \ref{th_Lyapunov_filtration_backward},
$\chi^{-}_{m^--1} \in\{-\lambda_1,\ldots,-\lambda_l\}$ and  $\chi^-_{m^-}=-\lambda_l.$
Consequently, we deduce from  Theorem \ref{thm_Oseledec_Brownian_version} and   Part 1) of Proposition  \ref{prop_backward_set_classification} that
 there  exists  $1\leq t< l$ such that $  V^-_{m^--1}(\omega(0)) :=\bigoplus_{j=1}^t  H_j(\omega)$
for $\mu$-almost every $x\in X$ and   for almost every $\omega\in \widehat\Omega(L_x).$ 
  This, combined with  the previous  decomposition of  $V_m(\omega(0)),$ implies that for $\mu$-almost  every $x\in X,$ 
  and for almost every $\omega\in \widehat\Omega( L_x),$
$$ \{0\}\not= H_l(\omega)\subset  V_m(\omega(0)) \setminus  V^-_{m^--1}(\omega(0)) .$$  
Note  by Corollary \ref{cor_leafwise_Oseledec} and Corollary  \ref{cor_leafwise_Oseledec_backward} that  
$Y\ni x\mapsto V_m(x)$ and  $Y\ni x\mapsto V^-_{m^--i}(x) $ with $0\leq i\leq m^-$
 are measurable    $\mathcal A$-invariant   bundles.

Next, we are in the position to apply  Theorem \ref{thm_dual_splitting_model}  to the following context:  the  measurable $\mathcal A$-invariant bundle
$Y\ni x\mapsto U(x)$ is given by $U(x):=V_m(x),$ and  its $m^{-}$ $\mathcal A$-invariant subbundles
$Y\ni x\mapsto V^{-i}(x)$ are  given by $V^{-i}(x):= V_m(x) \cap  V^-_{m^--i}(x).$   
Restricting  $\mathcal A(\omega,\cdot)$ on $V_m(x)$ for  $\omega\in\Omega_x,$ and applying   Theorem \ref{thm_dual_splitting_model}
yields that
   $\chi_m=-\chi^-_{m^-},$ as  desired. 
So $\chi_m=\lambda_l.$  Therefore, we deduce from Theorem \ref{th_Lyapunov_filtration_Brownian_version} 
and (\ref{eq_characterization_V_m})   that $   V_m(x)=V_l(\omega)$ for $\mu$-almost every $x\in X$ and  for $W_x$-almost every path
 $\omega.$

 \noindent{\bf Step 2: }{\it Proof that $m=l.$ Moreover, for  every $1\leq i\leq m,$ we have that  $\chi_i=\lambda_i $  and that
  for $\mu$-almost every $x\in X,$ it holds that
$V_i(x)=V_i(\omega)$ for $W_x$-almost every path
 $\omega\in\Phi.$}

 We will use  a  duality argument.
Applying Theorem \ref{th_Lyapunov_filtration_Brownian_version} to  the cocycle $\mathcal A^{*-1},$ 
  we obtain 
   $l^*$ Lyapunov exponents
$\lambda^*_{l^*}<\cdots< \lambda^*_1$
and the Lyapunov forward filtration  
  $$   \{0\}\equiv V^*_{l^*+1}(\omega)\subset   V^*_{l^*}(\omega)\subset \cdots\subset  V^*_2(\omega)\subset V^*_1(\omega)=\R^d,
$$
for   $\bar\mu$-almost  every  $\omega\in \Omega(X,\Lc).$ 
Applying Step 1  to  the cocycle $\mathcal A^{*-1},$  we obtain, for $\mu$-almost every $x\in X,$  a  space $V^*(x)\subset \R^d,$
such that  $V^*_{l^*}(\omega)=V^*(x) $ for $W_x$-almost every $\omega.$

 By  \cite{Ruelle}  we know that   $l^*=l$ and 
   $\{\lambda^*_{l^*},\ldots ,\lambda^*_1\}=\{-\lambda_1,\ldots ,-\lambda_l\}.$ 
 and $V^*_i(\omega)$ is the   orthogonal complement of    $V_{l+2-i}(\omega)$ in $\R^d$ for   $\bar\mu$-almost  every  $\omega\in \Omega(X,\Lc).$
 In particular,  $V^*_l(\omega)$ is the   orthogonal complement of    $V_2(\omega)$ in $\R^d$ for   $\bar\mu$-almost  every  $\omega\in \Omega(X,\Lc).$
 Recall from the previous  paragraph that  $V^*_l(\omega)=V^*(x)$ for $W_x$-almost every $\omega.$
 So  $V^* (x)$ is the   orthogonal complement of    $V_2(\omega)$ in $\R^d$  for $W_x$-almost every $\omega.$  
Let $V'_2(x)$ be the  orthogonal complement of  $V^*(x)$ in $\R^d.$
 We deduce  that $V_2(\omega)=V'_2(x)$ for  $\mu$-almost every $x\in X$ and for $W_x$-almost every $\omega.$
This, combined with  the definition  of $V_2(x),$ implies that  $V_2'(x)=V_2(x)$ for  $\mu$-almost every $x\in X.$
 So there  exists a  Borel set $Y\subset X$ of full $\mu$-measure  such that $Y\ni x\mapsto  V_2(x)$ is an $\mathcal A$-invariant measurable  bundle of rank $r_2.$  
By    Part 3) of Theorem
 \ref{thm_selection_Grassmannian},
  there is  a bimeasurable bijection  between the $\mathcal A$-invariant
subbundle  $Y\ni x\mapsto V_2(x)$    and $Y\times \R^{r_2}$ covering  the identity and which is a linear  isometry\index{isometry, linear isometry}  on each fiber. Using this   bijection, 
the  restriction  of $\mathcal A$ on $V_2(x),$ $x\in Y,$ becomes a cocycle  $\mathcal  A'$  on  $ \R^{r_2}.$  
Note that  the Lyapunov  exponents of $\mathcal A'$ are $\lambda_l<\cdots<\lambda_2.$ 

We repeat  the  previous  argument to   $\mathcal A'$   and  using the above  bijection.
 Consequently, we may find   a    Borel set $Y\subset X$ of full $\mu$-measure  such that 
 $V_3(\omega)=V_3(x)$ for    every $x\in Y$ and for $W_x$-almost every $\omega.$
  
 By still repeating  this  argument $(l-3)$-times, we  may find a    Borel set $Y\subset X$ of full $\mu$-measure   $ 1\leq i\leq l,$ 
 such that
 $V_i(\omega)=V_i(x)$ for  every $1\leq i\leq l$ and   every $x\in Y$ and for $W_x$-almost every $\omega.$
In particular, $m=l.$

 \noindent{\bf Step 3: }{\it Proof that $m=m^-=l.$ Moreover,   for  every $1\leq i\leq m$  we have   that
$\chi_i=\lambda_i=-\chi^-_i$ and that  for $\mu$-almost every $x\in X,$ there exists a space $H_i(x)\subset\R^d$ such that $H_i(\omega)=H_i(x)$ for $W_x$-almost every path
 $\omega\in\Phi.$ } 

Recall from Step 2 that $m=l$ and $\chi_i=\lambda_i$  for $1\leq  i\leq m.$  
First  we will prove that  $\chi^-_{m^--1}=-\lambda_{m-1}.$
Combining  Step 2 and   Theorem \ref{thm_Oseledec_Brownian_version} and Part 1) of Proposition  \ref{prop_backward_set_classification}, we get that
\begin{equation}\label{eq_First_Main_Theorem_V_m-1}
V_{m-1}(\omega(0))= H_m(\omega)\oplus  H_{m-1}(\omega)
\end{equation} 
for $\mu$-almost every $x\in X$ and for almost every $\omega\in \widehat\Omega( L_x).$

Recall  from Theorem \ref{th_Lyapunov_filtration_backward} that  $\chi^-_{m^-}=-\lambda_m$  and
$\chi^{-}_{m^--1},\chi^{-}_{m^--2} \in\{-\lambda_1,\ldots,-\lambda_{m-1}\}.$ 
Consequently, we deduce from  Theorem \ref{thm_Oseledec_Brownian_version} and   Part 1) of Proposition  \ref{prop_backward_set_classification} that
 there  exists  $1\leq t\leq m-1$ such that
 \begin{equation}\label{eq_First_Main_Theorem_V-_m-1}
 V^-_{m^--1}(\omega(0)) =\bigoplus_{j=1}^t  H_j(\omega)
 \end{equation}
 for $\mu$-almost every $x\in X$ and   for almost every $\omega\in \widehat\Omega(L_x).$ 
In particular, we get $\chi^-_{m^--1}=-\lambda_t.$

In order to prove that  $\chi^-_{m^--1}=-\lambda_{m-1}.$ 
  it suffices  to show that the possibility  $t < m-1$ cannot  happen  since $t\geq  m-1$ implies that $t=m-1$  and hence 
$\chi^-_{m^--1}=-\lambda_t=-\lambda_{l-1}.$

Suppose in order to reach a contradiction that  $t<m-1.$ Using the    decompositions (\ref{eq_First_Main_Theorem_V_m-1})-(\ref{eq_First_Main_Theorem_V-_m-1})
and noting that   $V^-_{m^-}(x)=\R^d$,
  we have that, for $\mu$-almost  every $x\in X,$ 
$$ \{0\}\not= H_{m-1}(x)\subset \big ( V_{m-1}(x)\cap  V^-_{m^-}(x) \big ) \quad\text{and}\quad V_{m-1}(x)\cap  V^-_{m^--1}(x)=\{0\}.$$  
   Consider the $\mathcal A$-invariant bundle $x\mapsto U(x)$ is given by
$U(x):=V_{m-1}(x)\cap  V^-_{m^-}(x),$ and its $\mathcal A$-invariant subbundle $x\mapsto V_m(x).$
Let $d':=\dim U(x).$
By    Part 3) of Theorem
 \ref{thm_selection_Grassmannian},
  there is  a bimeasurable bijection $\Lambda$ between the $\mathcal A$-invariant
bundle  $Y\ni x\mapsto U(x)$    and $Y\times \R^{d'}$ covering  the identity and which is a linear isometry\index{isometry, linear isometry} on each fiber.
Therefore, we are in the position to apply  Theorem \ref{thm_splitting_model}  to the following context:
$
V(x):= \Lambda(x, V_m(x)). 
$
Restricting  $\mathcal A(\omega,\cdot)$ on $ V_{m-1}(x)\cap  V^-_{m^-}(x) $ for  $\omega\in\Omega_x$
and  using  $\Lambda,$  we  obtain  a  cocycle  $\mathcal A'$ of rank $d'$  given by
$$
\mathcal A'(\omega,t) u:=\Lambda\Big(\omega(t), \mathcal A(\omega,t)\Lambda^{-1}(\omega(0), u)\Big),\qquad u\in\R^{d'},\ \omega\in\Omega(X,\Lc),\ t\in\R^+.
$$
 Applying  Theorem \ref{thm_splitting_model} to the cocycle  $\mathcal A'$ yields that
    $V(x)=0,$ hence  $V_m(x)=0$ for al $x\in Y,$ which is    impossible.  
 Thus we have shown that   $\chi^-_{m^--1}=-\lambda_{m-1}.$  Consequently, 
  we deduce from  Theorem \ref{thm_Oseledec_Brownian_version} and   Part 1) of Proposition  \ref{prop_backward_set_classification} that
 $  V^-_{m^--1}(\omega(0)) :=\bigoplus_{j=1}^{m-1}  H_j(\omega)$
 for $\mu$-almost every $x\in X$ and   for almost every $\omega\in \widehat\Omega(L_x).$ 
 
 So there  exists a  Borel set $Y\subset X$ of full $\mu$-measure  such that $Y\ni x\mapsto  V^-_{m^--1}(x)$ is an $\mathcal A$-invariant measurable  bundle of rank $d-d_m.$ 
By    Part 3) of Theorem
 \ref{thm_selection_Grassmannian},
  there is  a bimeasurable bijection  between the $\mathcal A$-invariant
subbundle  $Y\ni x\mapsto V(x)$    and $Y\times \R^{d-d_m}$ covering  the identity and which is a linear isometry\index{isometry, linear isometry} on each fiber. Using this   bijection, 
the  restriction  of $\mathcal A$ on $  V^-_{m^--1}(x),$ $x\in Y,$ becomes a cocycle  $\mathcal  A'$  on  $ \R^{d-d_2}.$  
Note that  the Lyapunov  exponents of $\mathcal A'$ are $\lambda_{l-1}<\cdots<\lambda_1.$ 

We repeat  the  previous  argument to   $\mathcal A'$   and  using the above  bijection. More specifically,
consider  $U(x):=V_{m-2}(x)\cap V^-_{m^--1}(x)$ and $V(x):= V_{m-1}(x)$ for each $x\in Y.$
 Consequently, we may find   a    Borel set $Y\subset X$ of full $\mu$-measure  such that 
  $  V^-_{m^--2}(\omega(0)) :=\bigoplus_{j=1}^{m-2}  H_j(\omega)$ for    every $x\in Y$ and for $W_x$-almost every $\omega\in\Phi.$
  
 By still repeating  this  argument $(m-3)$-times, we  may find a    Borel set $Y\subset X$ of full $\mu$-measure   $ 1\leq i\leq l,$ 
 such that
 $  V^-_i(\omega(0)) :=\bigoplus_{j=1}^i  H_j(\omega)$ 
 for  every $1\leq i\leq m$ and   every $x\in Y$ and for $W_x$-almost every $\omega\in\Phi.$
In particular, $m^-=m.$

Setting  $H_i(x):= V_i(x)\cap V^-_i(x),$  we deduce that $H_i(\omega)=H_i(x)$ and for $W_x$-almost every path
 $\omega\in\Phi.$ 
   
\noindent{\bf Step 4: }{\it  End of the proof.} 

First, observe that by  Step 3,  $H_i(x)\subset V_i(x)\setminus  V_{i+1}(x)$ for $\mu$-almost  every $x\in X.$
Recall also from Step 3 that for such a point $x,$ $V_i(\omega)=V_i(x)$  
for $W_x$-almost every $\omega\in\Phi.$
Consequently, by Theorem     \ref{th_Lyapunov_filtration} $\lim\limits_{n\to \infty} {1\over  n}  \log {\| \mathcal{A}(\omega,n)v   \|\over  \| v\|}  =\chi_i,    
$ for  $W_x$-almost every $\omega\in\Phi$ and for every $v\in H_i(x).$
 
 Next, we  will prove  the  following weaker  version  of  assertion (iii):

{\it There  exists  a  set $Y\subset X$ of full $\mu$-measure  such that  for every   subset  $S\subset  N:=\{1,\ldots,m\},$
$$
 \lim_{n\to\infty}{1\over n}\log\sin {\big |\measuredangle \big (H_S(\omega(n)), H_{N\setminus S} (\omega(n))\big ) \big |}=0
$$
 for every $x\in Y$ and  $W_x$-almost every path $\omega\in\Phi.$}

 Although the  argument is  standard, we still reproduce it here for
the  sake of  completeness.
To this  end     consider the  function 
$$
\phi(\omega):=   \log\sin {\big |\measuredangle \big (H_S(\omega(0)), H_{N\setminus S} (\omega(0))\big ) \big |},\qquad  \omega\in\Omega.
$$
 Observe that
 $$
 \big |  \phi(T\omega)  
-   \phi(\omega)  \big |\leq   \log \max \{\| \mathcal{A}(\omega,1)\|,\| \mathcal{A}^{-1}(\omega,1)\|  \}. 
$$  
So  $\phi\circ T-\phi$ is  $\bar\mu$-integrable. Hence,   our desired conclusion  follows from Lemma \ref{lem_tail_term}.

Summarizing what has been  done  in Step 4, we have  shown that there  exists a (not necessarily saturated) Borel set $Y\subset X$ of 
full $\mu$-measure  such that
all   assertions  (i)--(iii) of   Theorem \ref{th_main_1}  hold. Moreover,   for  each
$x\in Y,$ there  exists  a  set $\Fc_x\subset \Omega_x$ of full  $W_x$-measure 
such that  identity (\ref{eq_property_ii}) and   identity (\ref{eq_property_iii}) hold  for  all $\omega\in \Fc_x.$
It remains to show that
 by shrinking  the  set   $Y$ a little    we  can find such a set  $Y$ which  is also  
  leafwise saturated.   

The following  result is  needed.
 \begin{lemma}\label{lem_leafwise_saturation}
 Let $\Xi\subset\Omega$ be a $T$-totally invariant    
subset of full $\bar\mu$-measure. Then there exists a leafwise saturated Borel subset
$Y\subset X$ of full $\mu$-measure such that for every $y\in Y,$   $\Xi$ is  of full $W_y$-measure.   
 \end{lemma}
 \begin{proof}  We say that  a  set $Z\subset X$ is {\it almost leafwise saturated}
 if $a\in Z$ implies that   the whole  leaf $L_a$ except a   null Lebesgue measure  set
 is  contained  in $Z,$ 
where the Lebesgue   measure   on $L_a$  is  induced  by the  Riemannian  metric $g$ on $L_a.$  
Since $\Xi\subset\Omega$ is  $T$-totally invariant, $T(\Omega\setminus \Xi)= \Omega\setminus \Xi.$
On the other hand,
 we deduce  from   $\bar\mu(\Xi)=1$ that  $\bar\mu(\Omega\setminus \Xi)=0.$
Hence,   $\bar\mu(T(\Omega\setminus \Xi))=0.$ So there exists an almost leafwise  saturated  subset $Z\subset X$ of full $\mu$-measure  such that
 for every $x\in Z,$   $W_x(T(\Omega\setminus \Xi)  )=0.$
Let $Y$ be the   leafwise saturation  of $Z.$ Clearly,  $\mu(Y)=1.$ By shrinking $Y$  a little  if necessary we may assume that $Y$ is  a Borel set.
Let $y$ be an arbitrary point   in $Y.$ Since for $\Vol$-almost every $x\in L_y$ we have  $x\in Z$ it follows that
$W_x (T(\Omega\setminus \Xi))=0$ for such a point $x.$   
Consequently, by  Proposition
 \ref{prop_Markov} (i), we get that
$$
W_y(\Omega\setminus \Xi )  \leq \int_{x\in L_y} p(x,y,1) W_x(T(\Omega\setminus \Xi)) d\Vol(x)=0.
$$
Hence,   $\Xi$ is  of full $W_y$-measure for  all $y\in Y.$
    \end{proof}

 Now  we resume the proof of the First Main Theorem.\index{theorem!First Main $\thicksim$}
By shrinking the  set $Y\subset X$ a little  we  may assume without loss of generality that
$Y$ is  almost leafwise  saturated  of full $\mu$-measure. Let $Y'$ be the leafwise saturation of $Y.$
Using the action of $\mathcal A$ on $\R^d,$  we  can extend   $m$ functions  $Y\ni x\mapsto H_i(x)$ to    $m$ functions  $Y'\ni x\mapsto H_i(x)$
as follows: given  any point   $x'\in Y',$ we find  a point $x\in L_{x'}\cap Y$ and  set
$
H_i(x')= \mathcal A(\omega,1) H_i(x) 
$
for any  path $\omega\in \Omega_x$ with $\omega(1)=x'.$  This  extension is well-defined (i.e. no monodromy problem  occurs) because $Y\ni x\mapsto H_i(x)$ is $\mathcal A$-invariant.
     Consider the set
 \begin{multline*}
 \Xi:=\left\lbrace  \omega\in\Omega(X,\Lc):\  \lim\limits_{n\to \infty} {1\over n} \log {\| \mathcal{A}(\omega,n)u   \|\over  \| u\|}=\chi_i,\   \forall u\in  H_i(x)\setminus \{0\},\ \forall\ 1\leq i\leq m, \right.\\
\left.   \&   
 \lim_{n\to\infty}{1\over n}\log\sin {\big |\measuredangle \big (H_S(\omega(n)), H_{N\setminus S} (\omega(n))\big ) \big |}=0,\quad        \forall S\subset  N:=\{1,\ldots,m\}\right\rbrace.
 \end{multline*}
By  Step 3  as  well  as  the  assertion  established  in the preceding paragraphs,   $\Xi$ is  of full $\bar\mu$-measure.
On the  other hand, using that  $x\mapsto H_i(x)$ is $\mathcal A$-invariant,  it is  straightforward  to  see that $\Xi$ is $T$-totally invariant. Therefore, applying
   Lemma \ref{lem_leafwise_saturation} yields a  leafwise  saturated Borel  set  $Y''\subset Y'$ of full $\mu$-measure  such that  $\Xi$ is  of full $W_y$-measure
 for every $y\in Y''.$ This  completes the proof of Theorem  \ref{th_main_1} in the  case  $\G=\N.$
 In the sequel  we  write $Y$ instead of $Y''$ for  simplicity.
 
\subsection{Case II:    $\G= \R^+$}
\label{subsection_First_Theorem_Case_II}

  We only need to establish  assertion (ii) and (iii) of Theorem  \ref{th_main_1}.
  Without loss of generality  we may assume that $t_0=1.$ By the hypothesis the  function
  $F:\ \Omega(X,\Lc)\to\R^+$ given by
  $$
  F(\omega):= \sup_{t\in [0,1]}\big | \log \|\mathcal{A}^{\pm 1}(\omega,t)\| \big |,\qquad \omega\in  \Omega(X,\Lc),
  $$  is   $\bar\mu$-integrable. Therefore,
 by Birkhoff ergodic theorem,\index{Birkhoff!$\thicksim$ ergodic theorem}
\index{theorem!Birkhoff ergodic $\thicksim$}   ${1\over n} F\circ T^n$
converge  to $0$ $\bar\mu$-almost everywhere  when the integer   $n$ tend to $\infty.$ 
On the  other hand,  for $n\leq t<n+1$ and for $u\in\R^d\setminus \{0\},$   we have  that
$$
\Big |  \log {\| \mathcal{A}(\omega,t)u   \|\over  \| u\|}
-   \log {\| \mathcal{A}(\omega,n)u   \|\over  \| u\|}\Big |\leq   \log \max \{\| \mathcal{A}(T^n\omega,t-n)\|,\|\mathcal{A}^{-1}(T^n\omega,t-n)\|  \}. 
$$  
The right hand side is  bounded by  $(F\circ T^n)(\omega).$
This, coupled   with  the  convergence  of  
$  {1\over  n}  \log {\| \mathcal{A}(\omega,n)u   \|\over  \| u\|}$  to $\chi_i$
when   $u\in H_i(x)\setminus \{0\}$ and  the integers $n$ tend to $\infty$ for  $W_x$-almost every  $\omega\in\Omega_x,$
and with the  convergence of  ${1\over n} F\circ T^n(\omega)$ to $0$ for $\bar\mu$-almost everywhere $\omega,$ implies that
$$
\lim\limits_{t\to \infty,\ t\in \R^+} {1\over t} \log {\| \mathcal{A}(\omega,t)u   \|\over  \| u\|}=\chi_i,\qquad  u\in  H_i(x)\setminus \{0\},$$
for $\bar\mu$-almost everywhere $\omega.$ Consequently, it is  sufficient to  apply Lemma \ref{lem_leafwise_saturation}
in order  to  conclude assertion (ii). 
  
  We turn to the proof of  assertion (iii). Fix  a subset  $S\subset  N:=\{1,\ldots,m\}$ and consider the  function 
$$
\phi(\omega):=   \log\sin {\big |\measuredangle \big (H_S(\omega(0)), H_{N\setminus S} (\omega(0))\big ) \big |},\qquad  \omega\in\Omega.
$$
 Observe that, for $n\leq  t<n+1,$
 $$
 \big |  \phi(T^t\omega)  
-   \phi(T^n\omega)  \big |\leq   \log \max \{\| \mathcal{A}(T^n\omega,t-n)\|,\|\mathcal{A}^{-1}(T^n\omega,t-n)\|  \}. 
$$  
The right hand side is  bounded by  $(F\circ T^n)(\omega).$
 This, coupled   with  the  limit 
$$
\lim\limits_{n\to \infty} {1\over n}  \log\sin {\big |\measuredangle \big (H_S(\omega(t)), H_{N\setminus S} (\omega(t))\big ) \big |}=0
$$ 
 and with the  convergence of  ${1\over n} F\circ T^n(\omega)$ to $0$ for $\bar\mu$-almost everywhere $\omega,$ implies that
 $$
\lim\limits_{t\to \infty,\ t\in \R^+} {1\over t}  \log\sin {\big |\measuredangle \big (H_S(\omega(t)), H_{N\setminus S} (\omega(t))\big ) \big |}=0
$$
 for $\bar\mu$-almost everywhere $\omega.$ Using  this and   applying Lemma \ref{lem_leafwise_saturation} again,
 assertion (iii) follows.  \qed
 
 \subsection{Proofs of the corollaries and Ledrappier type characterization}
\label{subsection_Ledrappier}
 
 Now  we arrive  at  the
 
 \noindent{\bf Proof of  Corollary \ref{cor1_th_main_1}.}

Let $Y$  be  the  set  given  by  Theorem \ref{th_main_1}.  Part (i) and (ii) of this  theorem implies that   the functions  $m$ and  $ \chi_i$  are leafwise constant. Using the assumption that $\mu$ is  ergodic and
removing  from $Y$  a null $\mu$-measure set if necessary, the   conclusion (except  the two identities) of the  corollary  follows.

If  $\G=\N,$ the two identities of the corollary hold  by   Ruelle's work\index{Ruelle}  \cite{Ruelle}.  
If $\G=\R^+,$ we argue as  in Case  II  of the proof of Theorem  \ref{th_main_1}.     
 \qed
 \smallskip
 
 \noindent{\bf Proof of  Corollary \ref{cor2_th_main_1}.} 
 
Using the  remark  following  Theorem  \ref{th_main_1}, the case  $k=1$ of the corollary  is  exactly  Corollary  \ref{cor1_th_main_1}. Now  we consider the case $k>1.$
 If in   the corollary  we  replace  $\chi(x;v_1,\ldots,v_k) $ by $\chi(\omega;v_1,\ldots,v_k),$ 
  $\omega\in\Omega_x,$   then  
 assertions (ii), (iii) and (iv) of the corollary follow from   Ruelle's work\index{Ruelle}  \cite{Ruelle}. So it suffices to  prove  assertion (i). To this end
  we apply assertion  (ii) of Theorem    \ref{th_main_1}   to the cocycle  $\mathcal A^{\wedge k}.$
Consequently, we may find a  leafwise saturated  Borel set $Y\subset X$ of full $\mu$-measure such  that for every $x\in Y,$ there  exists a
set $\Fc\subset \Omega_x$ of full $W_x$-measure such that   if   $v_1,\ldots, v_k\in \R^d$ are fixed, then $\chi(\omega;v_1,\ldots,v_k)$  
 is  constant  for all $\omega\in \Fc_x.$   We denote by 
    $\chi(x;v_1,\ldots,v_k)$ this common value. 
    Assertion (i) follows.
 \qed
 
   We arrive   at the  spectrum  description in terms of $\mathcal A$-weakly harmonic measures.
We  are inspired by Ledrappier  \cite[Proposition 5.1, pp. 328-329]{Ledrappier} who
studies the case of measurable maps.\index{Ledrappier!$\thicksim$ type characterization}
\begin{theorem}\label{thm_Ledrappier}
 We keep the   hypotheses,   notation  and  conclusions of  Corollary \ref{cor1_th_main_1}.
 So $Y$ is  a  leafwise saturated  Borel set of full $\mu$-measure given by this  corollary. 
\\
1) For each ergodic probability measure $\nu$ which is  also an  element  of $  \Har_\mu(X\times \P(\R^d)),$ there is a unique integer $1\leq i\leq m$  such that
\\
(i)  $\int_{X\times \P(\R^d)}\varphi  d\nu=\chi_i;$
\\
(ii) $\nu$ is  supported  by the total space  of the $\mathcal A$-invariant subbundle
$Y\ni x\mapsto \P(H_i(x)),$  i.e, 
$$ \nu\left \lbrace (x,u):\ x\in X \ \& \  u\in \P(  H_i(x))\right\rbrace 
 =1.$$
\\
2)  Conversely,
for each $1\leq i\leq m,$ there exists such a  measure $\nu.$ 

In particular, the spectrum (i.e. the set of all  Lyapunov exponents) of $\mathcal A$  is the set of values of 
 $\int_{X\times \P(\R^d)}\varphi d\nu$ as  $\nu$ runs over all probability  measures which are also  ergodic  elements  in  $  \Har_\mu(X\times \P(\R^d)).$
\end{theorem}
\begin{proof}
Recall from  Step 3 and Step 4  in Case  I in  the proof of Theorem \ref{th_main_1}   that 
 \begin{equation}\label{eq_coincidence}
H_i(\omega)=H_i(x)
\end{equation}
for every $x\in Y$ and  for $W_x$-almost every $\omega\in\Omega.$
Next, applying  Corollary \ref{cor_Birkhoff_ergodic_thm} to $\nu$ yields that
$\bar\nu$ is $T$-ergodic on  $\Omega_{1,\mathcal A}$ and
 $\hat\nu$ is $T$-ergodic on  $\widehat\Omega_{1,\mathcal A},$
 where
 $\bar\nu$ is the Wiener measure with initial  distribution $\nu$ given by  (\ref{eq_formula_bar_mu}), and
 $\hat\nu$ is the  natural extension of $\bar\nu$ on $\widehat\Omega_{1,\mathcal A}.$
 Recall from   Lemma   \ref{lem_identifications_spaces_dim1}
that  $\Omega_{1,\mathcal A}\equiv  \Omega\times \P(\R^d)$ 
 and $\widehat\Omega_{1,\mathcal A}\equiv \widehat\Omega\times \P(\R^d).$
 Using  all these and applying \cite[Proposition 5.1, pp. 328-329]{Ledrappier} to the  ergodic  map $T$ acting  on  $(\Omega_{1,\mathcal A}, \bar\nu)$
(resp. $(\widehat\Omega_{1,\mathcal A},\hat\nu)$) yields a unique integer $i$ with $1\leq i\leq m$ such that
$\int_{\Omega\times \P(\R^d)}\varphi  d\bar\nu=\chi_i$
and
that  
$$ \bar\nu\left \lbrace (\omega,u):\ \omega\in \Omega \ \& \  u\in \P(  H_i(\omega(0)))\right\rbrace 
 =1.$$
 Combining  the former equality with (\ref{eq_formula_bar_mu}), assertion (i) follows.
The  latter  equality, coupled  with    identity (\ref{eq_coincidence}) and (\ref{eq_formula_bar_mu}), implies assertion (ii), thus
proving Part 1).
 
 Now  we turn to Part 2).
By Theorem \ref{thm_selection_Grassmannian}, there is a bimeasurable bijection between 
the  bundle $Y\ni x\mapsto
 H_i(x)$ and $Y\times \R^{d_i}$    covering the identity and which is  linear on fibers.
 Using this  
 and  applying Part 1) of Lemma \ref{lem_existence_harmonic_measures} and   applying  Proposition \ref{prop_extremality_implies_ergodicity}, we may find  an ergodic  $\mathcal A$-weakly harmonic probability
 measure $\nu$ living  on    the 
leafwise  saturated  subset    $\{(x,\P H_i(x)):\ x\in Y\}$ of the lamination $(X_{1,\mathcal A},\Lc_{1,\mathcal A}),$
that is,
$$ \nu \Big(\left\{(x,\P H_i(x)):\ x\in Y\right\}\Big)=1.$$
 Arguing as in the proof of Part 1),  Part 2) follows.
 \end{proof}
\begin{remark}\label{R:End}
We close  the section  with the  following  discussion  on the optimality of the  hypotheses in Theorem \ref{th_main_1}.
This  issue has been  mentioned  in Remarks \ref{R:th_main_1} and  \ref{R:End_B1}.

 In fact, we only  use  Hypothesis (H2) and  Definition   \ref{defi_Standing_Hypotheses_harmonicity} (ii) 
 in order to obtain  Proposition \ref{prop_current_local}.
 On the  other  hand, as  already observed in Remark \ref{R:End_B1}, the leafwise  Laplacian\index{leafwise!$\thicksim$ Laplacian}\index{Laplacian, Laplace operator!leafwise $\thicksim$}
  is  not really needed in  the proof of  Theorem \ref{th_main_1},
 and 
 only a weaker version of this proposition (see assumption (iii) below) suffices for the validity of  the whole Appendix \ref{subsection_fibered_laminations}.
   
 Therefore, we  conclude that  Theorem \ref{th_main_1} still  remains  valid if  we make   the following  weaker
 assumptions (i)--(iii) on the Riemannian lamination $(X,\Lc,g)$ and on the measure $\mu.$ 
 \begin{itemize}
 \item[(i)] $(X,\Lc,g)$ satisfies Hypothesis (H1).
 
 \item[(ii)] $\mu$ is  weakly harmonic.
 
 \item[(iii)]
  Let $\U\simeq \B\times\T$ be a flow
box  which is relatively compact in $X$. Then, there is a positive Radon
measure $\nu$ on $\T$ and for $\nu$-almost every $t\in \T$ there is a
measurable positive  function $h_t$ on $\B$ 
such that if $K$ is compact in $\B,$ 
the integral $\int_\T \|h_t\|_{L^1(K)}d\nu(t)$ is finite and
$$\int  fd\mu=\int_\T \Big(\int_\B h_t(y) f(y,t) d\Vol_t(y)\Big) d\nu(t)
$$
for every  continuous compactly supported function    $f$ on $\U.$
Here $\Vol_t(y)$  denotes the  volume form on $\B$ induced by the metric tensor $g|_{\B\times \{t\}}.$
\end{itemize}
In particular,  in assumption (iii) above we  do not need that $h_t$ is  harmonic  on $\B$ 
with respect to the metric tensor $g|_{\B\times \{t\}}.$
\end{remark}

\section{Second Main Theorem and its corollaries}\index{theorem!Second Main $\thicksim$}
\label{subsection_Second_Main_Theorem}
Let $(X,\Lc,g)$ be  a  Riemannian lamination satisfying the Standing  Hypotheses.
In this  section we  will combine  Theorem  \ref{th_main_1} and  Candel's results \cite{Candel2}\index{Candel}
in order  to   establish   Theorem  \ref{th_main_2}.  
 Let $\dist$ be the  distance  function induced  by the  Riemannian  metric $g$  on  every leaf.
Following    Candel \cite{Candel2}\index{Candel} we  introduce  the  following terminology
\begin{definition} \rm
A cocycle $\mathcal A:\ \Omega(X,\Lc)\times \R^+\to \GL(d,\R)$ 
  is  said to be {\it moderate}\index{cocycle!moderate $\thicksim$}
if there  exist constants $C,R>0$ such that
 $$  \log {\| \mathcal{A}^{\pm 1}(\omega,t)   \|}
 \leq  C\dist(\omega(t),\omega(0))+R,\qquad  \omega\in\Omega,\ t\in\R^+.
$$
\end{definition}
Here is a simple  sufficient condition   for a moderate cocycle.  
\begin{lemma}\label{lem_criterion_moderate_cocycles}
 If 
 $\mathcal A$ is $\Cc^1$-differentiable  cocycle on a compact $\Cc^1$-smooth lamination,
 then $\mathcal A$ is moderate.
\end{lemma}
\begin{proof}
Choose  a  finite covering of $X$ by flow boxes $\Phi_i:  \U_i\to \B_i\times \T_i$ with  $\B_i$ simply connected. In any flow  box  
$\Phi_i,$   let 
$\alpha_i:\  \B_i\times \B_i\times \T_i\to\GL(d,\R)$  be the  local expression of  $\mathcal A$ (see Definition \ref{defi_local_expression} with the choice $t_0:=1$). 
So
$$
\alpha_i(x,y,s)=\mathcal A(\omega,1),\qquad  (x,y,s)\in  \B_i\times \B_i\times \T_i,
$$
where  $\omega$ is  any leaf path  such that $\omega(0)=\Phi_i^{-1}(x,s),$ $\omega(1)=\Phi_i^{-1}(y,s)$ 
and  $\omega[0,1]$ is  contained in the simply connected  plaque   $\Phi_i^{-1}(\cdot,s).$
We deduce from this formula and the  identity law in Definition \ref{defi_cocycle} that  $\alpha_i(x,x,s)=\id.$ Consequently,   $\| \alpha_i(x,y,s)- \id\|\leq  C\|x-y\|$
for a finite constant  $C$ independent of the flow box $\Phi_i.$ This implies the desired conclusion.
 \end{proof}
We  will prove  the following  
\begin{proposition}\label{prop_moderate_cocycles}
Let $\mathcal A$ be  a moderate cocycle. Then
 the function $F:\ \Omega(X,\Lc)\to \R^+$ defined by
$$  F(\omega):=
 \sup_{t\in[0,1]} \log^+ \|\mathcal{A}^{\pm 1}(\omega,t)\|,\qquad \omega\in \Omega(X,\Lc),
$$
is $\bar\mu$-integrable.  
\end{proposition}
\begin{proof}Since  $\mathcal A$ is  moderate, we get that
 $$
 \log^+ \|\mathcal{A}^{\pm 1} (\omega,t)\|\leq C\dist(\omega(0),\omega(t))+R ,\qquad  \omega\in\Omega(X,\Lc),\ t\in\R^+.
$$
  Therefore,
$$\int_{\Omega(X,\Lc)} \sup_{t\in[0,1]}\log^+ \|\mathcal{A}^{\pm 1} (\omega,t)\|  d\bar\mu(\omega)\leq R+C\int_{\Omega(X,\Lc)}  \sup_{t\in[0,1]}\dist(\omega(0),\omega(t))d\bar\mu(\omega). $$
By  formula (\ref{eq_formula_bar_mu}) we  may rewrite the integral  on the  right hand side as
$$
\int_X  \Big (\int_{\Omega_x} \sup_{t\in[0,1]}\dist(\omega(0),\omega(t))d W_x(\omega) \Big )d\mu(x).
$$  
We will prove that  the  inner  integral is  bounded     from above  by a  constant   independent  of $x.$ This  will imply that
 the function $${\Omega(X,\Lc)}\ni \omega\mapsto R+ C \sup_{t\in[0,1]} \dist(\omega(0),\omega(t)) $$ is   $\bar\mu$-integrable, 
and hence so is  the function $F.$  
To this  end 
  we focus on a single  $L$ passing through a  given fixed point $x.$ Observe  that
 $$\int_{\Omega_x}    \sup_{t\in[0,1]} \dist(\omega(0),\omega(1)) dW_x(\omega)=  \int_0^\infty  W_x\{\omega\in \Omega_x:\ \sup_{t\in[0,1]}\dist(\omega(0),\omega(t))>s  \} ds. $$
 The  following estimate is  needed.
\begin{lemma}\label{lem_Candel} There is  a finite  constant $c>0$ such that for all $s\geq 1,$
 $$ W_x\left\{\omega\in\Omega (X,\Lc):\  \sup_{t\in[0,1]}\dist(\omega(0),\omega(t))>s \right \}< c e^{-s^2}.$$
\end{lemma}
\begin{proof} It follows  by combining  Lemma 8.16 and Corollary 8.8   in   \cite{Candel2}.  
\end{proof}
Resuming the  proof of Proposition \ref{prop_moderate_cocycles},  Lemma \ref{lem_Candel}, applied to the right hand side of the last  equality,  shows that the integral $$\int_{\Omega_x}  \sup_{t\in[0,1]} \dist(\omega(0),\omega(t)) dW_x(\omega)$$ is bounded  from above  by a  constant   independent  of $x\in X.$
 This completes the proof. 
\end{proof}

Now  we  are in the position   to  prove   Theorem   \ref{th_main_2}.
The  proof is  divided into two steps.
\\
 \noindent{\bf Step I:} {\it Proof of assertions (i) and (ii).}

Since  $\mathcal A$ is $\Cc^1$-differentiable, it follows from    Lemma  \ref{lem_criterion_moderate_cocycles} that
$\mathcal A$ is  moderate. By Proposition \ref{prop_moderate_cocycles}, we get the integrability  condition:
$$
\int_{\Omega(X,\Lc)} \sup_{t\in[0,1]} \log^+ \|\mathcal{A}^{\pm 1}(\omega,t)\|d\bar\mu(\omega)<\infty.
$$
Consequently,   we are able  to apply  Theorem  \ref{th_main_1}. Hence,  assertions (i) and (ii) of  Theorem  \ref{th_main_2} follow.

 \noindent{\bf Step II:} {\it Proof of assertions (iii).}

First  we  will prove that  $\underline\chi_{\max}(\mathcal A)\leq  \chi_1\leq  \bar\chi_{\max}(\mathcal A).$ 
In fact, we only show that   $\chi_1\leq  \bar\chi_{\max}(\mathcal A)$ since
the  inequality  $\chi_1\geq  \underline\chi_{\max}(\mathcal A)$  can be  proved in the  same way.
The  proof is  divided into several  sub-steps.

 \noindent{\bf Sub-step II.1:} {\it Proof  that for every $x\in Y$ and  $u\in \P(\R^d)$  and $t>0,$
 \begin{equation}\label{eq2_formulas_lambda1} 
\int_{\Omega_x} \log {\|\mathcal A(\omega,t)u\| }dW_x(\omega)\leq  \int_0^t\big (D_s\bar\delta(\mathcal A)\big )  (x)ds.
\end{equation}
}
To prove  (\ref{eq2_formulas_lambda1}) we  fix  an  arbitrary point $x\in X$  and an arbitrary $u\in \P(\R^d).$ 
 Let $\pi:\ \widetilde L\to L$ be  the universal cover of the leaf $L:=L_x$  and fix $\tilde x\in \widetilde L$  that projects to $x.$
Recall  that the bijective lifting  $\pi_{\tilde x}^{-1}:\  \Omega_x\to\widetilde\Omega_{\tilde x}$  identifies  the two path-spaces canonically.
Following (\ref{eq_cocycle_covering_manifold_new})-(\ref{eq_function_f_stepII}) consider   the  specialization  $f:\  \widetilde L\to \R$ of $\mathcal A$ at $(\widetilde L, \tilde x; u)$ defined by
\begin{equation}\label{eq_thmain2_function_f}
f(\tilde y):= \log \| \mathcal A(\omega,1)u \|,\qquad  \tilde y\in \widetilde L,
\end{equation}
where  $\omega\in \Omega_x$ is any path  such that  $(\pi_{\tilde x}^{-1}\omega)(1)=\tilde y.$

 Fix  an arbitrary  point $y\in L$ and an arbitrary point $\tilde y\in\pi^{-1}(y).$
Let $v:= \mathcal A(\omega,1)u\in \P(\R^d),$ where  $\omega\in \Omega_x$ is any path  such that   $(\pi_{\tilde x}^{-1}\omega)(1)=\tilde y.$
 Let  $z$ be  an arbitrary point in  a     simply connected, connected open neighborhood of $y.$ 
On this  neighborhood  a branch of $\pi^{-1}$  such that $\pi^{-1}(y)=\tilde y$ is  well-defined.    
Set $\tilde z:=\pi^{-1}(z).$ 
 By (\ref{eq_specialization_comparison}), 
 we have that
   \begin{equation*}
 f_{v,\tilde y}(\tilde z) =
f(\tilde z) -f(\tilde y),
 \end{equation*}
where  $f$ is  defined  by   (\ref{eq_thmain2_function_f}). 
This, combined   with formula  (\ref{eq_formulas_delta}) and (\ref{eq_identity_f_f}), implies that
 $$
 \bar\delta(\mathcal A)(y)\geq \Delta_z f_{v,y}(y)= (\widetilde \Delta_{\tilde z}) f_{v,\tilde y}(\tilde y)= (\widetilde \Delta f)(\tilde y).
 $$ In summary, we have     proved  
the following crucial estimate:
\begin{equation}\label{eq3_formulas_lambda1}
\bar\delta(\mathcal A)(y)\geq  (\widetilde \Delta f)(\tilde y),\qquad y\in L,\  \tilde y\in \pi^{-1}(y).
\end{equation}
Consider the cocycle $\widetilde {\mathcal A}$ on  $\widetilde L$  defined by
$$  \widetilde{\mathcal  A}(\tilde \omega,t):=\mathcal A(\pi(\tilde \omega),t),\qquad t\in \R^+,\ \tilde \omega\in \Omega(\widetilde L).$$
Using    the  homotopy law for $\widetilde {\mathcal A}$  and  using  the  simple connectivity of $\widetilde L,$ we see that  
$$
f(\pi_t(\tilde \omega))= \log \|\widetilde{ \mathcal A}(\tilde\omega,t) u\| ,\qquad  \tilde \omega\in \Omega_{\tilde x}(\widetilde L),\ t\in\R^+.
$$  
 Consequently, we infer from (\ref{eq_E_x_log_A_t}) with $(\tilde z, v)=(\tilde x, u)$ that 
 \begin{equation}
\label{eq_expectation_A_x}
\Et_x[\log {\|\mathcal A(\cdot,t)u\|}  ]=\Et_{\tilde x} [\log {\|\widetilde{\mathcal A}(\cdot,t)u\|}  ]=  (D_t f)(\tilde x)=
 (D_t f)(\tilde x)- f(\tilde x),
\end{equation}
where the last equality holds  because of $f(\tilde x)=0$ by (\ref{eq_normalization}).

Recall  from  Definition 8.3 in \cite{Candel2} that  a function $h$ defined on a complete Riemannian manifold $M$ with  distance function
$\dist$ is said to be {\it  moderate} (with  constants $C,R>0$) if
$$
\log |h(y)-h(z)|\leq  C\dist(y,z)+R,\qquad  y,z\in M.
$$ 
  In particular,  every bounded  function is  moderate.
We need the following result.
\begin{lemma}\label{lem_D_t_Delta}
If $f,$ $|df|,$ and $\Delta f$  are moderate functions on $\widetilde L,$ then
$$
(D_tf)(\tilde x) - f(\tilde x)=\int_0^t D_s \Delta  f (\tilde x) ds.
$$ 
\end{lemma}
\begin{proof}
It follows  from the proof of Proposition 8.11 and  Theorem  8.13 in \cite{Candel2}. 
Note that here is the place where we make use of the hypotheses that $(X,\Lc)$ is a compact $\Cc^2$-differentiable lamination 
and the  leafwise metric $g$ is transversally continuous  as the proof of Proposition 8.11  in \cite{Candel2} requires these  assumptions. 
\end{proof}
\begin{lemma}\label{lem_smooth_implies_moderates}
Assume that the cocycle $\mathcal A$   is  $\Cc^2$-differentiable.  
Then   there  are  constants $C,R>0$   with the following property.  For every $x\in X$ and every $u\in\P(\R^d),$
let  $f$ be  the function defined  by (\ref{eq_function_f_stepII}),   then $f,$ $|df|,$ and $\Delta f$  are moderate functions on $\widetilde L_x$ with  constants $C,R.$
\end{lemma}
\begin{proof}
We  only prove  that  $\Delta  f$ is moderate since the other assertions can be  proved  similarly.
In fact, we  will  prove that $\Delta f$ is  bounded.
 For every $y\in L_x$  and  for  every path  $\omega\in \Omega_x$  with   $\omega(1)=y,$ let     $v:= \mathcal A(\omega,1)u\in \P(\R^d).$  So  the function $  f_{v,y}$ constructed  in  
(\ref{eq_function_f}) (with $(v,y)$ in place of $(u,x)$)
 is well-defined   on  any simply connected neighborhood of
$y$ in $L_x.$  
We will  show that  there is a  constant $C>0$ independent of $x,y$ and $u,v$  such that 
\begin{equation}\label{eq_lem_smooth_implies_moderates}
|\Delta f_{v,y}(z)| \leq  C
\end{equation}
for  every  $z$ in any simply connected  plaque passing through $y.$
By (\ref{eq_identity_f_f}) and by the  compactness of the lamination $(X,\Lc),$ 
(\ref{eq_lem_smooth_implies_moderates}) will   imply  that    $\Delta  f$ is bounded, and  hence  moderate.

To prove  (\ref{eq_lem_smooth_implies_moderates}) let  $\Phi:\ \U\to \B\times \T$ be a flow  box  containing  $y.$ 
 Let 
$\alpha:\  \B\times \B\times \T\to\GL(d,\R)$  be the  local expression of  $\mathcal A$ on a (see Definition \ref{defi_local_expression}). 
So there is  $s\in\T$ such that
$$
\alpha(y,z,s)=\mathcal A(\eta,1),\qquad  (y,z,s)\in  \B\times \B\times \T,
$$
where  $\eta$ is  any leaf path  such that $\eta(0)=\Phi_i^{-1}(y,s),$ $\eta(1)=\Phi_i^{-1}(z,s)$ 
and  $\eta[0,1]$ is  contained in the simply connected  plaque   $\Phi^{-1}(\cdot,s).$
Since  $\mathcal A$   is  $\Cc^{2}$-differentiable, we deduce from  the last equality that
$\alpha(y,\cdot,s)\in\Cc^{2}.$ This, combined  with  the  equality
$$
 f_{v,y}(z)= \log{\| \mathcal A(\eta,1) v\|}=\log{\| \alpha(y,z,s)v\|}, 
$$
implies   (\ref{eq_lem_smooth_implies_moderates}).
    \end{proof}
Coming  back  the proof of assertion (iii), recall from the hypotheses that $\mathcal A$ is  $\Cc^{2}$-differentiable. Therefore, 
 the function  $f$  in  Lemma \ref{lem_smooth_implies_moderates}  satisfies  the hypotheses of Lemma \ref{lem_D_t_Delta}.
Consequently, for every $u\in \P(\R^d),$
$$
\int_{\Omega_x} \log {\|\mathcal A(\omega,t)u\| }dW_x(\omega)= \int_0^t D_s \Delta  f (\tilde x) ds \leq  \int_0^t\big (D_s\bar\delta(\mathcal A)\big )  (x)ds,\qquad x\in X,\ t>0,
$$
   where the equality holds by combining  Lemma \ref{lem_D_t_Delta} and (\ref{eq_expectation_A_x}), 
and
the inequality  holds  by an application of inequality  (\ref{eq3_formulas_lambda1}).
   This  proves  (\ref{eq2_formulas_lambda1}). 
 
 \noindent{\bf Sub-step II.2:} {\it  End of the proof of the  inequality    
 $\chi_1\leq  \bar\chi_{\max}(\mathcal A).$
}  

By Theorem \ref{thm_Ledrappier} there exists  an ergodic  probability measure $ \nu$ which is an  element of  $  \Har_\mu(X\times P)$   such that
 $\int_{X\times P}\varphi d\nu=\chi_1,$  where $\varphi$ is  defined in  (\ref{eq_varphi_n}). Consequently, using  this together  with  Theorem \ref{thm_ergodic_for_C}
and   formula (\ref{eq_canonical_cocycle}) for
  the canonical cocycle $\mathcal C_{\mathcal A},$  we infer that
 $$
  \int_{\Omega\times\P(\R^d)}  \log {\|\mathcal A(\omega,1) u\|}  d\bar\nu(\omega,u) = \chi_1.
 $$
Using  formula  
(\ref{eq_formula_nu_x})
 we rewrite  the  left hand side as  
$$
 \int_X \Big ( \int_{u\in \P(\R^d) }\big(\int_{\Omega_x} \log {\|\mathcal A(\omega,1) u\|}
dW_x(\omega) \big)d\nu_x(u)\Big) d\mu(x).
$$
Next, applying   inequality  (\ref{eq2_formulas_lambda1}) to the  inner  integral and  recalling that
each $\nu_x$ is a probability measure on $\P(\R^d),$ we deduce  from the last two  equalities  that  
  $$
  \chi_1\leq \int_X \Big ( \int_0^1\big(D_s\bar\delta(\mathcal A) \big) (x) ds\Big) d\mu(x).
$$
  On the other hand, since $\mu$ is harmonic, we get that
$$
\int_{X}  \big(D_s\bar\delta(\mathcal A) \big) (x)  d\mu(x)= \int_{X}  \bar\delta(\mathcal A)  (x)  d\mu(x),\qquad s>0.
$$
Combining  this  and  the last inequality  and  formula (\ref{eq_formulas_chi})   together, it follows that  
 $\chi_1\leq  \bar\chi_{\max}(\mathcal A).$

Now  we turn to the proof  of  $
 \underline\chi_{\min}\leq \chi_m\leq  \bar\chi_{\min}.$
 Recall from \cite{Ruelle} that the Lyapunov exponents of  the cocycle $\mathcal A^{*-1}$
 are  $-\chi_1<\cdots <-\chi_m.$  Hence,  what has been done  before  shows that
    $\underline\chi_{\max} (\mathcal A^{*-1})\leq  -\chi_m\leq \bar\chi_{\max} (\mathcal A^{*-1}).$ This, coupled  with (\ref{eq_formulas_chi}),  completes the proof.
\hfill $\square$

\smallskip

\noindent{\bf Proof of Corollary \ref{cor_formulas_Lyapunov_exponents}.}

By  Corollary \ref{cor1_th_main_1} and  Corollary  \ref{cor2_th_main_1}, the  maximal  Lyapunov exponent $\chi_1(\mathcal A^{\wedge k})$
of the cocycle $\mathcal A^{\wedge k}$ is  equal to the sum
$
\sum_{i=1}^k\chi'_i.
$
On the  other hand,  by Theorem  \ref{th_main_2}, $\underline\chi_{\max}(\mathcal A^{\wedge k})\leq \chi_1(\mathcal A^{\wedge k})\leq \bar\chi_{\max}(\mathcal A^{\wedge k}).$ This, combined with the last  equality, completes the  proof.
\hfill $\square$

\smallskip

\noindent{\bf Proof of Corollary \ref{cor_th_main_2}.}

We apply  Theorem \ref{th_main_2} to   the  holonomy  cocycle of the  foliation $(X,\Lc).$  Since we know  by hypothesis that this cocycle  admits  $d$ distinct  Lyapunov 
 exponents with respect to $\mu,$ it follows that the integer $m$  given by Theorem \ref{th_main_2}  coincides with $d.$ Hence,  
in the Oseledec decomposition\index{Oseledec!$\thicksim$ decomposition}
in assertion (i) of this theorem, we have that  $\dim H_i(x)=1,$
$1\leq i\leq d$   for every $x$ in a leafwise saturated  Borel set $Y\subset X$ of full  $\mu$-measure. Clearly, for such a point $x$ the leaf $L_x$  is  holonomy  invariant.
\hfill $\square$


%% file: Appendix.tex

\appendix 
\addcontentsline{toc}{chapter} {APPENDICES}

 \chapter[Measure  Theory] {Measure  theory for   sample-path spaces}
\label{section_Appendix}


Let  $(X,\Lc,g)$ be  a Riemannian measurable lamination.    We  first develop the   measure  theory on the sample-path space $\Omega(X,\Lc)$ endowed
with  the $\sigma$-algebra $\Ac$ (introduced in Section \ref{subsection_Wiener_measures_with_holonomy}) and the extended  sample-path space
$\widehat\Omega(X,\Lc)$ endowed
with  the $\sigma$-algebra $\widehat\Ac$ (introduced in Section  \ref{subsection_extended_sample_path_spaces}).  Next, we prove  
    Theorem \ref{prop_Wiener_measure}, Theorem   \ref{thm_Wiener_measure_measurable}, 
Proposition    \ref{prop_algebras},   
    Proposition \ref{prop_Sigma_k_is_a_fibered_lamination} and Theorem   \ref{thm_Brownian_motions_new}.  
Finally, we  show  that  cylinder laminations  are all  continuous-like. 
These results have been stated  and used    in the previous  chapters.
The present chapter is  divided  into  several sections. 
Note that  the study of the $\sigma$-algebra $\widetilde \Ac$ (defined in Section  \ref{subsection_Brownian_motion_without_holonomy}) in the   context 
of Riemannian continuous  laminations is
thoroughly  investigated  in the  works \cite{CandelConlon1,CandelConlon2, Candel2}.
 The  main difference between  the measure theory  with $\widetilde \Ac$ and that with  $ \Ac$
is that  the holonomy phenomenon  plays a vital role in the latter context but not in the former one.
In what follows,  $\sqcup$ denotes the disjoint union.

\section{Multifunctions and  measurable selections}
\label{subsection_multifunctions}
 We start   the first part of this  chapter with a  review  on   the theory of measurable   multifunctions as  presented
in the lecture notes by  Castaing and Valadier in \cite{CastaingValadier}. %
\index{Castaing}%
\index{Valadier}%
Consider    a  measurable space  $(T,\Tc)$ and a separable locally complete   metric  space $S.$
A multifunction $\Gamma$\index{multifunction} from $T$ to  $S$   associates to each $t\in T$  a  nonempty subset $\Gamma(t)\subset S.$
 The graph of such  a multifunction $\Gamma$ is
$
G(\Gamma):=\lbrace  (t,s)\in T\times S:\ s\in\Gamma(t)\}.$
We say that  a multifunction $\Gamma$ is {\it  measurable} (with respect to the $\sigma$-algebra $\Tc$)
\index{multifunction!measurable $\thicksim$ w.r.t. a $\sigma$-algebra}  if  its  graph 
$
G(\Gamma)$ belongs to the product of two $\sigma$-algebras  $\Tc\otimes \Bc(S).$  
Here  $\Bc(S)$ denotes, as  usual, the Borel $\sigma$-algebra of $S.$

An important problem in the  theory of    multifunctions  is   to  prove the  existence  of a {\it measurable  selection}  of $\Gamma:$
 a {\it selection}\index{selection!measurable $\thicksim$} is  a map $f:\ T\to S$  such that  $f(t)\in\Gamma(t)$ for all $t\in T.$
 In this Memoir,  we are only concerned with a  weaker notion of  measurable  selections. 
Roughly speaking,
let $\mu$ be a finite  positive (but not necessarily complete) measure  defined on $(T,\Tc).$  
Then  we allow the selection not to be  defined on  a negligible set with respect to $\mu.$
More precisely, we have the following 

\begin{definition}\label{D:selection}
We say
that   a map $f:\  T'\to S$  is  a 
{\it measurable selection $\mu$-almost everywhere}\index{selection!measurable $\thicksim$ $\mu$-almost everywhere} of $\Gamma$ if
the following three conditions are fulfilled:
 
$\bullet $ $T'\in \Tc$ and $\mu(T\setminus T')=0;$

$\bullet$  $f:\  T'\to S$  is measurable, that is,  the measurability  is understood when $T'$ is  endowed with the $\sigma$-algebra $\Tc',$
which is the trace  of $\Tc$ on $T';$ 

 $\bullet$   $f(t)\in\Gamma(t)$ for all $t\in T'.$
 \end{definition}

  The following two results from   Chapter III, Section 0, 1 and 2  in \cite{CastaingValadier}   plays  an important  role in this  work.

\begin{theorem} \label{T:measurable_selection}  {\rm (see \cite[Theorem III.6]{CastaingValadier})} Let $(T,\Tc)$  be a measurable space\index{space!measurable $\thicksim$} and $S$ a separable    complete metric space\index{space!complete metric $\thicksim$}, and $\Gamma$
 be a   multifunction  map $T$ to non empty closed subsets of $S.$ 
 If for each  open set  $U$ in $S,$  $\Gamma^-(U):=\{t\in T:\  \Gamma(t)\cap U\not=\varnothing\}$ belongs to $\Tc,$
 then  $\Gamma$ admits a measurable selection.
 \end{theorem}

\begin{theorem} \label{thm_measurable_selection} Let $(T,\Tc,\mu)$ be a     finite positive (but not necessarily complete) measure  space, and $S$ a separable    complete metric space, and $\Gamma$
  a  multifunction  map $T$ to non empty closed subsets of $S$  which is  measurable with respect to the $\mu$-completion $\Tc_\mu$ of $\Tc.$ 
 Then  $\Gamma$ admits a measurable selection  $\mu$-almost everywhere.
 \end{theorem}
We also have  the following  measurable projection theorem.
 \begin{theorem}\label{thm_measurable_projection} {\rm (see \cite[p. 75]{CastaingValadier})}
 Let  $(X,\Sc,\mu)$ be  a  complete finite positive  measure space and  let $Y$ be a complete  separable metric space
endowed with  the Borel $\sigma$-algebra $\Bc(Y).$ Consider  the  natural projection  of $X\times Y$ onto $X.$ Then for every  $Z\in\Sc\otimes\Bc(Y),$  the  projection of $Z$ onto 
$X$ is  in $\Sc.$  
\end{theorem}  
The following refined version of  Theorem \ref{thm_measurable_projection}  is  sometimes useful.
 \begin{theorem}\label{thm_measurable_projection_new}  
 Let  $(X,\Sc,\mu)$ be  a  complete finite positive  measure space, and  let $Y$ (resp. $Z$) be a complete  separable metric space
endowed with  the Borel $\sigma$-algebra $\Bc(Y)$ (resp. $\Bc(Y)$). Consider  the  natural projection  of $X\times Y\times Z$ onto $X\times Y.$ Then
 for every  $W\in\Sc\otimes\Bc(Y)\otimes\Bc(Z),$  
there is a set $S\subset X$ of full $\mu$-measure such that
$W'\cap (S\times Y)$ is  in $\Sc\otimes\Bc(Y),$
where $W'$ is the  projection of $W$ onto 
$X\times Y$   
\end{theorem}  
\begin{proof}
Consider  the function $\nu:\ \Sc\otimes \Bc(Y)\to [0,\infty)$ defined by
$$
\nu(A):=\mu(\pi(A)),\qquad A\in \Sc\otimes \Bc(Y),
$$
where $\pi(A)$ is  the projection of $A\subset X\times Y$ onto $X.$
Note   that by   Theorem \ref{thm_measurable_projection}, $\pi(A)\in \Sc.$
So the above  formula  is well-defined. Moreover,  it can be checked that $(X\times Y,\Sc\otimes \Bc(Y),\nu)$
is  a finite positive  measure space.

Also   by   Theorem \ref{thm_measurable_projection}, we know  that $W'$ is in $(\Sc\otimes\Bc(Y))_\nu \otimes \Bc(Z),$
where $(\Sc\otimes\Bc(Y))_\nu$ is the $\nu$-completion of the $\sigma$-algebra $\Sc\otimes\Bc(Y).$
Consequently,  we  may find  a $\nu$-negligible set $M\subset X\times Y$ such that $W'\setminus M$ is in  $\Sc\otimes\Bc(Y).$ 
Using  the above explicit  formula for $\nu,$ there is  a set $N\in \Sc$ of null $\mu$-measure  such that
$M\subset N\times Y.$
 Set $S:=X\setminus N\in \Sc. $  Clearly, $S$ is of full $\mu$-measure. Moreover, 
$$W'\cap (S\times Y)=(W'\setminus M) \cap (S\times Y)\in \Sc\otimes\Bc(Y),$$
as  desired.  
\end{proof}

When  $\Gamma$ is a  one-to-finite map, 
the  following 
 measurable selection theorem is very useful.
 \begin{theorem}\label{thm_measurable_selection_new}
 Let  $(T,\Tc,\mu)$ be a     finite positive (but not necessarily complete) measure  space, and $S$ a separable   complete metric space, and 
  $N\in\N\setminus\{0\}.$
Let $\Gamma$  be a    multifunction  from $T$ to nonempty finite  subsets of cardinal $\leq N$ of $S$
  which is  measurable with respect to the $\mu$-completion $\Tc_\mu$ of $\Tc.$ 
Then  $\Gamma$ admits  $n\leq N$    measurable maps $s_i:\  T_i\to S$  with $1\leq i\leq n$ satisfying the following properties:
\begin{itemize}
\item $T_i\in\Tc$ for $1\leq i\leq n;$ 
\item there is a set $T_0\in\Tc$ of  full $\mu$-measure  such that
 $ \Gamma(t)=\{  s_i(t):\  t\in T_i\ \text{and}\  1\leq i\leq n\}$ for all $t\in T_0;$   
\item  if $t\in T_i\cap T_j$ and $i\not=j,$ then $s_i(t)\not=s_j(t).$
\end{itemize}
 \end{theorem}
\begin{proof}  Suppose the theorem true for $N-1.$
Applying  Theorem \ref{thm_measurable_selection} yields a measurable selection $s_1:\ T_1\to S$ $\mu$-almost everywhere.
Consider  the one-to-finite map  $\Gamma'$ whose graph is $G(\Gamma'):= G(\Gamma)\setminus G(s_1).$
Let $T''_2$ be  the projection of $G(\Gamma')$ onto the $T$-component. By Theorem \ref{thm_measurable_projection},
there is  $T'_2\in \Tc$ such that $T'_2\subset T''_2$ and that $\mu(T''_2\setminus T'_2)=0.$
Consider   the multifunction $\Gamma_2$ defined on $T'_2$ whose  graph is $G(\Gamma_2):= (T'_2\times S)\cap G(\Gamma').$
Applying the  hypothesis of induction  to  the multifunction $\Gamma_2$ 
yields $(n-1)$   measurable maps $s_i:\  T_i\to S$  for $i=2,\ldots,n$  with the described properties.
So $s_1,\ldots,s_n$ are the desired maps, and the theorem  is proved for $N.$
\end{proof} 

The following result is  useful.
 \begin{lemma}   \label{lem_sup_measurable_book}
 Let  $(T,\Tc,\mu)$ be  a   finite positive (but not necessarily complete)   measure space, and $S$ a complete  separable  metric  space.
Let $\phi:\ T\times S\to\overline \R$ be  a  measurable  function.
Then there is  a  set $T'\in \Tc$ such that
$\mu(T\setminus T')=0$ and  that  the  function $T'\ni t\mapsto M(t):=\sup_{s\in S} \phi(t,s) $ is measurable.
 \end{lemma}
 \begin{proof}
 We only need to check that $\{ t\in T:\ M(t)>\alpha\}$ is  $\mu$-measurable  for all $\alpha\in \R.$
 Note that the last set is  equal to the image  under the projection onto $T$  of the  set
 $\{ (t,s)\in T\times S: \ \phi(t,s)>\alpha\}.$
 This  image is  $\mu$-measurable  by  Theorem \ref{thm_measurable_projection}. 
 \end{proof}

 We also need the following result.  
 
 \begin{theorem}\label{thm_selection_Grassmannian}
 Let  $(X,\Sc,\mu)$ be  a   finite positive (but not necessarily complete)  measure space, and let $x\mapsto V_x$ of $X$ into  $\Gr_k(\R^d)$ be  a  map, where  $0\leq k\leq d$ are   given  integers.
 Then  the following are equivalent:  
\begin{itemize}
\item[1)] There is  a set $X_1\in \Sc$ such that $\mu(X\setminus X_1)=0$ and that
   $X_1\ni x\mapsto V_x$ is a measurable map, 
 where   $\Gr_k(\R^d)$ is endowed with  the Borel $\sigma$-algebra $\Bc(\Gr_k(\R^d)).$
\item[2)] There is  a set $X_2\in \Sc$ such that $\mu(X\setminus X_2)=0$ and that there   are measurable maps $v_1,\ldots,v_k:\ X_2\to \R^d$ such that  for all $x\in X_2,$
$\{v_1(x),\ldots,v_k(x)\}$ is an orthonormal basis for $V_x.$
\item[3)] There is  a set $X_3\in \Sc$ such that $\mu(X\setminus X_3)=0$ and that there is a  bi-measurable  bijection $\Lambda$ from
$\{(x,v):\ x\in X_3,\ v\in V_x\} $ onto $X_3\times \R^k$  which is  a linear isometry\index{isometry, linear isometry}  on each fiber and  covers the  identity map of $X,$ that is, for  each $x\in X_3,$
$\Lambda(x,\cdot)$ is a  linear  isometry\index{isometry, linear isometry} from $V_x$ onto $\{x\}\times\R^k.$ 
\end{itemize}
\end{theorem}
\begin{proof}
The implications 2)$\Rightarrow$3) and 3)$\Rightarrow$1)  are easy.
So we  only need to establish the implication 1)$\Rightarrow$2).

For $k=1,$ this  implication is  trivial.

Suppose  the implication true for $k-1$ with some $k\geq 2.$ We need to prove it for $k.$
Applying   Theorem  \ref{thm_measurable_selection}, we may choose   a
set $X'_2\in \Sc$ with $\mu(X\setminus X'_2)=0$ and 
measurable map $v_k:\ X'_2\to \R^d$ such that  for all $x\in X'_2,$
$v_k(x)\in V_x.$
For each $x\in X,$ let $W_x$ be the  orthogonal component to $V_x$ in $\R^d$ with respect to the Euclidean inner product of $\R^d.$
By the hypothesis of induction, we may find  a
set $X_2\subset X'_2$ with $X_2\in \Sc$ and $\mu(X'_2\setminus X_2)=0$ and measurable maps $v_1,\ldots,v_{k-1}:\ X_2\to \R^d$ such that  for all $x\in X_2,$
$\{v_1(x),\ldots,v_{k-1}(x)\}$ is an orthonormal basis for $W_x.$
 This completes the proof of the implication for $k.$
\end{proof}

\section[$\sigma$-algebras]{$\sigma$-algebras:  approximations and measurability}
\label{subsection_approximations_and_measurability}

  In this  section we recall  some results of the measure theory   and prove  some new  ones. The  main reference is  Dudley's book \cite{Dudley}. 
 A {\it simple function} on  a measurable  space  $(S,\Sc)$ is any finite  sum 
$$
f:=\sum  a_i\otextbf_{A_i},\qquad  \text{where}\  a_i\in\R,\ A_i\in \Sc.
$$
The following result is  elementary.
\begin{proposition}\label{prop_simple_functions}
Let   $f$ be a  measurable   bounded function on $(S,\Sc).$ Then there exists   two sequences  of simple  functions $(g_n)_{n=1}^\infty$ and  $(h_n)_{n=1}^\infty$  such that
$g_n\searrow  f$ and  $h_n\nearrow  f$  as $n\to\infty.$
\end{proposition}


\begin{definition}\rm 
Let $(S,\Sc,\nu)$ be   a  positive measure  space and  $\Bc$ a  family of elements of $\Sc.$
We say  that  a  subset $A\subset S$ is  {\it  approximable}  by $\Bc$ if 
 there exists a  sequence
$(A_n)_{n=1}^\infty$ of subsets of $S$  such that
\\
$\bullet$ each $A_n$ is a countable  union of  elements of $\Bc$  and that
  $A\subset A_n$ and   $\nu(A_n\setminus A)\to 0$ as $n\to\infty.$ 
\\
$\bullet$  the sequence $(A_n)_{n=1}^\infty$  is  decreasing, i.e., $A_{n+1}\subset A_n$ for all $n.$

  We  say  that  a  family  $\Dc$ of elements of $\Sc$ is   approximable  by  $\Bc$ if each
  element of $\Dc$ is    approximable  by  $\Bc.$
\end{definition}

\begin{proposition}\label{prop_measure_theory}
1) For any  set $Z$ and  algebra $\Bc$ of subsets of $Z,$ any  countably additive  functions $\nu$ from $\Bc$ into  $[0,\infty]$  extends
to  a measure (still denoted by $\nu$) on the $\sigma$-algebra $\Sc$  generated by $\Bc.$ More explicitly, $\nu$ is  defined  as follows:
$$
\nu(A):=\inf \left\lbrace \sum_{i=1}^\infty \nu(B_i):\  B_i\in  \Bc,\  A\subset \bigcup_{i=1}^\infty B_i  \right\rbrace, \qquad A\in \Sc.
$$
2)
If, moreover, $\nu$  is  $\sigma$-finite,
then  $\Sc$ is approximable  by  $\Bc.$ 
\end{proposition}
\begin{proof}
Part 1)  follows from Theorem 3.1.4  in \cite{Dudley}.  

Part 2)  in the case  when $\nu$ is   finite follows immediately  from the  formula in Part 1).

When  $\nu$ is  $\sigma$-finite, we fix   a sequence $(Z_m)_{m=1}^\infty\subset\Sc$  such that  $Z=\cup_{m=1}^\infty Z_n$ and $\nu(Z_m)<\infty$ for each $m.$
Next, we apply  the  previous  case  to   each set $ Z_m$ in order  to obtain  a  decreasing sequence
$(A_{mn})_{n=1}^\infty$ of subsets of $Z_m$  such that each $A_{nm}$ is a countable  union of  elements of $\Bc$  and that
  $A\cap Z_m\subset A_{mn}$ and   $\nu(A_{mn}\setminus( A\cap Z_m))<2^{-(n+m)}.$ 
  Now  letting $A_n:=\bigcup_{m=1}^\infty A_{mn}$ for each $n\geq 1,$ we see easily that
  the sequence $(A_n)_{n=1}^\infty$  satisfies  the  desired conclusion.
 \end{proof}
 The following criterion is  very useful. 
 \begin{proposition}\label{prop_approximation_measure_theory}
 Let $(S,\Sc,\nu)$ be   a  $\sigma$-finite positive measure  space and  $\Dc^0$ and $\Bc$ two  families of elements of $\Sc.$
 Assume  in addition that the intersection of  two  sets of $\Bc$ may be represented as  a  countable union of sets in $\Bc$
and  that  $\Dc^0$ is   approximable  by $\Bc.$ 
Starting from $\Dc^0,$ we define  inductively   the sequence of  families $(\Dc^N)_{N=1}^\infty$  and  a new family $\Dc$ of   elements of $\Sc$  as follows:
 \begin{equation*}
\begin{split}                
 \Dc^{2k+1}&:=\left  \lbrace   A\in \Sc:\  A=\cup_{n=1}^\infty A_n,\qquad    A_n\in \Dc^{2k}    \  \right\rbrace,\qquad k\in\N;\\
 \Dc^{2k}&:=\left  \lbrace   A\in \Sc:\  A=\cap_{n=1}^\infty A_n,\   A_n\in \Dc^{2k-1}\ \text{and}\  A_{n+1}\subset A_n    \right\rbrace,\ k\in\N;\\
  \Dc&:=\bigcup_{k=1}^\infty \Dc^k.
 \end{split}
 \end{equation*}
  Then  the $\sigma$-algebra  generated  by all elements of 
 $\Dc$ is  also   approximable  by $\Bc.$
\end{proposition}
\begin{proof}
We prove  by induction on    $N$   that  $\Dc_N$ is approximable by $\Bc.$
By the assumption,  $\Dc^N$ is  approximable  by $\Bc$ for $N=0.$
Suppose   that   $\Dc^{N-1}$ is  approximable  by $\Bc$ for some $N\geq 1.$
Let $A\in \Dc^N.$ We need  to show that $A$ is  approximable by $\Bc.$
There are two cases.
\\
{\bf Case  $N=2k+1:$}
  
Since $A$ can be  represented as  $A=\cup_{n=1}^\infty A_n$ with    $A_n\in \Dc^{N-1},$
we obtain, by the  hypothesis of induction, for each $n,$ 
   a   decreasing  sequence
$(A_{mn})_{n=1}^\infty$    such that $A_n\subset A_{nm}$ and  that each $A_{nm}$ is a countable  union of  elements of $\Bc$  and that
     $\nu(A_{mn}\setminus A_n)<2^{-(n+m)}.$ 
  Now  letting $B_m:=\bigcup_{n=1}^\infty A_{mn},$ we see easily that
  $(B_m)_{m=1}^\infty$ is  decreasing, $A\subset B_m$ and     $\nu(B_m\setminus A)<2^{-m}$ as  
   desired.
\\
{\bf Case  $N=2k:$}
 
 Since $A$ can be  represented as  
$A=\cap_{n=1}^\infty A_n$ with   $ A_n\in \Dc^{N-1}$ and $  A_{n+1}\subset A_n,$ 
 we obtain, by the  hypothesis of induction, for each $n,$ 
   a    sequence
$(B_n)_{n=1}^\infty$    such that each $B_n$ is a countable  union of  elements of $\Bc$  and that
  $A_n\subset B_n$ and   $\nu(B_n\setminus A_n)<2^{-n}.$  Replacing  each  $B_n$ with $B_n\cap B_{n-1}\cap\cdots\cap  B_1$ and using the assumption that
the intersection of  two  sets of $\Bc$ may be represented as  a  countable union of sets in $\Bc,$   
we may assume  in addition  that $B_{n+1}\subset B_n.$
 Since  $\nu(A_n\setminus A)\to 0,$
it follows  that  $\nu(B_n\setminus A)\to 0.$ This  completes the proof in the last case.
\end{proof}

 \begin{proposition}\label{prop_criterion_sigma_algebra}
 Let $\Bc$ be an  algebra  of subsets of  a  given  set $Z.$
 Let $\Ac$ be  a  family of  subsets of  $Z$ such that
 \\
 $\bullet$ $ \Bc\subset \Ac;$
 \\
 $\bullet$  if $(A_n)_{n=1}^\infty\subset \Ac$ such that $A_n\subset A_{n+1}$ for all $n,$ then
 $\cup_{n=1}^\infty A_n\in\Ac;$
 \\
 $\bullet$  if $(A_n)_{n=1}^\infty\subset \Ac$ such that $A_{n+1}\subset A_n$ for all $n,$ then
 $\cap_{n=1}^\infty A_n\in\Ac.$
 
 Then $\Ac$ contains the $\sigma$-algebra  generated  by $\Bc.$
 \end{proposition}
 \begin{proof}
  For a collection $T$ of subsets of $Z,$   let
\\
 $\bullet$   $T_\sigma$  be all countable unions of elements of $T,$
\\
$\bullet$    $ T_\delta$ be all countable intersections of elements of $T,$
  \\
$\bullet$   $T_{\delta\sigma}:=(T_{\delta})_{\sigma}.$

Now define by transfinite induction a sequence $\Bc^m,$ where $m$ is an ordinal number, in the following manner:
\\
  $\bullet$  For the base case of the definition, let $\Bc^0:= \Bc.$ 
 \\
$\bullet$    If $i$ is not a limit ordinal, then $i$ has an immediately preceding ordinal $i - 1.$ Let
$  \Bc^i := [\Bc^{i-1}]_{\delta \sigma}.$
\\
$\bullet$
    If $i$ is a limit ordinal, set
$\Bc^i = \bigcup_{j < i} \Bc^j.$ 

Then  we can show that  $\mathfrak{B}:=  \Bc^{\omega_1}$ is  the $\sigma$-algebra  generated by $\Bc,$ where
$\omega_1$   is the first uncountable ordinal number.
We can prove by transfinite induction on the ordinal number $i$ that
each element $ A$ of $\Bc^i$ can be written as $A=\cup_{n=1}^\infty A_n$ with $A_n\nearrow A$ as $n\nearrow\infty$ and  each $A_n$
is of the form $A_n=\cap_{m=1}^\infty A_{nm}$ with $A_{nm}\searrow A_n$    as $m\nearrow\infty,$ and $A_{nm}\in\Bc^{i-1}.$ 
Summarizing  what has been done so far, we have shown that if $\Bc^{i-1}\subset \Ac$ then  $\Bc^{i}\subset \Ac$
for all $i<\omega_1.$ Hence $\mathfrak{B}\subset \Ac,$ and 
the   proof is  thereby completed.
 \end{proof}
 
 \begin{proposition}\label{prop_integral_dependance_measurably_on_parameter} 
 Let $(S,\Sc,\mu)$ be a  positive  finite measure space  and $(T,\Tc)$ a  measurable space.
 Let $f:\ S\times T\to \R$ be a  measurable bounded function, where 
  $S\times T$ is  endowed with the $\sigma$-algebra $\Sc\otimes\Tc.$
  Consider the  function  $F=\Phi(f):\  T\to \R$  defined by
  $$
  F(t):=\int_S  f(s,t) d\mu(s),\qquad  t\in T.
  $$
 Then  $F$ is measurable.
 \end{proposition}
 \begin{proof}
Since  $f$ is  measurable and  bounded, Proposition \ref{prop_simple_functions} yields a  sequence  of simple  functions $(f_n)_{n=1}^\infty$ such that
$f_n\searrow  f.$ 
By Lebesgue dominated  convergence,\index{Lebesgue!$\thicksim$ dominated convergence theorem}\index{theorem!Lebesgue dominated convergence $\thicksim$} we get that $\Phi(f_n)\searrow \Phi(f)=F.$ Consequently,
if all $\Phi(f_n)$ are measurable, so is  $F.$ Therefore,
 we may assume without loss of generality that
$f$ is  a simple  function, that is,
$$
f:=\sum  a_i\otextbf_{A_i},\qquad  \text{where}\  a_i\in\R,\ A_i\in \Sc\otimes\Tc.
$$
This  implies that
$$
F=\Phi(f)= \sum  a_i\Phi( \otextbf_{A_i}).
$$
Hence,  we are reduced to  the case  where $f:=\otextbf_A$ with  $A\in\Sc\otimes\Tc.$

To prove  the last assertion 
let $\mathfrak{A}$ be the family of   all sets  $A=\cup_{i\in I} S_i\times T_i,$ where $S_i\in \Sc$ and $T_i\in \Tc,$
and the index set $I$ is  finite. Note that
$\mathfrak{A}$ is an algebra on $S\times T$ which generates the  $\sigma$-algebra $\Sc\otimes\Tc.$
Moreover, each   such set $A$ can be  expressed  as  a disjoint finite union
$A=\sqcup_{i\in I} S_i\times T_i.$
 Using  the above  expression  for  such a set $A$ and  the  equality $f=\otextbf_A,$ we infer that 
 $$
  F(t)=\Phi(f)=\sum_{i\in I} \mu(S_i)\otextbf_{T_i}(t),\qquad  t\in T.
  $$
Hence, $F$ is  measurable for all $A\in \mathfrak{A}.$

Let $\mathcal A$ be  the  family  of  all sets $A\subset  S\times T$ such that  $S\ni s\mapsto \otextbf_A(s,t)$
is  measurable  for all $ t\in T$ and that
$\Phi(\otextbf_A)$ is    measurable.
 The  previous paragraph shows that $ \mathfrak{A}\subset\mathcal A.$

 Next, suppose that  $(A_n)_{n=1}^\infty\subset \mathcal A$  
and that   either $A_n\searrow A$ or $A_n\nearrow A.$  Let $f:=\otextbf_A$ and $F:=\Phi(f).$   
By Lebesgue dominated  convergence\index{Lebesgue!$\thicksim$ dominated convergence theorem}\index{theorem!Lebesgue dominated convergence $\thicksim$}, we get that either $\Phi(f_n)\searrow F$ or $\Phi(f_n)\nearrow F.$ So $F$ is  also measurable.
Hence, $A\in \mathcal A.$
Consequently,  by Proposition \ref{prop_criterion_sigma_algebra},
$\Sc\otimes\Tc\subset\mathcal A.$
In particular, $F$ is  well-defined and  measurable  for  $f:=\otextbf_A$ with $A\in\Sc\otimes\Tc.$
This completes the proof.
\end{proof}

\section{$\sigma$-algebra $\Ac$ on a leaf}
\label{subsection_algebra_on_a_leaf}
The main purpose of this  section is  to provide  the necessary material in order to prove Theorem \ref{prop_Wiener_measure} (i) and Proposition   \ref{prop_algebras} (i).
More precisely, this  section is devoted to
 the measure theory on sample-path spaces associated to  a single leaf.
  For this purpose 
 we need  to introduce some notation and  terminology as  well as  some preparatory results.

Fix  a  point $x\in X$ and let $L:= L_x$ and $\pi:\ \widetilde L\to L$ the  universal covering projection. 
Recall  from  Section \ref{subsection_measurability_issue}  that  a  set $A\subset \Omega(L)$ is  said to be a  cylinder image if  $A=\pi\circ \tilde A$ for  some cylinder set $\tilde A\subset \Omega(\widetilde L).$ 
 Let $\Dc^1(L)$ denote  following family of subsets of $\Omega(L):$
 $$ \Dc^1(L):=\left  \lbrace   A\in \Omega(L):\      A=\cup_{n=1}^\infty A_n,\  \text{$A_n$ is a cylinder image}    \  \right\rbrace.$$
 Starting from $\Dc^1(L),$ we define  inductively   the sequence of  families $(\Dc^N(L))_{N=1}^\infty$  and  a new family $\Dc(L)$ of  subsets of $\Omega(L)$  as follows:
 \begin{equation}\label{eq_algebra_Dc_L}
\begin{split}                
 \Dc^{2k}(L)&:=\left  \lbrace   A\in \Omega(L):\  A=\cap_{n=1}^\infty A_n,\   A_n\in \Dc^{2k-1}(L)\ \text{and}\  A_{n+1}\subset A_n    \right\rbrace,\ k\in\N;\\
 \Dc^{2k+1}(L)&:=\left  \lbrace   A\in \Omega(L):\  A=\cup_{n=1}^\infty A_n,\qquad    A_n\in \Dc^{2k}(L)    \  \right\rbrace,\qquad k\in\N;\\
 \Dc(L)&:=\bigcup_{k=1}^\infty \Dc^k(L).
 \end{split}
 \end{equation}
 Note that   $(\Dc^N(L))_{N=1}^\infty$  is  increasing,   that is, $\Dc^N(L)\subset \Dc^{N+1}(L).$  
 The following result will be   the main ingredient  in the proof of assertion (i) of both Theorem \ref{prop_Wiener_measure} and    Proposition  \ref{prop_algebras}.
 \begin{proposition}\label{prop_algebras_leaf}
 $   \Dc(L) $  is an algebra.
 \end{proposition}
 Prior to the proof of Proposition  \ref{prop_algebras_leaf} we need  to introduce some notion.
 A connected open set $U\subset L$ is  said to be  {\it trivializing} if we write $\pi^{-1}(U)$ as the disjoint union of its connected  components
$\tilde U_i$ then every restriction $\pi|_{\tilde U_i}:\ \tilde U_i\to U_i $ is   homeomorphic.
Clearly,   every simply connected  domain $U\subset L$    is trivializing. 
  A Borel set $\tilde A\subset\widetilde L$ is  said to be  {\it good} if there is an open neighborhood $\tilde U$ of
$\tilde A$  in  $\widetilde L$ such that $\pi$ maps $\tilde U$ homeomorphically onto  a trivializing open set in $L.$
 A cylinder  set $\tilde C:=C(\{t_i,\tilde B_i\})$ in  $\Omega(\widetilde L)$ is  said to be  {\it good}
if 
all  (Borel) sets  $\tilde B_i$  are  good.
A set  $C\subset \Omega(L)$ is  said to be  a {\it good cylinder image} if  there is  a good  cylinder set $\tilde C $
such that $C=\pi\circ\tilde C.$ 
Let us  point out the following remarkable {\it lifting property} of good  cylinder sets.
\begin{lemma}\label{lem_lifting_property}
If $\tilde C:=C(\{t_i,\tilde B_i\})$ is a  good cylinder set in  $\Omega(\widetilde L),$
then 
$$\pi^{-1}(\pi\circ \tilde  C)=\bigsqcup_{\gamma\in\pi_1(L)}C(\{t_i,\gamma (\tilde B_i)\}).$$
\end{lemma}
\begin{proof} If $U$ is a trivializing set, then we have that
 \begin{equation}\label{eq_trivializing_leaf}
\pi^{-1}(U)=  \bigsqcup_{\gamma\in\pi_1(L)} \gamma (U) 
\end{equation}  and that each  restriction $\pi|_{\gamma (U)}:\ \gamma(U)\to U $ is   homeomorphic.
 Using this  property the lemma  follows.
\end{proof}

The properties of various notion of goodness are collected  in the following 
\begin{lemma}\label{lem_algebras_leaf}
(i) If $A$ and $B $  are good cylinder images, then   $A\cap B$ is    a countable union of good cylinder images. \\
 (ii) If $A$ and $B $  are good cylinder images, then   $A\setminus B$ is    a countable union of mutually disjoint  good cylinder images. \\
(iii)  Every cylinder image is  a countable union of  good cylinder images.\\
(iv) The intersection of two   cylinder images is  a countable union of   cylinder images. 
\\
(v) If $A$ is   a   good cylinder image, then  $\Omega(L)\setminus A$ is a  
countable  union of cylinder images.
 \end{lemma}
 Taking for granted  the lemma,
   we arrive at  the

\smallskip  
\noindent{\bf End of the proof of Proposition   \ref{prop_algebras_leaf}.}
Consider  the   family
$\Dc^0(L)$ of all good cylinder images and  recall  from (\ref{eq_algebra_Dc_L}) the  sequence of families $(\Dc^L(L))_{N=1}^\infty.$  
First note that   $\Dc^k(L)\subset \Dc^{k+1}(L)$ for all $k\in \N.$
Using this  increasing property, we  deduce   that to prove  the proposition  it is  sufficient to show  that:

\noindent{\bf Claim.} {\it If $A, B\in \Dc^N(L)$ for some $N\in \N,$ then  $A\cup B\in\Dc^N(L)$ and  $\Omega(L)\setminus A\in \Dc^{N+1}(L).$}
 
For $N=0,$   Claim follows  from  Lemma \ref{lem_algebras_leaf} (v).
Suppose  Claim  true  for $N-1$ we  need to show  it true  for $N.$ To do this 
fix  two sets  $A,$ $B\in \Dc^N(L).$ Consider two cases.

\noindent{\bf Case 1: $N$ is  even:}

 Let $(A_n)_{n=1}^\infty$ and $(B_n)_{n=1}^\infty$ be two decreasing sequences of elements in $\Dc^{N-1}(L)$ such that
 $A=\cap_{n=1}^\infty A_n$ and $B=\cap_{n=1}^\infty B_n.$ Note from (\ref{eq_algebra_Dc_L}) that
each $A_n$ (resp. $B_n$) is a  countable  union of elements in $\Dc^{N-2}(L).$   
 Clearly,  $(A_n\cup B_n)_{n=1}^\infty$ is a  decreasing sequences of elements in $\Dc^{N-1}(L)$ such that $A\cup B=\cap_{n=1}^\infty (A_n\cup B_n).$
 Hence, $A\cup B\in \Dc^N(L)\subset \Dc^{N+1}(L).$
 
 To prove  that  $\Omega(L)\setminus A\in \Dc^{N+1}(L)$ we  write
$
\Omega(L)\setminus A=\bigcup_{n=1}^\infty \big (\Omega(L)\setminus   A_n\big ).
$
  By the  hypothesis of induction,  each $\Omega(L)\setminus  A_n$ is an element in $\Dc^N(L).$
We deduce   from the last equality and  from (\ref{eq_algebra_Dc_L}) and from  the fact that $N$ is  even  that  $
\Omega(L)\setminus A\in \Dc^{N+1}(L),$  as  desired.

\noindent{\bf Case 2: $N$ is  odd:}

It follows  from   (\ref{eq_algebra_Dc_L}) and  the oddness of $N$ that a finite union of elements in $\Dc^{N}$ is also an element   in $\Dc^{N}.$
To complete  the proof we need to show  that  $\Omega(L)\setminus A\in \Dc^{N+1}(L)$
for each $A=\cup_{n=1}^\infty A_n$ with $A_n\in \Dc^{N-1}.$
    Write
$$
\Omega(L)\setminus A=\bigcap_{n=1}^\infty \Big ( \Omega(L)\setminus  \bigcup_{i=1}^n A_i\Big)
$$
 By the  hypothesis of induction,  each $\Omega(L)\setminus  \bigcup_{i=1}^n A_i $ is an element in $\Dc^N(L).$
 Since  $\Omega(L)\setminus  \bigcup_{i=1}^n A_i $ is  decreasing on $n,$ it follows that $
\Omega(L)\setminus A\in \Dc^{N+1}(L),$  as  desired.
\hfill $\square$
 
 It remains to us  to  establish   Lemma  \ref{lem_algebras_leaf}.
 \\ 
 \noindent {\bf Proof of  assertion (i) of Lemma  \ref{lem_algebras_leaf}.}
 Let  $A:=\pi\circ\tilde A,$  $B =\pi\circ\tilde B,$ where  $\tilde A:= C(\{t_i,\tilde A_i\}:p) $ and   $\tilde B:= C(\{s_j,\tilde B_j\}:q), $
and  $\tilde A_i,$ $\tilde B_j$  are good subsets of $\widetilde L,$ and  $0\leq  t_1<t_2<\cdots <t_p$  and 
 $0\leq  s_1<s_2<\cdots <s_q$
are  sets of increasing times.
  The proof    is  divided into two cases.  

\noindent  {\bf Case  1:} {\it The two sets of times  are equal, i.e., $p=q$ and  $t_i=s_i$ for $1\leq i\leq p.$} 

If there is  some  $i$ such that $\pi(\tilde A_i)\cap \pi(\tilde B_i)=\varnothing,$ then  
$A\cap B=\varnothing$ because  $\omega\in A\cap B$ implies  $\omega(t_i)\in \pi(\tilde A_i)\cap \pi(\tilde B_i).$
If this  case  happens, there is  nothing  to prove. Therefore, we  may assume that 
  $D_i:=\pi(\tilde A_i)\cap \pi(\tilde B_i)\not=\varnothing$ for every $1\leq i\leq p.$ Moreover,  using that
$\tilde A_i$  and $\tilde  B_i$ are good and  replacing  $\tilde A_i$ (resp. $\tilde  B_i$) by
  $\tilde A_i\cap \pi^{-1}(D_i)$ (resp. $\tilde  B_i\cap \pi^{-1}(D_i)$), we may assume without loss of generality that
  $\pi(\tilde A_i)=  \pi(\tilde B_i)=D_i$ for  $1\leq i\leq p.$
 Using  the   goodness assumption of $\tilde A_i$ (resp. $\tilde B_i$)    we may  find,
 for each $1\leq i\leq m,$    open sets $\tilde U_i,\ \tilde  V_i\subset  \widetilde L$  and a  trivializing 
open set $W_i\subset L$ such that
$\tilde A_i\subset  \tilde U_i,$  $\tilde B_i\subset  \tilde V_i,$ and  
$\pi(\tilde U_i)=\pi(\tilde  V_i)=W_i.$
Fix  a point $c\in D_1$ and let $\tilde a$ (resp. $\tilde b$) be the unique point in $\pi^{-1}(c)$ lying on $\tilde A_1$ (resp. $\tilde B_1$). 
Let $\gamma\in \pi_1(L)$ be the  unique  deck-transformation sending $\tilde b$ to $\tilde a.$ 
By  shrinking $\tilde U_1$ and $\tilde V_1$  if necessary, we may   assume  without loss of generality that 
 $\gamma(\tilde V_1)=\tilde  U_1.$
Setting 
$$\tilde C_i:=\gamma(\tilde  B_i),\ 1\leq i\leq p\quad\text{and}\quad \tilde C:=C(\{t_i,\tilde C_i\}:p),$$
 we obtain, using Lemma \ref{lem_lifting_property}, that
$ \pi\circ \tilde C=\pi\circ \tilde B=B.$  Now pick an arbitrary path
 $\omega\in \pi\circ  \tilde A \cap \pi\circ \tilde C.$ Let  $y:=\pi(\omega(t_1))\in D_1$ and $\tilde y_a= \pi^{-1}(y) \cap \tilde A_1$
and  $\tilde y_b= \pi^{-1}(y) \cap \tilde B_1.$  Clearly, by Lemma   \ref{lem_lifting_property} again,  $\pi^{-1}_{\tilde y_a}\omega\in   \tilde A$ and  $\pi^{-1}_{\tilde y_b}\omega\in   \tilde B.$ This implies that
$$\pi^{-1}_{\tilde y_a}\omega=\gamma(\pi^{-1}_{\tilde y_b}\omega)\in   \tilde A\cap  \tilde C.$$ 
Thus we have shown  that $ \pi\circ  \tilde A \cap \pi\circ \tilde C\subset  \pi \circ (\tilde A\cap \tilde C).$ Since the  inverse  inclusion is trivial,
we obtain  that  $ \pi\circ  \tilde A \cap \pi\circ \tilde C=  \pi \circ (\tilde A\cap \tilde C).$
  Hence,
 $$A\cap B=\pi\circ  \tilde A \cap \pi\circ \tilde C=\pi \circ (\tilde A\cap \tilde C) ,$$
  which finishes the proof because  $\tilde A\cap \tilde C$ is a  good cylinder set.
 \\
 \noindent  {\bf Case  2:}  {\it The general case.}
 
 Suppose  that assertion (i) is proved  when the  cardinal  of the  symmetric difference of  $\{ t_1,t_2,\ldots ,t_p\}$ and   
 $\{ s_1, s_2,\cdots ,s_q\}$ is  $\leq  r.$
 We will prove  by induction  that assertion (i) also  holds when the  cardinal of the above  symmetric difference $\leq  r+1.$
 To do this consider  the case where this  cardinal is  equal to $r+1.$ Pick $t_i$ which does not belong to $\{s_1,\ldots,s_q\}.$
Let $
U_i$ be a trivializing open neighborhood  of  $\pi(\tilde A_i).$  
 Write  $\pi^{-1}(U_i)$ as  the union of its connected  components $U_{ij},$ where  $j\in J$   and the index set $J$ is  at most  countable.
 Let $\tilde A_{ij}:=\pi^{-1}_{ij}( \pi(\tilde A_i))\subset  U_{ij},$ where $\pi^{-1}_{ij}$
is the inverse of the homeomorphism $\pi|_{U_{ij}}:\ U_{ij}\to  U_i.$ 
For $j\in J$  consider the following  good cylinder set $\tilde B'_j:=C(\{s_1,\tilde B_1\},\ldots, \{s_q,\tilde B_q\} ,\{t_i, \tilde A_{ij}\}:q+1). $
Since  the  cardinal  of the  symmetric difference of  $\{ t_1,t_2,\ldots ,t_p\}$ and   
 $\{ s_1, s_2,\cdots ,s_q, t_i\}$ is  $\leq  r,$ it follows  from the  hypothesis of induction that
 $\pi\circ\tilde A\cap \pi\circ\tilde B'_j$ is  a countable union of good cylinder images.
  This, combined  with the  equality
 $$ A\cap B=\pi\circ\tilde A\cap \pi\circ\tilde B  =\bigcup_{j\in J} \pi\circ\tilde A\cap \pi\circ\tilde B'_j,$$ 
 implies  Case 2, where  the  last equality follows  from  (\ref{eq_trivializing_leaf}).
 
\noindent {\bf Proof of  assertion (ii) of Lemma  \ref{lem_algebras_leaf}.} 
Using the notation introduced  in  the  proof of assertion (i),  we also consider  two cases.

\noindent  {\bf Case  1:} {\it The two sets of times  are equal, i.e., $p=q$ and  $t_i=s_i$ for $1\leq i\leq p.$} 

If  $A\cap B=\varnothing$  then  one  get  that $A\setminus  B=A.$ So assertion (ii) is  trivially true.
 If  $A\cap B\not=\varnothing$ we proceed  as  in Case  1 of the  proof of  assertion (i). Consequently, we  can show that
 $$ A\setminus  B=\pi\circ  \tilde A \setminus \pi\circ \tilde C=\pi \circ (\tilde A\setminus \tilde C). $$
 Since we may write  $\pi \circ (\tilde A\setminus \tilde C) $ as the union of $p$ mutually disjoint good cylinder images
$$ \bigsqcup_{i=1}^p \pi\circ C\Big (\{t_1,\tilde A_1\},\ldots,\{t_{i-1},\tilde A_{i-1}\},\{t_i,\tilde A_i\setminus \tilde C_i\},\{t_{i+1},\tilde A_{i+1}\},\ldots,   \{t_p,\tilde A_p\}    :p\Big),$$
the  desired conclusion follows.

\noindent  {\bf Case  2:}  {\it The general case.}
 
 Suppose  that assertion (ii) is proved  when the  cardinal  of the  symmetric difference of  $\{ t_1,t_2,\ldots ,t_p\}$ and   
 $\{ s_1, s_2,\cdots ,s_q\}$ is  $\leq  r.$
 We will prove  by induction  that assertion (ii) also  holds when the  cardinal of the above  symmetric difference $\leq  r+1.$
As  the  arguments are  quite similar to those given in Case  2 of  the proof of assertion (i),   a  detailed proof is   left to the  interested reader. 
 
\noindent {\bf Proof of  assertion (iii) of Lemma  \ref{lem_algebras_leaf}.}
 Let  $A=\pi\circ \tilde A,$ where  $\tilde A:=C(\{t_i,\tilde A_i\}:m)$   is a cylinder set in  $\Omega(\widetilde L).$ 
Since  for every $\tilde x\in \widetilde L$ there is an open neighborhood $\tilde U$ of $\tilde x$ such that $\pi(\tilde U)$ is trivializing, we may  write  each $\tilde A_i$ as a countable union of good  sets  $\tilde A_{ij}.$ Consequently, 
using that
$$
\tilde A=\bigcup_{j_1,\ldots,j_m\in\N}C(\{t_i,\tilde A_{ij_i}\}:m),
$$ 
 the assertion follows.

 \noindent {\bf Proof of  assertion (iv) of Lemma  \ref{lem_algebras_leaf}.}
  It follows from  combining  assertion (i) and  assertion (iii).
  
 \noindent {\bf Proof of  assertion (v) of Lemma  \ref{lem_algebras_leaf}.}
  Let  $A=\pi\circ \tilde A,$ where  $\tilde A:=C(\{t_i,\tilde A_i\}:m)$   is a good cylinder set in  $\Omega(\widetilde L).$  If $A=\varnothing$ then we write  $\Omega(L)$ as the image of the cylinder $C(\{0,\widetilde L\}),$ and hence   the  desired conclusion follows from   assertion (iii). Therefore, we may suppose without loss of generality that  
all $\tilde  A_i$ are nonempty. Let $
U_i$ be a trivializing open neighborhood  of  $\pi(\tilde A_i)$ for $1\leq i\leq m.$
 Write  $\pi^{-1}(U_i)$ as  the union of its connected  components $U_{ij},$ where  $j\in J$ and  the index  set $J$
is  at most countable.
 Let $\tilde A_{ij}:=\pi^{-1}_{ij}( \pi(\tilde A_i))\subset  U_{ij},$ where $\pi^{-1}_{ij}$
is the inverse of the homeomorphism $\pi|_{U_{ij}}:\ U_{ij}\to  U_i.$
We assume without  loss of generality that $\tilde A_i=\tilde A_{i0}$ for $1\leq i\leq m.$ Observe  that
$ \pi\circ C(\{t_i,\pi^{-1}(\pi(\tilde A_i))\}:m)$ is  the disjoint union  of (at most countable) cylinder images
$$
\pi\circ C\big (\{t_1,\tilde A_{10}\}, \{t_2,\tilde A_{2j_2}\},\ldots,  \{t_m,\tilde A_{mj_m}\}:m \big),\qquad  (j_2,\ldots,j_m)\in J^{m-1}.
$$
On the other hand,  assertion (iii) and the  following trivial  equality
$$
\Omega(L)\setminus \pi\circ C(\{t_i,\pi^{-1}(\pi(\tilde A_i))\}:m)=\bigcup_{i=1}^m \pi\circ  C(\{ t_i,\tilde B_i\}: 1),
$$
  where $\tilde B_i:=\widetilde L\setminus\pi^{-1}(   \pi(\tilde A_i)),$  implies  that the set on the  left hand  side  is  a  countable  union of good  cylinder images.
  This, combined  with  the  previous disjoint union, implies that
\begin{multline*}
\Omega(L)\setminus A=\bigcup_{i=1}^m \pi\circ  C(\{ t_i,\tilde B_i\}: 1)\\
\sqcup \bigsqcup_{ (j_2,\ldots,j_m)\in J^{m-1}\setminus {(0,\ldots,0)}}\pi\circ C\big (\{t_1,\tilde A_{10}\}, \{t_2,\tilde A_{2j_2}\},\ldots,  \{t_m,\tilde A_{mj_m}\}:m \big),
\end{multline*}
proving 
assertion (iv).
  \hfill
$\square$


Now we arrive  at the 
  
\smallskip

\noindent{\bf End of the proof of  Theorem \ref{prop_Wiener_measure} (i).} Fix a point $x\in X$ and let $L:=L_x,$ and
 $\Dc_x:=\{ A\in \Dc(L):\  A\subset \Omega_x\},$   where  $\Dc(L)$ is  given in (\ref{eq_algebra_Dc_L}). Since we know from
  Lemma  \ref{lem_algebras_leaf} that  $\Dc(L)$ is  an  algebra, so is $\Dc_x.$  
For every $\tilde x\in \pi^{-1}(x),$ we  define  a  probability measure $W_x^{\tilde x}$ on $(\Omega_x,\Ac_x)$
as follows:
$$
W_x^{\tilde x}(A):= W_{\tilde x}  (\pi^{-1}_{\tilde  x}A),  \qquad A\in \Ac_x,
$$
where $W_{\tilde x}$ is the probability measure  on $(\Omega(\widetilde L),\widetilde \Ac(\widetilde L))$
given    by (\ref{eq_defi_W_x}).    
 Next, we will prove that $W_x(A)$ given  by formula (\ref{eq_formula_W_x})  is  well-defined  for every $A\in \Dc_x.$  This
is  equivalent to    showing the  following 
\\
{\bf Claim. }{\it 
$ W^{\tilde x_1}_x  (A)=  W^{\tilde x_2}_x  (A)$ for all $A\in\Dc_x$ 
and all points  $\tilde x_1,$ $\tilde x_2\in \pi^{-1}(x),$ in other words,
 all $ W^{\tilde x}_x$      with $\tilde x\in \pi^{-1}(x)$ coincide on  $\Dc_x(L).$}
 
To do this   our idea is  to prove this  coincidence   on good cylinder images, then on  $\Dc^1(L),$  $\Dc^2(L),\ldots,$
and finally on $\Dc(L).$

First, consider the case where $A=\pi\circ \tilde A$ and $\tilde A$ is a good cylinder set, that is, 
  $\tilde A:= C(\{t_i,\tilde A_i\}:m), $  where $t_1=0$ and $\tilde A_1=\{ \tilde x_1\}.$
 In this case  we may find,
for every $1\leq i\leq m,$   an   open set $\tilde U_i\subset  \widetilde L$ and a trivializing open set $U_i\subset L$  such that 
 $\tilde A_i\subset  \tilde U_i$   and that 
$\pi|_{\tilde U_i}$ is  homeomorphic  onto its  image $\pi(\tilde U_i)=U_i.$ 
 Let $\gamma\in \pi_1(L)$ be the  unique  deck-transformation sending $\tilde x_1$ to $\tilde x_2.$ Since $\tilde A=\pi_{\tilde  x_1}^{-1}A,$ it follows that
$
\pi_{\tilde  x_2}^{-1}A=  C(\{t_i,\gamma(\tilde A_i)\}). $ So it suffices to check that
$$
W_{\tilde x_1}\big (   C(\{t_i,\tilde A_i\}) \big  )= W_{\gamma(\tilde x_1)}\big (   C(\{t_i,\gamma(\tilde A_i)\}) \big ). 
$$
 But this equality  is  an immediate consequence of     the invariance   under the deck-transformations of the  heat kernel  (see (\ref{eq2_heat_kernel})).  So  all $ W^{\tilde x}_x$      with $\tilde x\in \pi^{-1}(x)$ coincide on   good cylinder images.

 

Next, we   show   that  all $W^{\tilde x}_x$  with $\tilde x\in \pi^{-1}(x)$ coincide on the union of two good cylinder images.
Indeed, let $A=B\cup C,$ where $B$ and $C$ are good cylinder images.
Writing $A=(B\setminus C)\sqcup C$ and  applying    Lemma \ref{lem_algebras_leaf} (ii) to $B\setminus C,$
we may find  a countable family of disjoint  good cylinder images $(B_i)_{i=1}^\infty$ such that
  $B\setminus C=\bigcup_{i=1}^\infty B_i.$  Now using the $\sigma$-additivity of a measure, we infer that
  $$
  W^{\tilde x}_x(A)= W^{\tilde x}_x(C)+\sum_{i=1}^\infty W^{\tilde x}_x(B_i).   
  $$
So  $W^{\tilde x}_x(A)$  does not depend on $\tilde x\in\pi^{-1}(x)$ as  desired.

In the next stage  we will show that  all $W^{\tilde x}_x$  coincide on   each
finite union of good cylinder images. Let $A=\cup_{i= 1}^n A_i,$ each $A_i$ being  a  good  cylinder image. 
Writing $$A= \Big (\big (\cup_{i= 1}^{n-1} A_i\big) \setminus A_n\Big ) \sqcup A_n,   $$
and using   Lemma \ref{lem_algebras_leaf} (ii), we  show by induction on $n$ that $A$ can be expressed 
as a countable  union of disjoint   good  cylinder images. 
Now using the $\sigma$-additivity of a measure, we infer that  $W^{\tilde x}_x(A)$ does not depend on  $\tilde x\in\pi^{-1}(x),$ as  desired.

Each element $A\in\Dc^1(L)$ may be  written as the  union of an increasing sequence $(A_n)_{n=1}^\infty$ of subsets of $\Omega(L),$ each set $A_n$ being
a finite union of good cylinder images.
Using the $\sigma$-additivity of a measure again, we deduce that  $W^{\tilde x}_x(A)=\lim_{n\to\infty} W^{\tilde x}_x(A_n).$ So all   $W^{\tilde x}_x$  with $\tilde x\in \pi^{-1}(x)$ coincide  on all elements of $\Dc^1(L).$
  
 Next,  since  each  element of  $A\in \Dc^2(L)$ may be expressed as the intersection of a decreasing  sequence $(A_n)_{n=1}^\infty\subset \Dc^1(L),$
  it follows that  $W^{\tilde x}_x(A)=\lim_{n\to\infty} W^{\tilde x}_x(A_n).$ So all  $W^{\tilde x}_x$  with $\tilde x\in \pi^{-1}(x)$  coincide   on  $\Dc^2(L).$
    Repeating the above  argument and  using    Proposition \ref{prop_algebras_leaf},
  we  can show that all   $W^{\tilde x}_x$  with $\tilde x\in \pi^{-1}(x)$  coincide  on  the  algebra $\Dc(L)=\cup_{k=1}^\infty \Dc^k(L),$ proving  our  claim.   

 We denote by $W_x$  the  restriction of   $W^{\tilde x}_x$ on $\Dc(L)$ which is  independent of $\tilde x\in\pi^{-1}(x).$
 So $W_x$ is  a countably additive function  from $\Dc(L)$ to $[0,1].$
Since  the $\sigma$-algebra $\Ac_x$ is generated by the algebra $\Dc(L)\cap \Ac_x,$ we deduce from
 Part 1)  of Proposition  \ref{prop_measure_theory}  that  $W_x$  extends to a probability measures (still denoted by)
$W_x$    on  $(\Omega_x,\Ac_x).$   This  completes the proof.
\hfill $\square$

\noindent{\bf End of the proof of  Proposition   \ref{prop_algebras} (i).} 
By the  construction  (\ref{eq_algebra_Dc_L}) and  Lemma  \ref{lem_algebras_leaf} (iv)
and  Proposition \ref{prop_algebras_leaf}, we  are in the  position to apply  Proposition \ref{prop_Wiener_measure} (i) and 
   Proposition  \ref {prop_approximation_measure_theory}. Hence,   assertion (i) follows. 
  \hfill $\square$
    \section{Holonomy maps}
    \label{subsection_holonomy_maps}
We will define the notion of the holonomy map along a path  
and the notion of flow  tubes.
A similar  (but slightly different) formulation of the  holonomy map could be found in the  textbook \cite[Chapter 2]{CandelConlon1}. 
We need  the following terminology and notation. 
A {\it multivalued  map} $f:\ Y\to Z$  is  given by its  graph $\Gamma(f)\subset Y\times Z.$
For each $y\in Y$ we denote by $f(y)$ the set $\{z\in Z:\  (y,z)\in \Gamma(f)\}.$   The {\it  domain of definition} $\Dom(f)$ of $f$ is the set $\{y\in Y:\ f(y)\not=\varnothing\}\subset Y$ and 
the {\it range}  $\Range(f)$ of $f$ is  the subset $f(Y):=\cup_{y\in \Dom(f)} f(y)\subset Z.$ If, moreover, $f$ is univalued and one-to-one,  then
$\Dom(f^{-1})=\Range(f)$ and  $\Range(f^{-1})=\Dom(f).$ For another multivalued map $g:\ Z\to W,$ we define $\Dom(g\circ f)$
as the set of all points $y\in Y$ such that
the composition  $(g\circ f)(y)$ is  nonempty, i.e, the set of all 
$y\in \Dom(f)$ such that $f(y)\cap \Dom(g)\not=\varnothing.$  
The   germ of a local  homeomorphism $f$ at a point $x\in\Dom(f)$  is the equivalent class of all  local homeomorphisms defined  on a neighborhood of $x$ and agreeing with $f$ on a  neighborhood of $x.$ 

Let $(X,\Lc)$ be a lamination     and  set $\Omega:=\Omega(X,\Lc).$ 
 Consider an atlas $\Lc$ of $X$ with (at most) countable and  locally  finite  charts 
$$\Phi_\alpha:\U_\alpha\rightarrow \B_\alpha\times \T_\alpha,$$
where $\T_\alpha$ is a locally compact metric space, $\B_\alpha$ is a 
domain in $\R^k$ and
$\Phi_\alpha$ is a homeomorphism defined on 
an open subset $\U_\alpha$  of
$X$. 
  A set $S\subset X$ is  said  to be a  {\it  continuous transversal}\index{transversal!continuous $\thicksim$} 
  if there is    a flow box $\U$  with chart  $\Phi:\  \U\to \B\times \T,$  
and a connected open subset $V$   of $\T,$ and a  continuous  map $V\ni t\mapsto s(t)\in \B$ such that
$$S= \{\Phi^{-1}(s(t),t):\  t\in V\}.$$
 Note that $S$ is  a continuous  image  of an open subset of $\T$   and  that
$S$ intersects  every plaque
 $\Phi^{-1}(\cdot,t)$  with $t\in V$  of $\U$   in exactly one point and that
$S$ does not intersect other  plaques of $\U.$  
 Hence, if we  fix  a  point $x\in S,$ then   for every  sufficiently small
open neighborhood  $U$ of $x,$    $S\cap U$  is  still a continuous  transversal at $x.$   Consequently, 
if  $x$ is also contained in another flow  box $\U',$ then by  shrinking  $S$  if necessary (that is, by
replacing  $S$ with $S\cap U$ as above), $S$ is also a continuous transversal at $x\in\U'.$
So the  germ of  a continuous  transversal at a point  is independent of     flow boxes. 

Let $\omega\in\Omega$  and  set $t_0:=0.$    Let $t_1>0$ such that  $\omega[t_0,t_1]$ is  contained  in a single   flow box $\U.$
Let $S_0$ (resp. $S_1$)  be  a continuous   transversal  at $x_0:=\omega(t_0)$  (resp. at $x_1:=\omega(t_1)$). We may choose    an  open neighborhood $V$ of $t_0$
in $\T$ such that by  shrinking $S_0$ and $S_1$ if necessary,
$$S_0= \{\Phi^{-1}(s_0(t),t):\  t\in V\}\quad\text{and}\quad  S_1= \{\Phi^{-1}(s_1(t),t):\  t\in V\}. $$
We define  the  holonomy map $h_{\omega,t_1}:\ S_0 \to S_1$  as  follows:
$$
 h_{\omega,t_1}(x):=   \Phi^{-1}(s_1(t),t),\quad  x\in S_0,
$$
where $t=t_x\in V$ is uniquely determined  by $\Phi(x)=(s_0(t),t).$
In summary, we have  shown that  $\Dom (h_{\omega,t_1})$  (resp. $\Range(h_{\omega,t_1})$)   is an open neighborhood of $x_0$ in $S_0$ (resp. of $x_1$ in $S_1$).
 In other  words,  the germ at $x_0$  of  $ h_{\omega,t_1}$ is a well-defined homeomorphism.

Now  we   define  the  holonomy map $h_{\omega,t}:\ S \to S'$ in the  general case.  Here   $S$ (resp. $S'$)  is  a continuous transversal  at $x_0:=\omega(0)$  (resp. at $x_t:=\omega(t)$).
Fix a  finite  subdivision  $0=t_0<t_1<\ldots  <t_k=t$ of $[0,t]$ such that   $\omega[t_i,t_{i+1}]$ is  contained in a flow  box  $\U_i.$ 
Choose  a continuous  transversal  $S_i$ at $t_i$ such that $S_0=S$ and $S_k=S'.$
The construction given in the  previous paragraph shows that  $ h_{T^{t_i}\omega,t_{i+1}-t_i}$  is a well-defined homeomorphism from an open neighborhood of
$x_i$ in  $S_i$ onto   an open neighborhood of
$x_{i+1}$ in $S_{i+1}.$  
 The  {\it holonomy map}  along $\omega$  at time $t$    is,  by definition,
the composition 
$$h_{\omega,t}:= h_{T^{t_{k-1}}\omega,t_k-t_{k-1}} \circ  \cdots \circ h_{T^{t_i}\omega,t_{i+1}-t_i}\circ \cdots\circ h_{\omega,t_1}
.$$    
 This a    a well-defined homeomorphism from an open neighborhood of
$\omega(0)$ in  $S$ onto   an open neighborhood of
$\omega(t)$ in $S'.$  The germ of  $h_{\omega,t}$ at $x_0$ depends only on the  homotopy type  of $\omega[0,t].$

Now we introduce  the   notion of flow tubes  which   generalizes the notion of  flow boxes. Flow tubes   are more flexible than  flow boxes. Roughly speaking, a
flow  tube can be as  thin  and  as  long as   we  like, whereas  flow boxes are  rigid.
In some  sense, a flow tube is  a  chain  of small flow boxes.
 However, contrary to the flow boxes,   a plaque of a given   
  flow tube  may meet several  plaques  of another adjacent flow tube.  
\begin{definition}
\label{defi_flow_tube}\rm 
An open set $\U\subset X$ is  said  to be 
a {\it  flow tube}\index{flow tube}  of $(X,\Lc)$ if
there is a  continuous   transversal  $\T$ such that for each $x\in\T,$
there is  a relatively compact connected  open  subset $\U_x$ of the leaf $L_x$ with the following  properties:  
\\
$\bullet$  $x\in \U_x$ and 
 $\U_x\cap \U_{x'}=\varnothing$ for $x,x'\in\T$ with $x\not= x';$
\\
$\bullet$ $\bigcup_{x\in\T}\U_x=\U.$
 
$\T$ is  said to be  a {\it transversal}\index{flow tube!transversal of a $\thicksim$}\index{transversal!$\thicksim$ of a flow tube}  of the  flow tube $\U.$ For each $x\in \T,$ the  set $\U_x$
is  said to be the {\it  plaque}\index{flow tube!plaque of a $\thicksim$}\index{plaque!$\thicksim$ of a flow tube} at $x$  of $\U.$

The  {\it sample-path space  of  a flow tube $\U$   up to  time $N\geq 0$}
\index{space!sample-path $\thicksim$!$\thicksim$ of a flow tube up to a given time}
 is, by definition, the subspace of $\Omega:=\Omega(X,\Lc)$ consisting of
 all 
$\omega\in\Omega$ such that $ \omega [0,N]$ is  fully contained in a single  plaque $\U_x$  for some $x\in\T.$  This space is  denoted by $\Omega(N,\U).$  
\nomenclature[c1g]{$\Omega(N,\U)$}{sample-path space of a flow tube $\U$ up to a time $N>0$}


Let $\pi:\ (\widetilde X,\widetilde\Lc)\to (X,\Lc)$ be the covering lamination projection.
An open set $\widetilde\U\subset \widetilde X$ is  said  to be 
a {\it good flow tube} in $(\widetilde X,\widetilde\Lc)$  if its {\it image}
$\U:=\pi(\widetilde \U)\subset X$ is  also a flow tube in $(X,\Lc)$
and if the restriction of the projection $\pi|_{\widetilde \U  }:\ \widetilde \U\to  \U$ is  a  homeomorphism
which maps each plaque  of  $\widetilde \U$ onto  a   plaque  of  $\U.$
For a  transversal $\widetilde\T$  of $\widetilde\U,$ the set  $ \T:=\pi(\widetilde\T)$ is a transversal of $\U.$



 A  pair of conjugate  flow  tubes  $(\widetilde \U',\widetilde \U'' )$
 is the  data of two good  flow tubes $\widetilde \U',$ $\widetilde \U''$in $(\widetilde X,\widetilde \Lc)$ such that
 they  have a  common image, i.e.,  $\pi(\widetilde \U')=\pi(\widetilde \U'').$
 \end{definition}
\begin{remark}
Unlike flow  boxes, flow tubes do not possess, in general, the  product structure of a transversal times a plaque.
However, flow  tubes still have a somehow weaker structure of a semi-product:  $\{(x,y):\ x\in\T,\ y\in\U_x\},$
where each plaque  $\U_x$ may vary when $x\in \T.$
\end{remark}

 \section[Metrizability and  separability]{Metrizability and  separability of sample-path spaces}
 \label{subsection_sample-path_spaces}
   
   Let $(X,\Lc)$ be  a lamination. 
  Let $\pi:\ (\widetilde X,\widetilde\Lc)\to (X,\Lc)$ be the covering lamination projection.
Let $\Omega:=\Omega(X,\Lc)$ and $\widetilde\Omega:=\Omega(\widetilde X,\widetilde\Lc).$ 
   The  main  purpose of this  section is to prove the following remarkable result of the flow tubes.
 \begin{theorem}  \label{thm_separability}
 1) There exists   
 a  countable  family of pairs of conjugate  flow  tubes  $(\{\widetilde \U'_i,\widetilde \U''_i\} )_{i\in\N}$
 such that
for every $N>0,$
\begin{equation*} 
\begin{split}
&\left\lbrace     
(\tilde \omega',\tilde \omega'')\in  \widetilde   \Omega\times\widetilde   \Omega:\ \pi\circ \tilde \omega'(t)= \pi\circ\tilde \omega''(t),\ \forall t\in[0,N]\right\rbrace\\
 &\subset \bigcup_{i\in\N} \Omega(N, \widetilde\U'_i)\times \Omega(N, \widetilde\U''_i).
 \end{split}
\end{equation*}
2)  For  each $i\in\N$ let $\U_i:=\pi(\widetilde \U'_i).$ Then  for every $N>0$
$$  \Omega=\bigcup_{i\in\N}\Omega(N, \U_i).$$
\end{theorem}   
 \begin{proof}
 Since Part 2) is  an immediate consequence of Part 1), we only need to prove the latter part.
 We fix  an  atlas  $\Lc$ of $X$  with at most  countable  charts $\Phi_\alpha:\ \U_\alpha\to \B_\alpha\times\T_\alpha.$
 Suppose   without loss of generality that $\B_\alpha=\B=[0,1]^k$ for all $\alpha,$  where $k$ is the  real dimension of the leaves.
 Fix a transversal (still denoted by) $\T_\alpha:= \Phi_\alpha^{-1}(\{0\}\times \T_\alpha) $ for each flow box $\U_\alpha.$ Since   each $\T_\alpha$ is  a separable metric space,
we may find  a  countable  basis $\Tc_\alpha$ of nonempty open subsets of $\T_\alpha.$  Let    
 $\Tc:=\cup_\alpha \Tc_\alpha.$ So $\Tc$ is  also countable.
 
 Before  proceeding further, 
look at  the  cube $\B=[0,1]^k.$ For each $m\in\N$ consider the  subdivision of $\B$ into $2^{mk}$ smaller cubes
 $$
 \left\lbrack {d_1\over 2^m}, {d_1+1\over 2^m}\right\rbrack \times \cdots\times \left\lbrack{d_k\over 2^m}, {d_k+1\over 2^m} \right\rbrack,
 $$
 where the integers $d_1,\ldots,d_k$ range  over  $0,\ldots, 2^m-1.$
 Each  such a cube  is  said to be   a {\it cube of order $m.$}

 For each $m\in\N$  let $\Fc_m$ be  the following  family  of open subsets of $X$
$$
\Uc_m:=\left\lbrace \Phi_\alpha^{-1}(B\times S):\  B\ \text{a  cube of order $m$ and}\ S\in \Tc_\alpha\right\rbrace.
$$ 
 Each element of $\Uc_m$ is   said to be  a {\it small  flow box of order $m$}\index{flow box!small $\thicksim$}.
Let $\Fc:=\cup_{m\in\N}\Fc_m.$ Clearly, $\Fc$ is  countable.  
The   proof is  divided into several steps.
\\
{\bf Step 1:}  {\it Construction of the holonomy map  for each small  flow box.}

Consider the small flow box    $U:=\Phi_\alpha^{-1}(B\times S)\subset X$ of order $m,$  where 
$B$ is a  cube  of order $m$ and  $S\in \Tc_\alpha$ is an open set. 
Let $V$ (resp. $W$) be 
two open subsets of $U.$
Since $U$ is  a  flow tube with  plaques $U_s:=  \Phi_\alpha^{-1}(B\times \{s\})$ for  $s\in S,$
 consider  the following multivalued map  $h_{U;V,W}:\  U\to U$ 
whose  the graph is    
 $$
\Gamma(h_{U;V,W}):=\left\{ (x,y)\in V\times  W:\ \text{$x$ and $y$ are  on  the  same plaque of $U$ } \right\}.
 $$
 Consider,  for each $N>0,$ the    sample-path space associated to $(U;V,W):$
 $$
 \Omega_{V,W}(N,U):= \left\lbrace \omega\in \Omega(N,U):\ \omega(0)\in V\ \text{and}\  \omega(N)\in W   \right\rbrace.
 $$
 {\bf Step 2:}  {\it    Construction of the holonomy map  for a chain
 of small flow boxes.}
 

A {\it chain}  $\mathcal U$ of small  flow  boxes\index{chain!$\thicksim$ of small flow boxes}   is a collection of  small  flow  boxes $(U_i)_{i=0}^p$
such that $U_i\cap U_{i+1}\not=\varnothing$ for $0\leq i\leq p-1.$
  For each  $0\leq i\leq p-1,$ let $V_i:=U_i\cap U_{i-1}$  and $W_i:= U_i\cap U_{i+1},$ where $U_{-1}:=U_0$ and $U_{p+1}:=U_p.$
So $V_{i+1}=W_i$ for $0\leq i\leq p-1.$ The  {\it holonomy map  of the chain  $\mathcal U$} is,
by definition, the multivalued map $h_{\mathcal U}:\ U_0\to U_p$ given by
\begin{equation*}
h_{\mathcal U}:= h_{V_{p}, W_{p}}\circ\cdots \circ h_{V_0,W_0},
\end{equation*}
where  each multivalued function in the  right hand side is given by Step 1.
Note   that $\Dom(h_{\mathcal U})$ and $\Range (h_{\mathcal U})$ are (possibly empty) open  subsets of $X.$ 
 For each $N>0$   the    sample-path space associated to the chain $\mathcal U$ is  given by
\begin{equation*}
\begin{split}
\Omega(N,\mathcal U)&:= \left\lbrace \omega\in\Omega:\  \exists (t_i)_{i=0}^p\subset [0,N]\ \text{such that}\ 0=t_0<\cdots< t_p=N\right.\\
& \left.    \text{and that}\ T^{t_i}\omega\in\Omega_{V_i,W_i}(t_{i+1}-t_i, U_i)\ \text{for}\ 0\leq i\leq p-1 \right\rbrace.
\end{split}
\end{equation*}
If we  fix  a transversal $S_0$ at  a point $x_0\in\Dom(h_{\mathcal U})$ and a transversal $S_p$
at a point $x_p\in h_{\mathcal U}(x_0)\subset \Range(h_{\mathcal U}),$ then 
$h_{\mathcal U}$ induces  a   well-defined  univalued map (still denoted by $h_{\mathcal U}$) from a  sufficiently  small  open neighborhood
of $x_0$ in $S_0$ onto  an open neighborhood
of $x_p$ in $S_p.$ The latter map is even  homeomorphic. 
Moreover, the germ  of $h_{\mathcal U}$  
 at  $x_0$ coincides with the germ of $h_{\omega,N}$ at $x_0,$
 where $N>0$  and $\omega\in\Omega(N,\mathcal U)$ is a  path  such that $\omega(0)=x_0$ and $\omega(N)=x_p.$
 
 \noindent {\bf Step 3:}  {\it    Construction of a  flow tube   associated  to a  good  chain.}
 
 Let $N>0$ be  a  fixed time.
 Let $\mathcal U $  be  a  chain of small flow boxes   such that $\T:=\Dom(h_{\mathcal U})\not=\varnothing.$
 For  each $x\in \T$  the  open subset of $L_x$  given by
 $$
 \U_x:=\left\lbrace \omega(t):\ \omega\in\Omega(N,\mathcal U),\ \omega(0)=x,\  t\in[0,N]   \right\rbrace
 $$
 is  said  to be  the {\it plaque} at $x$ associated  to $\mathcal U.$
Observe that  $x\in\U_x$ and  $\U:=\cup_{x\in\T}\U_x$ is  an open subset of $X.$
We say  that  the chain $\mathcal U$ is {\it  good} if  $\U$ is  a  flow  tube (see Definition \ref{defi_flow_tube}) with the transversal $\T$
and the  plaque $\U_x$ for each $x\in \T$ as  above\index{chain!good chain of small flow boxes}. In this case 
 $\U=\U_{\mathcal U}$ is  said to be  {\it the  flow tube associated to the  good chain $\mathcal U.$}\index{flow tube!$\thicksim$ associated to a good chain of small flow boxes}
 \nomenclature[c1i]{$\U_{\mathcal U}$}{flow tube associated to a good chain $\mathcal U$}
Clearly,   a chain $\mathcal U$ is  good if and only if $\T\not=\varnothing$ and
$\U_x\cap \U_{x'}=\varnothing$ for all $x,x'\in\T$ with $x\not=x'.$ For each $m\in\N$
let $\Sc_m$ be  the  set of all good chains  each small flow box of which is  of order $\geq m.$
So $\Sc_m$ is countable  and  $(\Sc_m)_{m=0}^\infty$ is decreasing. 
Let $\Sc:=\Sc_0.$
The following result is needed.
\begin{lemma}\label{lem_existence_good_chain}
For  each $N>0$    and $\omega\in\Omega,$ and for each open neighborhood  $\V$ of
 $\omega[0,N],$
there  exist a  a good chain $\mathcal U\in \Sc$ such that all  small flow boxes of $\mathcal U$ are contained in $\V$
and that 
$\omega\in\Omega(N, \U_{\mathcal U}),$ where $\U_{\mathcal U}$ is the  flow  tube associated with $\mathcal U.$  
\end{lemma}
  \begin{proof}
  Using the  compactness of $\omega[0,N]$ and shrinking  $\V$ if necessary, we may
  suppose  without loss of generality that
   $\V$   is a  sufficiently  small tubular neighborhood  of $\omega[0,N]$
  which is  also  a  flow tube.
  So $\omega\in\Omega(N,\V).$
  By decreasing the  size of small flow boxes  if necessary (that is, by increasing $m$),  we may cover  $\omega[0,N]$ by a good chain $\mathcal U\in \Sc_m$
  such that its small flow boxes are all contained in $\V$ and that $\omega\in\Omega(N, \U_{\mathcal U}),$
  as asserted.
   \end{proof}
\begin{remark}\label{R:lem_existence_good_chain}
 An immediate consequence of Lemma \ref{lem_existence_good_chain} is that
 $$\Omega=\bigcup_{\mathcal U\in\Sc}  \Omega(N, \mathcal U)\quad\text{and}\quad \Omega=\bigcup_{\mathcal U\in\Sc} \Omega(N, \U_{\mathcal U}).$$
\end{remark}
 {\bf Step 4:}  {\it    End of the proof.}
 
 For every flow  tube $\V,$ let  $\Hc(\V)$ be the set of all homotopies $\alpha:\ \V\times [0,1]\to X$ (see Section  \ref{subsection_Covering_laminations}).
 We  see  easily that $\Hc(\V)$ is  at most  countable.
 Now let $ (\tilde \omega',\tilde \omega'')\in  \widetilde   \Omega\times\widetilde   \Omega$ such that $ \pi\circ \tilde \omega'(t)= \pi\circ\tilde \omega''(t)$ for all $t\in[0,N].$
 Let $\omega:=\pi\circ\tilde\omega'\in\Omega.$ 
Using the  compactness of $\omega[0,N]$ we may
  choose  a flow  tube  $\V$  which is a  sufficiently  small tubular neighborhood  of $\omega[0,N]$
  such that  there  are homotopies  $\alpha'$ and $\alpha''\in  \Hc(\V)$   such that
  $$
  \tilde\omega'(t)=(\omega(t),[\alpha'_{\omega(t)}])\quad \text{and}\quad\tilde \omega''(t)=(\omega(t),[\alpha''_{\omega(t)}]),\quad t\in[0,N].
  $$
By Lemma  \ref{lem_existence_good_chain},
    we may cover  $\omega[0,N]$ by a good chain $\mathcal U\in \Sc$
  such that its small flow boxes are all contained in $\V$ and that $\omega \in\Omega(N, \U_{\mathcal U}).$
  By  restricting  $\alpha',\alpha''\in \Hc(\V)$ to $\U_{\mathcal U},$ we may consider them as  elements of  $\Hc(\U_{\mathcal U} ).$
Since  $\Sc$ is  countable  we may write  $\U_{\mathcal U}:=\U_j $ for $j\in\N.$ Moreover, for each $j\in\N,$
  let $I_j$ be the set indexing
  all  pairs $(\alpha'_i, \alpha''_i)$ as above. Clearly, $I_j$ is  at most countable. For $i\in I_j$ let
  $$\widetilde\U'_i:= \{  (x,[\alpha'_{ix}]):\ x\in \U_j\}\quad\text{and}\quad \widetilde\U''_i:= \{  (x,[\alpha''_{ix}]):\ x\in \U_j\}.$$
  So   $(\widetilde \U'_i,\widetilde \U''_i )$ is a  pair of conjugate  flow  tubes.
  By lifting  $\omega[0,N]$ to   $\tilde\omega'[0,N]$ and  $\tilde\omega''[0,N]$ and  using  that $\omega\in\Omega(N, \U_{\mathcal U}),$ we get that  
$$ (\tilde \omega',\tilde \omega'')\in\Omega(N, \widetilde\U'_i)\times\Omega(N, \widetilde\U''_i).$$
  By  taking the  above  membership over all $j\in \N$ and all $i\in I_j,$ the theorem  follows.
 \end{proof}

   Although the  following result is  not  used  in this  work, it is of  independent interest.  
   
\begin{theorem}\label{thm_seperable_widehat_Omega}
Let $(X,\Lc,g)$ be  a  Riemannian  lamination.
\\
1)  There is  a  metric $\dist$ on $\Omega:=\Omega(X,\Lc)$  such that the metric space $ (\Omega,\dist)$ is  separable
and that its Borel $\sigma$-algebra  coincides with $\Ac:=\Ac(\Omega).$ 
\\
2)  There is  a  metric $\dist$ on $\widehat\Omega:=\widehat\Omega(X,\Lc)$  such that the metric space $ (\widehat\Omega,\dist)$ is  separable
and that its Borel $\sigma$-algebra  coincides with $\widehat\Ac:= \widehat\Ac(\Omega).$
 \end{theorem} 
 We only give  the proof of the  separability  in both  assertions 1) and  2). We  leave to the  interested  reader to  verify
 the  coincidence between  $\sigma$-algebras   stated in these  assertions.
 \begin{proof}
First we prove assertion 1).
  Since the the induced topology on  each  flow box  is metrizable  and
$X$ is  paracompact,  the topological space $X$ is  metrizable. Replacing  $\rho$ by ${\rho\over 1+\rho}$ if necessary, we may  
assume without loss of generality that there exists a metric $\rho$ on $X$ which induces the topology   of $X$ and which satisfies
$\rho  \leq 1.$
Consider the following metrics on $\Omega :$
\begin{eqnarray*}
\dist(\omega,\omega')&:=& \sum_{n=1}^\infty {1\over 2^n}\sup_{t\in[0,n]} \rho (\omega(t),\omega'(t)),\qquad  \omega,\ \omega'\in  \Omega;\\
\dist_N(\omega,\omega')&:=& \sup_{0\leq t\leq N} \rho (\omega(t),\omega'(t)),\qquad  N\in\N\setminus \{0\}\ \text{and}\ \omega,\ \omega'\in  \Omega.
\end{eqnarray*}
  For every $\omega\in  \Omega$ and $r>0$ and $N\in \N\setminus \{0\},$ consider the       balls 
$$
\B(\omega,r):=\{ \omega'\in \Omega:\  \dist(\omega,\omega') <r\}\ \text{and}\
 \B_N(\omega,r):= 
\{ \omega'\in\Omega:\  \dist_N(\omega,\omega') <r\}  .
 $$
 Using $\rho\leq 1$ it  is  not hard  to check that
 $$
 \B_N(\omega,r)\subset \B(\omega,r+1/2^N)\quad\text{and}\quad  \B(\omega,r)\subset \B_N(\omega,2^N r).
 $$
 Therefore,  we  see that  in order to show that the metric space $(\Omega,\dist)$ is  separable, it  
 suffices to prove  that for every $N\in \N\setminus\{0\},$  the  space $\Omega_N$ equipped with the metric $\dist_N$
is  separable, where $\Omega_N$ consists of all leafwise continuous  paths $\omega:\  [0,N]\to (X,\Lc).$
We may assume without loss of generality that $N=1.$  

First  we consider the  case  where  the lamination $(X,\Lc)$ is  reduced to a single  leaf $L.$ The separability
of the metric space $\Cc ( [0,1],L)$ of all  continuous  map $\omega:\ [0,1]\to L$ is  well-known.

Now we consider the general case. Fix a countable dense sequence  $(x_n)_{n=1}^\infty\subset X.$   
 For each $n\geq 1$  we fix  a countable  number of balls  $B(\omega_{nm},r_{nm})$ such that $\omega_{nm}\subset L_{x_n}$  and that
this family of balls, when restricted to the leaf $L_{x_n},$
 constitutes a basis of open subsets of  $\Cc([0,1],L_{x_n}).$
We leave  it to  the interested  reader to check that the countable  family   $\{  B(\omega_{nm},s_{nm}):\ n,m\geq 1\ \text{and}\ s_{nm}\in \Q,\ s_{nm}>0\}$
forms a  basis of open subsets of $\Omega.$  Hence,   the metric space $(\Omega,\dist)$ is  separable.


To prove   assertion 2)
 consider   the following metric on $\widehat\Omega :$
$$
\dist(\omega,\omega'):= \sum_{n=1}^\infty {1\over 2^n}\sup_{t\in[-n,n]} \rho (\omega(t),\omega'(t)),\qquad  \omega,\ \omega'\in  \widehat\Omega.
$$
The rest of the  proof 
 is  analogous  to  that of assertion 1).
   \end{proof}

 \section[The leafwise diagonal]{The leafwise diagonal is    Borel measurable}
 \label{subsection_leafwise_diagonal}

The main result of this  section is the following
\begin{proposition}\label{prop_diagonal_measurable}
Let $(X,\Lc)$ be a  lamination. Then
the  leafwise  diagonal  defined by
$$
\Gc:=\left\lbrace (x,y)\in  X^2:\  x \ \text{and}\  y\ \text{are on  the  same leaf}  \right\rbrace
$$
is Borel measurable.
\end{proposition} 
\begin{proof}
We will use the  terminology and notation introduced in Section  \ref{subsection_sample-path_spaces}. 
Fix an  atlas $\Lc$ of $X$ with  at most  countable charts $\U_\alpha.$
 Fix  a  transversal  $\T_\alpha$ for each  flow box $\U_\alpha.$ Let $\T:=\bigcup_\alpha \T_\alpha.$  We also fix  a countable basis $\Tc_\alpha$  of nonempty open subsets of $\T_\alpha.$ Let $\Tc:=\cup_\alpha \Tc_\alpha.$ So $\Tc$ is  countable. 
 Let $\Sc$ be the (countable) set of all good  chains.
 
   Observe that two given  points 
$t_1$ and $t_2\in\T$ are on the same  leaf  if and only if there  exists a  chain $\mathcal U $ of flow boxes and  an open set $S\in\Sc$  such that $S\subset \Dom(h_{\mathcal U})$
and that
 $t_1\in S$ and $t_2=h_{\mathcal U}(t_1).$
 Let $(\{\mathcal U_n, S_n\})_{n\in I}$ be  the   sequence of  all possible pairs consisting of  a good chain $\mathcal U_n$ and  an element $S_n\in\Sc$  such that   $S_n\subset \Dom(h_{\mathcal U_n}).$
Note that $I$ is  at most countable.
Consider the following  subset  of $\T^2:$
  $$
  G=  \bigcup_{n\in I}  \{ (t,  h_{\mathcal U_n}(t):\ t\in  S_n\}  .
  $$
Each set   $\{ (t,  h_{\mathcal U_n}(t):\ t\in  S_n\}$  is  a Borel   subset of $X^2$  because
$h_{\mathcal U_n}$ is a continuous  map. Hence, $G$ is also a  Borel set.
We need the following  result whose the  proof is  left to the  interested reader.
 \begin{lemma}
 Let $\S$ be  a  Borel subset of a transversal $\T_\alpha$  of $(X,\Lc).$
 Then $\Satur(\S)$ is  also a Borel set.
 \end{lemma}
 Let $(X^2, \Lc^2)$ be the product 
 of the lamination $(X,\Lc)$ with  itself. More precisely,
  $(X^2, \Lc^2)$ is also a lamination  whose the  leaf  $L_{(x,x')}$ passing  through a point $(x,x')$
  is  $L_x\times L_{x'}.$
 Observe that $\Gc=\Satur(G),$ where $\Satur$ is  taken  in $(X^2, \Lc^2)$.  
Recall   that $G$ is  a  Borel subset of $\T^2.$   Applying    the  above  lemma  to $G$ yields that
 $\Gc$ is   a Borel set.
 \end{proof}

 \section{$\sigma$-algebra $\Ac$ on a lamination}
 \label{subsection_algebra_on_a_lamination}

The main purpose  of this  section is  to complete the proof of
 Theorem  \ref{prop_Wiener_measure} (ii),   
 Theorem \ref{thm_Wiener_measure_measurable} and  Proposition   \ref{prop_algebras} (ii).  
Let $(X,\Lc,g)$ be a Riemannian lamination  satisfying the Standing Hypotheses  and
$\pi:\ (\widetilde X,\widetilde \Lc)\to (X,\Lc)$ the covering lamination  projection.
  Set $\Omega:=\Omega(X,\Lc)$ and  $\widetilde\Omega:=\Omega(\widetilde X,\widetilde\Lc).$ 
  Consider the $\sigma$-algebras  $\Ac:=\Ac(\Omega)$ and  $\widetilde\Ac:=\widetilde\Ac(\Omega)$  on $\Omega,$  and  
the $\sigma$-algebra $\Ac(\widetilde\Omega)$ on $\widetilde\Omega$ introduced in Section  \ref{subsection_Brownian_motion_without_holonomy}
and  Section  \ref{subsection_Wiener_measures_with_holonomy}.
Recall  from  Section \ref{subsection_measurability_issue}  that  a  set $A\subset \Omega$ is  said to be a  cylinder image if  $A=\pi\circ \tilde A$ for  some cylinder set $\tilde A\subset \widetilde\Omega .$ 
 Note that $\Ac(\widetilde\Omega)=\widetilde\Ac(\widetilde\Omega).$
 Recall the following  terminology. For  $A\subset \Omega,$ let $\pi^{-1}A:=\{ \tilde\omega\in \widetilde\Omega:\ \pi\circ\tilde\omega\in A\}\subset \widetilde\Omega.$
 For a  family $\Fc$ of subsets of $\Omega,$ let    $\pi^{-1}\Fc$ be the  family $\{ \pi^{-1}  A:\ A\in\Fc\}.$ 
 For   $\tilde A\subset\widetilde \Omega,$ let $\pi\circ \tilde A:=\{   \pi\circ\tilde\omega:\ \tilde\omega\in \tilde A\}\subset \Omega.$ 
 For a  family $\widetilde\Fc$ of subsets of $\widetilde\Omega,$ let    $\pi\circ\widetilde \Fc$ be the  family $\{ \pi\circ  \tilde A:\ \tilde A\in\widetilde \Fc\}.$
The  following result is the main technical  tool in this  section. 
  \begin{proposition}\label{prop_algebras_laminations}
 $\widetilde\Ac\subset\Ac$ and   $\pi^{-1}(\Ac)\subset\Ac(\widetilde \Omega)$ and  $\pi \circ\Ac(\widetilde \Omega)=\Ac.$ 
 \end{proposition}
 Taking  for granted Proposition \ref{prop_algebras_laminations}, 
we  arrive  at the
 
 \noindent{\bf End of the proof of   Theorem \ref{prop_Wiener_measure} (ii).} It follows from the  second inclusion 
 in Proposition \ref{prop_algebras_laminations}. \qed

%

\medskip

 For the proof of  Theorem \ref{thm_Wiener_measure_measurable} (i) we need the following result. 
 \begin{proposition}\label{prop_measurable_heat_kernels}
The  function  $\Phi:\ X\times X\times\R^+\to\R^+$  defined  by
$$
\Phi(x,y,t):=\begin{cases}
p(x,y,t), &  \text{$x$ and $y$  are on the same leaf};\\
0, & \text{otherwise};
\end{cases}
$$
is    Borel measurable.
\end{proposition}
 \begin{proof}
 Using  identity (\ref{eq1_heat_kernel})
and the fact that $\pi:\ \widetilde X\to X$ is  locally homeomorphic, 
we may  asssume without loss of generality 
that $X=\widetilde X,$ that is, all leaves of $(X,\Lc)$ are   simply connected.
Fix  a transversal $\T$ and straighten all leaves  passing  through $\T.$   Using Proposition \ref{prop_diagonal_measurable},
it suffices   to show that  the  heat kernel of the leaf $L_\tau$ ($\tau\in \T$) given by $p_\tau(x,y,t):= p(x,y,t)$  with  $x,y\in L_\tau$ and $t\in\R^+,$ 
depends Borel measurably  on $\tau\in \T.$
Let $n$ be the dimension of the leaves. By identifying with $\R^n$
the tangent spaces $T_\tau L_\tau$ of the  leaf $L_\tau$ at  the  point $\tau\in \T,$
we obtain  by   Hopf-Rinow theorem\index{Hopf!Hopf-Rinow theorem}\index{Rinow!Hopf-Rinow theorem}
\index{theorem!Hopf-Rinow $\thicksim$} a family of  (surjective and locally diffeomorphic)  exponential maps     $\exp_\tau:\  \R^n\to L_\tau,$ which depends
measurably on the  parameter $\tau.$
For $N\in \N$ let  $\B_N$ be the open ball centered at $0$  with radius $N$ and 
let $\B_{N,\tau}:= \exp_\tau(\B_N)$ for $\tau\in \T.$ Since  $\B_{N,\tau}$ is a bounded  regular domains in $L_\tau,$
there exists a heat kernel  $p_{N,\tau}$  for $\B_{N,\tau}.$ 
Moreover,  by the  construction of the heat kernel  described in Appendix  B.6 in \cite{CandelConlon2},
the family  $p_{N,\tau}( \exp_\tau(u), \exp_\tau(v),s) $ depends Borel measurably on $(\tau;u,v,t)\in \T\times \B_N\times \B_N\times
\R^+.$ Here  we make  a full use of the  assumptions  that the geometry of $L_\tau$ are uniformly bounded and that the leafwise complete metric $g,$ when restricted   to each flow box, depends Borel measurably on plaques. 
On the other hand, since $\B_{N,\tau}\nearrow L_\tau$ as $N\nearrow\infty,$ it is  well-known (see \cite{Chavel} or Appendix  B.6 in \cite{CandelConlon2}) that
\begin{multline*}
p_\tau(\exp_\tau(u), \exp_\tau(v),t)=\limsup_{N\to\infty,\ N\in\N}    p_{N,\tau}( \exp_\tau(u), \exp_\tau(v),t),\\
 (\tau;u,v,t)\in \T\times \R^n\times \R^n\times
\R^+,
\end{multline*} 
where $p_\tau$ is  the heat kernel of the leaf $L_\tau.$
 This  formula  implies   the desired measurability of the  function $\Phi.$  
\end{proof}
  
\noindent{\bf  End  of the proof of Theorem \ref{thm_Wiener_measure_measurable} (i).}
By Proposition \ref{prop_algebras_laminations}, $\pi^{-1}B\in\Ac(\widetilde \Omega).$
On the  other hand,  by Lemma \ref{lem_change_formula}
(i), $W_x(B)=W_{\tilde x}(\pi^{-1}B)$ for every  $\tilde x\in\pi^{-1}(x).$
Moreover,  $\pi$ is a covering projection which is  locally homeomorphic.
Consequently,  we are  reduced  to showing that for every $\tilde B \in \Ac(\widetilde \Omega),$
 $\widetilde X\ni \tilde x\mapsto    W_{\tilde x}(\tilde B) \in[0,1]$  is
 Borel measurable. So we  may  asssume without loss of generality 
that $X=\widetilde X,$ that is, all leaves of $(X,\Lc)$ are   simply connected.
 The proof is divided into three steps.

\noindent{\bf Step 1:} {\it  For a cylinder set $ B:= C(\{t_i,  A_i\}), $
the  function $X\ni x\mapsto W_x(B)$ is measurable.}

Let $\T$ be a transversal and $A_1,\ldots, A_m$ Borel subsets of $\Satur(\T).$
Let $f:\ X\to\R$ be a Borel measurable  function. We  rewrite (\ref{eq_diffusions}) as  follows:
 $$
 D_tf(x):=\int_{L_x} \Phi(x,y,t) f(y) d\Vol_{L_x} (y),\qquad x\in X,
 $$
 where $\Phi$ is  the measurable  function given in Proposition \ref{prop_measurable_heat_kernels}
above. Next, fix  a point $x_0\in \T$ and consider the measure  $\mu:=\Vol_{L_{x_0}}$ on $L_{x_0}.$
Using the  straightening  of leaves passing through $\T$ we can  write, for each $x\in \T,$ 
 $$\Phi(x,y,t)  d\Vol_{L_x} (y)=\Phi(x,y,t)\Psi(x,y,t)  d\Vol_{L_{x_0}} (y)$$
 for some positive measurable  function  $\Psi.$
Consequently,  applying 
 Proposition \ref{prop_integral_dependance_measurably_on_parameter} yields that
 the function $\R^+\times X\ni(t,x)\mapsto (D_tf)(x)$ is  Borel measurable
for each  Borel measurable  function $f:\ X\to\R.$ 
 This, combined  with formula  (\ref{eq_formula_W_x_without_holonomy}), implies that
 $X\ni x\mapsto W_x(B)$ is  Borel measurable, as  desired.

\noindent{\bf Step 2:} {\it Let $\Dc$ be the family of  all  finite unions of cylinder sets.
Then  for every $B\in\Dc,$ the function  $X\ni x\mapsto W_x(B)$ is  Borel measurable.}

By Proposition \ref{prop_cylinder_sets},  we may write $B$ as a  finite  union of mutually disjoint  cylinder sets
 $B=\bigsqcup_{i=1}^k B_i.$
Since $W_x(B)=\sum_{i=1}^k W_x(B_i)$ and  each function $X\ni x\mapsto W_x(B_i)$ is Borel measurable, the  desired conclusion  follows.
 
\noindent{\bf Step 3:} {\it End of the proof.}  

Let $\Cc$ be the family of all $ B\in\Ac$ such that the function $X\ni x\mapsto W_x(B)$ is Borel measurable.
By Step 2, we get that $\Dc\subset \Cc.$ Next,
observe  that  if $(A_n)_{n=1}^\infty\subset \Cc$ such that  $A_n\nearrow A$ (resp.  $A_n\searrow A$)  as $n\nearrow \infty,$
then  $A\in\Cc$ because  $W_x(A_n) \nearrow W_x(A)$  (resp.   $W_x(A_n) \searrow W_x(A)$).
Consequently,
applying Proposition \ref{prop_criterion_sigma_algebra} yields that $\Ac\subset \Cc.$
 Hence, $\Ac=\Cc.$
 This  completes the proof.
\qed

\noindent{\bf End of the proof of   Theorem   \ref{thm_Wiener_measure_measurable} (ii).} 
For each $B\in\Ac,$ by Theorem \ref{thm_Wiener_measure_measurable} (i),
the bounded function $x\mapsto W_x(B)$ is  Borel measurable. Consequently,
$\bar\mu(B)$  given in (\ref{eq_formula_bar_mu}) is well-defined.
To conclude  assertion (ii)
  we need  to  show that $\bar\mu$ is  countably  additive  function from
$\Ac$ to $\R^+.$  Let $A=\cup_{n=1}^\infty A_n,$ where  $A_n\in \Ac$ with $A_n\cap A_m=\varnothing$ for $n\not=m.$
By  Proposition \ref{prop_algebras_laminations}, $\pi^{-1}A\in\Ac(\widetilde \Omega).$
On the  other hand,  by Lemma \ref{lem_change_formula} for every $x\in X$ and for every  $\tilde x\in\pi^{-1}(x).$
$$
W_x(A)=W_{\tilde x}(\pi^{-1}A)= \sum_{n=1}^\infty 
 W_{\tilde x}(\pi^{-1}A_n)= \sum_{n=1}^\infty W_x(A_n).
 $$
  So
 integrating both sides on $(X,\mu),$ and  using formula  (\ref{eq_formula_bar_mu})
we get $\bar\mu(A)= \sum_{n=1}^\infty \bar\mu(A_n)$ as  desired.
 \qed

 \noindent{\bf End of the proof of   Proposition        \ref{prop_algebras} (ii).}
It remains  to show that    $\Ac$ is   approximable by cylinder images when $X=\widetilde X.$
In this  case  $\Ac=\widetilde \Ac:=\widetilde\Ac(\Omega),$  in particular, cylinder images coincide with cylinder sets. 
  Let $\Dc$ be  the family of all  sets $A\subset \Omega$ which are  finite  unions of cylinder sets.
  By Proposition \ref{prop_cylinder_sets}, $\Dc$ is the algebra on $\Omega$  generated  by  cylinder sets.
 So $\Dc\subset \Ac.$ By  Theorem   \ref{thm_Wiener_measure_measurable} (ii),  $\bar\mu$ is   countably  additive  on
$\Dc.$ Consequently, applying Proposition \ref{prop_measure_theory} yields that
$\Ac$ is   approximable by cylinder sets. 
 \qed

 \noindent{\bf End of the proof of   Proposition        \ref{prop_algebras_new}.}
We proceed as  in the proof of Proposition        \ref{prop_algebras} (ii). Let $\Dc$ be  the family of all  sets $A\subset \Omega(\Sigma)$ which are  finite  unions of cylinder sets.
Next, we  make  use   of  Proposition \ref{prop_cylinder_sets_new} instead of     Proposition \ref{prop_cylinder_sets}  in the present context. 
Moreover, in order to show that $\bar\nu$ is  countably  additive on $\Dc,$  it suffices  to  use formula (\ref{eq_formula_bar_mu_new}) and 
   Definition \ref{defi_fibered_laminations_new} (iv). Finally,  applying Proposition \ref{prop_measure_theory} yields that
$\Ac(\Sigma)$ is   approximable by cylinder sets. 
 \qed

 \par 
Before  proceeding to the proof of Proposition \ref{prop_algebras_laminations}, we make some comments on our  method
in this  section.
The  approach adopted  in Section \ref{subsection_algebra_on_a_leaf} which is  heavily  based  on Lemma \ref{lem_lifting_property} and  which results in  Proposition \ref{prop_algebras_leaf}  does not work in the  present context of general laminations. Indeed,
 the holonomy phenomenon  in the case of a  lamination is   much more   complicated   than that in the  case of a single leaf. More concretely,
 instead of  trivializing open sets in the context of a single  leaf  as in  Proposition \ref{prop_algebras_leaf},  we have to work with flow boxes in the context of laminations. However, there exists, in general, no flow  box whose every    plaque     is simultaneously  trivializing.

Our new  approach consists in replacing  cylinder images  with {\it directed  cylinder images}.  Roughly speaking,
 directed  cylinder images   take into account  the holonomy  whereas  the  non-directed ones  do not so.
 
 In proving   Proposition \ref{prop_algebras_laminations} our new  approach consists of two
steps.
First, we use
the separability  and  holonomy result  developed in Section  \ref{subsection_holonomy_maps} and  \ref{subsection_sample-path_spaces} above. Second, we
replace  cylinder images  with   their preimages  in $\widetilde\Omega $   and  express  these preimages in terms of the so-called {\it directed  cylinder sets.}   
\begin{definition}\label{defi_directed_cylinder_sets} \rm 
 Let $\widetilde \U$ be   a  good  flow tube and $N>0$  a given time.
  
A {\it  directed  cylinder set}\index{cylinder!directed $\thicksim$ set} $\tilde A$  with respect to $(N,\widetilde \U)$  
is  the intersection of   a  cylinder set  $C\big (\{t_i,\tilde A_i\}:m )$ in $\widetilde\Omega$ and  the sample-path space 
$\Omega(N,\widetilde \U)$    satisfying the time  requirement $N\geq  t_m=\max t_i.$

 A set $A\subset \Omega$ is  said to be a {\it  directed  cylinder  image}\index{cylinder!directed $\thicksim$ image} (with respect to $(N,\widetilde\U)$) if  $A=\pi\circ\tilde A$  for  some
    directed  cylinder set $\tilde A\subset\widetilde \Omega$ (with respect to $(N,\widetilde\U)$).
  \end{definition}
  
  We  fix   a  countable  family of pairs of conjugate  flow  tubes  $(\{\widetilde \U'_i,\widetilde \U''_i )_{i\in\N}$
satisfying   the conclusion of Theorem  \ref{thm_separability}. More concretely,
let $\U_i$ be  the common image  of $\widetilde \U'_i$ and $\widetilde \U''_i,$ and let 
$\pi'_i:=\pi|_{\widetilde \U'_i}:\ \widetilde \U'_i\to \U_i$ and  $\pi''_i:=\pi|_{\widetilde \U''_i}:\ \widetilde \U''_i\to \U_i$
be two  homeomorphisms
which maps  plaques onto  plaques.
Then  
for every $N>0,$
\begin{equation}\label{eq_cover_good_flow_tubes}
\begin{split}
&\left\lbrace     
(\tilde \omega',\tilde \omega'')\in  \widetilde   \Omega\times\widetilde   \Omega:\ (\pi\circ \tilde \omega')(t)= (\pi\circ\tilde \omega'')(t),\ \forall t\in[0,N]\right\rbrace\\
 &\subset \bigcup_{i\in\N} \Omega(N, \widetilde\U'_i)\times \Omega(N, \widetilde\U''_i).
 \end{split}
\end{equation}
 
Now we  establish some properties  of  directed  cylinder images.

\begin{lemma}\label{lem_directed_cylinder_images_in_Ac}
 Let $\widetilde \U,\ \widetilde \V$ be   two  good  flow tubes and $N>0$  a given time.
\\
1)  Then  $\Omega(N,\widetilde \U)$ is  a  countable union of increasing  sets, each  being a countable  intersection of decreasing cylinder sets.
\\
2)  $\Omega(N,\pi(\widetilde\U))$ is  a  countable union of increasing  sets, each  being a countable    intersection of decreasing cylinder images.
 \\ 
 3) Every  directed cylinder   set (resp. image)   is  a  countable union of increasing  sets, each  being a  countable   intersection of decreasing  cylinder sets (resp. images).
 \\
 4) Every cylinder image   may be represented as a countable union of directed cylinder images (with respect to  flow tubes
$\widetilde\U'_i$).
 \\
5) $\Ac$ coincides with   the  $\sigma$-algebra  on  $\Omega$ generated  by   directed  cylinder images (with respect to  flow tubes
$\widetilde\U'_i$).
 \end{lemma}
\begin{proof} Fix  an increasing   sequence $(\widetilde F_j)_{j=0}^\infty$  of
compact  subsets of $\widetilde\U$  such that,
for every $\omega\in\Omega (N,\widetilde \U),$  there exists $\tilde x\in \widetilde \T$ and  $j\in\N$ such that
$\omega[0,N]\subset \widetilde \U_{\tilde x}\cap \widetilde F_j,$  where  $\widetilde \T$ is  a transversal of  $\widetilde \U.$
  Note that $\widetilde F_j\nearrow\widetilde\U$ as $j\nearrow\infty.$

\noindent{\bf Proof of Part 1).}
Using the  continuity  of each path  in a  sample-path space and  using the density of the rational numbers  in $[0,N],$   we obtain that
\begin{equation}\label{eq_Part1_lem_directed_cylinder_images} 
\Omega(N,\widetilde \U)=\bigcup_{j\in\N}\Big (\bigcap_{i\in\N} C( \{s,\widetilde F_j\}:  s\in S_i)\Big),
\end{equation}
where $\{S_i:\ i\in\N\}$ is  the family of   all  finite sets  $S$  of  rational numbers in $[0,N].$ 
Replacing $S_i$ with $S_0\cup\cdots\cup S_i,$ the  last intersection  does   not change.
Therefore, we may assume  without loss  of generality that  $S_0\subset S_1\subset S_2\cdots,$ and  hence
$\Omega(N,\widetilde \U)$   is  equal to  a  countable union of increasing  sets, each  being 
a countable  intersection of decreasing cylinder sets. This  proves   Part 1).
\\
\noindent{\bf Proof of Part 2).} In what  follows  set $\U:= \pi(\widetilde\U).$
The proof of    Part 2) will be complete  if one can show  that
\begin{equation}\label{eq_Part2_lem_directed_cylinder_images} 
 \Omega(N, \U)=\bigcup_{j\in\N}\Big ( \bigcap_{i\in\N} \pi\circ C( \{s,\widetilde F_j\}:s\in S_i)\Big),
\end{equation}
Let $\omega\in  \Omega(N, \U).$ Since $\pi|_{\widetilde \U}:\ \widetilde \U\to  \U$ is
a homeomorphism which maps plaques onto plaques, we see that  $(\pi|_{\widetilde \U})^{-1}(\omega)\in \Omega(N,\widetilde \U),$
and hence  $(\pi|_{\widetilde \U})^{-1}(\omega)\in C( \{s,\widetilde \U\}:s\in S_i)$ for every $i\in\N.$ 
This, combined with the property of $(\widetilde F_j)_{j=0}^\infty,$  implies that 
$\omega\in \bigcup_{j\in\N}\bigcap_{i\in\N} \pi\circ C( \{s,\widetilde F_j\}:s\in S_i).$

Conversely,   we pick  an arbitrary path $\omega \in \bigcap_{i\in\N} \pi\circ C( \{s,\widetilde F_j\}:s\in S_i) $ for some $j\in\N,$ and show that  $\omega\in \Omega(N, \U).$ The choice of $\omega$
implies that $\omega(t)\in  \U$ for all $t\in\Q\cap [0,N].$
Since  $\omega$ is  a leafwise  continuous map and  $\U$ is  a flow tube and  the intersection of  $\pi(\widetilde F_j)$ with
each plaque  of $\U$  is  compact, we  infer that
 $\omega[0,N]$ is  contained in  a plaque  of  $ \U.$ Hence,  $\omega\in \Omega(N, \U),$ as  desired.



\noindent{\bf Proof of Part 3).}
Let $\tilde A= C (\{t_i,\tilde A_i\}:p )$  be  a cylinder set in $\widetilde\Omega ,$ and
let $A:=\pi\circ\tilde A$     its images. Fix a time $N\geq t_p.$
Arguing as in the proof of  (\ref{eq_Part1_lem_directed_cylinder_images}), we  see that
$$ 
\tilde A\cap \Omega(N,\widetilde \U)=\bigcup_{j\in\N}\Big (\tilde A\cap \bigcap_{i\in\N} C( \{s,\widetilde F_j\}:  s\in S_i)\Big).
$$
So every  directed cylinder   set    is  a  countable union of increasing  sets, each  being a  countable   intersection of decreasing  cylinder sets.
Next,
arguing as in the proof of  (\ref{eq_Part2_lem_directed_cylinder_images}), we  see that
 $$
\pi\circ(\tilde A\cap \Omega(N,\widetilde \U))=  A\cap \Omega(N, \U)= \bigcup_{j\in\N} \Big (\bigcap_{i\in\N} \pi\circ(\tilde A\cap C( \{s,\widetilde F_j\}:s\in S_i))\Big).
$$
So every  directed cylinder   image    is  a  countable union of increasing  sets, each  being a  countable   intersection of decreasing  cylinder images.  

\noindent{\bf Proof of Part 4).}
We first  deduce  from   
 (\ref{eq_cover_good_flow_tubes}) that
 $$
     \widetilde\Omega  = \bigcup_{i\in\N} \Omega(N, \widetilde\U'_i) .
 $$
 This   implies that
 $$ \tilde A=   \bigcup_{i\in\N} C (\{t_i,\tilde A_i\}:p )\cap\Omega(N,\widetilde \U'_i).$$
Acting $\pi$ on  both sides,  Part 4) follows.
 
\noindent{\bf Proof of Part 5).}
 Recall that $\Ac$ is  the $\sigma$-algebra  generated by all cylinder images. This, coupled with Part 4) implies that  $\Ac$  is contained  in  the  $\sigma$-algebra  on  $\Omega$ generated 
 by   directed  cylinder images (with respect to  the flow tubes
$\widetilde\U'_i$). By Part 3)  the inverse   inclusion  is  also true.
Hence,   Part 5) follows. 
 \end{proof}
\begin{remark}\label{R:R+}
Observe  that  the proof of Lemma \ref{lem_directed_cylinder_images_in_Ac}   is not valid  any more if we
replace the  set of times $\R^+$ by  a discrete    semi-group $\N t_0$ for some $t_0>0.$ 
This  observation shows that in order to deal with the holonomy phenomenon, we have to work with cylinders whose  times vary in the whole $\R^+.$  
\end{remark}

 The  following result is very useful.
\begin{lemma}\label{lem_preimage_cylinder_image}
 For  each  cylinder image $ A,$ its  preimage  
$\pi^{-1}( A)$ may be represented  as  
$ \bigcup_{i\in\N}  A_i,$
where  $A_i$ is  a    directed cylinder set with respect  to $\widetilde\U'_i$.  
\end{lemma}    
\begin{proof}  
Let $A=\pi\circ\tilde A,$ where
$
\tilde A= C (\{t_j,\tilde A_j\}:p ) ,$ and fix  a time  $N\geq  t_p=\max t_j.$
 For each  $1\leq  j\leq  p$ and  each  $i\in\N$ consider the  Borel subset $\tilde B^i_j$ of $\U'_i$  given by
$$
\tilde B^i_j:=  (\pi'_i)^{-1}  ( \pi''_i( \tilde A_j\cap \U''_i)).
$$
Since $\pi'_i:=\pi|_{\widetilde \U'_i}:\ \widetilde \U'_i\to \U_i$ and  $\pi''_i:=\pi|_{\widetilde \U''_i}:\ \widetilde \U''_i\to \U_i$
are two  homeomorphisms
which maps  plaques onto  plaques one  can show that
$$
\pi\circ \Big(  C (\{t_j,\tilde B^i_j\}:p )\cap  \Omega(N,\U'_i)\Big)
=\pi\circ \Big(  C (\{t_j,\tilde A_j\}:p )\cap  \Omega(N,\U''_i)\Big).
$$
Since the right hand  side  is  contained in  $A,$ it follows  that
$$
\bigcup_{i\in\N}   C (\{t_j,\tilde B^i_j\}:p )\cap  \Omega(N,\U'_i)\subset \pi^{-1}(A).
$$
Consequently,  the proof will be  complete if one  can show that the  above  inclusion is,  in fact,
an  equality. To do this
pick an  arbitrary path  $\tilde\omega'\in \pi^{-1}(A).$ So there  exists  $\tilde\omega''\in \tilde A$ such that
$(\pi\circ \tilde \omega')(t)= (\pi\circ\tilde \omega'')(t)$ for all $t\in[0,N].$ So  by
(\ref{eq_cover_good_flow_tubes}), there  exists $i\in\N$ such that
 $
(\tilde \omega',\tilde \omega'')\in    \Omega(N, \widetilde\U'_i)\times \Omega(N, \widetilde\U''_i).$
  Hence,  $\tilde \omega''\in   C (\{t_j,\tilde A_j\}:p )\cap  \Omega(N,\U''_i)$ and  
 $\tilde \omega'\in C (\{t_j,\tilde B^i_j\}:p )\cap  \Omega(N,\U'_i),$ as  desired.
 
\end{proof}
 




Now  we arrive at the
\smallskip

 \noindent{\bf End of the  proof of   Proposition    \ref{prop_algebras_laminations}.}
 
  \noindent{\bf Proof of $\widetilde\Ac\subset \Ac:$} 
 Let $A:= C(\{A_i,t_i\}:p)$ be  a cylinder set in $\Omega.$
For each  $1\leq i\leq p$  let  $\tilde A_i:=\pi^{-1}(A_i).$  Consider  the cylinder set
$\tilde A:=  C(\{\tilde A_i,t_i\}:p)$  in $\widetilde\Omega.$
We can check  that  $A=\pi\circ \tilde A.$ So  every cylinder set is  also a cylinder image.
Hence,  $\widetilde\Ac\subset \Ac,$ as  desired.

 \noindent{\bf Proof of $\pi^{-1}(\Ac)\subset \Ac(\widetilde \Omega):$} 
Recall that $\Ac$ is  the $\sigma$-algebra generated by all cylinder  images.
Consequently,  the family $\{\pi^{-1}(A):\ A\in \Ac\}$  is  the $\sigma$-algebra on $\widetilde\Omega$
generated  by all sets of the form $\pi^{-1}(A)$ with $A$ a  cylinder image.
On the other hand,  
combining  Lemma \ref{lem_preimage_cylinder_image} and Part 3) of Lemma \ref{lem_directed_cylinder_images_in_Ac}
it follows  that $\pi^{-1}(A)\in \Ac(\widetilde \Omega)$ for each  cylinder image  $A.$
Hence,  $\pi^{-1}(\Ac)\subset \Ac(\widetilde \Omega),$ as  asserted.

 \noindent{\bf Proof of $\Ac=\pi\circ  \Ac(\widetilde \Omega):$}
Since  we have  shown in the previous paragraph that  $\pi^{-1}(\Ac)\subset \Ac(\widetilde \Omega),$ it follows that $\Ac\subset \pi\circ  \Ac(\widetilde \Omega).$
Therefore,  it remains  to  show  the inverse  inclusion  $ \pi\circ  \Ac(\widetilde \Omega)\subset \Ac.$
To this  end consider  the  family
$$
\widetilde\Cc:=\left\lbrace  \tilde A\in\Ac(\widetilde\Omega):\ \pi^{-1}(\pi\circ \tilde A)=\tilde A\ \text{and}\ \pi\circ\tilde A\in  \Ac  \right\rbrace.
$$
It is  worthy noting that by  the third  $\bullet$ in Definition \ref{defi_invariance_under_deck-transformations}, 
for  a  set $\tilde A\in\Ac(\widetilde\Omega),$ the  equality  $\pi^{-1}(\pi\circ \tilde A)=\tilde A$ holds if  and only
if 
 $\tilde A$ is  invariant  under deck-transformations.

 Let  $\tilde A$ be  a cylinder set in $\widetilde\Omega.$ Then  $\pi\circ \tilde A\in\Ac,$ and hence
 $\pi^{-1}(\pi\circ \tilde A)\in  \Ac(\widetilde\Omega)$ by using the  inclusion  $\pi^{-1}(\Ac)\subset \Ac(\widetilde \Omega)$ that 
 we have  already proved.
  Moreover, since $\pi\circ(\pi^{-1}(\pi\circ \tilde A))=\pi^{-1}(\pi\circ \tilde A),$
  we infer  that $\pi^{-1}(\pi\circ \tilde A)\in  \widetilde\Cc.$ 
  Let $\widetilde\Dc$ be the  algebra  of all finite unions of cylinder sets in $\widetilde\Omega.$
  We deduce easily from the previous discussion that
  $\pi^{-1}(\pi\circ \tilde A)\in  \widetilde\Cc$ for each $\tilde A\in \widetilde\Dc.$ 

Next,  if $(\tilde A_n)_{n=1}^\infty\subset \widetilde\Cc$ such that $\tilde A_n\subset \tilde A_{n+1}$ for all $n,$ then
 $\pi\circ (\cup_{n=1}^\infty \tilde A_n)= \cup_{n=1}^\infty \pi \circ\tilde A_n \in\Ac$
and  $\pi^{-1}\big (\pi\circ (\cup_{n=1}^\infty \tilde A_n)\big) = \cup_{n=1}^\infty \tilde A_n.$
Hence, $\cup_{n=1}^\infty \tilde A_n\in 
 \widetilde\Cc.$ 
 
Analogously,   if $(\tilde A_n)_{n=1}^\infty\subset  \widetilde\Cc$ such that $\tilde A_{n+1}\subset \tilde A_n$ for all $n,$ then
 $\pi\circ (\cap_{n=1}^\infty \tilde A_n)= \cap_{n=1}^\infty \pi \circ\tilde A_n \in\Ac$ since  each $\tilde A_n$ is  invariant  under deck-transformations.  Consequently,
we can show that  $\cap_{n=1}^\infty \tilde A_n\in   \widetilde\Cc.$
 Therefore,  applying Proposition \ref{prop_criterion_sigma_algebra}
 yields that $\pi\circ \widetilde\Ec\subset \Ac,$
 where  $\widetilde\Ec$ is the $\sigma$-algebra generated  by all the sets of the form $\pi^{-1}(\pi\circ \tilde A),$
 with $\tilde A$ a cylinder set. Using the transfinite induction and the identity $\pi\circ (\pi^{-1}(\pi\circ \tilde A))= \pi\circ \tilde A ,$
 it is  not difficult to show that $\pi\circ \widetilde\Ec=\pi\circ \Ac(\widetilde\Omega).$
 Hence,  $ \pi\circ  \Ac(\widetilde \Omega)\subset \Ac,$ as  desired.
 \hfill $\square$ 

\noindent {\bf End of the proof of  Proposition \ref{prop_cocycle_criterion}.}
The  proof is  divided into three steps. 
In the first  two steps
we  show that   the measurability of   local  expressions  implies   the measurable law.
 The last step is  devoted to the proof of the inverse implication.
  By Definition \ref{defi_local_expression} assume without loss of generality that $t_0=1.$

To start with the first implication   
it is  sufficient  to  show that   the map $\mathcal A(\cdot,t):\ \Omega\times\R^d\to\R^d$  given by
$(\omega,u)\mapsto  \mathcal A(\omega,t)u$ is  measurable for every fixed $t\in \G.$ Without loss 
of generality  we may assume that $t=1.$  Working with the covering lamination $ (\widetilde X,\widetilde \Lc)$ and
transferring the results back to $(X,\Lc)$ via  the projection $\pi,$ we may also assume that
$X=\widetilde X,$ that is, all leaves are simply connected. This implies that $\Ac=\widetilde\Ac.$
Moreover, we will make full use of the consequence of the homotopy law in Definition \ref{defi_cocycle} that
$\mathcal A(\omega,t)$  depends  only on $\omega(0)$ and $\omega(t)$ for each $t>0.$
 Let  $O\subset \R^d$ be  a  Borel set. 

\noindent{\bf Step 1: } {\it Given a flow  box $\Phi:\ \U\to \B\times \T,$
   the set 
\begin{multline*}
S_{\U,O}:=\left\lbrace  (\omega,u)\in  \Omega\times\R^d:\   \text{$\omega(0)$ and $\omega(1)$ 
live in a common plaque of $\U$  and}\right.\\
\left.     \mathcal A(\omega,1)u\in O\right\rbrace
\end{multline*}
is  measurable.}
 To do this   let $\alpha$ be the local expression of $\mathcal A$  on this  flow  box. By hypothesis,  $\alpha$
is  measurable. Consequently,  $\{(x,y,t,u)\in \B\times \B\times \T\times\R^d: \alpha(x,y,t)u\in O\}$ is a measurable set.
Hence,   the  set   $$P:=\{ (\Phi^{-1}(x,t),\Phi^{-1}(y,t),u):\  (x,y,t,u)\in \B\times \B\times \T\times\R^d: \alpha(x,y,t)u\in O\}$$
is  also measurable  in $X\times X\times \R^d.$

By the  construction of  the $\sigma$-algebra  $\Ac,$  we  see that, for  any set $Q$ belonging to  the product of Borel  
$\sigma$-algebras 
$\Bc(X)\times \Bc(X)\times\Bc (\R^d),$
the  {\it generalized  cylinder}
$$
C(0,1;Q):=  \left\lbrace (\omega,u)\in\Omega\times\R^d:\  (\omega(0),\omega(1),u)\in Q  \right\rbrace
$$
belongs to the product of $\sigma$-algebras  $\Ac\times \Bc(\R^d).$   
Thus, $
C(0,1;Q)$ is  measurable.  This,  combined   with  the  equality  
  $S_{\U,O}=
C(0,1;P)$ 
 implies that $S_{\U,O}$ is also  measurable  as  desired.
 
 \noindent{\bf Step 2: }{\it 
 Measurability of   local  expressions  implies    measurable law.}
 
 Next, we consider   a given  flow  tube $\U$  and  define  $S_{\U,O}$ as in the  case of a flow  box.
 Observe that  each flow  tube may be covered  by a finite  number of flow  boxes  and  that $\mathcal A(\omega,1)$  depends  only on $\omega(0)$ and $\omega(1),$
Consequently, the  argument  used in Step 1 still works in the context of flow tubes
using the local expression of $\mathcal A$  on different finite flow  boxes that  covers $\U$ and  making the obviously necessary changes.  
 So we can also  prove  that $S_{\U,O}$  is  measurable for each flow tube $\U.$

On the other hand, by Part 2)  of Theorem
 \ref{thm_separability},
 there exists   
 a  countable  family of    flow  tubes  $(  \U_i )_{i\in\N}$
 such that
 $\Omega=\bigcup_{i\in\N} \Omega(1,\U_i).$
 Therefore,
 $$\left\lbrace  (\omega,u)\in  \Omega\times\R^d:\    \mathcal A(\omega,1)u\in O\right\rbrace
=\bigcup_{i\in\N}S_{\U_i,O} $$
is also  measurable. Since this  is true for each  Borel set $O,$  Step 2 is complete.

 \noindent{\bf Step 3: }{\it 
    Measurable law implies measurability of   local  expressions.}
   Let  $O\subset \R^d$ be  a  Borel set  and $\U$ a  flow  box.  Since $\mathcal A(\omega,1)$  depends  only on $\omega(0)$ and $\omega(1),$
we only  need to show  that the set 
\begin{eqnarray*}
\left\lbrace (x,y,u)\in   \U^2\times\R^d:\   \text{$x$ and $y$ 
live in a common plaque of $\U$ and}\  \right.\\
\left.  \text{there is $\omega\in\Omega$ with $\omega(0)=x$ and $\omega(1)=y$ and}\ \mathcal A(\omega,1)u\in O   \right\rbrace 
\end{eqnarray*}
    is measurable. This is reduced, in turn,  to showing that  the set $S_{\U,O}$ is  measurable.
  Write the  last set as
    $$
    S_{\U,O}=\Omega(1,\U)\cap  \left\lbrace (\omega,u)\in \Omega\times\R^d:\ \mathcal A(\omega,1)u\in O   \right\rbrace.
    $$
    The first  set on the right hand side is measurable  by  Part 2) of Lemma \ref{lem_directed_cylinder_images_in_Ac}, whereas  the second one
    is  measurable  by  the measurable  law  applied to the cocycle $\mathcal A.$
    This completes the proof.
 \hfill $\square$

\section{The cylinder laminations are Riemannian continuous-like}
\label{section_cylinder_continuous_like}
 
 Let $(X,\Lc,g)$ be a Riemannian lamination satisfying  the Standing Hypotheses and  set  $\Omega:=\Omega(X,\Lc)$ as usual.
 Let  $\mathcal{A}:\ \Omega\times \N \to  \GL(d,\R)      $ be     a (multiplicative) cocycle.  Let $k$ be  an integer with $1\leq k\leq d.$
 
 Following   Definition  \ref{defi_cylinder_lamination} and Remark \ref{rem_cylinder_lamination},   let  $(X_{k,\mathcal A},\Lc_{k,\mathcal A})$ be 
the     cylinder lamination of rank $k$\index{lamination!cylinder $\thicksim$}  of the cocycle   $\mathcal A.$    
Note that   $X_{k,\mathcal A}=X\times \Gr_{k}(\R^d)$ which is independent of $\mathcal A.$
 Let  $\Omega_{k,\mathcal A}:=\Omega  (X_{k,\mathcal A},\Lc_{k,\mathcal A}).$ 
 Clearly, when $k=d$ we have that $(X,\Lc)\equiv (X_{d,\mathcal A}, \Lc_{d,\mathcal A}).$
 The leaves  of $(X_{k,\mathcal A},\Lc_{k,\mathcal A})  $ are  equipped with the  metric $\pr_1^*g,$
 where $\pr_1:\ X\times \Gr_{k}(\R^d)\to X $ is   the canonical  projection on the  first factor.
Hence, $(X_{k,\mathcal A},\Lc_{k,\mathcal A},\pr_1^*g )  $ is a  Riemannian  measurable lamination. 

Let  $(\widetilde X,\widetilde\Lc)$ be the covering lamination of $(X,\Lc)$ together with the corresponding   covering    lamination projection
$$
\pi:\ (\widetilde X,\widetilde\Lc,\pi^*g)\to (X,\Lc,g).
$$
     Set  $\widetilde\Omega:=\Omega(\widetilde X,\widetilde\Lc)$ as usual.
Following (\ref{eq_cocycle_on_cover}) we construct a  cocycle  $\widetilde{\mathcal A}$ on $(\widetilde X,\widetilde\Lc)$ given by the formula:
\begin{equation*}
\widetilde{\mathcal A}(\tilde\omega,t):=\mathcal A(\pi\circ \tilde\omega,t),\qquad t\in\R^+,\ \tilde\omega\in\widetilde\Omega .
\end{equation*}
Consider the cylinder lamination of rank $k$ $(\widetilde{X}_{k,\widetilde{\mathcal A}},\widetilde{\Lc}_{k,\widetilde{\mathcal A}})  $
of the  cocycle  $
\widetilde{\mathcal A}.$ The leaves  of $(X_{k,\mathcal A},\Lc_{k,\mathcal A})  $ are  equipped with the  metric $\tilde\pr_1^*(\pi^*g),$
 where $\tilde \pr_1:\ \widetilde X\times \Gr_{k}(\R^d)\to \widetilde X $ is   the canonical  projection on the  first factor.
Hence, $( \widetilde X_{k, \widetilde {\mathcal A}}, \widetilde \Lc_{k, \widetilde{ \mathcal A}}, \tilde\pr_1^*(\pi^*g) )  $ is a  Riemannian  measurable lamination. 
 Consider the  canonical projection
$$
\Pi:\ \widetilde X_{k, \widetilde {\mathcal A}}=\widetilde X\times \Gr_{k}(\R^d)\to X\times \Gr_{k}(\R^d)=X_{k,\mathcal A},
$$
given by $$\Pi(\tilde x,U):=(\pi(\tilde x),U),\qquad (\tilde x,U)\in \widetilde X_{k, \widetilde{ \mathcal A}}.$$
 It is immediate  to see that this is a covering measurable lamination
in the sense of Section   \ref{subsection_Covering_laminations}. Moreover,  we have the diagram:
\begin{equation}\label{e:pi_cover_cyl}
\Pi:\ \Big(\widetilde X_{k, \widetilde {\mathcal A}},\widetilde\Lc_{k, \widetilde {\mathcal A}},   \tilde\pr_1^*(\pi^*g)   \Big) \to  \Big(X_{k,\mathcal A},
\Lc_{k, \mathcal A},\pr_1^*g\Big),
\end{equation}
that is, $\Pi^*(\pr_1^*g)=\tilde\pr_1^*(\pi^*g).$
The purpose  of this section is   to the following result.
\begin{theorem}\label{T:cylinder_lami_is_conti_like}  Under the above hypotheses and  notation, 
$(X_{k,\mathcal A},\Lc_{k,\mathcal A},\pr_1^*g )  $
is   a  Riemannian   continuous-like lamination, and its  covering measurable lamination is 
$( \widetilde X_{k, \widetilde {\mathcal A}}, \widetilde \Lc_{k, \widetilde{ \mathcal A}} )  ,$ and its covering lamination projection
is given by  (\ref{e:pi_cover_cyl}).
\end{theorem}
\begin{proof}
The materials  developed in Appendix  \ref{subsection_holonomy_maps}, \ref{subsection_sample-path_spaces}, 
 \ref{subsection_leafwise_diagonal} and  \ref{subsection_algebra_on_a_lamination} for  Riemannian (continuous) laminations  are still valid
in the  present context of cylinder laminations making  the following relevant  modifications.  The  notions on the left-hand side
of the  following  dictionary correspond to  the cylinder lamination $\Big(X_{k,\mathcal A},
\Lc_{k, \mathcal A},\pr_1^*g\Big):$
\begin{eqnarray*}
\text{flow  box}& =&\text{flow box   of $(X,\Lc)$  $ \times$   $\Gr_{k}(\R^d)$},\\
\text{small flow box}& =&\text{small flow box of $(X,\Lc)$  $ \times$   $\Gr_{k}(\R^d)$},\\
\text{transversal of a flow  box}& =&\text{transversal of a flow box of $(X,\Lc)$ $ \times$  $\Gr_{k}(\R^d)$},\\
\text{transversal of a small flow box}& =&\text{transversal of a small flow box of $(X,\Lc)$ $ \times$  $\Gr_{k}(\R^d)$},\\
\text{flow   tube}& =&\text{flow  tube of $(X,\Lc)$  $ \times$   $\Gr_{k}(\R^d)$},\\
\text{transversal of a flow tube}& =&\text{transversal of a flow tube of $(X,\Lc)$ $ \times$  $\Gr_{k}(\R^d)$}.
\end{eqnarray*}
Analogously, we define  the notion of a  chain (resp. a  good chain) of small  flow  boxes
and the notion of the  flow tube  associated to such a good  chain.

Similarly as in Step 4  in the proof of Theorem   \ref{thm_separability}, we define  the set of all homotopies
of a flow tube $\V\times \Gr_{k}(\R^d)$ of the cylinder lamination as  follows.
By the above rule of modifications, $\V$ is a flow  tube of the Riemannian lamination $(X,\Lc,g).$
Define   $\Hc(\V\times \Gr_{k}(\R^d))$ to be  simply  the set  $\Hc(\V)$  of all homotopies $\alpha:\ \V\times [0,1]\to X$ (see Section  \ref{subsection_Covering_laminations}).
 Clearly,  $\Hc(\V\times \Gr_{k}(\R^d) )$ is   at most  countable.
We leave it to the interested reader to fill in the details of the proof.
\end{proof}

 \section{The  extended sample-path space  is of full outer measure}\index{measure!outer $\thicksim$}
 \label{section_full_outer_measure}

 Recall  that $(X,\Lc,g)$  is a  Riemannian measurable lamination satisfying Hypothesis (H1).
The main purpose  of this  section is  to prove 
 Theorem
\ref{thm_Brownian_motions_new}, that is, $\widehat\Omega:=\widehat\Omega(X,\Lc)$ is  of full outer measure.\index{measure!outer $\thicksim$}
 In fact, we  will adapt the proof of the corresponding result 
for the sample-path space  $\Omega:=\Omega(X,\Lc)$ (see Theorem C.2.13 in  \cite{CandelConlon2}). 
Prior to the proof, we  need  to introduce  some more notation and terminology.
Recall from Subsection  \ref{ss:Wiener-III} that  $\widehat{\mathfrak S}$ (resp. $\widehat{\mathfrak C}$ )  is the algebra (resp. the  $\sigma$-algebra)  generated by all cylinder sets
with real time,   and that  for each $n\in \N,$  $\widehat{\mathfrak C}_{-n}$  is  the  $\sigma$-algebra  generated by all cylinder sets
with  time $\geq  -n.$

Let $T:=T^1$ be the shift-transformation of unit-time\index{shift-transformation!$\thicksim$ of unit-time} defined in (\ref{eq_shift_real}).   
\begin{lemma}\label{lem_countable_agreement}
If $A\in  \widehat{\mathfrak C},$  then there is a  countable  subset $Q\subset \R$  with the   property that, whenever $\omega\in X^\R$ agrees  with some $\eta\in A$
on the set $Q,$ then $\omega\in A.$
\end{lemma}
\begin{remark} The lemma is a  generalization of  Corollary C.2.3 in \cite{CandelConlon2} to the context of extended sample-path spaces.
\end{remark}
\begin{proof}
 Let $\Dc$ be the  family of all elements $A\in  \widehat{\mathfrak C}$ having the  property described in   the lemma. So we need
 to show that  $\Dc= \widehat{\mathfrak C}.$
 
 It can be easily  checked that  an element which can be represented as a finite union of cylinder sets  belongs to $\Dc.$
On the other hand, we can show, by a similar argument  as in the proof of  Proposition \ref{prop_cylinder_sets},
that the algebra $\widehat{\mathfrak S}$ consists of all finite unions of cylinder sets. Hence,  $\widehat{\mathfrak S}\subset\widehat{\mathfrak C}. $
  
  Next, we can  check  that 
 \\
$\bullet$ if $(A_n)_{n=1}^\infty\subset \Dc$ such that $A_n\subset A_{n+1}$ for all $n,$ then
 $\cup_{n=1}^\infty A_n\in\Dc;$
 \\
 $\bullet$  if $(A_n)_{n=1}^\infty\subset \Dc$ such that $A_{n+1}\subset A_n$ for all $n,$ then
 $\cap_{n=1}^\infty A_n\in\Dc.$
 
 Consequently, applying Proposition \ref{prop_criterion_sigma_algebra} 
 the lemma follows.
  \end{proof}
Define  the  function  distance  $\dist:\ X\times X\to[0,\infty]$ as follows.
If two points $x,y\in X$ are in  a  common leaf, then set $\dist(x,y)$ to be the distance  between these points
with respect to the metric on  this leaf induced  by $g.$  Otherwise, we set simply $\dist(x,y):=\infty.$
\begin{lemma}\label{lem_countable_estimate}
Let $F$
 be  a  countable  subset of an interval $[a,b]$
and let $B$ be  the set
$$
B:= \bigcup_{s,t\in F:\ |s-t|<\delta} \{\omega\in X^{[0,\infty)}:\  \dist(\omega(s),\omega(t))\geq \epsilon\}.
$$
Then  for every $x\in X,$
$$
W_x(B)\leq  2(b-a){H(\epsilon,2\delta)\over\delta}.
$$
Here  $H:\ (0,\infty)\times (0,\infty)\to [0,\infty)$ is a  function  that satisfies the growth  condition  
$$
\lim_{t\to 0} H(\epsilon,t)/t=0.
$$  
\end{lemma}
\begin{proof}
See  Corollary C.4.1, Lemma C.4.2  and Lemma C.4.3 in \cite{CandelConlon2}.
\end{proof}
\begin{remark} \label{rem_countable_estimate}
An  interesting point  of Lemma \ref{lem_countable_estimate}
is that the upper bound  for $
W_x(B)$  does  not depend on the set $Q\subset [a,b]$  provided that  it is at most countable.
Moreover,  the estimate  depends only on $\epsilon,\delta$ and $b-a.$
The  function $H$ is  determined in terms of the  geometry of the leaves of the lamination.
Consequently, since  $\{s+n:\ s\in F\}\subset [a+n,b+n],$
we  infer  from  Lemma \ref{lem_countable_estimate} that 
$$
W_x(T^nB)\leq  2(b-a){H(\epsilon,2\delta)\over\delta},\qquad n\in \N.
$$
\end{remark}
\noindent{\bf End of the proof of  Theorem
\ref{thm_Brownian_motions_new}.}
Suppose  in order to reach  a contradiction that  there exists a  set $A\in \widehat{\mathfrak C}$ that is  disjoint from $\widehat\Omega$ and  with
$\widehat W^*_{x_0}(A)>0$ for some  point $x_0\in X.$
By  Lemma \ref{lem_countable_agreement}, there is a  countable subset $Q\subset \R$  such that,   if $\omega\in X^\R$ and there exists
$\omega'\in A$ for which $\omega(t)=\omega'(t)$ for all $t\in Q,$ then $\omega\in A.$

Let $Q_p:= Q\cap [-p,p],$ $p\in \N\setminus \{0\}.$ For each  $p,q,r\in \N\setminus \{0\},$ let
  $$
  B(p,q,r):=\bigcup_{s,t\in Q_p:\ |t-s|<1/r}  \{\omega\in X^\R:\  \dist(\omega(s),\omega(t)) \geq  1/q\}.
  $$
Consider  the  sets
$$
B_m:= \bigcup_{p=1}^m \bigcup_{q=1}^\infty \bigcap_{r=1}^\infty B(p,q,r)\quad\text{and}\quad B:=\bigcup_{m=1}^\infty B_m.
$$
Observe   that  $B(p,q,r)\in \widehat {\mathfrak C}_{-p}$ and that  $B(p,q,r)$ decreases as $r\to\infty.$ 
So $B_m\in \widehat {\mathfrak C}_{-m}$ and $ B\in  \widehat {\mathfrak C}.$
Moreover, by Lemma \ref{lem_countable_estimate} and Remark  \ref{rem_countable_estimate},  we  infer that, given any  $\epsilon>0$  and  any $p,q\in\N\setminus \{0\},$  
there is $r(\epsilon,p,q)$ such that 
$$
 W_x(T^n B(p,q,r))< \epsilon/ 2^{p+q},\qquad n\geq p,\ r\geq r(\epsilon,p,q),\ x\in X.
$$
 This implies that, for every $n\geq  m,$ 
 $$W_x(T^nB_m)\leq  \sum_{p=1}^m\sum_{q=1}^\infty W_x\big (T^nB(p,q,r(\epsilon,p,q))\big )<
 \sum_{p=1}^m\sum_{q=1}^\infty\epsilon/ 2^{p+q}<
\epsilon.$$  Since $\epsilon>0$ is  arbitrary, we have  shown that  $W_x(T^nB_m)=0$ for all $n\geq m$ and  $x\in X.$
 Hence, by formulas (\ref{eq_formula_W_m_extended})-(\ref{eq_formula_W_extended}),
 $$
 \widehat W^*_{x_0}(B_m)= \lim_{n\to\infty} D_1\big( D_n(W_\bullet(  T^nB_m))\big)(x_0)=0,
 $$
 which implies, in turn,  that 
 $$
 \widehat W^*_{x_0}(B)=\lim_{m\to\infty}\widehat W^*_{x_0}(B_m)= 0.  
 $$
 Let $C:= X^\R\setminus B.$ Then $\widehat\Omega\subset C$ because elements of $C$  have  the  property of being
 uniformly  continuous  when  restricted to each $Q_p.$ On the other hand, the intersection $A\cap C$ is  not empty
 because $ \widehat W^*_{x_0}(C)=1- \widehat W^*_{x_0}(B)=1$ and  by  our assumption $\widehat W^*_{x_0}(A)>0.$
If $\omega\in A\cap C,$ then $\omega$ is  uniformly  continuous  when restricted to each interval $Q_n,$ so there is $\omega'\in\widehat\Omega$  
 which agrees with $\omega$ on $Q.$ By the property of $Q,$ this  implies that $\omega'$ also belongs to $A,$ contradicting the  assumption that
$A\cap  \widehat\Omega=\varnothing.$
 This completes  the proof.
 \qed


 \chapter[Harmonic  measure theory and ergodic theory] {Harmonic  measure  theory and ergodic theory  for sample-path spaces}


  \section{Fibered laminations}
  \label{subsection_fibered_laminations}
  
 The purpose  of this  section is  to 
complete  the proofs of Lemma \ref{lem_sup_measurable},  Proposition \ref{prop_Sigma_k_is_a_weakly_fibered_lamination} and
 Proposition    \ref{prop_Sigma_k_is_a_fibered_lamination} which have been stated and partially proved  in  Section \ref{subsection_Stratifications} above. This section is  divided into two parts.
The first one discuss the relation between the saturations of  sets
and harmonic probability measures. Here  we realize   the  important difference of harmonic measures  in comparison with weakly harmonic ones.
Moreover, the  geometric and topological  aspects of a (continuous)  lamination come into play in our study.
The second part is devoted  to the proofs of   the  above mentioned results.

\subsection{Harmonic measures and geometry of continuous laminations}
\label{subsection_harmonic_measures}
To start with the first part, let $(X,\Lc,g)$ be a Riemannian lamination  satisfying the Standing Hypotheses.
Let   $\U$ be    a flow tube   with  transversal $\T.$
For a set $Y\subset\U,$
\\
$\bullet$ the {\it projection} of $Y$ onto $\T,$ denoted by $\T_Y,$ is given by
$$
\T_Y:=\{t\in\T:\ \U_t\cap Y\not=\varnothing\},
$$   
where $\U_t$ is  the plaque  passing through $t\in\T;$
\\
$\bullet$ the  {\it plaque-saturation}\index{plaque!$\thicksim$-saturation (in a flow tube)} of $Y$ in $\U,$ denoted by $\Satur_{\U}(Y),$ is given by
$$
\Satur_{\U}(Y):= \bigcup_{t\in\T:\ \U_t\cap Y\not=\varnothing} \U_t;
$$
$\bullet$ $Y$ is  said to be  {\it plaque-saturated}\index{plaque!$\thicksim$-saturated (in a flow tube)} if $
Y=\Satur_{\U}(Y).$
\nomenclature[b9f]{$ \Satur_\U(Z)$}{plaque-saturation  of a set $Z$  in  a  flow tube $\U$}

 Fix  a (at most)  countable  cover  of $X$ by  flow tubes $\U_i$ with  transversal $\T_i$ indexed by the (at most countable)
set  $I.$  
For each $i\in I$  set $I_i:=\{j\in I:\  \U_j\cap U_i\not=\varnothing\}.$
 For a single  set $Y\subset X,$ or more generally for an array of sets  $(Y_{i})_{i\in I}$  with
$Y_i\subset \U_i$   for each $i\in I,$
we define  an increasing  sequence $(Y_{ip})_{p=0}^\infty$ of plaque-saturated  sets in $\U_i$ as  follows:

$\bullet$ If we are in the case  of a single set $Y$ then  we set $Y_i:=Y\cap  \U_i$ for each $i\in I.$

$\bullet$  For each $i\in I,$ set  
\begin{equation}\label{eq_approximation_Y}
Y_{ip}:=\begin{cases}
\Satur_{\U_i} \Big (\U_i\cap  \cup_{j\in I_i} Y_{j,p-1}   \Big), &  p\geq 1;
\\
 \Satur_{\U_i}( Y_i), &  p= 0.
 \end{cases}
\end{equation}
The  following result allows us to compute  the  saturation of such a set $Y$ (resp.  the saturation of the union $Y:=\bigcup_{i\in I}Y_{i}$)   using  the  above approximating  sequence.  
\begin{proposition} \label{prop_plaque_saturated_approximation}
1) Under the above  hypotheses and notation,  then
for each $i\in I,$ $Y_{ip}\nearrow \Satur(Y)\cap \U_i$ as $p\nearrow\infty.$
\\
2)  If, moreover,   $(\T_i)_{Y_i}$ is a Borel set for each $i\in I,$
then  $\Satur(Y)$ is also a Borel set.  
\end{proposition}
\begin{proof}
We leave  the proof Part 1) to the  interested reader since  it is not difficult.

Now  we turn to Part 2). By the  hypothesis, $Y_{i0}$ is a Borel set for each $i\in I.$
 This, combined with (\ref{eq_approximation_Y}), implies that 
each $Y_{ip}$ is  a Borel set.
This, coupled with  Part 1) implies that $\Satur(Y)\cap \U_i$  is a Borel set, as  asserted. 
\end{proof}

In what follows,  $\mu$ is a  harmonic   measure on $(X,\Lc,g).$
  The following elementary result  is   needed  (see Exercise 2.4.16 in  \cite{CandelConlon2}).
\begin{lemma}\label{lem_change_of_small_flow_boxes} Let $\Phi:\ \U\to \B\times \T$ and $\Phi':\  \U'\to\B'\times \T'$  be two  (small)  flow boxes\index{flow box!small $\thicksim$}  
with transversal $\T$ and $\T'$  respectively. Let $\lambda$
(resp. $\lambda'$)
be the  measure   defined  on $\T$ (resp. on $\T'$) which is  given by the local decomposition of a harmonic measure $\mu$ 
on $\U$ (resp. $\U'$) thanks to Proposition   \ref{prop_current_local}. Assume that the  change of coordinates $\Phi\circ \Phi'$ is  of the form
$$
(x,t)\mapsto(x',t'),\qquad x'=\Psi(x,t),\ t'=\Lambda(t)
$$
Then  the measure $\Lambda^*\lambda'$ is absolutely continuous  with respect to  $\lambda$ on $\Dom(\Lambda).$ 
\end{lemma}

 Let $\T$ be  a  transversal of a  flow tube  $\U$ of $(X,\Lc) .$
Fix  a  partition of unity subordinate to a finite covering of $\U$ by small flow boxes.
We applying  Lemma \ref{lem_change_of_small_flow_boxes} while  traveling different small flow boxes. 
Consequently, we obtain the  following version of   Proposition  \ref{prop_current_local} 
in the context of flow tubes. 
\begin{proposition} \label{prop_local_currents_flow_tube} Under the above  hypotheses and  notation,
 there  exists a finite  positive Borel measure $\lambda$ on $\T$ such that
  for $\lambda$-almost every $t\in\T$ there  is  a  harmonic function  $h_t>0$ defined on the plaque
$\U_t$
 with  the following  two properties:
 \begin{itemize}
 \item [$\bullet$]
  for every compact set $K\subset \U,$
\begin{equation*} 
\int_{\T} \Big(\int_{\U_t} \otextbf_{K}(y) h_t(y) d\Vol_t(y)\Big) d\lambda(t)<\infty;
\end{equation*}
 \item [$\bullet$] for every Borel set $Y\subset \U,$
\begin{equation*} 
\mu(Y)=\int_Y  d\mu=\int_{\T} \Big(\int_{\U_t} \otextbf_{Y}(y) h_t(y) d\Vol_t(y)\Big) d\lambda(t),
\end{equation*}
where $\Vol_t(y)$ denotes the  volume form induced  by the metric tensor $g$ on the plaque $\U_t.$
\end{itemize}
  \end{proposition}

In the sequel let $\lambda_i$ be  the  measure on the transversal $\T_i$ when we apply Proposition \ref{prop_local_currents_flow_tube}
to $\mu|_{\U_i}.$
 The relation between $\mu$ and   the measure $\lambda$ associated to $\mu$  on a transversal in a flow tube
is  described  in the following result. 

\begin{proposition} \label{prop_current_local_consequence} 
Let $Y\subset X$ be a   set.
\\
 1) Assume that $Y$ is  a Borel set and that  $Y$ is  contained in a flow   tube $\U$ with transversal $\T.$ Then $\T_Y$ is  $\lambda$-measurable,
where $\lambda$ is the measure on $\T$ given by  Proposition \ref{prop_local_currents_flow_tube}. 
\\
2) If   $\lambda_i((\T_i)_{Y\cap \U_i})=0$ for every $i\in I$ then  $\mu(Y)=0.$ Conversely,
if $Y$ is a  Borel set  such that  $\mu(Y)=0$  and that $\Vol_a(L_a\cap Y)>0$  for every  $a\in Y,$ then    $\lambda(\T_{Y\cap\U})=0$
for every flow tube $\U$ with transversal $\T.$
Here $\Vol_a$ denotes the  volume form on $L_a$ induced  by the metric  tensor $g|_{L_a}$
and  $\lambda$ is given by  Proposition \ref{prop_local_currents_flow_tube}.  
\\
3) If $Y$ is  a Borel set such that  $\mu(Y)=0$ and  that $\Vol_{L_a}(L_a\cap Y)>0$  for every  $a\in Y,$ then    $\mu(\Satur(Y))=0.$ 
\\
4) If $Y$ is  a Borel set, Then there exist a  leafwise  saturated  Borel set $Z$ and  a leafwise saturated  set $E$ with $\mu(E)=0$ such that
$\Satur(Y)=Z\cup E$ and that
 $\T_Z$ is   a Borel set for any transversal $\T$ of each flow box $\U.$
 \\
 5)  If $Y$ is  a leafwise saturated set of full $\mu$-measure, then there exists a  leafwise  saturated  Borel subset $Z\subset Y$  such that
$\mu(Y\setminus Z)=0$ and that
 $\T_Z$ is   a Borel set for any transversal $\T$ of each flow box $\U.$
   \end{proposition}
\begin{proof}
 To prove  Part 1) we assume  first that $\U$ is  a flow box. So we can write $\U\simeq \B\times \T,$ where $\B$ is a domain in $\R^n.$ Applying  Theorem \ref{thm_measurable_projection} to $Y$ yields that $\T_Y$ is $\lambda$-measurable.
 When $\U$ is a  general flow tube, we use  a finite covering of $\U$ by small flow boxes by applying
 Lemma \ref{lem_change_of_small_flow_boxes}. Part 1) follows.
  
  The  first assertion of Part 2) holds by applying Proposition  \ref{prop_local_currents_flow_tube}
to each  flow tube $\U_i.$ Note that $Y$ need not to be a Borel set.

To prove the  second assertion of Part 2), suppose in order to get a contradiction that 
$\lambda_i((\T_i)_{Y\cap\U_i})\not=0$ for some $i.$
Since  we know by Part 1) that $(\T_i)_{Y\cap\U_i}$ is $\lambda_i$-measurable,
it follows that 
$\lambda_i((\T_i)_{Y\cap\U_i})>0.$
 Next, we apply   Proposition \ref{prop_plaque_saturated_approximation} to  the  saturation of $(\T_i)_{Y\cap\U_i}.$
Putting this together  with the  assumption that   $\Vol_a(L_a\cap Y)>0$  for every  $a\in Y,$  and  applying Lemma \ref{lem_change_of_small_flow_boxes} and  using a  partition of unity,
 we may find   a   flow  tube $\U$ with transversal $\T$ and a Borel set $Z\subset \U\cap Y$ and  a  Borel set $S\subset \T
\cap\Satur( (\T_i)_{Y\cap\U_i}  )$
 such that 
 \\
 $\bullet$  $\lambda(S)>0,$ where $\lambda$ is the measure on $\T$  given by Proposition
\ref{prop_local_currents_flow_tube};
\\
 $\bullet$ $\Vol_t(\U_t\cap Z)>0$ for each $t\in S.$

Inserting  these into the equality in Proposition  \ref{prop_local_currents_flow_tube},
we get that $\mu(Z)>0.$  Hence, $\mu(Y)>0,$ which is  a contradiction.
The  second assertion of Part 2) is thus   complete.

Now we turn to Part 3).
Since $\Satur(Y)=  \cup_{i\in I}\Satur(  (\T_i)_{Y\cap\U_i}    )),$  we only need to show that
$\mu(\Satur(  (\T_i)_{Y\cap\U_i}    ))=0$ for each $i\in I.$ Fix  such an $i_0\in I.$
By the second assertion of Part 2), we get that $\lambda_{i_0}((\T_{i_0})_{Y\cap\U_{i_0}})=0.$ 
 Consequently, applying Proposition \ref{prop_local_currents_flow_tube} to $\U_{i_0},$  we can show that 
$\mu(Z)=0,$ as  desired.
 
 Next, we prove Part 4). Since  $\Satur(Y)=  \cup_{i\in I}\Satur(   Y\cap\U_i    ),
$ we may  assume  without loss of generality that $Y$  is  contained in a flow   tube $\U$ with transversal $\T.$
By Part 1),  $\T_Y$ is  $\lambda$-measurable. So we can write $\T_Y=S\sqcup E,$ where  $S\subset \T$ is a  Borel set and 
$F\subset \T$ with $\lambda(F)=0.$ Here $\lambda$ is the measure on $\T$ given by  Proposition \ref{prop_local_currents_flow_tube}.
 Let $Z:=\Satur(S)$ and $E:=\Satur(F).$ Since $S$ is a Borel set  we know by Proposition \ref{prop_plaque_saturated_approximation} 
that so is $Z.$ Moreover,
by  the  first assertion of Part 2), $\mu(E)=0.$ 
This finishes Part 4).

To prove  Part 5), consider a flow  tube $\U$ with  transversal $\T.$
Since $Y$ is  of full $\mu$-measure, there is  a Borel subset $Y'\subset Y$ such that
$\mu(Y\setminus Y')=0.$
 By Part 1), $\T_{Y'\cap \U}$ is  $\lambda$-measurable,
where $\lambda$ is the measure on $\T$ given by  Proposition \ref{prop_local_currents_flow_tube}.
Let $\T'$ be a  Borel set  such that $\T'\subset \T_{Y'\cap \U}$ and that $\lambda(\T_{Y'\cap \U}\setminus \T')=0.$
Let $Z:=\Satur(\T').$ So
 $$\mu( (Y\setminus Z)\cap \U)=\mu( (Y'\setminus Z)\cap \U)=0,$$
where the last equality follows  from   Part 2).

For each  flow  box $\U_i$ in a  fixed  (at most) countable cover of $X$ by flow  boxes indexed by $I,$
we constructs such a leafwise saturated set $Z_i.$ Now it suffices  to choose $Z:=\cup_{i\in I} Z_i.$     
\end{proof}
 
 \subsection{Fibered laminations}
 \label{subsection_fibered_laminations_new}
 
 In the second part of the section,
   let $(X,\Lc,g)$ be a Riemannian continuous-like lamination 
  together with    the corresponding   covering  lamination projection $\pi:\ (\widetilde X,\widetilde\Lc)\to (X,\Lc).$
   Let  $\Sigma$ be  a Hausdorff  topological space  $\Sigma,$ and $\iota:\  \Sigma\to \widetilde X$ a  measurable projection    
 such that  for every $y\in \Sigma,$  there exists a  set $\Sigma_y\subset \Sigma$
 satisfying  Definition \ref{defi_fibered_laminations} (i)--(ii).
 Set $\tau:=\pi\circ\iota:\ \Sigma\to X.$

 Next, we  generalize Definition
\ref{defi_flow_tube} to the present context. 
 
\begin{definition}\label{defi_fibered_flow_tube}
 A set $\U_{\Sigma}\subset \Sigma$ is  said  to be 
a {\it flow tube}\index{flow tube!$\thicksim$ of a measurable projection}   if there is  a  good flow  tube $\widetilde\U$ in $(\widetilde X,\widetilde \Lc)$
such that 
$\U_\Sigma=\iota^{-1}(\widetilde\U). $  In this  case $\U_{\Sigma} $ is  often said  to be {\it  associated to} the good flow  tube $\widetilde\U.$

So the image $\U:=\pi(\widetilde \U)$  is  also  a flow tube in $(X,\Lc),$ and  the  restriction
of  the projection $\pi|_{ \widetilde\U  }:\ \widetilde\U\to  \U$ is a homeomorphism which map
plaques onto plaques. 

Let $\T$  be a  transversal of the flow tube $\U$. For each $x\in \T,$ let $\tilde x:=(\pi|_{ \widetilde\U})^{-1}(x)\in\widetilde \U.$ For such a point $\tilde x$ and
 for each $y\in\iota^{-1}(\tilde x)\subset \U_\Sigma,$ the  set
$\U_{\Sigma,y}:=\Sigma_y\cap \iota^{-1}(\widetilde\U_{\tilde x})$ is  said to be  the {\it plaque passing through $y$} of $\U_\Sigma,$
where $\widetilde \U_x$ is  the  plaque of $\widetilde\U$ passing through $\tilde x$ given by Definition \ref{defi_flow_tube}.

  A set  $V_\Sigma\subset \U_{\Sigma}$ is  said to be {\it plaque-saturated} if  for every $y\in V_\Sigma,$ the  plaque passing through $y$ of $\U_\Sigma$
  is  also  contained in $V_\Sigma.$
  

The  {\it sample-path space  of  a flow tube $\U_\Sigma$   up to  time $N\geq 0$}
\index{space!sample-path $\thicksim$!$\thicksim$ of a flow tube up to a given time}
    is, by definition, the subspace of $\Omega(\Sigma)$ consisting of
 all 
$\omega\in\Omega(\Sigma)$ such that $ \omega [0,N]$ is  fully contained in a single  plaque $\U_{\Sigma,y}$  for some $y\in\tau^{-1}(x)$ and $x\in\T.$  This space is  denoted by 
$\Omega(N,\U_\Sigma).$  
\nomenclature[c1h]{$\Omega(N,\U_\Sigma)$}{sample-path space of a flow tube $\U_\Sigma$ up to a time $N>0,$ where $\U_\Sigma$ is a  flow tube of a measurable projection $\iota:\ \Sigma\to\widetilde X$}
\end{definition}
 
 Next, we  generalize the notion of {\it directed cylinder sets} and {\it directed cylinder images} in Definition \ref{defi_directed_cylinder_sets} to  the present context.

\begin{definition}\label{defi_fibered_directed_cylinder_sets} \rm 
 Let $ \U_\Sigma$ be   a    flow tube of a $\Sigma$ associated to  a  good flow tube $\widetilde\U$
 (see Definition \ref{defi_fibered_flow_tube})
 and $N>0$  a given time.
  
A {\it  directed  cylinder set} $A$  with respect to $(N,\U_\Sigma)$  
is  the intersection of   a  cylinder set  $C\big (\{t_j, A_j\}:m )$ in $\Omega(\Sigma)$ and  the sample-path space 
$\Omega(N, \U_\Sigma)$    satisfying the time  requirement $N\geq  t_m=\max t_j.$

 A set $B\subset \Omega:=\Omega(X,\Lc)$ is  said to be a {\it  directed  cylinder  image} (with respect to $(N,\U_\Sigma)$) if  $B=\tau\circ A$  for  some
    directed  cylinder set $ A\subset \Omega(\Sigma)$ (with respect to $(N,\U_\Sigma)$).
  \end{definition}
  
 Prior  to the proofs of  Part 3) and Part 4)  of Lemma \ref{lem_sup_measurable} and the proof of  Part 2) of   Proposition \ref{prop_Sigma_k_is_a_weakly_fibered_lamination} and 
 the proof of  Proposition    \ref{prop_Sigma_k_is_a_fibered_lamination}, we  develop some common facts.
 For the sake of simplicity, 
in what follow  we write  $d$ (resp. $\Sigma$) instead of $d_m$ (resp. $\Sigma_k$) which  has appeared in the proof of Step 2 of Theorem  \ref{th_Lyapunov_filtration}
in Section \ref{subsection_Stratifications}. Let $\overline\Sigma:=\Satur(\Sigma)$ in  $(\widetilde X_{k,\widetilde{\mathcal A}},\widetilde \Lc_{k,\widetilde{\mathcal A}}).$
So $\overline\Sigma$ plays the role of $\overline\Sigma_k$ in Section \ref{subsection_Stratifications}.
Recall  that $\widetilde X_{k,\widetilde{\mathcal A}}
=\widetilde X\times\Gr_k(\R^d).$ Recall  also the canonical projections
$$\iota:\  \widetilde X\times\Gr_k(\R^d)\to \widetilde X\quad\text{and}\quad  \pi:\  \widetilde X\to X.$$

 Combining    Part 1) of  
 Lemma \ref{lem_sup_measurable}  and the formula for $\Sigma$ given in (\ref{eq_Sigma_k}), we get that
 $\Sigma=\cup_{l=1}^\infty \Sigma_l,$  where $(\Sigma_l)_{l=1}^\infty$ is the  increasing  sequence of   subsets of $\widetilde X\times \Gr_k(\R^d)$ given by:
 \begin{equation}\label{e:Sigma^l}
 \Sigma^l:= \{ (\tilde x,U)\in (\widetilde X\setminus \pi^{-1}(\Nc_{k+1}))\times \Gr_k(\R^d):\ M_k(\pi(\tilde x),U)>1/l\}.
 \end{equation}
 By  Part 1) of Lemma \ref{lem_sup_measurable},
 $\Sigma^l$ is a Borel set. On the  other  hand,
 by  Lemma  \ref{lem_stratifications},  we get  that
\begin{equation}\label{eq_finite_fiber}
\# \{U:\ (\tilde x,U)\in \Sigma^l\}<l, \qquad \tilde x\in \widetilde X\setminus \pi^{-1}(\Nc_{k+1}).
\end{equation}
Set
\begin{equation*}
\Sigma^{'l}:=(\pi\circ \iota) (\Sigma^l)\quad\text{for each  $l\in\N\setminus\{0\}$ and}\quad \Sigma':=(\pi\circ\iota)(\Sigma). 
\end{equation*} 
  Let  $(\U_i)_{i\in \N}$ be 
an  atlas of $ X$  consisting a countable family of flow boxes.  For each $i\in\N,$ let $\T_i$ be a transversal of $\U_i,$
let $\pi_{\U_i}$ be the projection from $\U_i$ onto $\T_i,$
and let $\lambda_i$ be the  positive  finite Borel measure on $\T_i$  given by applying Proposition  \ref{prop_local_currents_flow_tube} 
 to the restriction of $\mu$ on $\U_i.$
 Recall from formula (\ref{eq_approximation_Y}) that
 for a set $Y\subset  X,$ or more generally,  for an array of  sets $(Y_{i})_{i\in \N}$ with $Y_{i}\subset\U_i,$
we define  an increasing  sequence $(Y_{ip})_{p=0}^\infty$ of plaque-saturated  sets in $\U_i.$ 
Therefore,
 we may apply Proposition \ref{prop_plaque_saturated_approximation} to each $\Sigma^{'l}$ the construction of
a sequence  of  plaque-saturated sets $(\Sigma^{'l}_{ip})_{i,p=0}^\infty.$  
 Consequently, we have that
 \begin{equation}\label{e:Satur_Sigma^l}
\Satur(\Sigma^{'l})\nearrow  \Satur(\Sigma')=(\pi\circ\iota)(\overline\Sigma)\quad\text{and}\quad \Sigma^{'l}_{ip} \nearrow
 \U_i\cap \Satur(\Sigma^{'l})
 \end{equation}
   as $l$ (resp. $p$) tends to $\infty.$ Moreover,
   \begin{equation}\label{e:Nc_k_vs_Sigma_l}
   \Nc_k=\bigcup_{l=1}^\infty \Sigma^{'l}.
   \end{equation}

Now  we arrive  at
  the 
  \\
 \noindent
{\bf End of the proof of Part 3)   of Lemma \ref{lem_sup_measurable}.}
Applying  Theorem \ref{thm_measurable_projection} to each Borel set   
$\Sigma^l\subset\widetilde X\times \Gr_k(\R^d),$  we infer that the set $\Sigma^{'l}_i:=\pi_{\U_i} (\Sigma^{'l}\cap \U_i)\subset \T_i$
 is $\lambda_i$-measurable.
So  we may  choose a Borel  subset $E^l_{i}\subset \T_i$ such that
$  \Sigma^{'l}_{i}\subset  E^l_{i}$ and that $\lambda_i(E^l_{i}\setminus \Sigma^{'l}_{i})=0$ for all $i\in\N.$
Next, we  apply Proposition \ref{prop_plaque_saturated_approximation} to the   array $(E^l_{i})_{i=0}^\infty$
(resp.   $(\Sigma^{'l}_{i})_{i=0}^\infty$) and using (\ref{e:Satur_Sigma^l}).
Consequently, we
obtain a leafwise  saturated Borel set $E^l$ (resp. the set $\Satur(\bigcup_{i=0}^\infty \Sigma^{'l}_i)$).
Note  that the latter set is  equal to
 $\Satur(\Sigma^{'l}).$
 So $\Satur(\Sigma^{'l})\subset E^l.$
By the first assertion in Part 2) of Proposition  \ref{prop_current_local_consequence},  we get that
$\mu(E^l\setminus \Satur(\Sigma^{'l}))=0.$

Finally, we choose the Borel set
$
E:=\bigcup_{l=1}^\infty E^l.$
The previous  paragraph shows that
$E$ is a leafwise saturated Borel set  which contains  $\bigcup_{l=1}^\infty\Satur(\Sigma^{'l})$ and that
$$\mu\Big(E\setminus \bigcup_{l=1}^\infty\Satur(\Sigma^{'l})\Big)=0.$$
This, combined  with (\ref{e:Nc_k_vs_Sigma_l}), gives the  desired property  of $E.$
\qed

\noindent
{\bf End of the proof of Part 4)    of Lemma \ref{lem_sup_measurable}.}
We only  need to prove the  nontrivial  implication $\mu(\Nc_k)=0\Rightarrow\mu(\Satur(\Nc_k))=0.$
We argue as in the proof of Part 3) of the lemma. The  main change is that we choose 
 a Borel  subset $E^l_{i}\subset \T_i$ such that
$  E^l_{i}\subset \Sigma^{'l}_{i} $ and that $\lambda_i( \Sigma^{'l}_{i}\setminus E^l_{i})=0$ for all $i\in\N.$
Let $F^l_i$ be the  Borel subset of $\U_i$   defined  by
$$F^l_i:=(\Sigma^{'l}\cap \U_i)\cap \Satur_{\U_i}( E^l_{i}).$$
So $F^l_i\subset \Sigma^{'l}$ and  $\Satur(F^l_i)=\Satur( E^l_{i}).$
Let $F:=\bigcup_{l=1}^\infty\bigcup_{i=0}^\infty F^l_i.$ So the Borel set $F$ is  contained in $\Nc_k$ by (\ref{e:Nc_k_vs_Sigma_l})
and  $\Satur(F)=\bigcup_{l=1}^\infty \Satur( E^l )=E.$
In summary,
 we have obtained a Borel set $F$ and a leafwise saturated Borel set $E$ such that $F\subset \Nc_k$ and that
$E=\Satur(F)\subset \Satur(\Nc_k)$ and that $\mu( \Satur(\Nc_k)\setminus E)=0$  and that
 $\Vol_{L_a}(L_a\cap E)>0$ for every $a\in F.$ 
By 
 Part 3) of Proposition  \ref {prop_current_local_consequence}, $\mu(\Satur(F))=0.$
 Hence,  $\mu(E)=0$ and $\mu( \Satur(\Nc_k ))=0.$   
\qed


\noindent {\bf End of the proof of     Part 2) of   Proposition \ref{prop_Sigma_k_is_a_weakly_fibered_lamination}.} 
 Let  $(\widetilde\U_i)_{i\in \N}$ be 
an  atlas of $ \widetilde X$ which consists of  a countable family of flow boxes. 
  For each $i\in\N,$ let $\widetilde\T_i$ be a transversal of $\widetilde\U_i.$
  We may assume  that the image  $\U_i:=\pi(\widetilde\U_i)$ is  also a   flow box  and the  restriction of
the projection  $\pi|_{\widetilde\U_i}:\  \widetilde\U_i\to\U_i$ is a homeomorphism which maps each plaque of $\widetilde\U_i$
onto a  plaque of $\U_i.$
Let $\pi_{\widetilde \U_i}$ be the projection from $\widetilde\U_i$ onto $\widetilde\T_i,$
and let $\tilde\lambda_i$ be the  positive  finite Borel measure on $\widetilde\T_i$  given by applying Proposition  \ref{prop_local_currents_flow_tube} 
 to the restriction of $\pi^*\mu$ on $\widetilde \U_i.$

For each $i\in \N,$ consider $\widetilde\U_i\times  \Gr_k(\R^d)$
as  a  flow  box of the cylinder lamination $(\widetilde X_{k,\mathcal A},\widetilde \Lc_{k,\mathcal A})$
and   $\widetilde \T_i\times \Gr_k(\R^d)$ as its  transversal (see Definition \ref{defi_cylinder_lamination}).
 For each $l\in\N\setminus\{0\},$  let
 $\Sigma^l_i$ be the image of  $\Sigma^l\cap \widetilde\U_i\times  \Gr_k(\R^d) $ under the   projection 
from the flow box $\widetilde\U_i\times  \Gr_k(\R^d)$
  onto   its transversal $\widetilde \T_i\times \Gr_k(\R^d).$  
   Moreover, let 
 \begin{equation}\label{e:widetilde_Sigma_l}
\widetilde G^l_i :=\iota(\Sigma^l_i  )\subset \widetilde \T_i.
\end{equation}
 Since  $\Sigma^l$ is  a  Borel set, applying  Theorem \ref{thm_measurable_projection_new}  yields  
 a Borel set $  \widetilde F^l_i\subset  \widetilde \T_i$ such that   $  \widetilde F^l_i\subset \widetilde G^l_i$  and  that
 \begin{equation}\label{e:widetilde_G}
   \tilde\lambda_j(\widetilde G^l_i\setminus\widetilde F^l_i)=0 \quad\text{and}\quad   \Sigma^l_i\cap \big(\widetilde F^l_i\times\Gr_k(\R^d)\big)\in
\Bc(\widetilde \T_i)_{\tilde\lambda_i}\otimes  \Bc (\Gr_k(\R^d)).
    \end{equation}
    Here $\Bc(\widetilde \T_i)_{\tilde\lambda_i}$ denotes, as  usual, the $\tilde\lambda_i$-completion of $\Bc(\widetilde \T_i).$
    The  following lemma  is  needed.
    \begin{lemma}\label{L:Harnack}
    There is a Borel measurable function $c:\ \widetilde \T_i\to (0,\infty)$  such that for all $\tilde x, \tilde  y \in  
    \widetilde\U_{i,\tilde t}$ and $\tilde z\in \widetilde  L_{\tilde t},$  we have that
    $$
       p(\tilde x , \tilde z, 1)\leq c(\tilde t) \cdot  p(\tilde y , \tilde z, 1).
    $$
    Here $\widetilde\U_{i,\tilde t}$ is  the plaque of the flow box $\widetilde\U_i$ passing through $\tilde t\in \widetilde \T_i,$
    and  $\widetilde  L_{\tilde t}$ is the leaf of $(\widetilde X,\widetilde \Lc)$ passing through $\tilde t,$
and $p(\cdot,\cdot,\cdot)$ is,  as usual, the  heat kernel  on $\widetilde  L_{\tilde t}.$ 
    \end{lemma}
    \begin{proof}
    First, fix  a  point   $\tilde t\in \widetilde \T_i.$ Applying   Harnack's inequality\index{Harnack!Harnack's inequality} for non-negative solutions of the heat  equation
(see, for example, Grigor'yan's  survey \cite{Grigoryan}\index{Grigor'yan}) to the  
function $p(\cdot,\tilde z, t)$  on a  bounded  smooth domain  $\D_{\tilde t}$ in $\widetilde  L_{\tilde t}$  such that
$\widetilde\U_{i,\tilde t}\Subset \D_{\tilde t}$, the desired estimate  follows for some constant  $c(\tilde t)>0.$
In fact, we may identify $\D_{\tilde t}$ with a smooth bounded  domain in $\R^{n_0},$ where $n_0$ is the dimension of the lamination $(X,\Lc).$
    The  optimal quantity  $c(\tilde t)$ depends on the  geometry of the  manifold  $(\D_{\tilde t}, \pi^*g|_{\D_{\tilde t}}).$
    Since $(X,\Lc,g)$ satisfies  Hypothesis  (H1), we deduce that $ \pi^*g|_{\D_{\tilde t}}$  depends Borel
measurably on $\tilde t\in  \widetilde \T_i$,   and hence  $c(\tilde t)$ depends Borel  measurably on $\tilde t\in  \widetilde \T_i.$
     \end{proof}
    
     For each $j\in\N\setminus \{0\},$ consider the following  Borel subset of $
 \widetilde F^l_i:$
 $$
 \widetilde F^l_{ij}:=\left\lbrace \tilde t\in \widetilde F^l_i:\   c(\tilde t)\leq  j   \right\rbrace,
 $$
 where $c:\ \widetilde \T_i\to (0,\infty)$ is the Borel measurable function given by Lemma \ref{L:Harnack}.
 Clearly,  
 \begin{equation}\label{e:F^l_i,j}
 \widetilde F^l_{ij}\nearrow 
 \widetilde F^l_{i}\qquad\text{ as}\ j\to\infty.
 \end{equation}
 The following lemma  will be proved later on.
\begin{lemma}\label{L:finite_image}
For every $\tilde t\in \widetilde F^l_{ij},$ the  set
$$\big (\{\tilde t\}\times \Gr_k(\R^d)\big)\cap\Sigma^l_i$$
 is  nonempty and  of cardinal $\leq jl$.
\end{lemma}
Consider the set  is  
$$\big (  \widetilde F^l_{ij}\times \Gr_k(\R^d)\big)\cap \Sigma^l_i.
 $$
 By Lemma \ref{L:finite_image}, this  is  clearly  the  graph of a  multifunction from  $\widetilde F^l_{ij}$
 to nonempty finite  subsets of cardinal $\leq jl$ of $\Gr_k(\R^d).$ 
 Moreover, recall from (\ref{e:widetilde_G}) that $\big (  \widetilde F^l_{i}\times \Gr_k(\R^d)\big)\cap \Sigma^l_i
 $
   is in $\Bc(\widetilde \T_i)_{\tilde\lambda_i}\otimes \Bc (\Gr_k(\R^d)).$ Consequently, the above  multifunction
  is  measurable  with respect to $\Bc(\widetilde \T_i)_{\tilde\lambda_i} .$
Applying   Theorem
\ref{thm_measurable_selection_new}, %
we may find  a Borel set $\widetilde E^l_{ij}\subset \widetilde \T_i$ and 
 $\kappa(l,i,j) $ $(\leq jl)$ Borel  measurable selections $ s^l_{ijp}:\  \widetilde E^l_{ijp}\to \Gr_k(\R^d)$  with $1\leq p\leq \kappa(l,i,j)$ satisfying the following properties:
\begin{itemize}
\item[(i)] $\widetilde E^l_{ijp}$ is a Borel set in $\widetilde \T_i;$  

\item[(ii)]  $\widetilde E^l_{ij}\subset  \widetilde F^l_{ij}$ and  $\tilde\lambda_i(\widetilde F^l_{ij}\setminus  \widetilde E^l_{ij} )=0,$    where $\widetilde E^l_{ij}:=\bigcup_{p=1}^{\kappa(l,i,j)}
\widetilde E^l_{ijp};$   
\item  [(iii)']
 $ \iota^{-1}(\tilde x)\cap \Sigma^l_i
=\{  (\tilde x, s^l_{ijp}(\tilde x)):\    1\leq p\leq \kappa(l,i,j)\}$
 for all   
$\tilde x\in  \widetilde E^l_{ij},$ with the  convention that $(\tilde x, s^l_{ijp}(\tilde x))=\varnothing$ if 
$\tilde x\not\in \widetilde E^l_{ijp}.$
\end{itemize}
 Consider   the Borel measurable one-to-one map  $ \tilde s^l_{ijp}:\  \widetilde E^l_{ijp}\to \widetilde X\times \Gr_k(\R^d)$   associated  to each
 selections $ s^l_{ijp}:$
 $$
\tilde s^l_{ijp}(\tilde x):= (\tilde x, s^l_{ijp}(\tilde x)),\qquad  \tilde  x\in  \widetilde E^l_{ijp}.
 $$
So item (iii)' may be  rewritten as
\begin{itemize}
\item  [(iii)]
 $ \iota^{-1}(\tilde x)\cap \Sigma^l_i
=\{  \tilde s^l_{ijp}(\tilde x):\    1\leq p\leq \kappa(l,i,j)\}$
 for all   
$\tilde x\in  \widetilde E^l_{ij},$ with the  convention that $\tilde s^l_{ijp}(\tilde x)=\varnothing$ if 
$\tilde x\not\in \widetilde E^l_{ijp}.$ In particular,  $ \tilde s^l_{ijp}$ is  a  local section of $\iota.$
\end{itemize}

Using Lusin's theorem \cite{Dudley} \index{Lusin!$\thicksim$'s theorem}\index{theorem!Lusin's $\thicksim$} and by removing from
  each $\widetilde E^l_{ijp}$ a $\tilde\lambda_i$-negligible set, we may write  $\widetilde E^l_{ijp}$
  as  a countable union of increasing  compact subsets of   $\widetilde \T_i$  such that  the restriction of  $\tilde s^l_{ijp}$
on each  such  compact set is  continuous. Hence, we have  shown the validity of 
items (i)--(iii) above as well as  the following additional item:
  \begin{itemize}
\item  [(iv)]  both $\widetilde E^l_{ijp}$ and 
$\tilde s^l_{ijp}(\widetilde E^l_{ijp})$ are Borel sets, and the surjective  map $\tilde s^l_{ijp}:\ \widetilde E^l_{ijp}\to \tilde s^l_{ijp}(\widetilde E^l_{ijp})$ is Borel bi-measurable.
\end{itemize}

Consider the  following subset of  $\overline\Sigma=\Satur(\Sigma):$ 
\begin{equation*}
\overline\Sigma^*:=  \bigcup_{l=1}^\infty\bigcup_{i=0}^\infty\bigcup_{j=1}^\infty  \Satur\Big (\iota^{-1}(\widetilde E^l_{ij})\cap \Sigma^l_i\Big).
\end{equation*}
In fact, putting  (\ref{e:widetilde_Sigma_l}), (\ref{e:widetilde_G}), (\ref{e:F^l_i,j}) and the items (i)--(iii) above  together,  and applying  the first assertion of Part 2) of Proposition
   \ref {prop_current_local_consequence},
we infer that $\overline\Sigma^*\subset\overline\Sigma$ is  a leafwise saturated Borel set  in $\widetilde X\times\Gr_k(\R^d)$
and that  $\mu ((\pi\circ\iota)(\overline\Sigma\setminus\overline\Sigma^*))=0.$ In summary,
by removing from $ \overline\Sigma$ a leafwise-saturated set such that its image under $\pi\circ\iota$ is of null $\mu$-measure,
{\bf we may  assume without loss of generality that} 
\begin{equation}\label{e:overline_Sigma'}
  \overline\Sigma=\overline\Sigma^*=\bigcup_{l=1}^\infty\bigcup_{i=0}^\infty\bigcup_{j=1}^\infty  \Satur\Big (\iota^{-1}(\widetilde E^l_{ij})\cap \Sigma^l_i\Big).
\end{equation}
Write $\widetilde\Omega:=\Omega(\widetilde X,\widetilde\Lc)$ as  usual.
Recall  from Remark \ref{R:lem_existence_good_chain}
  that  there is a  countable  family $\widetilde\Sc:= (\widetilde{\mathcal U}_q)_{q\in\N}$ of good chains on $(\widetilde X,\widetilde\Lc)$ such that
 \begin{equation}\label{e:cover_local_sections}
\widetilde\Omega=\bigcup_{q=0}^\infty \Omega(1, \U_{\widetilde{\mathcal U}_q}).
\end{equation}
where $\U_{\widetilde{\mathcal U}_q}$ is   the flow tube associated to the good chain $\widetilde{\mathcal U}_q.$
For each $q\in \N,$ let $\T_{\widetilde{\mathcal U}_q}$ be a transversal  of the  flow tube $\U_{\widetilde{\mathcal U}_q}.$  
We may write  $\T_{\widetilde{\mathcal U}_q}$ as  the union $\bigcup_{r=0}^\infty V_{iqr}$ of its open subsets  such that  
\begin{itemize}
\item[(v)] if $\tilde z$ and $\tilde w$ are in $\widetilde \U_i\cap \Satur_{\U_{\widetilde{\mathcal U}_q}}(V_{iqr})$ and
if they  are not in a  common plaque   of the  flow box  $\widetilde \U_i,$ then they are also not in  a common
plaque of the  flow tube  $\U_{\widetilde{\mathcal U}_q}.$
\end{itemize}
 Indeed, for each point $t\in \T_{\widetilde{\mathcal U}_q},$ we may find  a small open neighborhood  $V_{iqr}$
 of $t$ satisfying  item (v). This, coupled  with a compactness argument
and the fact that  $\T_{\widetilde{\mathcal U}_q}$ is  a separable locally compact metric space,  finishes the claim.

Now  we  are  in  the position  to construct 
 a   countable family of   local sections   of $\iota$ which generates all fibers of $\iota.$
 For every $l,j\in\N\setminus\{0\}$ and  every $i,q\in \N$ and $1\leq p\leq \kappa(l,i,j),$ and every $r\in \N,$
write  $N:=(l,i,j,p,q,r)\in \N^6:$
if $$\widetilde E_N:= \Satur_{\U_{\widetilde{\mathcal U}_q}}\Big (\Satur_{\U_{\widetilde{\mathcal U}_q}}(V_{iqr}) \cap \widetilde E^l_{ijp}\not=\varnothing\Big),$$ then  we construct 
\begin{equation}\label{e:s_N}
\tilde s_N= \tilde s^l_{ijpqr}:\ \widetilde E_N\to \widetilde X\times\Gr_k(\R^d)
\end{equation} 
 such that $\tilde s_N=\tilde s^l_{ijp}$ on  $\widetilde E^l_{ijp}\cap \widetilde E_N$ and that
if $\tilde x$ and $\tilde x'$ are in a  common plaque of   the  flow tube  $\U_{\widetilde{\mathcal U}_q},$  then $\tilde s^l_{ijp} (\tilde x)$
and  $\tilde s^l_{ijp} (\tilde x')$ are also   in the same leaf of the cylinder lamination $(\widetilde X_{k,\widetilde{\mathcal A}},\widetilde \Lc_{k,\widetilde{\mathcal A}}).$
Roughly speaking, we    propagate $\tilde s^l_{ijp}$ from $\widetilde E^l_{ijp}\cap \widetilde E_N$
to  $\widetilde E_N$ by moving $\tilde s^l_{ijp}$ on leaves  of the cylinder lamination $(\widetilde X_{k,\widetilde{\mathcal A}},\widetilde \Lc_{k,\widetilde{\mathcal A}}).$

To show that $\tilde s_N$ is well-defined,  consider three points
 $\tilde x,\tilde x',\tilde y$ lying in a common plaque  of   the  flow tube  $\U_{\widetilde{\mathcal U}_q}$ such that
$\tilde x,\tilde x'\in \widetilde E^l_{ijp}\cap \widetilde E_N$ and  that $\tilde y\in \widetilde E_N$
and  that $\tilde s^l_{ijp}(\tilde x)$ and $(\tilde y,W)$ (resp. $\tilde s^l_{ijp}(\tilde x')$ and $(\tilde y,W')$)
are     in the same leaf of the cylinder lamination 
$(\widetilde X_{k,\widetilde{\mathcal A}},\widetilde \Lc_{k,\widetilde{\mathcal A}}),$
where $W,W'\in \Gr_k(\R^d).$
 We need to show that $W=W'.$ By item (v), $\tilde x= \tilde x'.$ So both $(\tilde y,W)$ and $(\tilde y,W')$
 are  in  a common  leaf of the cylinder lamination $(\widetilde X_{k,\mathcal A},\widetilde \Lc_{k,\mathcal A}).$
 By (\ref{eq_cocycle_on_cover}), 
there is  a  path $\tilde \omega\in\widetilde\Omega$ such that $\tilde\omega(0)=\tilde\omega(1)=\tilde y$
 and that $\mathcal A(\pi\circ \tilde\omega,1)W=W'.$
 Since  $\tilde\omega(0)=\tilde\omega(1)$, it follows that $\pi\circ \tilde\omega[0,1]$ is null-homotopic.
Therefore,  the homotopy law in Definition  \ref{defi_cocycle}
 implies  that  $\mathcal A(\pi\circ \tilde\omega,1)=\id.$ Hence, $W=W',$ proving  that $\tilde s_N$ is well-defined.

Now  we  are  in the position to verify that the  countable  family  $\tilde s_N :\ \widetilde E_N\to \widetilde X\times\Gr_k(\R^d)$ satisfies 
Definition \ref{defi_fibered_laminations} (iii).
First observe  that $\tilde s_N$ is  a local  section, that is, 
$\iota(\tilde s_N(\tilde x))=\tilde x,$   $\tilde x\in \widetilde E_N.$ Indeed, this assertion  follows
by combining   (\ref{e:s_N}) and item (iii) above.

Next, recall  from item (iv) above
that both $\widetilde E^l_{ijp}$ and 
$\tilde s^l_{ijp}(\widetilde E^l_{ijp})$ are Borel sets and   the surjective  map $\tilde s^l_{ijp}:\ \widetilde E^l_{ijp}\to \tilde s^l_{ijp}(\widetilde E^l_{ijp})$ is Borel  bi-measurable.
Consequently,  it follows    from  the  construction  in  (\ref{e:s_N}) that 
 $\widetilde E_N$ and 
$\tilde s_N(\widetilde E_N)$ are Borel sets and the surjective map
 $\tilde s_N:\  \widetilde E_N\to \tilde s_N(\widetilde E_N)$ is also bi-measurable.

To prove that  the local sections $\tilde s_N$ generate all fibers of $\iota,$
let  $\tilde  x_0\in \widetilde X$ and pick an arbitrary  $\tilde  y_0\in\iota^{-1}(\tilde x_0)\cap \overline\Sigma'.$
By (\ref{e:overline_Sigma'}), there are $i\in \N$ and  $j,l\in \N\setminus\{0\}$ such that 
\begin{equation}\label{e:tilde y_0}
\tilde y_0\in
 \Satur\Big (\iota^{-1}(\widetilde E^l_{ij})\cap \Sigma^l_i\Big).
 \end{equation}
By item (iii)  above  and (\ref{e:cover_local_sections})  and (\ref{e:s_N}),
there exists $1\leq  p\leq  \kappa(l,i,j)$ and $q\in\N$ and $r\in\N$ such that for 
$N:=(l,i,j,p,q,r),$  $\tilde x_0\in E_N$ and $\tilde y_0=\tilde s_N(\tilde x_0).$
Hence, we have shown that
$$
\iota^{-1}(\tilde x)=\left  \lbrace  \tilde s_N(\tilde x):\  \tilde x\in \widetilde E_N\ \text{and}\ N\in \N^6 \right\rbrace,\qquad \tilde x\in \widetilde X.
$$
This completes the proof modulo Lemma \ref{L:finite_image}.
\qed

\noindent{\bf End of the proof of Lemma \ref{L:finite_image}.} 
Putting (\ref{e:widetilde_Sigma_l}), (\ref{e:widetilde_G}) and (\ref{e:F^l_i,j}) together,
we infer that for every $\tilde t\in \widetilde F^l_{ij},$ the  set
$\big (\{\tilde t\}\times \Gr_k(\R^d)\big)\cap\Sigma^l_i$
 is  nonempty.

It  remains to show  that the cardinal of such a set is$\leq jl.$ To this  end,  we will  improve quantitatively the proof  of statement 
(\ref{e:coutable_fiber}) replacing   the usual Lebesgue measure on leaves by a more ingenious measure
related to the heat kernels. In what follows, $\Sigma$ play the role of $\Sigma_k$ in the partial proof of
Proposition     \ref{prop_Sigma_k_is_a_weakly_fibered_lamination}, and $\Sigma^l$ is given  in (\ref{e:Sigma^l}).
 Moreover, we will  keep the other notation introduced in the course of  the partial proof of
Proposition     \ref{prop_Sigma_k_is_a_weakly_fibered_lamination}.
   Fix  an arbitrary point $\tilde x_0\in \widetilde F^l_{ij},$ and  write  
$$\big (\{\tilde x_0\}\times \Gr_k(\R^d)\big)\cap\Sigma^l_i=\{ (\tilde x_0, U_i):\ i\in I   \}.$$
    We need to show that the cardinal of the  index set $I$ is at most $jl.$
Suppose without loss of generality that $\tilde x_0\not\in\pi^{-1}(\Satur(\Nc_{k+1}))$ since otherwise $\iota^{-1}(\tilde x_0)=\varnothing$
by the construction of $\Sigma$ and $\overline{\Sigma}.$
For each $i\in I,$  there exists  $(\tilde x_i,V_i)\in \Sigma^l  \cap \widetilde\U_i\times  \Gr_k(\R^d) $ such that 
 $(\tilde x_i,V_i)$ on the  same leaf as $ (\tilde x_0, U_i)$ in $\overline\Sigma.$
The membership  $(\tilde x_i,V_i)\in \Sigma^l$ implies, by the  definition of $\Sigma^l,$ that instead of  (\ref{e:W_positive})
we have that 
  \begin{equation}\label{e:W_positive_new}
W_{\pi(\tilde x_i)}\Big(\widehat{\widetilde A}_k\cap \Omega((\overline\Sigma)_{(\tilde x_0, U_i)} )\Big) >1/l,\qquad i\in I,
\end{equation}
where $(\overline\Sigma)_{(\tilde x, U)}$ denotes the leaf of $\overline\Sigma$ passing through the point $(\tilde x, U),$
and  $\Omega((\overline\Sigma)_{(\tilde x, U)} )$ denotes the  space of all  continuous  paths $\omega$ defined on $[0,\infty)$ with image fully contained in  this leaf.
By Part 1) of Proposition     \ref{prop_Sigma_k_is_a_weakly_fibered_lamination},    the set $\widehat{\widetilde A}_k\cap \Omega\big((\overline\Sigma)_{(\tilde x_0, U_i)}\big)$ is $T$-totally invariant.
Consequently, applying  Proposition \ref{prop_Markov} (i) yields that for each $i\in I,$ instead of (\ref{eq_coutability_I}) we have the following estimate
\begin{equation*}
\int_{\widetilde L_{\tilde x_0}} p(\tilde x,\tilde x_i,1) W_{\tilde x}\Big(\widehat{\widetilde A}_k\cap \Omega((\overline\Sigma)_{(\tilde x_0, U_i)} )\Big)
 d\Vol_{\widetilde L_{\tilde x_0}}(\tilde x) >1/l.
\end{equation*}
Applying  Lemma  \ref{L:Harnack} and using that $\tilde x_0\in  \widetilde F^l_{ij},$ the last  inequality gives that
\begin{equation}
 \label{eq_coutability_I_new}
\int_{\widetilde L_{\tilde x_0}} p(\tilde x,\tilde x_0,1) W_{\tilde x}\Big(\widehat{\widetilde A}_k\cap \Omega((\overline\Sigma)_{(\tilde x_0, U_i)} )\Big)
 d\Vol_{\widetilde L_{\tilde x_0}}(\tilde x) >1/jl.
\end{equation}
Note that  $\widetilde L_{\tilde x_0}\cap \pi^{-1}(\Nc_{k+1})=\varnothing $ because
  $\tilde x_0\not\in\pi^{-1}(\Satur(\Nc_{k+1})).$
Consequently, as in (\ref{e:sum<1}),
we deduce from Lemma \ref{lem_stratifications}
 that, for each $n\geq 1$ and for each $\tilde x\in \widetilde L_{\tilde x_0},$
 \begin{equation*}
 \sum_{i\in I} W_{\tilde x}\Big(\widehat{\widetilde A}_k\cap \Omega((\overline\Sigma)_{(\tilde x_0, U_i)} )\Big)\leq 1.  
 \end{equation*}
  Integrating  the above inequality  over $\widetilde L_{\tilde x_0}$  with respect  to the probability measure $p(\tilde x,\tilde x_0,1)d\Vol_{\widetilde L_{\tilde x_0}}(\tilde x),$
 we get, instead of (\ref{e:integration_sum}), that
 \begin{equation*}
\sum_{i\in I}\int_{\widetilde L_{\tilde x_0} } p(\tilde x,\tilde x_0,1) W_{\tilde x}\Big(\widehat{\widetilde A}_k\cap \Omega((\overline\Sigma)_{(\tilde x_0, U_i)} ) \Big)d\Vol_{\widetilde L_{\tilde x_0}}(\tilde x)\leq  1.
\end{equation*}
 This,  combined with  (\ref{eq_coutability_I_new}), implies that the cardinal of    $I$ is at most $jl,$  as  desired.
\qed

\noindent
{\bf End of the proof of  Part 1) of Proposition    \ref{prop_Sigma_k_is_a_fibered_lamination}.} 
We only  need  to check  Definition  \ref{defi_fibered_laminations_new} (iv).
By  Theorem  \ref{T:cylinder_lami_is_conti_like}
and by Definition \ref{defi_continuity_like} (iv), we know that, for  each Borel subset $A\subset \widetilde X\times \Gr_k(\R^{d}),$  the map
$$
\widetilde X\times \Gr_k(\R^{d})\ni y\mapsto  W_y(A)\in[0,1]
 $$
is  Borel measurable.
On the  other  hand, recall from  (\ref{eq_Borel_Sigma}) that
  $\overline\Sigma_k$ is  a Borel subset of $\widetilde X\times \Gr_k(\R^{d}).$
 Consequently, for  each Borel subset $A\subset \overline\Sigma_k ,$  the map
$$
\overline\Sigma_k\ni y\mapsto  W_y(A)\in[0,1]
 $$
is  also Borel measurable, as desired.\qed

\noindent {\bf End of the proof of Proof of Part 2) of  Proposition    \ref{prop_Sigma_k_is_a_fibered_lamination}.}
Consider a  Borel probability measure $\mu$ on $X.$  Let $\iota:\ \overline\Sigma\to \widetilde X$ be a     fibered lamination over $(X,\Lc,g),$
and set $\tau:=\pi\circ\iota:\ \overline\Sigma\to X.$
 Following formula (\ref{eq_formula_bar_mu_new}), consider the  $\sigma$-finite  measure $\bar\nu:=\tau^*\bar\mu$ on $(\Omega(\overline\Sigma),\Ac(\overline\Sigma)).$
   
The following lemma is  needed.
 \begin{lemma}\label{lem_measurable_projection}
 If, for every  cylinder set $A\subset\Omega(\overline\Sigma),$
 the  image $\tau\circ A\subset \Omega(X,\Lc)$ is  $\bar\mu$-measurable,
 then  Definition  \ref{defi_fibered_laminations} (v)  is  satisfied, i.e.,  $\mu$ respects $\overline\Sigma.$  
  \end{lemma}
  \begin{proof}
  First  consider the case  where  $A$ is  a  countable   union of cylinder sets $A_i$  in $\Omega(\overline\Sigma).$
  Since  $\tau\circ A=\cup_{i=1}^\infty \tau\circ A_i, $ and each image  $\tau\circ A_i$  is  $\bar\mu$-measurable,
 so is     $\tau\circ A.$
  
 Now  we turn  to  the  general case where $A\in \Ac(\overline\Sigma).$ 
Note that the leaves of $\overline\Sigma$ are all simply connected.
 Consequently,
 applying  Proposition \ref{prop_algebras_new}
  yields  a decreasing  sequence
$(A_n),$  each $A_n$ being  a  countable union of mutually disjoint  cylinder sets such that $A\subset A_n$ and  that   $\bar\nu(A_n\setminus A)\to 0$ as $n\to\infty.$  By  the previous  case,
$(\tau\circ A_n)_{n=1}^\infty$  is a  decreasing  sequence  of $\bar\mu$-measurable  sets containing  $\tau\circ A.$
Moreover, we have that by 
$$
\bar\mu(\tau\circ A_n\setminus \tau\circ A)\leq \bar\mu(\tau\circ (A_n\setminus  A))\leq \bar\nu (A_n\setminus A)\to 0\qquad\text{as}\ n\to\infty, 
$$
where the  second inequality holds by Lemma \ref{lem_mu_nu}.
 Thus, $\tau\circ A$ is  $\bar\mu$-measurable.
 \end{proof}

By Lemma \ref{lem_measurable_projection},
  we need to show that  for every  cylinder set $A\subset\Omega(\overline\Sigma),$
 the  cylinder image $\tau\circ A\subset \Omega:=\Omega(X,\Lc)$ is  $\bar\mu$-measurable.
 Fix  such a cylinder set $A:=C\big (\{t_s, A_s\cap \overline\Sigma\}:m ),$ where
 each $A_s$ is a Borel set of $\widetilde X_{k,\widetilde{\mathcal A}}=\widetilde X\times\Gr_k(\R^d).$
 For the sake of simplicity, assume  without loss of generality that $t_m\leq 1.$

For $N:=(l,i,j,p,q,r)\in \N^6,$
 we set $F_N:= \tilde s_N (\widetilde E_N)\subset \overline\Sigma,$ where the local section $\tilde s_N$ as well as its domain of definition  $\widetilde E_N$ are given by   
(\ref{e:s_N}). Here we  make the convention that  $F_N:=\varnothing$ if  $\widetilde E_N=\varnothing.$
Observe  that if  $\widetilde E_N\not=\varnothing,$ then $\widetilde E_N$ 
 is a  plaque-saturated Borel subset of the  flow  tube 
$\U_{\widetilde{\mathcal U}_q},$
and 
   $F_N$   is a nonempty  plaque-saturated Borel subset of the  flow  tube 
$\U_{\widetilde{\mathcal U}_q}\times\Gr_k(\R^d)$
in the   cylinder lamination $(\widetilde X_{k,\mathcal A},\widetilde \Lc_{k,\mathcal A}).$
Let $\U_{\overline\Sigma,q}$ be the  flow  tube of the fibered lamination  $\overline\Sigma$ associated to the  good flow tube  $\U_{\widetilde{\mathcal U}_q}.$ Let $\U_{\mathcal U_q}:=\pi(\U_{\widetilde{\mathcal U}_q})$ be the flow tube in the lamination  $(X,\Lc)$
(see Definition \ref{defi_fibered_flow_tube}).
So $F_N,$ when it is nonempty,  is also a  plaque-saturated Borel subset of the  flow  tube  $\U_{\overline\Sigma,q}.$
 
 We deduce  easily from  (\ref{e:cover_local_sections}) that
 $$
 \Omega(\overline\Sigma)=\bigcup_{q=0}^\infty \Omega(1, \U_{\overline\Sigma,q}).
 $$
 On the other hand, summarizing what has been done  from  (\ref{e:cover_local_sections}) to the end of   the proof of     Part 2) of   Proposition \ref{prop_Sigma_k_is_a_weakly_fibered_lamination}, we have shown that
 for every $\widetilde \omega\in \Omega(\overline\Sigma),$
 there is  $N\in\N^6$ such that  $\omega[0,1]$ is contained in a plaque of  $F_N.$
 Consequently, we  get that 
 $$
 A=A\cap \Omega(\overline\Sigma) =\bigcup_{N=(l,i,j,p,q,r)\in\N^6} \Omega(1, \U_{\overline\Sigma,q})\cap C(\{t_s,A_s\cap F_N\}:m)  .
$$
  Recall that $\tau=\pi\circ\iota.$ So in order to prove  that $\tau\circ A$ is $\bar\mu$-measurable, it suffices to check that each set
  $\pi(A_N)$ is also $\bar\mu$-measurable, where
  \begin{equation*}
 A_N:=\iota\circ \Big (\Omega(1, \U_{\overline\Sigma,q})\cap C(\{t_s,A_s\cap F_N\}:m)\Big).
 \end{equation*}
 To this  end,  suppose without loss of generality that
$\widetilde E_N$ (and hence $F_N$) is  nonempty.
 Since  $\iota(\U_{\overline\Sigma,q})\subset \U_{\widetilde{\mathcal U}_q},$ it follows that
   $A_N\subset  A'_N, $ where
 $$A'_N:=\Omega(1, \U_{\widetilde{\mathcal U}_q})\cap C(\{t_s,B_{s,N}\}:m),\quad\text{with}\quad B_{s,N}:=\iota(A_s\cap F_N).$$
 Pick an arbitrary path $\tilde\omega\in A'_N.$ So  $\tilde\omega\in\Omega(1, \U_{\widetilde{\mathcal U}_q}).$
 Since $B_{s,N}\subset \widetilde E_N$ and   $\widetilde E_N$ 
 is a  plaque-saturated Borel subset of the  flow  tube 
$\U_{\widetilde{\mathcal U}_q},$ there is a unique   path $\eta\in \Omega(\overline\Sigma)$ such that
$\eta(t)=\tilde s_N(\tilde\omega(t))$ for   $t\in[0,1]$ and that $\tilde\omega=\iota\circ \eta$ (see Lemma  \ref{lem_identifications_spaces}).
Since  $\tilde s_N$ is a  local section of $\iota,$ we infer that  
 $$\eta\in \Omega(1, \U_{\overline\Sigma,q})\cap C(\{t_s,A_s\cap F_N\}:m).$$
So we have shown that $A_N=  A'_N. $ Hence, it remains to show that
$\pi\circ A'_N$ is $\bar\mu$-measurable.

Using that the projection $\pi|_{\U_{\widetilde{\mathcal U}_q}}:\ \U_{\widetilde{\mathcal U}_q}\to \U_{\mathcal U_q}$ is a  homeomorphism
which maps each plaque of $\U_{\widetilde{\mathcal U}_q}$ onto a plaque of $\U_{\mathcal U_q},$  we can show that
$\pi\circ A'_N=A''_N,$
where 
$$A''_N:=\Omega(1, \U_{\mathcal U_q})\cap C(\{t_s,C_{s,N}\}:m),\quad\text{with}\quad C_{s,N}:=\pi(B_{s,N}).$$
So it remains to show that $A''_N$ is $\bar\mu$-measurable.
A    problem arises as   $C_{s,N}=\pi(B_{s,N})=\pi\circ\iota (A_s\cap F_N) $ need not to be a Borel set of $ X$
although  $A_s\cap F_N$ is a Borel subset of $\widetilde X\times\Gr_k(\R^d).$ In the remainder of the proof,
we deal with cylinder sets $C(\{t_s,B_s\}:m),$ where each $B_s$ need not to be  a Borel set.
 However, applying  Theorem \ref{thm_measurable_projection} and using that $\pi$ is locally  homeomorphic,
we may find  a Borel set $D_{s,N}\subset X$ such that
$D_{s,N}\subset C_{s,N}$ and that $\mu(C_{s,N}\setminus D_{s,N})=0.$
Consider the following  set
$$
A'''_N:=\Omega(1, \U_{\mathcal U_q})\cap C(\{t_s,D_{s,N}\}:m)\subset  A''_N. 
$$
By  Part 3) of Lemma  \ref{lem_directed_cylinder_images_in_Ac}, each $A'''_N$ belongs to $\Ac(\Omega ).$
On  the other hand, we  see easily that
$$
 A''_N\setminus A'''_N\subset \bigcup_{s=1}^m  C(\{t_s, C_{s,N}\setminus D_{s,N}\}:1).
$$
The following lemma is  needed.
\begin{lemma}\label{lem_null_mu_measure_set}
For  every   set $E\subset X$ with $\mu(E)=0$ and $t\in\R^+,$ it holds that
 $$\bar\mu( C(\{t,E\}:1)=0.$$ 
\end{lemma}
\begin{proof}
Let $E'$ be a Borel subset of $X$ such that $E\subset E'$ and $\mu(E')=0.$
As $ C(\{t,E\}:1)\subset  C(\{t,E'\}:1),$ it suffices to show that  $\bar\mu( C(\{t,E'\}:1)=0.$ 
Since  $\mu$ is weakly harmonic, it follows  from  Theorem   \ref{thm_invariant_measures}  that
$$\bar\mu( C(\{t,E'\})=\bar\mu( C(\{0,E'\}).$$
On the other hand, by formula  (\ref{eq_formula_bar_mu}), $\bar\mu( C(\{0,E'\})=\mu(E')=0.$
This   completes the proof.\footnote{It is  relevant to note  that  here is  the place where the assumption of weak harmonicity of $\mu$
is fully   used (see also Remark \ref{R:R+}).}
\end{proof}
 Applying  Lemma \ref{lem_null_mu_measure_set} yields that $\bar\mu( C(\{t_s, C_{s,N}\setminus D_{s,N}\}:1))=0$
 for $1\leq s\leq m.$
 So  $$\bar\mu(A''_N\setminus A'''_N)\leq \sum_{s=1}^m \bar\mu( C(\{t_s, C_{s,N}\setminus D_{s,N}\}:1))=0.$$ 
This, combined with the  fact that  $A'''_N\in \Ac(\Omega),$
 implies that $A''_N$ is $\bar\mu$-measurable. \qed 
 \begin{remark}\label{R:End_B1}
We close  this  section  with the  following  discussion  on the optimality of the  hypotheses in Proposition  \ref{prop_local_currents_flow_tube}.
This  issue  is  related  to  Remarks \ref{R:th_main_1} and  \ref{R:End}.

 Observe  that only the following  weaker version of   Proposition  \ref{prop_local_currents_flow_tube}
 (and hence  Proposition \ref{prop_current_local})
 suffices for the validity of  the whole section.

 {\it  Let $\U\simeq \B\times\T$ be a flow
box  which is relatively compact in $X$. Then, there is a positive Radon
measure $\nu$ on $\T$ and for $\nu$-almost every $t\in \T$ there is a
measurable positive  function $h_t$ on $\B$ 
such that if $K$ is compact in $\B,$ 
the integral $\int_\T \|h_t\|_{L^1(K)}d\nu(t)$ is finite and
$$\int  fd\mu=\int_\T \Big(\int_\B h_t(y) f(y,t) d\Vol_t(y)\Big) d\nu(t)
$$
for every  continuous compactly supported function    $f$ on $\U.$
Here $\Vol_t(y)$  denotes the  volume form on $\B$ induced by the metric tensor $g|_{\B\times \{t\}}.$
}

In particular,   we  do not need that $h_t$ is  harmonic  on $\B$ 
with respect to the metric tensor $g|_{\B\times \{t\}}.$ Therefore,
the leafwise  Laplacian  is  not really needed in this  section.

 Finally, we  conclude that   the  whole section still  remains  valid if  we make   the following  weaker
 assumptions (i)--(iii) on the Riemannian lamination $(X,\Lc,g)$ and on the measure $\mu.$ 
 \begin{itemize}
 \item[(i)] $(X,\Lc,g)$ satisfies Hypothesis (H1).
 
 \item[(ii)] $\mu$ is  weakly harmonic.
 
 \item[(iii)] $\mu$  satisfies the above  weaker  version of  Proposition  \ref{prop_local_currents_flow_tube}.
 
 \end{itemize}
\end{remark}

  \section[A lamination  and its covering]{Relation between a lamination  and its covering lamination}
  
Let $(X,\Lc,g)$ be a  Riemannian continuous-like lamination  with  its covering  lamination projection
 $\pi:\ (\widetilde X,\widetilde \Lc,\pi^*g)\to (X,\Lc,g) .$  Following Definition 
 \ref{defi_algebras_Ac},  
let $\Omega:=\Omega(X,\Lc),$ $\widetilde\Omega:=\Omega(\widetilde X,\widetilde\Lc)$ be two  sample-path spaces, and let  $\Ac=\Ac(\Omega),$ $ \Ac(\widetilde\Omega)$ be two $\sigma$-algebras on them
respectively.   Note that $ \Ac(\widetilde\Omega)=\widetilde \Ac(\widetilde\Omega).$
For every $\omega\in\Omega $ let $\pi^{-1}(\omega)$ be the set
$$
\left\lbrace  \tilde \omega\in \widetilde\Omega:\ \pi\circ\tilde \omega=\omega\right\rbrace.
$$

We say  that  a function $\tilde f:\  \widetilde X\to \R$  (resp.   $\tilde F:\  \widetilde\Omega\to \R$)
is {\it constant  on fibers}\index{function!constant on fibers $\thicksim$}  if 
$\tilde f(\tilde x_1)=\tilde f(\tilde x_2)$ for all $x\in X$ and all  $\tilde x_1,\tilde x_2\in \pi^{-1}(x)$
(resp.             if            $\tilde F(\tilde \omega_1)=\tilde F(\tilde \omega_2)$ for all $\omega\in \Omega$
 and all  $\tilde \omega_1,\tilde \omega_2\in \pi^{-1}(\omega)$).
 Observe that  a function $\tilde f:\  \widetilde X\to \R$  (resp.   $\tilde F:\  \widetilde\Omega\to \R$) can be  written as $\tilde f=f\circ \pi$  for some function  $ f:\  X\to \R$
(resp.  can be written as $\tilde F=F\circ \pi$  for some function $ F:\  \Omega\to \R$)
if and only if $\tilde f$ (resp. $\tilde F$) is   constant  on fibers.

The  following result is needed  in the sequel.
 
\begin{lemma}\label{lem_change_formula}
(i)  Let  $\  A\in \Ac ,$ let
$x$ be a point in $X,$  and let $\tilde x_1,$ $\tilde x_2$ be two points of $\pi^{-1}(x).$
Then $$W_{\tilde x_1}(\pi^{-1} A)=  W_{\tilde x_2}(\pi^{-1} A)=
W_x( A) .$$
\\
(ii) Let $\tilde x\in\widetilde X,$ $\tilde A\in \Ac_{\tilde x}$ and  $x:=\pi(\tilde x).$
Then  $W_{\tilde x}(\tilde A)=W_x(\pi\circ \tilde A).$
\\
(iii)
 For a  bounded  measurable  function $F:\ \Omega(X,\Lc)\to\R,$   let 
$\tilde F:\ \Omega(\widetilde X,\widetilde \Lc)\to\R$ be the function defined  by
$\tilde F=F\circ \pi.$ To  each $x\in X$ we associate  an arbitrary (but fixed) element $\tilde x\in  \pi^{-1}(x).$
 Then for every positive finite  Borel  measure $\mu$  on $X,$
$$
\int_{\Omega } F(\omega) d\bar\mu(\omega)=  \int_X\Big (\int_{ \widetilde \Omega} \tilde F(\tilde\omega) dW_{\tilde x}(\tilde \omega)\Big )  d\mu(x).
$$
\end{lemma}
\begin{proof} 
Since   $  A\in \Ac ,$  we get  $  A\cap  \Omega( L_x)\in\Ac( L_x).$
Consequently, combining  Definition  \ref{defi_continuity_like} (ii)-(iii) and formula  (\ref{eq_formula_W_x}) yields that
$$W_{\tilde x_1}(\pi^{-1} A)=  W_{\tilde x_1}(\pi^{-1} (A\cap  \Omega( L_x)))=
W_x( A\cap  \Omega( L_x))=W_x(A) .$$
This proves  assertion (i).

To prove   assertion (ii) observe that
$$
W_x(\pi\circ \tilde A)=W_{\tilde x} (\pi^{-1}(\pi\circ \tilde A))= W_{\tilde x} (\pi^{-1}_{\tilde x}(\pi\circ \tilde A)),
$$
 where the  first  equality holds by  assertion (i).
 Since $\tilde A\in \Ac_{\tilde x},$  we have that $\pi^{-1}_{\tilde x}(\pi\circ \tilde A)=\tilde A.$
 Hence,  $W_x(\pi\circ \tilde A)=W_{\tilde x}(\tilde A),$ as  asserted.
 
To prove  assertion (iii), we first consider the case  where $F:=\otextbf_A$ for some $A\in\Ac.$
In this  case, the assertion  holds  by combining assertion (i) above and  formula (\ref{eq_formula_bar_mu}).
The general case  deduces  from  the  above case using Proposition  \ref{prop_simple_functions}. 
 \end{proof}   
 
 \section[Invariance of   Wiener  measures]{Invariance  of   Wiener  measures  with harmonic  initial  distribution}
 \label{section_invariance}
 This  section is  devoted  to the proof of   
   Theorem   \ref{thm_invariant_measures}. Let $(X,\Lc,g)$ be a  Riemannian continuous-like lamination  with  its covering  lamination projection
 $\pi:\ (\widetilde X,\widetilde \Lc,\pi^*g)\to (X,\Lc,g) .$  By Section \ref{subsection_Wiener_measures_with_holonomy},
 we construct  a  $\sigma$-algebra $\Ac$ on  $\Omega:=\Omega(X,\Lc)$ and a $\sigma$-algebra $\widetilde\Ac=\widetilde\Ac(\widetilde\Omega)=\Ac(\widetilde\Omega)$ on  $\widetilde\Omega:=\Omega(\widetilde X,\widetilde\Lc).$ 
 \begin{definition}\label{D:diffusion_measure}
 Let $\nu$ be a  positive finite Borel measure on $X$ and $t\in\R^+,$ then  $D_t\nu$ is 
 the   positive finite Borel measure  on $X$ (unique  in the  sense  of $\nu$-almost  everywhere) satisfies the following  condition
$$ 
 \int_X   D_tf(x) d\nu(x)= \int_X  f(x) d(D_t\nu)(x)
$$
 for every  bounded   measurable  function $f:\ X\to \R.$
\end{definition}
\begin{remark} It is  clear from the  definition  and the  identity  $$\int_{L_x} p(x,y,t)d\Vol_{L_x} (y)=1,\qquad x\in X,\ t\in\R^+,$$ (see  Chavel\index{Chavel} \cite{Chavel}) that  the masses of $\nu$ and  $D_t\nu$ are the same.
In particular,
when  $\nu$ is a  probability measure, so is  $D_t\nu.$
\end{remark}
   The following result is  needed.
   \begin{lemma}\label{lem_invariant_measures} Let $\mu$ be  a positive  finite Borel measure on $X.$
   Then, for every $ t\in\R^+$ and for every bounded  measurable  function $F:\ (\Omega,\Ac)\to (\R,\Bc(\R)),$
   $$
   \int_{\omega\in \Omega} F(T^t\omega) d\bar\mu(\omega)= \int_{x\in X}\Big (\int_{\Omega} F(\omega) dW_x(\omega)\Big )  d(D_t\mu)(x).
   $$
   \end{lemma}
   Taking  Lemma \ref{lem_invariant_measures} for granted, we  arrive at the

   \smallskip
   
   \noindent{\bf End  of the proof of
 Theorem   \ref{thm_invariant_measures}.} 
As  already observed  after the statement of  Theorem   \ref{thm_invariant_measures},  we only need to prove Part 1).
By Lemma  \ref{lem_invariant_measures}, Part 1)     will follow if one can show  that $D_1\mu=\mu$  (resp.
 $D_t\mu=\mu$ for all $t\in\R^+$).  But  this  identity  is an immediate consequence of  the  assumption that $\mu$  is very weakly  harmonic  (resp.  weakly  harmonic). 
   \qed
    
   \smallskip
   
   \noindent{\bf Proof of Lemma   \ref{lem_invariant_measures}.} 
Fix  a  time $t\geq 0.$ We will  show that for    each element $A\in\Ac,$
\begin{equation}\label{eq1_lem_invariant_measures}
 \int_{\omega\in \Omega} \otextbf_A(T^t\omega) d\bar\mu(\omega)\leq  \int_{x\in X}\Big (\int_{\Omega} \otextbf_A(\omega) dW_x(\omega)\Big )  d(D_t\mu)(x),
\end{equation}
where   $\otextbf_A$ is, as  usual,  the characteristic function   of  $A.$
 We  only  need  to prove  the lemma  for  every function $F$  which is  the characteristic function   of an element $A\in\Ac.$
 Taking (\ref{eq1_lem_invariant_measures}) for granted, and  applying (\ref{eq1_lem_invariant_measures}) to both $A$ and $\Omega\setminus A$ for each $A\in  \Ac,$ and  summing up both sides of the  obtained two inequalities, and noting  that
$\otextbf_A(\omega)+\otextbf_{\Omega\setminus A}(\omega)=1$ for all $\omega\in\Omega,$ we  deduce that 
 (\ref{eq1_lem_invariant_measures}) is, in fact, an equality. Using this  and  approximating  each 
  bounded  measurable  function $F:\ \Omega\to \R$ by simple  functions, i.e., by functions  which are a finite linear combination
  of  characteristic functions  (see  Proposition  \ref{prop_simple_functions}), the lemma follows.

So it remains  to establish   (\ref{eq1_lem_invariant_measures}). Set  
  $ A':= \pi^{-1}(A).$ By  Theorem \ref{prop_Wiener_measure},  $A'\in \widetilde\Ac.$
  By Part 1) of Proposition \ref{prop_cylinder_sets},
  let $\mathfrak{S}( \widetilde\Omega )$ be the  algebra on $\widetilde\Omega$ consisting  of  all sets  which are a  finite  union
 of cylinder sets in $\widetilde\Omega.$ Consider  the $\sigma$-finite  measure
$\nu:=\pi^\ast (D_t\mu)$ on  $\widetilde X.$ Let $\bar\nu$ be the  Wiener measure   with initial  distribution $\nu$
defined  on $\widetilde\Ac$ 
  by  formula (\ref{eq_formula_bar_mu}). So $\bar\nu$ is  countably additive on $\mathfrak{S}( \widetilde\Omega ).$ 
   Fix  an arbitrary $\epsilon>0.$  
Consequently, applying   Proposition \ref{prop_measure_theory} to the  measure space $(\widetilde \Omega,\Ac(\widetilde\Omega),\bar\nu)$
and the  algebra  $\mathfrak{S}( \widetilde\Omega )$ generating the $\sigma$-algebra $\widetilde\Ac$ yields  a  set  $\tilde A$  which  is  a countable union of cylinder sets
in $\widetilde\Omega$
 such that  
 \begin{equation}\label{eq_tildeA_lem_invariant_measures}
 A'\subset \tilde A\quad\text{and}\quad \bar\nu( \tilde A\setminus  A')<\epsilon.
 \end{equation}
 By Part 2) of Proposition \ref{prop_cylinder_sets}, we  may write  
 $$\tilde A=\bigcup_{p\in \N} \tilde A^p  :=\bigcup_{p\in \N}  C(\{t^p_i, \tilde A^p_i\}:  m_p), $$
 where
 the cylinder sets   $\tilde A^p=C(\{t^p_i, \tilde A^p_i\}:  m_p)$  are mutually disjoint.
 Consequently,    
   we obtain  a  disjoint  countable decomposition 
 $$\tilde B:=  (T^t)^{-1} \tilde A=\bigcup_{p\in \N} \tilde B^p  :=\bigcup_{p\in \N}  C(\{t^p_i+t, \tilde A^p_i\}:  m_p). $$
 By (\ref{eq_formula_W_x_without_holonomy}), we get, for  every $\tilde x\in \widetilde X,$ that
$$
W_{\tilde x}(\tilde B^p)=\Big (\widetilde D_{t^p_1+t}(\chi_{\tilde A^p_1}\widetilde D_{t^p_2-t^p_1}(\chi_{\tilde A^p_2}\cdots\chi_{\tilde A^p_{m_p-1}} \widetilde D_{t^p_{m_p}-t^p_{m_p-1}}(\chi_{\tilde A^p_m}      )\cdots))\Big) (\tilde x).
$$
 Consider the function  $\tilde H:\ \widetilde X\to [0,1]$ given  by
   $$
   \tilde H:=\sum_{p\in \N}\widetilde D_{t^p_1} \Big (\chi_{\tilde A^p_1}
  \widetilde D_{t^p_2-t^p_1}(\chi_{\tilde A^p_2}\cdots\chi_{\tilde A^p_{m_p-1}} \widetilde D_{t^p_{m_p}-t^p_{m_p-1}}(\chi_{\tilde A^p_m}    )\cdots)\Big).                            
   $$
Observe that, for  every $\tilde x\in\widetilde X,$
$$ 
\int_{\widetilde \Omega_{\tilde x}} \otextbf_{\tilde A}(T^t\tilde \omega)  dW_{\tilde x}(\tilde \omega)
= W_{\tilde x}(\tilde B)= \sum_{p\in \N} W_{\tilde x}(\tilde B^p)
  =(\widetilde D_t\tilde H)(\tilde x).
  $$
  This, combined  with $A'\subset\tilde A,$ implies that, for     every $\tilde x\in  \widetilde X,$ 
  $$
  \int_{\widetilde \Omega_{\tilde x}} \otextbf_{ A'}(T^t\tilde \omega)  dW_{\tilde x}(\tilde \omega)\leq  \int_{\widetilde \Omega_{\tilde x}} \otextbf_{\tilde A}(T^t\tilde \omega)  dW_{\tilde x}(\tilde \omega)
 =(\widetilde D_t\tilde H)(\tilde x).  
  $$
  Applying Lemma \ref{lem_change_formula}  (i) to the left hand side and using  that $\pi^{-1}(A)=A',$ we get,  for  every $x\in X$ and  every $\tilde x\in \pi^{-1}(x),$ that
  $$
  \int_{ \Omega_x} \otextbf_A(T^t\omega)  dW_x( \omega)\leq (\widetilde D_t\tilde H)(\tilde x).
  $$
 By Definition \ref{defi_continuity_like} (i), we may find  a Borel measurable map $s:\ X\to\widetilde X$ such that $s(x)\in\pi^{-1}(x),$ $x\in X.$ Therefore, integrating both sides of    
  the last  line with respect  to  $\mu$ and  applying Lemma \ref{lem_change_formula}  (ii), 
  we infer that
 $$
   \int_{\omega\in \Omega} \otextbf_A (T^t\omega)  d\bar\mu(\omega)= \int_{x\in X}\Big (\int_{ \Omega_x} \otextbf_A (T^t \omega)  dW_x( \omega)\Big )  d\mu(x)
\leq \int_X (\widetilde D_t\tilde H)(s( x)) d\mu(x).
   $$ 
  Since  $\epsilon>0$  is  arbitrarily fixed, we conclude that the  proof of (\ref{eq1_lem_invariant_measures}) will be complete  if one  can show  that 
   \begin{equation}\label{eq_reduction_lem_invariant_measures}
 \int_X (\widetilde D_t\tilde H)(s( x)) d\mu(x)\leq  \epsilon+
    \int_{x\in X}\Big (\int_{\Omega} \otextbf_A(\omega) dW_x(\omega)\Big )  d(D_t\mu)(x).
    \end{equation}
   To this end  observe that,  for  every $\tilde x\in\widetilde X,$
  $$
  \tilde H(\tilde x)= W_{\tilde x}(\tilde A)= W_{\tilde x}( A')+ W_{\tilde x}(\tilde A\setminus A')
 $$
Acting  $\widetilde D_t$ on the last  equality,  we obtain that, for every $x\in X,$
 $$\widetilde D_t\tilde H(s( x))= (\widetilde D_t  W_{\bullet}( A'))(s( x))+ (\widetilde D_t  W_{\bullet}( \tilde A\setminus A'))(s( x)).
  $$
   By Lemma \ref{lem_change_formula} (i), 
 $W_{\tilde x}(A')=W_x(A)$  for every $\tilde x\in\pi^{-1}(x).$ So $W_{\bullet}(A')$ is   constant on fibers and
   $  W_{\bullet}(A')  =W_{\bullet}(A)\circ \pi.$ Hence, 
  by  an application of  Proposition   \ref{prop_heat_difusions_between_L_and_its_universal_covering} we get that
  $\widetilde D_t  W_{\bullet}( A')=(D_t W_{\bullet}(A)) \circ \pi.$
  Putting all these  together, we get that
  $$
  \int_X (\widetilde D_t\tilde H)(s( x)) d\mu(x)=\int_X (D_t W_{\bullet}(A))(x)  d\mu(x)
  +\int_X (\widetilde D_t     W_{\bullet}( \tilde A\setminus A'))   (s( x)) d\mu(x).
  $$
  Note that by Definition \ref{D:diffusion_measure},  
$$
 \int_X(D_t W_{\bullet}(A))(x)  d\mu(x)  = \int_X   W_{\bullet}(A) d(D_t\mu)= \int_{x\in X}\Big (\int_{\Omega} \otextbf_A(\omega) dW_x(\omega)\Big )  d(D_t\mu)(x).
 $$
Consequently, (\ref{eq_reduction_lem_invariant_measures}) is reduced to showing that
\begin{equation}\label{eq_second_reduction_lem_invariant_measures}
\int_X (\widetilde D_t     W_{\bullet}( \tilde A\setminus A'))   (s( x)) d\mu(x)<\epsilon.
  \end{equation}
To do this  we   
use  $\nu=D_t\mu$ and apply Lemma \ref{lem_change_formula}  (ii) in order to obtain 
\begin{eqnarray*}
  \bar\nu( \tilde A\setminus  A')&=& \int_X \Big (  \sum_{\tilde x\in\pi^{-1}(x)} W_{\tilde x} (\tilde A\setminus  A')   \Big) d (D_t\mu)(x)\\
  &=&  \int_X (D_t K) (x) d \mu(x),
 \end{eqnarray*}
where $K:\ X\to  \R^+$  given by
$$ K(x):=  \sum_{\tilde x\in\pi^{-1}(x)} W_{\tilde x} (\tilde A\setminus  A'),\qquad  x\in X. $$
 Therefore,  this, coupled with (\ref{eq_tildeA_lem_invariant_measures}), gives that
 $$
  \int_X (D_t K) (x) d \mu(x)\leq   \bar\nu( \tilde A\setminus  A')<\epsilon.
 $$
Using   formula (\ref{eq1_heat_kernel})   we  get that for every $x\in X$ and $\tilde x\in\pi^{-1}(x),$
$$
(D_t K)(x)= \int_{y\in L_x}\Big (\sum_{ \tilde y\in\pi^{-1}(y)}\tilde p(\tilde x,  \tilde y,t)\Big) \Big (\sum_{\tilde y'\in\pi^{-1}(y)} W_{\tilde y'} (\tilde A\setminus  A')
\Big)d\Vol_{L_x}(y).
$$
  The right hand  side  in the  last line  
is   greater than
$$
\int_{y\in L_x}\Big (\sum_{ \tilde y\in\pi^{-1}(y)}\tilde p(\tilde x,  \tilde y,t)\Big)  W_{\tilde y} (\tilde A\setminus  A')
d\Vol_{L_x}(y)=   (\widetilde D_t     W_{\bullet}( \tilde A\setminus A'))   (\tilde x)  .
$$    
Choosing $\tilde x:=s(x)$ and  integrating  the last  inequality over $X,$ we  get that
$$
\int_X (\widetilde D_t     W_{\bullet}( \tilde A\setminus A'))   (s( x)) d\mu(x)\leq
 \int_X (D_t K) (x) d \mu(x)\leq   \bar\nu( \tilde A\setminus  A')<\epsilon.
$$
This  proves (\ref{eq_second_reduction_lem_invariant_measures}), and  thereby completes  the lemma.
\hfill $\square$
 \section[Ergodicity  of   Wiener  measures]{Ergodicity  of   Wiener  measures  with ergodic harmonic initial distribution}
 \label{subsection_Ergodicity}

The  first part of this  section is  devoted  to the proof of the following theorem  which has already  been  used in the proof of 
Theorem \ref{thm_totally_invariant_set_in_covering_lamination}. This  result  may be  regarded as Akcoglu's ergodic theorem %
\index{Akcoglu!$\thicksim$'s ergodic theorem}%
\index{theorem!Akcoglu's ergodic $\thicksim$}%
   (see Theorem 2.6 in \cite[p. 190]{Krengel})
in the  context of measurable  laminations.
\begin{theorem}\label{lem_Akcoglu}
Let $(X,\Lc,g)$ be  a  Riemannian measurable lamination satisfying  Hypothesis (H1). Let $\mu$ be a  probability Borel measure on $X$
which is ergodic on $(X,\Lc).$ Assume in addition that $\mu$ is very weakly  harmonic, i.e, 
$$
\int_X D_1f d\mu= \int_X fd\mu,\qquad  f\in L^1(X,\mu).
$$
(i) If  $D_1f=f$  $\mu$-almost everywhere for some $f\in L^1(X,\mu),$  then  $f=\const$ $\mu$-almost everywhere.
\\
(ii) For every $f\in L^1(X,\mu),$   ${1\over n}\sum_{i=0}^{n-1} D_kf$ converges  to   $\int_X f d\mu $  $\mu$-almost everywhere.
\end{theorem}
\begin{proof}
To  prove  assertion (i), suppose in order to reach a contradiction  that there is a non-constant function  $f\in L^1(X,\mu)$ such that  $D_1f=f$  $\mu$-almost everywhere.  
Let $f^+=\max(f,0).$ Using that $D_1f=f$ $\mu$-almost everywhere, it is  easy to see that  $f^+\leq  D_1f^+$ $\mu$-almost everywhere. This, combined with  the  assumption that $\mu$ is very weakly harmonic, implies that
$$
\int_X f^+d\mu\leq\int_X D_1f^+d\mu=\int_X f^+d\mu.
$$
Hence, $f^+= D_1f^+$ $\mu$-almost everywhere. 

Using $-f$ instead of $f,$  the above argument also shows that  $f^-= D_1f^-$ $\mu$-almost everywhere, where $f^-=\max(-f,0).$  More generally, by similar arguments we can prove that  $f_{\alpha\beta}=D_1f_{\alpha\beta}$ $\mu$-almost everywhere, where,
for any real numbers  $\alpha,$ $\beta$ with $\alpha <\beta,$
$$
 f_{\alpha\beta}(x):=
 \begin{cases}
 1, &  f(x)\geq \beta;\\
 {f(x)-\alpha\over  \beta-\alpha}, &  \alpha <f(x)<\beta;\\
 0, &  f(x)\leq \alpha.
 \end{cases}
$$  
Since  $f$ is  not equal to a constant  $\mu$-almost everywhere, there exists $\beta\in\R$ such that 
the set $A:=\{x\in X:\ f(x)\geq \beta\}$ satisfies $0<\mu(A)<1.$
Setting  $\alpha_n:=\beta-{1\over n},$ $n\in \N^*,$ 
we get that  $f_{\alpha_n\beta}\to\chi_A$ as 
 $n\to\infty.$ Since all functions $f_{\alpha_n\beta}$ are bounded we deduce from the equality $f_{\alpha_n\beta}=D_1f_{\alpha_n\beta}$   
and the last limit that  $\chi_A=D_1\chi_A.$ Hence, there is  a leafwise saturated  Borel set $A_0\subset X$ such that
$\mu\big((A\setminus A_0)\cup (A_0\setminus A)\big) =0.$  So $\mu(A_0)=\mu(A)$
and  $0<\mu(A_0)<1.$ This  contradicts the assumption that $\mu$ is  ergodic. 
 Hence,  assertion (i)  follows.

 To prove  assertion (ii), recall from (\ref{eq3_heat_kernel}) that $\|D_1f\|_{L^\infty(X)}\leq \|f\|_{L^\infty(X)}  $  for every  bounded  measurable  function $f.$
This, combined with  the $D_1$-invariance of $\mu$ and an interpolation argument, implies that 
$D_1$ 
is  a positive contraction  in all $L^p(X,\mu)$ $(1<p<\infty).$
Therefore, by Akcoglu's ergodic theorem\index{Akcoglu!$\thicksim$'s ergodic theorem}\index{theorem!Akcoglu's ergodic $\thicksim$}  (see Theorem 2.6 in \cite[p. 190]{Krengel}),  
${1\over n}\sum_{i=0}^{n-1} D_kf$ converges  to a function  $f^*\in L^1(X,\mu)$  $\mu$-almost everywhere.
This implies that  $D_1f^*=f^*$   $\mu$-almost everywhere and  $\int  f^*d\mu=\int fd\mu.$
 By assertion (i), $f^*$  is  equal to a  constant  $\mu$-almost everywhere. 
So this  constant is $\int fd\mu.$ This completes the proof.
\end{proof}

The remainder of the  section  is  devoted to the 

\smallskip

\noindent{\bf End of the 
proof of Theorem \ref{thm_ergodic_measures}.} By Proposition \ref{P:lami-is-cont-like}, Part 2) is clearly a consequence
of Part 1).  
 The proof of Part 1)  is divided  into  two steps.
\\
{\bf Step 1: }  {\it Ergodicity $T$   on $(\Omega,\Ac,\bar\mu)$ implies the ergodicity of $\mu.$}

Let $A$ be a leafwise  saturated  Borel subset of $X.$ Let $\Omega(A)$ be the  set consisting of all
$\omega\in\Omega(X,\Lc)$ which are contained  in $A.$ Clearly, $\Omega(A)$ is $T$-invariant.
By the ergodicity of   $T$   acting on $(\Omega,\Ac,\bar\mu),$    $\bar\mu(\Omega(A))$ is either  $1$ or $0.$   
Hence,  we deduce from  formula (\ref{eq_formula_bar_mu}) that $\mu(A)$ is either  $1$ or $0.$

\noindent {\bf Step 2: }  {\it Ergodicity  of $\mu$ implies  the ergodicity of $T$   on $(\Omega,\Ac,\bar\mu).$}

 Let 
   $\pi:\ (\widetilde X,\widetilde\Lc,\pi^*g)\to (X,\Lc,g)$ be the    corresponding   covering  lamination projection
associated to the  Riemannian continuous-like lamination $(X,\Lc,g),$ and  set, as  usual,
 $\widetilde \Omega:=\Omega(\widetilde X,\widetilde\Lc).$
 We need to prove that $\bar\mu(A)$ is  equal to either $0$ or $1$ for every $T$-totally invariant set $A\in \Ac.$
 Fix  such a set  $A$ and let $\tilde A:=\pi^{-1}(A)\in\Ac(\widetilde \Omega).$
 By Remark \ref{rem_fibered_laminations} and \ref{rem_fibered_laminations_new},
 consider the  trivial
   fibered lamination $\Sigma$ over $(X,\Lc,g ),$ that is,   $\Sigma:=\widetilde X$ and   $\iota:\  \Sigma\to \widetilde X$
is the  identity. We know  by Remark \ref{rem_fibered_laminations_new} that $\mu$ respects this fibered lamination.
  Clearly, $\Ac(\widetilde \Omega)=\Ac(\Sigma),$ and hence
$\tilde A\in\Ac(\Sigma)$ is  also $T$-totally invariant.  
By Theorem \ref{thm_totally_invariant_set_in_covering_lamination},
    $  \| \otextbf_{\tilde A}  \|_\ast$ is equal to  either $0$ or $1.$
   On the other hand, $ \otextbf_{\tilde A}=   \otextbf_{ A} \circ\pi$ because   $\tilde A:=\pi^{-1}(A).$
Consequently,  
$$  \| \otextbf_{\tilde A}  \|_\ast=\int_{x\in X}  W_x(A)d\mu(x)=\bar\mu(A).$$ 
Hence, $\bar\mu(A)$ is  equal to either $0$ or $1$ as  desired.
\qed

Here is a  version of Theorem \ref{thm_ergodic_measures}  for cylinder laminations.
This  result has been repeatedly used in  Chapter \ref{section_Main_Theorems}.

\begin{corollary}\label{cor_Birkhoff_ergodic_thm}
Let    $\mathcal A$ be a cocycle on a Riemannian lamination $(X,\Lc,g)$ and $\mu$ a  harmonic probability measure
on $X$ which is  ergodic. 
Let $\nu$ be  an  element in $\Har_\mu(X_{\mathcal A})$ which is also ergodic.
Let $\bar\nu$ be the Wiener measure with initial  distribution $\nu$ given by  (\ref{eq_formula_bar_mu})
(this  is  possible  by Definition \ref{defi_continuity_like} and  Theorem  \ref{T:cylinder_lami_is_conti_like}).
Let $\hat\nu$ be the  natural extension of $\bar\nu$ on $\widehat\Omega_{\mathcal A}.$
\\
1)   Then $\bar\nu$ is $T$-ergodic on  $\Omega_{\mathcal A}$
and $\hat\nu$ is $T$-ergodic on  $\widehat\Omega_{\mathcal A}$
\\
2)  
Let  $f:\  \Omega_{\mathcal A}\to\R$ be a $\bar\mu$-integrable function.
Then  there is a  constant  $ \alpha$     
such that for $\nu$-almost every $(x,u)\in X_{\mathcal A},$  and for $W_x$-almost every $\omega,$
$
{1\over n} \sum_{i=0}^{n-1}f(T^i(\omega,u))\to  \alpha.
$
Moreover, $\int_{\Omega_{\mathcal A}} fd\bar\nu= \alpha.$
\\
3)  Let $\hat\nu$ be the  natural extension of $\bar\nu$ on $\widehat\Omega_{\mathcal A}.$
Let  $f:\  \widehat\Omega_{\mathcal A}\to\R$ be a $\hat\mu$-integrable function.
Then  there is a  constant $\alpha$     
such that for $\hat\nu$-almost every $(\hat\omega,u)\in \widehat\Omega_{\mathcal A},$  
$
{1\over n} \sum_{i=0}^{n-1} f(T^{-i}(\hat\omega,u))\to  \alpha.
$
Moreover, $\int_{\widehat\Omega_{\mathcal A}} fd\hat\nu= \alpha.$
\end{corollary}
\begin{proof}

 
Part 2) and  3)  follows from  combining Part 1) and the Birkhoff ergodic theorem\index{Birkhoff!$\thicksim$ ergodic theorem}
\index{theorem!Birkhoff ergodic $\thicksim$}.
So it suffices to prove Part 1). By \cite[p. 241]{CornfeldFominSinai}, if $\bar\nu$ is $T$-ergodic on  $\Omega_{\mathcal A},$
then $\hat\nu$ is $T$-ergodic on  $\widehat\Omega_{\mathcal A}.$
So it remains  to  show that  $\bar\nu$ is $T$-ergodic on  $\Omega_{\mathcal A}.$

Arguing as in  Step 2 of the proof of  Theorem \ref{thm_ergodic_measures} and replacing the lamination $(X,\Lc,g)$ (resp.  the measure $\mu$) with
$(X_{\mathcal A}, \Lc_{\mathcal A},\pr_1^*g)$ (resp.  with  $\nu\in \Har_\mu(X_{\mathcal A})$), the  desired conclusion follows.
 The  point  here is that the lamination $(X_{\mathcal A}, \Lc_{\mathcal A},\pr_1^*g)$ is  Riemannian  continuous-like
 by Theorem  \ref{T:cylinder_lami_is_conti_like}. Therefore, we  still apply Remark \ref{rem_fibered_laminations} and \ref{rem_fibered_laminations_new} and Theorem \ref{thm_totally_invariant_set_in_covering_lamination}.
\end{proof}
 
 \smallskip
 

%% file: biblio.tex

\bibliographystyle{amsalpha}
